\documentclass{memo-l}

\usepackage{amsmath}
\usepackage{amsfonts}
\usepackage{amsthm}
\usepackage{amssymb}
\usepackage{bm}
\usepackage{url}
\usepackage{xcolor}
\usepackage{framed}
\usepackage[normalem]{ulem}

\usepackage{mathtools}
\usepackage{tikz}
\usepackage{tikz-cd}
\usetikzlibrary{matrix,arrows,positioning,calc,shapes}
\usetikzlibrary{arrows.meta}

\usepackage{algorithm}
\usepackage[noend]{algorithmic}
\usepackage{booktabs}
\usepackage{graphicx}
\usepackage{verbatim}

\usepackage{dsfont}
\usepackage{mathrsfs}
\usepackage{subcaption}
\usepackage[cmtip,arrow]{xy}
\usepackage{pb-diagram,pb-xy}
\usepackage{multirow}
\usepackage{phoenician}

\usepackage[hidelinks]{hyperref}

\newtheorem{theorem}{Theorem}[section]

\theoremstyle{plain}  
\newtheorem{thm}{Theorem}[section]
\newtheorem{lem}[thm]{Lemma}
\newtheorem{prop}[thm]{Proposition}
\newtheorem{cor}[thm]{Corollary}

\theoremstyle{definition}  
\newtheorem{defn}[thm]{Definition}

\newtheorem{ex}[thm]{Example}

\theoremstyle{remark}  
\newtheorem{rem}[thm]{Remark}

\newcommand{\setof}[1]{\left\{ {#1}\right\}}
\newcommand{\setdef}[2]{\left\{{#1}\,\left|\,{#2}\right.\right\}}

\newcommand{\bg}{{\bf g}}

\newcommand{\bq}{{\bf q}}
\newcommand{\bu}{{\bf u}}
\newcommand{\bv}{{\bf v}}
\newcommand{\bw}{{\bf w}}

\newcommand{\bE}{{\bf E}}

\newcommand{\bG}{{\bf G}}

\newcommand{\bL}{{\bf L}}

\newcommand{\F}{{\mathbb{F}}}
\newcommand{\I}{{\mathbb{I}}}

\newcommand{\N}{{\mathbb{N}}}
\renewcommand{\P}{{\mathbb{P}}}

\newcommand{\R}{{\mathbb{R}}}

\newcommand{\Z}{{\mathbb{Z}}}

\newcommand{\bzero}{{\bf 0}}
\newcommand{\bone}{{\bf 1}}

\newcommand{\cB}{{\mathcal B}}

\newcommand{\cD}{\mathcal{D}}

\newcommand{\cF}{{\mathcal F}}
\newcommand{\cG}{{\mathcal G}}
\newcommand{\cH}{{\mathcal H}}
\newcommand{\cI}{{\mathcal I}}

\newcommand{\cK}{{\mathcal K}}

\newcommand{\cM}{{\mathcal M}}
\newcommand{\cN}{{\mathcal N}}

\newcommand{\cQ}{{\mathcal Q}}

\newcommand{\cS}{{\mathcal S}}
\newcommand{\cT}{{\mathcal T}}
\newcommand{\cU}{{\mathcal U}}
\newcommand{\cV}{{\mathcal V}}
\newcommand{\cW}{{\mathcal W}}
\newcommand{\cX}{{\mathcal X}}
\newcommand{\cY}{{\mathcal Y}}
\newcommand{\cZ}{{\mathcal Z}}

\newcommand{\sA}{{\mathsf A}}

\newcommand{\sI}{{\mathsf I}}
\newcommand{\sJ}{{\mathsf J}}
\newcommand{\sK}{{\mathsf K}}
\newcommand{\sL}{{\mathsf L}}
\newcommand{\sM}{{\mathsf M}}
\newcommand{\sN}{{\mathsf N}}
\newcommand{\sO}{{\mathsf O}}
\newcommand{\sP}{{\mathsf P}}
\newcommand{\sQ}{{\mathsf Q}}
\newcommand{\sR}{{\mathsf R}}

\newcommand{\sT}{{\mathsf T}}

\newcommand{\poeZ}{\leq_\Z}

\newcommand{\poeN}{\leq_\N}

\newcommand{\mvmap}{\rightrightarrows}

\newcommand{\sAtt}{{\mathsf{ Att}}}

\newcommand{\sInvset}{{\mathsf{ Invset}}}

\newcommand{\sMG}{{\mathsf{ MG}}}

\newcommand{\sSub}{{\mathsf{ Sub}}}

\newcommand{\sABlock}{{\mathsf{ ABlock}}}

\newcommand{\darrow}{\leftrightarrow}

\newcommand{\Inv}{\mathop{\mathrm{Inv}}\nolimits}

\newcommand{\Int}{\mathop{\mathrm{int}}\nolimits} 
\newcommand{\cl}{\mathop{\mathrm{cl}}\nolimits}
\newcommand{\bdy}{\mathop{\mathrm{bd}}\nolimits}

\newcommand{\support}{\mathop{\rm supp}}

\newcommand{\id}{\mathop{\mathrm{id}}\nolimits}

\DeclareMathOperator{\sgn}{sgn}

\definecolor{gray85}{gray}{0.85} 
\definecolor{gray8}{gray}{0.8} 
\definecolor{gray7}{gray}{0.7} 
\definecolor{gray6}{gray}{0.6} 
\definecolor{gray5}{gray}{0.5} 
\definecolor{gray4}{gray}{0.4} 
\definecolor{gray35}{gray}{0.35} 

\colorlet{grading21}{blue!100!white}

\colorlet{grading20}{blue!80!white}
\colorlet{grading19}{orange!80!white}

\colorlet{grading18}{blue!60!white}
\colorlet{grading17}{orange!60!white}
\colorlet{grading16}{green!60!white}

\colorlet{grading15}{blue!40!white}
\colorlet{grading14}{orange!40!white}
\colorlet{grading13}{green!40!white}
\colorlet{grading12}{red!40!white}

\colorlet{grading11}{blue!20!white}
\colorlet{grading10}{orange!20!white}
\colorlet{grading9}{green!20!white}
\colorlet{grading8}{red!20!white}

\colorlet{grading7}{teal!80!white}

\colorlet{grading6}{teal!60!white}
\colorlet{grading5}{yellow!80!white}

\colorlet{grading4}{teal!40!white}
\colorlet{grading3}{yellow!40!white}

\colorlet{grading2}{teal!20!white}
\colorlet{grading1}{yellow!20!white}

\newcommand{\blup}{{b}}

\newcommand{\btor}{\bar{b}}

\newcommand{\dec}{\hat{\xi}}

\newcommand{\Ex}{\mathrm{Ex}}

\newcommand{\ctoy}{\mathrm{sb}}
\newcommand{\proj}[1]{\mathrm{proj}_{#1}}
\newcommand{\cross}{\mathrm{Cross}}
\newcommand{\sstarr}{\mathrm{star}_\cX}
\newcommand{\posetF}[1]{\leq_{\bar{\cF}_{#1}}}

\newcommand{\Janus}{\mathcal{X}_\mathtt{J}}
\newcommand{\precJanus}{\prec_\mathtt{J}}

\newcommand{\gradJanus}{\pi_\mathtt{J}}

\newcommand{\mCellX}{\mathfrak{mc}_\cX}
\newcommand{\mSuppX}{\mathfrak{ms}_\cX}

\newcommand{\Spineb}{\mathfrak{sp}_b}
\newcommand{\SpineJ}{\mathfrak{sp}_{\mathtt{J}}}
\newcommand{\mTile}{\mathfrak{mt}_{\mathtt{J}}}
\newcommand{\Add}{\mathfrak{Add}_{\mathtt{J}}}
\newcommand{\Rem}{\mathfrak{Rem}_{\mathtt{J}}}
\newcommand{\Pl}{\mathfrak{F2}}

\newcommand{\mBn}[1]{B_2(\bar{S}(\blup({#1})))\cap\Janus^{(N)}}
\newcommand{\FibTri}{\mathfrak{F3}}
\newcommand{\mFiber}{\mathfrak{mF2}}
\newcommand{\mAdd}{\mathfrak{mAdd}_{\mathtt{J}}}
\newcommand{\mRem}{\mathfrak{mRem}_{\mathtt{J}}}

\newcommand{\gab}{\mathrm{GB}}
\newcommand{\Rec}{\mathrm{Rect}}
\newcommand{\direc}{\mathrm{dir}}

\newcommand{\activeset}{\mathrm{Act}}
\newcommand{\rmap}[1]{o_{#1}}
\newcommand{\rook}{\Phi}
\newcommand{\cycleof}{\mathrm{Orb}}
\newcommand{\lap}{\bL}

\newcommand{\bbarI}{{\bar{\bar{I}}}}
\newcommand{\bbarbI}{{\bar{\bar{\I}}}}

\newcommand{\bbdy}{{\mathrm{bdy}}}

\newcommand{\minD}{{\mathrm{minD}_\cX}}
\newcommand{\maxD}{{\mathrm{maxD}_\cX}}

\newcommand{\PosDriftCells}{{\mathrm{d}_{+}}}
\newcommand{\NegDriftCells}{{\mathrm{d}_{-}}}
\newcommand{\PosNegDriftCells}{{{\mathrm{d}_{\pm}}}}

\newcommand{\PosIndecisiveDrifts}{{\mathrm{d}_{+}^{-1}}}
\newcommand{\NegIndecisiveDrifts}{{\mathrm{d}_{-}^{-1}}}
\newcommand{\PDcells}{\mathrm{PD}}

\newcommand{\TP}{{\mathrm{TP}}}
\newcommand{\Geo}{{\mathrm{Geo}}}

\newcommand{\Nul}{{\mathrm{Nul}}}
\newcommand{\recG}{{\bf RecG}}
\newcommand{\mdpt}{{\mathrm{Mid}}}

\newcommand{\sMR}{{\mathsf{MR}}}
\newcommand{\SCC}{{\mathsf{SCC}}}
\newcommand{\RC}{{\mathsf{RC}}}
\newcommand{\GRC}{{\mathsf{GRC}}}
\newcommand{\CRC}{{\mathsf{CRC}}}
\newcommand{\Dec}{{\mathsf{Back}}}
\newcommand{\Back}{{\mathsf{Back}}}
\newcommand{\Stable}{{\mathcal{SC}}}
\newcommand{\Top}{{\mathrm{Top}}}

\newcommand{\trans}{\mathrm{Trans}}

\newcommand{\defcell}[4]{\left[ \left(\begin{smallmatrix}
        #1 \\ #2
    \end{smallmatrix}\right),\left(\begin{smallmatrix}
        #3 \\ #4
    \end{smallmatrix}\right)\right]}

\newcommand{\defcellb}[6]{\left[ \left(\begin{smallmatrix}
        #1 \\ #2 \\ #3
    \end{smallmatrix}\right),\left(\begin{smallmatrix}
        #4 \\ #5 \\ #6
    \end{smallmatrix}\right)\right]}  

\newcommand{\defcellc}[8]{\left[ \left(\begin{smallmatrix}
        #1 \\ #2 \\ #3 \\ 0 \\ 0
    \end{smallmatrix}\right),\left(\begin{smallmatrix}
        #4 \\ #5 \\ #6 \\ #7 \\ #8
    \end{smallmatrix}\right)\right]}      

\DeclareMathOperator{\source}{Source}
\DeclareMathOperator{\target}{Target}

\newtheorem{lemma}[theorem]{Lemma}

\numberwithin{section}{chapter}
\numberwithin{equation}{chapter}

\makeindex

\begin{document}

\frontmatter

\title{Global Dynamics of Ordinary Differential Equations: Wall Labelings, Conley Complexes, and Ramp Systems}

\author{Marcio Gameiro}
\address{Department of Mathematics, Rutgers University, Piscataway, NJ, 08854, USA}
\curraddr{}
\email{gameiro@math.rutgers.edu}
\thanks{}

\author{Tom\'{a}\v{s} Gedeon}
\address{Department of Mathematical Sciences, Montana State University, Bozeman, MT, 59717, USA}
\curraddr{}
\email{tgedeon@montana.edu}
\thanks{}

\author{Hiroshi Kokubu}
\address{Department of Mathematics, Kyoto University, Kyoto 606-8502, Japan}
\curraddr{}
\email{kokubu.hiroshi.3u@kyoto-u.ac.jp}
\thanks{}

\author{Konstantin Mischaikow}
\address{Department of Mathematics, Rutgers University, Piscataway, NJ, 08854, USA}
\curraddr{}
\email{mischaik@math.rutgers.edu}
\thanks{}

\author{Hiroe Oka}
\address{Ryukoku University, Seta, Otsu 520-2123, Japan}
\curraddr{}
\email{hiroecho@gmail.com}
\thanks{}

\author{Bernardo Rivas}
\address{Department of Mathematics, Rutgers University, Piscataway, NJ, 08854, USA}
\curraddr{}
\email{bad162@math.rutgers.edu}
\thanks{}

\author{Ewerton Vieira}
\address{Department of Mathematics, Rutgers University, Piscataway, NJ, 08854, USA}
\curraddr{}
\email{er691@math.rutgers.edu}
\thanks{}

\author{\vspace{0.6cm} With an appendix by Daniel Gameiro}
\address{Department of Electrical and Computer Engineering, Rutgers University, Piscataway, NJ, 08854, USA}
\curraddr{}
\email{dtg62@scarletmail.rutgers.edu}
\thanks{}


\begin{abstract}
We introduce a combinatorial topological framework for characterizing the global dynamics of ordinary differential equations (ODEs). The approach is motivated by the study of gene regulatory networks, which are often modeled by ODEs that are not explicitly derived from first principles. 

The proposed method involves constructing a combinatorial model from a set of parameters and then embedding the model into a continuous setting in such a way that the algebraic topological invariants are preserved. In this manuscript, we build upon the software Dynamic Signatures Generated by Regulatory Networks (DSGRN), a software package that is used to explore the dynamics generated by a regulatory network. By extending its functionalities, we deduce the global dynamical information of the ODE and extract information regarding equilibria, periodic orbits, connecting orbits and bifurcations. 

We validate our results through algebraic topological tools and analytical bounds, and the effectiveness of this framework is demonstrated through several examples and possible future directions. 
\end{abstract}

\maketitle

\tableofcontents

\mainmatter
\counterwithout{figure}{chapter}
\counterwithout{table}{chapter}

\part{Overview}
\label{part:I}

\chapter{Introduction}
\label{sec:intro}
    
This manuscript presents \emph{initial} steps towards the following two seemingly disparate goals:
\begin{description}
    \item[Goal 1] 
    Given  an ordinary differential equation (ODE) defined by a specific family of vector fields,
    \begin{equation}
    \label{eq:generalODE}
        \dot{x} = f(x,\lambda),\quad x\in \R^N, \lambda\in \R^M,
    \end{equation}
    \emph{provide an algorithm that outputs rigorous characterizations of the global dynamics for explicit open sets of parameters.
    } 
    \item[Goal 2] \emph{Consider a regulatory network, e.g., a diagram such as that shown in Figure~\ref{fig:RN}, the dynamics of which is modeled by a set of  ordinary differential equation.
    Provide a transparent means of characterizing the global dynamics over all of parameter space.}
\end{description}
A feature common to both goals is that they are impossible to achieve using traditional interpretations of what is meant by solving a differential equation.
At the risk of oversimplification, a standard numerical approach begins with an explicit nonlinearity and requires computing all trajectories over all specified parameter values, i.e., a continuum of computations.
The qualitative perspective of dynamical systems with its focus on invariant sets and conjugacy classes is also, in general, not computable \cite{foreman:rudolph:weiss}.
From the perspective of systems biology, the dynamics associated with the regulatory network of Figure~\ref{fig:RN} can be viewed as describing concentrations of mRNA or proteins associated with the nodes of the network.
However, the regulatory network does not provide an explicit ODE model derived from first principles, and thus, neither numerical nor qualitative techniques are immediately applicable.

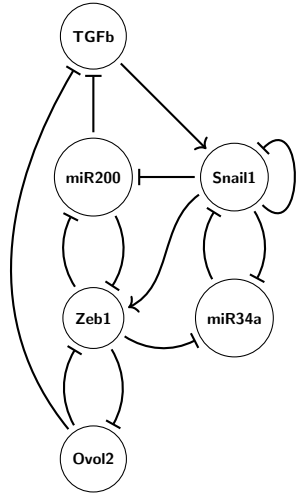
\begin{figure*}[!htb]
\centering
\begin{tikzpicture}
[main node/.style={circle,fill=white!20,draw,font=\sffamily\tiny\bfseries},scale=2.5]
\node[main node] (1) at (0,0) {Ovol2};
\node[main node] (2) at (0.75,0.75) {miR34a};
\node[main node] (3) at (0,0.75) {Zeb1};
\node[main node] (4) at (0,1.5) {miR200};
\node[main node] (5) at (0.75,1.5) {Snail1};
\node[main node] (6) at (0,2.25) {TGFb};
\path[thick]
(1) edge[-|, shorten <= 2pt, shorten >= 2pt, bend left] (3)
(1) edge[-|, shorten <= 2pt, shorten >= 2pt, bend left, out=40] (6)
(2) edge[-|, shorten <= 2pt, shorten >= 2pt, bend left] (5)
(3) edge[-|, shorten <= 2pt, shorten >= 2pt, bend left] (1)
(3) edge[-|, shorten <= 2pt, shorten >= 2pt, bend right] (2)
(3) edge[-|, shorten <= 2pt, shorten >= 2pt, bend left] (4)
(4) edge[-|, shorten <= 2pt, shorten >= 2pt, bend left] (3)
(4) edge[-|, shorten <= 2pt, shorten >= 2pt] (6)
(5) edge[-|, shorten <= 2pt, shorten >= 2pt, bend left] (2)
(5) edge[->, shorten <= 2pt, shorten >= 2pt, bend left, out=-20] (3)
(5) edge[-|, shorten <= 2pt, shorten >= 2pt] (4)
(5) edge[ -|, loop, shorten <= 2pt, shorten >= 2pt, distance=10pt, thick, out=-45, in=45] (5) 
(6) edge[->, shorten <= 2pt, shorten >= 2pt] (5);
\end{tikzpicture}
\caption{Regulatory network for reversible epithelial-to-mesenchymal transition as proposed by \cite[Figure 2]{hong}. Pointed edges represent activation while dull edges represent repression.}
\label{fig:RN}
\end{figure*}

To achieve  {\bf Goals 1} and {\bf 2} we propose an alternative paradigm for characterizing dynamics based on order theory and algebraic topology.
Figure~\ref{fig:2D_ramp_ODE_intro_dynamics} (associated with {\bf Goal 2}) and Figure~\ref{fig:phase_portraits_intro1} (associated with {\bf Goal 1})  are meant to provide with minimal technicalities a  high level  intuition into the type of results our approach provides and to suggest why achievement of {\bf Goal 1} and {\bf Goal 2} are closely related.

The information carried by Figure~\ref{fig:2D_ramp_ODE_intro_dynamics} takes the form of combinatorics and homological algebra.
In particular it provides no topological, geometric or analytic information. 
This is to be expected since it is generated from the regulatory network shown in Figure~\ref{fig:network_parameter_regions_intro}(A), which is essentially a combinatorial object.
Part~\ref{part:II} of this monograph details the mathematical theory and algorithms that allow us to pass from the input data of Figure~\ref{fig:network_parameter_regions_intro}(A) to the combinatorial/homological information of Figure~\ref{fig:2D_ramp_ODE_intro_dynamics}.

\begin{figure}[!htpb]
\begin{subfigure}{0.32\textwidth}
\centering
\includegraphics[width=\linewidth]{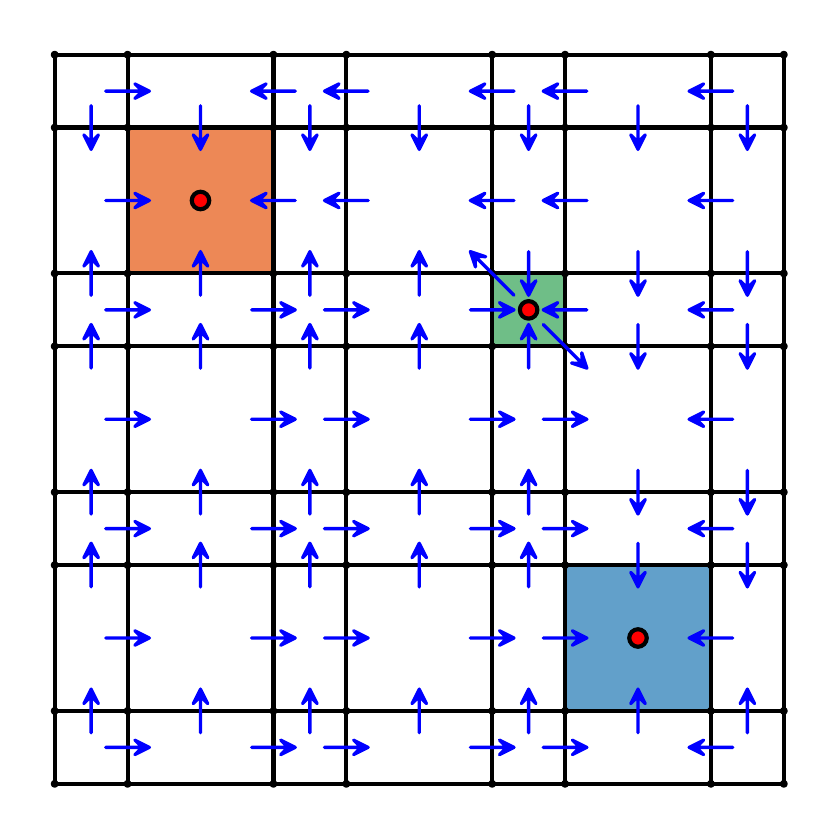}
\caption{Combinatorial model for parameter node $974$.}
\label{fig:2D_ramp_ODE_intro_dynamics_a}
\end{subfigure}
\begin{subfigure}{0.32\textwidth}
\centering
\includegraphics[width=\linewidth]{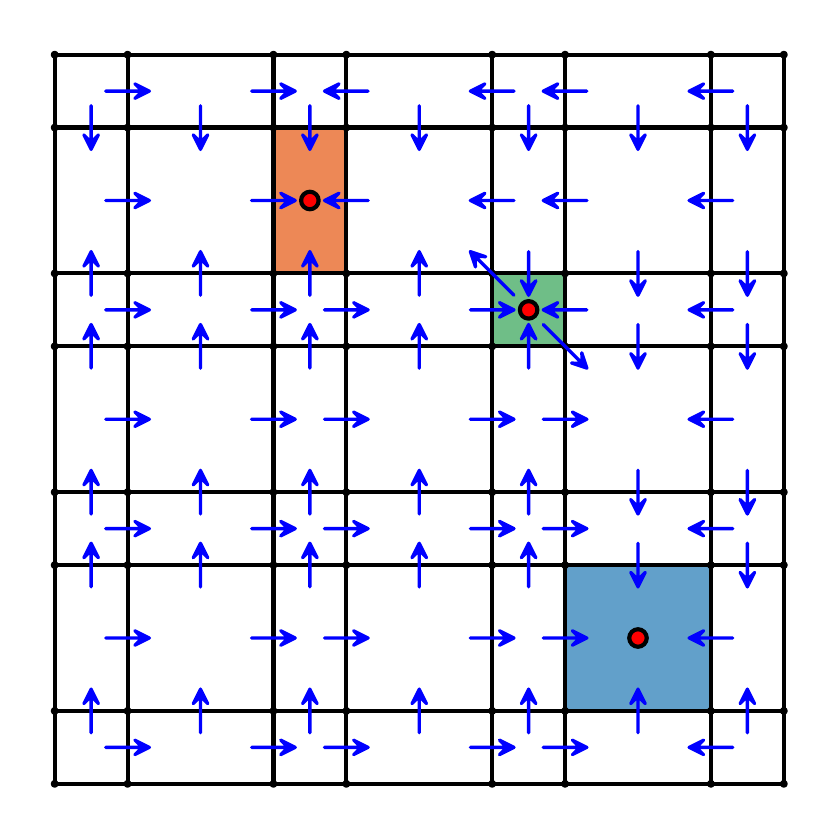}
\caption{Combinatorial model for node $975$.}
\label{fig:2D_ramp_ODE_intro_dynamics_b}
\end{subfigure}
\begin{subfigure}{0.32\textwidth}
\centering
\includegraphics[width=\linewidth]{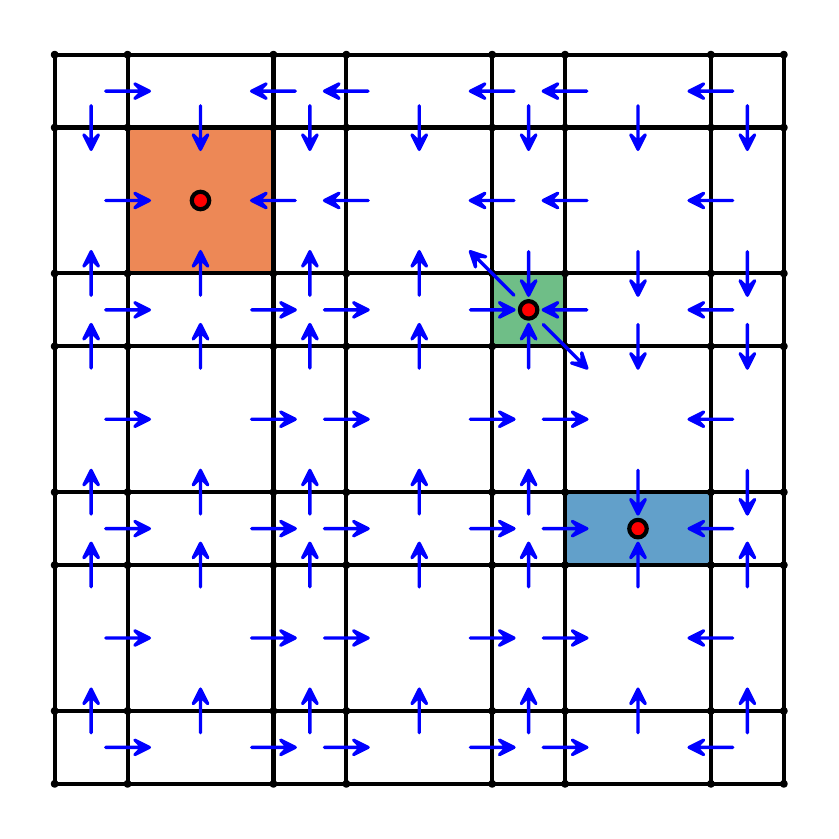}
\caption{Combinatorial model for node $1054$.}
\label{fig:2D_ramp_ODE_intro_dynamics_c}
\end{subfigure}

\begin{subfigure}{0.32\textwidth}
\centering
\includegraphics[width=\linewidth]{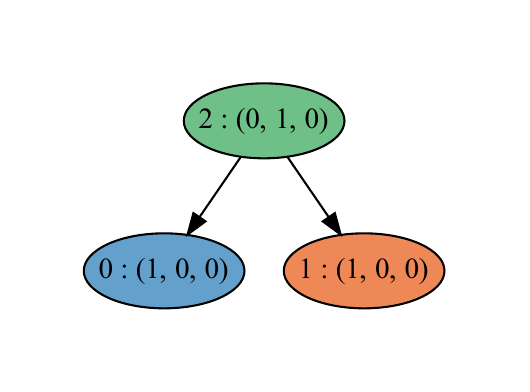}
\caption{Morse Graph for parameter node $974$.}
\label{fig:2D_ramp_ODE_intro_dynamics_d}
\end{subfigure}
\begin{subfigure}{0.32\textwidth}
\centering
\includegraphics[width=\linewidth]{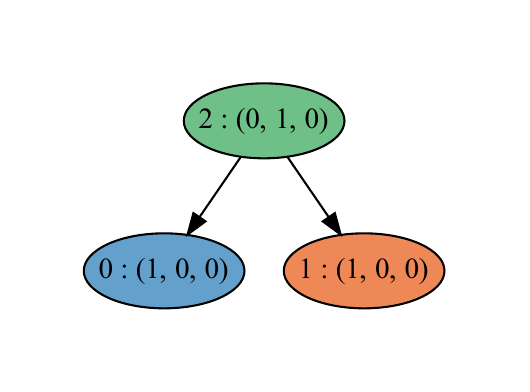}
\caption{Morse Graph for parameter node $975$.}
\label{fig:2D_ramp_ODE_intro_dynamics_e}
\end{subfigure}
\begin{subfigure}{0.32\textwidth}
\centering
\includegraphics[width=\linewidth]{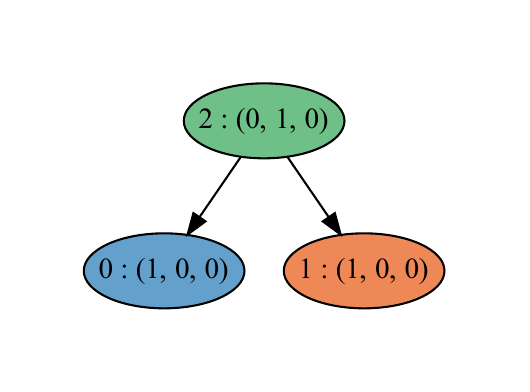}
\caption{Morse Graph for parameter node $1054$.}
\label{fig:2D_ramp_ODE_intro_dynamics_f}
\end{subfigure}
\caption{Combinatorial models  for the indicated parameter nodes (A)-(C), and the corresponding Morse graphs (D)-(F).
The combinatorial models are generated by the regulatory network shown in Figure~\ref{fig:network_parameter_regions_intro}(A).
This network generates exactly 1,600 combinatorial models of which three are shown.
}
\label{fig:2D_ramp_ODE_intro_dynamics}
\end{figure}

Figure~\ref{fig:phase_portraits_intro1} shows ``corresponding'' phase portraits that are generated by the following system of two-dimensional ODEs 
\begin{equation}
\label{eq:2DexampleODE}
\begin{aligned}
\dot{x}_1 & = - \gamma_1 x_1 + r_{1,1}(x_1) r_{1,2}(x_2) \\
\dot{x}_2 & = - \gamma_2 x_2 + r_{2,1}(x_1) r_{2,2}(x_2),
\end{aligned}
\end{equation}
where
\begin{equation}
\label{eq:neg_pos_ramp}
r_{i, j}(x) =
\begin{cases}
\nu_{i, j, 1}, & \text{if}~ x < \theta_{i, j} - h_{i, j} \\
L_{i, j} (x), & \text{if}~ \theta_{i, j} - h_{i, j} \leq x \leq \theta_{i, j} + h_{i, j} \\
\nu_{i, j, 2}, & \text{if}~ x > \theta_{i, j} + h_{i, j}
\end{cases}
\end{equation}
and 
$L_{i, j} (x) = \dfrac{\nu_{i, j, 2} - \nu_{i, j, 1}}{2 h_{i, j}}(x - \theta_{i, j}) + \dfrac{\nu_{i, j, 1} + \nu_{i, j, 2}}{2}$, which is just a linear interpolation between $\nu_{i, j, 1}$ and $\nu_{i, j, 2}$,
at the sets of parameter values given in Table~\ref{tab:parameters_intro1}.
System \eqref{eq:2DexampleODE} is a special case of a \emph{ramp system} (see Chapter~\ref{sec:ramp}).
Part~\ref{part:III} of this monograph explains how ramp systems are associated with regulatory networks, and in particular, why \eqref{eq:2DexampleODE} is associated with the regulatory network of Figure~\ref{fig:network_parameter_regions_intro}(A).
Furthermore, it provides an explicit mapping from the cubical cell complexes of Figure~\ref{fig:2D_ramp_ODE_intro_dynamics} (A)-(C) to the phase portraits of Figure~\ref{fig:phase_portraits_intro1}, and it provides analytic proofs that the combinatorial/homological information of Figure~\ref{fig:2D_ramp_ODE_intro_dynamics} provide valid information concerning the dynamics of \eqref{eq:2DexampleODE}.

\begin{table}[!htpb]
\centering
\renewcommand{\arraystretch}{1.2}
\setlength{\tabcolsep}{12pt}
\begin{tabular}{@{}lll@{}}
\toprule
Parameter set $1$ & Parameter set $2$ & Parameter set $3$ \\
\midrule
$
\begin{aligned}
& \nu_{1, 1, 1} = 3.7, \ \nu_{1, 1, 2} = 1.4 \\
& \nu_{1, 2, 1} = 10.7, \ \nu_{1, 2, 2} = 0.1 \\
& \nu_{2, 1, 1} = 9.2, \ \nu_{2, 1, 2} = 0.2 \\
& \nu_{2, 2, 1} = 6.2, \ \nu_{2, 2, 2} = 1.4 \\
& \theta_{1, 1} = 6.4, \ \theta_{1, 2} = 5.6 \\
& \theta_{2, 1} = 11.1, \ \theta_{2, 2} = 1.8 \\
& h_{1, 1} = 0.3, \ h_{1, 2} = 0.35 \\
& h_{2, 1} = 0.6, \ h_{2, 2} = 0.3 \\
& \gamma_1 = 1, \ \gamma_2 = 1
\end{aligned}
$
&
$
\begin{aligned}
& \nu_{1, 1, 1} = 6.5, \ \nu_{1, 1, 2} = 4.2 \\
& \nu_{1, 2, 1} = 7.6, \ \nu_{1, 2, 2} = 1.7 \\
& \nu_{2, 1, 1} = 4.7, \ \nu_{2, 1, 2} = 0.6 \\
& \nu_{2, 2, 1} = 5.2, \ \nu_{2, 2, 2} = 2.6 \\
& \theta_{1, 1} = 10.3, \ \theta_{1, 2} = 8.5 \\
& \theta_{2, 1} = 28.2, \ \theta_{2, 2} = 5.7 \\
& h_{1, 1} = 0.5, \ h_{1, 2} = 0.5 \\
& h_{2, 1} = 0.5, \ h_{2, 2} = 0.5 \\
& \gamma_1 = 1, \ \gamma_2 = 1
\end{aligned}
$
&
$
\begin{aligned}
& \nu_{1, 1, 1} = 7.2, \ \nu_{1, 1, 2} = 5 \\
& \nu_{1, 2, 1} = 3, \ \nu_{1, 2, 2} = 0.4 \\
& \nu_{2, 1, 1} = 5.5, \ \nu_{2, 1, 2} = 1 \\
& \nu_{2, 2, 1} = 9, \ \nu_{2, 2, 2} = 5.5 \\
& \theta_{1, 1} = 3.6, \ \theta_{1, 2} = 13 \\
& \theta_{2, 1} = 13.5, \ \theta_{2, 2} = 8.1 \\
& h_{1, 1} = 0.5, \ h_{1, 2} = 0.6 \\
& h_{2, 1} = 0.5, \ h_{2, 2} = 0.5 \\
& \gamma_1 = 1, \ \gamma_2 = 1
\end{aligned}
$
\\
\bottomrule
\end{tabular}
\caption{Three sets of parameter values for system \eqref{eq:2DexampleODE}.}
\label{tab:parameters_intro1}
\end{table}

\begin{figure}[!htpb]
\centering
\includegraphics[height=0.31\linewidth]{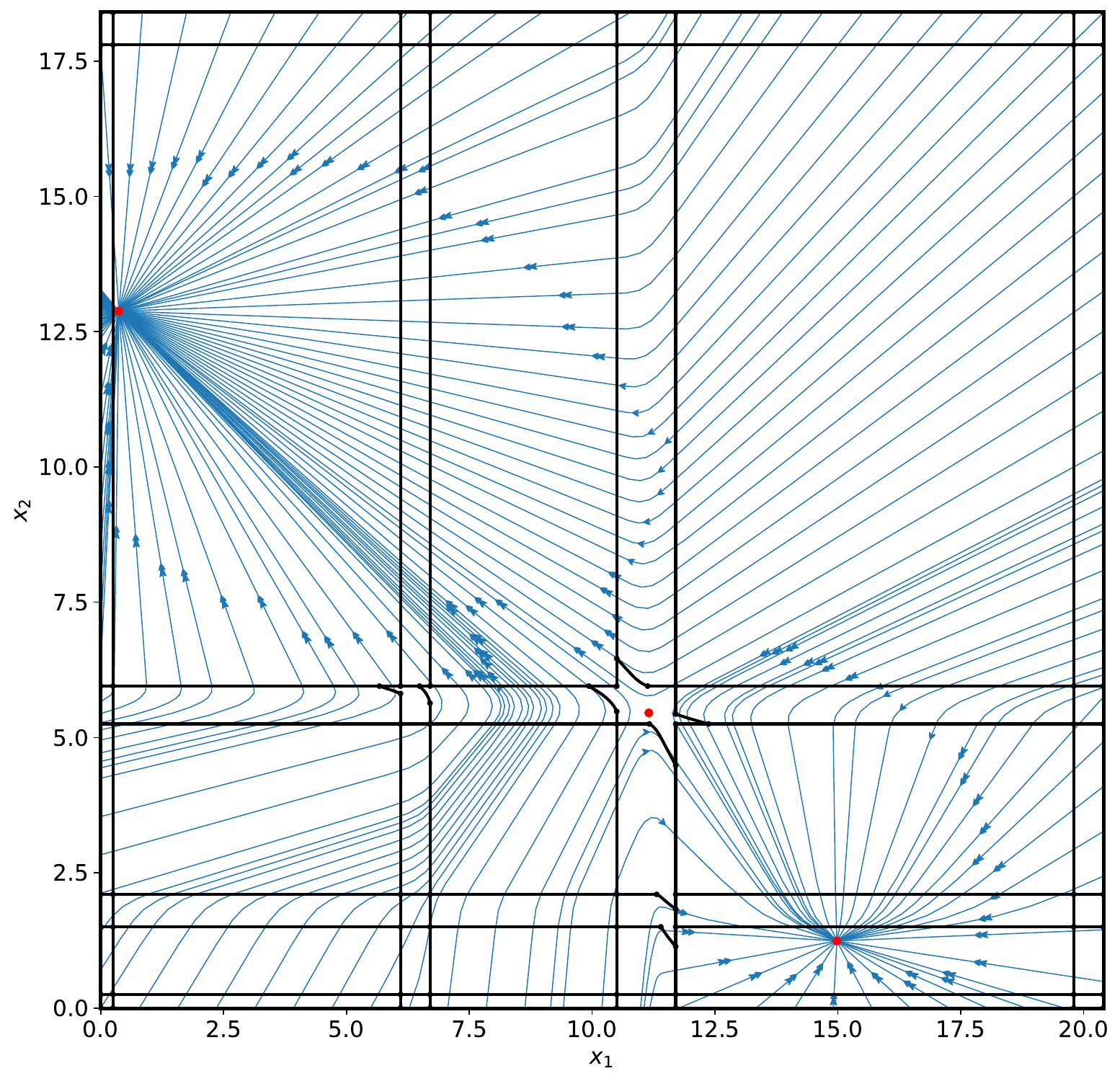}
\includegraphics[height=0.31\linewidth]{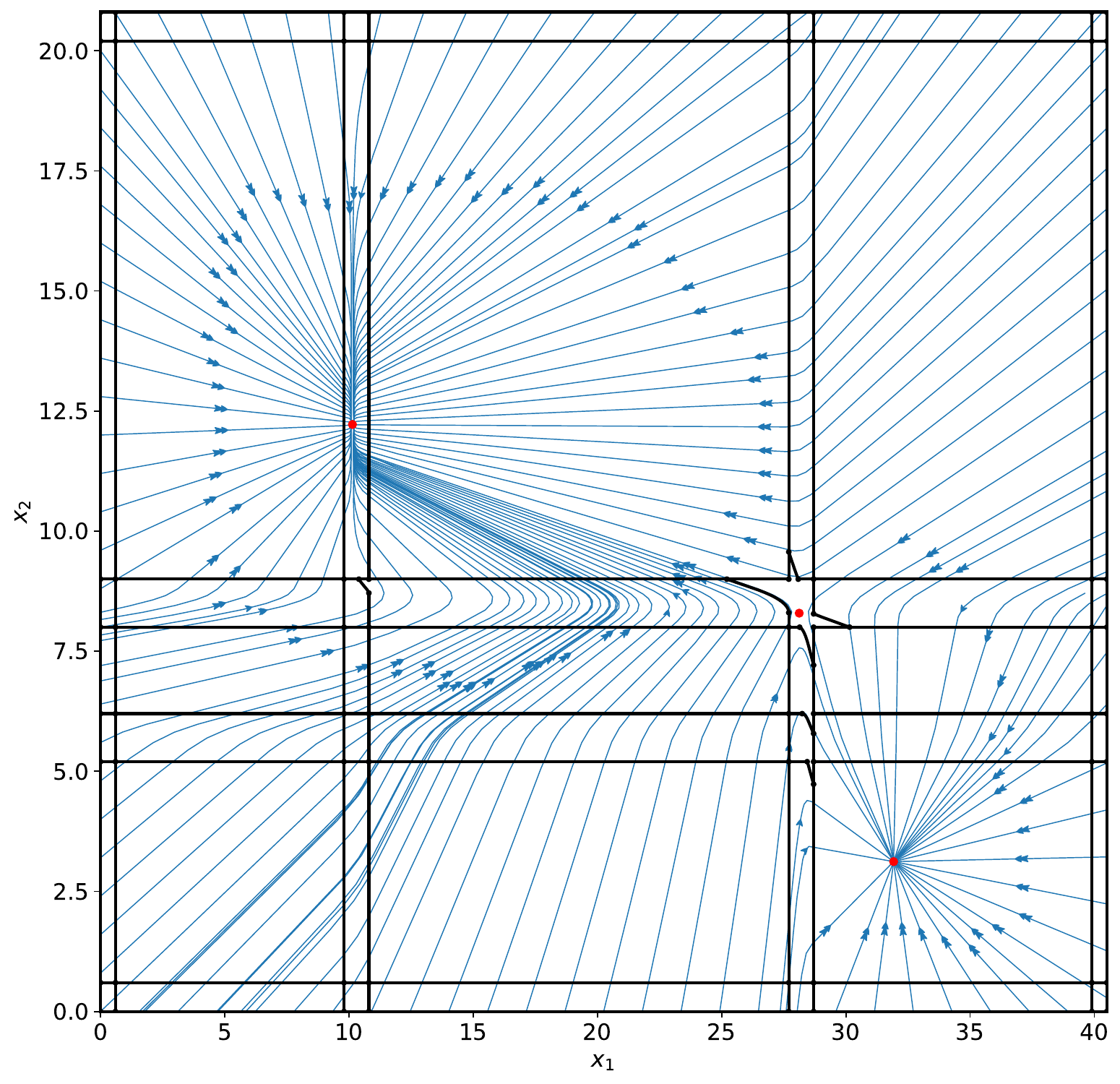}
\includegraphics[height=0.31\linewidth]{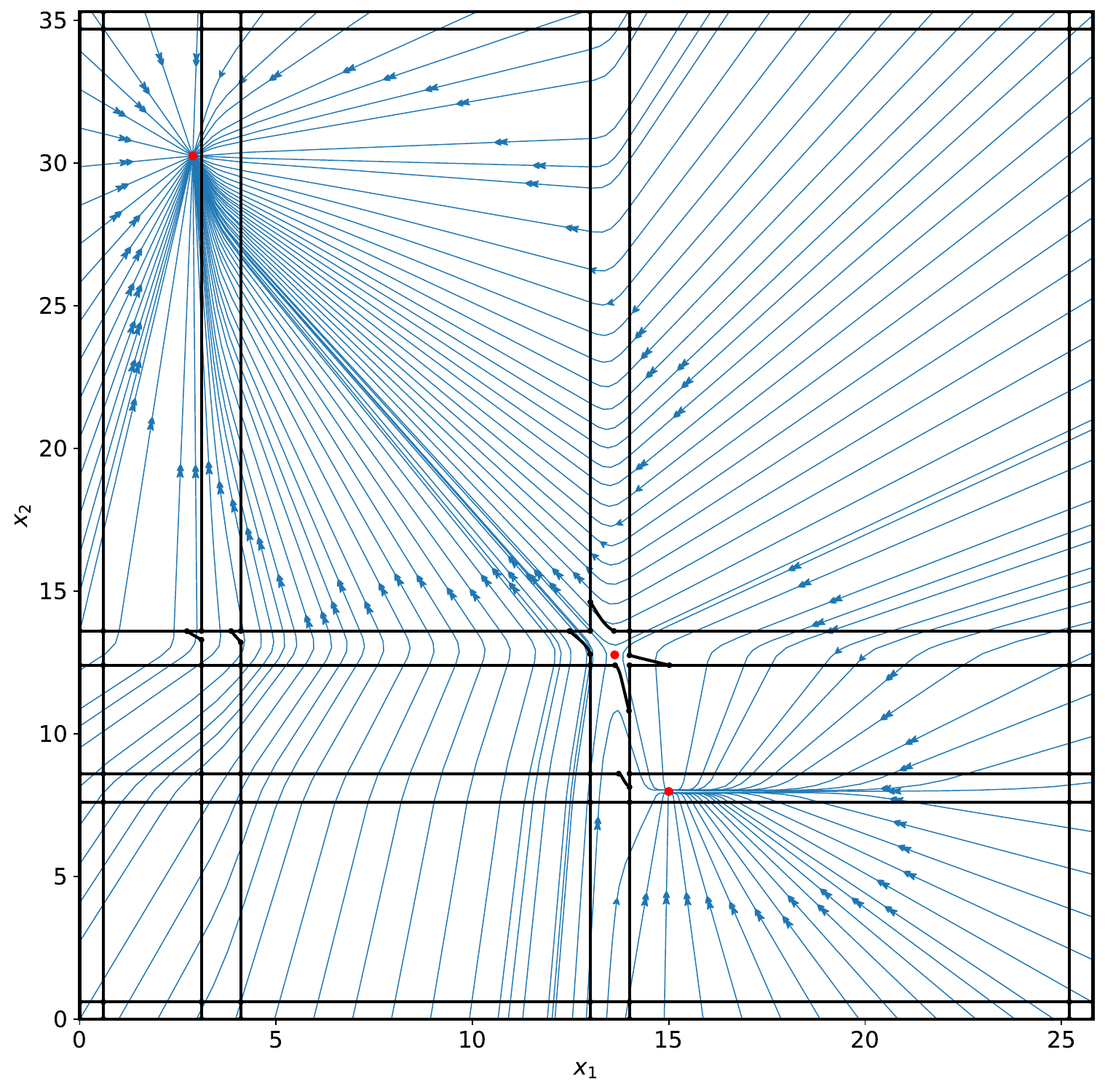}
\caption{Phase portraits and geometric realizations for system for system \eqref{eq:2DexampleODE} at parameter values given in Table~\ref{tab:parameters_intro1}.}
\label{fig:phase_portraits_intro1}
\end{figure}

\begin{figure}[!htpb]
\begin{subfigure}{0.25\textwidth}
\centering
\begin{tikzpicture}
[main node/.style={circle,fill=white!20,draw},scale=2.5]
\node[main node] (1) at (0,0) {1};
\node[main node] (2) at (0,0.75) {2};
\path[thick]
(1) edge[-|, shorten <= 2pt, shorten >= 2pt, bend left] (2)
(1) edge[-|, loop left, distance=10pt, shorten <= 2pt, shorten >= 2pt, thick, in=300, out=240] (1)
(2) edge[-|, shorten <= 2pt, shorten >= 2pt, bend left] (1)
(2) edge[-|, loop right, distance=10pt, shorten <= 2pt, shorten >= 2pt, thick, in=120, out=60] (2);
\end{tikzpicture}
\caption{DSGRN Network.}
\label{fig:network_2_intro}
\end{subfigure}
\begin{subfigure}{0.7\textwidth}
\centering
\renewcommand{\arraystretch}{1.2}
\begin{tabular}{@{}ll@{}}
\toprule
Parameter node $974$ &
$
\begin{aligned}
& x_1 : p_0, p_1 < \gamma_1 \theta_{1,1} < \gamma_1 \theta_{2,1} < p_2, p_3 \\
& x_2 : p_0, p_2 < \gamma_2 \theta_{2,2} < \gamma_2 \theta_{2,1} < p_1, p_3 \\
\end{aligned}
$
\\
\midrule
Parameter node $975$ &
$
\begin{aligned}
& x_1 : p_0 < \gamma_1 \theta_{1,1} < p_1 < \gamma_1 \theta_{2,1} < p_2, p_3 \\
& x_2 : p_0, p_2 < \gamma_2 \theta_{2,2} < \gamma_2 \theta_{2,1} < p_1, p_3 \\
\end{aligned}
$
\\
\midrule
Parameter node $1054$ &
$
\begin{aligned}
& x_1 : p_0, p_1 < \gamma_1 \theta_{1,1} < \gamma_1 \theta_{2,1} < p_2, p_3 \\
& x_2 : p_0 < \gamma_2 \theta_{2,2} < p_2 < \gamma_2 \theta_{2,1} < p_1, p_3 \\
\end{aligned}
$
\\
\bottomrule
\end{tabular}
\caption{Inequalities defining parameter regions $974$, $975$, and $1054$.}
\label{fig:parameter_regions_intro}
\end{subfigure}
\caption{DSGRN Network and inequalities defining parameter regions in the parameter graph of the network, where the \emph{input polynomials} $p_j$ are given by $p_0 = \ell_{1,1} \ell_{1,2}$, $p_1 = u_{1,1} \ell_{1,2}$, $p_2 = \ell_{1,1} u_{1,2}$, and $p_3 = u_{1,1} u_{1,2}$ for node $x_1$ and $p_0 = \ell_{2,1} \ell_{2,2}$, $p_1 = u_{2,1} \ell_{2,2}$, $p_2 = \ell_{2,1} u_{2,2}$, and $p_3 = u_{2,1} u_{2,2}$ for node $x_2$.}
\label{fig:network_parameter_regions_intro}
\end{figure}

We hasten to add that most of our results are dimension independent and we consider explicit higher dimensional examples in Chapter~\ref{sec:examples}.

Because it may appear to be
counter intuitive, we re-iterate that the results for system \eqref{eq:2DexampleODE} shown in Figure~\ref{fig:phase_portraits_intro1} are obtained from our analysis of the dynamics of the regulatory network shown in Figure~\ref{fig:network_parameter_regions_intro} (A).
This is in line with our perspective that trying to directly analyze the global dynamics of an ODE is too difficult.
Instead we propose the following pipeline.

\begin{description}
    \item[Step 1] Consider a simpler problem that can be represented using a combinatorial model.
    \item[Step 2] Use the combinatorial model to 
    \begin{description}
        \item[a] identify the global structure of the dynamics, and
        \item[b] compute algebraic topological invariants that imply the existence of dynamics.
    \end{description}
    \item[Step 3] Using analysis embed the combinatorial model into a continuous setting in such a way that the algebraic topological invariants are preserved.
    \item[Step 4] 
    Check that the algebraic topological invariants associated with the embedding constructed in step 3 are still valid for the original ODE of interest.
\end{description}

With regard to {\bf Step 1}, observe that the network of Figure~\ref{fig:network_parameter_regions_intro} (A) provides a simple caricature of the interactions between variables in the ODE \eqref{eq:2DexampleODE}.
In particular, the directed edges $1 \dashv 1 $ and $2 \dashv 1$ are associated with the fact that $\dot{x}_1$ is monotone decreasing as a function of $x_1$  and $x_2$, respectively. Similarly, $1 \dashv 2$ and $2 \dashv 2$ indicates that $\dot{x}_2$ is monotone decreasing as a function of $x_1$ and $x_2$, respectively.

The model ODE \eqref{eq:2DexampleODE} is clearly parameterized, thus the associated combinatorial model should also be parameter dependent.
The parameterization of the combinatorial model is based on the assumption that each node $n$ (associated to a real variable $x_n$) has a decay rate $\gamma_n >0$, that there is a minimal $\ell_{n,m}>0$ and maximal $u_{n,m}$ effect of $x_m$ on the growth rate of $x_n$, and a threshold $\theta_{n,m}$ at which the nonlinear interactions of these parameters potentially change the sign of the rate of change of $x_n$.\footnote{We allow for more complex expressions of the growth rate than just a minimum and maximum, which is why the expression for $r_{n,m}$ in \eqref{eq:neg_pos_ramp} involves parameters labeled $\nu_{n,m,k}$ as opposed to $\ell_{n,m}$ and $u_{n,m}$.}

The software, Dynamic Signatures Generated by Regulatory Networks (DSGRN) \cite{DSGRN}, takes a regulatory network and produces an explicit subdivision of parameter space where each element of the subdivision is given by an explicit semi-algebraic set \cite{cummins:gedeon:harker:mischaikow:mok,kepley:mischaikow:zhang}.
For the simple network of Figure~\ref{fig:network_parameter_regions_intro} (A) DSGRN produces a subdivision of the parameter space $(0,\infty)^{14}$ with exactly $1600$ regions.
This information is encoded in the \emph{parameter graph}.
Each node of the parameter graph represents a region.
An edge between nodes implies that the corresponding regions share a codimension 1 hypersurface.
The explicit representations of the regions associated with the nodes of the parameter graph used to produce the results of Figure~\ref{fig:2D_ramp_ODE_intro_dynamics} are given in Figure~\ref{fig:network_parameter_regions_intro} (B).

\begin{rem}
The parameter $\nu$ for a ramp system derived from a DSGRN network is defined in terms of the DSGRN parameters $\ell$ and $u$ as follows (see Section~\ref{sec:DSGRN}): For an edge $m \to n$ of the regulatory network we define $\nu_{n, m, 1} = \ell_{n, m}$ and $\nu_{n, m, 2} = u_{n, m}$, and for an edge $m \dashv n$ we define $\nu_{n, m, 1} = u_{n, m}$ and $\nu_{n, m, 2} = \ell_{n, m}$.
DSGRN represents inequalities of the type given in Figure~\ref{fig:network_parameter_regions_intro}(B), such as $p_0 < \gamma_1 \theta_{1,1} < p_1 < \gamma_1 \theta_{2,1} < p_2, p_3$ and $p_0, p_2 < \gamma_2 \theta_{2,2} < \gamma_2 \theta_{2,1} < p_1, p_3$ by $(p_0, \gamma_1 \theta_{1,1}, p_1, \gamma_1 \theta_{2,1}, p_2, p_3)$ and $(p_0, p_2, \gamma_2 \theta_{2,2}, \gamma_2 \theta_{2,1}, p_1, p_3)$, respectively.
\end{rem}

In this paper we define four families of combinatorial models ($\cF_0$, $\cF_1$, $\cF_2$, and $\cF_3$) that take the form of combinatorial multivalued maps or equivalently directed graphs.\footnote{In the context of Boolean models these multivalued maps are often call state transition graphs.}
The states  or vertices are the cells of an abstract cubical complex $\cX$ that is determined by the number of thresholds $\theta_{n, m}$.
Of fundamental importance is the fact that a unique multivalued map is assigned to each node in the parameter graph, thus given any regulatory network there are only finitely many associated combinatorial models. 
Figures~\ref{fig:2D_ramp_ODE_intro_dynamics} (A), (B), and (C) provide \emph{pictorial} representations of the multivalued maps for three nodes of the parameter graph of the network in in Figure~\ref{fig:network_parameter_regions_intro} (A).
We emphasize pictorial, since at this stage in the pipeline we are still working with purely combinatorial structures.

We now turn to {\bf Step 2}.
Our immediate goal is a compact characterization of the dynamics associated with the multivalued map $\cF$.
The characterization of global structure of the dynamics takes the form of an acyclic directed graph or, equivalently, a partially ordered set (poset) called a \emph{Morse graph}.
The nodes of the Morse graph capture the existence of recurrent components (nontrivial strongly connected components) in the multivalued map and the directed edges indicate reachability between recurrent components.
Figures~\ref{fig:2D_ramp_ODE_intro_dynamics} (D), (E), and (F) show the Morse graphs for the multivalued maps of Figures~\ref{fig:2D_ramp_ODE_intro_dynamics} (A), (B), and (C), respectively.

Our long term goal is to use this data to deduce information about the dynamics of an ODE.
Note that in the language of nonlinear dynamics the nodes are suggestive of recurrent dynamics and the edges indicate gradient-like dynamics.
As presented the Morse graph is a purely combinatorial object, and thus, in general, there is no certainty that this data translates into meaningful information about the dynamics.
To transfer the combinatorial data to statements about continuous dynamics we turn to algebraic topology.

Our fundamental tool is a homological invariant called the \emph{Conley index} \cite{conley:cbms}.
As described later we use the Conley index to identify nontrivial dynamics.
For the moment, it is sufficient to remark that the Conley index takes the form of a sequence of homology groups.\footnote{Throughout this paper we use field coefficients and therefore the homology groups are vector spaces.}
To access the Conley indices for a fixed combinatorial model $\cF$ on a cubial cell complex $\cX$ we compute a new chain complex, called a \emph{Conley complex}. 
While the Morse graph of $\cF$ and a Conley complex of $\cF$ carry different information, they are closely related.
In particular, each node $p$ of the Morse graph has an associated Conley index $CH_*(p)$.
The chains of the Conley complex consist of the direct sum of the Conley indices of the nodes of the Morse graph.
The boundary of the Conley complex allows one to compute additional Conley indices.
With regard to our example, this information is embedded in Figures~\ref{fig:2D_ramp_ODE_intro_dynamics} (D), (E), and (F), where the calculations were done using $\Z_2$ coefficients. The triple associated with each Morse node $p$ are the first three Betti numbers of the associated homology groups, i.e., $(\beta_0, \beta_1, \beta_2)$. For these examples (this is not true in general) each edge in the Morse graph corresponds to a nontrivial entry in the boundary operator of the Conley complex.

At this point we claim success with respect to {\bf Goal 2}.
Given a regulatory network, DSGRN produces a finite parameter graph and  
associates to each node of the paramater graph a multivalued map.
From the multivalued map we can compute the Morse graph and an associated Conley complex.
Therefore, for each regulatory network, application of {\bf Step 1} and {\bf Step 2} produces a finite collection of characterizations of dynamics. Furthermore, with the exclusion of a finite set of hypersurfaces, the characterizations cover all parameter values and the computations are very efficient. To compute the dynamics for all $1600$ nodes of the parameter graph of the network in Figure~\ref{fig:network_parameter_regions_intro} takes about $8$ seconds on a laptop using a single core.

Admittedly the characterization of dynamics is given in terms of a poset (the Morse graph) and homological expressions.
However, as is made clear in the next step there are well defined classical dynamical interpretations of this order structure and homology in the context of ODEs.
This is an optimal result in the sense that no specific ODE has been selected at this point, and yet we are providing information that can be given classical interpretations about the dynamics of ODEs.

We now consider {\bf Step 3}.
While {\bf Step 1} and {\bf Step 2} are purely combinatorial/algebraic, {\bf Step 3} is purely analytic.
We are mapping the abstract cubical complex on which $\cF_i$ is defined into $\R^N$ ($N$ is the number of variables) in such a way that certain transversality conditions with respect to the vector field of an ODE are satisfied.
Success in {\bf Step 3} implies that the user of our techniques will achieve {\bf Goal 1} by performing the computations associated with {\bf Step 1} and {\bf Step 2}; \emph{there is no need to do any numerical computations}.
Ideally, a user would declare the family of ODEs of interest, perform the combinatorial/algebraic computations, and understand the global dynamics of the ODE.

We are far from achieving this ideal.
However, as a first step in this direction, we introduce the family of ramp systems (see Section~\ref{sec:ramp}),
\begin{equation}
\label{eq:rampODEintro}
\dot{x}_n  = -\gamma_n x_n + E_n(x;\nu,\theta,h), \quad n = 1, \ldots, N
\end{equation}
where $\gamma$, $\nu$, and $\theta$ are the parameters introduced in the first step of the pipeline (the parameters $\ell$ and $u$ are replaced by $\nu$ since we need more freedom to express how one node impacts the rate of production of another node) and $h$ is an additional parameter.

As is shown in Part~\ref{part:III} of this manuscript for these ODEs we have succeeded in {\bf Step 3}; given appropriate parameter values $\gamma$, $\nu$, $\theta$, and $h$ we provide a proof that the combinatorial/algebraic information of {\bf Step 1} and {\bf Step 2} characterizes the dynamics of the ramp system.
We admittedly designed the nonlinearities of the ramp systems to simplify the necessary analysis.
However, we encourage the reader, as they are studying the details of Part~\ref{part:III}, to keep in mind that local monotonicity and identification of equilibria are the key elements of the proofs.
We expand on this comment in Chapter~\ref{sec:futurework}.

To carry out the analysis we choose an admissible fixed parameter $(\gamma, \nu, \theta, h)$ for the ramp system under the assumption that $(\gamma, \nu, \theta)$ belongs to a region $R(k)$ associated with a node $k$ of the parameter graph. From the original cell complex $\cX$ we construct an abstract cell complex, called the \emph{Janus complex}. We embed the Janus complex into phase space to obtain a regular CW-decomposition of phase space such that along boundaries of appropriate cells the vector field for the ramp system is transverse. Up to homotopy the constructed regular CW-decomposition is valid for all parameter values in $R(k)$. The constraint for which we produce explicit bounds is on the set of $h$ for which the transversality is guaranteed. These bounds depend explicitly on $(\gamma, \nu, \theta)$. 

Using the transversality conditions of the CW-decomposition we obtain three fundamental consequences.
\begin{description}
\item[R1] The  Morse graph for the combinatorial model is a Morse decomposition for the ramp system  \cite{kalies:mischaikow:vandervorst:18}.
\item[R2] For each node $p$ in the Morse graph, the computed homology group $CH_*(p)$ is the homological Conley index \cite{conley:cbms, salamon, mischaikow:mrozek:02} for the associated Morse set for the ramp system.
\item[R3] A boundary operator for the Conley complex of the combinatorial model is a connection matrix \cite{franzosa:89} for the Morse decomposition of the ramp system.
\end{description}

Recall that the Conley index can be used to provide information about the existence and structure of invariant sets \cite{conley:cbms, mischaikow:mrozek:02}.
The most fundamental result is that if $CH_*(p)\neq 0$, then the associated invariant set is non-empty.
The connection matrix can be used to identify heteroclinic orbits \cite{franzosa:89,fiedler:mischaikow,hattori:mischaikow}, 
suggest the existence of global bifurcations \cite{mccord:mischaikow:92,kokubu:mischaikow:oka}, and more generally the global structure of invariant sets \cite{mccord:mischaikow:96,mischaikow:95,mccord, gedeon:mischaikow:95}.

\begin{rem}
\label{rem:vandenberg:lessard}
Significant advances are being made in the use of rigorous validation or computer assisted proofs for the analysis of differential equations \cite{vandenberg:lessard}.
These methods typically involve the validation of approximate solutions that have been identified using traditional numerical methods.
While the work presented here fits into this framework -- given a nonlinear equation the computer is used to prove the existence of particular solutions -- it is fundamentally different in that our method simultaneously identifies the solutions of interest.
\end{rem}

{\bf Step 4} of the pipeline is to transfer the information about the dynamics of the ramp system to the original ODE of interest. 
If the original ODE of interest is a ramp system, then this step is not necessary.
With an eye towards more general systems,  Morse decompositions and Conley indices are stable with respect to perturbation \cite{conley:cbms}.
Thus, any information about the ramp system is immediately applicable to any ODE sufficiently close in the $C^0$ topology to the ramp system.
While the analytic form of a ramp system is special, adopting a more qualitative perspective allows us to focus on the fact that ramp systems are built out of piecewise monotone functions, e.g., \eqref{eq:neg_pos_ramp}.
In particular,  if an ODE mimics the piecewise monotone structure of a ramp system sufficiently well, then our combinatorial/algebraic characterization of the dynamics of the ramp system should apply to the ODE.
However, achieving {\bf Goal 1} requires a stronger result; given an ODE of interest we want to choose a ramp system sufficiently close so that the transversality results mentioned in the discussion of {\bf Step 3} are applicable to the ODE.

We do not have rigorous a priori bounds for choosing such a ramp system.
However, experimentally it appears that tight bounds are not necessary \cite{kepley:mischaikow:queirolo} and thus a reasonable guess may, in many cases, be sufficient.
Consider the ODE system
\begin{equation}
\label{eq:Hill_2D_example_ODE}
\begin{aligned}
\dot{x}_1 & = - \gamma_1 x_1 + H_{1, 1}^{-}(x_1) H_{1, 2}^{-}(x_2) \\
\dot{x}_2 & = -\gamma_2 x_2 +  H_{2, 1}^{-}(x_1) H_{2, 2}^{-}(x_2),
\end{aligned}
\end{equation}
where
\begin{equation}
\label{eq:Hillrepressor}
H_{i, j}^{-}(x) = \nu_{i, j, 2} + (\nu_{i, j, 1} - \nu_{i, j, 2}) \frac{\theta_{i, j}^s}{x^s + \theta_{i, j}^s},
\end{equation}
with parameter values given by the parameter set $1$ in Table~\ref{tab:parameters_intro1} and $s = 10$ and the ramp system whose parameter value lies in the region given by the parameter node $974$ in Figure~\ref{fig:network_parameter_regions_intro}(B).

Our embedding of the Janus complex into phase space explicitly depends on particular perturbations of the ramp system.
We can use this information to numerically identify the above mentioned boundaries of appropriate cells and then numerically check that the vector field for the ODE of interest is transverse to these boundary regions.
In particular, piecewise linear approximation of the numerical images of the embeddings for the ramp systems chosen to study the ODE \eqref{eq:Hill_2D_example_ODE} at parameter values given in Table~\ref{tab:parameters_intro1} and $s = 10$ are shown in Figure~\ref{fig:phase_portrait_Hill_intro}. Numerically, these piecewise linear curves are transverse to the vector field of the ODE \eqref{eq:Hill_2D_example_ODE}, and therefore, the conclusions {\bf R1}, {\bf R2}, and {\bf R3} are valid for the ODE \eqref{eq:Hill_2D_example_ODE}.

\begin{figure}[!htpb]
\centering
\includegraphics[width=0.9\linewidth]{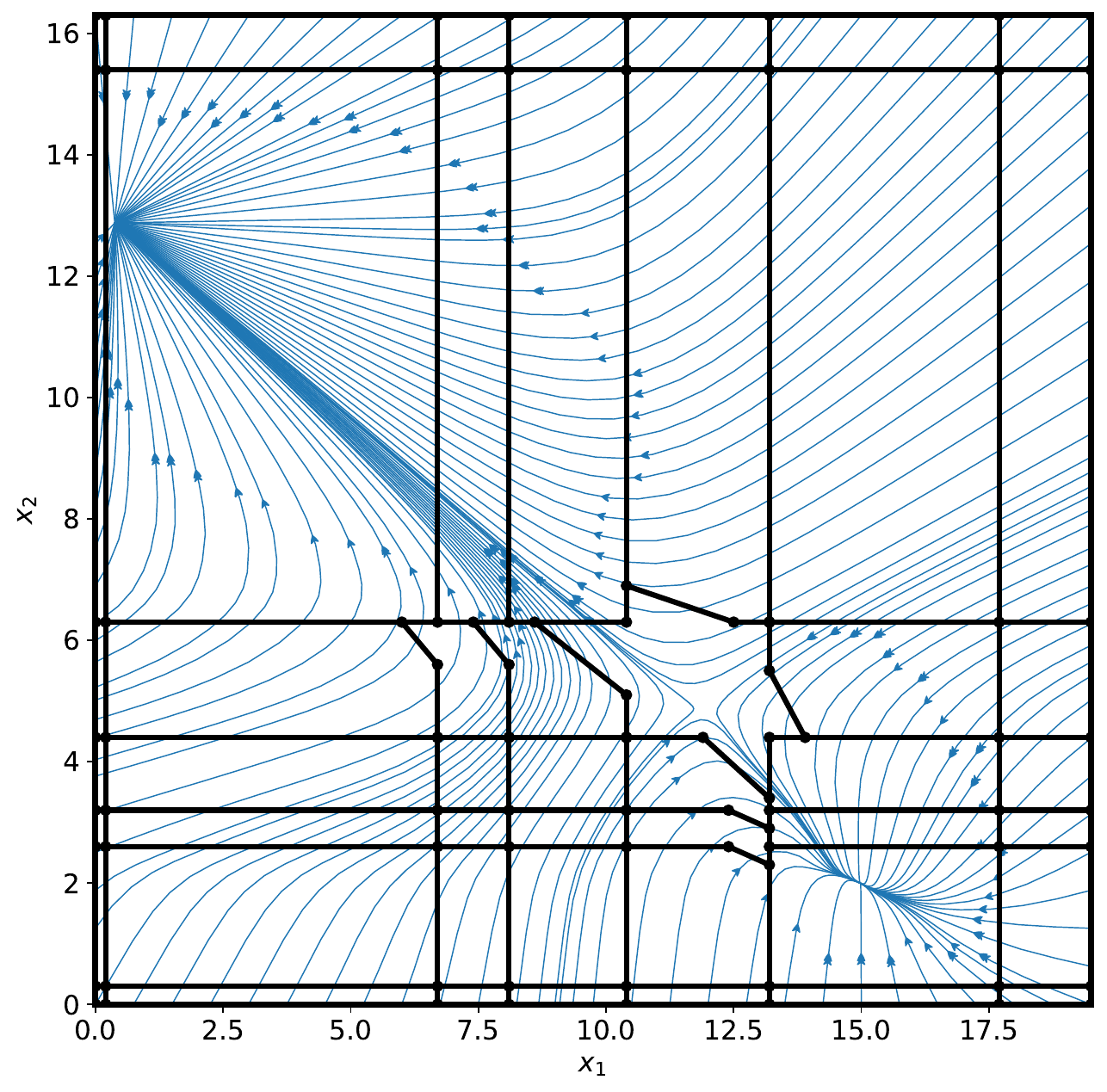}
\caption{Phase portrait and geometric realization for system \eqref{eq:Hill_2D_example_ODE} at parameter values given by the parameter set $1$ in Table~\ref{tab:parameters_intro1} and $s = 10$.}
\label{fig:phase_portrait_Hill_intro}
\end{figure}

Because it is outside of the scope of this effort we do not seek to rigorously validate the numerical check of transversality along the curves. 
However, modulo this step, there are a variety of results concerning the dynamics of the ODE \eqref{eq:Hill_2D_example_ODE} at the parameter values of Table~\ref{tab:parameters_intro1} that follow from the combinatorial/homological computations and the transversality analysis of this paper. Developing methods to efficiently represent the boundaries of the appropriate cells where to check transversality for a given ODE system are beyond the scope of this paper and is left for future work (see Chapter~\ref{sec:futurework}). We note however that we provide explicit representations of the manifolds where transversality is valid for ramp system. The analytical representations of the manifolds given in this paper can be used to construct (piecewise linear) approximations of these manifolds where transversality needs to be checked for a given non-ramp ODE system.

Let $\varphi\colon \R\times [0,\infty)^2 \to [0,\infty)^2$ denote a flow associated with \eqref{eq:2DexampleODE} at parameter values given in Table~\ref{tab:parameters_intro1}.
The following statements are true (justification is provided in Section~\ref{sec:examples}).
\begin{enumerate}
    \item There exists a global attractor $K \subset \prod_{n = 1}^{2}{ [0, \gab_n(\gamma, \nu, \theta, h)] }$ for $\varphi$ (see Proposition~\ref{prop:globalAttractor}).
    \item There exists a semi-conjugacy of $\varphi\colon \R\times K\to K$ onto the dynamics of $\dot{x} = x(1-x^2)$ restricted to $[-1,1]\subset \R$, i.e., a continuous surjective function $\psi\colon K\to [-1,1]$ such that trajectories of $\varphi$ are mapped in a time preserving manner onto trajectories of $\dot{x} = x(1-x^2)$.
    \item The invariant sets $\psi^{-1}(1)$, $\psi^{-1}(-1)$, and $\psi^{-1}(0)$ contain fixed points.
\end{enumerate}

The two-dimensional ODE \eqref{eq:2DexampleODE} was chosen for simplicity of exposition and is used throughout the paper.
Nevertheless, it is important to point out that as presented \eqref{eq:2DexampleODE} is an 18-parameter system.
Via the parameter graph we know that 1,600 such computations identifies the dynamics of \eqref{eq:2DexampleODE} for essentially all parameter values except $h_{i,j}$ (this is discussed in Chapter~\ref{sec:ramp}).
For these parameter values, we provide upper bounds (see Chapter~\ref{sec:ramp}).

Again we remind the reader  that most of our analysis is dimension independent and higher dimensional examples are discussed in Chapter~\ref{sec:examples}.
For the moment, in an attempt to whet the reader's appetite we consider two three-dimensional examples to indicate the complexity of the dynamics that can be extracted.

\begin{ex}
\label{ex:saddle_saddle_3D_intro}
Consider the system
\begin{align}
\label{eq:ramp_system_intro_1}
\dot{x}_1 & = -\gamma_1 x_1 + r_{1,1}(x_1) + r_{1,2}(x_2) + r_{1,3}(x_3) \nonumber \\
\dot{x}_2 & = -\gamma_2 x_2 + r_{2,1}(x_1) \left( r_{2,2}(x_2) + r_{2,3}(x_3) \right) \\
\dot{x}_3 & = -\gamma_3 x_3 + r_{3,2}(x_2) \left( r_{3,1}(x_1) + r_{3,3}(x_3) \right) \nonumber,
\end{align}
where $r_{i, j}$ is given by \eqref{eq:neg_pos_ramp}.

This is a specific choice of the aforementioned ramp system \eqref{eq:rampODEintro}.
We discuss the dynamics of this ODE in greater detail in Example~\ref{ex:saddle_saddle_3D_example} using the concepts and techniques developed in the paper.

For the moment, we focus on simplicity of exposition with the goal of providing an explicit realization of some of the discussion of this introduction.
Thus we fix the parameter values as presented in Table~\ref{tab:parameters_ramp_system_intro_2}.
Let $\varphi$ be the flow for \eqref{eq:ramp_system_intro_1} restricted to the positive orthant $[0,\infty)^3$.
By Proposition~\ref{prop:globalAttractor}, $\varphi$ has a global attractor $K$. 
A Morse graph and associated Conley indices  are presented in Figure~\ref{fig:mg_3d_example_intro_1}. There are multiple admissible connection matrices.
\begin{table}[!htpb]
\centering
\renewcommand{\arraystretch}{1.2}
\begin{tabular}{@{}llllll@{}}
\toprule
$\nu_{1, 1, 1} = 1.01$  & $\nu_{1, 1, 2} = 4.0$   & $\nu_{1, 2, 1} = 1.0$  & $\nu_{1, 2, 2} = 4.0$   & $\nu_{1, 3, 1} = 1.0$  & $\nu_{1, 3, 2} = 2.0$   \\
$\nu_{2, 1, 1} = 0.875$ & $\nu_{2, 1, 2} = 0.797$ & $\nu_{2, 2, 1} = 0.22$ & $\nu_{2, 2, 2} = 0.875$ & $\nu_{2, 3, 1} = 0.44$ & $\nu_{2, 3, 2} = 0.875$ \\
$\nu_{3, 1, 1} = 0.76$  & $\nu_{3, 1, 2} = 1.0$   & $\nu_{3, 2, 1} = 1.0$  & $\nu_{3, 2, 2} = 0.85$  & $\nu_{3, 3, 1} = 0.5$  & $\nu_{3, 3, 2} = 1.0$   \\
$\theta_{1, 1} = 6.5$   & $\theta_{1, 2} = 1.497$ & $\theta_{1, 3} = 1.87$ & $\theta_{2, 1} = 8.0$   & $\theta_{2, 2} = 1.0$  & $\theta_{2, 3} = 1.16$  \\
$\theta_{3, 1} = 3.5$   & $\theta_{3, 2} = 1.46$  & $\theta_{3, 3} = 1.61$ & $\gamma_1 = \gamma_2 = 1$ &   $\gamma_3 = 1.2$   & $h_{i,j} = 0.1$         \\
\bottomrule
\end{tabular}
\caption{Parameters values for \eqref{eq:ramp_system_intro_1}.}
\label{tab:parameters_ramp_system_intro_2}
\end{table}

\begin{figure*}[!htb]
\centering
\includegraphics[width=0.98\textwidth]{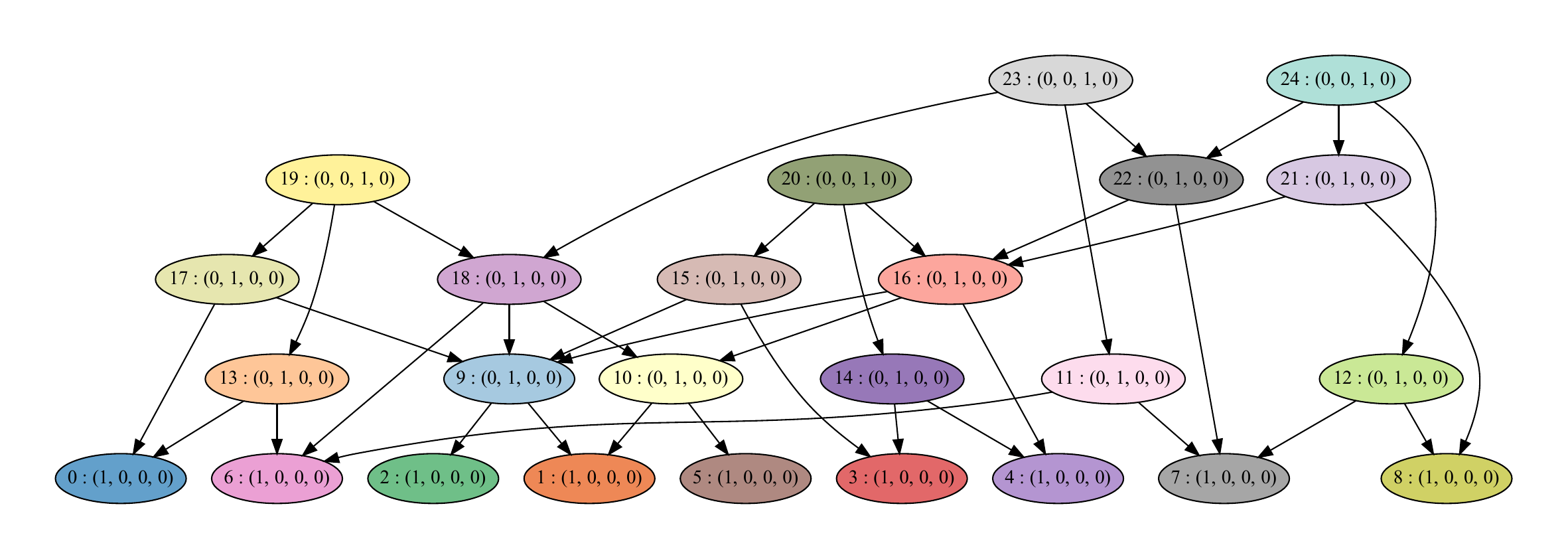}
\caption{Morse graph for the ODE system \eqref{eq:ramp_system_intro_1} with the parameter values in Table~\ref{tab:parameters_ramp_system_intro_2}. The Morse graph corresponds to parameter node $52,718,681,992$ of the regulatory network in Figure~\ref{fig:3d_example_1RN}.}
\label{fig:mg_3d_example_intro_1}
\end{figure*}

An application of Theorem~\ref{thm:dynamics} implies that the Morse graph is a Morse decomposition for $\varphi$ restricted to $K$.
The Conley index of each node in the Morse graph is nonzero, thus associated with each node $p$ is a nontrivial invariant set that we denote by $M(p)$.
Therefore, the Morse decomposition is, in fact, a Morse representation for $\varphi$ restricted to $K$ \cite{kalies:mischaikow:vandervorst:18}.
Furthermore the Conley index of each node is the index of a hyperbolic fixed point.
In particular, nodes labeled $(1,0,0,0)$, $(0,1,0,0)$, $(0,0,1,0)$, and $(0,0,0,1)$, have Conley indices of hyperbolic fixed points with unstable manifolds of dimension $0$, $1$, $2$, and $3$, respectively.
It is not true that this suffices to conclude that the Morse sets are hyperbolic fixed points,
however, for the purpose of this discussion we encourage the reader to make this assumption.\footnote{The reader that refuses this suggestion must, in the discussion that follows, replace $\lim_{t\to \infty}$ and $\lim_{t\to -\infty}$ by omega and alpha limits.} 
We will discuss this subtlety further in Example~\ref{ex:saddle_saddle_3D_example}.

The fact that the Morse graph describes a Morse representation leads to the following result.
Let $x\in [0,\infty)^3$.
Then $\lim_{t\to \infty}x(t) = M(p)$ for some $p\in \setof{0, 1, \ldots, 24}$.
Since $K\subset [0,\infty)^3$ this holds for $x\in K$.
However, in addition, if $x\in K$, then $\lim_{t\to -\infty}x(t) = M(q)$ and $p\leq q$ under the partial order described by the Morse graph.
Furthermore, $p=q$ if and only if $x\in M(p)$.
Thus, the partial order of the Morse graph provides considerable restrictions on the existence of possible connecting orbits.

The connection matrices can be used to deduce the existence of connecting orbits.
We return to this question in Example~\ref{ex:saddle_saddle_3D_example}.
For the moment we remark that with the exceptions of 
\begin{equation}
\label{eq:saddle-saddle}
22 \to 16,\ 21 \to 16,\ 16 \to 10,\ \text{and}~ 16 \to 9,   
\end{equation}
each arrow $q\to p$ in the Morse graph indicates the existence of a trajectory satisfying $\lim_{t\to \infty}x(t) = M(p)$ and $\lim_{t\to -\infty}x(t) = M(q)$.

Continuing with the assumption that the Morse sets are hyperbolic fixed points, the arrows of \eqref{eq:saddle-saddle} suggest the existence of heteroclinic orbits between fixed points all of which have unstable manifolds of dimension $1$.
Given that the exact choice of the parameter values given by Table~\ref{tab:parameters_ramp_system_intro_2} was essentially chosen at random it is highly unlikely that these heteroclinic orbits exist for \eqref{eq:ramp_system_intro_1} at these parameter values.

Nevertheless, we believe that this information from the Morse graph is extremely valuable.
A more detailed justification for this claim is presented in Example~\ref{ex:saddle_saddle_3D_example}, for the moment recall the previous discussion concerning {\bf Step 1}.
Our computations are based on the regulatory network $RN$ associated with \eqref{eq:ramp_system_intro_1} (see Figure~\ref{fig:3d_example_1RN})
not on the ODEs directly.
The parameter graph for $RN$ has 87,280,405,632 parameter nodes.
Our computation is based on parameter node 52,718,681,992 (to understand why this particular parameter node was selected see Section~\ref{sec:ramp2rook}) and, in particular, these computations are valid for an explicit open set $U\subset (0,\infty)^{39}$ of parameter values (see Chapter~\ref{sec:ramp}).
Our claim is that codimension-one hypersurfaces of bifurcations associated with saddle to saddle connections indicated by \eqref{eq:saddle-saddle} occur within the region $U$.

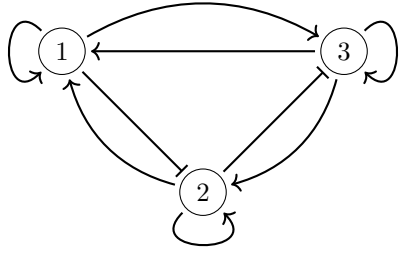
\begin{figure*}[!htb]
\centering
\begin{tikzpicture}
[main node/.style={circle,fill=white!20,draw}, scale=2.5]
\node[main node] (1) at (-0.75,0.75) {1};
\node[main node] (2) at (0,0) {2};
\node[main node] (3) at (0.75,0.75) {3};
\path[thick]
(1) edge[->, loop, shorten <= 2pt, shorten >= 2pt, distance=10pt, thick, out=135, in=225] (1) 
(1) edge[-|, shorten <= 2pt, shorten >= 2pt] (2)
(1) edge[->, shorten <= 2pt, shorten >= 2pt, bend left] (3)
(2) edge[->, shorten <= 2pt, shorten >= 2pt, bend left] (1)
(2) edge[->, loop, shorten <= 2pt, shorten >= 2pt, distance=10pt, thick, out=225, in=-45] (2) 
(2) edge[-|, shorten <= 2pt, shorten >= 2pt,] (3)
(3) edge[->, shorten <= 2pt, shorten >= 2pt] (1)
(3) edge[->, shorten <= 2pt, shorten >= 2pt, bend left] (2)
(3) edge[->, loop, shorten <= 2pt, shorten >= 2pt, distance=10pt, thick, out=45, in=-45] (3) ;
\end{tikzpicture}
\caption{Regulatory network $RN$ for ramp system \eqref{eq:ramp_system_intro_1}.
This network generates exactly 87,280,405,632 combinatorial models.}
\label{fig:3d_example_1RN}
\end{figure*}
\end{ex}

We include the following example to emphasize that our approach can identify the existence of nontrivial recurrent dynamics, in this case periodic orbits.

\begin{ex}
\label{ex:introPeriodic}
Consider the ramp system
\begin{align}
\label{eq:ramp_system_intro_periodic}
\dot{x}_1 & = -\gamma_1 x_1 + r_{1,1}(x_1) r_{1,2}(x_2) r_{1,3}(x_3) \nonumber \\
\dot{x}_2 & = -\gamma_2 x_2 + r_{2,1}(x_1) r_{2,2}(x_2) \\
\dot{x}_3 & = -\gamma_3 x_3 + r_{3,2}(x_2) r_{3,3}(x_3) \nonumber,
\end{align}
where $r_{i, j}$ is given by \eqref{eq:neg_pos_ramp}, with the parameter values in Table~\ref{tab:parameters_ramp_system_intro_periodic}.
Let $\varphi$ be the flow for \eqref{eq:ramp_system_intro_periodic} restricted to the positive orthant $[0,\infty)^3$.
By Proposition~\ref{prop:globalAttractor}, $\varphi$ has a global attractor $K$.
A Morse graph and associated Conley indices for this system is presented in Figure~\ref{fig:mg_3d_example_intro_periodic}. 

\begin{table}[!htpb]
\centering
\renewcommand{\arraystretch}{1.2}
\begin{tabular}{@{}llllll@{}}
\toprule
$\nu_{1, 1, 1} = 1.80$ & $\nu_{1, 1, 2} = 8.56$  & $\nu_{1, 2, 1} = 13.07$ & $\nu_{1, 2, 2} = 3.25$ & $\nu_{1, 3, 1} = 20.10$ & $\nu_{1, 3, 2} = 1.07$  \\
$\nu_{2, 1, 1} = 2.44$ & $\nu_{2, 1, 2} = 0.84$  & $\nu_{2, 2, 1} = 0.16$  & $\nu_{2, 2, 2} = 6.10$ & $\nu_{3, 2, 1} = 2.39$  & $\nu_{3, 2, 2} = 1.36$  \\
$\nu_{3, 3, 1} = 0.05$ & $\nu_{3, 3, 2} = 5.03$  & $\theta_{1, 1} = 27.17$ & $\theta_{1, 2} = 2.26$ & $\theta_{1, 3} = 11.73$ & $\theta_{2, 1} = 39.10$ \\
$\theta_{2, 2} = 1.25$ & $\theta_{3, 2} = 10.47$ & $\theta_{3, 3} = 6.70$  & $\gamma_1 = 1$ & $\gamma_2 = \gamma_3 = 0.5$ & $h_{i,j} = 0.5$ \\
\bottomrule
\end{tabular}
\caption{Parameters values for \eqref{eq:ramp_system_intro_periodic}.}
\label{tab:parameters_ramp_system_intro_periodic}
\end{table}

\begin{figure*}[!htb]
\centering
\includegraphics[width=0.95\textwidth]{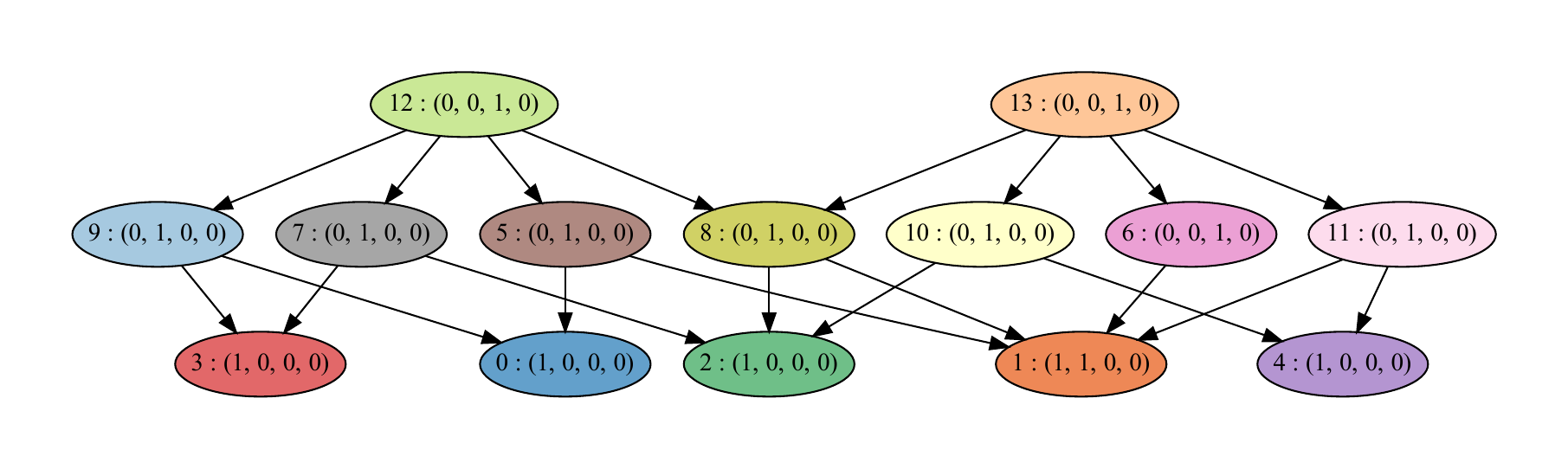}
\caption{Morse graph for the ODE system \eqref{eq:ramp_system_intro_periodic} with the parameter values in Table~\ref{tab:parameters_ramp_system_intro_periodic}. The Morse graph corresponds to parameter node $2,472,287$.}
\label{fig:mg_3d_example_intro_periodic}
\end{figure*}

An application of Theorem~\ref{thm:dynamics} implies that the Morse graph is a Morse decomposition for $\varphi$ restricted to the global attractor $K$.
The fact that the Conley index of each node in the Morse graph is nonzero, implies that the Morse decomposition is a Morse  representation.

As in Example~\ref{ex:saddle_saddle_3D_intro} we encourage the reader to assume that each Morse set $M(p)$ with Conley index 
\[
CH_k(p;\Z_2) \cong
\begin{cases}
\Z_2 & \text{if $k=d$} \\
0 & \text{otherwise}
\end{cases}
\]
is a hyperbolic fixed point with unstable manifold of dimension $d$.
Under this assumption each $q\to p$ arrow implies the existence of a heteroclinic orbit satisfying $\lim_{t\to \infty}x(t) = M(p)$ and $\lim_{t\to -\infty}x(t) = M(q)$.

There is an exception to this discussion and that is node $p = 1$ for which the Conley index is
\[
CH_k(p;\Z_2) \cong
\begin{cases}
\Z_2 & \text{if $k=0,1$} \\
0 & \text{otherwise.}
\end{cases}
\]
This is the Conley index of a stable hyperbolic periodic orbit.
As explained in Example~\ref{ex:periodicOrbit3}, the machinery developed in this manuscript allows us to identify a region of phase space that contains the Morse set $M(1)$ and furthermore to conclude that $M(1)$ contains a periodic orbit of \eqref{eq:ramp_system_intro_periodic}.
Given this information, it is relatively straightforward to obtain a numerical representation of such a periodic orbit as indicated in Figure~\ref{fig:periodicOrbit}

\begin{figure}[!htb]
\centering
\includegraphics[width=0.85\textwidth]{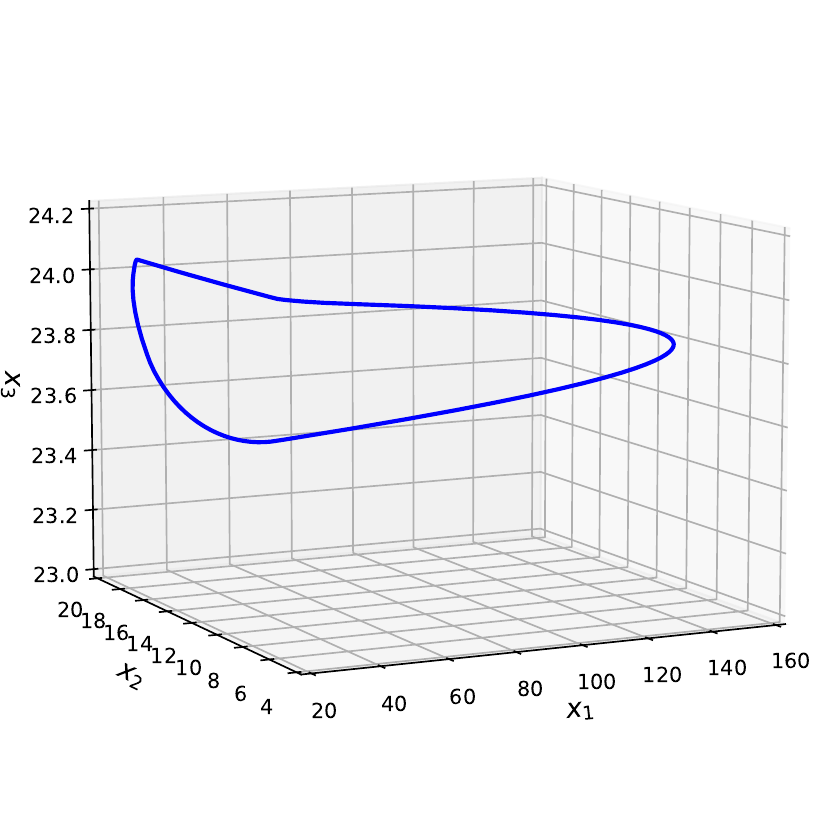}
\caption{Periodic orbit that belongs to Morse set $M(1)$ for \eqref{eq:ramp_system_intro_periodic} at the parameter values given in Table~\ref{tab:parameters_ramp_system_intro_periodic}.}
\label{fig:periodicOrbit}
\end{figure}
\end{ex}

In this section we have attempted to provide a high level introduction to the novel perspective of the analysis of systems of ODEs being presented in this monograph.
We have stated our goals of efficiently capturing the global dynamics of ODEs and provided examples of the types of results that can be determined.
We have claimed that this is done using combinatorial and algebraic topological techniques, but have not provide any substantial insight into the details and how the combinatorial/algebraic results can be related to the continuous dynamics of ODEs.
We attempt to rectify this in the next two chapters.

An implementation of the methods presented in this monograph is included with the DSGRN software \cite{DSGRN_Rook_Field}. A Jupyter notebook with commented code to reproduce the examples presented in Chapters~\ref{sec:intro} and \ref{sec:examples} is available at \cite{Rook_Field_Paper_Repo}.

\chapter{Prelude}
\label{sec:prelude}

The introduction presents the goals of the monograph.
Ideally, the description of the pipeline by which these goals are to be achieved makes clear on a nontechnical level the mathematical challenges that need to be addressed.
We hasten to add -- and this is discussed explicitly in Chapter~\ref{sec:futurework} -- that we do not claim to have fully resolved the challenges.
These challenges are nontrivial and involve addressing multiple technical mathematical issues.
With this in mind we use this prelude to discuss how this work relates to classical dynamics, provide an outline of the organization of this paper  on a slightly more technical level, and use this to indicate how the steps of the pipeline are achieved. 
Precise statements of our results require the language of the machinery we develop in this monograph and thus are presented in Chapter~\ref{sec:examples}.
 
{\bf Step 4} of the pipeline involves translating  the combinatorial/homological information derived during the earlier steps into the language of flows, i.e., solutions to ODEs.
To do this we take advantage of two fundamental contributions by C. Conley \cite{conley:cbms}: a decomposition of invariant sets into gradient-like and recurrent dynamics, and an algebraic topological invariant, called the \emph{Conley index}, from which existence and structure of the invariant sets can be recovered.

We begin with the decomposition, but present it in two stages.
The first describes the global dynamics in the language of invariant sets.
This is an existential framework, we presume to know the objects of interest.
The second describes the global dynamics using objects that can be computed.
There is a mapping from the computable objects to invariant sets, but typically there may be invariant objects that are not revealed by the computations and computed objects may correspond to empty invariant sets.
We use the Conley index to guide our understanding of this correlation.

Let $\varphi\colon \R\times X\to X$ be a flow on a compact metric space $X$.
A set $S\subset X$ is \emph{invariant} if $\varphi(\R,S) = S$.
A set $A\subset X$ is an \emph{attractor} if there exists a compact set $K\subset X$ such that 
\[
A = \omega(K,\varphi) := \bigcap_{t\geq 0}\cl(\varphi((t,\infty),K)) \subset \Int(K)
\]
where $\cl$ and $\Int$ denote closure and interior, respectively.
The set of all attractors of $\varphi$ is denoted by $\sAtt(\varphi)$ and forms a bounded distributive lattice \cite{kalies:mischaikow:vandervorst:14}.
Conley's decomposition theorem \cite{conley:cbms} states that
\[
CR(\varphi) := \bigcap_{A\in\sAtt(\varphi)} A\cup \setof{x\in X\mid \omega(x,\varphi)\cap A = \emptyset}
\]
is an invariant set, called the \emph{chain recurrent set}, and there exists a continuous function $V\colon X \to [0,1]$, called a \emph{Lyapunov function} satisfying
\[
\begin{cases}
V(\varphi(x,t)) = V(x) &\text{if $x\in CR(\varphi)$,} \\
V(\varphi(x,t)) < V(x) &\text{otherwise,}
\end{cases}\quad\text{for all $t>0$.}
\]
Furthermore, for any Lyapunov function $W$ if $W(\varphi(y,t))= W(y)$, for all $y\in Y\subset X$, then $CR(\varphi)\subset Y$.
Loosely stated, all recurrent dynamics of $\varphi$ occurs in $CR(\varphi)$ and  $\varphi$ exhibits gradient-like dynamics off of $CR(\varphi)$.

Since $\sAtt(\varphi)$ may contain countably infinitely many elements \cite{conley:cbms}, we cannot in general expect to be able to compute $CR(\varphi)$.
With this in mind let $\sA$ be a finite sublattice of $\sAtt(\varphi)$ containing both the minimal attractor $\emptyset$ and the maximal attractor $\omega(X,\varphi)$.
This gives rise to a \emph{Morse representation} $(\sM\sR(\sA),\leq_\varphi)$ \cite{kalies:mischaikow:vandervorst:18}.
The elements of $\sM\sR(\sA)$ are called called \emph{Morse sets}. 
Morse sets are non-empty invariant sets and can be defined in terms of  components of 
\[
\bigcap_{A\in\sA} A\cup \setof{x\in X\mid \omega(x,\varphi)\cap A = \emptyset}.
\]
The partial order $\leq_\varphi$ is derived from the flow. In particular, given distinct Morse sets $M$ and $M'$, set $M <_\varphi M'$  if there exists $x\in X$ such that $\omega(x,\varphi)\subset M$ and $\alpha(x,\varphi)\subset M'$, and extend the relation transitively.
Observe that the partial order $\leq_\varphi$ provides explicit information about the gradient-like structure of the dynamics; the dynamics moves from higher ordered invariant sets to lower ordered invariant sets.

In general we cannot expect to be able to compute $(\sM\sR(\sA),\leq_\varphi)$.
Morse sets need not be  stable with respect to perturbations.
The simplest example is a degenerate fixed point associated with  a saddle-node bifurcation; under arbitrarily small perturbations the Morse set may vanish.
With this in mind, we turn to computable objects.

\begin{defn}
\label{defn:attractingblock}
A compact set $K\subset X$ is an \emph{attracting block} for a flow $\varphi$ if $\varphi(t,K)\subset \Int(K)$ for all $t>0$.
\end{defn}

The set of attracting blocks $\sABlock(\varphi)$ is a bounded distributive lattice with a partial order relation given by set inclusion \cite{kalies:mischaikow:vandervorst:14}.
Furthermore,
\begin{align*}
\omega\colon \sABlock(\varphi) &\to \sAtt(\varphi) \\
K & \mapsto \omega(K)
\end{align*}
is a bounded lattice epimorphism.
Since typically $\sABlock(\varphi)$ is uncountable, it is clear that $\omega$ is  not a monomorphism.

Let $\sK$ be a finite sublattice of $\sABlock(\varphi)$ that contains $\emptyset$ and $X$.
Recall that a \emph{join irreducible element} of $\sK$ is an attracting block that has a unique immediate predecessor with respect to the partial order of inclusion.
The set of all join irreducible elements of $\sK$ is denoted by $\sJ^\vee(\sK)$ and is a poset with partial order given by inclusion.

Given $K\in \sJ^\vee(\sK)$, let $K^< \in \sK$ denote its unique \emph{immediate predecessor}, i.e., $K^< \subset K$ with $K^< \neq K$ and if there exists $K'\in \sK$ such that $K^< \subset K' \subset K$, then $K' = K^<$ or $K' = K$.
 
The \emph{Morse tiling} associated with $\sK$ is given by
\[
\sM\sT(\sK) := \setof{\cl(K\setminus K^<)\subset X\mid K\in \sJ^\vee(\sK)}.
\]

Observe that $\sA : = \omega(\sK)$ is a finite lattice of attractors that contains $\emptyset$ and $\omega(X,\varphi)$.
As such $(\sM\sR(\sA),\leq_\varphi)$ is well defined.

The poset $\sJ^\vee(\sK)$ is called a \emph{Morse decomposition} for $\varphi$ since there exists an order preserving embedding $\pi\colon \sM\sR(\sA) \to \sJ^\vee(\sK)$.
For further details see \cite{kalies:mischaikow:vandervorst:18}.

Our claim, which we state in the form of the following objective, is that attracting blocks and Morse decompositions are computable.

\begin{description}
\item[O1] Compute a \emph{Morse decomposition} of $\varphi$, i.e., an order-embedding $\pi\colon \sM\sR\hookrightarrow \sJ^\vee(\sK)$.
\end{description}

Understanding the structure of the recurrent dynamics of a differential equation is of considerable interest.
While a Morse decomposition does not directly address this challenge, it does provide information about the location of recurrent dynamics; Morse sets capture recurrent dynamics and Morse sets are subsets of Morse tiles.
Therefore, the first question we can ask is the following: given a Morse tile $\cl(K\setminus K^<)$, is 
\[
\Inv(\cl(K\setminus K^<),\varphi) := \setof{x\in \cl(K\setminus K^<)\mid \varphi(\R,x)\subset \cl(K\setminus K^<) } \neq \emptyset?
\]

To answer this we employ the (homological) Conley index \cite{conley:cbms, salamon, mischaikow:mrozek:02}.
For the purposes of this paper it is sufficient to define the Conley index as follows. 
Let $K_i\in \sABlock(\varphi)$, $i=0,1$, and assume that $K_0\subset K_1$.
The \emph{Conley index} of $S = \Inv(K_1\setminus K_0,\varphi)$ is given by
\[
CH_*(S) := H_*(K_1,K_0).
\]
The most fundamental result is the following \cite{conley:cbms}.
\begin{thm}
\label{thm:CH}
    Let $S$ be an isolated invariant set. If $CH_*(S) \neq 0$, then $S\neq \emptyset$.
\end{thm}
The converse need not be true.
More refined theorems lead to the existence of equilibria \cite{mccord:89}, periodic orbits \cite{mccord:mischaikow:mrozek}, and heteroclinic orbits \cite{conley:cbms}.

Given a Morse decomposition, the Morse graph and the associated Conley indices can be organized as a chain complex.
Let $\sK$ be a finite sublattice of $\sABlock(\varphi)$ that contains $\emptyset$ and $X$.
In a misuse of notation, given $K\in \sJ^\vee(\sK)$, set 
\[
CH_*(K;\F) = CH_*(\Inv(\cl(K\setminus K^<),\varphi);\F) = H_*(K,K^<;\F).
\]
Furthermore, assume that $\F$ is a field.
Franzosa \cite{franzosa:89} proved that there exists a boundary operator $\Delta$, called a \emph{connection matrix}, that acts on chains that are given as a direct sum of the Conley indices of the isolated invariant sets of the Morse tiling, i.e.
\begin{equation}
\label{eq:DeltaPrelude}
\Delta \colon \bigoplus_{K\in\sJ^\vee(\sK)} CH_*(K;\F) \to \bigoplus_{K\in\sJ^\vee(\sK)} CH_*(K;\F)
\end{equation}
from which Conley indices of other invariant sets can be recovered.

Even a cursory discussion of the connection matrix goes beyond the scope of this manuscript.
For our purposes it is sufficient to note that a connection matrix (connection matrices need not be unique) satisfies the following properties.
\begin{itemize}
    \item If $\Delta(K',K) \colon CH_*(K;\F) \to CH_*(K';\F)$ is not the trivial map, then $K' \subset K$. 
    \item It is a boundary operator, i.e.\  $\Delta(K',K)\left( CH_k(K);\F \right)\subset CH_{k-1}(K';\F)$ and $\Delta\circ \Delta = 0$.
\end{itemize}
As indicated in Chapter~\ref{sec:intro} the connection matrix can be used to characterize  global dynamics.
For specific examples the reader is referred to \cite{reineck, mccord:mischaikow:92, mccord:mischaikow:96, mischaikow:95, gedeon:mischaikow:95, kokubu:mischaikow:oka, kokubu:mischaikow:nishiura:oka:takaishi, maier-paape:mischaikow:wanner:07, day:hiraoka:mischaikow:ogawa}.
We make use of these techniques in Chapter~\ref{sec:examples}.

The fact that this algebraic topological data provides information about the recurrent and global dynamics of the flow leads us to our second objective.
\begin{description}
\item[O2] Given a Morse decomposition of $\varphi$ compute the set of connection matrices.
\end{description}

Observe that as described a Morse decomposition provides a combinatorial representation of the gradient-like dynamics and the associated connection matrices provide information about  potential recurrent dynamics and connecting orbits   for \eqref{eq:generalODE} at a single parameter value.
Thus, our third objective is as follows.
\begin{description}
\item[O3] Provide an explicit open region $\Omega$ of parameter space for which the Morse decomposition is valid.
\end{description}

We declare 
\begin{quote}
\emph{
    Given an ODE \eqref{eq:generalODE}, if  objectives {\bf O1} - {\bf O3} are achieved, then {\bf Goal 1} is satisfied.
    }
\end{quote}

A priori we cannot assume to know the finite sublattice of attractors $\sA$ upon which our description of {\bf O1}-{\bf O3} is based, thus the appropriate interpretation of our declaration is that given a level of resolution we will identify a Morse decomposition and connection matrix.
As is shown in  \cite{kalies:mischaikow:vandervorst:15} given any finite lattice of attractors $\sA$ there exists a discretization of phase space and a combinatorial description of the dynamics that takes the form of a finite lattice of attracting blocks that allows for the identification of $\sA$ and the associated Conley index information.
However, from a practical perspective this is an existential result.

A key contribution of this paper is to provide a practical method for carrying out the computations that produces an interesting sublattice $\sA$.
As discussed below this is done using combinatorial techniques.
At least conceptually this is compatible with the finite poset conditions of {\bf O1}.
However, the flow $\varphi$ that appears explicitly or implicitly in {\bf O1}-{\bf O3} is the flow associated with \eqref{eq:generalODE}.
Thus, an essential step to achieving the objectives is to pass from the combinatorial results to results about differential equations without explicit knowledge of the associated flow.
In Chapter~\ref{sec:geometrizationCellComplex} we provide an abstract framework in which this can be done.

The methods presented in Part~\ref{part:II}
are purely combinatorial/algebraic in nature; they involve no analysis or topology.
It is in this part that we carry out {\bf Step 1} and {\bf Step 2} of the pipeline.
The combinatorial model is developed in three stages in Chapters~\ref{sec:RookFields} and \ref{sec:rules}.
\begin{enumerate}
    \item We introduce the concept of a \emph{wall-labeling} on an abstract cubical complex $\cX$. 
    This provides, with certain constraints, signed pairings between top dimensional cells of $\cX$ and their co-dimension one faces. 
    The wall labeling provides \emph{all} the data used in the combinatorial calculations.
    \item We re-interpret the wall labeling as a \emph{rook field}.
    Rook fields can be viewed as a combinatorial vector field that is assigned to each cell of the cubical complex $\cX$ and whose coordinates assumes values $-1$, $0$, or $1$.
\item We use the rook field to define four families $\cF_i\colon \cX \mvmap \cX$, $i=0,\ldots, 3$ of combinatorial models that take the form of \emph{multivalued maps}, i.e., for each $\xi\in\cX$, $\cF_i(\xi)$ is a non-empty set of cells in $\cX$. 
The most naive model $\cF_0$ does nothing more than ensure that the dynamics we are modeling can be represented by a flow.
The models $\cF_1$, $\cF_2$, and $\cF_3$ provide potentially ever greater refinement of the dynamics.
\end{enumerate}
For the sake of clarity of exposition we provide minimal motivation for the conditions imposed on the models $\cF_1$, $\cF_2$, and $\cF_3$ when they are introduced in Chapter~\ref{sec:rules}.
For the moment it is sufficient to remark that they are chosen to allow us to identify hypersurfaces that are transverse to the vector field of the ODE of interest.
This is made clear in Chapters~\ref{sec:R1Dynamics}, \ref{sec:R2Dynamics}, and \ref{sec:R3Dynamics}.

In Chapter~\ref{sec:lattice} we address {\bf Step 2} of the pipeline. 
In particular, in Section~\ref{sec:CCalgorithm} we define the \emph{Morse graph} of the combinatorial model $\cF_i$. 
This the combinatorial equivalent of the above mentioned Morse decompositions and provides the means by which we aim to complete {\bf O1} and {\bf Step 2a}.
Furthermore, we make use of the work of \cite{harker:mischaikow:spendlove,harker:templates:21} to compute \emph{Conley complexes} for $\cF_i$, i.e., chain complexes of the form
\begin{equation}
    \label{eq:CCPrelude}
\Delta \colon \bigoplus_{\cM\in\sMG(\cF_i)} CH_*(\cM;\F) \to CH_*(\cM;\F),
\end{equation}
and thereby aim to complete {\bf O2}.

There is a significant gap between the definitions provided in Chapters~\ref{sec:RookFields} and \ref{sec:rules} and the structures and homological conditions required to perform the computations discussed in Section~\ref{sec:CCalgorithm}.
Most of Chapter~\ref{sec:lattice} is dedicated towards resolving this gap.
In Section~\ref{sec:compInvset+} we discuss the computation of 
the lattices of \emph{forward invariant sets} $\sInvset^+(\cF_i) := \setof{\cN\subset\cX\mid \cF_i(\cN)\subset \cN}$ and an associated order structure called a \emph{D-grading}.
This is used to identify a lattice of cell complexes from which connection matrices are derived.
From the homological perspective we require \emph{admissible gradings} that are discussed in Section~\ref{sec:AdmissbleD-grading}.
As a consequence we need to show that a D-grading leads to an admissible grading. 
We are not able to do this using the cell complex  $\cX$.
Thus, in Section~\ref{sec:blowup_complex} we introduce the \emph{blow-up complex} $\cX_b$ and transfer the D-grading on $\cX$ to a D-grading on $\cX_b$.

The theorems of Section~\ref{sec:Dgradblowup} are fundamental in that they demonstrate that the D-gradings on $\cX_b$ generated by our combinatorial models $\cF_i$ are admissible gradings, i.e., they can be used to determine connection matrices. Furthermore, we show that these admissible gradings give rise to AB-lattices, the importance of which is made clear in Chapter~\ref{sec:geometrizationCellComplex}.
Having identified that our constructions lead to admissible gradings, in Section~\ref{sec:ConleyComplex} we discuss Conley complexes, which sets us up for the computations discussed in Section~\ref{sec:CCalgorithm}.

{\bf Goals 1} and {\bf 2} and objectives {\bf O1} - {\bf O3} are statements about solutions to differential equations.
However, as we again emphasize the definitions and results of Part~\ref{part:II} only involve combinatorics and homological algebra.
There is no topology, let alone differential equations.
The focus of Part~\ref{part:III} is to demonstrate that for an interesting class of ODEs these goals and objectives are realized by the computations of Section~\ref{sec:CCalgorithm}. 

In Chapter~\ref{sec:ramp} we discuss ramp systems and their associated wall labelings.
Ramp systems, a special class of ODEs examples of which are given in Chapter~\ref{sec:intro}, are defined in Section~\ref{sec:DefinitionRampSystems}.
In Section~\ref{sec:ramp2rook} we make precise how to go from ramp system parameters to a wall labeling. 
Thus, given a wall labeling we can apply the content of Part~\ref{part:II} to compute $\cF_i$ and obtain the results described in Section~\ref{sec:CCalgorithm}.
However, this does not necessarily imply that the Morse graph and Conley complex results  apply to the ramp system ODE.
Our analysis requires us to make additional restrictions on the parameter values; the restrictions depend on the choice of combinatorial model $\cF_i$ and are presented in Section~\ref{sec:ramp2rook}.
 
While every ramp system can be associated with a wall labeling, in Section~\ref{sec:wall-ramp-property} we make clear that there are wall labelings that cannot be realized by ramp systems.
Though discussed in greater detail in Section~\ref{sec:alternativeNonlinearities} this suggests that the techniques presented in this monograph can be applied to a much broad class of ODEs.

Recall that  {\bf Goal 2} involves going from a regulatory network to a characterization of the global dynamics of associated ODEs.
Our software DSGRN, discussed briefly in Chapter~\ref{sec:intro}, is introduced to a greater extent in Section~\ref{sec:DSGRN}.
Salient points for the current discussion are as follows.
First, given a regulatory network $RN$, the choice of node in the associated parameter graph, and a choice of $i=1,2,3$, the DSGRN software produces a combinatorial model $\cF_i$ from which a Morse graph $\sMG(\cF_i)$ and a Conley complex is computed.
Second, the DSGRN parameters are a subset of ramp system parameters that can easily be embedded in the parameter regions described in Section~\ref{sec:ramp2rook}.
Third, the current implementation DSGRN can only be identified with a  special class of ramp systems.
However this special class is sufficiently rich that the computations for most of the examples discussed in Chapter~\ref{sec:examples} are done using DSGRN.\footnote{We provide an example in Chapter~\ref{sec:examples} of a ramp system for which DSGRN cannot be directly applied.} 

Having introduced ramp systems the remainder of Part~\ref{part:III} is dedicated to proving that {\bf Goals 1} and {\bf 2} and objectives {\bf O1} - {\bf O3} are achieved.
While our strategy is to provide essentially independent proofs of these results for $\cF_i$, $i=0,1,2,3$, the unifying idea is to take the abstract cell complex $\cX_b$ and produce an associated CW complex that we refer to as a \emph{geometrization} with the property that boundaries of specific cells are transverse to the vector field of the ramp system.
The theoretical justification for this approach is provided in Chapter~\ref{sec:geometrizationCellComplex} where it is made clear (see Remark~\ref{rem:O3}) that if our approach is successfully employed, then weak version of objective {\bf O3} is satisfied.

The most primitive geometrization, called a \emph{rectangular geometrization}, is defined in Chapter~\ref{sec:geometrization01}.
In  Chapter~\ref{sec:R0} we demonstrate how given a ramp system and the combinatorial model $\cF_0$ this rectangular geometrization can be used to prove that the Morse graph $\sMG(\cF_0)$ defines a Morse decomposition, to determine the connection matrix, and to conclude the existence of a fixed point for the flow.
This chapter is included for pedagogical purposes since the results obtained using $\cF_0$ are trivial.
However, the essential ideas of how the combinatorial information from Part~\ref{part:II} carries over to information about the dynamics of ramp systems is present in this example.

As suggested above the combinatorial models $\cF_1$, $\cF_2$, and $\cF_3$ capture different behaviors of vector fields.
Roughly speaking $\cF_1$ identifies dynamics that is dominated by the vector field being transverse to the coordinate axes, $\cF_2$ organizes dynamics associated with the vector field being tangent to the coordinate axes, and $\cF_3$ captures dynamics near equilibria.

In Chapter~\ref{sec:R1Dynamics} we consider the combinatorial model $\cF_1$ and show that rectangular geometrizations can employed. 
As is indicated in Part~\ref{part:II} the combinatorial dynamics associated with $\cF_1$ can be nontrivial.
This implies that $\cF_1$ can be used to extract useful information about dynamics of ramp systems

Chapters~\ref{sec:janusComplex} - \ref{sec:R2Dynamics} focus on the combinatorial model $\cF_2$.
The dynamics captured by $\cF_2$ requires a more subtle geometric realization than that which can be captured by the  rectangular geometrization. 
Thus, in Chapter~\ref{sec:janusComplex} we introduce the \emph{Janus complex} $\Janus$, which is a refinement of the blow-up complex $\cX_b$.

The D-grading on $\cX_b$ naturally extends to a D-grading on $\Janus$, but to construct the geometrization we need to introduce a modified D-grading. 
This is the content of Chapter~\ref{sec:P2-grading}.
In particular, the construction of the modified D-grading is done in Sections~\ref{sec:2notation} - \ref{sec:pi2}.
The thesis of this monograph is that the computations done in Section~\ref{sec:CCalgorithm} are sufficient to characterize the dynamics of ramp systems.
Thus, rather than performing new computations using the modified D-grading on $\Janus$ we prove in Section~\ref{sec:pi2Equivalent} that the  Morse graph and Conley complexes in the modified setting are the same as those of obtained using the D-grading on $\cX$ and $\cX_b$.

Chapters~\ref{sec:janusComplex} and \ref{sec:P2-grading} provide refined and modified combinatorial structures.
To use these structures to construct the desired geometrization is a question of analysis which
is presented in Chapter~\ref{sec:R2Dynamics}. 
The analytic subtleties arise from the fact that we need to track the behavior of the flow as it passes through nullclines of the vector field.

We conclude Part~\ref{part:III} with Chapters~\ref{sec:P3-Grading} and \ref{sec:R3Dynamics} that address geometrization in the context of $\cF_3$.
Our construction of $\cF_3$ is restricted to two and three dimensional systems.
The possibility of extending this work to higher dimensional systems is discussed in Section~\ref{sec:equilibria}.
As in the case of $\cF_2$ there are two issues that need to be addressed. 
The first, done in Chapter~\ref{sec:P3-Grading}, is the combinatorial work, the appropriate modification of the D-grading using the Janus complex $\Janus$, and the second, done in Chapter~\ref{sec:R3Dynamics} is using analysis to construct the desired geometrization.

The results of Chapters~\ref{sec:R1Dynamics}, \ref{sec:R2Dynamics}, and \ref{sec:R3Dynamics} imply that in the context of ramp systems  {\bf Goals 1} and {\bf 2} and objectives {\bf O1} - {\bf O3} are met for combinatorial models $\cF_1$, $\cF_2$, and $\cF_3$, respectively.

Part~\ref{part:IV} of the manuscript has two chapters.
Chapter~\ref{sec:examples} consists of examples meant to demonstrate the power of our techniques.
In particular, we combine the computations of Section~\ref{sec:CCalgorithm} with classical results that allow us to use Conley index information to draw conclusions about the structure of the dynamics and the existence of bifurcations.

In Chapter~\ref{sec:futurework} we focus on three directions of future work that we believe will lead to improvements of the results presented here.
In Section~\ref{sec:equilibria} we indicate why we believe that the restriction of $\cF_3$ to two and three dimensional systems can be lifted.
Finally, in Section~\ref{sec:alternativeNonlinearities} we discuss the challenges of moving to ODEs that are not ramp systems.
Using a simple example  we show how a system that does not have the form of a ramp system can be transformed to a ramp system, but we also highlight that care needs to be taken to insure that an appropriate ramp system model is chosen.

Finally, we finish the manuscript with an appendix that describes software that allows the user to visualize the combinatorial constructions of Part~\ref{part:II} in dimensions two and three.

\chapter{Geometrization of a Cell Complex}
\label{sec:geometrizationCellComplex}

As is discussed in the previous chapters the mathematical flavors of Part~\ref{part:II} and Part~\ref{part:III} are quite distinct.
The language and techniques employed in  Part~\ref{part:II} come from combinatorics, order theory, and homological algebra.
This is also where the descriptions of the computations we perform are presented.
The focus of  Part~\ref{part:III} is on ODEs, geometry and analysis.
The philosophy behind this partition is that it is via Part~\ref{part:II} that we obtain the characterization of the global dynamics of an ODE, and the purpose of Part~\ref{part:III} is to identify which ODEs have been solved.
The goal of this chapter is to present the mathematics that ties these two distinct parts of the manuscript together.
The key idea is that of a flow transverse geometric realization (Part~\ref{part:III}) of a lattice ordered chain complex (Part~\ref{part:II}).
Thus, we begin by recalling a few elementary definitions from order theory and algebraic topology.
We then discuss regular CW decompositions and conclude with results about the structure of flows.

\section{Posets and Lattices}
\label{sec:posetsLattices}

We begin with a few classical definitions.
\begin{defn}
\label{defn:poset}
A \emph{partially ordered set} (poset) $(\sP,\leq)$ consists of a set $\sP$ with a reflexive, antisymetric,  transitive relation $\leq$. Given $p, q \in \sP$ we denote $p < q$ to indicate that $p \leq q$ and $p \neq q$.
\end{defn}

\begin{defn}
\label{defn:meet_join}
Let $(\sP, \leq)$ be a poset and let $p, q\in \sP$. If $\{ p, q \}$ have a least upper bound it is called the \emph{join of $p$ and $q$} and is denoted by $p \vee q$, that is,
\[
p \vee q := \inf \setof{r\in\sP\mid p\leq r\ \text{and}\ q \leq r},
\]
if the $\inf$ exists. If $\{ p, q \}$ have a greatest lower bound it is called the \emph{meet of $p$ and $q$} and is denoted by $p \wedge q$, that is,
\[
p \wedge q := \sup \setof{r\in\sP\mid r\leq p\ \text{and}\ r \leq q},
\]
if the $\sup$ exists. 
\end{defn}
The meet and join may not exist for a given pair $p, q\in \sP$, but if they exist they are unique.

\begin{defn}
\label{defn:lattice}
A \emph{lattice} is a poset in which every pair of elements $p,q\in \sP$ has a (unique) join $p \vee q$ and a (unique) meet $p \wedge q$. 
\end{defn}

Throughout most of this monograph we focus our attention to finite distributive lattices, i.e., lattices with finitely many elements for which the operations $\vee$ and $\wedge$ are distributive.

\begin{ex}
\label{ex:integerlattice}
Let $\N = \setof{0, 1, 2, \cdots}$ and $N$ be a positive integer.
The partial order $\poeN$ on $\N^N$ is defined by 
\[
x \poeN y\quad\text{if and only if}\quad x_n\leq y_n\ \text{for all $n=1,\ldots, N$.}
\]
We leave it to the reader to check that if $z = x\vee_\N y$, then $z_n = \max\{x_n,y_n\}$ for all $n=1,\ldots, N$, and if $z = x\wedge_\N y$, then $z_n = \min\{x_n,y_n\}$ for all $n=1,\ldots, N$.
Thus, $\poeN$ defines a lattice on $\N^N$.
\end{ex}

We employ two functorial relations between posets and lattices that we define below. 

\begin{defn}
    \label{defn:downset}
Let $(\sP,\leq)$ denote a finite poset and let $p\in \sP$. 
The \emph{downset} of $p$ is
\[
\sO(p) := \setdef{q\in\sP}{q\leq p}.
\]
The collection of all downsets in $\sP$ generates a finite distributive lattice denoted by $\sO(\sP)$ where the operations are $\wedge = \cap$ and $\vee = \cup$.
\end{defn}

\begin{defn}
    \label{defn:joinirreducible}
Let $\sL$ denote a finite distributive lattice.   
An element $p\in \sL$ is \emph{join-irreducible} if it has a unique immediate predecessor or equivalently, if $p=q\vee r$ implies that $p = q$ or $p=r$.
The collection of all join irreducible elements of $\sL$ is denoted by $\sJ^\vee(\sL)$. Since $\sJ^\vee(\sL)\subset \sL$, $\sJ^\vee(\sL)$ is a poset.
\end{defn}

A finite lattice always has a unique minimal element $\bzero$ and a unique maximal element $\bone$ and hence is a \emph{bounded lattice}.
We denote our lattices as $\bzero\bone$-lattices.
Recall that if $\sL$ and $\sL'$ are two $\bzero\bone$-lattices, then a $\bzero\bone$-lattice morphism $f\colon \sL\to\sL'$ satisfies
\[
f(L_0\vee L_1) = f(L_0)\vee f(L_1)\quad\text{and}\quad f(L_0\wedge L_1) = f(L_0)\wedge f(L_1)
\]
and preserves the minimal and maximal elements, i.e.,
\[
f(\bzero) = \bzero' \quad\text{and}\quad f(\bone) = \bone' .
\]

The following theorem (see \cite[Theorem 5.19]{davey:priestley}) relates finite distributive lattices with finite posets.

\begin{thm}[Birkhoff]
\label{thm:birkhoff}
Let $\sP$ and $\sQ$ be finite posets and let $\sL = \sO(\sP)$
and $\sK = \sO(\sQ)$.
Given a $\bzero\bone$-homorphism $f\colon \sL \to \sK$, there is an associated
order-preserving map $\phi_f\colon \sQ \to \sP$ defined by
\[
\phi_f(q) := \min\setdef{ p\in \sP}{q\in\sO(p)}\quad \text{for all $q\in \sQ$.}
\]

Given an order-preserving map $\phi\colon \sQ \to \sP$, there is an associated $\bzero\bone$-homorphism $f_\phi\colon \sL \to \sK$  defined by
\[
f_\phi(a) := \phi^{-1}(a)\quad\text{for all $a \in \sL$.}
\]
Equivalently,
\[
\phi(y) \in a\quad\text{if and only if}\quad \text{$y \in f_\phi(a)$ for all $a\in\sL$, $y\in\sQ$.}
\]
The maps $f\mapsto \phi_f$ and $\phi\mapsto f_\phi$ establish a one-to-one correspondence between $\bzero\bone$-homorphisms from $\sL$ to $\sK$ and order-preserving
maps from $\sQ$ to $\sP$.
Further,
\begin{enumerate}
    \item[(i)] $f$ is one-to one if and only if $\phi_f$ is onto,
    \item[(ii)] $f$ is onto if and only if $\phi_f$ is an order-embedding.
\end{enumerate}
\end{thm}
An important consequence of Theorem~\ref{thm:birkhoff} is that if $\sL$ is a finite distributive lattice and $\sP$ is a finite poset, then
\begin{enumerate}
    \item[(i)] $\sL \cong \sO(\sJ^\vee(\sL))$ as lattices, and
    \item[(ii)] $\sP \cong \sJ^\vee(\sO(\sP))$ as posets.
\end{enumerate}

\section{Cell Complexes}

We now discuss cell complexes and their associated chain complexes (for a more general discussion see \cite{lefschetz}). 

\begin{defn}
\label{defn:cell_complex}
A {\em cell complex} $\cZ = (\cZ,\preceq, \dim, \kappa)$ is a finite partially ordered set with two functions:
the {\em dimension} function
\[
\dim\colon \cZ\to \N,
\]
and the {\em incidence} function
\[
\kappa\colon \cZ\times\cZ \to \F
\]
where throughout this monograph we assume $\F$ is a field.
These functions  satisfy the following three conditions:
\begin{enumerate}
\item[(i)]  if $\zeta \prec \zeta'$, then $\dim \zeta < \dim \zeta'$;
\item[(ii)] for each $\zeta,\zeta' \in \cZ$, $\kappa(\zeta',\zeta) \neq 0\text{ implies } \zeta \prec \zeta' \text{ and } \dim \zeta' = \dim \zeta + 1 $; and
\item[(iii)] for each $\zeta, \zeta'' \in \cZ$, the sum 
\begin{equation}
\label{eq:incidence}
\sum_{\zeta'\in\cZ} \kappa(\zeta'',\zeta')\kappa(\zeta',\zeta)=0.
\end{equation}
\end{enumerate}

An element $\zeta \in \cZ$ is  a {\em cell} of {\em dimension} $k$ if $\dim \zeta = k$ and we denote the set of $k$-dimensional cells by
\[
\cZ^{(k)} := \setof{\zeta\in\cZ\mid \dim \zeta = k}.
\]
\end{defn}

As is made clear below, it is useful to consider a finite distributive $\bzero\bone$-lattice structure on a cell complex $\cZ$.

\begin{defn}
A cell complex $\cZ = (\cZ,\preceq, \dim, \kappa)$ is \emph{uniform} of dimension $N$ if there exists a positive integer $N$ such that
\begin{enumerate}
    \item if $\zeta\in\cZ$ is maximal with respect to $\preceq$, then $\dim(\zeta) = N$, and
    \item if $\zeta \in \cZ^{(N-1)}$, then there exist at most two cells $\mu_i\in \cZ^{(N)}$ such that $\zeta \preceq \mu_i$.
\end{enumerate}
Furthermore, if $\zeta\in\cZ$ is minimal with respect to $\preceq$, then $\dim(\zeta) = 0$.
\end{defn}

As indicated in the previous chapter we use the language of posets to describe the gradient-like behavior of dynamics.
To emphasize the difference between topology, which through a geometric realization is associated with a chain complex, and dynamics we make use of the following notation.

\begin{rem}
\label{rem:cellularlattice}
By definition a cell complex $\cZ$ is a poset $(\cZ,\preceq)$.
The lattice of downsets $\sO(\cZ)$ is a $\bzero\bone$-lattice whose elements are subsets of $\cZ$, the partial order is inclusion, and the lattice operations are $\wedge = \cap$ and $\vee = \cup$.
\end{rem}

\begin{defn}
\label{defn:closure}
Let $\zeta\in \cZ$. 
The \emph{closure} of $\zeta$ is $\cl(\zeta) := \sO(\zeta)$, i.e., the downset of $\zeta$.
Given $\cU\subset \cZ$, the  \emph{closure} of $\cU$ is $\cl(\cU):= \bigcup _{\zeta\in \cU}\sO(\zeta)\in \sO(\cZ)$ and $\cU$ is \emph{closed} if $\cU = \sO(\cU)$.    
\end{defn}

Recall \cite{lefschetz} that if $\cN\subset \cZ$ is closed, then $\cN$ is a cell complex.

\begin{defn}
\label{def:boundary_prime}
Let $\cZ$ be a uniform cell complex of dimension $N$.
The \emph{boundary} of $\cZ$ is 
\[
\bbdy(\cZ) := \cl\left( \setdef{ \zeta\in \cZ^{(N-1)}}{ \text{there exists a unique $\mu\in \cZ^{(N)}$ such that $\zeta\prec\mu$}} \right)
\]
\end{defn}

\begin{defn}
\label{defn:kchain}
Given a cell complex $\cZ$ and a field $\F$, the {\em $k$-chains} of $\cZ$ over $\F$ is the  vector space  $C_k(\cZ;\F)$ with basis $\cZ^{(k)}$, i.e.
\[
C_k(\cZ;\F) := \setdef{\sum_{\zeta_i\in\cZ^{(k)}}a_i\zeta_i}{a_i\in\F}.
\]
The associated {\em boundary operator} is the collection of linear maps $\partial_k\colon C_k(\cZ)\to C_{k-1}(\cZ)$
defined by
\[
\partial_k \zeta = \sum_{\zeta'\in C_{k-1}(\cZ)}\kappa(\zeta,\zeta')\zeta'.
\]
\end{defn}

\begin{defn}
\label{defn:subchainlattice}
Given a cell complex $\cZ$, the \emph{lattice of subchain complexes} over the field $\F$ is given by 
\[
\sSub(C_*(\cZ);\F) = \setof{C_*(\cN;\F) \mid \cN^{\text{closed}}\subset \cZ }
\]
where 
\[
C_*(\cN_0;\F) \wedge C_*(\cN_0;\F) := C_*(\cN_0\cap \cN_1;\F)
\]
and
\[
C_*(\cN_0;\F) \vee C_*(\cN_0;\F) := C_*(\cN_0\cup \cN_1;\F).
\]
\end{defn}
Observe that $C_*(\emptyset;\F)$ and $C_*(\cZ;\F)$ are minimal and maximal elements of $\sSub(C_*(\cZ);\F)$, and thus $\sSub(C_*(\cZ);\F)$ is a $\bzero\bone$-lattice.

\begin{defn}
\label{defn:ABlattice}  
Let $\cZ$ be a uniform cell complex of dimension $N$. 
A sublattice $\sN$ of $\sO(\cZ)$ is an \emph{AB-lattice}\footnote{As is indicated in the next section, using an appropriate geometrization an AB-lattice defines a lattice of attracting blocks.} if the following conditions are satisfied.
\begin{enumerate}
\item $\emptyset,\cZ\in\sN$.
\item If $\cN\in \sN\setminus\setof{\emptyset}$, then $\cN$ is a uniform cell complex of dimension $N$.    
\end{enumerate}
\end{defn}

The following result follows directly from the definitions of an \emph{AB-lattice} sublattice of $\cZ$ and the lattice of subchain complexes.

\begin{prop}
    \label{prop:sublatticeC(X)}
Let $\cZ$ be a uniform cell complex of dimension $N$.
Let $\sN$ be an AB-sublattice of $\sO(\cZ)$.
Then, 
\begin{align*}
    \iota\colon \sN &\to \sSub(C_*(\cZ);\F)  \\
    \cN &\mapsto \iota(\cN) := C_*(\cN;\F)
\end{align*}
is a $\bzero\bone$-lattice monomorphism.
\end{prop}

\section{Geometrization}

We now introduce the idea of geometrization.

\begin{defn}
\label{defn:geometrization}
Let $\cZ$ be a uniform cell complex of dimension $N$.
For $n\geq 1$ let $B^n$ be homeomorphic to the closed unit ball in $\R^n$ and set $B^0$ to be a single point.
A \emph{geometrization} of $\cZ$ in $\R^N$ consists of a collection of functions 
\[
\bG :=\setof{g_\zeta \colon B^n \to \R^N\mid \zeta \in \cZ^{(n)},\ n=0, \ldots, N}
\]
that satisfies the following conditions:
\begin{enumerate}
    \item[(i)] each $g_\zeta$ is a homeomorphism onto its image $\bg(\zeta) := g_\zeta(B^n)$;
    \item[(ii)] if $\zeta_0\neq \zeta_1$ and $\dim(\zeta_i)=n_i$, then $g_{\zeta_0}(\Int(B^{n_0})) \cap g_{\zeta_1}(\Int(B^{n_1})) = \emptyset$; and
    \item[(iii)] for any $\zeta'\in \cZ^{(n)}
    $\begin{equation*}
    \bigcup_{\zeta\prec\zeta'}\bg(\zeta) = g_{\zeta'}(\bdy(B^n))
    \end{equation*}
    where $\bdy(B^n)$ denotes the topological boundary of $B^n$.
\end{enumerate}
The set $\bg(\zeta)\subset \R^N$ is called the \emph{geometric realization} of $\zeta\in\cZ$ under $\bG$ and 
\[
X:= \bigcup_{\zeta\in \cZ}\bg(\zeta) \subset \R^N
\]
denotes the geometric realization of $\cZ$.
\end{defn}

We now relate geometrizations with vector fields.
Let $\cZ$ be a uniform cell complex of dimension $N$ and let $\bG$ be a geometrization  of $\cZ$ in $\R^N$.
Consider $\zeta \in \bbdy(\cZ)$ such that $\dim(\zeta) = N-1$.
By definition, this implies that there exists a unique $\mu\in \cZ^{(N)}$ such that $\zeta\prec \mu$.
Assume that $\bg(\zeta)\subset \R^N$ is a smooth $(N-1)$-dimensional manifold with boundary.
Then, at $x\in g_\zeta(\Int(B^{N-1}))$ there are two choices of normal vector to $\bg(\zeta)$.
We denote the normal vector pointing inward to $g_\mu(B^N)$ by  $z_\zeta(x)$ and refer to it as the \emph{inward pointing normal vector}.

\begin{defn}
\label{defn:aligned}
Let $\cZ$ be a uniform cell complex of dimension $N$, and let $\bG$ be a geometrization of $\cZ$ in $\R^N$. 
Consider a Lipschitz continuous vector field $f \colon X \to \R^N$ generating a flow $\varphi$ via $\dot{x} = f(x)$. 
Let $\cN \subset \cZ$ be a uniform subcomplex of dimension $N$, and let $\zeta \in \bdy(\cN)$ with $\dim(\zeta) = N-1$. 

The flow $\varphi$ is said to be \emph{inward-pointing} at $x \in g_\zeta(B^{N-1})$ if there exists $t > 0$ such that $\varphi((0,t),x) \subset \Int(\bg(\cN))$. 
If $\varphi$ is inward-pointing at every $x \in g_\zeta(B^{N-1})$, we say that $\varphi$ is inward-pointing at $\zeta$. 
Finally, if $\varphi$ is inward-pointing at every $\zeta \in \bdy(\cN)^{(N-1)}$, then we say that \emph{$\bG$ is aligned with the vector field $f$ over $\cN$}.
\end{defn}

\begin{thm}
\label{thm:inward}
Let $\cZ$ be a uniform cell complex of dimension $N$, and let $\bG$ be a geometrization of $\cZ$ in $\R^N$. 
Consider a Lipschitz continuous vector field $f \colon X \to \R^N$ generating a flow $\varphi$ via $\dot{x} = f(x)$. 
Let $\cN \subset \cZ$ be a uniform subcomplex of dimension $N$ and define $L =\bg\left(\cN\right) \subset X$.

The following statements hold:
\begin{enumerate}
    \item Let $\zeta \in \bdy(\cN)^{(N-1)}$ and $x \in g_\zeta(\Int(B^{N-1}))$. If $\langle f(x), z_\zeta(x) \rangle > 0$, then $\varphi$ is inward-pointing at $x$.
    \item If $\varphi$ is inward-pointing at $x$ for all $x \in g_\zeta(\Int(B^{N-1}))$ and all $\zeta \in \bdy(\cN)^{(N-1)}$, then $\varphi(t, L) \subset L$ for all $t > 0$.
\end{enumerate}
\end{thm}

\begin{proof}
The item (1) follows from classical results [Theorem 9.24, \cite{Lee2012}]. To prove (2), suppose for the sake of contradiction that there exists $x \in L$ and $t>0$ such that $\varphi(t,x) \notin L$. Without loss of generality, assume that $x \in \bdy(L)$ and $$\inf\setdef{s>0}{\varphi(s,x)\notin L}=0.$$ By assumption, $x \notin g_\zeta(\Int(B^{N-1})$ for all $\zeta \in \bdy(\cN)^{(N-1)}$. Since $\cN$ is uniform of dimension $N$, there exist $\zeta \in \bdy(\cN)^{(N-1)}$ such that $x \in \bg(\zeta)$. By continuity with respect to initial conditions, for any $t>0$ there exists a neighborhood $V_t$ of $x$ and $y \in V_t \cap g_\zeta(\Int(B^{N-1})) \neq \emptyset$ such that $\varphi(t,y) \notin L$, which contradicts (1). 
\end{proof}

In light of Theorem~\ref{thm:inward}, to conclude that $\bG$ is aligned with $f$ over $\cN$, it is sufficient to verify that trajectories starting at the boundary of $\bg(\zeta)$ for $\zeta \in \bdy(\cN)^{(N-1)}$ do not stay on the boundary. That is equivalent to verifying that $\varphi$ is inward-pointing. In the event that $f$ is tangent to $\bg(\zeta)$, the condition may be verified explicitly. This leads us to the following corollary. 

\begin{cor}
\label{cor:inward}
Let $\cZ$ be a uniform cell complex of dimension $N$, and let $\bG$ be a geometrization of $\cZ$ in $\R^N$. 
Consider a Lipschitz continuous vector field $f \colon X \to \R^N$ generating a flow $\varphi$ via $\dot{x} = f(x)$. 
Let $\cN \subset \cZ$ be a uniform subcomplex of dimension $N$.

If, for all $\zeta \in \bdy(\cN)^{(N-1)}$, the following conditions hold:
\begin{enumerate}
    \item $\langle f(x), z_\zeta(x) \rangle > 0$ for all $x \in g_\zeta(\Int(B^{N-1}))$, and
    \item $\varphi$ is inward-pointing at every $x \in g_\zeta(\bdy(B^{N-1}))$,
\end{enumerate}
then $\bG$ is aligned with $f$ over $\cN$.
\end{cor}
We summarize the results in the following theorem. 
\begin{thm}
\label{thm:LipAttBlock}
Let $f$ be a Lipschitz continuous vector field on $\R^N$.
Let $\varphi$ be the flow generated by $\dot{x} = f(x)$.
Let $\cZ$ be a uniform cell complex of dimension $N$. 
Let $\cN\subset\cZ$ be a uniform sub-complex of dimension $N$.
Let $\bG$ be a geometrization of $\cZ$ in $\R^N$ that is aligned with $f$ over $\cN$.
Then, $L \coloneqq \bigcup_{\zeta\in\cN}\bg(\zeta)$ is an attracting block for $\varphi$.
\end{thm}

\begin{rem}
\label{rem:O3}
Assume that a geometrization $\bG$ is aligned with a vector field $f$ over $\cN$.
Condition (2) of Definition~\ref{defn:aligned} implies that there exists $\epsilon_0>0$ such that if $f_0\colon X\to \R^N$ is a Lipschitz continuous vector field and $\sup_{x\in X}\|f(x) -f_0(x)\| < \epsilon$, then $\bG$ is aligned with a vector field $f_0$ over $\cN$.
Therefore, the fact that we use geometrizations that are aligned with a vector field to transfer information from the combinatorial setting to that of flows implies that a weak form of objective {\bf O3} as described in Chapter~\ref{sec:prelude} is satisfied, e.g., we know that the results are true for an open set of parameter values, but we have not explicitly stated what that set is.
\end{rem}

\begin{cor}
\label{cor:existenceAttractorLattice}
Let $\cZ$ be cell complex that is uniform of dimension $N$.
Let $\sL$ be a finite distributive $\bzero\bone$-lattice.
Let $\lambda\colon \sL \to \sO(\cZ)$ be a $\bzero\bone$-lattice monomorphism such that $\sN = \lambda(\sL)$ is an AB-lattice sublattice of $\sO(\cZ)$.
Let $\bG$ be a geometrization of $\cZ$ in $\R^N$ such that  $\bG$ is aligned with $f$ for all  $\cN\in \sN$.
Then,
\[
\setdef{\bigcup_{\zeta\in\lambda(L)}\bg(\zeta)}{L\in\sL}
\]
is a lattice of attracting blocks for $\varphi$.
\end{cor}

In the context of this monograph it is hard to overstate the importance of Corollary~\ref{cor:existenceAttractorLattice}; it is the conduit by which combinatorial/homological information is translated into information about continuous dynamical systems.
To emphasize this we rephrase the hypothesis of Corollary~\ref{cor:existenceAttractorLattice} and draw statements about nonlinear dynamics.

\begin{thm}
\label{thm:dynamics}
Fix a cell complex $\cZ$ that is uniform of dimension $N$, a $\bzero\bone$-lattice monomorphism $\lambda\colon \sL \to \sO(\cZ)$ such that $\sN = \lambda(\sL)$ is an AB-lattice sublattice of $\sO(\cZ)$, and a geometrization $\bG$ of $\cZ$ in $\R^N$.
Consider $\dot{x} = f(x)$, where $x\in\R^N$ and $f$ is Lipschitz continuous over $X:= \lambda(\bone)$.  
If $\bG$ is aligned with $f$ for all  $\cN\in \sN$, then the following statements are true.
\begin{enumerate}
\item Let $\varphi$ denote the flow associated with $\dot{x} = f(x)$. 
Then, $X$ is a trapping region for $\varphi$, i.e., $\varphi\colon [0,\infty)\times X \to X$.
\item $(\sJ^\vee(\sL),\leq)$ is a Morse decomposition of $\varphi$ restricted to $X$.
\item For $p\in \sJ^\vee(\sL)$ set 
    \[
    M(p) := \Inv\left(\bg(\lambda(p))\setminus \bg(\lambda(p^<)) ;\varphi \right).
    \]
    Then, 
    \[
    \sQ := \setdef{p\in \sJ^\vee(\sL)}{M(p)\neq \emptyset} 
    \]
    with ordering $\leq$ is a Morse representation of $\varphi$.
\item For $p\in \sJ^\vee(\sL)$, the homology Conley index of $M(p)$ is
    \[
    CH_*(p) = H_*(\bg(\lambda(p)), \bg(\lambda(p^<))).
    \]
    If $p\in \sJ^\vee(\sL)\setminus \sQ$, then $CH_*(p)=0$.
\end{enumerate}
\end{thm}

No proof of Theorem~\ref{thm:dynamics} is necessary as it is, for the most part, a restatement of classical results. 
Theorem~\ref{thm:dynamics} item (1) follows from Theorem~\ref{thm:LipAttBlock}.
Theorem~\ref{thm:dynamics} item (2) and Theorem~\ref{thm:dynamics} item (3) follow from \cite{kalies:mischaikow:vandervorst:18}.
Theorem~\ref{thm:dynamics} item (4) follows from the realization that $(\bg(\lambda(p)), \bg(\lambda(p^<)))$ is an index pair.

\begin{rem}
\label{rem:Fconnectionmatrix}
The fundamental theorem of Franzosa \cite[Theorem 3.8]{franzosa:89} states that given a Morse decomposition there exists a connection matrix.
Franzosa's theorem is constructive and 
Theorem~\ref{thm:dynamics} items (1) and (4) provide the necessary input.
This implies that Theorem~\ref{thm:dynamics} determines a set of connection matrices for $\varphi$ given the Morse decomposition $(\sJ^\vee(\sL),\leq)$.
For this monograph we apply the algorithm  for computing connection matrices  presented in \cite{harker:mischaikow:spendlove} and implemented in \cite{harker:templates:21} using  input derived from $(\sJ^\vee(\sL),\leq)$.
\end{rem}

\part{Combinatorial/Homological Dynamics}
\label{part:II}

\chapter{Wall Labelings and Rook Fields}
\label{sec:RookFields}

As indicated in Chapter~\ref{sec:prelude} our combinatorial representation and homological characterization is based on abstract cubical complexes.
These are described in Section~\ref{section:cubical}.
Wall labelings and rook fields, defined in Section~\ref{sec:rook}, are used to translate regulatory networks to combinatorial vector fields defined on cubical complexes.
In Section~\ref{sec:rookproperties} we identify essential characteristics of rook fields.

\section{Cubical Complexes}
\label{section:cubical}

For $n=1,\ldots, N$, let $K(n)$ be a positive integer, set $I_n := \setof{0,\ldots, K(n)+1} \subset \Z$ and define 
\begin{equation}
\label{eq:II}
\I := \prod_{n=1}^N I_n\subset \Z^N.
\end{equation}
Throughout this manuscript
\[
\bv\in \I \quad \text{and} \quad \bq,\bw\in \setof{0,1}^N.
\]
Set
\[
\dim(\bw) = \sum_{n=1}^N \bw_n.
\]
Particular elements of interest in $\setof{0,1}^N$ are ${\bf 1}=(1,\ldots, 1)$,  ${\bf 0}=(0,\ldots, 0)$, ${\bf 1}^{(n)}=(1,\ldots,1,0,1\ldots 1)$ where  the $n$-th coordinate is $0$,  and ${\bf 0}^{(n)} = {\bf 1}-{\bf 1}^{(n)}$.

\begin{defn}
\label{defn:Xcomplex}
The $N$-dimensional \emph{cubical cell complex} generated by  $\I$ is denoted by $\cX(\I) = (\cX,\preceq,\dim,\kappa)$ and defined as follows. The set of cells is
defined by
\[
\cX = \setof{ \xi = [\bv,\bw] \mid \bv\in \I, \bw\in \setof{0,1}^N, \text{ and } \bv + \bw \in \I}.
\]
The \emph{dimension} of a cell is given by
\[
\dim([\bv,\bw]) = \dim(\bw)
\]
and we set $\cX^{(n)} := \setof{\xi\in\cX \mid \dim(\xi) = n}$.

The partial order $\preceq$, called the \emph{face relation}, indicates which cells are faces of other cells and is defined as
\begin{equation}\label{eq:is_a_face}
    [\bv,\bw]\preceq[\bv',\bw']\ \text{if and only if $\bv = \bv'+\bq$ and $\bw +\bq \poeN \bw'$}. 
\end{equation}

The \emph{incidence number} $\kappa$ is defined by
\[
\kappa\left( \left[\bv,\bw \right], \left[ \bv,\bw - {\bf 0}^{(i)}\right] \right) = -1
\quad
\text{and}
\quad
\kappa\left( \left[\bv,\bw \right], \left[ \bv +{\bf 0}^{(i)} ,\bw - {\bf 0}^{(i)}\right] \right) =1
\]
under the assumption that $\bw_i = 1$, and $\kappa(\xi, \xi') = 0$ for all other pairs $(\xi, \xi') \in \cX \times \cX$. 
\end{defn}

\begin{ex}
\label{ex:cubicalcomplex}
As emphasized in Chapter~\ref{sec:prelude} the cubical complex $\cX(\I)$ is a purely combinatorial object. 
However, for the purposes of providing intuition it is useful to use diagrams to organize this combinatorial data. 
In particular, consider $N=2$ and assume $K(n)=2$ for $n=1,2$. 
Then $\I = \prod_{n=1}^2 \setof{0,1,2,3}$.
The elements of $\cX$ can be visualized as in Figure~\ref{fig:position_vectors}.    
\end{ex}

We leave it to the reader to check that if $\xi = [\bv,\bw]\in \bbdy(\cX)$ (see Definition~\ref{def:boundary_prime}), then there exists $n$ such that $\bv_n=0$ or $\bv_n = K(n)+1$, and $\bw_n =0$. 

\begin{defn}
    \label{defn:topstar}
    Given a $N$-dimensional cubical cell complex $\cX(\I)$, we refer to a cell $\mu \in \cX^{(N)}$ as a \emph{top cell}. The \emph{top star} of $\xi \in \cX$  is given by
    \begin{equation}
        \label{eq:topstar}
        \Top_\cX(\xi) := \setof{\mu\in \cX^{(N)}\mid \xi\preceq \mu}.
    \end{equation}
\end{defn}
\begin{defn}
    \label{defn:JeJi}
    Let $\cX(\I)$ be a $N$-dimensional cubical cell complex. The \emph{essential} and \emph{inessential} directions of $\xi = [\bv,\bw] \in \cX$ are given by
    \begin{equation}
    \label{eq:JeJi}
    J_e(\xi) = J_e([\bv,\bw]) \coloneqq \setof{n\mid \bw_n=1} \text{ and } J_i(\xi) = J_i([\bv,\bw]) \coloneqq \setof{n\mid \bw_n=0},
    \end{equation}
    respectively.
\end{defn}

The following proposition follows directly from the definition of essential and inessential directions.

\begin{prop}
\label{prop:JtauJsigma}
If $\xi \preceq \xi'$, then 
\begin{enumerate}
\item[(i)] $J_e(\xi)\subseteq J_e(\xi')$,
\item[(ii)] $J_i(\xi')\subseteq J_i(\xi)$, and
\item[(iii)] $J_i(\xi)\setminus J_i(\xi') = J_i(\xi)\cap J_e(\xi')$.
\end{enumerate}
\end{prop}

\begin{defn}
\label{defn:Ex}
Let $\xi \preceq \xi'$. 
The \emph{extensions of $\xi$ in $\xi'$} are defined to be
\begin{equation}
    \label{eq:Ex}
    \Ex(\xi,\xi') := J_i(\xi)\cap J_e(\xi') .
\end{equation}
\end{defn}

\begin{defn}
\label{defn:rpvector}
The \emph{relative position vector} of 
faces in $\cX$ is given by the function
$
p\colon \setof{(\xi,\xi')\in \cX\times \cX \mid \xi \preceq\xi'} \to \setof{0,\pm 1}^N
$
defined componentwise by
\[
p_n(\xi,\xi') = p_n([\bv,\bw],[\bv',\bw']) := (-1)^{\bv_n-\bv'_n}\left(\bw_n-\bw'_n \right)
\]
\end{defn}

\begin{figure}
\centering
    \begin{tikzpicture}[scale=0.25]
    \fill[gray!10!white] (0,0) rectangle (24,24);
    \foreach \i in {0,...,3}{
        \draw(8*\i,-1) node{$\i$}; 
        \draw(-1,8*\i) node{$\i$}; 
       }
    \draw[step=8cm,black, thick] (0,0) grid (24,24);
    \foreach \i in {0,...,3}{
      \foreach \j in {0,...,3}{
        \draw[black, fill=black] (8*\i,8*\j) circle (2ex);
      }
    }
    \foreach \i in {0,...,2}{
      \draw[blue, thick] (4+8*\i,1) node{$\downarrow$}; 
      \draw[blue, thick] (4+8*\i,23) node{$\uparrow$}; 
      \draw[blue, thick] (1, 4+8*\i) node{$\leftarrow$}; 
      \draw[blue, thick] (23.0,4+8*\i) node{$\rightarrow$}; 
    }
    \foreach \i in {0,...,2}{
        \draw[blue, thick] (4+8*\i,6.9) node{$\uparrow$};
        \draw[blue, thick] (4+8*\i,14.7) node{$\uparrow$};
        \draw[blue, thick] (4+8*\i,16.9) node{$\downarrow$};
        \draw[blue, thick] (4+8*\i,8.9) node{$\downarrow$};
      \draw[blue, thick] (7.0,4+8*\i) node{$\rightarrow$};
      \draw[blue, thick] (9,4+8*\i) node{$\leftarrow$};
      \draw[blue, thick] (15.0,4+8*\i) node{$\rightarrow$};
      \draw[blue, thick] (17,4+8*\i) node{$\leftarrow$};
    }
     \foreach \i in {0,...,2}{
      \draw[blue, thick] (1,7+8*\i) node{$\nwarrow$};
        \draw[blue, thick] (9,7+8*\i) node{$\nwarrow$};
        \draw[blue, thick] (17,7+8*\i) node{$\nwarrow$};
        \draw[blue, thick] (7,7+8*\i) node{$\nearrow$};
        \draw[blue, thick] (15,7+8*\i) node{$\nearrow$};
        \draw[blue, thick] (23,7+8*\i) node{$\nearrow$};
        \draw[blue, thick] (1,1+8*\i) node{$\swarrow$};
        \draw[blue, thick] (9,1+8*\i) node{$\swarrow$};
        \draw[blue, thick] (17,1+8*\i) node{$\swarrow$};
        \draw[blue, thick] (7,1+8*\i) node{$\searrow$};
        \draw[blue, thick] (15,1+8*\i) node{$\searrow$};
        \draw[blue, thick] (23,1+8*\i) node{$\searrow$};
    }
    \end{tikzpicture}
\caption{
A visualization of the abstract cubical cell complex $\cX(\I)$ where $\I = \setof{0,1,2,3}^2$. 
The vertices and two-dimensional cells of $\cX$ can be expressed as $\defcell{n_1}{n_2}{0}{0}$, for  $n_i\in \setof{0,1,2,3}$, and $\defcell{n_1}{n_2}{1}{1}$, for  $n_i\in \setof{0,1,2}$, respectively.
The edges that are visualized as horizontal and vertical segments  take the form $\defcell{n_1}{n_2}{1}{0}$, where $n_1\in \setof{0,1,2}$ and $n_2\in \setof{0,1,2,3}$,  and $\defcell{n_1}{n_2}{0}{1}$, where $n_1\in \setof{0,1,2,3}$ and $n_2\in \setof{0,1,2}$, respectively.
The blue arrows provide a visualization of the relative position vectors $p(\xi,\mu)$ for $\mu\in \cX^{(2)}$ (see Definition~\ref{defn:rpvector}). 
}
\label{fig:position_vectors}
\end{figure}
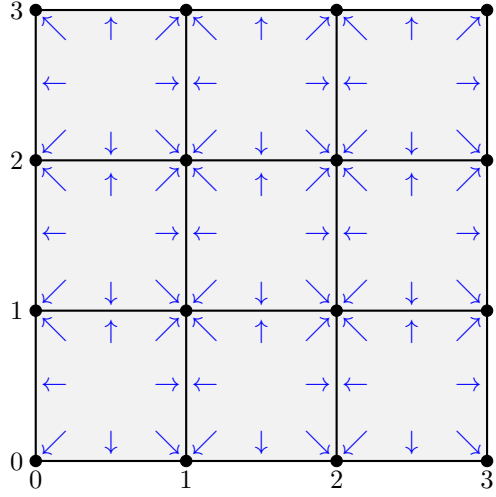

We remind the reader that $\cX$ is an abstract cubical cell complex, and hence, the position vector is a discrete function with no geometric meaning.
However, for the purposes of developing intuition with respect to its use in the sections that follow the reader may find it useful to recognize that the relative position vector ``points''  from the higher dimensional cell to the lower dimensional cell as indicated in Figure~\ref{fig:position_vectors}. 

We leave to the reader to check the proof the following proposition.
\begin{prop}\label{prop:position_trio}
    If $\xi\preceq_\cX\xi'\preceq_\cX\xi''$ then $p_n(\xi,\xi')=p_n(\xi,\xi'')$ for all $n\in\Ex(\xi,\xi')$. 
\end{prop}

Given two elements $\xi$ and $\xi'$ in the poset $(\cX, \preceq)$ their meet exists if and only if they have a common face and their join exists if and only if they have a common co-face. The following proposition provides explicit formulas for the meet and joint of $\xi$ and $\xi'$ when they exist.

\begin{prop}
\label{prop:MeetJoinFormula}
Let $\cX(\I) = (\cX,\preceq,\dim,\kappa)$ be a cubical cell complex and consider any two cells $\xi =[\bv,\bw] \in \cX$ and $\xi'=[\bv',\bw'] \in \cX$.
\begin{enumerate}
    \item The join of $\xi$ and $\xi'$ exists if and only if $[\bv \wedge_\N \bv',\bzero] \preceq \xi,\xi'$. In that case, it is given by 
    \[
       \xi \vee_\preceq \xi' = [\bv \wedge_\N \bv',(\bv+\bw)\vee_\N(\bv'+\bw')-(\bv\wedge_\N\bv')].
    \]
    \item The meet of $\xi$ and $\xi'$ exists if and only if $\xi,\xi'\preceq [\bv\vee_\N\bv',\bone]$. In that case, it is given by 
    \[
       \xi \wedge_\preceq \xi' = [\bv\vee_\N\bv',(\bv+\bw)\wedge_\N(\bv'+\bw')-(\bv\vee_\N\bv')]. 
    \]
\end{enumerate}
\end{prop}

\begin{proof}
We prove (1) and leave (2) to the reader.

Let $\sigma = [\bv'', \bw''] \in \cX$ be the join of $\xi$ and $\xi'$. By Definition~\ref{defn:meet_join}, $\sigma \preceq \xi, \xi'$, meaning there exist $\bq, \bq' \in \{0,1\}^N$ such that 
    \begin{align*}
        \bv'' & = \bv + \bq, & \bw'' + \bq & \leq_\N \bw, \\ 
        \bv'' & = \bv' + \bq', & \bw'' + \bq' & \leq_\N \bw'.
    \end{align*}
    Since $\bv'' = \bv + \bq = \bv' + \bq'$, we have $\bv - \bv' = \bq' - \bq$, implying $|\bv_n - \bv'_n| \leq 1$ for each $n \in \{1, \ldots, N\}$. Hence, $\bv \wedge_\N \bv' - \bv \in \{0,1\}^N$ and 
    \[
        \bv \wedge_\N \bv' = \bv + (\bv \wedge_\N \bv' - \bv).
    \]
    If we show that $\bv \wedge_\N \bv' - \bv \leq_\N \bw$, then $[\bv \wedge_\N \bv', \bzero] \preceq \xi$. Assume for contradiction that there exists an index $n$ such that 
    \[
        (\bv \wedge_\N \bv' - \bv)_n > \bw_n.
    \]
    Since $(\bv \wedge_\N \bv' - \bv)_n$ and $\bw_n$ are in $\{0,1\}$, it follows that 
    \[
        (\bv \wedge_\N \bv' - \bv)_n = 1 \quad \text{and} \quad \bw_n = 0.
    \]
    Given $\bw_n = 0$, we have $\bq_n = 0$, so $\bv''_n = \bv_n$. Thus, $\bv_n = \bv'_n + \bq'_n$, implying $\bv'_n \leq \bv_n$ and $(\bv \wedge_\N \bv')_n = \bv_n$. This contradicts $(\bv \wedge_\N \bv' - \bv)_n = 0$. The same reasoning implies that $[\bv\wedge_\N\bv',\bzero]$ is a face of $\xi'$. 

    To prove the converse, assume $[\bv \wedge_\N \bv', \bzero] \preceq \xi, \xi'$, implying 
    \[
        S = \{ \xi'' \in \cX : \xi'' \preceq \xi, \xi'' \preceq \xi' \} \neq \emptyset.
    \]
    Consider $\bw''$ that maximizes the dimension of $\sigma = [\bv \wedge_\N \bv', \bw''] \in S$. Then there exist $\bq, \bq' \in \{0,1\}^N$ such that 
    \begin{align*}
        \bv \wedge_\N \bv' & = \bv + \bq, & \bw'' + \bq & \leq_\N \bw, \\ 
        \bv \wedge_\N \bv' & = \bv' + \bq', & \bw'' + \bq' & \leq_\N \bw'.
    \end{align*}
    Since $\bw'' \leq_\N \bw - \bq$ and $\bw'' \leq_\N \bw' - \bq'$, we have 
    \[
        \bw'' \leq (\bw - \bq) \vee_\N (\bw' - \bq').
    \]
    The dimension of $\bw''$ (see Definition~\ref{defn:Xcomplex}) is maximized when 
    \begin{align*}
        \bw'' & = (\bw - \bq) \vee_\N (\bw' - \bq') \\
              & = (\bv + \bw - \bv \wedge_\N \bv') \vee_\N (\bv' + \bw' - \bv \wedge_\N \bv') \\
              & = (\bv + \bw) \vee_\N (\bv' + \bw') - \bv \wedge_\N \bv'.
    \end{align*}
    To show that the $\sigma$ provided is the join of $\xi$ and $\xi'$, consider $\alpha = [\bv'', \bw''] \in S$. Then there exist $\bq, \bq' \in \{0,1\}^N$ such that 
    \begin{align*}
        \bv'' & = \bv + \bq, & \bw'' + \bq & \leq_\N \bw, \\ 
        \bv'' & = \bv' + \bq', & \bw'' + \bq' & \leq_\N \bw'.
    \end{align*}
    Thus, $\bv'' = \bv \wedge_\N \bv' + \bq \vee_\N \bq'$, and we need to prove that
    \[
        \bw'' + \bq \vee_\N \bq' \leq_\N (\bv + \bw) \vee_\N (\bv' + \bw') - \bv''.
    \]
    Observe that $\bv + \bw - \bv'' = \bw - \bq \geq \bw''$ and similarly, $\bw' - \bq' \geq \bw''$, so 
    \[
        (\bv + \bw) \vee_\N (\bv' + \bw') - \bv'' = (\bv + \bw) \vee_\N (\bv' + \bw') - (\bv \wedge_\N\bv') - \bq\vee_\N\bq' \geq \bw'',
    \]
    implying $\alpha \preceq \sigma$.
\end{proof}

Recall that the extension of a cell in another is given by Definition~\ref{defn:Ex}. 
\begin{prop}
    \label{prop:sameposition-meetjoin}
    Let $\xi,\xi' \in \cX$ be incomparable cells, i.e., neither $\xi \preceq \xi'$ nor $\xi' \preceq \xi$. If $\xi \vee_\preceq \xi'$ and $\xi \wedge_\preceq \xi'$ are well-defined, then
    \[
    p_n(\xi',\xi \vee_\preceq \xi') = p_n(\xi \wedge_\preceq \xi',\xi), \quad \text{for all}\ n \in \Ex(\xi',\xi \vee_\preceq \xi').
    \]
\end{prop}
\begin{proof}
    Given $\xi=[\bv,\bw]$ and $\xi'=[\bv',\bw']$, it is enough to show that
    \[
        p_n(\xi',\xi \vee_\preceq \xi') \cdot p_n(\xi \wedge_\preceq \xi',\xi) = 1, \quad \forall n \in \Ex(\xi',\xi \vee_\preceq \xi').
    \]
    Let $n \in \Ex(\xi',\xi \vee_\preceq \xi')$ and denote the vertices of $\xi \wedge_\preceq \xi'$ and $\xi \vee_\preceq \xi'$ respectively by
    \[
        \bv_n^\wedge  = \min\{\bv_n,\bv_n'\}, \quad
        \bv_n^\vee  = \max\{\bv_n,\bv_n'\}, 
    \]
    as in Proposition~\ref{prop:MeetJoinFormula}. 
    
    Then, by Definition~\ref{defn:rpvector} and Proposition~\ref{prop:JtauJsigma}, 
    \[
        p_n(\xi',\xi \vee_\preceq \xi') =  (-1)^{\bv_n' - \bv_n^\vee (0-1)} = (-1)^{\bv_n' - \bv_n^\vee+1}.
    \]
    Similarly, since $\Ex(\xi',\xi \vee_\preceq \xi') \subseteq \Ex(\xi \wedge_\preceq \xi',\xi)$ by Proposition~\ref{prop:MeetJoinFormula}, it follows that
    \[
    p_n(\xi \wedge_\preceq \xi',\xi) =  (-1)^{\bv_n^\wedge - \bv_n} (0-1) = (-1)^{\bv_n^\wedge - \bv_n+1}.
    \]
    Thus,
    \[
         p_n(\xi',\xi \vee_\preceq \xi') \cdot p_n(\xi \wedge_\preceq \xi',\xi) = (-1)^{(\bv_n' - \bv_n^\vee+(\bv_n^\wedge - \bv_n)} = (-1)^{c(n)},
    \]
    where $c(n):=(\bv_n' - \bv_n^\wedge)+(\bv_n^\vee - \bv_n)$.
    Again, by Proposition~\ref{prop:MeetJoinFormula}, $c(n)=0$ or $c(n)=2 ( \bv'-\bv)$, and the result follows.
\end{proof}

\section{Wall Labeling and Rook Fields}
\label{sec:rook}

Our characterization of combinatorial dynamics on an $N$-dimensional cubical cell complex $\cX$ is based on information organized using top cells $\cX^{(N)}$. 

\begin{defn}
\label{defn:wall}
Given an $N$-dimensional cubical cell complex $\cX$ the set of \emph{top pairs} is given by
\[
\TP(\cX) : = \setof{(\xi, \mu) \in \cX \times \cX^{(N)} \mid \xi \preceq \mu}.
\]
Of particular interest is the subcollection where $\xi \in \cX^{(N-1)}$.
The set of \emph{walls} of $\cX$ is 
\[
W(\cX) := \setof{(\xi, \mu) \in \cX^{(N-1)} \times \cX^{(N)} \mid \xi \prec \mu}.
\]
A wall $(\xi, \mu)\in W(\cX)$ is called an \emph{$n$-wall} if $J_i(\xi) = \setof{n}$.
\end{defn}

In a slight misuse of language we say that $\xi\in\cX^{(N-1)}$ is a wall of $\mu\in\cX^{(N)}$ to indicate that $(\xi, \mu)$ is a wall.

A top cell $\mu = [\bv, {\bf 1}] \in \cX^{(N)}$ has two $n$-walls
for each $n$,
\[
\mu^-_n = [\bv, {\bf 1}^{(n)}]\quad \text{and}\quad \mu^+_n = [\bv + {\bf 0}^{(n)}, {\bf 1}^{(n)}],
\] 
called the \emph{left} and \emph{right $n$-walls} of $\mu$, respectively.
Two distinct top cells $\mu,\mu'\in\cX^{(N)}$ are \emph{$n$-adjacent} if $\mu$ and $\mu'$ share an \emph{$n$-wall} $\xi$, i.e., $(\xi, \mu)$ and $(\xi, \mu')$ are $n$-walls.

Let $\xi_0 \in\cX$. 
The \emph{walls} of $\xi_0\in\cX$ are
\begin{equation*}
    W(\xi_0) := \setof{(\xi, \mu) \in W(\cX) \mid  \xi_0 \preceq \xi}.
\end{equation*}
We leave the proof of the following proposition to the reader.
\begin{prop}
\label{prop:mu*-general}
Consider an $N$-dimensional cubical complex $\cX$. Let $\xi \in \cX$ with $\dim(\xi) < N$. 
Fix $\mu \in \cX$ such that $\xi \prec \mu$.
Then for each $n\in \Ex(\xi,\mu)$ there exists a unique face $\mu_n^*(\xi,\mu) \in \cX$ of $\mu$ such that
\[
\xi \preceq \mu_n^*(\xi,\mu) \prec \mu \text{ and } \Ex(\mu_n^*(\xi,\mu),\mu)=\setof{n}.
\]
\end{prop}
If $\mu \in \Top_\cX(\xi)$, then Proposition~\ref{prop:mu*-general} can be reformulated as follows.
\begin{prop}
\label{prop:mu*}
Consider an $N$-dimensional cubical complex $\cX$. Let $\xi \in \cX$ with $\dim(\xi) < N$. 
Fix $\mu \in \Top_\cX(\xi)$. 
Then for each $n\in J_i(\xi)$ there exists a unique $n$-wall $\mu_n^*(\xi,\mu)$ of $\mu$ such that
\[
\xi \preceq \mu_n^*(\xi,\mu) \prec \mu.
\]
\end{prop}
\begin{lem}
\label{lem:nwallJe}
Consider an $n$-wall $(\mu_n,\mu) \in W(\xi)$.
Let $\xi\in\cX$.
If $n\in J_e(\xi)$, then $\xi \not\preceq \mu_n$.
\end{lem}
\begin{proof}
By Proposition~\ref{prop:JtauJsigma}, if $\xi\preceq \mu_n$, then $J_e(\xi)\subseteq J_e(\mu_n)$.
But, by assumption $n\in J_e(\xi)$, and by definition $n\in J_i(\mu_n)$, a contradiction.
\end{proof}
We recall that $\Ex(\xi,\xi')=J_i(\xi)\cap J_e(\xi')$ by \eqref{eq:Ex}.  
\begin{prop}
\label{prop:equal_position}
Let $\xi,\xi'\in \cX$ be such that $\xi\preceq\xi'$, and $\Ex(\xi,\xi')=\setof{n}$. 
If $(\mu_n,\mu) \in W(\xi)$ is a $n$-wall of $\xi$ with $\mu \in \Top_\cX(\xi')$, then $p_n(\xi,\xi')=p_n(\mu_n,\mu)$.
\end{prop}
\begin{proof}  
We leave it to the reader to check that if $\xi' = \mu$, then the result follows directly from Definition~\ref{defn:rpvector}.
By Lemma~\ref{lem:nwallJe}, $\xi' \not\preceq \mu_n$.
This in turn implies that $\xi'$ and $\mu_n$ are not comparable, and hence by Proposition~\ref{prop:sameposition-meetjoin},
\[ 
p_n(\mu_n,\xi' \vee_\preceq \mu_n) = p_n(\xi' \wedge_\preceq \mu_n,\xi'), \quad \text{for all}\ n \in \Ex(\mu_n,\xi' \vee_\preceq \mu_n),
\]
where $\xi' \vee_\preceq \mu_n = \mu$, $\xi' \wedge_\preceq \mu_n = \xi$, and $\Ex(\mu_n,\mu)=\setof{n}$. Thus, 
\[
p_n(\xi,\xi')=p_n(\mu_n,\mu). 
\]
\end{proof}

\begin{defn}
\label{def:wall_labeling}
Let $\cX$ be an $N$ dimensional cubical complex.
A function $\omega \colon W(\cX) \to \setof{\pm 1}$ is a \emph{wall labeling} if for each $\sigma \in \cX^{(0)}$ there exists a map
\[
\tilde{o}_{\sigma}\colon \setof{1, \ldots, N} \to \setof{1,\ldots, N},
\]
called a \emph{local inducement map}, satisfying the following two conditions.
\begin{enumerate}
\item[(i)] Let $\mu, \mu'\in \Top_\cX(\sigma)$ be $n$-adjacent.
If $k \neq n$ and $k \neq \tilde{o}_\sigma(n)$, then
\[
\omega(\mu^-_k, \mu) = \omega(\mu'^-_k, \mu') \quad \text{and} \quad \omega(\mu^+_k, \mu) = \omega(\mu'^+_k, \mu').
\]
\item[(ii)] Let $(\xi, \mu)$ and $(\xi, \mu') \in W(\sigma)$ be $n$-walls. If $n \neq \tilde{o}_\sigma(n)$, then
\[
\omega(\xi, \mu) = \omega(\xi, \mu').
\]
\end{enumerate}
\end{defn}

\begin{rem}
For the purposes of this manuscript the importance and minimality of the concept of a wall labeling cannot be over emphasized.
Our characterization of dynamics of ODEs is derived exclusively from the data of the wall labeling.
Furthermore, modulo the constraints, a wall labeling is a simple tabular data structure -- 
given a cubical complex each wall is assigned either $1$ or $-1$ -- whose size is the order of $2N$ times the number of top cells in $\cX$.    
\end{rem}

Figure~\ref{fig:wall_labeling}(A) provides an example of a wall labeling $\omega\colon \cX \to \setof{\pm 1}$ on the cubical complex $\cX(\I)$ where $\I = \setof{0,1,2,3}^2$. While computations are performed using this algebraic information, the vector representation of $\omega$ shown in Figure~\ref{fig:wall_labeling}(B) may provide greater intuition as to how $\omega$ can be related to dynamics. 

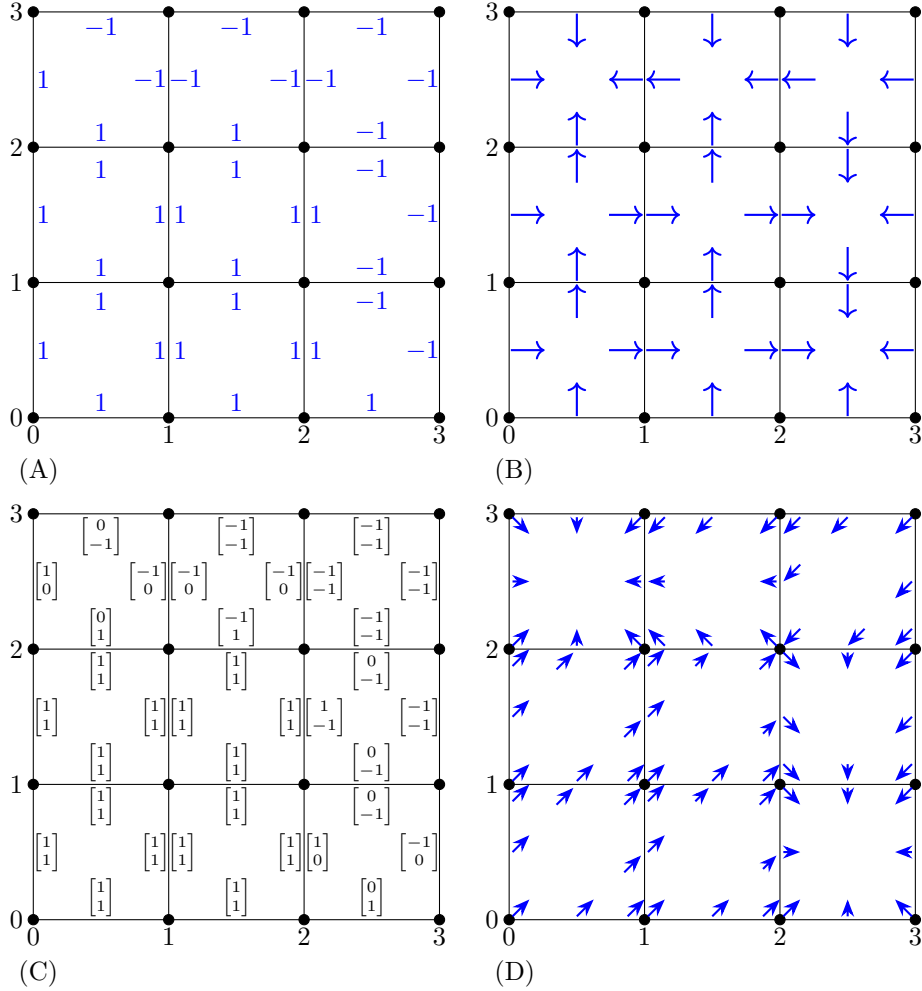
\begin{figure}
\begin{picture}(400,400)(0,-10)

\put(10,180){(A)}
\put(0,190){
\begin{tikzpicture}[scale=0.225]
\foreach \i in {0,...,3}{
    \draw(8*\i,-1) node{$\i$}; 
    \draw(-1,8*\i) node{$\i$}; 
   }
\draw[step=8cm,black] (0,0) grid (24,24);
\foreach \i in {0,...,3}{
  \foreach \j in {0,...,3}{
    \draw[black, fill=black] (8*\i,8*\j) circle (2ex);
  }
}
\foreach \i in {0,...,2}{
  \draw[blue, thick] (4+8*\i,0.9) node{$1$}; 
  \draw[blue, thick] (4+8*\i,23.1) node{$-1$}; 
  \draw[blue, thick] (0.6, 4+8*\i) node{$1$}; 
  \draw[blue, thick] (23.0,4+8*\i) node{$-1$}; 
}
\foreach \i in {0,...,1}{
  \draw[blue, thick] (4+8*\i,6.9) node{$1$};
  \draw[blue, thick] (4+8*\i,8.9) node{$1$};
  \draw[blue, thick] (4+8*\i,14.7) node{$1$};
  \draw[blue, thick] (4+8*\i,16.9) node{$1$};
  \draw[blue, thick] (7.5,4+8*\i) node{$1$};
  \draw[blue, thick] (8.7,4+8*\i) node{$1$};
  \draw[blue, thick] (15.5,4+8*\i) node{$1$};
  \draw[blue, thick] (16.7,4+8*\i) node{$1$};
}
\foreach \i in {2}{
  \draw[blue, thick] (4+8*\i,6.9) node{$-1$};
  \draw[blue, thick] (4+8*\i,8.9) node{$-1$};
  \draw[blue, thick] (4+8*\i,14.7) node{$-1$};
  \draw[blue, thick] (4+8*\i,16.9) node{$-1$};
  \draw[blue, thick] (6.9,4+8*\i) node{$-1$};
  \draw[blue, thick] (9.0,4+8*\i) node{$-1$};
  \draw[blue, thick] (14.9,4+8*\i) node{$-1$};
  \draw[blue, thick] (17.0,4+8*\i) node{$-1$};
}
\end{tikzpicture}
}
\put(10,-10){(C)}
\put(0,0){
\begin{tikzpicture}[scale=0.225]

\foreach \i in {0,...,3}{
    \draw(8*\i,-1) node{$\i$}; 
    \draw(-1,8*\i) node{$\i$}; 
   } 
\draw[step=8cm,black] (0,0) grid (24,24);
\foreach \i in {0,...,3}{
  \foreach \j in {0,...,3}{
    \draw[black, fill=black] (8*\i,8*\j) circle (2ex);
  }
}
\foreach \i in {0,...,1}{
  \foreach \j in {0,...,1}{
    \draw[thick] (8*\i + 0.8, 8*\j + 4) node{\tiny{$\begin{bmatrix} 1 \\ 1 \end{bmatrix}$}};
    \draw[thick] (8*\i + 7.2, 8*\j + 4) node{\tiny{$\begin{bmatrix} 1 \\ 1 \end{bmatrix}$}};
    \draw[thick] (8*\i + 4, 8*\j + 1.3) node{\tiny{$\begin{bmatrix} 1 \\ 1 \end{bmatrix}$}};
    \draw[thick] (8*\i + 4, 8*\j + 6.7) node{\tiny{$\begin{bmatrix} 1 \\ 1 \end{bmatrix}$}};
  }
}
\def \j {2}
\def \i {0}
\draw[thick] (\i * 8 + 0.8, \j * 8 + 4) node{\tiny{$\begin{bmatrix} 1 \\ 0 \end{bmatrix}$}};
\draw[thick] (\i * 8 + 7.2 - 0.4, \j * 8 + 4) node{\tiny{$\begin{bmatrix} -1 \\ 0 \end{bmatrix}$}};
\draw[thick] (\i * 8 + 4, \j * 8 + 1.3) node{\tiny{$\begin{bmatrix} 0 \\ 1 \end{bmatrix}$}};
\draw[thick] (\i * 8 + 4, \j * 8 + 6.7) node{\tiny{$\begin{bmatrix} 0 \\ -1 \end{bmatrix}$}};
\def \i {1}
\draw[thick] (\i * 8 + 0.8 + 0.4, \j * 8 + 4) node{\tiny{$\begin{bmatrix} -1 \\ 0 \end{bmatrix}$}};
\draw[thick] (\i * 8 + 7.2 - 0.4, \j * 8 + 4) node{\tiny{$\begin{bmatrix} -1 \\ 0 \end{bmatrix}$}};
\draw[thick] (\i * 8 + 4, \j * 8 + 1.3) node{\tiny{$\begin{bmatrix} -1 \\ 1 \end{bmatrix}$}};
\draw[thick] (\i * 8 + 4, \j * 8 + 6.7) node{\tiny{$\begin{bmatrix} -1 \\ -1 \end{bmatrix}$}};
\def \i {2}
\draw[thick] (\i * 8 + 0.8 + 0.4, \j * 8 + 4) node{\tiny{$\begin{bmatrix} -1 \\ -1 \end{bmatrix}$}};
\draw[thick] (\i * 8 + 7.2 - 0.4, \j * 8 + 4) node{\tiny{$\begin{bmatrix} -1 \\ -1 \end{bmatrix}$}};
\draw[thick] (\i * 8 + 4, \j * 8 + 1.3) node{\tiny{$\begin{bmatrix} -1 \\ -1 \end{bmatrix}$}};
\draw[thick] (\i * 8 + 4, \j * 8 + 6.7) node{\tiny{$\begin{bmatrix} -1 \\ -1 \end{bmatrix}$}};
\def \i {2}
\def \j {1}
\draw[thick] (\i * 8 + 0.8 + 0.4, \j * 8 + 4) node{\tiny{$\begin{bmatrix} 1 \\ -1 \end{bmatrix}$}};
\draw[thick] (\i * 8 + 7.2 - 0.4, \j * 8 + 4) node{\tiny{$\begin{bmatrix} -1 \\ -1 \end{bmatrix}$}};
\draw[thick] (\i * 8 + 4, \j * 8 + 1.3) node{\tiny{$\begin{bmatrix} 0 \\ -1 \end{bmatrix}$}};
\draw[thick] (\i * 8 + 4, \j * 8 + 6.7) node{\tiny{$\begin{bmatrix} 0 \\ -1 \end{bmatrix}$}};
\def \j {0}
\draw[thick] (\i * 8 + 0.8, \j * 8 + 4) node{\tiny{$\begin{bmatrix} 1 \\ 0 \end{bmatrix}$}};
\draw[thick] (\i * 8 + 7.2 - 0.4, \j * 8 + 4) node{\tiny{$\begin{bmatrix} -1 \\ 0 \end{bmatrix}$}};
\draw[thick] (\i * 8 + 4, \j * 8 + 1.3) node{\tiny{$\begin{bmatrix} 0 \\ 1 \end{bmatrix}$}};
\draw[thick] (\i * 8 + 4, \j * 8 + 6.7) node{\tiny{$\begin{bmatrix} 0 \\ -1 \end{bmatrix}$}};
\end{tikzpicture}
}

\put(190,180){(B)}
\put(180,190){
\begin{tikzpicture}[scale=0.225]
\draw[step=8cm,black] (0,0) grid (24,24);

\foreach \i in {0,...,3}{
    \foreach \j in {0,...,3}{
    \draw[black, fill=black] (8*\i,8*\j) circle (2ex);
    }
}

\foreach \i in {0,...,3}{
    \draw(8*\i,-1) node{$\i$}; 
    \draw(-1,8*\i) node{$\i$}; 
}

\foreach \i in {0,...,2}{
    \draw[->, blue, thick] (4+8*\i,0.1) -- (4+8*\i,2.1); 
    \draw[->, blue, thick] (4+8*\i,23.9) -- (4+8*\i,21.9); 
    \draw[->, blue, thick] (0.1, 4+8*\i) -- (2.1,4+8*\i); 
    \draw[->, blue, thick] (23.9,4+8*\i) -- (21.9,4+8*\i); 
}
\foreach \i in {0,...,1}{
    \draw[->, blue, thick] (4+8*\i,5.9) -- (4+8*\i,7.9); 
    \draw[->, blue, thick] (4+8*\i,8.1) -- (4+8*\i,10.1);
    \draw[->, blue, thick] (4+8*\i,13.9) -- (4+8*\i,15.9); 
    \draw[->, blue, thick] (4+8*\i,16.1) -- (4+8*\i,18.1);  
    \draw[->, blue, thick] (5.9,4+8*\i) -- (7.9,4+8*\i); 
    \draw[->, blue, thick] (8.1,4+8*\i) -- (10.1,4+8*\i);
    \draw[->, blue, thick] (13.9,4+8*\i) -- (15.9,4+8*\i); 
    \draw[->, blue, thick] (16.1,4+8*\i) -- (18.1,4+8*\i);  
}
\foreach \i in {2}{
    \draw[->, blue, thick] (4+8*\i,7.9) -- (4+8*\i,5.9); 
    \draw[->, blue, thick] (4+8*\i,10.1) -- (4+8*\i,8.1);
    \draw[->, blue, thick] (4+8*\i,18.1) -- (4+8*\i,16.1); 
    \draw[->, blue, thick] (4+8*\i,15.9) -- (4+8*\i,13.9);  
    \draw[->, blue, thick] (7.9,4+8*\i) -- (5.9,4+8*\i); 
    \draw[->, blue, thick] (10.1,4+8*\i) -- (8.1,4+8*\i);
    \draw[->, blue, thick] (18.1,4+8*\i) -- (16.1,4+8*\i); 
    \draw[->, blue, thick] (15.9,4+8*\i) -- (13.9,4+8*\i);  
}
\end{tikzpicture}
}

\put(190,-10){(D)}
\put(180,0)
{
\begin{tikzpicture}[scale=0.225]
\draw[step=8cm,black] (0,0) grid (24,24);

\foreach \i in {0,...,3}{
    \foreach \j in {0,...,3}{
    \draw[black, fill=black] (8*\i,8*\j) circle (2ex);
    }
}

\foreach \i in {0,...,3}{
    \draw(8*\i,-1) node{$\i$}; 
    \draw(-1,8*\i) node{$\i$}; 
}
    \draw[-Stealth,blue,thick] (0.2,4) -- ++ (1,1); 
    \draw[-Stealth,blue,thick] (4,0.2) -- ++ (1,1); 
    \draw[-Stealth,blue,thick] (6.8,2.8) -- ++ (1,1); 
    \draw[-Stealth,blue,thick] (2.8,6.8) -- ++ (1,1); 

    \draw[-Stealth,blue,thick] (0.2,0.2) -- ++ (1,1);
    \draw[-Stealth,blue,thick] (0.2,7) -- ++ (1,1);
    \draw[-Stealth,blue,thick] (7,0.2) -- ++ (1,1);
    \draw[-Stealth,blue,thick] (6.8,6.8) -- ++ (1,1);
    \draw[-Stealth,blue,thick] (0.2,12) -- ++ (1,1); 
    \draw[-Stealth,blue,thick] (4,8.2) -- ++ (1,1); 
    \draw[-Stealth,blue,thick] (6.8,10.8) -- ++ (1,1); 
    \draw[-Stealth,blue,thick] (2.8,14.8) -- ++ (1,1); 

    \draw[-Stealth,blue,thick] (0.2,8.2) -- ++ (1,1);
    \draw[-Stealth,blue,thick] (0.2,15) -- ++ (1,1);
    \draw[-Stealth,blue,thick] (7,8.2) -- ++ (1,1);
    \draw[-Stealth,blue,thick] (6.8,14.8) -- ++ (1,1);
    \draw[-Stealth,blue,thick] (0.2,20) -- ++ (1,0); 
    \draw[-Stealth,blue,thick] (4,16.2) -- ++ (0,1); 
    \draw[-Stealth,blue,thick] (7.8,20) -- ++ (-1,0); 
    \draw[-Stealth,blue,thick] (4,23.8) -- ++ (0,-1);

    \draw[-Stealth,blue,thick] (0.2,16.2) -- ++ (1,1);
    \draw[-Stealth,blue,thick] (0.2,23.8) -- ++ (1,-1);
    \draw[-Stealth,blue,thick] (7.8,16.2) -- ++ (-1,1);
    \draw[-Stealth,blue,thick] (7.8,23.8) -- ++ (-1,-1);
    \draw[-Stealth,blue,thick] (8.2,4) -- ++ (1,1); 
    \draw[-Stealth,blue,thick] (12,0.2) -- ++ (1,1); 
    \draw[-Stealth,blue,thick] (15,3) -- ++ (0.8,0.8); 
    \draw[-Stealth,blue,thick] (11,7) -- ++ (0.8,0.8); 

    \draw[-Stealth,blue,thick] (8.2,0.2) -- ++ (1,1);
    \draw[-Stealth,blue,thick] (8.2,7) -- ++ (1,1);
    \draw[-Stealth,blue,thick] (15,0.2) -- ++ (1,1);
    \draw[-Stealth,blue,thick] (14.8,6.8) -- ++ (1,1);
    \draw[-Stealth,blue,thick] (8.2,12) -- ++ (1,1); 
    \draw[-Stealth,blue,thick] (12,8.2) -- ++ (1,1); 
    \draw[-Stealth,blue,thick] (15,11) -- ++ (0.8,0.8); 
    \draw[-Stealth,blue,thick] (11,15) -- ++ (0.8,0.8); 

    \draw[-Stealth,blue,thick] (8.2,8.2) -- ++ (1,1);
    \draw[-Stealth,blue,thick] (8.2,15) -- ++ (1,1);
    \draw[-Stealth,blue,thick] (15,8.2) -- ++ (1,1);
    \draw[-Stealth,blue,thick] (14.8,14.8) -- ++ (1,1);
    \draw[-Stealth,blue,thick] (12,16.2) -- ++ (-1,1); 
    \draw[-Stealth,blue,thick] (9.2,20) -- ++ (-1,0); 
    \draw[-Stealth,blue,thick] (12,23.8) -- ++ (-1,-1); 
    \draw[-Stealth,blue,thick] (15.8,20) -- ++ (-1,0);

    \draw[-Stealth,blue,thick] (9.2,16.2) -- ++ (-1,1);
    \draw[-Stealth,blue,thick] (9.2,23.8) -- ++ (-1,-1);
    \draw[-Stealth,blue,thick] (15.8,16.2) -- ++ (-1,1);
    \draw[-Stealth,blue,thick] (15.8,23.8) -- ++ (-1,-1);
    \draw[-Stealth,blue,thick] (20,0.2) -- ++ (0,1); 
    \draw[-Stealth,blue,thick] (16.2,4) -- ++ (1,0); 
    \draw[-Stealth,blue,thick] (20,7.8) -- ++ (0,-1); 
    \draw[-Stealth,blue,thick] (23.8,4) -- ++ (-1,0); 

    \draw[-Stealth,blue,thick] (16.2,0.2) -- ++ (1,1);
    \draw[-Stealth,blue,thick] (16.2,7.8) -- ++ (1,-1);
    \draw[-Stealth,blue,thick] (23.8,0.2) -- ++ (-1,1);
    \draw[-Stealth,blue,thick] (23.8,7.8) -- ++ (-1,-1);
    \draw[-Stealth,blue,thick] (20,9.2) -- ++ (0,-1); 
    \draw[-Stealth,blue,thick] (16.2,12) -- ++ (1,-1); 
    \draw[-Stealth,blue,thick] (20,15.8) -- ++ (0,-1); 
    \draw[-Stealth,blue,thick] (23.8,12) -- ++ (-1,-1); 

    \draw[-Stealth,blue,thick] (16.2,9.2) -- ++ (1,-1);
    \draw[-Stealth,blue,thick] (16.2,15.8) -- ++ (1,-1);
    \draw[-Stealth,blue,thick] (23.8,9.2) -- ++ (-1,-1);
    \draw[-Stealth,blue,thick] (23.8,15.8) -- ++ (-1,-1);    
    \draw[-Stealth,blue,thick] (21,17.2) -- ++ (-1,-1); 
    \draw[-Stealth,blue,thick] (17.2,21) -- ++ (-1,-1); 
    \draw[-Stealth,blue,thick] (20,23.8) -- ++ (-1,-1); 
    \draw[-Stealth,blue,thick] (23.8,20) -- ++ (-1,-1);

    \draw[-Stealth,blue,thick] (17.2,17.2) -- ++ (-1,-1);
    \draw[-Stealth,blue,thick] (17.2,23.8) -- ++ (-1,-1);
    \draw[-Stealth,blue,thick] (23.8,17.2) -- ++ (-1,-1);
    \draw[-Stealth,blue,thick] (23.8,23.8) -- ++ (-1,-1);
\end{tikzpicture}
}
\end{picture}
\caption{All four figures are geometric representations of $\cX(\I)$ where $\I = \prod_{n=1}^2\setof{0,1,2,3}$, i.e., $K(n)=2$ for $n=1,2$.
(A) Visual representation of a wall labeling $\omega\colon W(\cX) \to \setof{\pm 1}$.
(B) A vector representation of the wall labeling $\omega$.
(C) Rook field restricted to the walls, $\rook \colon W(\cX) \to \setof{0,\pm 1}^2$,  associated with wall labeling $\omega$.
(D) A vector representation of the rook field $\rook\colon TP(\cX) \to \setof{0,\pm 1}^N$.
}
\label{fig:wall_labeling}
\end{figure}

As the following example shows, to have a wall labeling it must be the case that for each pair of $n$-adjacent top cells, only the labels corresponding to the $n$-wall and the $\tilde{o}_\sigma(n)$-wall are allowed to change.

\begin{ex}
An example of a function $\omega\colon W(\cX)\to \setof{\pm 1}$ that is not a wall labeling is shown in Figure~\ref{fig:notWallLabeling}.
In particular, there is no local inducement map $\tilde{o}_{\sigma}\colon \setof{1,2} \to \setof{1,2}$ at the vertex $\sigma = \defcell{1}{1}{0}{0}$.
We prove this by contradiction.

Consider $\mu,\mu'\in \Top_\cX(\sigma)$ given by $\mu = \defcell{0}{1}{1}{1}$ and $\mu' = \defcell{1}{1}{1}{1}$.
Observe that $\mu$ and $\mu'$ are $1$-adjacent.

Assume that $\tilde{o}_{\sigma}(1) = 2$.
Let $\xi = \defcell{1}{1}{0}{1}$.
Then $(\xi,\mu), (\xi,\mu')\in W(\sigma)$ are $1$-walls.    
By Definition~\ref{def:wall_labeling}~(ii), if $\tilde{o}_{\sigma}$ is a local inducement map, then
\[
1 = \omega (\xi, \mu) = \omega (\xi, \mu') = -1,
\]
which is a contradiction. 

Therefore, if a local inducement map exists it must be the case that $\tilde{o}_{\sigma}(1) = 1$.
Definition~\ref{def:wall_labeling}~(i) can be applied with $n=1$ and $k=2$, which again leads to a contradiction.
\end{ex}

\begin{figure}
\centering
\begin{tikzpicture}[scale=0.3]
    \draw[step=8cm,black] (0,0) grid (16,16);
    \foreach \i in {0,...,2}{
        \foreach \j in {0,...,2}{
        \draw[black, fill=black] (8*\i,8*\j) circle (2ex);
        }
    }
    \foreach \i in {0,...,2}{
        \draw(-1,8*\i ) node{$\i$}; 
    }
    \foreach \i in {0,...,2}{
        \draw(8*\i,-1.1) node{$\i$}; 
    }
    \foreach \i in {0,...,1}{
      \draw[blue, thick] (4+8*\i,0.9) node{$1$}; 
      \draw[blue, thick] (4+8*\i,15.1) node{$-1$}; 
      \draw[blue, thick] (0.6, 4+8*\i) node{$1$}; 
      \draw[blue, thick] (15.0,4+8*\i) node{$-1$}; 
    }
    \foreach \i in {0,...,1}{
      \draw[blue, thick] (4+8*\i,6.9) node{$1$};
      \draw[blue, thick] (7,4+8*\i) node{$1$};
    }
        \draw[blue, thick] (4,8.9) node{$1$};
      \draw[blue, thick] (12,8.9) node{$-1$};
       \draw[blue, thick] (9,4) node{$1$};
       \draw[blue, thick] (9,12) node{$-1$};
\end{tikzpicture}
\caption{
An example of a function $\omega\colon W(\cX)\to \setof{\pm 1}$ that is not a wall labeling.}
\label{fig:notWallLabeling}
\end{figure}

\begin{defn}
\label{defn:dissipativewall}
Let $\cX$ be an $N$-dimensional cubical complex. A wall labeling $\omega \colon W(\cX) \to \setof{-1, 1}$ is \emph{strongly dissipative} if for every wall 
$(\xi,\mu) \in W(\cX)$ with $\xi \in \bbdy(\cX)$ (see Definition~\ref{def:boundary_prime}) we have
\[
\omega(\xi,\mu) = -p(\xi,\mu).
\]
\end{defn}
It is left to the reader to check that the wall labeling shown in Figure~\ref{fig:wall_labeling}(A) is strongly dissipative.

Given a wall labeling on an $N$-dimensional cubical complex $\cX$ we define a discrete vector field as follows.

\begin{defn}
\label{def:rookfield}
Let $\omega \colon W(\cX)\to \setof{\pm 1}$ be a wall labeling on an $N$-dimensional cubical complex $\cX$.
The associated \emph{rook field} $\rook\colon TP(\cX) \to \setof{0,\pm 1}^N$ is defined as follows. Given an $n$-wall, $(\xi, \mu) \in W(\cX)$, set
\begin{equation}
\label{eq:rookwalls}
\rook_k(\xi, \mu) :=
\begin{cases}
\omega(\xi, \mu), & \text{if}~ k = n \\
\omega(\mu^-_k, \mu), & \text{if}~ k \neq n ~\text{and}~ \omega(\mu^-_k, \mu) = \omega(\mu^+_k, \mu) \\
0, & \text{if}~ k \neq n ~\text{and}~ \omega(\mu^-_k, \mu) \neq \omega(\mu^+_k, \mu).
\end{cases}
\end{equation}
For $(\xi, \mu) \in TP(\cX) \setminus W(\cX)$, define
\begin{equation}
\label{eq:PhiSigmaKappa}
\rook_n(\xi, \mu) :=
\begin{cases}
\rook_n(\mu^-_n, \mu), & \text{if}~ n \in J_e(\xi) ~\text{and}~ \rook_n(\mu^-_n, \mu) = \rook_n(\mu^+_n, \mu) \\
0, & \text{if}~ n \in J_e(\xi) ~\text{and}~ \rook_n(\mu^-_n, \mu) \neq \rook_n(\mu^+_n, \mu) \\
\rook_n(\mu_n^*(\xi, \mu),\mu), & \text{if}~ n \in J_i(\xi) 
\end{cases}
\end{equation}
where $J_e(\xi)$ are the essential directions of $\xi$, $J_i(\xi)$ are the inessential directions of $\xi$ and $\mu_n^*$ is the $n$-wall given by Proposition~\ref{prop:mu*}.
\end{defn}

We state the following corollary of Proposition~\ref{prop:mu*} for future reference.

\begin{cor}
\label{cor:mu*}
Consider an $N$-dimensional cubical complex $\cX$. Let $\xi \preceq \xi' \prec \mu \in \cX$ with $\dim(\mu) = N$ and let $n \in J_i(\xi')$.
Then,
\[
\rook_n(\xi,\mu) = \rook_n(\xi',\mu) = \omega(\mu_n^*(\xi,\mu),\mu).
\]
\end{cor}

Figure~\ref{fig:wall_labeling}(C) shows a rook field restricted to the walls, i.e., $\rook\colon W(\cX)\to \setof{0,\pm 1}^N$ and Figure~\ref{fig:wall_labeling}(D) provides a vector representation of the associated rook field $\rook\colon TP(\cX) \to \setof{0,\pm 1}^2$.

\begin{rem}
Figure~\ref{fig:wall_labeling}(d) suggests a phase plane diagram for a $2$-dimensional ODE.
Hopefully, this encourages the reader to believe that rook fields can be related to continuous dynamics.
However, conceptually it is misleading to think of Figure~\ref{fig:wall_labeling}(D) as providing, in any traditional sense, an approximation of a continuous vector field.
In this Part we  use the rook field to characterize dynamics in terms of Morse graphs and Conley complexes.
It is only in Part~\ref{part:III} that we identify a class of differential equations for which the combinatorial/homological classification is valid. 

As is made clear via Definition~\ref{def:rookfield}, Figure~\ref{fig:wall_labeling}(D) contains no more information than that provided by Figure~\ref{fig:wall_labeling}(A); it is just a re-organization of the tabular information of the wall labeling.
As such, the rook field is merely a convenient organization of the data provided by the wall labeling.
\end{rem}

\begin{lem}
\label{rem:o_sigmas_agree}
Let $\cX$ be an $N$-dimensional cubical complex.
Let $\mu, \mu' \in \cX^{(N)}$ be $n$-adjacent with shared wall $\xi$.
Consider vertices $\sigma,\sigma'\in \cX^{(0)}$ such that $\sigma, \sigma' \preceq \xi$. 
If $\rook_k(\xi, \mu) \neq \rook_k(\xi, \mu')$, then $\tilde{o}_\sigma(n) = \tilde{o}_{\sigma'}(n) = k$. 
\end{lem}
\begin{proof}
Since $(\xi, \mu)$ and $(\xi, \mu')$ are $n$-walls, it follows from Definition~\ref{def:rookfield} that
$\rook_k(\xi, \mu)$ and $\rook_k(\xi, \mu')$ are given by \eqref{eq:rookwalls}. There are two cases to consider. 

If $k = n$, then
\[
\rook_k(\xi,\mu)=\omega(\xi,\mu)
\quad\text{and}\quad  \rook_k(\xi,\mu')=\omega(\xi,\mu').
\]
Hence, $\rook_k(\xi, \mu) \neq \rook_k(\xi, \mu')$ is equivalent to $\omega(\xi,\mu)\neq \omega(\xi,\mu')$.
Thus, by Definition~\ref{def:wall_labeling}~(ii) it follows that $\tilde{o}_\sigma(n) = n$ for any $\sigma \in \cX^{(0)}$ such that $\sigma \preceq \xi$.

Now assume that $k \neq n$.
Since $\rook_k(\xi, \mu) \neq \rook_k(\xi, \mu')$ from \eqref{eq:rookwalls} we can assume, without loss of generality, that $\rook_k(\xi, \mu) = \omega(\mu^-_k, \mu)$ and $\rook_k(\xi, \mu') = 0$. This implies that $\omega(\mu^-_k, \mu) = \omega(\mu^+_k, \mu)$ and $\omega(\mu'^-_k, \mu') \neq \omega(\mu'^+_k, \mu')$. 
Hence,  by Definition~\ref{def:wall_labeling}(i) $\tilde{o}_\sigma(n) = k$ for any $\sigma \in \cX^{(0)}$ such that $\sigma \preceq \xi$.
\end{proof}

\begin{defn}
\label{def:active_regulation}
Let $\xi \in \cX$ and $k \in \setof{1,\ldots,N}$. A direction $n \in J_i(\xi)$ \emph{actively regulates} $k$ at $\xi$ if there exist $n$-walls $(\xi_n,\mu), (\xi_n,\mu') \in W(\xi)$ such that
\[
\rook_k(\xi_n, \mu) \neq \rook_k(\xi_n, \mu').
\]
Define
\[
\activeset(\xi) := \setof{n\in J_i(\xi)\mid \text{$n$ actively regulates $k$ for some $k\in \setof{1,\dots,N}$}}.
\]
\end{defn}

The proof of the following proposition follows directly from the definition of active regulation and is left to the reader.
\begin{prop}
\label{prop:act_regulation_faces}
Let $\xi \preceq \xi' \in \cX$. If $n$ actively regulates $k$ at $\xi'$, then $n$ actively regulates $k$ at $\xi$.
\end{prop}

\begin{ex}
\label{ex:activeRegulation}
Consider the rook field $\rook$ shown in Figure~\ref{fig:wall_labeling}. Let $\xi = \defcell{1}{2}{0}{0}$. Then $n = 2$ actively regulates $k = 1$ at $\xi$ since for the following $2$-walls
\[
\left(\defcell{0}{2}{1}{0},\defcell{0}{2}{1}{1}\right)
\quad \text{and} \quad 
\left(\defcell{0}{2}{1}{0},\defcell{0}{1}{1}{1}\right)
\]
we have that $\rook_1$ disagrees. It is left to the reader to check that $\activeset(\xi) = \setof{2}$.
\end{ex}

\begin{ex}
\label{ex:Actempty}
Consider the rook field $\rook$ shown in Figure~\ref{fig:wall_labeling}. Let $\xi = \defcell{1}{1}{0}{0}$.
Then, $\rook_k(\xi,\mu) = 1$ for $k = 0, 1$ and for all $\mu\in \Top_\cX(\xi)$. Therefore $\activeset(\xi)=\emptyset$.
\end{ex}

Notice that given $\xi \in \cX$ and $n \in \activeset(\xi)$ it follows from Lemma~\ref{rem:o_sigmas_agree} that $\tilde{o}_\sigma(n) = \tilde{o}_{\sigma'}(n)$ for any $\sigma, \sigma' \in \cX^{(0)}$ with $\sigma, \sigma' \preceq \xi$. Therefore, the following map is well-defined.

\begin{defn}
\label{def:regulationmap}
Given $\xi \in \cX$, we define the \emph{regulation map} $\rmap{\xi}$ at $\xi$ as $\tilde{o}_\sigma$ restricted to $\activeset(\xi)$, that is,
\begin{align*}
\rmap{\xi} : \activeset(\xi) & \to \setof{1,\dots,N} \\
n & \mapsto \tilde{o}_\sigma(n)
\end{align*}
for any vertex $\sigma \in \cX^{(0)}$ with $\sigma \preceq \xi$.
\end{defn}

\section{Properties of rook fields}
\label{sec:rookproperties}

As is clear from \eqref{eq:PhiSigmaKappa}, for a fixed $\xi \in \cX$ the values of $ \rook_n(\xi, \mu)$ may depend on $\mu \in \Top_\cX(\xi)$. 
To keep track of these values we use the following notation:

\begin{equation}
\label{eq:Rn}
R_n(\xi) := \setof{\rook_n(\xi ,\mu) \mid \mu \in \Top_\cX(\xi)}.
\end{equation}

We leave it to the reader to check the following proposition.
\begin{prop}\label{prop:face_R_n}
If $\xi,\xi'\in\cX$ and $\xi\preceq\xi'$, then $R_n(\xi')\subset R_n(\xi)$ for all $n\in\setof{1, \ldots, N}$.
\end{prop}

\begin{defn}
\label{def:namedirections}
Let $\xi\in \cX$.
The  \emph{gradient}, \emph{neutral}, and \emph{opaque directions} of  $\xi$ are
\begin{align*}
G(\xi) &:= \setof{n\mid R_n(\xi)=\setof{1}\text{ or }  R_n(\xi)=\setof{-1}}, \\
N(\xi) & := \setof{n\mid 0 \in R_n(\xi)},\ \text{and} \\
O(\xi) &:= \setof{n\mid R_n(\xi) = \setof{\pm 1}},
\end{align*}
respectively. 
\end{defn}
 
\begin{prop}
\label{prop:Direction-Decomposition}
Given $\xi\in \cX$, the sets $G(\xi)$, $N(\xi)$, and  $O(\xi)$ are mutually disjoint.
Furthermore,
\begin{enumerate}
\item[(i)] $G(\xi) \cup N(\xi)\cup O(\xi) = \setof{1,\ldots, N}$, and
\item[(ii)] $J_i(\xi) \subseteq O(\xi) \cup G(\xi)$.
\end{enumerate}
\end{prop}

\begin{proof}
The given sets are clearly mutually disjoint. To see that $G(\xi) \cup N(\xi)\cup O(\xi) = \setof{1,\ldots, N}$ observe that for any $n \in \setof{1,\dots,N}$, there are at most four possibilities for $R_n(\xi)$.
If $0 \in R_n(\xi)$, then $n\in N(\xi)$.
If $R_n(\xi)=\{\pm 1\}$, then $n\in O(\xi)$.
Finally, if $R_n(\xi) = \{1\}$ or $R_n(\xi)=\{-1\}$, then $n\in G(\xi)$. 

To prove the second claim, recall that by Definition~\ref{def:rookfield}, if $\rook_n(\xi, \mu) = 0$ then $n \in J_e(\xi)$. Thus $J_i(\xi) \cap N(\xi) = \emptyset$ and so by item (i) it follows that $J_i(\xi) \subseteq  O(\xi) \cup G(\xi)$.
\end{proof}
Note that Proposition~\ref{prop:Direction-Decomposition} does not exclude that a gradient direction might be essential or inessential. To address that, we introduce the following notation.
\begin{defn}
\label{defn:gradient-inessential-essential}
Let $\xi \in \cX$. The \emph{gradient essential} and \emph{gradient inessential directions} of $\xi$ are
\[
G_e(\xi) \coloneqq G(\xi) \cap J_e(\xi) \quad \text{and} \quad G_i(\xi) \coloneqq G(\xi) \cap J_i(\xi),
\]
respectively. Similarly, the \emph{opaque essential} and \emph{opaque inessential directions} of $\xi$ are respectively
\[
O_e(\xi) \coloneqq O(\xi) \cap J_e(\xi) \quad \text{and} \quad O_i(\xi) \coloneqq O(\xi) \cap J_i(\xi).
\]
\end{defn}
\begin{lem}
\label{lem:gradientsubset}
Let $\xi \preceq \xi'$. If $n \in J_e(\xi) \cup J_i(\xi')$ and $n \in G(\xi)$, then $n \in G(\xi')$.
\end{lem}
\begin{proof}
By definition, $n \in G(\xi)$ implies that $R_n(\xi) = \setof{r_n}$ with $r_n\neq 0$. Hence we have
\[
\rook_n(\xi, \mu) = r_n, \quad \forall \mu \in \Top_\cX(\xi).
\]

If $n \in J_e(\xi)$, Definition~\ref{def:rookfield} implies that
\[
\rook_n(\mu^-_n, \mu) = \rook_n(\mu^+_n, \mu) = \rook_n(\xi, \mu) = r_n, \quad \forall \mu \in \Top_\cX(\xi).
\]
Since $\Top_\cX(\xi') \subseteq \Top_\cX(\xi)$, the equality holds when restricted to $\Top_\cX(\xi')$, i.e., 
\[
\rook_n(\mu^-_n, \mu) = \rook_n(\mu^+_n, \mu) = r_n, \quad \forall \mu \in \Top_\cX(\xi'),
\]
thus $\rook_n(\xi', \mu) = r_n$ for all $\mu \in \Top_\cX(\xi')$. This implies that $R_n(\xi') = \{ r_n \}$ and therefore that $n \in G(\xi')$.

If $n \in J_i(\xi')$, Corollary~\ref{cor:mu*} implies that
\[
\rook_n(\xi', \mu) = \rook_n(\xi, \mu), \quad \forall \mu \in \Top_\cX(\xi').
\]
Hence $\rook_n(\xi', \mu) = r_n$ for all $\mu \in \Top_\cX(\xi')$, which imply that $n \in G(\xi')$.
\end{proof}

\begin{prop}
\label{prop:GradPsiIndependent}
Let $\xi \preceq \xi'$. 
If $n \in G(\xi')$,
then $\rook_n(\xi', \mu)$ and $\rook_n(\xi, \mu)$ are independent of $\mu \in \Top_\cX(\xi')$.
\end{prop}
 
\begin{proof}
By definition, $n \in G(\xi')$ implies that $R_n(\xi')$ is a singleton set. 
Hence, by \eqref{eq:Rn} it follows that $\rook_n(\xi', \mu)$ is independent of $\mu \in \Top_\cX(\xi')$. 
Therefore, it is sufficient to show that $\rook_n(\xi, \mu) = \rook_n(\xi', \mu)$ for all $\mu \in \Top_\cX(\xi')$.
By Corollary~\ref{cor:mu*} this is true if $n \in J_i(\xi')$.

So assume that $n \in J_e(\xi')$. Since $n \in G(\xi')$ it follows from Definition~\ref{def:rookfield} that
\begin{equation}
\label{eq:same_Phi}
\rook_n(\xi', \mu) = \rook_n(\mu^-_n, \mu) = \rook_n(\mu^+_n, \mu).
\end{equation}
We have two cases to consider for $\rook_n(\xi, \mu)$. If $n \in J_i(\xi)$, then Definition~\ref{def:rookfield} implies that
\[
\rook_n(\xi, \mu) = \rook_n(\mu^*_n, \mu),
\]
where $\mu^*_n \in \setof{\mu^-_n, \mu^+_n}$ is the $n$-wall given by Proposition~\ref{prop:mu*}. Thus it follows from \eqref{eq:same_Phi} that $\rook_n(\xi, \mu) = \rook_n(\xi', \mu)$ for all $\mu \in \Top_\cX(\xi')$.

In contrast,
if $n \in J_e(\xi)$ it follows from Definition~\ref{def:rookfield} and \eqref{eq:same_Phi} that
\[
\rook_n(\xi, \mu) = \rook_n(\mu^-_n, \mu) = \rook_n(\xi', \mu),
\]
for all $\mu \in \Top_\cX(\xi')$.
\end{proof}

When focusing on directions that are actively regulated, it is of particular interest to understand when the regulation map has a cyclic behavior. 

\begin{defn}
\label{def:acyclic-direction}
Let $\xi \in \cX$. A direction $n$ is \emph{cyclic} at $\xi$ if $n \in \activeset(\xi)$ and there exists an integer $k \geq 1$ such that $\rmap \xi ^i(n) \in \activeset(\xi)$ for all $1 \leq i \leq k$ and
\[
\rmap\xi^k(n) = n.
\]
In this case we define the \emph{forward orbit} of $n$ as the set given by the iterates of $n$ by the map $\rmap\xi$, that is,
\begin{equation}
\label{orbitS}
\cycleof_\xi(n) = \setdef{\rmap\xi^i(n)}{i = 0, 1, \dots, \ell }.
\end{equation}
The minimal such $\ell$ is called the \emph{length} of the cycle of $n$. If the length of the cycle of $n$ is one, then we say that $n$ is \emph{fixed}.
\end{defn}

\begin{ex}
\label{ex:cyclingcells}
Returning to Figure~\ref{fig:wall_labeling}(D) let $\xi = \defcell{2}{2}{0}{0}$ and observe that for the following $1$-walls and $2$-walls of $\xi$ we have
\begin{align*}
\rook_2\left( \defcell{2}{2}{0}{1},\defcell{1}{2}{1}{1}\right) & \neq \rook_2\left( \defcell{2}{2}{0}{1},\defcell{2}{2}{1}{1}\right) \\
\rook_1\left( \defcell{2}{2}{1}{0},\defcell{2}{1}{1}{1}\right) & \neq \rook_1\left( \defcell{2}{2}{1}{0},\defcell{2}{2}{1}{1}\right). 
\end{align*}
Thus, $\activeset(\xi)=\setof{1,2}$ and $o_\xi(1)=2$ and $o_\xi(2)=1$. Finally,  $\cycleof_\xi(n) =\setof{1,2}$ and the length of the cycle is $2$.
\end{ex}

\begin{prop}
\label{prop:wallcycle}
Let $\mu, \mu' \in \cX^{(N)}$ be $k$-adjacent with shared $k$-wall $\xi$. 
If $k \in \rmap\xi(\activeset(\xi))$, then $k$ is fixed.  
\end{prop}
\begin{proof}
By definition $\activeset(\xi)\subset J_i(\xi) = \setof{k}$.
By assumption $k \in \rmap\xi(\activeset(\xi))$, which implies that $\xi$ is fixed.
\end{proof}
We now turn our attention to special types of cells.
\begin{defn}
\label{defn:opaque}
A cell $\xi \in \cX \setminus \cX^{(N)}$ is \emph{opaque} if $J_i(\xi) \subseteq O(\xi)$. 
\end{defn}

\begin{prop}
\label{prop:opaque}
If $\xi$ is opaque, then 
\begin{enumerate}
\item[(i)] the map $\rmap \xi \colon \activeset(\xi) \subseteq J_i(\xi) \to J_i(\xi)$ is a bijection,
\item[(ii)] $J_i(\xi)=\activeset(\xi)= O(\xi)$, and
\item[(iii)] $J_e(\xi) = G(\xi)\cup N_0(\xi)$ where
\[
N_0(\xi) = \setof{n\mid R_n(\xi)=\setof{0}}\subset N(\xi).
\]
\end{enumerate}
\end{prop}

\begin{proof}
We first prove (i) by contradiction. Assume that $\rmap\xi$ is not surjective onto $J_i(\xi)$, that is, assume that there exists $k \in J_i(\xi)$ such that $k \not\in \rmap\xi(\activeset(\xi))$. This implies that 
\[
R_k(\xi) = \setof{\rook_k(\xi,\mu) \mid \mu\in \Top_\cX(\xi)} \neq \setof{\pm 1},
\]
since for $k \in J_i(\xi)$ and $k \not\in \rmap\xi(\activeset(\xi))$ we must have $\rook_k(\xi,\mu) = \rook_k(\xi,\mu')$ for all $\mu, \mu' \in \Top_\cX(\xi)$. Since $\xi$ is opaque we have that $k \in J_i(\xi) \subseteq O(\xi)$. Now $R_k(\xi) \neq \setof{\pm 1}$ with $k \in O(\xi)$ is a contradiction. Thus $\rmap\xi$ is surjective. Since $\activeset(\xi) \subseteq J_i(\xi)$ is finite, the map $\rmap\xi \colon \activeset(\xi)\subseteq J_i(\xi) \to J_i(\xi)$ is a bijection.

To prove (ii), note that $J_i(\xi)=\activeset(\xi)$ follows from (i). Additionally, $J_i(\xi) \subseteq O(\xi)$ follows from opaqueness. To prove $O(\xi) \subseteq J_i(\xi)$, let $k\in O(\xi)$. By the definition of opaque direction, 
\[
R_k(\xi) = \setof{\rook_k(\xi,\mu)\mid \mu\in  \Top_\cX(\xi)} = \setof{\pm 1}.
\] 
By the definition of $\rook$, there exists $\mu,\mu'\in \Top_\cX(\xi)$ such that $\rook_k(\xi,\mu)\neq\rook_k(\xi,\mu')$. So there exists $n \in \activeset(\xi)$ such that $\rmap\xi(n)=k$.  Thus, $k \in \rmap\xi(\activeset(\xi))$ and hence $k \in J_i(\xi)$. Therefore, $O(\xi)=J_i(\xi)$. 

We prove (iii) by contradiction. By (ii) and Proposition~\ref{prop:Direction-Decomposition}, $J_e(\xi) = G(\xi) \cup N(\xi)$. Assume there exists $m \in N(\xi) \subset J_e(\xi)$ such that $\setof{0,1} \subset R_m(\xi)$ or $\setof{0,-1} \subset R_m(\xi)$. In either case, there exists $\mu,\mu'\in \Top_\cX(\xi)$ such that $\rook_m(\xi,\mu) \neq \rook_m(\xi,\mu')$. This implies that there exists $n\in \activeset(\xi)$ such that $\rmap\xi(n)=m$, which contradicts (i). Therefore
$N(\xi) = N_0(\xi)$ which proves (iii).
\end{proof}

\begin{defn}
\label{def:eqcell}
A cell $\xi \in \cX$ is an \emph{equilibrium cell} if $G(\xi) = \emptyset$. A cell $\xi \in \cX$ is \emph{regular} if it is not an equilibrium cell.
\end{defn}

\begin{prop}
\label{prop:eqcell-is-opaque}
Let $\xi \in \cX$. If $\xi$ is an equilibrium cell and $\xi \not\in \cX^{(N)}$, then $\xi$ is opaque.
\end{prop}
\begin{proof}
Let $n \in J_i(\xi)$. It follows from Definition~\ref{def:rookfield} that $\rook_n(\xi, \mu) \in \setof{\pm 1}$ for all $\mu \in  \Top_\cX(\xi)$. Hence $n \in G(\xi) \cup O(\xi) = O(\xi)$.
\end{proof}

\begin{prop}
Let $\xi\in\cX$. If $\xi$ is opaque and $J_e(\xi)=N_0(\xi)$, then $\xi$ is an equilibrium cell.
\end{prop}

\begin{proof}
Since $\xi$ opaque, by Proposition~\ref{prop:opaque} it follows that
\[
N_0(\xi) = J_e(\xi) = G(\xi)\cup N_0(\xi)
\]
which, by Proposition~\ref{prop:Direction-Decomposition}, implies that $G(\xi) = \emptyset$.
\end{proof}

\begin{defn}
\label{def:exit_face}
Let $\xi\prec \xi'\in\cX$.
\begin{enumerate}
\item[(i)] $\xi$ is an \emph{exit face} of $\xi'$ if $\rook_n(\xi,\mu) = p_n(\xi,\xi')$ for all $n\in \Ex(\xi,\xi')$ and all $\mu \in \Top_\cX(\xi')$.
The set of exit faces of $\xi'$ is denoted by $E^-(\xi')$.
\item[(ii)]  $\xi$ is an \emph{entrance face} of $\xi'$ if $\rook_n(\xi,\mu) = -p_n(\xi,\xi')$ for all $n\in \Ex(\xi,\xi')$ and all $\mu \in \Top_\cX(\xi')$.
The set of entrance faces of $\xi'$ is denoted by $E^+(\xi')$.
\end{enumerate}
\end{defn}

The following result is an immediate consequence of the definition of an exit or entrance face.
\begin{lem}
\label{lem:exitwall}
    If $(\xi,\mu)\in W(\cX)$, then either $\xi\in E^+(\mu)$ or $\xi\in E^-(\mu)$.
\end{lem}

\begin{ex}
\label{ex:exitface}
The wall labeling shown in Figure~\ref{fig:exitset} is obtained by restricting the complex on which the wall labeling in Figure~\ref{fig:wall_labeling} is defined.
Let $\xi = \defcell{1}{2}{0}{0}$ and
$\hat{\xi} = \defcell{1}{1}{0}{1}$.
Then $\xi\prec \xi'$ and $\Ex(\xi,\xi') = \setof{2}$.    
Note that $\Top_\cX(\xi') = \{ \hat{\xi}_0',\hat{\xi}_1' \}$.
Since $2\in J_i(\xi)$, by \eqref{eq:PhiSigmaKappa} and \eqref{eq:rookwalls}
\[
\rook_2 (\xi ,\hat{\xi}_1' ) = \rook_2 (\xi',\hat{\xi}_1' ) =  1 = p_2(\xi,\xi').
\]
Similarly, 
\[
\rook_2 (\xi ,\hat{\xi}_0' ) = \rook_2 (\xi',\hat{\xi}_0' ) = 1 = p_2(\xi,\xi').
\]
Thus, $\xi\in E^-(\xi')$.

A similar argument shows that $\xi\in E^+\left(\tau\right)$.

In contrast consider $\xi_1' = \defcell{1}{2}{1}{0}$.
Then $\xi\prec \xi_1'$, $\Ex(\xi,\xi_1') = \setof{1}$, and $\Top_\cX(\xi_1') = \{ \mu_1,\hat{\xi}_1' \}$. Since $1\in J_i(\xi)$, by \eqref{eq:PhiSigmaKappa} and \eqref{eq:rookwalls}
\[
\rook_1 (\xi, \hat{\xi}_1') = \rook_1 (\xi_1', \hat{\xi}_1') = 1 = -p_1(\xi,\xi_1')
\]
while
\[
\rook_1 (\xi, \mu_1) = \rook_1 (\xi_1', \mu_1) =  -1 = p_1(\xi,\xi_1').
\]
Thus, $\xi\not\in E^\pm(\xi_1')$.
\end{ex}

\begin{figure}[h]
\centering
    \begin{tikzpicture}[scale=0.3]
    \draw[step=8cm,black] (0,0) grid (16,16);
    
    \foreach \i in {0,...,2}{
        \foreach \j in {0,...,2}{
        \draw[black, fill=black] (8*\i,8*\j) circle (2ex);
        }
    }
    
    \foreach \i in {1,...,3}{
        \draw(-1,8*\i -8) node{$\i$}; 
    }
    \foreach \i in {0,...,2}{
        \draw(8*\i,-1) node{$\i$}; 
    }
    
    \foreach \i in {0,...,1}{
        \draw[->, blue, thick] (4+8*\i,0.1) -- (4+8*\i,2.1); 
        \draw[->, blue, thick] (4+8*\i,15.9) -- (4+8*\i,13.9); 
    }
    \foreach \i in {1,...,2}{
        \draw[->, blue, thick] (0.1, 4+8*\i -8) -- (2.1,4+8*\i -8); 
    }
    \foreach \i in {0,...,1}{
        \draw[->, blue, thick] (4+8*\i,5.9) -- (4+8*\i,7.9); 
        \draw[->, blue, thick] (4+8*\i,8.1) -- (4+8*\i,10.1);
    }
        \draw[->, blue, thick] (8.1,4) -- (10.1,4);
        \draw[->, blue, thick] (5.9,4) -- (7.9,4);
        \draw[->, blue, thick] (13.9,4) -- (15.9,4); 
    \foreach \i in {2}{
        \draw[->, blue, thick] (7.9,4+8*\i -8) -- (5.9,4+8*\i -8); 
         \draw[->, blue, thick] (10.1,4+8*\i -8) -- (8.1,4+8*\i -8);
         \draw[->, blue, thick] (15.9,4+8*\i -8) -- (13.9,4+8*\i -8);  
    }
    
    \draw(4,12) node{$\mu_0$};
    \draw(12,12)  node{$\mu_1$};
    \draw(4,4) node{$\hat\xi_0'$};
    \draw(12,4)  node{$\hat\xi_1'$};
    \draw(5.5,7) node{$\xi'_0$};
    \draw(8.5,14)  node{$\tau$};
    \draw(8.5,2)  node{$\hat\xi$};
    \draw(13.5,7)  node{$\xi'_1$};
    \draw(9,9)  node{$\xi$};
    \end{tikzpicture}
\caption{A subcomplex and associated wall labeling from Figure~\ref{fig:wall_labeling} (B).
Observe that $\xi = \defcell{1}{2}{0}{0}$ belongs to the exit and entrance set of $\hat{\xi} = \defcell{1}{1}{0}{1}$ and $\tau = \defcell{1}{2}{0}{1}$, respectively.
However, $\xi$ is not in the exit or entrance sets of $\xi_0'=\defcell{0}{2}{1}{0}$ nor $\xi_1'=\defcell{2}{2}{1}{0}$.}
\label{fig:exitset}
\end{figure}
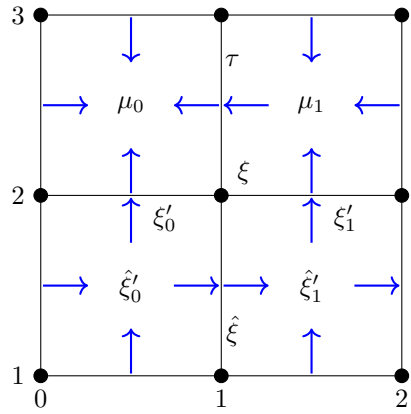

\chapter{Combinatorial dynamics induced by wall labelings}
\label{sec:rules}

To simplify the exposition, \emph{for the remainder of this manuscript we only consider wall labelings that are strongly dissipative} (see Definition~\ref{defn:dissipativewall}).
While this assumption is not necessary, without it there are numerous subtle technicalities that need to be dealt with.
For example, consider Theorem~\ref{thm:dynamics}.
Without the assumption of strongly dissipative, there may not exist a geometrization $\bG$ such that $\lambda(\bone)$ is aligned with $f$ and hence $X=\lambda(\bone)$ need not be a trapping region.
This in turn implies that we need to understand the behavior of trajectories that pass through the boundary of $X$. 
While, this can be done for specific examples we do not provide a general algorithm.

Throughout this chapter we assume that we are given an $N$-dimensional cubical cell complex $\cX = \cX(\I)$, a fixed strongly dissipative  wall labeling $\omega\colon W(\cX)\to \setof{\pm 1} $, and the derived rook field $\rook \colon TP(\cX)\to \setof{0,\pm1}^N$.
We use this information to define a hierarchy of \emph{multivalued maps} $\cF\colon \cX\mvmap \cX$ satisfying 
\begin{equation}
\label{eq:mvmapdefn}
    \emptyset \neq \cF(\xi)\subset \cX,\quad\text{for all $\xi\in\cX$,}
\end{equation}
that provide progressively finer models of the associated dynamics associated with $\rook$.

\begin{defn}\label{def:refinement}
A \emph{refinement} of a multivalued map  $\cF\colon \cX\mvmap \cX$ is any multivalued map  $\cF' \colon \cX\mvmap \cX$ such that 
$\emptyset \neq \cF'(\xi) \subseteq \cF(\xi)$ for all $\xi \in \cX$. 
\end{defn}

The defining condition \eqref{eq:mvmapdefn} of a multivalued map is extremely weak.
For the purposes of this paper we want to restrict our attention to multivalued maps that are reasonable models for ODEs.
This leads to the following definition.

\begin{defn}
\label{def:Rule0}    
The \emph{trivial multivalued map} is given by 
\[
\cF_0 \colon \cX\mvmap \cX,
\]
where $\xi \in \cF_0(\xi')$ and $\xi' \in \cF_0(\xi)$ whenever $\xi \preceq \xi'$ and $\dim(\xi')-\dim(\xi) \leq 1$. 
\end{defn}

The following notation is convenient.
\begin{defn}
\label{def:darrow}
Given a multivalued map $\cF \colon \cX \mvmap \cX$ we say that there is a \emph{double edge} between $\xi,\xi'\in\cX$ if $\xi \in \cF(\xi')$ and $\xi' \in \cF (\xi)$.
We denote the existence of such a double edge by $\xi \darrow_\cF \xi'$ with the convention that $\xi\preceq \xi'$.
\end{defn}

Observe that from the trivial combinatorial multivalued map $\cF_0$, if $\xi \preceq \xi'$ and $\dim(\xi')-\dim(\xi)\leq 1$,
then $\xi \darrow_{\cF_0} \xi'$. 
In Chapter~\ref{sec:R0} it is shown that $\cF_0$ can be used to identify that the ODEs considered in this paper (see Chapter~\ref{sec:ramp}) have a global attractor in the positive orthant and that this global attractor contains a fixed point.
While this is a correct statement, it provides limited information about the structure of the dynamics on the attractor. 
In the applications the presence  of a double edge $\xi \darrow_{\cF} \xi'$ allows for the existence of a recurrent set for the ODE with trajectories that oscillate between regions identified by $\xi$ and $\xi'$.
In the  ODEs described in Chapter~\ref{sec:ramp} this is not the case.
Therefore, the challenge is to modify $\cF_0$ while still correctly capturing the dynamics of the ODE. 
Thus, the remainder of this chapter focuses on the following problem: given a face relation $\xi \preceq \xi'$, when do we insist that either $\xi \in \cF(\xi')$ or $\xi' \in \cF(\xi)$, but not both? 

Sections~\ref{sec:local-conditions} and \ref{sec:quasi-local} provides conditions that allow for the removal of  edges, thus resolving the ambiguity of some double edges.  
Section~\ref{sec:cycle-condition}  resolves some double edges, but also add new edges. 
We provide minimal justification in this Chapter for the imposed conditions other than to assert that they are tied to the dynamics of the ODEs of interest.
However, the exact relationship between vector fields and our combinatorial representation is subtle.
Thus, for the sake of clarity, the focus of this Chapter is on the combinatorics.
The justification of these conditions is presented in Part~\ref{part:III} in the context of ramp systems.

\section{Local conditions}
\label{sec:local-conditions}

In this section, we impose a set of conditions that allows us to remove some edges from the map $\cF_0$ to construct a map $\cF_1$. 
This new map allows for richer information about the dynamics of the ODEs and is obtained by two conditions.
The first, given by Definition~\ref{def:Rule1.1}, says that if a cell has a gradient direction, then it should not contain any recurrent dynamics of the ODE. 
The second, given by Definition~\ref{def:Rule1.2}, is an expanded version of the same observation, and is applied to pairs of cells that are adjacent with respect to the face relation and whose dimensions differ by one.

\begin{defn}
\label{def:Rule1.1}
    Define $\cF_{1.1}\colon \cX\mvmap\cX$ to be the maximal combinatorial multivalued map that is a refinement of $\cF_0$ and satisfies Condition 1.1.
    \begin{description}
        \item[Condition 1.1]  Let $\xi\in\cX$. If $G(\xi) \neq \emptyset$ (see \ref{def:namedirections}), then  $\xi \notin \cF_{1.1}(\xi)$.
    \end{description}
\end{defn}

\begin{defn}
\label{def:Rule1.2}
Define $\cF_{1.2} \colon \cX\mvmap \cX$ to be the maximal combinatorial multivalued map that is a refinement of $\cF_0$ and satisfies Condition 1.2.

\begin{description}
\item[Condition 1.2] Let $\xi\prec\xi'\in\cX$ and assume $\dim(\xi') - \dim(\xi)=1$.
Recall from Definition~\ref{def:exit_face} that $E^-(\xi')$ and $E^+(\xi')$ denote the exit and entrance faces of $\xi'$, respectively.
\begin{enumerate}
    \item[(i)] If $\xi\in E^-(\xi')$,  then $\xi'\not\in \cF_{1.2}(\xi)$.
    \item[(ii)] If $\xi \in E^+(\xi')$,  then $\xi\not\in \cF_{1.2}(\xi')$.
\end{enumerate}
\end{description}
\end{defn}

\begin{prop}
\label{prop:arrowExists}
If $\xi \prec \xi'$ and $\dim(\xi')-\dim(\xi) = 1$, then $\xi \in \cF_{1.2}(\xi')$ or $\xi' \in \cF_{1.2}(\xi)$.
\end{prop}

\begin{proof}
If $\xi \darrow_{\cF_{1.2}}\xi'$, then we are done, and if not then Condition 1.2 is satisfied.
The contrapositive of Condition 1.2 implies that if $\xi \not\in \cF_{1.2}(\xi')$, then  $\xi \in E^+(\xi')$, and if $\xi' \not\in \cF_{1.2}(\xi)$, then $\xi \in \cF_{1.2}(\xi')$.
\end{proof}

\begin{defn}
    \label{def:Rule1}
    Define $\cF_1\colon \cX\mvmap \cX$ by 
    \begin{equation}
    \label{eq:defnF1}
        \cF_1(\xi) := \cF_{1.1}(\xi)\cap \cF_{1.2}(\xi).
    \end{equation}
\end{defn}

Proposition~\ref{prop:F1-welldefined}, whose proof makes use of the following lemma, shows that $\cF_1$ satisfies $\eqref{eq:mvmapdefn}$. 
\begin{lem}
    \label{lem:gradient-implies-non-empty}
    Let $\xi \in \cX$. If $G(\xi) \neq \emptyset$, then there exists $\xi' \in \cX$ with $|\dim(\xi')-\dim(\xi)|=1$ such that $\xi' \in E^-(\xi)$ or $\xi \in E^+(\xi')$. 
\end{lem}
\begin{proof}
For $\xi = [\bv,\bw]$, we describe $\xi'=[\bv',\bw'] \in \cX$ with $\dim(\xi')=\dim(\xi)\pm 1$ such that $\xi'\in E^\pm(\xi)$. Let $n \in G(\xi)$. Then by \eqref{eq:Rn} and Proposition~\ref{prop:Direction-Decomposition}, 
\[
    R_n(\xi) = \setdef{ \rook_n(\xi,\mu) }{ \mu \in \Top_\cX(\xi) } = \setof{r_n}\quad\text{where $r_n =\pm 1$.}
\]
If $n \in J_i(\xi)$, then define $\xi' \coloneqq \left[\bv-\frac{1-r_n}{2}\bzero^{(n)},\bw+\bzero^{(n)} \right]$ so that $\xi \prec \xi'$ with $\Ex(\xi,\xi')=\setof{n}$. For any $\mu \in \Top_\cX(\xi')$, 
\begin{align*}
p_n(\xi,\xi') & =(-1)^{(\bv_n-\frac{1+r_n}{2})-\bv_n}((\bw_n+1)-\bw_n) = (-1)^{-\frac{1+r_n}{2}} \\ 
& = -r_n = -\rook_n(\xi',\mu) = -\rook_n(\xi,\mu)
\end{align*}
where the last equality follows from Corollary~\ref{cor:mu*}.
Thus, by Definition~\ref{def:exit_face} $\xi\in E^+(\xi')$.

Similarly, if $n \in J_e(\xi)$, then define $\xi' \coloneqq \left[\bv+\frac{1+r_n}{2}\bzero^{(n)},\bw-\bzero^{(n)} \right] \prec \xi$. We leave it to the reader verify that $\xi' \in E^-(\xi)$.
\end{proof}

\begin{prop}
    \label{prop:F1-welldefined}
    If $\cF_1 \colon \cX \mvmap \cX$ is given by \eqref{eq:defnF1}, then $\cF_1(\xi) \neq \emptyset$ for all $\xi \in \cX$.
\end{prop}
\begin{proof}
    Notice that $\xi \in \cF_{1.2}(\xi)$ for any $\xi \in \cX$, so if $G(\xi) = \emptyset$, then $\xi \in \cF_{1.1}(\xi)$ too. Thus, $\xi \in \cF_1(\xi)$ and $\cF_1(\xi) \neq \emptyset$.

    If $G(\xi)\neq\emptyset$, then by Lemma~\ref{lem:gradient-implies-non-empty} there exists $\xi' \in \cX$, $|\dim(\xi')-\dim(\xi)|=1$, such that $\xi' \in E^-(\xi)$ or $\xi \in E^+(\xi')$. In either case, $\xi'\in \cF_{1}(\xi)$, so $\cF_1(\xi)\neq\emptyset$.
\end{proof}

\begin{ex}
\label{ex:F1}
Recall the wall labeling and rook field of Figure~\ref{fig:wall_labeling}.    
The associated directed graph $\cF_1\colon \cX \mvmap \cX$ is shown in Figure~\ref{fig:F1mvmap}.
To emphasize that $\cF_1$ is a directed graph, we represent the cell complex $\cX$ via a collection of vertices; one vertex for each cell. 
It is left to the reader to check that $G(\xi)\neq \emptyset$ for every cell $\xi$ except $\defcell{0}{2}{1}{1}$, $\defcell{2}{0}{1}{1}$, and $\defcell{2}{2}{0}{0}$.
Thus, by Condition 1.1 only these latter three cells have self edges in $\cF_{1.2}$.    
We indicate the existence of a self edge by drawing the vertex in blue (see Figure~\ref{fig:F1mvmap}).
A vertex without a self edge is drawn in black.

The double edges $\xi \darrow_{\cF_1}\xi'$ that remain in Figure~\ref{fig:F1mvmap} when $\xi \neq \xi'$ are all of the form $\xi\in \cX^{(0)}$ and $\xi'\in \cX^{(1)}$.
It is left to the reader to check that these are exactly the cases where $\xi\not\in E^\pm(\xi')$, i.e., when $\cF_{1.2}$ cannot be applied.
\end{ex}

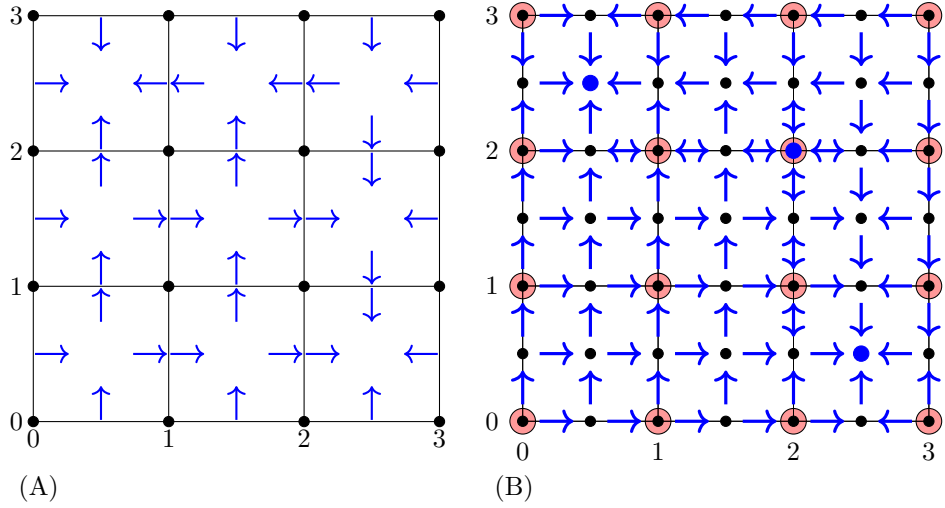
\begin{figure}
\begin{picture}(400,210)(0,0)

\put(10,0){(A)}
\put(0,15){
\begin{tikzpicture}[scale=0.225]
\draw[step=8cm,black] (0,0) grid (24,24);

\foreach \i in {0,...,3}{
    \foreach \j in {0,...,3}{
    \draw[black, fill=black] (8*\i,8*\j) circle (2ex);
    }
}

\foreach \i in {0,...,3}{
    \draw(8*\i,-1) node{$\i$}; 
    \draw(-1,8*\i) node{$\i$}; 
}

\foreach \i in {0,...,2}{
    \draw[->, blue, thick] (4+8*\i,0.1) -- (4+8*\i,2.1); 
    \draw[->, blue, thick] (4+8*\i,23.9) -- (4+8*\i,21.9); 
    \draw[->, blue, thick] (0.1, 4+8*\i) -- (2.1,4+8*\i); 
    \draw[->, blue, thick] (23.9,4+8*\i) -- (21.9,4+8*\i); 
}
\foreach \i in {0,...,1}{
    \draw[->, blue, thick] (4+8*\i,5.9) -- (4+8*\i,7.9); 
    \draw[->, blue, thick] (4+8*\i,8.1) -- (4+8*\i,10.1);
    \draw[->, blue, thick] (4+8*\i,13.9) -- (4+8*\i,15.9); 
    \draw[->, blue, thick] (4+8*\i,16.1) -- (4+8*\i,18.1);  
    \draw[->, blue, thick] (5.9,4+8*\i) -- (7.9,4+8*\i); 
    \draw[->, blue, thick] (8.1,4+8*\i) -- (10.1,4+8*\i);
    \draw[->, blue, thick] (13.9,4+8*\i) -- (15.9,4+8*\i); 
    \draw[->, blue, thick] (16.1,4+8*\i) -- (18.1,4+8*\i);  
}
\foreach \i in {2}{
    \draw[->, blue, thick] (4+8*\i,7.9) -- (4+8*\i,5.9); 
    \draw[->, blue, thick] (4+8*\i,10.1) -- (4+8*\i,8.1);
    \draw[->, blue, thick] (4+8*\i,18.1) -- (4+8*\i,16.1); 
    \draw[->, blue, thick] (4+8*\i,15.9) -- (4+8*\i,13.9);  
    \draw[->, blue, thick] (7.9,4+8*\i) -- (5.9,4+8*\i); 
    \draw[->, blue, thick] (10.1,4+8*\i) -- (8.1,4+8*\i);
    \draw[->, blue, thick] (18.1,4+8*\i) -- (16.1,4+8*\i); 
    \draw[->, blue, thick] (15.9,4+8*\i) -- (13.9,4+8*\i);  
}
\end{tikzpicture}
}

\put(190,0){(B)}
\put(180,10){
    \begin{tikzpicture}[scale=0.225]
    \foreach \i in {0,...,3}{
        \draw(8*\i,-1.8) node{$\i$}; 
        \draw(-1.8,8*\i) node{$\i$}; 
       }
    \draw[step=8cm, black, ultra thin] (0,0) grid (24,24);
    \foreach \i in {0,...,3}{
      \foreach \j in {0,...,3}{
        \draw[black, fill=red, fill opacity=0.4] (8*\i,8*\j) circle (5ex);
      }
    }
    \draw[step=8cm,black, ultra thin] (0,0) grid (24,24);
    \foreach \i in {0,...,6}{
      \foreach \j in {0,...,6}{
        \draw[black, fill=black] (4*\i,4*\j) circle (2ex);
      }
    }
        \draw[blue, fill=blue] (4,20) circle (3ex);
        \draw[blue, fill=blue] (20,4) circle (3ex);
        \draw[blue, fill=blue] (16,16) circle (3ex);
    \foreach \i in {0,...,6}{
        \draw[->, blue, very thick] (1, 4*\i) -- (3, 4*\i);  
    }
    \foreach \i in {0,...,3}{
        \draw[->, blue, very thick] (5, 4*\i) -- (7, 4*\i);  
        \draw[->, blue, very thick] (9, 4*\i) -- (11, 4*\i);  
        \draw[->, blue, very thick] (13, 4*\i) -- (15, 4*\i);  
        \draw[->, blue, very thick] (17, 4*\i) -- (19, 4*\i);  
     }
    \foreach \i in {0,...,6}{
        \draw[->, blue, very thick] (23, 4*\i) -- (21, 4*\i);  
    }
    \foreach \i in {1,...,4}{
        \draw[->, blue, very thick] (3+4*\i,24) -- (1+4*\i,24);  
        \draw[->, blue, very thick] (3+4*\i,20) -- (1+4*\i,20);  
    }
    \foreach \i in {0,...,6}{
        \draw[->, blue, very thick] (4*\i,23) -- (4*\i, 21);  
    }
    \foreach \i in {1,...,4}{
        \draw[->, blue, very thick] (20, 3+4*\i) -- (20,1+4*\i);  
        \draw[->, blue, very thick] (24, 3+4*\i) -- (24,1+4*\i);  
    }
    \foreach \i in {0,...,6}{
        \draw[->, blue, very thick] (4*\i,1) -- (4*\i, 3);  
    }
    \foreach \i in {0,...,3}{
        \draw[->, blue, very thick] (4*\i,5) -- (4*\i, 7);  
        \draw[->, blue, very thick] (4*\i,9) -- (4*\i, 11);  
        \draw[->, blue, very thick] (4*\i,13) -- (4*\i, 15);  
        \draw[->, blue, very thick] (4*\i,17) -- (4*\i, 19);  
    }
    \foreach \i in {1,...,4}{
    \draw[<->, blue, very thick] (3+4*\i,16) -- (1+4*\i,16);  
    }
    \foreach \i in {1,...,4}{
    \draw[<->, blue, very thick] (16, 3+4*\i) -- (16,1+4*\i);  
    }
    \end{tikzpicture}
}
\end{picture}
\caption{
(A) The vector representation of the wall labeling $\omega$ of Figure~\ref{fig:wall_labeling}.
(B)
Directed graph associated with $\cF_1 \colon \cX \mvmap \cX$ given wall labeling in (A). Each cell in $\cX$ (vertices, edges, and faces) is represented by a black or blue dot and are the nodes of the directed graph. The vertices of the cell complex $\cX$ are circled in red. Every black vertex is associated to a cell that has a gradient direction and hence by $\cF_{1.1}$ has no self-edge.
Large blue dots, corresponding to the cells
$\defcell{0}{2}{1}{1}$, $\defcell{2}{0}{1}{1}$, and $\defcell{2}{2}{0}{0}$
of $\cX$,
indicate vertices for which there are self-edges.
The remaining edges are removed by $\cF_{1.2}$.}
\label{fig:F1mvmap}
\end{figure}

\begin{prop}
    \label{prop:Rule1-DArrow-Implies-Opaque} 
    Let $\xi,\xi' \in \cX$ such that $\xi \prec \xi'$. If $\xi \darrow_{\cF_1} \xi'$, then $\Ex(\xi,\xi')=\setof{n}\subseteq O(\xi)$. 
\end{prop}
\begin{proof}
    Definition~\ref{def:Rule0} states that for $\xi,\xi' \in \cX$, $\xi \prec \xi'$, there is a double edge $\xi \darrow_{\cF_0}\xi'$ under $\cF_0$ if and only if $Ex(\xi,\xi')=\setof{n}$. Note that $n \in J_i(\xi)$, so by Proposition~\ref{prop:Direction-Decomposition} it follows that $n \in O(\xi)$ or $n \in G(\xi)$. Suppose for the sake of contradiction that $n \in G(\xi)$, so either $R_n(\xi)=\setof{1}$ or $R_n(\xi)=\setof{-1}$. Then either $\xi \in E^+(\xi')$ or $\xi \in E^-(\xi)$, and thus either $\xi' \notin \cF_1(\xi)$ or $\xi \notin \cF_1(\xi')$, contradicting the hypothesis.
\end{proof}

\section{Quasi-local conditions}
\label{sec:quasi-local}

Typically, e.g., Figure~\ref{fig:F1mvmap}, the directed graph $\cF_1$ contains double edges $\xi \darrow_{\cF_1} \xi'$.
As indicated in the introduction to this chapter, if $\cF_1$ were the optimal multivalued map that could characterize the dynamics, then this suggests the existence of recurrent dynamics between the geometric realizations of $\xi$ and $\xi'$.
With this in mind, consider $\xi = \defcell{1}{2}{0}{0}
$ and
    $\xi' = \defcell{1}{2}{1}{0}
$ in Figure~\ref{fig:F1mvmap}.
Observe that $\setof{2} \in G(\xi)\cap G(\xi')$.
This implies that there is a consistent nontrivial gradient direction associated within these cells that should prevent a recurrent set.
The objective of this section is to identify this condition and propose a modification to $\cF_1$ that replaces the double edge with a single edge. 

Recall from Definition~\ref{defn:Ex} that $\Ex(\xi,\xi')$ denotes the extension of $\xi$ in $\xi'$, that is, the directions that are inessential in $\xi$ but essential in $\xi'$. 

\begin{defn}
\label{defn:GO-pair}
Let $(\xi,\xi') \in \cX \times \cX$ be such that $\xi \prec \xi'$. We say that $(\xi,\xi')$ has a \emph{GO-pair} $(n_g,n_o) \in \setof{1,\ldots,N} \times \setof{1,\ldots,N}$ if $\Ex(\xi,\xi')=\setof{n_o}$ and if there exists $n_g \in G(\xi')$ and $n_g$-adjacent top cells $\mu,\mu' \in \Top_\cX(\xi')$ such that
\[
\rook_{n_o}(\xi,\mu) \neq \rook_{n_o}(\xi,\mu').
\]
\end{defn}

\begin{ex}
\label{ex:indecisive_drift}
Figure~\ref{fig:indecisive_drift_2} depicts the wall labeling from Figure~\ref{fig:wall_labeling}(b) restricted to cells with vertices in $\setof{1,2,3}\times \setof{0,1,2}$ as in Example~\ref{ex:exitface}. 
Observe that $G(\xi)=\activeset(\xi)=\setof{2}$ and $\rmap\xi(2)=1 \in O(\xi)$ as in Example~\ref{ex:cyclingcells}. So there is $2 \in G_i(\xi_k')$, $k=0,1$, that actively regulates $1 \in O(\xi)$ with $\Ex(\xi,\xi_k')=\setof{1}$, therefore $(\xi,\xi_k')$, $k=0,1$, both have a GO-pair $(2,1)$.  
\end{ex}

\begin{figure}[h]
    \begin{picture}(400,140)
    \put(0,0){
    \begin{tikzpicture}[scale=0.25]
    \draw[step=8cm,black] (0,0) grid (16,16);
    
    \foreach \i in {0,...,2}{
        \foreach \j in {0,...,2}{
        \draw[black, fill=black] (8*\i,8*\j) circle (2ex);
        }
    }
    
    \foreach \i in {1,...,3}{
        \draw(-1,8*\i -8) node{$\i$}; 
    }
    \foreach \i in {0,...,2}{
        \draw(8*\i,-1) node{$\i$}; 
    }
    
    \foreach \i in {0,...,1}{
        \draw[->, blue, thick] (4+8*\i,0.1) -- (4+8*\i,2.1); 
        \draw[->, blue, thick] (4+8*\i,15.9) -- (4+8*\i,13.9); 
    }
    \foreach \i in {1,...,2}{
        \draw[->, blue, thick] (0.1, 4+8*\i -8) -- (2.1,4+8*\i -8); 
    }
    \foreach \i in {0,...,1}{
        \draw[->, blue, thick] (4+8*\i,5.9) -- (4+8*\i,7.9); 
        \draw[->, blue, thick] (4+8*\i,8.1) -- (4+8*\i,10.1);
    }
        \draw[->, blue, thick] (8.1,4) -- (10.1,4);
        \draw[->, blue, thick] (5.9,4) -- (7.9,4);
        \draw[->, blue, thick] (13.9,4) -- (15.9,4); 
    \foreach \i in {2}{
        \draw[->, blue, thick] (7.9,4+8*\i -8) -- (5.9,4+8*\i -8); 
         \draw[->, blue, thick] (10.1,4+8*\i -8) -- (8.1,4+8*\i -8);
         \draw[->, blue, thick] (15.9,4+8*\i -8) -- (13.9,4+8*\i -8);  
    }
    
    \draw(4,12) node{$\mu_0$};
    \draw(12,12)  node{$\mu_1$};
    \draw(4,4) node{$\hat\xi_0'$};
    \draw(12,4)  node{$\hat\xi_1'$};
    \draw(5.5,7) node{$\xi'_0$};
    \draw(8.5,14)  node{$\tau$};
    \draw(8.5,2)  node{$\hat\xi$};
    \draw(13.5,7)  node{$\xi'_1$};
    \draw(9,9)  node{$\xi$};
    \end{tikzpicture}
    }
    \end{picture}
    \caption{A subcomplex and associated wall labeling from Figure~\ref{fig:wall_labeling}(b).
    The pairs $(\xi,\xi'_0) = \left( \defcell{1}{2}{0}{0},\defcell{0}{2}{1}{0} \right)$ and $(\xi,\xi'_1) =\left( \defcell{1}{2}{0}{0},\defcell{1}{2}{1}{0} \right)$ have GO-pair $(2,1)$ and exhibit indecisive drift. There are unique back walls given by $(\hat{\xi},\hat{\xi}_0')=\left( \defcell{1}{1}{0}{1}, \defcell{0}{1}{1}{1} \right)$ and $(\hat{\xi},\hat{\xi}_1')=\left( \defcell{1}{1}{0}{1},\defcell{1}{1}{1}{1} \right)$.}
    \label{fig:indecisive_drift_2}
\end{figure}
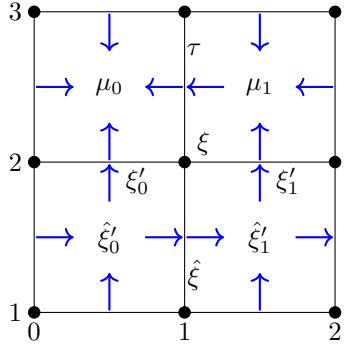

\begin{prop}
\label{prop:GO-pair-properties}
If $(\xi,\xi') \in \cX \times \cX$ has a GO-pair $(n_g,n_o)$, then
\begin{enumerate}
\item $\xi\darrow_{\cF_1}\xi'$,
\item $n_o \in O(\xi)$,
\item $n_g \in G_i(\xi')$,
\item $n_g \neq n_o$,
\item $n_g$ actively regulates $n_o$ at $\xi'$,
\item $n_g$ actively regulates $n_o$ at $\xi$.
\end{enumerate}
\end{prop}

\begin{proof}
Assume that $(\xi,\xi') \in \cX \times \cX$ has a GO-pair. By Definition~\ref{defn:GO-pair}, that means that $\xi \preceq \xi'$ and $\Ex(\xi,\xi')=\setof{n_o}$, hence $\xi \preceq \xi'$ and $\dim(\xi')-\dim(\xi)=1$. Therefore, $\xi\darrow_{\cF_0}\xi'$. 

To verify that $\xi\darrow_{\cF_1}\xi'$, we need to check that Conditions 1.1 and 1.2 of Definitions~\ref{def:Rule1.1} and \ref{def:Rule1.2} do not apply. Note that Condition 1.1 only applies when $\xi=\xi'$, which is not the case, so we're left with Condition 1.2.

We prove that Condition 1.2 does not apply by contradiction. Suppose for the sake of contradiction that it does. Then either $\xi \notin \cF_{1.2}(\xi')$ or $\xi' \notin \cF_{1.2}(\xi)$, i.e., $\xi$ is either an entrance or exit face of $\xi'$. That means that for every top cell $\mu \in \Top_{\cX}(\xi')$,
$\rook_{n_o}(\xi,\mu) = p_{n_o}(\xi,\xi')$,
or for every top cell $\mu \in \Top_{\cX}(\xi')$, 
$\rook_{n_o}(\xi,\mu) = -p_{n_o}(\xi,\xi')$.
In either case, $\rook_{n_o}(\xi,\mu)=\rook_{n_o}(\xi,\mu')$ for any $\mu,\mu' \in \Top_\cX(\xi')$. In particular, $\rook_{n_o}(\xi,\mu)=\rook_{n_o}(\xi,\mu')$ for any $n_g$-adjacent top cells $\mu,\mu' \in \Top_\cX(\xi')$, which contradicts the assumption that $(\xi,\xi')$ has a GO-pair. Therefore, Condition 1.2 does not apply and $\xi \in \cF_{i}(\xi')$ and $\xi' \in \cF_{i}(\xi)$ for both $i=1.1,1.2$. Thus, $\xi \in \cF_1(\xi')$ and $\xi' \in \cF_1(\xi)$, which yields $\xi\darrow_{\cF_1}\xi'$ as desired in (1). 

Given that $\xi\darrow_{\cF_1}\xi'$, (2) follows from Proposition~\ref{prop:Rule1-DArrow-Implies-Opaque} since $\setof{n_o}=\Ex(\xi,\xi')\subseteq O(\xi)$ and $\setof{n_o}=\Ex(\xi,\xi')\subseteq O(\xi')$.

To prove (3), note that $n_g \in G(\xi')$ by Definition~\ref{defn:GO-pair}, so it remains to show that $n_g \in J_i(\xi')$. Let $\mu,\mu'\in\Top_\cX(\xi')$ be the $n_g$-adjacent top cells that satisfy $\rook_{n_o}(\xi,\mu)\neq \rook_{n_o}(\xi,\mu')$. If $\xi_{n_g} \in \cX^{(N-1)}$ is such that $\xi' \preceq \mu$ and $\xi_{n_g} \preceq \mu'$, then $\xi'\preceq \xi_{n_g}$. Therefore, by Proposition~\ref{prop:JtauJsigma}, $\setof{n_g}=J_i(\xi_{n_g})\subseteq J_i(\xi')$, so $n_g \in G(\xi')\cap J_i(\xi') = G_i(\xi)$.

Note that (4) follows directly from (3), since we have $n_g \in J_i(\xi')$ and $n_o \in J_e(\xi')$.

Observe that by Proposition~\ref{prop:mu*}, $\mu_{n_g}^*(\xi,\mu) = \xi_{n_g} = \mu_{n_g}^*(\xi,\mu)$, so the Definition~\ref{def:rookfield} of $\rook$ yields the equalities in
\[
\rook_{n_o}(\xi_{n_g},\mu) = \rook_{n_o}(\xi,\mu) \neq \rook_{n_o}(\xi,\mu') = \rook_{n_o}(\xi_{n_g},\mu').
\]
Thus, $(\xi_{n_g},\mu),(\xi_{n_g},\mu') \in W(\xi')$ are $n_g$-walls that differ at the $n_o$-entry of $\rook$ which, by Definition~\ref{def:active_regulation}, means that $n_g$ actively regulates $n_o$.

Finally, note that (6) follows directly from (5) and Proposition~\ref{prop:act_regulation_faces}.
\end{proof}

\begin{defn}
\label{defn:indecisive}
Let $(\xi,\xi') \in \cX \times \cX$ be such that $\xi \prec \xi'$. We say that $(\xi,\xi')$ exhibits \emph{indecisive drift} if $(\xi,\xi')$ has a GO-pair $(n_g,n_o)$ and $\rmap\xi^{-1}(n) = \setof{n}$ for all $n \in O_i(\xi) \setminus \setof{n_o}$.

Given a rook field $\rook$ we denote the set of pairs $(\xi,\xi')$ that exhibit indecisive drift by $\cD(\rook)$.
\end{defn}

\begin{prop}
\label{prop:unique_ext_GO_pairs}
Assume $(\xi_0, \xi_0'), (\xi_1, \xi_1') \in \cD(\rook)$.
If $\xi_0 \preceq \xi_1$ or $\xi_0' \preceq \xi_1'$, then $\Ex(\xi_0, \xi_0') = \Ex(\xi_1, \xi_1')$.
\end{prop}

\begin{proof}
Let $(n_g, n_o)$ and $(n'_g, n'_o)$ be GO-pairs for $(\xi_0, \xi_0')$ and $(\xi_1, \xi_1')$, respectively. It follows from Proposition~\ref{prop:GO-pair-properties} that $n_g$ actively regulates $n_o$ at $\xi_0$ and $n'_g$ actively regulates $n'_o$ at $\xi_1$, that is, $\rmap{\xi_0}(n_g) = n_o$ and $\rmap{\xi_1}(n'_g) = n'_o$.

We will show that if $\xi_0 \preceq \xi_1$ or $\xi_0' \preceq \xi_1'$, then $\rmap{\xi_0}(n'_g) = n'_o$. The result then follows, since this implies that $\rmap{\xi_0}^{-1}(n_o) \neq \setof{n_o}$ and $\rmap{\xi_0}^{-1}(n'_o) \neq \setof{n'_o}$, 
which is contradiction with the uniqueness condition in Definition~\ref{defn:indecisive} if $n_o \neq n'_o$. Therefore we must have $n_o = n'_o$ and hence
\[
\Ex(\xi_0, \xi_0') = \setof{n_o} = \setof{n'_o} = \Ex(\xi_1, \xi_1').
\]
Assume that $\xi_0 \preceq \xi_1$.  
Since $\rmap{\xi_1}(n'_g) = n'_o$ by  Proposition~\ref{prop:act_regulation_faces} $\rmap{\xi_0}(n'_g) = n'_o$ as desired. 

Now suppose that $\xi_0' \preceq \xi_1'$. Since $(n'_g, n'_o)$ is a GO-pair for $(\xi_1, \xi_1')$, Proposition~\ref{prop:GO-pair-properties} implies that $n'_g$ actively regulates $n'_o$ at $\xi_1'$. It follows from Definition~\ref{def:active_regulation} that there exist $n'_g$-walls $(\xi_{n'_g}, \mu), (\xi_{n'_g}, \mu') \in W(\xi_1') \subseteq W(\xi_0')$ such that $\rook_{n'_o}(\xi_{n'_g}, \mu) \neq \rook_{n'_o}(\xi_{n'_g}, \mu')$. This implies that $(\xi_{n'_g}, \mu), (\xi_{n'_g}, \mu') \in W(\xi_0') \subseteq W(\xi_0)$ are $n'_g$-walls of $\xi_0$ with $\rook_{n'_o}(\xi_{n'_g}, \mu) \neq \rook_{n'_o}(\xi_{n'_g}, \mu')$. Since we also have that $n'_g \in J_i(\xi_1') = J_i(\xi_0') \subset J_i(\xi_0)$ it follows that $n'_g$ actively regulates $n'_o$ at $\xi_0$, that is, $\rmap{\xi_0}(n'_g) = n'_o$ as desired.
\end{proof}

\begin{prop}
\label{prop:not_L_GO_pairs}
There are no triples $\xi \prec \xi' \prec \xi''$ in $\cX$ such that $(\xi,\xi')$ and $(\xi',\xi'')$ exhibit indecisive drift. 
\end{prop}
\begin{proof}
Assume that $(\xi,\xi'),(\xi',\xi'')\in\cD(\rook)$ with GO-pairs $(n_g, n_o)$ and $(n'_g, n'_o)$, respectively.
By definition of a GO-pair (Definition~\ref{defn:GO-pair}) $n_o\in \Ex(\xi,\xi')$, and hence, $n_o \in J_e(\xi')$.
Similarly, $n_o'\in J_i(\xi')$. 
Thus, $n_o \neq n_o'$. 

By Proposition~\ref{prop:GO-pair-properties},  $n_g$ actively regulates $n_o$ at $\xi_0$ and $n'_g$ actively regulates $n'_o$ at $\xi'$. 
By Proposition~\ref{prop:act_regulation_faces}, $n'_g$ actively regulates $n'_o$ at $\xi$ since $\xi \preceq \xi'$, hence 
$n_g'\in \rmap{\xi}^{-1}(n_o')$. 

In contrast, the assumption that $(\xi,\xi')\in\cD(\rook)$ and $n_o' \in O_i(\xi)\setminus\setof{n_o}$, implies that $\rmap\xi^{-1}(n_o')=\setof{n_o'}$. Thus, $n_g' = n_o'$, which contradicts the fact that $(n_g',n_o')$ is a GO-pair. 
\end{proof}

Recall that $\xi\leftrightarrow_{\cF_1}\xi'$ for a pair $(\xi,\xi')$ that exhibits indecisive drift. In order to identify which edge should be removed from $\xi\leftrightarrow_{\cF_1}\xi'$ we draw information from walls that are in backwards direction with respect to the gradient directions of $\xi'$. We begin by defining the set of back walls for a generic pair $(\xi,\xi')$, and then we prove that for a pair $(\xi,\xi')$ that exhibits indecisive drift, the wall labeling information contained in the back walls are independent of the choice of a particular back wall. 

\begin{defn}
\label{defn:back-walls}
Let $(\xi,\xi')\in\cX\times\cX$ such that $\xi\prec\xi'$ and  $\Ex(\xi,\xi')=\setof{n_o}$. If $\xi=[\bv,\bw]$ and $\xi'=[\bv',\bw']$, the set of \emph{back walls} of $(\xi,\xi')$ is defined by
\begin{equation*}
\Back(\xi,\xi') \coloneqq \setdef{ \left( \left[ \bv-\hat{\bv}, \bone^{(n_o)} \right],\left[ \bv'-\hat{\bv}, \bone \right] \right)}{\hat{\bv} = \sum_{n \in J_i(\xi')} \frac{1+r_n}{2} \bzero^{(n)},\ r_n \in R_n(\xi)}.
\end{equation*}
\end{defn}

Two relevant observations that validate Definition~\ref{defn:back-walls} are as follows.
First, by Proposition~\ref{prop:Direction-Decomposition}(ii), $r_n=\pm 1$, and hence $\hat{\bv}_n$ is an integer.
Second if $(\dec,\dec')$ is a back wall for $(\xi,\xi')$, then $(\dec,\dec') \in W(\xi)$, i.e., it is a wall of $\xi$.
\begin{rem}\label{rem:same_position}
Let $(\xi,\xi') \in \cD(\rook)$ and assume that $\Ex(\xi,\xi')=\setof{n_o}$. 
If $(\dec_0,\dec'_0), (\dec_1,\dec'_1)\in \Back(\xi,\xi')$, 
then $p_{n_o}(\dec_0,\dec'_0) = p_{n_o}(\dec_1,\dec'_1)$.
\end{rem}
\begin{ex}
\label{ex:backpairs}
Returning to Example~\ref{ex:indecisive_drift} and Figure~\ref{fig:indecisive_drift_2} consider the pairs $(\xi,\xi'_0) = \left( \defcell{1}{2}{0}{0},\defcell{0}{2}{1}{0} \right)$ and $(\xi,\xi'_1) = \left( \defcell{1}{2}{0}{0}, \defcell{1}{2}{1}{0} \right)$ that have GO-pair $(2,1)$. 
Note that $2 \in G(\xi)$, $1 \in O(\xi)$ and $1 \notin \activeset(\xi)$, so $(\xi,\xi'_0)$ and $(\xi,\xi'_1)$ exhibit indecisive drift. Moreover, $J_i(\xi_k')=\setof{2}$ and $R_{2}(\xi_k') = \setof{1}$ so $\hat{\bv} = \left(\begin{smallmatrix}
    0 \\ 1
\end{smallmatrix}\right)$. Thus, 
\[
    \hat{\xi} = [\bv-\hat{\bv},1^{(1)}] = \defcell{1}{2-1}{0}{1} = \defcell{1}{1}{0}{1}
\]
Similarly, $\hat{\xi}'_0=\setof{\defcell{0}{1}{1}{1}}$ and $\hat{\xi}'_1=\setof{\defcell{1}{1}{1}{1}}$.
Therefore the back walls for $(\xi,\xi'_0)$ and $(\xi,\xi'_1)$ are
\[
\Back(\xi,\xi'_0) = \setof{\left( \defcell{1}{1}{0}{1}, \defcell{0}{1}{1}{1} \right)}
        \ \ \text{and}\ \ 
\Back(\xi,\xi'_1) = \setof{\left( \defcell{1}{1}{0}{1},\defcell{1}{1}{1}{1} \right)}, 
\]
\end{ex}

\begin{prop}\label{prop:back_subset}
Let $(\xi_0, \xi_0'),(\xi_1, \xi_1')\in\cX\times\cX$ satisfy $\xi_0\preceq_\cX \xi_0'\preceq_\cX \xi_1'$, $\xi_0\preceq_\cX\xi_1\preceq_\cX \xi_1'$, and $\Ex(\xi_0, \xi_0')= \Ex(\xi_1, \xi_1') = \setof{n_o}$.
If $\Ex(\xi_0', \xi_1')\subset O_i(\xi_0')$, then 
\[
\Back(\xi_1, \xi_1')\subset \Back(\xi_0, \xi_0').
\]
\end{prop}

\begin{proof}
Let $\xi_0=[\bv,\bw]$, $\xi'_0=[\bv',\bw']$, $\xi_1=[\bu,\tilde\bw]$ and $\xi'_1=[\bu',\tilde\bw']$. 
By definition 
\[
\left( \left[ \bu-\hat\bu, \bone^{(n_o)} \right],\left[ \bu'-\hat\bu, \bone \right] \right)\in \Back(\xi_1,\xi_1')
\]
and 
\[
\left( \left[ \bv-\hat\bv, \bone^{(n_o)} \right],\left[ \bv'-\hat\bv, \bone \right] \right)\in \Back(\xi_0,\xi_0')
\]
if and only if 
\[
\hat\bu = \sum_{n \in J_i(\xi'_1)} \frac{1+\bar{r}_n}{2} \bzero^{(n)}
\]
and
\begin{equation}\label{eq:v_hat}
\hat\bv = \sum_{n \in J_i(\xi'_0)} \frac{1+r_n}{2} \bzero^{(n)} =\sum_{n \in J_i(\xi'_1)} \frac{1+r_n}{2} \bzero^{(n)} + \sum_{n \in \Ex(\xi'_0,\xi'_1)} \frac{1+r_n}{2} \bzero^{(n)}
\end{equation}
where $r_n\in R_n(\xi'_0)$ and $\bar{r}_n\in R_n(\xi'_1)$, and the last equation arises from the assumption that $\xi'_0\preceq_\cX \xi'_1$ and Proposition~\ref{prop:JtauJsigma}.

By hypothesis $\Ex(\xi'_0,\xi'_1)\subset O_i(\xi'_0)\subset O_i(\xi_0)$, where the second containment follows from the assumption that $\xi_0\prec\xi_0'$.

For each $n \in J_i(\xi'_0)$ we will select $r_n\in R(\xi_0)$
such that \[\hat{\bv} = \sum_{n \in J_i(\xi'_0)} \frac{1+r_n}{2} \bzero^{(n)}\] satisfies
\[\left[ \bu-\hat\bu, \bone^{(n_o)} \right] = \left[ \bv-\hat{\bv}, \bone^{(n_o)} \right] 
    \text{ and }
    \left[ \bu'-\hat\bu, \bone \right] = \left[ \bv'-\hat{\bv}, \bone \right].
    \]
This implies that 
    \[\left( \left[ \bu-\hat\bu, \bone^{(n_o)} \right],\left[ \bu'-\hat\bu, \bone \right] \right)\in \Back(\xi_0,\xi_0').\]
    Therefore, $\Back(\xi_1, \xi_1')\subset \Back(\xi_0, \xi_0')$.
Now, we show how to choose the desired $r_n\in R(\xi_0)$.

Recall that $\Ex(\xi'_0,\xi'_1)\subset O_i(\xi'_0)\subset O_i(\xi_0)$,  then $R_n(\xi_0)=\setof{\pm1}$ for any $n\in \Ex(\xi'_0,\xi'_1)$. Since $\xi_0'\preceq_\cX \xi_1'$ implies that $|\bu'_n-\bv'_n |\leq 1$, we can choose $r_n\in R_n(\xi_0)=\setof{\pm1}$ for any $n\in \Ex(\xi'_0,\xi'_1)$ such that
    \[\bu' = \bv' - \sum_{n \in \Ex(\xi'_0,\xi'_1)} \frac{1+r_n}{2} \bzero^{(n)}.\]
    By hypothesis $\Ex(\xi_0,\xi_1)=\Ex(\xi_0',\xi_1')$, we leave it to the reader to check that the same choice of $r_n\in\setof{\pm1}$ for any $n\in \Ex(\xi'_0,\xi'_1)$ provides that  
    \[\bu = \bv - \sum_{n \in \Ex(\xi_0,\xi_1)} \frac{1+r_n}{2} \bzero^{(n)}.\]

    It remains to choose $r_n\in R_n(\xi_0)$ in \eqref{eq:v_hat} such that $\hat\bv_n=\hat\bu_n$ for any $n\in J_i(\xi'_1)$. Notice that, it is sufficient to show that $R_n(\xi_1)\subset R_n(\xi_0)$ for any $n\in J_i(\xi'_1)$. Indeed, since $\xi_0\preceq_\cX \xi_1$, by Proposition~\ref{prop:face_R_n}, it follows that $R_n(\xi_1)\subset R_n(\xi_0)$ for all $n\in J_i(\xi'_1)$.
\end{proof}

\begin{ex}
\label{ex:3D_Dec_pairs} 
We consider GO-pairs, indecisive drift, and back walls in the context of a rook field $\rook$ defined on a three dimensional cell complex $\cX=\cX(\I)$ with $\I=\prod_{n=1}^3 \setof{0,\ldots,K(n)}$.
Let $\xi = [\bv,\bzero]$ where $\bv = (\bv_1,\bv_2,\bv_3)$ and $0 < \bv_n < K(n)$ for each $n=1,2,3$.
For all $\mu\in\Top_\cX(\xi)$ and $n=1,2,3$, set $\rook_n(\xi,\mu) = -1$, \emph{except} for 
\[
\rook_2\left( \xi, \defcellb{v_1}{v_2}{v_3}{1}{1}{1} \right) =
\rook_2\left( \xi, \defcellb{v_1}{v_2-1}{v_3}{1}{1}{1} \right) = 1.
\]
Let 
\[
\xi_0'  = \defcellb{v_1}{v_2-1}{v_3}{0}{1}{0}\quad\text{and}
\quad\xi_1'  = \defcellb{v_1}{v_2}{v_3}{0}{1}{0}
\]
and observe that $\xi \prec \xi_i'$, $i=0,1$. 

We claim that $(\xi,\xi_i')$ has a GO-pair $(n_g,n_o)=(1,2)$.
Observe that $\Ex(\xi,\xi_i') = \setof{2}$.
Since $R_1(\xi) = R_3(\xi) = \setof{-1}$ and $R_2(\xi) = \setof{\pm 1}$, it follows that $G(\xi)= \setof{1,3}$ and $O(\xi) = \setof{2}$.
Let 
\[
\mu_0 = \defcellb{v_1}{v_2-1}{v_3}{1}{1}{1}   \quad\text{and}\quad
\mu_0' = \defcellb{v_1-1}{v_2-1}{v_3}{1}{1}{1}  
\]
and
\[
\mu_1 = \defcellb{v_1}{v_2}{v_3}{1}{1}{1}   \quad\text{and}\quad
\mu_1' = \defcellb{v_1-1}{v_2}{v_3}{1}{1}{1} . 
\]
Then, for $i=0$ or $i=1$, $\mu_i$ and $\mu_i'$ are $1$-adjacent top cells and $\rook_2(\xi,\mu_i)\neq \rook_2(\xi,\mu_i')$.
Therefore, $(1,2)$ is a GO-pair.

We further claim that $(\xi,\xi_i')$ exhibits indecisive drift.
Observe that $\left(O(\xi) \cap J_i(\xi)\right)\setminus \setof{2}  = \emptyset$, thus it is sufficient that $(\xi,\xi_i')$ has a GO-pair.

To determine the set of back walls of $(\xi,\xi_i')$ we note that $J_i(\xi_i') = \setof{1,3}$.
Thus, for $i=0,2$, there is an unique $\hat{\bv} \in \setof{0,1}^3$ given by
\[
\hat{\bv} = \frac{1+r_1}{2}\bzero^{(1)} + \frac{1+r_3}{2}\bzero^{(3)} = \bzero
\]
since $R_1(\xi) = R_3(\xi) = \setof{-1}$.
Thus, by definition the set of back walls of $(\xi,\xi_0')$ and $(\xi,\xi_1')$ are
\[
\Dec(\xi,\xi_0') = \setof{\left( \defcellb{v_1}{v_2}{v_3}{1}{0}{1} ,\defcellb{v_1}{v_2-1}{v_3}{1}{1}{1} \right)}
\]
and
\[
\Dec(\xi,\xi_1') = \setof{\left( \defcellb{v_1}{v_2}{v_3}{1}{0}{1} ,\defcellb{v_1}{v_2}{v_3}{1}{1}{1} \right)}.
\]
\end{ex}
The following results shows that $\rook_{n_o}(\dec,\dec')$ is independent of the choice of back wall for a pair $(\xi, \xi')$ that exhibits indecisive drift. First, we address the structure of $\Back(\xi,\xi')$ when $(\xi,\xi')\in\cD(\rook)$.
\begin{lemma}
    \label{lem:back_wall_structure}
    Let $(\xi,\xi') \in \cX \times \cX$ be a pair that satisfies $\xi \prec \xi'$ and let $\Ex(\xi,\xi')=\setof{n_o}$. 
    
    If $G=(V,E)$ is the graph given by $$V=\setdef{ \dec' \in \cX^{(N)}}{\exists \dec \in \cX^{(N-1)} \text{ such that } (\dec,\dec') \in \Back(\xi,\xi')}$$
    and $(\dec'_0,\dec'_1) \in E$ if and only if $\dec'_0$ and $\dec'_1$ are $n$-adjacent, then $G$ is connected. 
\end{lemma}

\begin{proof}
We show that there exists a graph isomorphism between the $k$-dimensional hypercube graph $Q_{k}$ (the graph formed from the vertices and edges of a $k$-dimensional hypercube) and $G$, where $k$ is the number of inessential opaque directions of $\xi'$, i.e., $k = \#(O_i(\xi'))$. Note that if $O_i(\xi')=\emptyset$, then $\Back(\xi,\xi')$ consists of a single element. Otherwise, assume that $O_i(\xi')=\setof{n_1,\ldots,n_k}$.

Let $(r_1, \ldots, r_k) \in \setof{-1,1}^k$ denote the vertices of $Q_k$. Note that $f \colon \setof{-1,1}^k \to V$ given by 
\[
f(r_1,\ldots,r_n) = \left[ \bv'- \left( \sum_{n \in G_i(\xi')} \frac{1+s_n}{2} \bzero^{(n)} + \sum_{i=1}^k \frac{1+r_i}{2} \bzero^{(n_i)} \right), \bone \right],
\]
where $\xi'=[\bv',\bw']$ and $s_n = \rook_n(\xi',\cdot)$, is a bijection. Moreover, given any edge $e=((r_1,\ldots,r_k),(r'_1,\ldots,r'_k))$ in the hypercube graph, there exists an unique $r_\ell$ that differs between vertices. Note that $f(r_1,\ldots,r_k)$ and $f(r'_1,\ldots,r'_k)$ are $n_\ell$-adjacent with shared wall given by 
\[
\left[ \bv'-\left( \sum_{n \in G_i(\xi')} \frac{1+s_n}{2} \bzero^{(n)} + \sum_{\substack{i=1 \\ i \neq \ell}}^k \frac{1+r_i}{2} \bzero^{(n_i)} \right), \bone^{(n_\ell)} \right]. 
\]
\end{proof}

\begin{prop}
\label{prop:back_wall_well_defined}
Assume $(\xi, \xi')\in\cD(\rook)$ and let $\Ex(\xi,\xi')=\setof{n_o}$. 
If $(\dec_0,\dec'_0),(\dec_1,\dec'_1)\in \Dec(\xi,\xi')$, then 
\[
    \rook_{n_o} (\dec_0,\dec'_0) = \rook_{n_o} (\dec_1,\dec'_1). 
\]
\end{prop}
\begin{proof}
If $\Dec(\xi,\xi')=\setof{(\dec,\dec')}$ there is nothing to prove, so assume that $\Dec(\xi,\xi')$ has at least two different elements, and let $(\dec_0,\dec'_0),(\dec_1,\dec'_1)\in \Dec(\xi,\xi')$ be distinct pairs. Assume, for the sake of contradiction, that
\[
    \rook_{n_o}(\dec_0,\dec_0') \neq \rook_{n_o}(\dec_1,\dec_1'). 
\]
If $\dec'_0$ and $\dec'_1$ are $n$-adjacent, then $\rmap\xi(n)=n_o$. Note that $n \in O_i(\xi')$, so that by Definition~\ref{defn:indecisive}, either $\rmap\xi(n)=n$ or $n=n_o$. The former contradicts $\rmap\xi(n)=n_o$ while the latter contradicts $n_o \in J_e(\xi')$. Therefore, any pair whose top-cells are adjacent satisfy $\rook_{n_o}(\dec_0,\dec_0') =\rook_{n_o}(\dec_1,\dec_1')$. Given the connected structure of $\Back(\xi,\xi')$ described in Lemma~\ref{lem:back_wall_structure}, the result holds for any two pairs. 
\end{proof}
Now, we describe conditions under which we refine $\cF_1$ to obtain $\cF_2$.
\begin{defn}\label{defn:F2}
    We define $\cF_{2.1} \colon \cX \mvmap \cX$ to be the maximal refinement of the trivial map $\cF_0$ that satisfies the following condition.
    \begin{description}
        \item[Condition 2.1] Suppose that $(\xi,\xi')\in \cD(\rook)$ and let $(\dec,\dec')\in\Dec(\xi,\xi')$.
        \begin{itemize}
            \item If $\dec\in E^-(\dec')$, then $\xi' \not\in \cF_{2.1}(\xi)$.
            \item If $\dec\in E^+(\dec')$, then $\xi \not\in \cF_{2.1}(\xi')$.
        \end{itemize}
    \end{description}
\end{defn}
\begin{rem}
Recall from Definition~\ref{def:exit_face} that $\dec\in E^\pm(\dec')$ is determined by $\rook_{n_o}(\dec,\dec')$ and $p_{n_o}(\dec,\dec')$. 
Thus, by Remark~\ref{rem:same_position} and Proposition~\ref{prop:back_wall_well_defined},  Definition~\ref{defn:F2} is independent of the choice of a back wall.
\end{rem}
\begin{defn}
    \label{def:Rule2}
    Define $\cF_2 \colon \cX \mvmap \cX$ by
    \begin{equation}
        \label{eq:defnF2}
        \cF_2(\xi) \coloneqq \cF_1(\xi) \cap \cF_{2.1}(\xi).
    \end{equation}
\end{defn}
\begin{prop}
    \label{prop:F2-arrow-exists}
    For any pair $\xi,\xi' \in \cX$ such that $\xi \prec \xi'$ and $\dim(\xi')=\dim(\xi)+1$, either $\xi' \in \cF_2(\xi)$ or $\xi \in \cF_2(\xi)$.
\end{prop}
\begin{proof}
    It is enough to show that $\cF_1$ given by \eqref{eq:defnF1} and $\cF_{2.1}$ given by \eqref{eq:defnF2} do not remove both edges when $(\xi,\xi')$ exhibit indecisive drift. Let $\Ex(\xi,\xi')=\setof{n_o}$. 
    
    Suppose that $\xi' \notin \cF_1(\xi)$. By Condition 1.2, that means $\xi \in E^-(\xi')$, hence $\rook_{n_o}(\xi,\mu) = p_{n_o}(\xi,\xi')$
    for all $\mu \in \Top_\cX(\xi')$. In particular, $\rook_{n_o}(\xi,\dec')=p_{n_o}(\xi,\xi')$. Thus, by Corollary~\ref{cor:mu*},
    \[
    \rook_{n_o}(\dec,\dec')=p_{n_o}(\xi,\xi')=p_{n_o}(\dec,\dec'),
    \]
    so $\dec \in E^-(\dec')$ and Condition 2.1 only requires $\xi' \notin \cF_{2.1}(\xi)$, hence  $\xi \in \cF_1(\xi')$.
    
    Similarly, if $\xi \notin \cF_1(\xi')$, then $\dec \in E^+(\dec')$ and $\xi \notin \cF_{2.1}(\xi')$, hence  $\xi' \in \cF_1(\xi)$.

    Thus, $\cF_{2.1}$ either agrees with $\cF_1$ or refines double edges of $\cF_1$ when $(\xi,\xi')$ exhibit indecisive drift, so $\xi \in \cF_2(\xi')$ or $\xi' \in \cF_2(\xi)$.
\end{proof}
In light of Proposition~\ref{prop:F1-welldefined}, the same argument shows the following. 
\begin{cor}
    \label{cor:F2-well-defined}
    If $\cF_2 \colon \cX \mvmap \cX$ is given by \eqref{eq:defnF2}, then $\cF_2(\xi) \neq \emptyset$ for all $\xi \in \cX$. 
\end{cor}
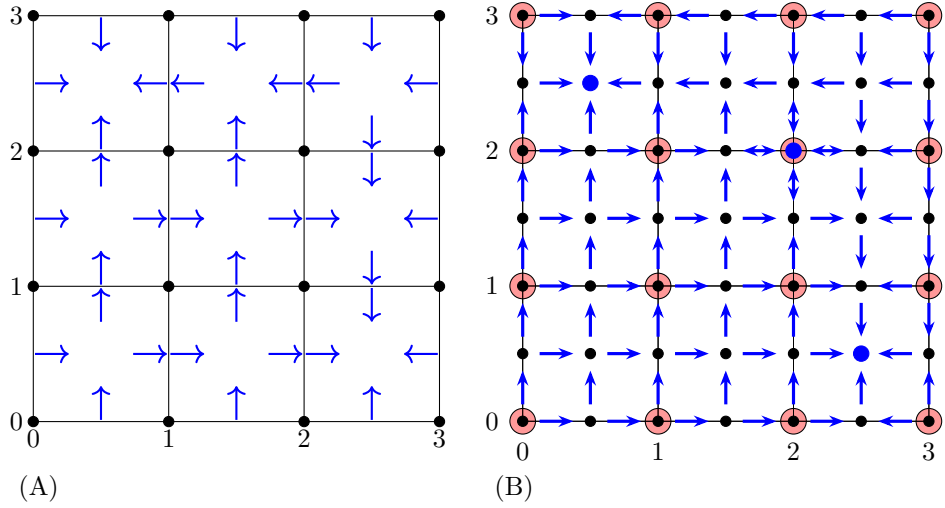
\begin{figure}
\begin{picture}(400,210)(0,0)
\put(10,0){(A)}
\put(0,15){
\begin{tikzpicture}[scale=0.225]
\draw[step=8cm,black] (0,0) grid (24,24);

\foreach \i in {0,...,3}{
    \foreach \j in {0,...,3}{
    \draw[black, fill=black] (8*\i,8*\j) circle (2ex);
    }
}

\foreach \i in {0,...,3}{
    \draw(8*\i,-1) node{$\i$}; 
    \draw(-1,8*\i) node{$\i$}; 
}

\foreach \i in {0,...,2}{
    \draw[->, blue, thick] (4+8*\i,0.1) -- (4+8*\i,2.1); 
    \draw[->, blue, thick] (4+8*\i,23.9) -- (4+8*\i,21.9); 
    \draw[->, blue, thick] (0.1, 4+8*\i) -- (2.1,4+8*\i); 
    \draw[->, blue, thick] (23.9,4+8*\i) -- (21.9,4+8*\i); 
}
\foreach \i in {0,...,1}{
    \draw[->, blue, thick] (4+8*\i,5.9) -- (4+8*\i,7.9); 
    \draw[->, blue, thick] (4+8*\i,8.1) -- (4+8*\i,10.1);
    \draw[->, blue, thick] (4+8*\i,13.9) -- (4+8*\i,15.9); 
    \draw[->, blue, thick] (4+8*\i,16.1) -- (4+8*\i,18.1);  
    \draw[->, blue, thick] (5.9,4+8*\i) -- (7.9,4+8*\i); 
    \draw[->, blue, thick] (8.1,4+8*\i) -- (10.1,4+8*\i);
    \draw[->, blue, thick] (13.9,4+8*\i) -- (15.9,4+8*\i); 
    \draw[->, blue, thick] (16.1,4+8*\i) -- (18.1,4+8*\i);  
}
\foreach \i in {2}{
    \draw[->, blue, thick] (4+8*\i,7.9) -- (4+8*\i,5.9); 
    \draw[->, blue, thick] (4+8*\i,10.1) -- (4+8*\i,8.1);
    \draw[->, blue, thick] (4+8*\i,18.1) -- (4+8*\i,16.1); 
    \draw[->, blue, thick] (4+8*\i,15.9) -- (4+8*\i,13.9);  
    \draw[->, blue, thick] (7.9,4+8*\i) -- (5.9,4+8*\i); 
    \draw[->, blue, thick] (10.1,4+8*\i) -- (8.1,4+8*\i);
    \draw[->, blue, thick] (18.1,4+8*\i) -- (16.1,4+8*\i); 
    \draw[->, blue, thick] (15.9,4+8*\i) -- (13.9,4+8*\i);  
}
\end{tikzpicture}
}
\put(190,0){(B)}
\put(180,10){
\begin{tikzpicture}[scale=0.225]
\foreach \i in {0,...,3}{
    \draw(8*\i,-1.8) node{$\i$}; 
    \draw(-1.8,8*\i) node{$\i$}; 
   }
\draw[step=8cm, black, ultra thin] (0,0) grid (24,24);
\foreach \i in {0,...,3}{
  \foreach \j in {0,...,3}{
    \draw[black, fill=red, fill opacity=0.4] (8*\i,8*\j) circle (5ex);
  }
}
\draw[step=8cm,black, ultra thin] (0,0) grid (24,24);
\foreach \i in {0,...,6}{
  \foreach \j in {0,...,6}{
    \draw[black, fill=black] (4*\i,4*\j) circle (2ex);
  }
}
    \draw[blue, fill=blue] (4,20) circle (3ex);
    \draw[blue, fill=blue] (20,4) circle (3ex);
    \draw[blue, fill=blue] (16,16) circle (3ex);
\foreach \i in {0,...,6}{
    \draw[-{Stealth[length=2mm]}, blue, very thick] (1, 4*\i) -- (3, 4*\i);  
}
\foreach \i in {0,...,3}{
    \draw[-{Stealth[length=2mm]}, blue, very thick] (13, 4*\i) -- (15, 4*\i);  
    \draw[-{Stealth[length=2mm]}, blue, very thick] (17, 4*\i) -- (19, 4*\i);  
 }
\foreach \i in {0,...,4}{
    \draw[-{Stealth[length=2mm]}, blue, very thick] (5, 4*\i) -- (7, 4*\i);  
    \draw[-{Stealth[length=2mm]}, blue, very thick] (9, 4*\i) -- (11, 4*\i);  
 }
\foreach \i in {0,...,6}{
    \draw[-{Stealth[length=2mm]}, blue, very thick] (23, 4*\i) -- (21, 4*\i);  
}
\foreach \i in {1,...,4}{
\draw[-{Stealth[length=2mm]}, blue, very thick] (3+4*\i,24) -- (1+4*\i,24);  
\draw[-{Stealth[length=2mm]}, blue, very thick] (3+4*\i,20) -- (1+4*\i,20);  
}

\foreach \i in {0,...,6}{
    \draw[-{Stealth[length=2mm]}, blue, very thick] (4*\i,23) -- (4*\i, 21);  
}
\foreach \i in {1,...,4}{
\draw[-{Stealth[length=2mm]}, blue, very thick] (20, 3+4*\i) -- (20,1+4*\i);  
\draw[-{Stealth[length=2mm]}, blue, very thick] (24, 3+4*\i) -- (24,1+4*\i);  
}

\foreach \i in {0,...,6}{
    \draw[-{Stealth[length=2mm]}, blue, very thick] (4*\i,1) -- (4*\i, 3);  
}
\foreach \i in {0,...,3}{
    \draw[-{Stealth[length=2mm]}, blue, very thick] (4*\i,13) -- (4*\i, 15);  
    \draw[-{Stealth[length=2mm]}, blue, very thick] (4*\i,17) -- (4*\i, 19);  
}
\foreach \i in {0,...,4}{
    \draw[-{Stealth[length=2mm]}, blue, very thick] (4*\i,5) -- (4*\i, 7);  
    \draw[-{Stealth[length=2mm]}, blue, very thick] (4*\i,9) -- (4*\i, 11);  
}

\foreach \i in {3,...,4}{
\draw[{Stealth[length=2mm]}-{Stealth[length=2mm]}, blue, very thick] (3+4*\i,16) -- (1+4*\i,16);  
}

\foreach \i in {3,...,4}{
\draw[{Stealth[length=2mm]}-{Stealth[length=2mm]}, blue, very thick] (16, 3+4*\i) -- (16,1+4*\i);  
}
\end{tikzpicture}
}
\end{picture}
\caption{ 
(A) The vector representation of the wall labeling $\omega$ of Figure~\ref{fig:wall_labeling}.
(B) Directed graph associated with $\cF_2\colon \cX\mvmap \cX$ 
 given wall labeling in (A).
Each cell is denoted by a vertex: a vertex within a square denotes a two-cell, a vertex on an edge denotes an edge, and the zero-cells are denoted by the orange disks.
The large blue disks indicate the existence of a self edge.}
\label{fig:F2mvmap}
\end{figure}
\begin{ex}
    \label{ex:F2}
    Consider the wall labeling given in Figure~\ref{fig:wall_labeling}. Notice that $\cF_1 : \cX \mvmap \cX$ described in Figure~\ref{fig:F1mvmap} has double edges between cells 
    \[
         \defcell{0}{2}{1}{0} 
         \darrow_{\cF_1} 
         \defcell{1}{2}{0}{0} 
         \darrow_{\cF_1} 
         \defcell{1}{2}{1}{0} 
         \darrow_{\cF_1} 
         \defcell{2}{2}{0}{0} 
         \darrow_{\cF_1} 
         \defcell{2}{2}{1}{0} 
    \]
    and
    \[
        \defcell{2}{0}{0}{1}
        \darrow_{\cF_1}
        \defcell{2}{1}{0}{0}
        \darrow_{\cF_1}
        \defcell{2}{1}{0}{1}
        \darrow_{\cF_1}
        \defcell{2}{2}{0}{0}
        \darrow_{\cF_1}
        \defcell{2}{2}{0}{1}.
    \]
    Except at pairs $(\xi,\xi')$ that involve $\defcell{2}{2}{0}{0}$, all others exhibit indecisive drift. 
    
Let $\xi = \defcell{1}{2}{0}{0}$ and $\xi_0'=\defcell{0}{2}{1}{0}$ as in Example~\ref{ex:backpairs} where 
\[
\Back(\xi,\xi'_0) = \setof{(\dec,\dec'_0)}= \setof{(\defcell{1}{1}{0}{1},\defcell{0}{1}{1}{1})}. 
\]
Note that $\rook_1(\dec,\dec'_0)=1=p_1(\dec,\dec'_0)$, hence $\dec \in E^-(\dec'_0)$. 
Thus, by Condition 2.1, $\xi \notin \cF_{2.1}(\xi')$. 
\end{ex}

\begin{prop}
\label{prop:c-tranverse}
Let $\xi\in\cX$. If there is no $\xi' \in \cX$ such that $(\xi,\xi')\in\cD(\rook)$ or $(\xi',\xi)\in\cD(\rook)$, then $\cF_2(\xi) = \cF_1(\xi)$ and $\cF_2^{-1}(\xi) = \cF_1^{-1}(\xi)$.
\end{prop}

\begin{proof}
By Definition~\ref{defn:F2} $\cF_2(\xi) = \cF_1(\xi)$ and $\cF_2^{-1}(\xi) = \cF_1^{-1}(\xi)$ unless Condition 2.1 is satisfied for some $(\xi,\xi')\in\cD(\rook)$ or $(\xi',\xi)\in\cD(\rook)$.
By hypothesis, this is not the case.
\end{proof}

\section{Resolving cycles of local inducement maps}
\label{sec:cycle-condition}
Our goal is to produce $\cF_3\colon \cX\mvmap \cX$ by analyzing cases where $\cF_2$ did not refine $\cF_1$. This is done by removing double arrows where possible or adding arrows wherever necessary. 

We observe that in Definition~\ref{defn:indecisive}, the multivalued map $\cF_2$ hasn't addressed cases in which the direction in $\Ex(\xi,\xi')=\setof{n_o}$ actively regulates a different opaque direction $n_o' \in O(\xi)$. So we turn our attention to a particular subcase, when $\rmap\xi \colon \activeset(\xi)\to\activeset(\xi)$ is a bijection and there is an opaque direction $n_o \in O(\xi)$ that is cyclic at $\xi$ with length of the cycle $k \geq 2$ (see Definition~\ref{def:acyclic-direction}).
\begin{defn}
    Let $S$ be a finite set and $\sigma : S \to S$ a bijection. A \emph{cycle decomposition} for $\sigma$ is an expression of $\sigma$ as the product of disjoint cycles $\sigma_i = (a_{i_1} \, a_{i_2} \, \ldots \ a_{i_{n_i}})$ where $ \sigma_i : S \to S$ satisfies $\sigma_i(a_{i_k})=a_{i_{k+1}}$ and $\sigma_i(x)=x$ for any $x \in S\setminus S_i$ with $S_i \coloneqq \setof{a_{i_1},\ldots,a_{i_{n_i}}}$, and $S_i \cap S_j = \emptyset$ whenever $i\neq j$. 
\end{defn}

\begin{prop}
\label{prop:face_cycle}  
Let $\xi\preceq_\cX\tau\in\cX$.
Let $\sigma$ be cycle for $\rmap{\tau}$.
Then, $\sigma$ is a cycle for $\rmap{\xi}$.
\end{prop}
\begin{proof}
Notice that $\activeset(\tau) \subset \activeset(\xi)$, so if $\rmap\tau|_{S_\sigma} = \sigma$, then $\rmap\xi|_{S_\sigma} = \sigma$. That is, $\sigma$ is a cycle in $\rmap\xi$.
\end{proof}

\begin{defn}
\label{defn:partially_opaque}
A cell $\xi \in \cX \setminus \cX^{(N)}$ is called \emph{semi-opaque} if $\rmap \xi \colon \activeset(\xi) \to \activeset(\xi)$ is a bijection.
\end{defn}
Note that every opaque cell is semi-opaque by Proposition~\ref{prop:opaque}.                                                                

\begin{defn}
    \label{defn:lap-number}
    Let $\xi=[\bv,\bw] \in \cX$ be a semi-opaque cell.
    For each $n \in J_i(\xi)$, define $\kappa_n^- \in \cX^{(N-1)}$ and $\kappa \in \cX^{(N)}$ by
    \[
        \kappa_n^- = [\bv,\bone^{(n)}], \quad \kappa = [\bv,\bone]. 
    \]
    Let $$\rmap\xi = \sigma_{\xi,1} \sigma_{\xi,2} \ldots \sigma_{\xi,r}$$
    be the decomposition of the permutation $\rmap\xi$ into disjoint cycles. 
    For each cycle $\sigma=(n_1 \ \ldots \ n_k)$ of length $k$, define the support of the cycle by $S_\sigma \coloneqq \setof{n_1,\ldots,n_k} \subset \activeset(\xi)$ and assign a non-negative integer $\lap_{\xi,\sigma}\colon \Top_\cX(\xi) \to \mathbb{N}$ to all top cells of $\xi$ by 
    \[
        \lap_{\xi,\sigma}(\mu) \coloneqq \# \setdef{ n \in S_\sigma } 
        {\rook_{\sigma(n)}(\kappa_{\sigma(n)}^-,\kappa) \cdot p_{n}(\xi,\mu) \cdot p_{\sigma(n)}(\xi,\mu) < 0}.
    \]
\end{defn}

\begin{ex}\label{ex:lap_number_2and3D}
    If $\sigma = (n)$, then $S_\sigma = \setof{n}$ and 
    \begin{align*}
        \lap_{\xi,\sigma}(\mu) = \begin{cases} 
        0 & \text{if } \omega(\kappa_n^-,\kappa) > 0\\ 
        1 & \text{if } \omega(\kappa_n^-,\kappa) < 0
        \end{cases}
    \end{align*}
    If $\sigma = (n_1 \ n_2)$, then $S_\sigma=\setof{n_1,n_2}$ and 
    \begin{align*}
        \lap_{\xi,\sigma}(\mu) = \begin{cases} 
        0 & \text{if } \omega(\kappa_{n}^-,\kappa)p_{n_1}(\xi,\mu)p_{n_2}(\xi,\mu)>0 \text{ for } n=n_1, n_2 \\
        1 & \text{if } \omega(\kappa_{n_1}^-,\kappa)\omega(\kappa_{n_2}^-,\kappa)<0\\
        2 & \text{if } \omega(\kappa_{n}^-,\kappa)p_{n_1}(\xi,\mu)p_{n_2}(\xi,\mu)<0 \text{ for } n=n_1, n_2\\
        \end{cases}
    \end{align*}
    If $\sigma = (n_1 \ n_2 \ n_3)$, then $S_\sigma=\setof{n_1,n_2,n_3}$ and $\lap_{\xi,\sigma}(\mu)$ simply counts how many $-1$'s are in
    \[
        \setof{ \omega(\kappa_{n_2}^-,\kappa)p_{n_1}(\xi,\mu)p_{n_2}(\xi,\mu), \omega(\kappa_{n_3}^-,\kappa)p_{n_2}(\xi,\mu)p_{n_3}(\xi,\mu), \omega(\kappa_{n_1}^-,\kappa)p_{n_3}(\xi,\mu)p_{n_1}(\xi,\mu)}.
    \]
    In particular, $\lap_{\xi,\sigma}(\kappa)=\#\setdef{i\in\setof{1,\ldots,k}}{\omega(\kappa_{n_i}^-,\kappa) = -1}$, where $k$ is the length of the cycle $\sigma$.
\end{ex}
    Note that by Proposition~\ref{prop:opaque}, every equilibrium cell $\xi \in \cX$ has a bijection $\rmap\xi\colon \activeset(\xi) \to \activeset(\xi)$. Moreover, note that if $n \in S_\sigma$ for some cycle $\sigma$, then $\cycleof_\xi(n)=S_\sigma$.

\begin{rem}
    \label{rem:lap-number}
    The definition of $\lap_{\xi,\sigma}$ is motivated by the definition of an integer-valued Lyapunov function, namely zero number or lap number, for cyclic feedback systems of ODEs \cite{mallet-paret:smith, gedeon:memoirs}. 
    This number characterizes the structure of unstable manifolds of an equilibrium in such systems. 
    By foreshadowing their existence, we identify the position of the manifolds in the combinatorial setting. This is further discussed in Chapter~\ref{sec:futurework}.
\end{rem}
\begin{defn}
    \label{defn:Rule3.1}
    We define $\cF_{3.1} : \cX \mvmap \cX$ to be the maximal refinement of $\cF_0$ that satisfies the following condition.
    \begin{description}
        \item[Condition 3.1] Let $\xi \in \cX$ be a semi-opaque cell.
        If $\sigma$ is a $k$-cycle in the cycle decomposition of $\rmap\xi$ with $k\geq 2$, then for any $\xi' \in \cX$, $\xi \prec \xi'$, such that $\Ex(\xi,\xi')\subseteq S_\sigma$, $\xi'\notin \cF_{3.1}(\xi)$. 
    \end{description}
\end{defn}
\begin{defn}
\label{defn:Rule3.2}
Let $\xi \in \cX$ be a semi-opaque cell. Consider the cycle decomposition of $\rmap{\xi} = \sigma_1 \ldots \sigma_r$. 
For each cycle $\sigma$ with length $| \sigma | \geq 2$, set
\[
\mathcal{N}(\xi,\sigma) = \setdef{\xi' \in \cX}{\xi\prec\xi', \Ex(\xi,\xi')\subseteq S_\sigma, \dim(\xi')-\dim(\xi)\geq 2}
\]
and
\begin{equation}
\label{defn:set_cells_Uns&Stable}
\cU(\xi,\sigma) = \setdef{ \xi'\in \mathcal{N}(\xi,\sigma)}{\lap_{\xi,\sigma}(\mu) < \frac{| \sigma |}{2},\ \forall \mu \in \Top_\cX(\xi')}.
\end{equation}

Define the set of \emph{unstable cells} of $\xi$ by
\begin{equation}
\label{defn:unstable_cells}
\cU(\xi) = \bigcup_{\stackrel{i = 1}{|\sigma_i| \geq 2}}^r \cU(\xi,\sigma_i),
\end{equation}
that is, the union is taken over all cycles $\sigma$ of length $k \geq 2$ in the cycle decomposition of $\rmap{\xi}$. 
If $\rmap\xi$ is not a bijection, let $\cU(\xi)=\emptyset$.
\end{defn}

\begin{defn}
    \label{defn:Rule3}
    We define $\cF_{3} \colon \cX \mvmap \cX$ by
    \[
        \cF_3(\xi) = \left( \cF_{2}(\xi) \cap \cF_{3.1}(\xi) \right)\cup \cU(\xi).
    \]
\end{defn}
\begin{prop}
    \label{prop:F3-welldefined}
    For any $\xi \in \cX$, $\cF_3(\xi) \neq \emptyset$. 
\end{prop}
\begin{proof}
    If $G(\xi)\neq\emptyset$, then by Lemma~\ref{lem:gradient-implies-non-empty} there exists $\xi'\in\cF_2(\xi)$. Note that $\xi'\in\cF_{3.1}(\xi)$ is unchanged since $\Ex(\xi,\xi')$ is contained in $G(\xi)$ and not $O(\xi)$.

    If $G(\xi)=\emptyset$, then by Proposition~\ref{defn:opaque}, $\rmap{\xi}$ is a bijection, and $\xi \in \cF_2(\xi)$ and $\xi \in \cF_{3.1}(\xi)$ since only $\cF_{1.2}$ removes self-edges. 
    Thus, $\cF_3(\xi)\neq \emptyset$. 
\end{proof}

\begin{thm}
\label{thm:2dDone}
    Let $\omega\colon W(\cX) \to \setof{\pm 1}$ be a wall labeling on an abstract two-dimensional cubical complex. 
    Then, $\cF_3\colon \cX\mvmap \cX$ has no double edges.
\end{thm}
\begin{proof}
    We verify that each pair $(\xi,\xi')\in \cX \times \cX$ with $\xi \prec \xi'$ and $\dim(\xi')=\dim(\xi)+1$ does not have a double edge under $\cF_3$. That is enough since $\cF_{3.2}$ only adds edges in the direction $\xi \rightarrow_{\cF_{3.2}} \xi'$ when $\dim(\xi')=\dim(\xi)+2$. 

    If $\dim(\xi)=1$, then $(\xi,\xi') \in W(\xi)$ so $\xi \in E^\pm(\xi')$ and either $\xi \notin \cF_3(\xi')$ or $\xi' \notin \cF_3(\xi)$. 

    If $\dim(\xi)=0$, let $\Ex(\xi,\xi')=\setof{n}$. If there is no double edge under $\cF_1$, we are done. If $\xi \darrow_{\cF_1} \xi'$, then by Lemma~\ref{prop:Rule1-DArrow-Implies-Opaque}, $n \in O(\xi)$. If $\rmap\xi(n)=n$ or $\rmap\xi(n) \notin \activeset(\xi)$, then $\cF_2$ removes the double edge. Otherwise,  $\rmap\xi(n)\neq n$ and $\rmap\xi(n) \in \activeset(\xi) \subseteq (\xi)$ and, since $N=2$, $n$ is cyclic of length $k=2$, in particular, $\rmap{\xi} :\setof{1,2}\to\setof{1,2}$ is a bijection. Thus, $\xi'\notin \cF_{3.1}(\xi)$ removes the double edge. 

    Therefore, $\cF_3 : \cX \mvmap \cX$ has no double edge. 
\end{proof}

\chapter{Lattice structures}
\label{sec:lattice}

Chapters~\ref{sec:RookFields} and \ref{sec:rules} provide algorithms to go from wall labelings to four combinatorial models, $\cF_i\colon \cX\mvmap\cX$ for $i\in\setof{0,1,2,3}$.
The goal of this Chapter is to demonstrate that these models can be used to derive a comibinatorial and homological characterization of global dynamics.

To achieve this goal requires some preparatory remarks and constructions.
We begin in Section~\ref{sec:compInvset+} by defining the D-grading (meant to suggest dynamics grading) of a combinatorial map $\cF$ and as is made clear this partial order is intimately tied to the lattice of forward invariant sets $\sInvset^+(\cF)$.
We also discuss restrictions on $\sInvset^+(\cF)$ due to the standing assumption that we are working with strongly dissipative wall labelings.

Though D-gradings capture the direction of non-recurrent dynamics, it is not clear that they provide a correct codification of the desired homological invariants.
Thus, in Section~\ref{sec:AdmissbleD-grading}, we introduce the notion of an admissible grading.
Furthermore,  the abstract cubical cell complex $\cX$ on which the combinatorial model $\cF$ is defined does not appear to be rich enough to express the desired algebraic topological information.
Hence, in Section~\ref{sec:blowup_complex} we define an associated blow-up complex $\cX_b$.

Section~\ref{sec:Dgradblowup} provides one of the major results of this monograph: the D-gradings defined from the combinatorial models $\cF_i\colon \cX\mvmap\cX$ for $i\in\setof{0,1,2,3}$ give rise to admissible gradings on $\cX_b$.
As a consequence we can compute the desired homological invariants.

These desired homological invariants, e.g., the Conley index and the Conley complex, are presented in Section~\ref{sec:ConleyComplex}. 
Section~\ref{sec:CCalgorithm} concludes this chapter with a discussion the desired computations and some simple examples.

\section{Computing $\sInvset^+(\cF)$}
\label{sec:compInvset+}

Let $\cZ$ denote a finite set.
Let $\cF\colon \cZ\mvmap \cZ$ denote a multivalued map, i.e., for every $\zeta\in\cZ$, $\cF(\zeta)$ is a non-empty subset of $\cZ$.
Equivalently, $\cF$ denotes a directed graph with vertices $\cZ$,  directed edges $\zeta_0\to\zeta_1$ if and only if $\zeta_1\in\cF(\zeta_0)$, and for every $\zeta_0\in\cZ$ there exists a vertex $\zeta_1$ for which there is an edge $\zeta_0\to\zeta_1$.
\emph{We use both concepts interchangeably.}
\begin{defn}
    \label{defn:actualscc}
    A \emph{path of length $k$} from a vertex $\zeta \in \cZ$ to a vertex $\zeta' \in \cZ$ is a sequence of vertices $(\zeta_0,\ldots,\zeta_k) \in \cZ^{k+1}$ such that $\zeta_{i+1} \in \cF(\zeta_i)$, $\zeta_0=\zeta$ and $\zeta_k=\zeta'$. 
    A \emph{strongly connected component} of a directed graph is a maximal set of vertices $\cY \subset \cZ$ such that for every pair $\zeta,\zeta' \in \cY$ there are paths in $\cY$ from $\zeta$ to $\zeta'$ and from $\zeta'$ to $\zeta$. 
\end{defn}
\begin{defn}
\label{defn:Dgrading}
Recall that the vertices $\cZ$ of a directed graph $\cF$ can be partitioned into the collection of strongly connected components $\SCC(\cF)$ \cite{cormen:leiserson:rivest:stein}.
We encode this partition by means of a quotient map $\pi \colon \cZ \to \SCC(\cF)$ that we refer to as the \emph{D-grading} for $\cF$.
A strongly connected component that contains at least one edge is called a \emph{recurrent component}.
\end{defn}

There are linear time algorithms for computing $\SCC(\cF)$ \cite{tarjan}.
\begin{defn}
\label{defn:SCC}
The \emph{weak condensation graph} of $\cF$ is the multivalued map $\bar{\cF} \colon \SCC(\cF) \mvmap \SCC(\cF)$ defined by
$\alpha_1 \in \bar{\cF}(\alpha_0)$ if and only if there exist $\zeta_0 \in \pi^{-1}(\alpha_0)$ and $\zeta_1 \in \pi^{-1}(\alpha_1)$ such that $\zeta_1 \in \cF(\zeta_0)$.
\end{defn}

\begin{ex}
\label{ex:trivialSCC}
Consider the multivalued maps $\cF_i\colon \cX\mvmap \cX$, where $i=0,1.1$ as in Definition~\ref{def:Rule0} and \ref{def:Rule1.1}.
Observe that $\xi \darrow_{\cF_i} \xi'$ for every $\xi\prec\xi'$ with $\dim(\xi')=\dim(\xi)+1$ and therefore there exists a single strongly connected component $\alpha\in \SCC(\cF_i)$.
Furthermore, since there exists at least one pair $\xi, \xi' \in \pi^{-1}(\alpha)$ such that $\xi \darrow_{\cF_i} \xi'$ it follows that $\alpha \in \bar{\cF}_i (\alpha)$, i.e., $\alpha$ is a recurrent component.
\end{ex}

Note that $\bar{\cF} \colon \SCC(\cF) \mvmap \SCC(\cF)$ is just the \emph{condensation graph} of $\cF$ \cite[Definition 1.1]{Fagnani2018} with self-edges added to recurrent components. 
Since the condensation graph of $\cF$  is a directed acyclic graph it imposes a partial order on $\SCC(\cF)$. 
Therefore, $\bar{\cF}$ defines the same partial order, namely $(\SCC(\cF), \leq_{\bar{\cF}})$ is a poset where the partial order is given by $\alpha' \leq_{\bar{\cF}} \alpha$ if and only if there exists $n\geq 0$ such that $\alpha'\in \bar{\cF}^n(\alpha)$. 

A set $\cN\subseteq \cZ$ is \emph{forward invariant} under $\cF$ if $\cF(\cN)\subseteq \cN$.
The collection of forward invariant sets is denoted by $\sInvset^+(\cF)$ and as is shown in \cite{kalies:mischaikow:vandervorst:14} is a distributive lattice under the operations of intersection $\cap$ and union $\cup$.
Observe that $\sInvset^+(\cF)$ is a $\bzero\bone$-sublattice of $\sO(\cZ)$, the downsets of $\cZ$.

\begin{prop}
\label{prop:Invset+}
Let $\cF \colon \cZ \mvmap \cZ$ be a multivalued map and let $\bar{\cF} \colon \SCC(\cF) \mvmap \SCC(\cF)$ be the
weak condensation graph of $\cF$.
Then, 
\[
\sInvset^+(\cF)\cong \sInvset^+(\bar{\cF}) = \sO(\SCC(\cF),\leq_{\bar{\cF}}).
\]
\end{prop}

\begin{proof}
We prove that the quotient map  $\pi\colon \cZ \to \SCC(\cF)$ induces an lattice isomorphism from $\sInvset^+(\cF)$ to $\sInvset^+(\bar{\cF})$.

We begin by showing that if $\cN\in \sInvset^+(\cF)$, then $\pi(\cN)\in \sInvset^+(\bar{\cF})$.
Let $\alpha_0\in\pi(\cN)$ and let $\alpha_1 \in \bar{\cF}(\alpha_0)$.
By Definition~\ref{defn:SCC} $\alpha_1 \in \bar{\cF}(\alpha_0)$ implies that there exists $\zeta_0\in\cN$ and $\zeta_1\in \pi^{-1}(\alpha_1)$ such that 
$\zeta_1\in\cF(\zeta_0)$.
Consider any $\zeta_0\in\cN$ such that $\pi(\zeta_0)=\alpha_0$.
By definition, $\cF(\cN)\subseteq \cN$ and hence if $\zeta_1\in\cF(\zeta_0)$, then $\zeta_1\in \cN$.
Thus, $\alpha_1=\pi(\zeta_1) \in \pi(\cN)$.
Therefore, $\bar{\cF}(\pi(\cN))\subseteq \pi(\cN)$.

We now show that if $\bar{\cN}\in \sInvset^+(\bar{\cF})$, then $\cN :=\pi^{-1}(\bar{\cN})\in \sInvset^+(\cF)$.
The proof is by contradiction.
Let $\zeta_0 \in \cN$, $\zeta_1\in \cF(\zeta_0)$ and assume that $\zeta_1\not\in \cN$. 
This implies that $\alpha_1 =\pi(\zeta_1) \not\in \pi(\cN)$.
Let $\alpha_0 = \pi(\zeta_0)\in \bar{\cN}$.
By Definition~\ref{defn:SCC},  $\alpha_1\in \bar{\cF}(\alpha_0)\subseteq \bar{\cN}$; a contradiction.

We leave it to the reader to check that $\pi$ induces a lattice morphism, i.e., that
\[
\pi(\cN_0\cup\cN_1) = \pi(\cN_0)\cup \pi(\cN_1)\quad\text{and}\quad
\pi(\cN_0\cap\cN_1) = \pi(\cN_0)\cap c(\cN_1).
\]
\end{proof}
The standing assumption that we are working with strongly dissipative wall labelings impacts the structure of $\sInvset^+(\cF_i)$, $i=1,2,3$, and hence, restricts the combinatorial/homological dynamics.
\begin{defn}
\label{defn:commonGradient}
Let $\cF\colon \cX\mvmap \cX$ be a multivalued map on an $N$ dimensional cubical complex.
A subset $\cS\subset \cX$ has a \emph{common gradient direction} $n \in \setof{1,\ldots, N}$ if $n\in G(\xi)$ for all $\xi\in \cS$.
\end{defn}

\begin{thm}
\label{thm:InvsetStronglyDiss}
Consider the multivalued map $\cF_i\colon \cX\mvmap \cX$, $i=1,2,3$, on an $N$-dimensional cubical complex $\cX$.
If $\cS \cap \bbdy(\cX)$ is a recurrent component under $\cF_i$, then $\cS \subset \bbdy(\cX)$, and furthermore, $\cS$ has at least one common gradient direction.
\end{thm}

An important consequence  of Theorem~\ref{thm:InvsetStronglyDiss} is that if $\cS$ is a recurrent component contained in $\bbdy(\cX)$, then it is not associated with recurrent dynamics of the associated ODEs introduced in Chapter~\ref{sec:ramp}.
The justification for this claim is delayed until Theorem~\ref{thm:regular_cell}.

Our proof of Theorem~\ref{thm:InvsetStronglyDiss} involves a more careful analysis of 
$\cF_i$ for $i=1,2,3$, and thus we introduce the following notation.

Given $n \in \setof{1,\ldots,N}$, we define the left and right $n$-boundaries of $\cX$, respectively, by 
\begin{equation}
    \label{eq:nbbdy}
    \begin{aligned}
        \cX_n^- & = \setdef{ [\bv,\bw ] \in \cX}{\bv_n=0, \bw_n=0},\\
        \cX_n^+ & = \setdef{ [\bv,\bw] \in \cX }{\bv_{n}=K(n)+1, \bw_n=0}.
    \end{aligned}
\end{equation}
Note that $\bbdy(\cX) = \bigcup_{n=1}^N (\cX_n^- \cup \cX_n^+)$. 
\begin{lem}
    \label{lem:GradBoundary}
    For each $n \in \setof{1,\ldots,N}$, and for each $\xi\in \cX_n^\pm,$ $n \in G(\xi)$.
\end{lem}
\begin{proof}
    We prove for $\cX_n^-$, since the case $\cX_n^+$ is analogous. 
    
    Let $\xi \in \cX_n^-$. To prove that $n \in G(\xi)$, we show that $$R_n(\xi)=\setdef{\rook_{n}(\xi,\mu)}{\mu\in\Top_\cX(\xi)}=\setof{1}.$$
    Let $\mu \in \Top_{\cX}(\xi)$. By definition of $\cX_n^-$, $n\in J_i(\xi)$, so by Proposition~\ref{cor:mu*},
    \[
        \rook_n(\xi,\mu) = \omega(\mu_n^{*}(\xi,\mu),\mu) = 1. 
    \]
    Thus, $R_n(\xi)=\setof{1}$ and $n \in G(\xi)$. Note that our choice of $\xi$ was arbitrary, so $n$ is a common gradient direction of $\cX_n^-$. 
\end{proof}
\begin{lem}
    \label{lem:F1bbdyX}
    If $\xi' \in \cX\setminus \bbdy(\cX)$, then $ \cF_1(\xi')\cap \bbdy(\cX)=\emptyset$.
\end{lem}
\begin{proof}
    It is enough to show that $\cX_n^\pm \cap \cF_1(\xi')$ for each $n = 1,\ldots,N$. As in Lemma~\ref{lem:GradBoundary}, we prove for the left $n$-boundary $\cX_n^-$. 

    First, observe that $\cF_1(\xi') \subseteq \cF_0(\xi')$. So if $\cX_n^- \cap \cF_1(\xi')\neq\emptyset$, then there must exist $\xi \in \cX_n^-$ such that $\xi \darrow_{\cF_0} \xi'$. 

    Let $\xi'\in \cX\setminus\bbdy(\cX) \subset \cX\setminus\cX_n^-$ and assume that $\xi\darrow_{\cF_0}\xi'$ for some $\xi \in \cX_n^-$. Writing $\xi'=[\bv',\bw']$ yields that either $\bv'_n>0$ or $\bw'_n=1$ by \eqref{eq:nbbdy}. Note that $\bv'_n>\bv_n=0$ is not possible due to \eqref{eq:is_a_face}, so $\bv'_n=\bv_n=0$ and $\bw_n'=1$. That is, $n \in J_e(\xi')$. Since $\xi \darrow_{\cF_0} \xi'$, it follows that $\Ex(\xi,\xi')=\setof{n}$ by Definition~\ref{def:Rule0}. 

    Finally, we show that $\xi \in E^+(\xi)$. Note that by Definition~\ref{defn:rpvector} and Lemma~\ref{lem:GradBoundary},
    \[
        p_n(\xi,\xi') = 1 = \rook_n(\xi,\mu),
    \]
    for every $\mu \in \Top_\cX(\xi') \subseteq \Top_\cX(\xi)$. Therefore, $\xi \in E^+(\xi)$ and $\xi \notin \cF_1(\xi')$. 
\end{proof}
\begin{proof}[Proof of Theorem~\ref{thm:InvsetStronglyDiss}]
    We prove it by contradiction. Suppose that $\cS \cap \bbdy(\cX)$ is a recurrent component and assume for the sake of contradiction that $\cS \not\subset \bbdy(\cX)$. Let $\xi \in \bbdy(\cX)$ and $\xi' \in \cX\setminus \bbdy(\cX)$ be elements such that there exists a path $(\zeta_0,\ldots,\zeta_k)$ from $\xi'$ to $\xi$. Note that $\xi \in \cX_n^-$ or $\xi \in \cX_n^+$ for some $n$. 
    Let be $i$ is the largest index with $\zeta_i \notin \cX_n^-$. Clearly, $0 \leq i < k$ since $\zeta_0 =\xi' \notin \bbdy(\cX)$ and $\zeta_k = \xi \in \cX_n^\pm$. Then $\zeta_{i+1} \in \cF_1(\zeta_i)$, which contradicts Lemma~\ref{lem:F1bbdyX}.
\end{proof}

\begin{ex}
Figure~\ref{fig:PosetF1}(A) shows a portion of the weak condensation graph $\bar{\cF}_1 \colon \SCC(\cF_1) \mvmap \SCC(\cF_1)$ for $\cF_1$ of Figure~\ref{fig:F1mvmap}.
For the sake of simplicity the boundary elements of $\cX(\I)$, $\I = \setof{0,1,2,3}^2$ are not shown.
By Theorem~\ref{thm:InvsetStronglyDiss}, if $\cS$ is a recurrent component in $\bbdy(\cX)$, then $\cS$ has a common gradient direction, and as discussed Section~\ref{sec:CCalgorithm} (see Definition~\ref{defn:MG} and the preceding discussion) this implies that $\cS$ is not of interest with regard to applications to ODEs.
Figure~\ref{fig:PosetF1}(B) shows the Hasse diagram for the corresponding subposet of $(\SCC(\cF_1),\leq_{\bar{\cF}_1})$.
In this figure the vertices of $\SCC(\cF_1)$ are indicated by the cells that belong to the pre-images of the D-grading $\pi\colon \cX \to \SCC(\cF_1)$.

Figure~\ref{fig:PosetF2} provides the same information as Figure~\ref{fig:PosetF1} but for $\cF_2$ of Figure~\ref{fig:F2mvmap}.
The fact that  $\cF_2$ has fewer double edges than $\cF_1$ leads to more strongly connected components (potentially fewer recurrent components) and to significant changes in the multivalued maps.
The refinement that results in  $\bar{\cF}_2 \colon \SCC(\cF_2) \mvmap \SCC(\cF_2)$ and $(\SCC(\cF_2),\posetF{2})$ allows for a more detailed understanding of the dynamics of ODEs modeled using  the wall labeling of Figure~\ref{fig:wall_labeling}.
\end{ex}

\begin{figure}
\begin{picture}(350,320)(0,0)

\put(110,0){(B)}
\put(100,10){
\begin{tikzpicture}[scale=0.2,>= stealth,->,shorten >=2pt,looseness=.5,auto,ampersand replacement=\&]
    \matrix [matrix of math nodes,
           column sep={1cm,between origins},
           row sep={1.0cm,between origins},
           nodes={rectangle}]
  {
|(11)| \defcell{0}{0}{1}{1} \& |(12)| \& |(13)| \& |(14)| \& |(15)|
\\
|(21)| \defcell{1}{0}{0}{1} \&  |(22)| \&   |(23)| \defcell{0}{1}{1}{0} \& |(24)| \& |(25)| 
\\
|(31)| \defcell{1}{0}{1}{1} \&   |(32)|  \&   |(33)| \defcell{1}{1}{0}{0} \& |(34)| \& |(35)| \defcell{0}{1}{1}{1} 
\\    
|(41)| \defcell{1}{1}{1}{0} \&   |(42)|  \&   |(43)| \defcell{1}{1}{0}{1} \& |(44)| \& |(45)|  
\\     
|(51)| \defcell{1}{1}{1}{1} \&   |(52)|  \&   |(53)| \& |(54)| \& |(55)| 
\defcell{2}{2}{1}{1}
     \\ 
     \\
     |(61)| \&   |(62)|  \&   |(63)|  
      \setof{ \begin{array}{c}
    \defcell{2}{0}{0}{1}, 
    \defcell{2}{1}{0}{0}, 
    \defcell{2}{1}{0}{1}, \\
    \defcell{0}{2}{1}{0}, 
    \defcell{1}{2}{0}{0}, 
    \defcell{1}{2}{1}{0}, \\
    \defcell{2}{2}{0}{0}, 
    \defcell{2}{2}{0}{1},
    \defcell{2}{2}{1}{0}
     \end{array}} \& |(64)|  \& |(65)| 
     \\ 
     \\
     |(71)| \defcell{2}{1}{1}{1} \&   |(72)| \&   |(73)| \& |(74)|  \& |(75)| \defcell{1}{2}{1}{1} 
        \\ 
      |(81)| \defcell{2}{1}{1}{0} \&   |(82)| \&   |(83)| \& |(84)| \& |(85)| \defcell{1}{2}{0}{1} 
        \\ 
      |(91)| \&   |(92)| \defcell{2}{1}{1}{0} \&   |(93)| \& |(94)| \defcell{1}{2}{0}{1} \& |(95)|
      \\ 
  };
    \draw (11) -- (21);
    \draw (11) -- (23);
    \draw (21) -- (31);
    \draw (21) -- (33);
    \draw (23) -- (33);
    \draw (23) -- (35);
    \draw (31) -- (41);
    \draw (33) -- (41);
    \draw (33) -- (43);
    \draw (35) -- (43);
    \draw (41) -- (51);
    \draw (43) -- (51);
    \draw (51) -- (63);
    \draw (55) -- (63);  
    \draw (63) -- (71);
    \draw (63) -- (75);
    \draw (71) -- (81);
    \draw (75) -- (85);
    \draw (81) -- (92);
    \draw (85) -- (94);
 
\end{tikzpicture}
}

\put(10,0){(A)}
\put(0,10){
\begin{tikzpicture}[scale=0.2,>=stealth,->,shorten >=2pt,looseness=.5,auto,ampersand replacement=\&]
    \matrix [matrix of math nodes,
           column sep={1cm,between origins},
           row sep={1cm,between origins},
           nodes={rectangle
           }]
  {
|(11)| \bullet  \& |(13)| \& |(15)|   \\
      |(21)| \bullet \&   |(23)| \bullet  \& |(25)|  \\
    |(31)| \bullet   \&   |(33)|\bullet  \& |(35)|\bullet    \\    
|(41)| \bullet   \&   |(43)|  \bullet  \& |(45)|    \\     
|(51)| \bullet   \&   |(53)|  \& |(55)| \bullet
         \\ \\
     |(61)|   \&   |(63)|  
       \bullet \& |(65)|  
     \\ \\
     |(71)|    \bullet \&   |(73)| \& |(75)| \bullet
       \\ 
      |(81)|   \bullet \&   |(83)| \& |(85)| \bullet
         \\ 
      |(91)|    \bullet \&   |(93)|  \& |(95)|   \bullet
        \\ 
  };
  \begin{scope}[]
    \draw (11) -- (21);
    \draw (11) -- (23);
    \draw (21) -- (31);
    \draw (21) -- (33);
    \draw (23) -- (33);
    \draw (23) -- (35);
    \draw (33) -- (41);
    \draw (33) -- (43);
    \draw (31) -- (41);
    \draw (31) -- (63);
    \draw (35) -- (43);
    \draw (35) -- (63);
    \draw (41) -- (51);
    \draw (41) -- (63);
    \draw (43) -- (63);
    \draw (43) -- (51);
    \draw (51) -- (63);
    \draw (55) -- (63);  
    \draw (63) to [loop right] (63);
    \draw (63) -- (71);
    \draw (63) -- (75);
    \draw (63) -- (81);
    \draw (63) -- (85);
    \draw (63) -- (91);
    \draw (63) -- (95);
    \draw (71) -- (81);
    \draw (75) -- (85);
    \draw (81) -- (91);
    \draw (85) -- (95);
    \draw (91) to [loop right] (91);
    \draw (95) to [loop left] (95);
    \end{scope}
\end{tikzpicture}}
\end{picture}
\caption{(A) The weak condensation graph $\bar{\cF}_1 \colon \SCC(\cF_1) \mvmap \SCC(\cF_1)$. (B) Hasse diagram for the poset  $(\SCC(\cF_1),\leq_{\bar{\cF}_1})$ that  indicates the pre-images of $\pi\colon \cX \to \SCC(\cF_1)$.}
\label{fig:PosetF1}
\end{figure}

\begin{figure}
\begin{picture}(370,320)(0,0)
\put(140,0){(B)}
\put(130,10){
\begin{tikzpicture}[scale=0.2,>=stealth,->,shorten >=2pt,looseness=.5,auto,ampersand replacement=\&]
    \matrix [matrix of math nodes,
           column sep={1cm,between origins},
           row sep={1.0cm,between origins},
           nodes={rectangle
           }]
  {
     |(11)| \defcell{0}{0}{1}{1} \& |(12)| \& |(13)| \& |(14)| \& |(15)|\& |(16)| \& |(17)| 
     \\
      |(21)|  \defcell{1}{0}{0}{1} \&  |(22)| \&   |(23)| \defcell{0}{1}{1}{0} \& |(24)| \& |(25)| 
     \\
    |(31)| \defcell{1}{0}{1}{1} \&   |(32)|  \&   |(33)|\defcell{1}{1}{0}{0} \& |(34)| \& |(35)|\defcell{0}{1}{1}{1} 
     \\    
|(41)| \defcell{2}{0}{0}{1} \&   |(42)|  \&   |(43)|\defcell{1}{1}{1}{0} \& |(44)| \& |(45)|  \defcell{1}{1}{0}{1}    \& |(46)| \& |(47)| \defcell{0}{2}{1}{0}
     \\     
|(51)| \defcell{2}{1}{0}{0}   \&   |(52)| \&   |(53)|  \defcell{1}{1}{1}{1} \& |(54)| \& |(55)|  \defcell{1}{2}{0}{0} 
          \& |(56)| \& |(57)| \defcell{2}{2}{1}{1}
     \\ 
|(61)|   \&  |(62)|
     \\
|(71)| \&   |(72)|  \&   |(73)|  
      \setof{ \begin{array}{c}
     \defcell{2}{1}{0}{1}, 
     \defcell{1}{2}{1}{0}, \\
     \defcell{2}{2}{0}{0}, 
     \defcell{2}{2}{0}{1},
     \defcell{2}{2}{1}{0}
     \end{array}} \& |(74)|  \& |(75)| 
     \\ 
     \\
     |(81)|    \defcell{2}{1}{1}{1} \&   |(82)| \&   |(83)| \& |(84)|  \& |(85)|\defcell{1}{2}{1}{1}  
        \\ 
      |(91)|     \defcell{2}{1}{1}{0} \&   |(92)|\&   |(93)| \& |(94)| \& |(95)| \defcell{1}{2}{0}{1} 
        \\ 
      |(101)| \&   |(102)|  \defcell{2}{1}{1}{0} \&   |(103)| \& |(104)|    \defcell{1}{2}{0}{1} \& |(105)|
      \\ 
  };
  \begin{scope}[]
    \draw (11) -- (21);
    \draw (11) -- (23);
    \draw (21) -- (31);
    \draw (21) -- (33);
    \draw (23) -- (33);
    \draw (23) -- (35);
    \draw (31) -- (41);
    \draw (31) -- (43);
    \draw (33) -- (43);
    \draw (33) -- (45);
    \draw (35) -- (45);
    \draw (35) -- (47);
    \draw (41) -- (51);
    \draw (43) -- (51);
    \draw (43) -- (53);
    \draw (45) -- (53);
    \draw (45) -- (55);
    \draw (47) -- (55);
    \draw (51) -- (73);
    \draw (53) -- (73);
    \draw (55) -- (73);
    \draw (57) -- (73);
    \draw (73) -- (81);
    \draw (73) -- (85);
    \draw (81) -- (91);
    \draw (85) -- (95);
    \draw (91) -- (102);
    \draw (95) -- (104);
    \end{scope}

\end{tikzpicture}
}

\put(10,0){(A)}
\put(0,10){
\begin{tikzpicture}[scale=0.2,>=stealth,->,shorten >=2pt,looseness=.5,auto,ampersand replacement=\&]
    \matrix [matrix of math nodes,
           column sep={1cm,between origins},
           row sep={1cm,between origins},
           nodes={rectangle
           }]
  {
|(11)| \bullet  \& |(13)| \& |(15)|  \& |(17)| \\
      |(21)| \bullet \&   |(23)| \bullet  \& |(25)|  \\
    |(31)| \bullet   \&   |(33)|\bullet  \& |(35)|\bullet    \\    
|(41)| \bullet   \&   |(43)|  \bullet  \& |(45)| \bullet  \& |(47)| \bullet    \\     
|(51)| \bullet   \&   |(53)| \bullet \& |(55)| \bullet  \& |(57)| \bullet
         \\ \\
     |(61)|   \&   |(63)|  
       \bullet \& |(65)|  
     \\ \\
     |(71)|    \bullet \&   |(73)| \& |(75)| \bullet
       \\ 
      |(81)|   \bullet \&   |(83)| \& |(85)| \bullet
         \\ 
      |(91)|    \bullet \&   |(93)|  \& |(95)|   \bullet
        \\ 
  };
  \begin{scope}[]
    \draw (11) -- (21);
    \draw (11) -- (23);
    \draw (21) -- (31);
    \draw (21) -- (33);
    \draw (23) -- (33);
    \draw (23) -- (35);
    \draw (31) -- (41);
    \draw (31) -- (43);
    \draw (33) -- (43);
    \draw (33) -- (45);
    \draw (35) -- (45);
    \draw (35) -- (47);
    \draw (41) -- (51);
    \draw (41) to [bend right] (91);
    \draw (43) -- (51);
    \draw (43) -- (53);
    \draw (45) -- (53);
    \draw (45) -- (55);
    \draw (47) -- (55);
    \draw (47) to [bend left] (95);
    \draw (51) -- (63);
    \draw (51) to [bend right] (81);
    \draw (55) -- (63);  
    \draw (55) to [bend left] (85);
    \draw (57) -- (63);
    \draw (63) to [loop right] (63);
    \draw (63) -- (71);
    \draw (63) -- (75);
    \draw (63) -- (81);
    \draw (63) -- (85);
    \draw (63) -- (91);
    \draw (63) -- (95);
    \draw (71) -- (81);
    \draw (75) -- (85);
    \draw (81) -- (91);
    \draw (85) -- (95);
    \draw (91) to [loop right] (91);
    \draw (95) to [loop left] (95);
    \end{scope}
\end{tikzpicture}
}

\end{picture}
\caption{(A) The weak condensation graph $\bar{\cF}_2 \colon \SCC(\cF_2) \mvmap \SCC(\cF_2)$. (B) Hasse diagram for the poset  $(\SCC(\cF_2),\posetF{2})$ that  indicates the pre-images of $\pi\colon \cX \to \SCC(\cF_2)$.
}
\label{fig:PosetF2}
\end{figure}

\section{Admissible Grading for Cell Complexes}
\label{sec:AdmissbleD-grading}

For clarity, the discussion in Section~\ref{sec:compInvset+} is presented on the level of combinatorial multivalued maps on finite sets.
The multivalued maps are essential for the definition of a D-grading. 
Our goal is to use algebraic topology to characterize dynamics. Therefore, we will eventually restrict to finite sets that form cell complexes and make use of D-gradings defined on these complexes.
However, the structures discussed in this section are only dependent upon the cell complex structure and are independent of the existence of a multivalued map.

\begin{defn}
\label{defn:admissibleGrading}
Let $(\cZ,\preceq)$ be a cubical cell complex and $(\sP,\leq)$ be a finite poset.
A surjective function $\pi\colon \cZ \to \sP$ is an \emph{admissible grading} on $(\cZ, \preceq)$ if 
\[
\pi(\zeta) = \min\setdef{\pi(\mu)}{\mu\in \Top_{\cZ}(\zeta)},
\]
where $\Top_\cZ(\zeta)$ are the top-dimensional cells of $\cZ$ that are cofaces of $\zeta$ (see Definition~\ref{defn:topstar}.)
\end{defn}
Note that Definition~\ref{defn:admissibleGrading} provides uniqueness of the minimum element in $\Top_\cZ(\zeta)$ for each $\zeta \in \cZ$.
\begin{prop}
\label{prop:gradingorder}
Let $\pi\colon (\cZ,\preceq)\to (\sP, \leq)$ be an admissible grading from  a cubical cell complex  $(\cZ,\preceq)$ to a poset $(\sP,\leq)$. 
Then, $\pi$ is an order preserving epimorphism and $\pi(\cZ^{(N)}) = \sP$.
\end{prop}
\begin{proof}
We first prove that $\pi$ is order-preserving.
If $\zeta \preceq \zeta'$ then $\Top_{\cZ}(\zeta')\subset \Top_{\cZ}(\zeta)$. 
Hence, by definition of $\pi$, we have that $\pi(\zeta)\leq_{\bar{\cF}} \pi(\zeta')$.
Observe that $\pi(\cZ) \subseteq \pi(\cZ^{(N)})$, and by Definition~\ref{defn:admissibleGrading}, $\pi(\cZ) = \sP$. Therefore, it follows that $\pi(\cZ^{(N)}) = \sP$.
\end{proof}

Applying 
Theorem~\ref{thm:birkhoff}
we obtain the following corollary.
\begin{cor}
    \label{cor:birkhoff} 
    If $\pi\colon (\cZ,\preceq)\to (\sP,\leq)$ is an  admissible grading from a cubical cell complex $(\cZ,\preceq)$ to a poset $(\sP,\leq)$, then
    \begin{align*}
        \sO(\pi) \colon \sO(\sP) & \to \sO(\cZ) \\
        U & \mapsto \cU \coloneqq \pi^{-1}(U)
    \end{align*}
    is a lattice monomorphism.
\end{cor}
As indicated in Proposition~\ref{prop:gradingorder} if  $\pi\colon (\cZ,\preceq)\to (\sP,\leq)$ is an  admissible grading, then $\pi(\cZ^{(N)}) = \sP$.
In what follows we consider the possibility of extending a function $\pi\colon \cZ^{(N)}\to \sP$ to an admissible grading.
\begin{defn}\label{defn:extendable}
Let $(\cZ, \preceq)$ be a cubical cell complex and $(\sP, \leq)$ a finite poset. A function $\pi\colon \cZ^{(N)} \to \sP$ is called an \emph{extendable grading} if there exists an extension $\widehat{\pi}\colon \cZ \to \sP$ such that $\widehat{\pi}$ is an admissible grading. We refer to $\widehat{\pi}$ as the extension of $\pi$, which is unique by Definition~\ref{defn:admissibleGrading}.
\end{defn}
As the following example shows, not all functions $\pi\colon \cZ^{(N)}\to\sP$ are extendable gradings.
\begin{ex}
\label{ex:nonuniqueness-grading}
Consider the one-dimensional cubical complex $\cZ$ with cells $\cZ=\setof{[(0),(0)], [(1),(0)], [(2),(0)], [(0),(1)], [(1),(1)]}$ (see Figure~\ref{fig:nonuniqueness-grading}(A)).
Let $\sP = \setof{p_0,p_1}$ with no relation (see Figure~\ref{fig:nonuniqueness-grading}(B)).
Define $\pi \colon \cZ^{(1)}\to\sP$ by 
$\pi([(0),(1)])=p_0$ and $\pi([(1),(1)])=p_1$.
There is no extension of $\pi$ that is an admissible grading.
Let $\zeta = [(1),(0)]$.
If we set $\pi(\zeta) = p_0$, then $\pi$ is not admissible with respect to the pair $[(1),(0)] \prec [(1),(1)]$
and if we set $\pi(\zeta) = p_1$, then $\pi$ is not admissible with respect to the pair $[(1),(0)] \prec [(0),(1)]$.
\end{ex}

\begin{figure}[h]
    \centering
    \begin{subfigure}{0.45\textwidth}
        \centering
        \begin{tikzpicture}
            \node[fill=black, circle, inner sep=1.5pt] (A) at (0,0) {};
            \draw (A) -- ++(2,0) node[midway, above] {$p_0$};
            \node[fill=black, circle, inner sep=1.5pt, label=below:{$\zeta$}] (B) at (2,0) {};
            \draw (B) -- ++(2,0) node[midway, above] {$p_1$};
            \node[fill=black, circle, inner sep=1.5pt] (C) at (4,0) {};
        \end{tikzpicture}
        \caption{Cubical complex $\cX$ with grading $\pi$.}
    \end{subfigure}
    \hfill
    \begin{subfigure}{0.45\textwidth}
        \centering
        \begin{tikzpicture}
            \node[circle, draw, minimum size=0.7cm] (p0) at (0,0) {$p_0$};
            \node[circle, draw, minimum size=0.7cm] (p1) at (2,0) {$p_1$};
        \end{tikzpicture}
        \caption{Poset $\sP$.}
    \end{subfigure}
    \caption{Example~\ref{ex:nonuniqueness-grading} illustrated.}
    \label{fig:nonuniqueness-grading}
\end{figure}

We conclude this section by addressing the following question.
If $\pi\colon \cZ^{(N)}\to\sP$ is an extendable grading of $\cZ$ onto $\sP$, what modifications of $\pi$ lead to a different extendable grading of $\cZ$ onto $\sP$?
\begin{prop}\label{prop:min_extendable}
Let $\pi\colon\cZ^{(N)} \to \sP$ be an extendable grading from a finite cell complex $(\cZ,\preceq)$ to a poset $(\sP,\leq)$.
Consider $\cV\subset \cZ^{(N)}$ and a function $\pi'\colon\cZ^{(N)} \to \sP$ that satisfies the following conditions.
\begin{enumerate}
    \item Restricted to $\cZ^{(N)}\setminus \cV$, $\pi = \pi'$.
    \item $\pi'(\cV) = p_\cV \in\sP$
    \item If $\mu\in\cV$,  $\zeta\in\cZ^{(N)}$ and $\cl(\mu)\cap\cl(\zeta)\neq\emptyset$, then $p_\cV\leq\pi(\zeta)$.
\end{enumerate}
Then, $\pi'$ is an extendable grading.
\end{prop}
\begin{proof}
Let $\zeta\in\cZ$.

If $\Top_{\cZ}(\zeta)\cap \cV = \emptyset$, then (1) yields equality of the restricted map, so
\[
\pi'(\zeta) \coloneq \min\setdef{\pi'(\mu)}{\mu\in \Top_{\cZ}(\zeta)} = \min\setdef{\pi(\mu)}{\mu\in \Top_{\cZ}(\zeta)}.
\]

If $\Top_{\cZ}(\zeta)\cap \cV \neq \emptyset$, then 
\[
\pi'(\zeta) \coloneq  \min\setdef{\pi'(\mu)}{\mu\in \Top_{\cZ}(\zeta)} =  \min\left\{ \setof{p_\cV} \cup \setdef{\pi(\mu)}{\mu\in \Top_{\cZ}(\zeta)\backslash \cV} \right\} = p_\cV
\]
where the first equality follows from (2) and last equality follows from (3).
In either case $\pi'(\zeta)$ is the unique minimum over $\setdef{\pi'(\mu)}{\mu\in \Top_{\cZ}(\zeta)}$.
\end{proof}

\begin{thm}
\label{thm:extension}
Let $\pi\colon\cZ^{(N)} \to \sP$ be an extendable grading from a finite cell complex $(\cZ,\preceq)$ to a poset $(\sP,\leq)$.
Consider $\cV\subset \cZ^{(N)}$ and a function $\pi'\colon\cZ^{(N)} \to \sP$ that satisfies the following conditions.
\begin{enumerate}
\item Restricted to $\cZ^{(N)}\setminus \cV$, $\pi = \pi'$.
\item Let $\mu\in\cZ^{(N)}$ and $\mu'\in\cV$ such that  $\pi'(\mu)$ and $\pi'(\mu')$ are not comparable, then $\cl(\mu)\cap\cl(\mu')=\emptyset$.
\item Let $\mu,\mu'\in\cZ^{(N)}$ and $\mu''\in\cV$ such that  $\pi(\mu),\pi(\mu')\leq\pi'(\mu'')$, and $\pi(\mu)$ and $\pi(\mu')$ are not comparable. Then \[\cl(\mu)\cap\cl(\mu')\cap\cl(\mu'')=\emptyset.\]
\end{enumerate}
Then, $\pi'$ is an extendable grading.
\end{thm}
\begin{proof}
We provide a proof by contradiction.
So assume that $\pi'$ is not extendable.
This implies that there exists $\zeta\in\cZ$ such that $\min\setdef{\pi'(\eta)}{\eta\in\Top_\cZ(\zeta)}$ is not unique.
Let \begin{equation}\label{eq:min_zeta}
M \coloneqq \setdef{\mu_i\in \Top_\cZ(\zeta)}{\pi'(\mu_i)\in \min\setdef{\pi'(\eta)}{\eta\in\Top_\cZ(\zeta)}, i=1,\ldots,I}.
\end{equation}
By assumption, $I \geq 2$ and without loss of generality we can assume that $\pi'(\mu_i)$ and $\pi'(\mu_j)$ are not comparable.
Indeed, since the elements of $M$ are defined as minima with respect to $\leq$, if $\pi'(\mu_{i}), \pi'(\mu_{j})\in M$ are comparable then $\pi'(\mu_{i}) = \pi'(\mu_{j})$.

Assume that there exists $i\in\setof{1,\ldots,I}$ such that $\mu_i\in\cV$. 
Since $\pi'(\mu_1)$ and $\pi'(\mu_2)$ are not comparable, $\pi'(\mu_i) \neq \pi'(\mu_1)$ or $\pi'(\mu_i) \neq \pi'(\mu_2)$.
Without loss of generality, assume $\pi'(\mu_i) \neq \pi'(\mu_1)$, so that $\pi'(\mu_i)$ and $\pi'(\mu_1)$ are not comparable.
However, $\zeta\in\cl(\mu_i)\cap\cl(\mu_1)$, which contradicts assumption (2).

Thus, for the remainder of the proof we can assume that $\mu_i\not\in\cV$, for all $i\in\setof{1,\ldots,I}$. 
By assumption (1) $\pi'(\mu_i) = \pi(\mu_i)$ for all $i = 1,\ldots, I$. 
Since $\pi$ is assumed to be extendable, there exists $\mu\in \Top_\cZ(\zeta)$ such that 
$
\pi(\mu) = \min\setdef{\pi(\eta)}{\eta\in\Top_\cZ(\zeta)}
$
and hence 
\begin{equation}
    \label{eq:pimu}
    \pi(\mu)\leq \pi(\mu_i)=\pi'(\mu_i)\quad \text{for all $i = 1,\ldots, I$.}
\end{equation}

We claim that 
\begin{equation}
\label{eq:mu_i}
\mu \not\in M \text{ and }
\pi(\mu) \neq \pi'(\mu_i)=\pi(\mu_i) \text{ for all } i\in\setof{1,\ldots,I}.
\end{equation} 
To prove this claim, we note that if there exists $j\in\setof{1,\ldots,I}$ such that $\mu=\mu_j$ or $\pi(\mu)=\pi(\mu_j)$, 
then
\[
\pi'(\mu_j)=\pi(\mu_j)=\pi(\mu)\leq \pi(\mu_i)=\pi'(\mu_i) \text{ for all } i\in\setof{1,\ldots,I},
\]
where the inequality follows from \eqref{eq:pimu}.
Hence, $\pi'(\mu_j)$ is the unique minimal element in $\pi(M)$, a contradiction with the assumption that $\min\setdef{\pi'(\eta)}{\eta\in\Top_\cZ(\zeta)}$ is not unique.

We use this claim to prove that $\mu\in\cV$.
Again we use a proof by contradiction.
If $\mu \not\in \cV$, then $\mu \in \cZ^{(N)}\setminus \cV$, and hence, by (1) $\pi(\mu) = \pi'(\mu)$.
In particular, by \eqref{eq:pimu}
\[
\pi'(\mu)=\pi(\mu)\leq \pi(\mu_i)=\pi'(\mu_i),
\]
for all $i = 1,\ldots, I$. 
Thus, there exists $j\in\setof{1,\ldots,I}$ such that $\pi'(\mu_j)=\pi'(\mu)$, a contradiction with \eqref{eq:mu_i}.
Therefore, $\mu\in\cV$.

Referring back to the claim, $\mu\not\in M$.
Thus, $\pi(\mu_1)=\pi'(\mu_1) < \pi'(\mu)$ and $\pi(\mu_2)=\pi'(\mu_2) < \pi'(\mu)$.
However, $\mu\in\cV$, $\pi(\mu_1),\pi(\mu_2) < \pi'(\mu)$, $\pi(\mu_1)$ and $\pi(\mu_2)$ are not comparable, and 
\[
\zeta\in \cl(\mu_j)\cap\cl(\mu_k)\cap\cl(\mu).
\]
This contradicts (3).
\end{proof}

\section{The Blow-up Complex}
\label{sec:blowup_complex}

The multivalued maps $\cF_i\colon \cX\mvmap \cX$, $i=0,1,2,3$, defined in Chapter~\ref{sec:rules} provide sufficient information for the computation of the associated Morse graphs $\sMG(\cF_i)$.
However, to the best of our knowledge the cell complex $\cX$ of Defintion~\ref{defn:Xcomplex} is not rich enough to express the algebraic topology associated with Conley indices. 
Thus we introduce the following cubical complex.

\begin{defn}
\label{defn:Xbcomplex}
Consider a cubical cell complex $\cX(\I) = (\cX,\preceq,\dim,\kappa)$ where $I_n = \setof{0,\ldots, K(n)+1}$, for $n=1,\ldots,N$.
Set $\bar{K}(n) = 2(K(n)+1)$.
The \emph{blowup cubical complex} $\cX(\bar{\I}) = (\cX_b,\preceq_b,\dim,\kappa)$ is the $N$-dimensional cubical cell complex generated by $\bar{\I} = \prod_{n=1}^N \bar{I}_n$ where
\[
\bar{I}_n := \setof{0,\ldots, \bar{K}(n) + 1},\quad n=1,\ldots, N.
\]
The \emph{blowup map} is the bijection $\blup \colon \cX  \to  \cX_b^{(N)}$ given by
\begin{equation}
\label{eq:defnb}
    \blup([\bv,\bw ])  := [2\bv + \bw,{\bf 1}].
\end{equation}
\end{defn}

\begin{rem}
\label{rem:xihat}
The blowup map is a bijection from cells of the cubical cell complex $\cX(\I)$ to top dimensional cells of the blowup complex $\cX(\bar{\I})$.
The inverse of $\blup$ is given by
\[
\blup^{-1}([\bar{\bv},\bone]) = \left[\frac{1}{2}(\bar{\bv}- (\bar{\bv}\mod 2)), \bar{\bv}\mod 2 \right].
\]
where $\bar{\bv}\mod 2\in\setof{0,1}^N$ is obtained by evaluating each coordinate of $\bv$ mod $2$.
\end{rem}

1We use the blowup map $\blup\colon \cX \to \cX_b^{(N)}$ to transport the lattice of forward invariant sets $\sInvset^+(\cF)$ and the D-grading $\pi\colon \cX\to \SCC(\cF)$ 
to structures defined in terms of $\cX_b$. 
This requires a technical understanding of properties of $\blup$ that is developed in the remainder of this section.

\begin{ex}
\label{ex:simplexXb}
Recall Example~\ref{ex:cubicalcomplex} where $N=2$, $K(1)=K(2)=2$, and hence $\I = \setof{0,1,2,3}^2$. 
The cubical complex $\cX(\I)$ is visualized in Figure~\ref{fig:position_vectors}.
This same visualization is presented in Figure~\ref{fig:Xb}(A) in blue.
For the associated blowup complex $\bar{K}(1)=\bar{K}(2) = 6$, and hence $\bar{\I} = \setof{0,\ldots, 7}^2$.
The blowup complex $\cX_b(\I)$ is shown in black in Figure~\ref{fig:Xb}(A).
\end{ex}

\begin{figure}
\centering
\begin{picture}(200,200)(0,-10)
\begin{tikzpicture}[scale=0.175]
\foreach \i in {0,...,7}{
    \draw(4*\i -2,-4) node{$\i$}; 
    \draw(-4,4*\i -2) node{$\i$}; 
   }
\fill[gray!10!white] (-2,-2) rectangle (26,26);
\draw[step=4cm,black,thick,xshift=-2cm,yshift=-2cm] (0,0) grid (28,28);
\foreach \i in {0,...,7}{
  \foreach \j in {0,...,7}{
    \draw[black, fill=black] (-2+4*\i,-2+4*\j) circle (2ex);
  }
}   
\draw[step=8cm,blue] (0,0) grid (24,24);
\foreach \i in {0,...,3}{
    \draw(8*\i,-4) node{\footnotesize$\color{blue}\i$}; 
    \draw(-4,8*\i) node{\footnotesize$\color{blue}\i$}; 
   }
\foreach \i in {0,...,3}{
  \foreach \j in {0,...,3}{
    \draw[black, fill=blue] (8*\i,8*\j) circle (1ex);
  }
}
\end{tikzpicture}
\end{picture}
\caption{
Blue represents the 2-dimensional cell complex $\cX(\I)$ where $\I = \setof{0,1,2,3}^2$ along with numbering of vertices (see Figure~\ref{fig:position_vectors}). 
Black represents the associated cell complex $\cX_b(\I)$ along with number of vertices as given by \eqref{eq:defnb}.
}
\label{fig:Xb}
\end{figure}
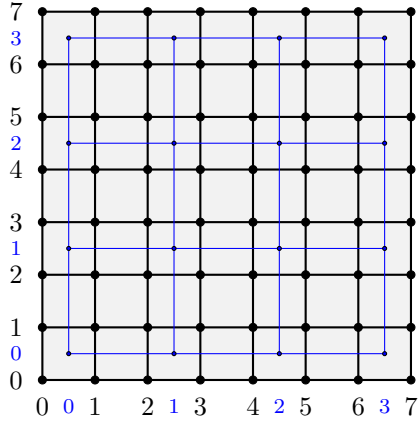

We leave the proof of the following lemma to the reader.
\begin{lem}
    \label{lem:codimension1face}
Consider top dimensional cells $\zeta_i\in \cX_b^{(N)}$, $i=0,1$.
Assume there exists $\zeta\in \cX_b^{(N-1)}$ such that $\zeta\prec_b \zeta_i$, $i=0,1$.
Then,
\[
|\dim(\blup^{-1}(\zeta_1))-\dim(\blup^{-1}(\zeta_0))| =1.
\]
\end{lem}

Recall that, for any $\xi\in\cX_b$,  $\Top_{\cX_b}(\xi)=\setdef{\mu\in \cX_b^{(N)}}{\xi\preceq_b\mu}$. 

\begin{ex}
\label{ex:T_of_vw}
Consider the blowup complex $\cX_b(\I)$ in  Example~\ref{ex:simplexXb} shown in black in Figure~\ref{fig:Xb}(A). Let $[\bv,\bw]=\left[ \left(\begin{smallmatrix}
        1 \\ 2
    \end{smallmatrix}\right),\left(\begin{smallmatrix}
        0 \\ 0
     \end{smallmatrix}\right)\right]\in\cX_b$, 
then 
\[
\Top_{\cX_b}(
\left[ \left(\begin{smallmatrix}
        1 \\ 2
    \end{smallmatrix}\right),\left(\begin{smallmatrix}
        0 \\ 0
     \end{smallmatrix}\right)\right]
) = 
\setof{
\left[ \left(\begin{smallmatrix}
        1 \\ 2
    \end{smallmatrix}\right),\left(\begin{smallmatrix}
        1 \\ 1
     \end{smallmatrix}\right)\right],
\left[ \left(\begin{smallmatrix}
        0 \\ 2
    \end{smallmatrix}\right),\left(\begin{smallmatrix}
        1 \\ 1
     \end{smallmatrix}\right)\right],
\left[ \left(\begin{smallmatrix}
        1 \\ 1
    \end{smallmatrix}\right),\left(\begin{smallmatrix}
        1 \\ 1
     \end{smallmatrix}\right)\right],
\left[ \left(\begin{smallmatrix}
        0 \\ 1
    \end{smallmatrix}\right),\left(\begin{smallmatrix}
        1 \\ 1
     \end{smallmatrix}\right)\right]
}.
\]
Notice that, $[\bv',\bw']\in \Top_{\cX_b}(
\left[ \left(\begin{smallmatrix}
        1 \\ 2
    \end{smallmatrix}\right),\left(\begin{smallmatrix}
        0 \\ 0
     \end{smallmatrix}\right)\right]
)$ is obtained by taking the position $\bv=
\left(\begin{smallmatrix}
        1 \\ 2
    \end{smallmatrix}\right)
$ and subtracting from it one element of
$\setof{
\left(\begin{smallmatrix}
        0 \\ 0
    \end{smallmatrix}\right),
\left(\begin{smallmatrix}
        1 \\ 0
    \end{smallmatrix}\right),
\left(\begin{smallmatrix}
        0 \\ 1
    \end{smallmatrix}\right),
\left(\begin{smallmatrix}
        1 \\ 1
    \end{smallmatrix}\right)
}$,
and fixing $\bw'=\left(\begin{smallmatrix}
        1 \\ 1
     \end{smallmatrix}\right)$.
\end{ex}

In the spirit of Example \ref{ex:T_of_vw}, we leave it to the reader to check the following lemma. It is simply a reformulation of the face relation in Definition~\ref{defn:Xcomplex}. 

\begin{lem}
\label{lem:top-cell-formula}
    Let $[\bar\bv,\bar\bw]\in\cX_b$, then
    \begin{equation}\label{eq:T_of_zeta}
    \Top_{\cX_b}([\bar\bv,\bar\bw]) = \setdef{ [\bar\bv-\bq,\bone]}{\bq \in\setof{0,1}^N,\bar\bw +\bq \poeZ \bone}. 
    \end{equation}
\end{lem}

Recall that the essential and inessential directions of a cell $[\bv,\bw]\in\cX$ are the sets 
\[
J_e([\bv,\bw]) = \setof{n\mid \bw_n=1} \quad \text{and} \quad  J_i([\bv,\bw]) = \setof{n\mid \bw_n=0},
\]
respectively.
The next lemma is a straightforward consequence of Lemma~\ref{lem:top-cell-formula}. 

\begin{lem}
\label{lem:q_property}
Let $[\bar\bv,\bar\bw]\in\cX_b$ and let $\bq\in\setof{0,1}^N$. Then, $[\bar\bv-\bq,\bone]\in\Top_{\cX_b}([\bar\bv,\bar\bw])$, if and only if,
$\bq_n=0 \text{ for all } n\in J_e([\bar\bv,\bar\bw])$.
\end{lem}

Explicit calculations using the definition of $\blup$ lead to the following result.
\begin{lem}
    \label{lem:dimensionblup}
Let $[\bar{\bv},\bone]\in \cX_b^{(N)}$ and let $[\bv,\bw] = \blup^{-1}([\bar{\bv},\bone])$.
Then,
\begin{equation}\label{eq:b_inverse_v}
    2\bv_n + \bw_n = \bar{\bv}_n,
\end{equation}
\begin{equation}\label{eq:b_inverse_w}
    \bw_n = \begin{cases}
        1, &\text{if $\bar{\bv}_n$ is odd},\\
        0, &\text{if $\bar{\bv}_n$ is even},
    \end{cases}
\end{equation}
and
\begin{equation}\label{eq:b_inverse_dim}
    \dim([\bv,\bw]) = \# \setdef{n}{\bar{\bv}_n \equiv 1 \pmod{2} }.
\end{equation}
\end{lem}

\begin{lem}
\label{lem:q_face}
Fix $[\bar{\bv},\bar{\bw}]\in\cX_b$ and $k\in J_i([\bar{\bv},\bar{\bw}])$.
Let $\bq\in\setof{0,1}^N$ be such that $[\bar{\bv}-\bq -\bzero^{(k)},\bone], [\bar{\bv}-\bq,\bone]\in\Top_{\cX_b}([\bar{\bv},\bar{\bw}])$.

If $\bar{\bv}_k-\bq_k -1$ is even, then
\[
\blup^{-1}([\bar{\bv}-\bq -\bzero^{(k)},\bone])\prec_\cX \blup^{-1}([\bar{\bv}-\bq,\bone]).
\]

If $\bar{\bv}_k-\bq_k -1$ is odd, then
\[
\blup^{-1}([\bar\bv-\bq ,\bone])\prec_\cX \blup^{-1}([\bar\bv-\bq-\bzero^{(k)},\bone]),
\]
\end{lem}

\begin{proof}
Let $[\bv',\bw']=\blup^{-1}([\bar\bv-\bq-\bzero^{(k)},\bone])$ and $[\bv'',\bw'']=\blup^{-1}([\bar\bv-\bq ,\bone])$. We want to prove that either $[\bv',\bw']\prec_\cX [\bv'',\bw'']$ or $[\bv'',\bw'']\prec_\cX [\bv',\bw']$.

If $\bar{\bv}_k-\bq_k-1$ is even, then $\bar{\bv}_k-\bq_k-1=2m$, $m\in\N$, and $\bar{\bv}_k-\bq_k = 2m + 1$ is odd. 
Hence, by construction $\bv'_n=\bv''_n,$ for all $n\in\setof{1,\ldots, N}\backslash\setof{k}$. 
If $\bar{\bv}_k-\bq_k-1$ is even, then $\bar{\bv}_k-\bq_k$ is odd, $k \in J_i([\bv',\bw'])$ and $k \in J_e([\bv'',\bw''])$ by Lemma~\ref{lem:dimensionblup}. Thus, $\bw'_k = 0$ and $\bw''_k=1$, that is, $\bw'+\bzero^{(k)}=\bw''$. Together with the fact that $\bv'=\bv''+\bzero^{(k)}$, Definition~\ref{defn:Xcomplex} yields $[\bv',\bw']\preceq_{\cX}[\bv'',\bw'']$. 

If $\bv_k-\bq_k -1$
is odd, then it follows analogously that $[\bv'',\bw'']\prec_\cX [\bv',\bw']$.
\end{proof}

\begin{prop}
\label{prop:lowest_highest}
Let $\zeta = [\bar{\bv},\bar{\bw}]\in\cX_b$.
There exists a unique lowest dimensional cell $\minD(\zeta)$ and a unique highest dimensional cell $\maxD(\zeta)$  in $\blup^{-1}\left(\Top_{\cX_b}(\zeta) \right)\subset \cX$.
Furthermore, $\minD(\zeta)= \blup^{-1}([\bar{\bv} -\bq,\bone])$  where
\[
\bq_n = \begin{cases}
    0 & \text{if $\bar{\bw}_n=1$,} \\
    \bar\bv_n \pmod{2} & \text{if $\bar{\bw}_n=0$,}
\end{cases}
\]
and $\maxD(\zeta)= \blup^{-1}([\bar{\bv} -\bq,\bone])$
where
\[
\bq_n = \begin{cases}
    0 & \text{if $\bar{\bw}_n=1$,} \\
    1 - \bar\bv_n \pmod{2} & \text{if $\bar{\bw}_n=0$.}
\end{cases}
\]
\end{prop}

\begin{proof}
By Lemma~\ref{lem:top-cell-formula}, $[\bar\bv-\bq,\bone] \in \Top_{\cX_b}([\bar\bv,\bar\bw])$ if and only if $\bar\bw+\bq \leq_{\Z} \bone$, i.e., $\bq_n = 0$ whenever $\bar\bw_n=1$. 
Thus, by Equation \eqref{eq:b_inverse_dim} the dimension of $b^{-1}([\bar\bv-\bq,\bone])$ is given by
\begin{equation}
\label{eq:dimensionpreimage}
\# \setdef{n}{\bar\bv_n-\bq_n \equiv 1 \pmod{2},\ \bar\bw_n=0}+\# \setdef{n}{\bar\bv_n \equiv 1 \pmod{2},\ \bar\bw_n=1}    
\end{equation}
where the second term in the sum is fixed.
Thus, the unique means of minimizing or maximizing \eqref{eq:dimensionpreimage} is to set $\bq_n = \bar\bv_n \pmod{2}$ or $\bq_n = 1-\bar\bv_n \pmod{2}$ when $\bar\bw_n=0$, respectively.
\end{proof}

Recall from \eqref{eq:is_a_face} that 
\begin{equation}\label{eq:is_a_face_two}
    [\bv,\bw]\preceq[\bv',\bw']\ \text{if and only if $\bv = \bv'+\bu$ and $\bw +\bu \poeZ \bw'$}, 
\end{equation}
where $\bu\in\setof{0,1}^N$.

\begin{lem}
    \label{lem:orderedcells}
Let $\zeta\in\cX_b$. Then $\xi\in \blup^{-1}(\Top_{\cX_b}(\zeta)) \subset \cX$ if and only if
\[
\minD(\zeta) \preceq_\cX \xi \preceq_\cX \maxD(\zeta).
\]
\end{lem}
\begin{proof}
Let $\zeta=[\bar\bv,\bar\bw]$, $\xi=[\bv^\xi,\bw^\xi]\in \cX$ and 
\[\maxD(\zeta)=[\bv^{\maxD(\zeta)},\bw^{\maxD(\zeta)}].\]
If $\xi\in \blup^{-1}(\Top_{\cX_b}(\zeta))$,
then by \eqref{eq:is_a_face_two} it is enough to prove that there exists $\bu\in\setof{0,1}^N$ such that 
\begin{equation}\label{eq:the_u}
    \bv^\xi = \bv^{\maxD(\zeta)}+\bu \quad\text{and}\quad \bw^\xi +\bu \poeZ \bw^{\maxD(\zeta)}.
\end{equation}

We begin by showing that $\bw^\xi\poeZ \bw^{\maxD(\zeta)}$. In fact, fix $n$ and assume that $\bw^{\maxD(\zeta)}_n=1$. Since $\bw_n^\xi\in\setof{0,1}$, then $\bw^\xi_n\leq1=\bw^{\maxD(\zeta)}_n$. Now, assume that $\bw^{\maxD(\zeta)}_n=0$, by \eqref{eq:b_inverse_w} it follows that $\bar\bv_n$ is even, i.e., there exists $a\in\N$ such that $\bar\bv_n=2a$. Using \eqref{eq:b_inverse_v} we have that $2\bv^\xi+\bw_n^\xi=\bar\bv_n=2a$, thus $\bw_n^\xi=0$. Therefore, $\bw^\xi \poeZ \bw^{\maxD(\zeta)}$.

Now we have the necessary ingredients to find $\bu$ that satisfies \eqref{eq:the_u}.

Let $\bq^\xi,\bq^{\maxD(\zeta)}\in\setof{0,1}^N$ such that $\blup(\xi)=[\bar\bv - \bq^\xi, \bone]$ and $\blup(\maxD(\zeta))=[\bar\bv - \bq^{\maxD(\zeta)}, \bone]$. By \eqref{eq:b_inverse_v}, we have that
    \begin{align}
        2\bv_n^{\maxD(\zeta)} + \bw_n^{\maxD(\zeta)} &= \bar\bv_n - \bq_n^{\maxD(\zeta)} \label{eq:orderedcells_one}\\
        2\bv_n^\xi + \bw_n^\xi &= \bar\bv_n - \bq_n^\xi. \label{eq:orderedcells_two}       
    \end{align}
Subtracting \eqref{eq:orderedcells_two} from \eqref{eq:orderedcells_one}, it follows that 
    \begin{equation}\label{eq:orderedcells_three}
    2(\bv^{\maxD(\zeta)}_n-\bv^\xi_n) + (\bw_n^{\maxD(\zeta)}-\bw_n^\xi) = \bq_n^\xi   - \bq_n^{\maxD(\zeta)} \in\setof{0,\pm 1}.
    \end{equation}   

Assume that $(\bw^{\maxD(\zeta)}_n-\bw^\xi_n)=0$, then by  \eqref{eq:orderedcells_three}$ (\bv^{\maxD(\zeta)}_n-\bv^\xi_n)=0$. 
Notice that \eqref{eq:is_a_face_two} is satisfied for $\bu=\bzero$, thus $\xi\preceq \maxD(\zeta)$.

Assume that $(\bw^{\maxD(\zeta)}_n-\bw^\xi_n)=1$, then  by \eqref{eq:orderedcells_three} either $(\bv^{\maxD(\zeta)}_n-\bv^\xi_n)=-1$ or $(\bv^{\maxD(\zeta)}_n-\bv^\xi_n)=0$. In both cases, \eqref{eq:is_a_face_two} is satisfied by selecting $\bu_n=1$ for $(\bv^{\maxD(\zeta)}_n-\bv^\xi_n)=-1$ or $\bu_n=0$ for $(\bv^{\maxD(\zeta)}_n-\bv^\xi_n)=0$. Hence $\xi\preceq \maxD(\zeta)$.

Therefore, $\xi\preceq \maxD(\zeta)$. Analogously, we can prove that $\minD(\zeta)\preceq \xi$.

We leave to the reader to prove the sufficient condition.
\end{proof}

\begin{rem}
\label{rem:cells_between}
Let $\zeta = [\bar\bv, \bar\bw] \in \cX_b$ and let $N' = \dim(\maxD(\zeta))$. For any $[\bar\bv - \bq,\bone], [\bar\bv - \bq',\bone] \in \Top_{\cX_b}(\zeta)$, define $\bu = \bq - \bq'$ and let $U = \{ j \in \{1, \ldots, N'\} \mid \bu_j \neq 0 \}$. Then,
\[
\left[ \bar\bv - \bq + \sum_{i \in U'} \bu_{j_i} \bzero^{(i)}, \bone \right] \in \Top_{\cX_b}(\zeta)
\]
for any subset $U' \subset U$.
\end{rem}

Recall that $p$ is the relative position vector (see Definition~\ref{defn:rpvector}). 
Given $\zeta\in\cX_b$ the following lemma relates the position vector of cells in $\blup^{-1}\left(\Top_{\cX_b}(\zeta) \right)$ with the lowest and highest dimensional cells in $\blup^{-1}\left(\Top_{\cX_b}(\zeta) \right)$.

\begin{lem}\label{lem:sigma_sigma_prime_position}
Let $\zeta\in\cX_b$ and let $\alpha,\beta\in \blup^{-1}(\Top_{\cX_b}(\zeta)) \subset \cX$ such that $\alpha\prec_\cX \beta$.
If $n\in\Ex(\alpha,\beta)=J_i(\alpha)\cap J_e(\beta)$, then 
\[
p_n(\minD(\zeta), \alpha)=p_n(\beta, \maxD(\zeta))=0
\]
and
\[
p_n(\minD(\zeta), \beta)=p_n(\minD(\zeta),\maxD(\zeta))=p_n(\alpha, \maxD(\zeta)),
\]
which is non-zero. 
\end{lem}
\begin{proof}
Since $\minD(\zeta)$ and $\maxD(\zeta)$ are the unique lowest and highest dimensional cells in $\blup^{-1}(\Top_{\cX_b}(\zeta))$, by Lemma~\ref{lem:orderedcells} \[
\minD(\zeta)\preceq_\cX\alpha\prec_\cX\beta\preceq_\cX\maxD(\zeta).
\]
Thus, $n\in\Ex(\minD(\zeta),\maxD(\zeta))$.

Consider
$\alpha=[\bv^\alpha, \bw^\alpha]$, $\beta=[\bv^\beta, \bw^\beta]$, $\minD(\zeta)=[\bv^{\minD(\zeta)}, \bw^{\minD(\zeta)}]$ and 
$\maxD(\zeta)=[\bv^{\maxD(\zeta)}, \bw^{\maxD(\zeta)}]$.
Notice that, $\alpha\prec_\cX\beta$ and $n\in\Ex(\alpha,\beta)$ implies that $\bw^\alpha_n=0$ and $\bw^\beta_n=1$. Similarly, $\minD(\zeta)\prec_\cX\maxD(\zeta)$ and $n\in\Ex(\alpha,\beta)$ implies that $\bw^{\minD(\zeta)}_n=0$ and $\bw^{\maxD(\zeta)}_n=1$.

The formula for the position vector is a consequence of Proposition~\ref{prop:position_trio}.
\end{proof}    

\begin{lem}\label{lem:alpha_less_beta}
    Fix $[\bar\bv,\bar{\bw}]\in\cX_b$ and $n\in J_i([\bar\bv,\bar{\bw}])$.
Let $\bq\in\setof{0,1}^N$ be such that $[\bar\bv-\bq,\bone], [\bar\bv-\bq +\bzero^{(n)},\bone]\in\Top_{\cX_b}([\bar\bv,\bar{\bw}])$. 

If $\alpha = \blup^{-1}([\bar\bv-\bq,\bone])$ and $\beta = \blup^{-1}([\bar\bv-\bq+\bzero^{(n)},\bone])$, then
\[
p_n(\alpha, \beta)= \begin{cases}
    -1 & \text{if } \alpha\preceq_\cX \beta \\ 
    1 & \text{if } \beta \preceq_\cX \alpha.
\end{cases}
\]

\end{lem}
\begin{proof}
Denote $[\bv, \bw]= \blup^{-1}([\bar\bv-\bq,\bone])$ and $[\bv', \bw']=\blup^{-1}([\bar\bv-\bq+\bzero^{(n)},\bone]$.\\

Assume that $[\bv, \bw]\prec_\cX [\bv', \bw']$. 
Then, \eqref{eq:b_inverse_w} implies that $\bw_n=0$ and $\bw'_n=1$. By \eqref{eq:b_inverse_v}, it follows that
\[
2\bv_n + \bw_n = \bar\bv_n -1 \quad \text{and} \quad 2\bv_n' + \bw_n' = \bar{\bv}_n.
\]
Hence,
\[
2\bv_n + 0 = \bar{\bv}_n -1 \quad \text{and} \quad 2\bv_n' + 1 = \bar{\bv}_n,
\]
thus $\bv_n=\bv_n'$.\\
From the formula for the relative position vector (Definition~\ref{defn:rpvector}), we have that
\[
 p_n([\bv,\bw],[\bv',\bw']) = (-1)^{\bv_n-\bv'_n}\left(\bw_n-\bw'_n \right) =  (-1)^{0}\left(0-1 \right) = -1.
\]
Assume that $[\bv', \bw']\prec_\cX [\bv, \bw]$. Analogously, it follows that 
\[
    p_n([\bv',\bw'],[\bv,\bw]) = 1.
\]
Notice that flipping the roles of $[\bv,\bw]$ and $[\bv',\bw']$ in \eqref{defn:rpvector} yields
$p_n([\bv,\bw], [\bv',\bw']) = -1$.
\end{proof}

\section{Extendable D-gradings and AB-lattices in the Blow-up Complex}
\label{sec:Dgradblowup}
Recall that given a cubical cell complex $\cX=\cX(\I)$ as in Definition~\ref{defn:Xcomplex} and a wall labeling $\omega \colon W(\cX) \to \setof{\pm 1}^N$, we defined a rook field $\rook : \TP(\cX)\to \setof{0,\pm1}^N$ and multivalued maps $\cF_i \colon \cX \mvmap \cX$, $i=0,1,2,3$ (see Definitions~\ref{def:Rule0}--\ref{defn:Rule3}). Each $\cF_i$ decomposes $\cX$ into strongly connected components. We refer to the quotient map $\pi : \cX \to \SCC(\cF_i)$ as the D-grading for $\cF_i$ (see Definition~\ref{defn:Dgrading}). 

The primary aim of this section is to prove the following theorem.
\begin{thm}
\label{thm:singlevalue}
Consider $\cF_i \colon \cX \mvmap \cX$ with $i \in \{0,1,2\}$ and D-grading $\pi \colon \cX \to \SCC(\cF_i)$. Define the map $\pi_b \colon \Top_{\cX_b} \to \SCC(\cF_i)$ by
\[
\pi_b(\zeta) = \pi(b^{-1}(\zeta)),
\]
where $b \colon \cX \to \cX^{(N)}_b$ is the blowup map. 
Then $\pi_b$ is an extendable grading. Furthermore, if $N \leq 3$, the same holds for $i=3$.
\end{thm}

We discuss the relevance of Theorem~\ref{thm:singlevalue}  before turning to its proof.
We begin with a definition.

\begin{defn}
Consider $\cF_i \colon \cX \mvmap \cX$ with $i \in \{0,1,2,3\}$ and D-grading $\pi \colon \cX \to \SCC(\cF_i)$.
The \emph{D-grading} $\pi_b\colon (\cX_b,\preceq_b) \to (\SCC(\cF_i),\posetF{i})$ is defined to be the extension of $\pi_b$ from Theorem~\ref{thm:singlevalue}.
\end{defn}

The importance of Corollary~\ref{cor:birkhoff} is the following theorem.
\begin{thm}
\label{thm:NFABlattice}
Consider the multivalued map $\cF_i\colon \cX\mvmap \cX$, $i\in\setof{0,1,2,3}$ with
D-grading $\pi_b\colon (\cX_b,\preceq_b)\to (\SCC(\cF_i),\posetF{i})$.
Then 
\[
\sN(\cF_i,\pi_b) \coloneqq \sO(\pi_b)(\sO(\SCC(\cF_i))) \subset \sO(\cX_b)
\]
is an AB-lattice (see Definition~\ref{defn:ABlattice}).
\end{thm}
\begin{proof}
Observe that $\emptyset,\SCC(\cF_i)\in \sO(\SCC(\cF_i))$, so 
\[
    \sO(\pi_b)(\emptyset) = \pi_b^{-1}(\emptyset) = \emptyset, \quad \sO(\pi_b)(\SCC(\cF_i)) = \pi_b^{-1}(\SCC(\cF_i)) = \cX_b,
\]
hence $\emptyset, \cX_b\in \sN(\cF_i,\pi_b)$. 
Assume that $\cK \in \sN(\cF_i,\pi_b)\setminus \emptyset$. 
That is, a non-empty $\cK\in \sO(\SCC(\cF_i))$.
By definition of $\sO(\SCC(\cF_i))$, 
\[
\cK = \bigcup_{p\in K} \sO(p)
\]
for some $K \subset \SCC(\cF_i)$.
Since 
$\blup(\pi^{-1}(p))\subseteq\cX_b^{(N)}$, 
$\sO(p)$ is a uniform cell complex of dimension $N$ and hence $\cK$ is a uniform cell complex of dimension $N$.
\end{proof}

Recall from the discussion in Chapter~\ref{sec:geometrizationCellComplex} that Corollary~\ref{cor:existenceAttractorLattice} is of fundamental importance as it provides the conduit by which combinatorial/homological information is translated into information about continuous dynamics.
The hypothesis of Corollary~\ref{cor:existenceAttractorLattice} requires the existence of an AB-lattice; $\sN(\cF_i,\pi_b)$ is  the AB-lattice used in applications.

The following proposition addresses the values of D-grading $\pi_b$ on the neighboring cells along the boundary of an element in the AB-lattice. This information is essential for determining the direction in which the continuous dynamics (vector field) should point, as discussed in Definition~\ref{defn:aligned} and Theorem~\ref{thm:LipAttBlock}.

\begin{prop}
    \label{lem:boundaryABlattice}
Consider $\cF_i\colon \cX\mvmap \cX$ with $i\in\setof{0,1,2,3}$ and D-grading $\pi\colon \cX \to \SCC(\cF_i)$.
Let $\cN \in \sN(\cF_i,\pi_b)$ (see Theorem~\ref{thm:NFABlattice}). 
Given $\zeta\in \bbdy(\cN)^{(N-1)}$, let $\sigma_0\in \cN^{(N)}$ be such that 
$\zeta\preceq_b \sigma_0$.
If there exists a distinct $\sigma_1\in \cX_b^{(N)}$ such that 
$\zeta\preceq_b \sigma_1$, then 
$\pi_b(\sigma_0) <_{\bar\cF_i} \pi_b(\sigma_1)$.
\end{prop}
\begin{proof}

Notice that $\cN \in \sN(\cF_i,\pi_b)$ can be decomposed as a union of down-sets 
\[
\cN=\bigcup_{p\in P}\sO(\pi_b)(\sO(p)),
\]
where $P\subset \SCC(\cF_i)$. Hence
    \begin{equation}\label{eq:N_iff_leq_p}
        \xi\in\cN \text{ if and only if } \pi_b(\xi) \leq_{\bar{\cF_i}} p \text{ for some } p\in P.
    \end{equation}
Given $\zeta\in \bbdy(\cN)^{(N-1)}$, by Definition~\ref{def:boundary_prime}, there exists a unique $\sigma_0 \in \cX_b^{(N)}$ such that $\zeta\preceq_b \sigma_0$ and $\sigma_0\in\cN$. Thus, if there exists a distinct $\sigma_1\in \cX_b^{(N)}$ such that $\zeta\preceq_b \sigma_1$, then $\sigma_1\notin \cN$.
    Note that,
    $\Top_{\cX_b}(\zeta) = \setof{\sigma_0, \sigma_1}$ then
    \begin{eqnarray*}
        \pi_b(\zeta) &=& \min\setdef{\pi(\xi)}{\xi\in b^{-1}(\Top_{\cX_b}(\zeta))} \\ &=& \min\setof{\pi(\blup^{-1}(\sigma_0)), \pi(\blup^{-1}(\sigma_1))} \\ 
        &=& \min\setof{\pi_b(\sigma_0), \pi_b(\sigma_1)}.
    \end{eqnarray*}
    Recall that, $\pi$ being a D-grading implies that  there exists a unique minimum element in $\min\setdef{\pi(\xi)}{\xi\in b^{-1}(\Top_{\cX_b}(\zeta))}$. Hence $\pi_b(\sigma_0)$ and $\pi_b(\sigma_1)$ are comparable, i.e., either $\pi_b(\sigma_0) <_{\bar{\cF_i}} \pi_b(\sigma_1)$ or $\pi_b(\sigma_1) \leq_{\bar{\cF_i}} \pi_b(\sigma_0)$. 
    
    Assume that $\pi_b(\sigma_1) \leq_{\bar{\cF_i}} \pi_b(\sigma_0)$. By \eqref{eq:N_iff_leq_p}, it follows that,
    for some $p\in P$,
    \[\pi_b(\sigma_1) \leq_{\bar{\cF_i}} \pi_b(\sigma_0) \leq_{\bar{\cF_i}} p.\]
    Consequently, by \eqref{eq:N_iff_leq_p}, $\sigma_1 \in \cN$, a contradiction. Therefore, $\pi_b(\sigma_0) <_{\bar{\cF_i}} \pi_b(\sigma_1)$.
\end{proof}

We provide separate proofs of Theorem~\ref{thm:singlevalue} for $i=0$, 1, 2, and 3.

\begin{proof}[Proof of Theorem~\ref{thm:singlevalue} for $i=0$]
By Example~\ref{ex:trivialSCC}, $\SCC(\cF_0)$ consists of a single element.
Thus, $\pi_b \colon \Top_{\cX_b}(\cX_b) \to \SCC(\cF_0)$ takes on a single value and therefore is an extendable grading.    
\end{proof}

Recall that for any $\zeta \in \cX_b$, $\minD(\zeta)$ and $\maxD(\zeta)$ denote the cells of minimum and maximum dimension  in $\blup^{-1}(\Top_\cX(\zeta))$, respectively.

\begin{lem}\label{lem:F_1-leq-geq}
Consider $\cF_1\colon \cX\mvmap \cX$ with D-grading $\pi\colon \cX \to \SCC(\cF_1)$. 
Fix $\zeta=[\bar{\bv},\bar{\bw}]\in\cX_b$ and $n\in J_i([\bar{\bv},\bar{\bw}])$. 
Let $\mu\in\Top_\cX(\maxD(\zeta))$.
Assume $[\bar{\bv}-\bq,\bone], [\bar{\bv}-\bq + \bzero^{(n)},\bone]\in \Top_{\cX_b}(\zeta)$.

If $\rook_n(\minD(\zeta),\mu)=-1$, then
\[
\pi\left(\blup^{-1}\left([\bar{\bv}-\bq,\bone]\right)\right)\posetF{1} \pi\left(\blup^{-1}([\bar{\bv}-\bq + \bzero^{(n)},\bone])\right).
\]

If $\rook_n(\minD(\zeta),\mu)=1$, then
\[
 \pi\left(\blup^{-1}([\bar{\bv}-\bq + \bzero^{(n)},\bone])\right) \leq_{\bar\cF_1} \pi\left(\blup^{-1}([\bar{\bv}-\bq,\bone])\right).
\]
\end{lem}

\begin{proof}
Set $\alpha=\blup^{-1}([\bar\bv-\bq,\bone])$ and $\beta=\blup^{-1}([\bar\bv-\bq + \bzero^{(n)},\bone])$.
By Lemma~\ref{lem:q_face}, either $\alpha\prec_\cX\beta$ or $\beta\prec_\cX\alpha$. 

We begin by assuming that $\alpha\prec_\cX\beta$. Then, Lemma~\ref{lem:alpha_less_beta} implies that $p_n(\alpha, \beta)=-1$. Furthermore, since $\minD(\zeta)$ and $\maxD(\zeta)$ are the unique lowest and highest dimensional cells in $\blup^{-1}(\Top_{\cX_b}(\zeta))$, by Lemma~\ref{lem:orderedcells} $\minD(\zeta)\preceq_\cX\alpha\prec_\cX\beta\preceq_\cX\maxD(\zeta)$.

Let $\mu\in\Top_\cX(\maxD(\zeta))$.
Then, $\mu$ is a top cell such that $\maxD(\zeta)\preceq_\cX\mu$, and thus, $\minD(\zeta)\preceq_\cX\alpha\prec_\cX\beta\preceq_\cX\maxD(\zeta)\preceq_\cX\mu$. 

Now assume that $\rook_n(\minD(\zeta),\mu)=-1$. 
Given that $\alpha\prec\mu$, then by Corollary~\ref{cor:mu*}
\[
    \rook_n(\alpha, \mu) = \rook_n(\minD(\zeta), \mu) = -1.
\]
Since $p_n(\alpha,\beta)=-1$, then $p_n(\alpha,\beta)=\rook_n(\alpha, \mu)$.    

As a consequence  $\alpha$ is not an entrance face of $\beta$,
and hence, $\beta\rightarrow_{\cF_1} \alpha$, that is, $\alpha \in \cF_1(\beta)$. 
Therefore, $\pi\left(\alpha\right)\posetF{1} \pi\left(\beta\right)$.

Now assume that $\rook_n(\minD(\zeta),\mu)=1$. 
Given that $\alpha\prec\mu$, then
\[
    \rook_n(\alpha, \mu) = \rook_n(\minD(\zeta), \mu) = 1.
\]
Since $p_n(\alpha,\beta)=-1$, then $p_n(\alpha,\beta)=-\rook_n(\alpha, \mu)$.
Hence,  $\alpha$ is not an exit face of $\beta$, and thus, $\alpha\rightarrow_{\cF_1} \beta$. 
Therefore, $\pi(\beta)\posetF{1}\pi(\alpha)$.

The result follows analogously, assuming that $\beta\prec_\cX\alpha$.
\end{proof}

\begin{proof}[Proof of Theorem~\ref{thm:singlevalue} for $i=1$]
The strategy of the proof is as follows.
Given $\zeta=[\bar\bv,\bar\bw]\in\cX_b$ we construct a cell $\gamma\in\blup^{-1}(\Top_{\cX_b}(\zeta))$ such that 
\[
\pi(\gamma) = \min\setdef{\pi(\xi)}{\xi \in \blup^{-1}(\Top_{\cX_b}(\zeta))}
\]
and show that $\pi(\gamma)$ is unique.

To construct $\gamma$ consider $\mu\in\Top_\cX(\maxD(\zeta))$. Choose $\bq\in\setof{0,1}^N$ such that
\[ 
\bq_n =
\begin{cases}
    \frac{1-\rook_n(\minD(\zeta),\mu)}{2}, & \text{if $n\in J_i([\bar\bv,\bar\bw]),$}\\
    0, & \text{otherwise.}
\end{cases}
\]
and set $\gamma = \blup^{-1}([\bar\bv -\bq,\bone])$. 
Note that, by definition, if $n\in J_e([\bar\bv,\bar\bw])$ then $\bq_n =0$.
Therefore, by Lemma~\ref{lem:q_property}, it follows that $\gamma\in \blup^{-1}(\Top_{\cX_b}(\zeta))$.

Consider $\gamma' \in \blup^{-1}(\Top_{\cX_b}(\zeta))$.
To complete the proof it suffices to show that $\pi(\gamma)\posetF{1}\pi(\gamma')$.
By Lemma~\ref{lem:top-cell-formula} there exists $\bq'\in\setof{0,1}^N$ such that $\gamma'=\blup^{-1}([\bar\bv -\bq',\bone])$.
Set $\bu=\bq-\bq'$ and $\setof{j_1, \ldots, j_\ell} = \setdef{j\in\setof{1, \ldots, N}}{\bu_j\neq 0}$. 

Notice that, if $\bq_n=1$ then 
$\rook_n(\minD(\zeta),\mu)=-1$. Thus, by Lemma~\ref{lem:F_1-leq-geq}, it follows that $\pi(\blup^{-1}([\bar\bv -\bq'',\bone])) \posetF{1} \pi(\blup^{-1}([\bar\bv -\bq''+\bzero^{(n)},\bone]))$ for any $\bq''\in\setof{0,1}^N$ such that $[\bar\bv-\bq'',\bone], [\bar\bv-\bq'' +\bzero^{(n)},\bone]\in \Top_{\cX_b}(\zeta)$. 

Hence, 
\begin{align*}
    \pi(\gamma) = \pi\left(\blup^{-1}\left([\bar\bv -\bq,\bone]\right)\right) &\posetF{1} \pi\left(\blup^{-1}\left([\bar\bv -\bq + \bu_{j_1}\bzero^{(j_1)},\bone]\right)\right)\\
    &\posetF{1} \pi\left(\blup^{-1}\left([\bar\bv -\bq + \bu_{j_1}\bzero^{(j_1)}+\bu_{j_2}\bzero^{(j_2)},\bone]\right)\right)\\
    & \vdots \\
    &\posetF{1} \pi\left(\blup^{-1}\left([\bar\bv -\bq + \sum_{i=1}^\ell\bu_{j_i}\bzero^{(j_i)},\bone]\right)\right)\\
    &= \pi\left(\blup^{-1}\left([\bar\bv -\bq + \bu,\bone]\right)\right)\\
    &= \pi\left(\blup^{-1}\left([\bar\bv -\bq',\bone]\right)\right) = \pi(\gamma').
\end{align*}
As a consequence $\pi(\gamma)\posetF{1} \pi(\gamma')$, for all $\gamma'\in\blup^{-1}(\Top_{\cX_b}(\zeta))$. Therefore $\pi(\gamma)$ is the unique minimal element in $\pi(\blup^{-1}(\Top_{\cX_b}(\zeta)))$.
\end{proof}

The proof Theorem~\ref{thm:singlevalue} for $i=2$ is more subtle, and thus, we introduce the following definition and lemmas.

\begin{defn}\label{defn:adrift}
Let $\zeta=[\bar\bv, \bar\bw]\in \cX_b$. 
We call $k\in J_i(\zeta)$ a \emph{indecisive drift direction} in $\Top_{\cX_b}(\zeta)$ if there exists $\bq\in\setof{0,1}^N$ with $[\bar\bv-\bq ,\bone], [\bar\bv-\bq+\bzero^{(k)},\bone]\in \Top_{\cX_b}(\zeta)$ such that the pairs
\[
\left(\blup^{-1}\left([\bar\bv-\bq,\bone]\right),\blup^{-1}([\bar\bv-\bq+\bzero^{(k)},\bone])\right)
\] 
or 
\[
\left(\blup^{-1}([\bar\bv-\bq+\bzero^{(k)},\bone]),\blup^{-1}\left([\bar\bv-\bq,\bone]\right)\right)
\] 
exhibit indecisive drift (see Definition~\ref{defn:indecisive}).
\end{defn}
The following lemma is the version of Lemma~\ref{lem:F_1-leq-geq} for $\cF_2$. 
\begin{lem}\label{lem:F_2-leq-geq}
Consider $\cF_2\colon \cX\mvmap \cX$ with D-grading $\pi\colon \cX \to \SCC(\cF_2)$. 
Let $\zeta=[\bar{\bv},\bar{\bw}]\in\cX_b$ and assume that $k\in J_i(\zeta)$ is not an indecisive drift direction in $M\subset\Top_{\cX_b}(\zeta)$. 
Let $\mu\in\Top_\cX(\maxD(\zeta))$.
Assume $[\bar{\bv}-\bq,\bone], [\bar{\bv}-\bq+\bzero^{(k)},\bone]\in M$.

If $\rook_n(\minD(\zeta),\mu)=-1$, then
\[
\pi\left(\blup^{-1}\left([\bar{\bv}-\bq,\bone]\right)\right)\posetF{2} \pi\left(\blup^{-1}([\bar{\bv}-\bq+\bzero^{(k)},\bone])\right).
\]

If $\rook_n(\minD(\zeta),\mu)=1$, then
\[
\pi\left(\blup^{-1}([\bar{\bv}-\bq+\bzero^{(k)},\bone])\right) \posetF{2}
\pi\left(\blup^{-1}\left([\bar{\bv}-\bq,\bone]\right)\right).
\]
\end{lem}

\begin{proof}
Since, $k$ is not an indecisive drift direction, neither 
 \[
    \left(\blup^{-1}\left([\bar\bv-\bq,\bone \right), \blup^{-1}([\bar\bv-\bq+\bzero^{(k)},\bone])\right)
\]
nor 
\[
    \left(\blup^{-1}\left([\bar\bv-\bq+\bzero^{(k)},\bone]\right)\right), \blup^{-1}\left([\bar\bv-\bq,\bone])\right),
\]
exhibit indecisive drift.
Hence,  by Definition~\ref{defn:F2}, 
\[
\cF_2\left(\blup^{-1}([\bar\bv-\bq,\bone])\right)=\cF_1\left(\blup^{-1}([\bar\bv-\bq,\bone])\right)
\] 
and 
\[
\cF_2\left(\blup^{-1}([\bar\bv-\bq+\bzero^{(k)},\bone])\right)=\cF_1\left(\blup^{-1}([\bar\bv-\bq+\bzero^{(k)},\bone])\right),
\] 
for all $\bq\in\setof{0,1}^N$ with $[\bar\bv-\bq,\bone], [\bar\bv-\bq+\bzero^{(k)},\bone]\in M$.
Therefore, the result follows by Lemma~\ref{lem:F_1-leq-geq}.
\end{proof}

\begin{proof}[Proof of Theorem~\ref{thm:singlevalue} for $i=2$]
The overall strategy of the proof is the same as for the case $i=1$; given $\zeta=[\bar\bv,\bar\bw]\in\cX_b$, construct $\gamma\in\cX$ such that 
\[
\pi(\gamma) = \min\setdef{\pi(\xi)}{\xi \in \blup^{-1}(\Top_{\cX_b}(\zeta))}
\]
and show that $\pi(\gamma)$ is unique.
However, the construction of $\gamma$ is more involved and makes use of the following setup.

As a preliminary step, note that if there are no indecisive drift directions in $\Top_{\cX_b}(\zeta)$, {\bf Condition 2.1} does not apply. Consequently, $\cF_2$ coincides with $\cF_1$, allowing us to apply the proof for $i=1$.

Thus, we can assume that there exists an indecisive drift direction $n$ in $\Top_{\cX_b}(\zeta)$. 
Let $(\xi_n,\xi_n')$ be a pair in $\blup^{-1}(\Top_{\cX_b}(\zeta))$ that exhibits indecisive drift. 

Let $\alpha^n\in\blup^{-1}(\Top_{\cX_b}(\zeta))$ such that
\[
\Ex(\alpha^n,\maxD(\zeta))=\setof{n}.
\]
Note that Lemma~\ref{lem:orderedcells} guarantees that $\alpha^n$ exist. 
Now, recall that $\Back(\xi_n,\xi_n')$ is given in Definition~\ref{defn:back-walls} and choose 
$(\hat\xi_n, \hat\xi_n')\in \Back(\xi_n,\xi_n')$. Let
\[ 
\delta_n =
\begin{cases}
    \rook_n(\hat\xi_n, \hat\xi_n'), &\text{if $(\hat\xi_n, \hat\xi_n')\in\Back(\alpha^n,\maxD(\zeta))$,}\\
    -\rook_n(\hat\xi_n, \hat\xi_n'), &\text{otherwise.}
\end{cases}
\]

Consider $\mu\in\Top_\cX(\maxD(\zeta))$ and choose $\bq\in\setof{0,1}^N$ such that
\[ 
\bq_n =
\begin{cases}
    (1-\delta_n)/2, &\text{if $n$ is an indecisive drift direction},\\
    \frac{1-\rook_n(\minD(\zeta),\mu)}{2}, &\text{if $n$ is not an indecisive drift direction} \\
    0, &\text{otherwise.}
\end{cases}
\]

Finally,  set $\gamma \coloneqq \blup^{-1}([\bar\bv -\bq,\bone])$. By Lemma~\ref{lem:q_property}, it follows that $\gamma\in \blup^{-1}(\Top_{\cX_b}(\zeta))$, since $\bq_n=0$ if  $n\in J_e([\bar\bv,\bar\bw])$.

It remains to show that $\pi(\gamma)=\min\setdef{\pi(\xi)}{\xi \in \blup^{-1}(\Top_{\cX_b}(\zeta))}$. 
Let $\gamma' \in \blup^{-1}(\Top_{\cX_b}(\zeta))$.
By Lemma~\ref{lem:top-cell-formula} there exists $\bq'\in\setof{0,1}^N$ such that $\gamma'=\blup^{-1}([\bar\bv -\bq',\bone])$.

Set $\bu=\bq-\bq'$ and let $U=\setdef{j\in\setof{1, \ldots, N}}{\bu_j\neq 0}$. Note that, by 
Remark~\ref{rem:cells_between},
    \[[\bar\bv -\bq + \sum_{i\in U' }\bu_{j_i}\bzero^{(i)},\bone]\in \Top_{\cX_b}(\zeta)\]
for any $U'\subset U$.
Observe that Proposition~\ref{prop:unique_ext_GO_pairs} and Proposition~\ref{prop:not_L_GO_pairs} allow us to order the set $U = \setof{j_1, \ldots, j_\ell}$ which has the following properties:
\begin{enumerate}
    \item[A)] Let 
    \[
    \eta_0 \coloneqq [\bar\bv -\bq,\bone]  \quad \text{and} \quad
    \eta_k \coloneqq [\bar\bv -\bq + \sum_{i=1}^{k}\bu_{j_i}\bzero^{(j_i)},\bone],
    \] 
    if $k\in\setof{1, \dots, \ell-1}$ then $(\eta_{k-1},\eta_{k})$ is not a pair that exhibit indecisive drift. 
    
    \item[B)] Let $k\in \setof{1, \ldots, \ell}$ and $k'\in \setof{1, \ldots, {\ell-1}}$ such that $k$ is indecisive drift direction and $k'$ is not an indecisive drift direction, then $k\leq k'$.
\end{enumerate}

It suffices to show that $\pi_b(\eta_{k-1})\posetF{2} \pi_b(\eta_{k})$ for all $k\in\setof{1, \dots, \ell}$, as this leads to conclude that  
\begin{align*}
    \pi(\gamma) = \pi\left(\blup^{-1}\left([\bar\bv -\bq,\bone]\right)\right) &= \pi_b(\eta_0) \\
    &\posetF{2} \pi_b(\eta_1)\\
    &\ \vdots\\ 
    &\posetF{2} \pi_b(\eta_{\ell}) \\
    &= \pi_b\left([\bar\bv -\bq + \sum_{i=1}^\ell\bu_{j_i}\bzero^{(j_i)},\bone]\right)\\
    &= \pi_b\left([\bar\bv -\bq + \bu,\bone]\right) = \pi_b\left([\bar\bv -\bq',\bone]\right) = \pi(\gamma').
\end{align*}
Hence $\pi(\gamma)\posetF{2} \pi(\gamma')$, for all $\gamma'\in\blup^{-1}(\Top_{\cX_b}(\zeta))$. Therefore $\pi(\gamma)$ is the unique minimal element in $\pi(\blup^{-1}(\Top_{\cX_b}(\zeta)))$.

Now we turn our focus to proving that $\pi_b(\eta_{k-1})\posetF{2} \pi_b(\eta_{k})$, for any $k\in\setof{1, \dots, \ell}$. 

We begin by assuming that $k$ is not an indecisive drift direction.
If $\rook_k(\minD(\zeta),\mu)=-1$, then $\bq_k=1$. 
Since $u_k\neq 0$ then $\bq'_k\neq 1$, which implies that $\bq'_k = 0$, and thus $u_k=1$. Hence, 
\begin{equation}\label{eq:u=1_n_ks}
    \eta_{k-1} = [\bar\bv -\bq + \sum_{i=1}^{k-1}\bu_{j_i}\bzero^{(j_i)},\bone]
    \quad \text{and} \quad 
    \eta_{k} = [\bar\bv -\bq + \sum_{i=1}^{k-1}\bu_{j_i}\bzero^{(j_i)} + \bzero^{(k)},\bone].
\end{equation}
It follows from Lemma~\ref{lem:F_2-leq-geq} that 
$\pi_b(\eta_{k-1})\posetF{2} \pi_b(\eta_{k})$. Similarly, if $\rook_k(\minD(\zeta),\mu)=1$ then $\bq_k=0$, thus $u_k=-1$. Hence, by Lemma~\ref{lem:F_2-leq-geq}, $\pi_b(\eta_{k-1})\posetF{2} \pi_b(\eta_{k})$.

For the remainder of the proof we assume that $k$ is an indecisive drift direction. 
We break this proof into two cases, $(\hat\xi_k, \hat\xi_k')\in\Back(\alpha^k,\maxD(\zeta))$ and $(\hat\xi_k, \hat\xi_k')\notin\Back(\alpha^k,\maxD(\zeta))$.

Assume $(\hat\xi_k, \hat\xi_k')\in\Back(\alpha^k,\maxD(\zeta))$.
    \begin{enumerate}
        \item If $\rook_k(\hat\xi_k, \hat\xi_k')=-1$ 
            then $\bq_k=1$, thus $u_k=1$. Hence, by \eqref{eq:u=1_n_ks} and Lemma~\ref{lem:alpha_less_beta}, it follows that
            \[
            p_k(\blup^{-1}(\eta_{k-1}),\blup^{-1}(\eta_{k})) = -1.\]
            Hence,
            \[p_k(\blup^{-1}(\eta_{k-1}),\blup^{-1}(\eta_{k})) = -1= \rook_k(\hat\xi_k, \hat\xi_k')\\
            = \rook_k(\blup^{-1}(\eta_{k-1}), \hat\xi_k') 
            \]

            \begin{enumerate}
                \item  If $(\blup^{-1}(\eta_{k-1}),\blup^{-1}(\eta_{k}))$ or $(\blup^{-1}(\eta_{k}), \blup^{-1}(\eta_{k-1}))$ is not a pair that exhibits indecisive drift, then
                $\blup^{-1}(\eta_{k}) \rightarrow_{\cF_2} \blup^{-1}(\eta_{k-1})$. Hence $\pi_b(\eta_{k-1})\posetF{2} \pi_b(\eta_{k})$.
                
                \item If $(\blup^{-1}(\eta_{k-1}),\blup^{-1}(\eta_{k}))$ or $(\blup^{-1}(\eta_{k}), \blup^{-1}(\eta_{k-1}))$ is a pair that exhibits indecisive drift then {\bf Condition 2.1} implies that $\blup^{-1}(\eta_{k}) \rightarrow_{\cF_2} \blup^{-1}(\eta_{k-1})$. Hence $\pi_b(\eta_{k-1})\posetF{2} \pi_b(\eta_{k})$.
            \end{enumerate}

        \item If $\rook_k(\hat\xi_k, \hat\xi_k')=1$, then $\bq_k=0$, thus $u_k=-1$. Similarly to item (1), it follows that $\pi_b(\eta_{k-1})\posetF{2} \pi_b(\eta_{k})$.
    \end{enumerate}

Assume $(\hat\xi_k, \hat\xi_k')\notin\Back(\alpha^k,\maxD(\zeta))$.
Then, 
    there exists $(\hat\mu, \hat\mu')\in\Back(\alpha^k,\maxD(\zeta))$ such that $\rook_k(\hat\xi_k, \hat\xi_k') = - \rook_k(\hat\mu, \hat\mu')$. 
    
    \begin{enumerate}
        \item If $\rook_k(\hat\xi_k, \hat\xi_k')=1$, 
            then $\bq_k=1$ and $u_k=1$.  Hence, by \eqref{eq:u=1_n_ks} and Lemma~\ref{lem:alpha_less_beta}, it follows that
            \[
            p_k(\blup^{-1}(\eta_{k-1}),\blup^{-1}(\eta_{k})) = -1.\]
            Since 
            \[\rook_k(\blup^{-1}(\eta_{k-1}), \hat\mu')  = \rook_k(\hat\mu, \hat\mu')
            =-\rook_k(\hat\xi_k, \hat\xi_k') = -1,
            \]
            then
            \[
            p_k(\blup^{-1}(\eta_{k-1}),\blup^{-1}(\eta_{k})) = \rook_k(\blup^{-1}(\eta_{k-1}), \hat\mu').\]
            \begin{enumerate}
                \item  If $(\blup^{-1}(\eta_{k-1}),\blup^{-1}(\eta_{k}))$ or $(\blup^{-1}(\eta_{k}), \blup^{-1}(\eta_{k-1}))$ is not a pair that exhibits indecisive drift, then $\blup^{-1}(\eta_{k}) \rightarrow_{\cF_2} \blup^{-1}(\eta_{k-1})$. Hence $\pi_b(\eta_{k-1})\posetF{2} \pi_b(\eta_{k})$.

                \item If $(\blup^{-1}(\eta_{k-1}),\blup^{-1}(\eta_{k}))$ or $(\blup^{-1}(\eta_{k}), \blup^{-1}(\eta_{k-1}))$ is a pair that exhibits indecisive drift then, by property A) of the $n_k$'s, it follows that $k=\ell$. Let $(n_g,\ell)$ be an associated GO-pair for $(\blup^{-1}(\eta_{\ell-1}),\blup^{-1}(\eta_\ell))$, , see Definition~\ref{defn:GO-pair}. 
                By the definition of GO-pairs, $\ell$ does not  actively regulates $\ell$, then $n_g\neq \ell$.
                By property B) of $\eta_k$'s, it follows that $n_g\notin\setof{1, \ldots, \ell-1}$. Then $n_g\notin\setof{1, \ldots, \ell}$, which implies that $\bu_{n_g}=0$, see definition of $\bu$, thus $\bq_{n_g}=\bq'_{n_g}$. 
                
                By the assumption that $(\hat\xi_\ell, \hat\xi_\ell')\notin\Back(\alpha^\ell,\maxD(\zeta))$, it follows that
                $\rook_{n_g}(\minD(\zeta), \maxD(\zeta))=-p_{n_g}(\minD(\zeta), \maxD(\zeta))$.\\ Let $[\bq'',\bw'']=\maxD(\zeta).$ Thus, by definition of $\bq$: \\
                If $p_{n_g}(\minD(\zeta), \maxD(\zeta))=1$, then  $\bq_{n_g}''=1$ and $\bq_{n_g}=1$; 
                \\
                If $p_{n_g}(\minD(\zeta), \maxD(\zeta))=-1$, then $\bq_{n_g}''=0$ and $\bq_{n_g}=0$. Hence, 
                \[p_{n_g}(\blup^{-1}([\bar\bv-\bq, \bone]), \maxD(\zeta))=0,\] which implies that $n_g\in J_e(\blup^{-1}([\bar\bv-\bq, \bone]))$. As a consequence $n_g\in J_e(\eta_\ell)$, which yields a contradiction given that $(n_g,\ell)$ is a GO-pair for $(\blup^{-1}(\eta_{\ell-1}), \blup^{-1}(\eta_{\ell}))$, i.e., $n_g\in J_i(\eta_\ell)$.
            \end{enumerate}
        \item If $\rook_k(\hat\xi_k, \hat\xi_k')=-1$, then similarly to item (1), it follows that $\pi_b(\eta_{k-1})\posetF{2} \pi_b(\eta_{k})$.
    \end{enumerate}
\end{proof}

For $i=3$ the new element in the proof of Theorem~\ref{thm:singlevalue} arises from cycles of the local inducement map.
We begin by stating a version of Lemma~\ref{lem:F_1-leq-geq} for $\cF_3$.
\begin{lem}\label{lem:F_3-leq-geq}
Consider $\cF_3\colon \cX\mvmap \cX$ with D-grading $\pi\colon \cX \to \SCC(\cF_3)$. 
Let $\zeta=[\bar{\bv},\bar{\bw}]\in\cX_b$ and assume that $k\in J_i(\zeta)$ is not an indecisive drift direction in $\Top_{\cX_b}(\zeta)$ and $k$ is not in an image of a cycle of any cell in $\blup^{-1}(\Top_{\cX_b}(\zeta))$.
Let $\mu\in\Top_\cX(\maxD(\zeta))$ and assume $[\bar{\bv}-\bq,\bone], [\bar{\bv}-\bq+\bzero^{(k)},\bone]\in \Top_{\cX_b}(\zeta)$.

If $\rook_k(\minD,\mu)=-1$, then
\[
\pi\left(\blup^{-1}\left([\bar{\bv}-\bq,\bone]\right)\right)\posetF{3} \pi\left(\blup^{-1}([\bar{\bv}-\bq+\bzero^{(k)},\bone])\right).
\]

If $\rook_k(\minD,\mu)=1$, then
\[
\pi\left(\blup^{-1}([\bar{\bv}-\bq+\bzero^{(k)},\bone])\right) \posetF{3} \pi\left(\blup^{-1}\left([\bar{\bv}-\bq,\bone]\right)\right).
\]
\end{lem}
\begin{proof}
By hypothesis $k$ is not an indecisive drift direction in $\Top_{\cX_b}(\zeta)$ and is not in an image of a cycle of any cell in $\blup^{-1}(\Top_{\cX_b}(\zeta))$, hence
\[
\cF_3\left(\blup^{-1}([\bar{\bv}-\bq+\bzero^{(k)},\bone])\right)=\cF_1\left(\blup^{-1}([\bv-\bq+\bzero^{(k)},\bone])\right)
\] 
and 
\[
\cF_3\left(\blup^{-1}([\bar\bv-\bq,\bone])\right)=\cF_1\left(\blup^{-1}([\bar\bv-\bq,\bone])\right),
\] 
for all $\bq\in\setof{0,1}^N$ such that $[\bar{\bv}-\bq+\bzero^{(k)},\bone], [\bv-\bq,\bone]\in M$.
Therefore, the result follows by Lemma~\ref{lem:F_1-leq-geq}.
\end{proof}

Recall from \eqref{defn:unstable_cells} that the set of unstable cells about a semi-opaque cell $\xi\in\cX$ is 
denoted by $\cU(\xi)$.
We define the \emph{stable cells} about $\xi$ by 
\begin{equation}
\label{eq:stable_cells}
\Stable(\xi) = \bigcup_{\sigma} \setdef{\xi' \in \cX}{\xi\prec\xi', \Ex(\xi,\xi')\subseteq S_\sigma}\backslash\cU(\xi),
\end{equation}
where $S_\sigma$ is the support of the cycle $\sigma$ defined in \eqref{defn:lap-number} and the union is taken over all cycles $\sigma$ of length $k\geq 2$ in the cycle decomposition of $\rmap{\xi}$.

\begin{lem}\label{lem:F_3-minimal}
Consider $\cF_3\colon \cX\mvmap \cX$  with D-grading $\pi\colon \cX \to \SCC(\cF_3)$. 
If $\xi\in\cX$ is a semi-opaque cell, then 
\[
\pi(\alpha) \posetF{3} \pi(\xi) \posetF{3} \pi(\beta),
\]
for any $\alpha\in\cU(\xi)$ and $\beta\in\Stable(\xi)$.
\end{lem}

\begin{proof}
By Definition~\ref{defn:Rule3}, it follows that $\alpha\in\cF_3(\xi)$ for any $\alpha\in\cU(\xi)$. Hence, $\pi(\alpha) \posetF{3} \pi(\xi)$.

We use a proof by induction on the difference of dimension of $\beta$ and $\xi$ to show that $\pi(\xi)\posetF{3} \pi(\beta)$ for any $\beta\in\Stable(\xi)$.
So assume that $\pi(\xi)\posetF{3} \pi(\beta)$ for all $\beta\in\Stable(\xi)$ such that $1\leq \dim\beta - \dim\xi\leq k$.

We begin by considering the base case, $k=1$.
Since $\xi$ is a semi-opaque cell, $\beta \leftrightarrow_{\cF_2} \xi$.
The assumption that $\beta \in \Stable(\xi)$ implies that {\bf Condition 3.1} is satisfied and hence, $\beta\not\in \cF_3(\xi)$ but $\xi\in\cF_3(\beta)$.

To perform the induction step assume $\dim\beta - \dim\xi = k+1$. 
We provide a proof by contradiction and thus assume that $\pi(\beta) <_{\cF_3} \pi(\xi)$. 
By \eqref{eq:stable_cells}, there exists a cycle $\sigma$ of the regulation map $o_\xi$ such that $\Ex(\xi,\beta)\subset S_\sigma$. 
We  use the set
\[
\cB\coloneqq \setdef{\tau\in\cX}{\xi\prec\tau\preceq\beta, \Ex(\xi,\tau)\subset S_\sigma}
\]
to obtain a contradiction with Theorem~\ref{thm:singlevalue} for $\cF_2$. 

We begin by showing that $\pi(\xi)\posetF{3} \pi(\tau)$, for all $\tau\in \cB\setminus\setof{\beta}$.
By assumption $\beta\in\Stable(\xi)$ and hence $\beta\notin \cU(\xi)$.
By Definition~\ref{defn:Rule3.2} there exists $\mu\in\Top_\cX(\beta)$ such that $\lap_{\xi,\sigma}(\mu) \geq \frac{| \sigma |}{2}$. 
Observe that for all $\tau\in \cB$, $\tau\preceq\mu$, and hence $\tau\not\in\cU(\xi)$.
As a consequence $\cB\subset \Stable(\xi)$.
Since, $\dim\beta - \dim\xi = k+1$, if  $\tau\in \cB\setminus\setof{\beta}$, then $\dim\tau - \dim\xi \leq k$.
Hence, by the induction hypothesis, $\pi(\xi)\posetF{3} \pi(\tau)$, for all $\tau\in \cB\setminus\setof{\beta}$.

For $i=2,3$ define
\begin{align*}
    \cF^{\cB}_i \colon \cB & \mvmap \cB \\
    \eta & \mapsto \cF_i(\eta)\cap \cB.
\end{align*}
We claim that 
\begin{equation}\label{eq:F3_F2}
    \cF^{\cB}_{3} = \cF^{\cB}_{2}.
\end{equation}
Observe that by Definition~\ref{defn:Rule3.1} it is sufficient to show that {\bf Condition 3.1} is not satisfied for any pairs of cell $\tau \leq_\cX \tau'\in\cB$.
We are assuming that the cycle $\sigma$ cannot be decomposed. 
By definition if $\tau \in \cB$, then $\xi \prec_\cX \tau$, and hence, by Proposition~\ref{prop:face_cycle} if the regulation map $o_\tau$ has a cycle it must be $S_{\sigma}$.
Let $n\in\Ex(\xi,\tau)\subset S_\sigma$ such that $n\in J_e(\tau)$.
By Definition~\ref{def:active_regulation} $n\in \activeset(\tau)$.
Thus $\sigma$ is not a cycle of $o_\tau$, and hence {\bf Condition 3.1} cannot be applied.

Recall that we are assuming that $\pi(\beta) <_{\bar\cF_3} \pi(\xi)$ and $\pi(\xi)\posetF{3} \pi(\tau)$, for all $\tau\in \cB\setminus \setof{\beta}$.

Thus $\pi(\beta) <_{\bar\cF_3} \pi(\tau)$, for any $\tau\in \cB\setminus \setof{\beta}$.
Let $\tau'\in\cB$ such that $\dim(\tau')-\dim(\xi)=1$.
By hypothesis $\xi$ is a semi-opaque cell, thus $\tau' \leftrightarrow_{\cF_2}\xi$, i.e., $\pi'(\tau') =_{\bar\cF_{2}} \pi'(\xi)$, where $\pi'\colon \cX \to \SCC(\cF_2)$. 

By \eqref{eq:F3_F2} and $\pi(\beta) <_{\bar\cF_3} \pi(\tau')$, it follows that  $\pi'(\beta) <_{\bar\cF^{\cB}_2} \pi'(\tau')$. Moreover, $\pi'(\tau')\leq_{\bar\cF^{\cB}_2} \pi'(\tau)$, for any $\tau\in \cB\setminus \setof{\beta}$. 
Hence,  $\pi'(\beta) \neq \pi'(\tau')$ and 
\[
\pi'(\beta), \pi'(\tau')\in\min\setdef{\pi'_b(\mu')}{\mu'\in \Top_{\cl(\blup(\cB))}(\zeta)},
\]
for any $\zeta\in \cl(\blup(\beta))\cap \cl(\blup(\tau'))$. This contradicts Theorem~\ref{thm:singlevalue} for $\cF^{\cB}_2$, since $\pi_b'$ is not an extendable grading for $\cF^{\cB}_2$, as it does not satisfied Definition~\ref{defn:admissibleGrading}.
\end{proof}

\begin{figure}
    \centering
\begin{tikzpicture}[scale=4]
    \coordinate (A) at (0,0,0);
    \coordinate (B) at (1,0,0);
    \coordinate (C) at (1,1,0);
    \coordinate (D) at (0,1,0);
    \coordinate (E) at (0,0,1);
    \coordinate (F) at (1,0,1);
    \coordinate (G) at (1,1,1);
    \coordinate (H) at (0,1,1);

    \draw[thick, dashed, blue] (A) -- (B);
    \draw[thick] (B) -- (C);
    \draw[thick] (C) -- (D);
    \draw[thick, dashed, blue] (D) -- (A);

    \draw[thick] (E) -- (F);
    \draw[thick, blue] (F) -- (G);
    \draw[thick, blue] (G) -- (H);
    \draw[thick] (H) -- (E);

    \draw[thick, dashed] (A) -- (E);
    \draw[thick, blue] (B) -- (F);
    \draw[thick] (C) -- (G);
    \draw[thick, blue] (D) -- (H);
    
    \node at (A) [below right] {$\mu_{\bone^{(1)}}$};
    \node at (B) [below right] {$\mu_{\bzero^{(3)}}$};
    \node at (C) [above right] {$\mu_{\bzero}$};
    \node at (D) [above left] {$\mu_{\bzero^{(2)}}$};
    \node at (E) [below left] {$\mu_{\bone}$};
    \node at (F) [below right] {$\mu_{\bone^{(2)}}$};
    \node at (G) [below right] {$\mu_{\bzero^{(1)}}$};
    \node at (H) [above left] {$\mu_{\bone^{(3)}}$};

    \draw[->, thick, red] (A) -- ($(A) + (-0.1,0,0)$); 
    \draw[->, thick, red] (A) -- ($(A) + (0,0.1,0)$); 
    \draw[->, thick, red] (A) -- ($(A) + (0,0,-0.2)$); 
    
    \draw[->, thick, red] (B) -- ($(B) + (-0.1,0,0)$); 
    \draw[->, thick, red] (B) -- ($(B) + (0,-0.1,0)$); 
    \draw[->, thick, red] (B) -- ($(B) + (0,0,-0.2)$); 
    
    \draw[->, thick, red] (C) -- ($(C) + (-0.1,0,0)$); 
    \draw[->, thick, red] (C) -- ($(C) + (0,-0.1,0)$); 
    \draw[->, thick, red] (C) -- ($(C) + (0,0,0.2)$); 
    
    \draw[->, thick, red] (D) -- ($(D) + (-0.1,0,0)$); 
    \draw[->, thick, red] (D) -- ($(D) + (0,0.1,0)$); 
    \draw[->, thick, red] (D) -- ($(D) + (0,0,0.2)$); 
    
    \draw[->, thick, red] (E) -- ($(E) + (0.1,0,0)$); 
    \draw[->, thick, red] (E) -- ($(E) + (0,0.1,0)$); 
    \draw[->, thick, red] (E) -- ($(E) + (0,0,-0.2)$); 
    
    \draw[->, thick, red] (F) -- ($(F) + (0.1,0,0)$); 
    \draw[->, thick, red] (F) -- ($(F) + (0,-0.1,0)$); 
    \draw[->, thick, red] (F) -- ($(F) + (0,0,-0.2)$); 
    
    \draw[->, thick, red] (G) -- ($(G) + (0.1,0,0)$); 
    \draw[->, thick, red] (G) -- ($(G) + (0,-0.1,0)$); 
    \draw[->, thick, red] (G) -- ($(G) + (0,0,0.2)$); 
    
    \draw[->, thick, red] (H) -- ($(H) + (0.1,0,0)$); 
    \draw[->, thick, red] (H) -- ($(H) + (0,0.1,0)$); 
    \draw[->, thick, red] (H) -- ($(H) + (0,0,0.2)$); 

    \coordinate (AB_mid) at ($(A)!0.5!(B)$);
    \coordinate (BC_mid) at ($(B)!0.5!(C)$);
    \coordinate (CD_mid) at ($(C)!0.5!(D)$);
    \coordinate (DA_mid) at ($(D)!0.5!(A)$);
    \coordinate (EF_mid) at ($(E)!0.5!(F)$);
    \coordinate (FG_mid) at ($(F)!0.5!(G)$);
    \coordinate (GH_mid) at ($(G)!0.5!(H)$);
    \coordinate (HE_mid) at ($(H)!0.5!(E)$);
    \coordinate (AE_mid) at ($(A)!0.5!(E)$);
    \coordinate (BF_mid) at ($(B)!0.5!(F)$);
    \coordinate (CG_mid) at ($(C)!0.5!(G)$);
    \coordinate (DH_mid) at ($(D)!0.5!(H)$);

    \node at (AB_mid) [below, blue] {$\tau_{31}$};
    \node at (DA_mid) [left, blue] {$\tau_{21}$};
    \node at (FG_mid) [right, blue] {$\tau_{12}$};
    \node at (GH_mid) [above, blue] {$\tau_{13}$};
    \node at (BF_mid) [right, blue] {$\tau_{32}$};
    \node at (DH_mid) [left, blue] {$\tau_{23}$};
\end{tikzpicture}
    \caption{The cube represents the zero dimensional cell $\xi$ in the dual complex of $\cX$. The top dimensional cell of $\xi$, $\Top_\cX(\xi)=\setdef{\mu_\bu\in\cX^{(3)}}{\bu\in\setof{0,1}^3}$, are represented by vertices in the dual complex of $\cX$. The wall labelling of each top cell in $\Top_\cX(\xi)$ are represented by vectors in red. The edges of the cube are the dual of the two dimensional cells and the edges 
    in blue are the dual of two dimensional cells in the unstable cell of $\xi$,  $\cU(\xi)$.}
    \label{fig:wall_labelling_repressilator}
\end{figure}
\begin{lem}\label{lem:F_3-q-same_grading} 
Fix $N=2$ or $3$. Consider $\cF_3\colon \cX\mvmap \cX$  with D-grading $\pi\colon \cX \to \SCC(\cF_3)$.

    Let $\zeta\in\cX_b$ and  $\xi\in\blup^{-1}(\Top_{\cX_b}(\zeta))$ be a semi-opaque cell. Then for all $\tau,\tau'\in\cU(\xi)\cap \blup^{-1}(\Top_{\cX_b}(\zeta))$,
    \[
    \pi(\tau)=\pi(\tau').
    \]
\end{lem}
\begin{proof}
The most direct proof of this lemma involves a cases by cases analysis.
We fully describe the cases for $N=3$ and leave the simpler cases for $N=2$ to the reader.
    
Consider $N=3$ and $\xi\in\cX^{(0)}$ a semi-opaque cell. Recall that Example~\ref{ex:lap_number_2and3D} describes all possible cases, and due to symmetry, they can be reduced to two:
\begin{enumerate}
    \item $n_1=1$, $n_2=2$, $n_3=3$ and  $w(\kappa_{1}^-,\kappa)=w(\kappa_{2}^-,\kappa)=w(\kappa_{3}^-,\kappa)=-1$, where $\kappa_n^-=[\bv, \bone^{(n)}]$ and $\kappa=[\bv, \bone]$;
    \item $n_1=1$, $n_2=2$, $n_3=3$ and $w(\kappa_{1}^-,\kappa)=w(\kappa_{2}^-,\kappa)=w(\kappa_{3}^-,\kappa)=1$.
\end{enumerate}
Notice that $o_\xi=\sigma=(1 \ 2 \ 3)$ is a cycle of length $k=3$ for both cases. We begin with case (1): 
From Definition~\ref{defn:lap-number}, it follows that

\[\lap_{\xi,\sigma}(\kappa)=\#\setdef{i\in\setof{1,\ldots,3}}{w(\kappa_{n_i}^-,\kappa) = -1}=3> \frac{3}{2},\] 
hence $\kappa\notin \cU(\xi)$,
see Figure~\ref{fig:wall_labelling_repressilator}.
     
 Define $\mu_\bu\coloneq [\bv -\bu, \bone]$, where $\bu\in\setof{0,1}^3$. With the notation above,  
 \[\Top_\cX(\xi)=\setdef{\mu_\bu\in\cX^{(3)}}{\bu\in\setof{0,1}^3},\] 
 and $p(\xi,\mu_\bu)=2\bu-\bone$, see Figure~\ref{fig:wall_labelling_repressilator}.

Recall that 
$\lap_{\xi,\sigma}(\mu_\bu)$ (see Definition~\ref{defn:Rule3}) counts how many times $-1$'s are in
     \[
     \setof{-(2\bu_1-1)(2\bu_2-1),-(2\bu_2-1)(2\bu_3-1),-(2\bu_3-1)(2\bu_1-1)}.
     \]
     Hence, 
     $\lap_{\xi,\sigma}(\mu_\bu)=1<\frac{3}{2}$ for \[\bu\in\setof{\bzero^{(1)},\bzero^{(2)},\bzero^{(3)},\bone^{(1)},\bone^{(2)},\bone^{(3)}}\] 
     and 
    $\lap_{\xi,\sigma}(\mu_\bu)=3>\frac{3}{2}$ for $\bu\in\setof{\bzero,\bone}$. As a consequence, the opposite top cells $[\bv-\bone, \bone]$ and $\kappa=[\bv-\bzero,\bone]$ are the only top cells not in $\cU(\xi)$. Hence, $\cU(\xi)$ does not have one dimensional cells and two dimensional cells that are faces of $[\bv-\bone, \bone]$ and $\kappa=[\bv-\bzero,\bone]$. As a consequence, we can obtain the two dimensional cells in $\cU(\xi)$ as follows:
    
    Let $\tau_{ij}\in\cX^{(2)}$ such that $\tau_{ij}\prec \mu_{\bzero^{(i)}}$ and $\tau_{ij}\prec \mu_{\bone^{(j)}}$, where $i=1,2,3$ and 
    $j\in\setof{i\pm 1}$ with $j\leq3$. Thus, 
    \[
    \tau_{ij}=[\bv-\bzero^{(i)}\vee_\Z\bone^{(j)}, \bzero^{(i)}\vee_\Z\bone^{(j)}].
    \]
    
    Since $o_\xi(i)=\sigma(i)\neq i$ for all $i=\setof{1,2,3}$, then
    \[
    \mu_{\bzero^{(i)}}\rightarrow_{\cF_3}\tau_{i\sigma(i)}, \quad
    \mu_{\bone^{(i)}}\rightarrow_{\cF_3}\tau_{\sigma(i)i}, \quad 
    \tau_{\sigma(i)i}\rightarrow_{\cF_3} \mu_{\bzero^{(\sigma(i))}} \quad \text{and}\quad 
    \tau_{i\sigma(i)}\rightarrow_{\cF_3} \mu_{\bone^{(\sigma(i))}},
    \]
    see Figure~\ref{fig:wall_labelling_repressilator}. Hence,
    \begin{align*}
        \mu_{\bzero^{(1)}}\rightarrow_{\cF_3}\tau_{12}\rightarrow_{\cF_3}\mu_{\bone^{(2)}}\rightarrow_{\cF_3}\tau_{32}\rightarrow_{\cF_3}\mu_{\bzero^{(3)}}\rightarrow_{\cF_3}\tau_{31}\rightarrow_{\cF_3}\mu_{\bone^{(1)}}
        \\
        \mu_{\bone^{(1)}}\rightarrow_{\cF_3}\tau_{21}\rightarrow_{\cF_3}\mu_{\bzero^{(2)}}\rightarrow_{\cF_3}\tau_{23}\rightarrow_{\cF_3}\mu_{\bone^{(3)}}\rightarrow_{\cF_3}\tau_{23}\rightarrow_{\cF_3}\mu_{\bzero^{(1)}}
    \end{align*}
    is a closed loop. Therefore $\pi(\alpha)=\pi(\alpha')$ for any $\alpha,\alpha\in\cU(\xi)$.

    For case (2), where $w(\kappa_{1}^-,\kappa)=w(\kappa_{2}^-,\kappa)=w(\kappa_{2}^-,\kappa)=1$, analogous computations show that  $\cU(\xi)=\setof{[\bv-\bone, \bone], \kappa=[\bv-\bzero,\bone]}$. Hence \[\# \left(\cU(\xi)\cap \blup^{-1}(\Top_{\cX_b}(\zeta))\right)\leq 1,\]
    thus the result follows.
\end{proof}

\begin{proof}[Proof of Theorem~\ref{thm:singlevalue} for $i=3$]
Recall that for $i=3$ we only consider $N=2$ or $3$.
Let $\zeta=[\bar\bv,\bar\bw]\in \cX_b$. 
We break the proof into two cases: (i)  $\blup^{-1}(\Top_{\cX_b}(\zeta))$ does not have a semi-opaque cell with a $2$-cycle or a $3$-cycle and (ii) it does.

(i) In this case {\bf Condition 3.1} does not apply and hence $\left.\cF_3\right|_{\blup^{-1}(\Top_{\cX_b}(\zeta))}=\left.\cF_2\right|_{\blup^{-1}(\Top_{\cX_b}(\zeta))}$.
Thus, the proof for $i=2$ applies.

(ii)  Let $\xi\in\blup^{-1}(\Top_{\cX_b}(\zeta))$ with an $\ell$-cycle $\sigma$, where $\ell =2$ or $3$.
We further break the proof into two cases: $\ell=3$ and $\ell=2$.

$(\boldsymbol{\ell = 3})$
We are assuming that $o_\xi=\sigma$ is a $\ell=3$ cycle.
This implies that $S_\sigma = \setof{1,2,3}$ and hence $N=3$.
Thus, by \eqref{eq:stable_cells}, we have that 
$\setdef{\tau\in\cX}{\xi\preceq_\cX \tau}=\cU(\xi)\cup\setof{\xi}\cup\Stable(\xi)$ is a disjoint union.
Hence, $\blup^{-1}(\Top_{\cX_b}(\zeta))$ is a disjoint union of unstable, semi-opaque, and stable cells.
Assume that $\blup^{-1}(\Top_{\cX_b}(\zeta)) \cap \cU(\xi) = \emptyset$, then  Lemma~\ref{lem:F_3-minimal} guarantees that $\pi(\xi)$ is the minimal element in $\pi(\blup^{-1}(\Top_{\cX_b}(\zeta)))$, since there are no unstable cells in $\blup^{-1}(\Top_{\cX_b}(\zeta))$.

Now assume that $\blup^{-1}(\Top_{\cX_b}(\zeta)) \cap \cU(\xi) \neq \emptyset$. 
Choose a cell $\gamma \in \blup^{-1}(\Top_{\cX_b}(\zeta)) \cap \cU(\xi)$. It follows from Lemma~\ref{lem:F_3-minimal} that $\pi(\gamma) \posetF{3} \pi(\eta)$ for any $\eta \in \blup^{-1}(\Top_{\cX_b}(\zeta))\setminus \cU(\xi)$. 
Lemma~\ref{lem:F_3-q-same_grading} implies that all elements of $\cU(\xi)$ have the same grading in $\blup^{-1}(\Top_{\cX_b}(\zeta))$. 
Hence, $\pi(\gamma)$ is the unique minimal element in $\pi(\blup^{-1}(\Top_{\cX_b}(\zeta)))$.

$(\boldsymbol{\ell = 2})$ For this case, observe that there exists $n'\in\setof{1,2,3}$ such that $n'\notin \sigma$. 
Notice that $n'$ is not an indecisive drift direction in $\Top_{\cX_b}(\zeta)$ since $o_{\xi}=\sigma$ and $n'\notin \sigma$, i.e., either no element in $o_{\xi}$ acts on $n'$ or $n'$ acts to itself. 
We can use $n'$ to find a subset of $\blup^{-1}(\Top_{\cX_b}(\zeta))$ that can realize the minimal element of $\blup^{-1}(\Top_{\cX_b}(\zeta))$ as follows.

Assume that $n'\notin J_i(\zeta)$. Then, for any $\eta,\eta'\in\blup^{-1}(\Top_{\cX_b}(\zeta))$ it follows that $\Ex(\eta,\eta')\subset \sigma$. Hence, we can follow the same argument done for $(\boldsymbol{\ell = 3})$. 

Now assume that $n'\in J_i(\zeta)$.
Let $\mu\in\Top_\cX(\maxD(\zeta))$. Choose $\bq'\in\setof{0,1}^N$ such that
\[ \bq'_n =
\begin{cases}
    1, &\text{if $n=n'$ and } \text{$p_n(\minD(\zeta),\maxD(\zeta)) = \rook_n(\minD(\zeta),\mu)$},\\
    0, &\text{otherwise.}
\end{cases}
\]
and define
\[
M\coloneq\setdef{[\bar\bv-\bq, \bone]\in \Top_{\cX_b}(\zeta)}{\bq_{n'}=\bq'_{n'}} \ \text{and} \  M^{\star}\coloneq\setdef{[\bar\bv-\bq, \bone]\in \Top_{\cX_b}(\zeta)}{\bq_{n'} = 1- \bq'_{n'}} 
\]
Clearly $M\cap M^{\star}=\emptyset$ and $\Top_{\cX_b}(\zeta) = M\cup M^\star$. Furthermore, for any $[\bar\bv-\bq, \bone]\in M$ we can define the corresponding $[\bar\bv-\bq^{\star}, \bone]\in M^{\star}$ by setting 
\[
   \bq^{\star}_n = \begin{cases}
       \bq_n, & \text{if}~ n\neq n',\\
       1-\bq_n, & \text{if}~ n = n'.\\
   \end{cases}
\]
By Lemma~\ref{lem:F_3-leq-geq}, $\pi(\blup^{-1}([\bar\bv-\bq, \bone]))\posetF{3} \pi(\blup^{-1}([\bar\bv-\bq^\star, \bone]))$ for any $[\bar\bv-\bq, \bone]\in M$. Hence, the minimal element of $\pi\left(\blup^{-1}(\Top_{\cX_b}(\zeta))\right)$ is in $\pi(\blup^{-1}(M))$. 

Now we can follow an analogous argument done for $(\boldsymbol{\ell = 3})$ by changing $\blup^{-1}(\Top_{\cX_b}(\zeta))$ to $M$. 

\end{proof}

\section{Conley complex}
\label{sec:ConleyComplex}

Weak condensation graphs and D-gradings play a central role in our characterization of the global dynamics by identifying regions potentially associated with recurrent dynamics and providing information about the direction of nonrecurrent dynamics. Rigorous identification of the recurrent and nonrecurrent dynamics for the ODE is done using the Conley index. Thus, we want to perform homological algebraic calculations that preserve this grading. With this in mind, we begin with an abstract discussion.

\begin{defn}
\label{defn:gradedcellcomplex}
Let $\cZ = (\cZ,\preceq_\cZ,\dim,\kappa)$ be a cubical cell complex and let $(\sP,\leq)$ be a finite poset. 
The cubical cell complex $\cZ$ is \emph{$P$-graded} if there exists an order preserving epimorphism $\lambda\colon (\cZ,\preceq_\cZ) \to (\sP,\leq)$.
\end{defn}

Since $\cZ$ forms a basis for the chain complex $C_*(\cZ;\F)$, we have that 
\[
C_*(\cZ;\F) = \bigoplus_{p\in\sP}C_*(\lambda^{-1}(p);\F).
\]

We remark that  $\lambda^{-1}(p)$ need not be a closed subcomplex of $\cZ$.
However, as the following argument makes clear we claim that we can associate a chain complex with $C_*(\lambda^{-1}(p);\F)$. 
To begin the argument we recall the discussion leading to Corollary~\ref{cor:birkhoff}, which in this context gives rise to a lattice monomorphism
\begin{equation}
    \label{eq:latticemono}
    \begin{aligned}
        \sO(\lambda) \colon \sO(\sP) & \to \sO(\cZ) \\
    U & \mapsto \cU :=\lambda^{-1}(U).
    \end{aligned}
\end{equation}
Identifying $\cU\in \sO(\cZ)$ with $C_*(\cU;\F)$ allows us to set
\[
\sO(\lambda)(\sO(p)) = C_*(\lambda^{-1}(\sO(p));\F), \quad p\in \sP.
\]
By Theorem~\ref{thm:birkhoff}, $\sO(p)\in \sJ^\vee(\sO(\sP))$.
Thus, there exists a unique immediate predecessor $\sO(p)^< \in \sO(\sP)$.
We leave it to the reader to check that $\zeta\in \lambda^{-1}(p)$ if and only if $\zeta\in \sO(p)\setminus \sO(p)^<$.
Therefore,
\begin{equation}
    \label{eq:C(p)}
    C_*(\lambda^{-1}(p);\F) \cong C_*(\lambda^{-1}(\sO(p)), \lambda^{-1}(\sO(p)^<);\F).
\end{equation}
\begin{defn}
    \label{defn:PgradedLinMap}
    Let $(\cZ_0,\preceq_0)$ and $(\cZ_1,\preceq_1)$ be $\sP$-graded cubical cell complexes with gradings $\lambda_i\colon \cZ_i\to \sP$ for $i=0,1$.
    A chain map  $\rook\colon C_*(\cZ_0;\F) \to C_*(\cZ_1;\F)$ is \emph{$\sP$-graded} if
    \[
    \rook_k\colon C_k(\lambda_0^{-1}(p);\F)\to \bigoplus_{q\leq p}C_{k}(\lambda_1^{-1}(q);\F)
    \]
    for all $p\in\sP$.    
\end{defn}
The next two definitions allow us to consider classes of $\sP$-graded chain complexes.
\begin{defn}
\label{defn:PgradedHomotopy}
Let $(\cZ_0,\preceq_0)$ and $(\cZ_1,\preceq_1)$ be $\sP$-graded cubical cell complexes with gradings $\lambda_i\colon \cZ_i\to \sP$ for $i=0,1$.
Let $\rook,\psi\colon C_*(\cZ_0;\F) \to C_*(\cZ_1;\F)$ be $\sP$-graded chain maps.
A chain homotopy $D \colon C_*(\cZ_0;\F) \to C_{*+1}(\cZ_1;\F)$ is \emph{$\sP$-graded} if 
\[
D \colon C_*(\lambda_0^{-1}(p);\F) \to \bigoplus_{q\leq p}C_{*+1}(\lambda_1^{-1}(q);\F), \quad \forall p \in \sP.
\]
If a $\sP$-graded chain homotopy exists, then $\rook$ and $\psi$ are \emph{$\sP$-graded homotopic} which is denoted by $\rook\sim_\sP \psi$.
\end{defn}
\begin{defn}
\label{defn:PgradedHomotopic}
Let $(\cZ_0,\preceq_0)$ and $(\cZ_1,\preceq_1)$ be $\sP$-graded cubical cell complexes.
A  $\sP$-graded chain map $\rook\colon C_*(\cZ_0;\F) \to C_*(\cZ_1;\F)$ is a \emph{$\sP$-graded chain equivalence} if there is a $\sP$-graded chain map $\psi\colon C_*(\cZ_1;\F) \to C_*(\cZ_0;\F)$ such that
\[
    \psi\circ\rook \sim_\sP \id\quad\text{and}\quad \rook\circ\psi \sim_\sP \id.
\]
In this case $C_*(\cZ_0;\F)$ and $C_*(\cZ_1;\F)$ are said to be \emph{$\sP$-graded chain equivalent}.
\end{defn}

\begin{defn}
\label{defn:conleycomplex}
Let $C_*(\cZ;\F)$ be a $\sP$-graded complex.
A $\sP$-graded chain complex $C_*(\cV;\F)$, with $\sP$-grading  $\lambda_\cV\colon (\cV,\preceq_{\cV})\to (\sP,\leq)$ and boundary operator
\[
\Delta\colon \bigoplus_{p\in\sP}C_*(\lambda_\cV^{-1}(p);\F) \to \bigoplus_{p\in\sP}C_{*-1}(\lambda_\cV^{-1}(p);\F)
\]
is a \emph{Conley complex} for $C_*(\cZ;\F)$ if
\begin{enumerate}
    \item $C_*(\cV;\F)$ is $\sP$-graded chain equivalent to $C_*(\cZ;\F)$, and 
    \item the boundary operator $\Delta$ is \emph{strictly order preserving}, that is
\begin{equation}
\label{eq:strictorderpreserving}
    \Delta\colon C_*(\lambda_\cV^{-1}(p);\F) \to \bigoplus_{q\prec_\cV p}C_{*-1}(\lambda_\cV^{-1}(q);\F).
\end{equation}
\end{enumerate}
The boundary operator $\Delta$ is called a \emph{connection matrix} for $C_*(\cZ;\F)$.
\end{defn}

Observe that \eqref{eq:strictorderpreserving} implies that 
\[
0 = \Delta_{p,p} \colon C_*(\lambda_\cV^{-1}(p);\F)\to C_{*-1}(\lambda_\cV^{-1}(p);\F)\quad\text{for all $p\in\sP$.}
\]
and therefore, by \eqref{eq:C(p)}, 
\[
C_*(\lambda_\cV^{-1}(p);\F) = H_*(\lambda_\cV^{-1}(p);\F) \cong H_*(\lambda^{-1}(\sO(p)), \lambda^{-1}(\sO(p)^<);\F).
\]

Recall that $I\subset \sP$ is \emph{convex} if given $p,q,r\in \sP$ such that $p,r\in I$ and $p < q < r$, then $q\in \sI$.
We leave it to the reader to check that if $I$ is convex, then $\sO(I)\in\sJ^\vee(\sO(\sP))$.

\begin{defn}
    \label{defn:HomologicalAlgebraConleyIndex}
Let $\cZ$ be a  $\sP$-graded cubical cell complex with grading $\lambda\colon (\cZ,\preceq_\cZ) \to (\sP,\leq)$.
Let $I\subset \sP$ be convex.
We define the \emph{Conley index} of $I$ to be
\[
CH_*(I) := H_*(\lambda^{-1}(\sO(I)), \lambda^{-1}(\sO(I)^<);\F).
\]
In particular, given $p\in\sP$
\[
CH_*(p) \cong H_*(\lambda^{-1}(\sO(p)), \lambda^{-1}(\sO(p)^<);\F).
\]
\end{defn}

\begin{rem}
\label{rem:conleycomplexConleyindex}
Consider a $\sP$-graded chain complex $C_*(\cZ;\F)$.
The fundamental result of \cite{harker:mischaikow:spendlove} is that there is an algorithm that produces an associated Conley complex $C_n(\cV;\F)$, with $\sP$-grading  $\lambda_\cV\colon (\cV,\preceq_{\cV})\to (\sP,\leq)$.
Using the Conley index notation,  we can write
\[
C_n(\cV;\F) \cong \bigoplus_{p\in\sP} CH_n(p;\F).
\]
Therefore, up to isomorphism, the chains of $C_n(\cV;\F)$ are uniquely determined.
However, as  Example~\ref{ex:nonuniquenessDelta} demonstrates the boundary operator, i.e., the connection matrix, need not be uniquely determined.
\end{rem}

The next two examples are meant to provide intuition into Conley complexes.

\begin{ex}
\label{ex:bistableP}
Consider the poset $\sP = \setof{0,0',1}$ with $0< 1$ and $0'<1$.
Figure~\ref{fig:simpleCC}(A) shows a two dimensional cubical complex $\cZ$.  
Let $\lambda\colon \cZ \to \sP$ be the $\sP$-grading of $\cZ$ where the values of the 2-dimensional cells of $\cZ$ are as indicated and the lower dimensional cells assume the minimal value of their top cells, i.e.,
\[
\lambda(\zeta) = \min\setdef{\lambda(\mu)}{\zeta\prec\mu\in \cZ^{(2)}}.
\]
Since $\lambda^{-1}(0)$ and $\lambda^{-1}(0')$ are the closed two cells indicated by $0$ and $0'$ in Figure~\ref{fig:simpleCC}(A),
\[
CH_k(0;\F)\cong CH_k(0';\F)  \cong \begin{cases}
    \F & \text{if $k=0$,}\\
    0 & \text{otherwise.}
\end{cases}
\]
The Conley index of $1$ is 
\[
CH_k(1;\F) \cong H_k(\cZ,\lambda^{-1}(0)\cup \lambda^{-1}(0');\F)  \cong\begin{cases}
    \F & \text{if $k=1$,}\\
    0 & \text{otherwise.}
\end{cases}
\]
The chains for the Conley complex are
\begin{equation}
\label{eq:reducedcomplex1}    
\bigoplus_{p\in \setof{0,0,1}}CH_*(p;\F),
\end{equation}
and thus, the nontrivial part of the connection matrix, 
\[
\Delta_1 \colon CH_1(1;\F) \to CH_0(0;\F)\oplus CH_0(0';\F),
\]
takes the form $\left[\begin{smallmatrix}
    1 \\ 1
\end{smallmatrix}\right]$ as we need to be able to recover the Conley index of $I$ for any convex subset of $\sP$ from the Conley complex.
\end{ex}

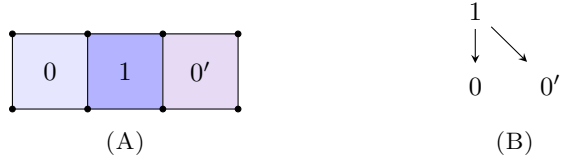
\begin{figure}[h]
    \centering
    \begin{subfigure}{0.4\textwidth}
         \centering
         \begin{tikzpicture}[scale=0.125]
            \fill[red!0!white!90!blue] (0,0) rectangle (8,8);
            \fill[red!0!white!70!blue] (8,0) rectangle (16,8);
            \fill[red!5!white!90!blue] (16,0) rectangle (24,8);
            \draw[step=8cm,black] (0,0) grid (24,8);
            \foreach \i in {0,...,3}{
                \foreach \j in {0,...,1}{
                    \draw[black, fill=black] (8*\i,8*\j) circle (2ex);
                }
            }
            \draw(4,4) node{$0$};
            \draw(12,4) node{$1$};
            \draw(20,4) node{$0'$};
            \end{tikzpicture}
            \caption{}
     \end{subfigure}
     \begin{subfigure}{0.4\textwidth}
        \centering
            \begin{tikzpicture}[scale=0.2,>=stealth,->,shorten      >=2pt,looseness=.5,auto,ampersand replacement=\&]
                \matrix [matrix of math nodes,
                column sep={1cm,between origins},
                row sep={1cm,between origins},
                nodes={rectangle}]
                {    
                |(41)| 1   \&   |(43)|    \\     
                |(51)| 0   \&   |(53)| 0'   \\ 
                };
                \begin{scope}[]
                    \draw (41) -- (51);
                    \draw (41) -- (53);
                \end{scope}
    \end{tikzpicture}
    \caption{}
    \end{subfigure}
    \caption{(A) Two-dimensional cubical complex with top-dimensional cells labeled by a $\sP$-grading given in Example~\ref{ex:bistableP}. (B) Hasse diagram for poset $\sP$.}
    \label{fig:simpleCC}
\end{figure}

\begin{ex}
\label{ex:nonuniquenessDelta}
Consider the two dimensional cubical complex $\cZ$ shown in Figure~\ref{fig:nonunique} and poset $\sP = \setof{0,0',0'',1,2,3,3',4,4'}$ with integer ordering and primes indicating incomparability.
Let $\lambda\colon \cZ \to \sP$ be the $\sP$-grading of $\cZ$ where the values of the 2-dimensional cells of $\cZ$ are as indicated in Figure~\ref{fig:nonunique} and the lower dimensional cells assume the minimal value of their top cells, i.e.,
\[
\lambda(\zeta) = \min\setdef{\lambda(\mu)}{\zeta\prec\mu\in \cZ^{(2)}}.
\]
It is left to the reader to check that
\[
CH_k(0;\F)\cong CH_k(0';\F) \cong CH_k(0'';\F)\cong \begin{cases}
    \F & \text{if $k=0$,}\\
    0 & \text{otherwise;}
\end{cases}
\]
\[
CH_k(1;\F)\cong CH_k(2;\F) \cong  \begin{cases}
    \F & \text{if $k=1$},\\
    0 & \text{otherwise;}
\end{cases}
\]
and
\[
H_*(3;\F)= H_*(3';\F) = H_*(4;\F)= H_*(4';\F) =0.
\]
Therefore, the chains for the Conley complex are
\begin{equation}
\label{eq:reducedcomplex2}    
\bigoplus_{p\in \setof{0,0,0'',1,2}}CH_*(p;\F) = \bigoplus_{p\in \setof{0,0,0'',1,2,3,3',4,4'}}CH_*(p;\F)
\end{equation}
The only nontrivial portion of the connection matrix is $\Delta_1$ which maps
\[
 CH_1(1;\F)\oplus CH_1(2;\F)\to CH_0(0;\F) \oplus CH_0(0';\F) \oplus CH_0(0'';\F). 
\]
It is left to the reader to check that there are two possible connection matrices
\[
\Delta_1 = \begin{pmatrix}
    1 & 1 \\
    1 & 0 \\
    0 & 1 \\
\end{pmatrix}
\quad\text{and}\quad
\Delta_1' = \begin{pmatrix}
    1 & 0 \\
    1 & 1 \\
    0 & 1 \\
\end{pmatrix}.
\]
It is easy to check that with both connection matrices the Conley complex recovers the same homology as that of $C_*(\cZ;\F)$ and that these are the only possible connection matrices.
What is more subtle is to check that both Conley complexes are $\sP$-chain equivalent to $C_*(\cZ;\F)$.
\end{ex}

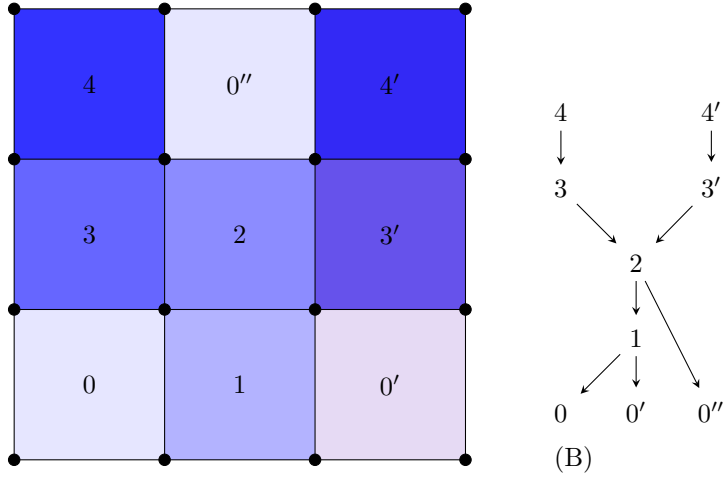
\begin{figure}
\begin{picture}(400,200)

\put(0,0){
\begin{tikzpicture}[scale=0.25]

\fill[red!0!white!20!blue] (0,16) rectangle (8,24);
\fill[red!0!white!90!blue] (8,16) rectangle (16,24);
\fill[red!20!white!20!blue] (16,16) rectangle (24,24);

\fill[red!0!white!40!blue] (0,8) rectangle (8,16);
\fill[red!0!white!55!blue] (8,8) rectangle (16,16);
\fill[red!20!white!40!blue] (16,8) rectangle (24,16);

\fill[red!0!white!90!blue] (0,0) rectangle (8,8);
\fill[red!0!white!70!blue] (8,0) rectangle (16,8);
\fill[red!5!white!90!blue] (16,0) rectangle (24,8);
\draw[step=8cm,black] (0,0) grid (24,24);
\foreach \i in {0,...,3}{
  \foreach \j in {0,...,3}{
    \draw[black, fill=black] (8*\i,8*\j) circle (2ex);
  }
}
\draw(4,4) node{$0$};
\draw(12,4) node{$1$};
\draw(20,4) node{$0'$};
\draw(4,12) node{$3$};
\draw(12,12) node{$2$};
\draw(20,12) node{$3'$};
\draw(4,20) node{$4$};
\draw(12,20) node{$0''$};
\draw(20,20) node{$4'$};

\end{tikzpicture}
}
\put(210,0){(B)}
\put(200,10){
\begin{tikzpicture}[scale=0.2,>=stealth,->,shorten >=2pt,looseness=.5,auto,ampersand replacement=\&]
    \matrix [matrix of math nodes,
           column sep={1cm,between origins},
           row sep={1cm,between origins},
           nodes={rectangle
           }]
  {
|(11)| 4  \& |(13)| \& |(15)| 4'  \& |(17)| \\
|(21)| 3 \&   |(23)|   \& |(25)| 3' \\
|(31)|    \&   |(33)| 2 \& |(35)|    \\    
|(41)|    \&   |(43)| 1   \& |(45)|   \& |(47)|    \\     
|(51)| 0   \&   |(53)| 0' \& |(55)| 0''  \& |(57)|  \\ 
  };
  \begin{scope}[]
    \draw (11) -- (21);
    \draw (15) -- (25);
    \draw (21) -- (33);
    \draw (25) -- (33);
    \draw (33) -- (43);
    \draw (33) -- (55);
    \draw (43) -- (51);
    \draw (43) -- (53);
    \end{scope}
\end{tikzpicture}
}

\end{picture}
\caption{(A) Two-dimensional cubical complex with top dimensional cells labeled by a $\sP$-grading given in Example~\ref{ex:nonuniquenessDelta}.
(B) Hasse diagram for poset $\sP$.
}
\label{fig:nonunique}
\end{figure}

There are two important observations to be made from Example~\ref{ex:nonuniquenessDelta}.
The first, as noted in Remark~\ref{rem:conleycomplexConleyindex}, is that connection matrices need not be unique.
The second, which is demonstrated by \eqref{eq:reducedcomplex2}, is that $CH_*(p;\F)=0$ for many $p\in \sP$.
As a consequence, in many instances  the `size' of the connection matrix is much smaller than the number of elements of $\sP$.
We return to these observations in the next section.

\section{From wall labelings to Morse graphs and Conley complexes}
\label{sec:CCalgorithm}

Let $\cX$ be an abstract cubical complex and let $\omega\colon W(\cX) \to \setof{\pm 1}$ be a fixed strongly dissipative wall labeling. 
Via Chapter~\ref{sec:rules} we construct four possible multivalued maps $\cF_i \colon \cX\mvmap \cX$, $i=0,1,2,3$.
As described in Section~\ref{sec:compInvset+}  we compute the weak condensation graph $\bar{\cF}_i\colon \SCC(\cF_i)\mvmap \SCC(\cF_i)$ and the  D-grading $\pi_b\colon (\cX_b,\preceq_b) \to (\SCC(\cF_i),\preceq_{\cF_i})$.
By Proposition~\ref{prop:gradingorder}, $\pi_b$ is an order preserving epimorphism.
Hence, by Definition~\ref{defn:gradedcellcomplex} $\cX_b$ is a $\SCC(\cF_i)$-graded cubical cell complex.
By Remark~\ref{rem:conleycomplexConleyindex}, we have the existence of a Conley complex that takes the form
\begin{equation}
\label{eq:CComplex_SCC}
\Delta \colon \bigoplus_{p\in\SCC(\cF_i)}CH_*\left(p;\F\right) \to \bigoplus_{p\in\SCC(\cF_i)}CH_{*-1}\left(p;\F\right).
\end{equation}

The goal of the above mentioned constructions is to provide rigorous information about the global structure of the dynamics of ODEs.
To foreshadow how these constructions are employed we note that, though the order relations associated with  the weak condensation graph $\bar{\cF}_i$ and the homological invariants encoded in the Conley complex are intertwined,  the weak condensation graph and the Conley complex also contain complementary information.
We use $\bar{\cF}_i$ to identify regions of phase space where recurrence may occur and to identify the direction of nonrecurrence, and we use the Conley complex to provide guarantees about the existence of recurrent and gradient-like dynamics.
In particular, a nontrivial Conley index implies the existence of recurrent dynamics \cite{conley:cbms}, but a priori we can derive no conclusion from a trivial Conley index.
In this section we suggest how to combine the Conley complex and the condensation graph information.

First note that, as is indicated in Example~\ref{ex:nonuniquenessDelta}, it is possible that $CH_*(p;\F) = 0$, in which case \eqref{eq:CComplex_SCC} can be simplified. 

We denote the set of recurrent components whose Conley index is non-trivial by
\[
\CRC(\cF_i) : = \setdef{p \in \SCC(\cF_i)}{CH_*(p;\F) \neq 0}.
\]
The Conley complex simplifies to
\begin{equation}
\label{eq:CComplex}
\Delta \colon \bigoplus_{p \in \CRC(\cF_i)} CH_*\left(p; \F \right) \to \bigoplus_{p \in \CRC(\cF_i)} CH_{*-1}\left(p; \F \right).
\end{equation}

As indicated above, starting with a wall labeling we  provide an efficient algorithm for identifying the condensation graph $\bar{\cF}_i$ and a Conley complex.
This algorithm is efficient.
In Section~\ref{sec:ramp} we discuss how the DSGRN software translates regulatory networks into large collections of well organized wall labelings. 
Thus we are capable of producing tremendous numbers of condensation graphs and Conley complexes.
The challenge is to have enough a priori intuition about condensation graphs and Conley complexes in order to interpret the computational results in the context of ODEs.

One application of the connection matrix is to identify the existence of heteroclinic orbits for ODEs. In order to describe the results that imply the existence of such orbits assume that $\cF_i\colon \cX\mvmap \cX$ is a combinatorial model for an ODE with flow $\varphi$ and that $\bG$ is an appropriate geometrization of $\cX_b$ (Definition~\ref{defn:rectGeoXb}). In the context of this work, a relevant result (stated below) gives conditions on $p, q \in \SCC(\cF_i)$ and $\Delta_{q,p} \colon CH_*(p;\F) \to CH_*(q;\F)$ under which there exists a connecting orbit for the ODE from $\sInvset(\bg(b(\pi^{-1}(p))),\varphi)$ to $\sInvset(\bg(b(\pi^{-1}(q))),\varphi)$. The conditions of the aforementioned result are given in terms of order-theoretic information on $p, q \in \SCC(\cF_i)$ and algebraic information on $\Delta_{q,p}$. Thus, in principle, better control on the partial order and the connection matrix may lead to stronger results concerning the dynamics of the ODE.
With this in mind we discuss the representation of the partial order obtained from $\SCC(\cF_i)$.

As indicated above, once \eqref{eq:CComplex_SCC} has been computed we simplify the Conley complex taking into account that $CH_*(p;\F) = 0$ if $p \not\in \CRC(\cF_i)$. Ideally we would like to know if $CH_*(p;\F)$ is trivial or not a priori to the computation. In general this is not possible, but a partial result is as follows. 
Define 
\[
\GRC(\cF_i) : = \setdef{p\in\SCC(\cF_i)}{ \bigcap_{\xi\in \pi^{-1}(p)}G(\xi)\neq \emptyset}.
\]
In words, $\GRC(\cF_i)$ consist of the recurrent components $p$ for which there exists a direction $n$ such that $n\in G(\xi)$ for all $\xi\in\pi^{-1}(p)$, or equivalently, all elements share a common gradient direction.

\begin{prop}
    \label{prop:CH=0}
Consider a multivalued map $\cF_i \colon \cX\mvmap \cX$, $i=0,1,2,3$, as in Chapter~\ref{sec:rules}.
If $p\in \GRC(\cF_i)$, then $CH_*(p)=0$.
\end{prop}
\begin{rem}
    \label{rem:CH=0}
In general the converse of Proposition~\ref{prop:CH=0} is not true.
\end{rem}

A direct proof of Proposition~\ref{prop:CH=0}, which we do not present, makes use of the arguments of the proof of Theorem~\ref{thm:singlevalue}.
An indirect proof, which is sufficient for this manuscript, arises from Theorem~\ref{thm:regular_cell}, which guarantees that if $\cF_i\colon \cX\mvmap \cX$ is a combinatorial model for an ODE with flow $\varphi$, $\bG$ is an appropriate geometrization of $\cX_b$, and $p\in \GRC(\cF_i)$, then 
\[
\sInvset(\bg(b(\pi^{-1}(p))),\varphi) = \emptyset.
\]
However,  the Conley index of the empty set is trivial \cite{conley:cbms}, and hence $CH_*(p)=0$.

In summary, from both the order-theoretic and algebraic perspectives, the elements of $\GRC(\cF_i)$ are negligible. This motivates the following definition.

\begin{defn}
\label{defn:MG}
The \emph{Morse graph} for a multivalued map $\cF_i\colon\cX\mvmap \cX$, $i=0,1,2,3$, constructed as in Chapter~\ref{sec:rules} is the poset
\[
\sMG(\cF_i) := \SCC(\cF_i)\setminus \GRC(\cF_i) 
\]
with partial order $\preceq_{\cF_i}$.
\end{defn}

Note that Proposition~\ref{prop:CH=0} implies that $\CRC(\cF_i)$ is a subposet of $\sMG(\cF_i)$. We can now state the following result (we provide a more explicit discussion in Chapter~\ref{sec:examples}).
\begin{thm}
Let $\cF_i \colon \cX \mvmap \cX$ be a combinatorial model for an ODE with flow $\varphi$, and let $\bG$ be a suitable geometrization of $\cX_b$. Consider $p, q \in \CRC(\cF_i)$, and assume that $q$ is an \emph{immediate predecessor} of $p$ in $\sMG(\cF_i)$. That is, assume $q < p$ and that if $r \in \sMG(\cF_i)$ satisfies $q \leq r \leq p$, then $r = q$ or $r = p$.

If the map $\Delta_{q, p} \colon CH_*(p; \F) \to CH_*(q; \F)$ is nonzero, then there exists a connecting orbit for the ODE from $\sInvset(\bg(b(\pi^{-1}(p))), \varphi)$ to $\sInvset(\bg(b(\pi^{-1}(q))), \varphi)$. Specifically, there exists an initial condition $x$ such that
\[
\alpha(x,\varphi)\subset \sInvset(\bg(b(\pi^{-1}(p))),\varphi)\quad\text{and}\quad\omega(x,\varphi)\subset
\sInvset(\bg(b(\pi^{-1}(q))),\varphi).
\]
\end{thm}
To compute a Conley complex we use \cite[Algorithm 6.8]{harker:templates:21} that takes as input a D-graded chain complex.
Applying \cite[Algorithm 6.8]{harker:templates:21} to  $\cX_b$ with $\SCC(\cF_i)$ grading $\pi_b$ we obtain a  Conley complex
\[
\Delta \colon \bigoplus_{p \in \CRC(\cF_i)} CH_*\left(p; \F \right) \to \bigoplus_{p \in \CRC(\cF_i)} CH_{*-1}\left(p; \F \right).
\]

While the Conley complex need not be unique the
chains of the Conley complex are unique up to isomorphism.
This implies that if
\[
\Delta' \colon \bigoplus_{p \in \CRC(\cF_i)} CH_*\left(p; \F \right) \to \bigoplus_{p \in \CRC(\cF_i)} CH_{*-1}\left(p; \F \right)
\]
is another Conley complex, then they are related by an order-preserving chain isomorphism 
\[
\psi\colon \bigoplus_{p \in \CRC(\cF_i)} CH_*\left(p; \F \right) \to \bigoplus_{p \in \CRC(\cF_i)} CH_{*-1}\left(p; \F \right),
\]
i.e., $\Delta' = \psi^{-1} \circ \Delta \circ \psi$.
For the purposes of this monograph, we are working with $\Z_2$ coefficients. 
Thus it is easy to list all possible $\psi$,  compute all possible conjugacies of $\Delta$, and thereby identify all Conley complexes.

\begin{ex}
\label{ex:connection_matrix_running_ex}
Consider the wall labeling in Figure~\ref{fig:wall_labeling}(A). Figure~\ref{fig:example_1_2d_results}(A) shows the edges of the directed graph associated with $\cF_3 \colon \cX\mvmap \cX$ superimposed on $\cX_b$. Marked in color are the 2-dimensional cells associated with $\setdef{\pi_b^{-1}(p)}{p\in \RC(\cF_3)}$. The Morse graph $\sMG(\cF_3)$ is shown in Figure~\ref{fig:example_1_2d_results}(B) with a corresponding color coding. The computations are done using $\F=\Z_2$. The Betti numbers for the Conley indices $CH_*(p;\Z_2)$ are given as vectors $(\beta_0,\beta_1,\beta_2)$.
For each $p\in \sMG(\cF_3)$, $CH_*(p;\Z_2)\neq 0$.

There is a unique Conley complex for $\cF_3$. The chains are 
\[
CH_0(0;\Z_2)\oplus CH_0(0';\Z_2) \oplus CH_1(1;\Z_2)
\]
and the nontrivial portion of the connection matrix is given by 
\[
\Delta_1 =
\begin{bmatrix}
1 \\
1
\end{bmatrix}.
\]
\end{ex}

\begin{figure*}[!htb]
\centering
\includegraphics[width=0.45\textwidth]{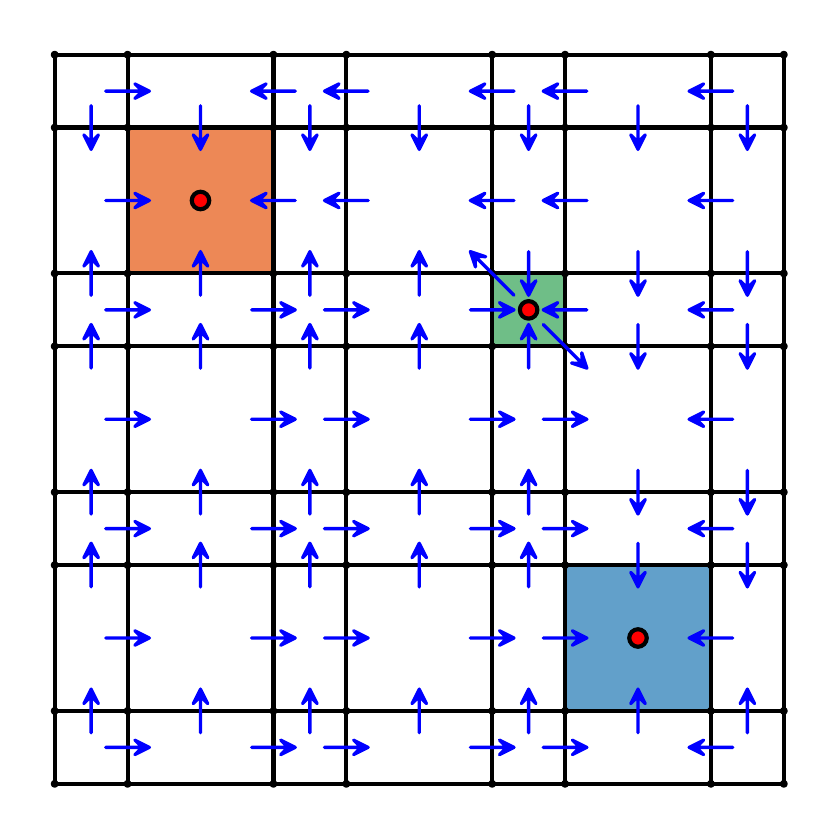}
\includegraphics[width=0.45\textwidth]{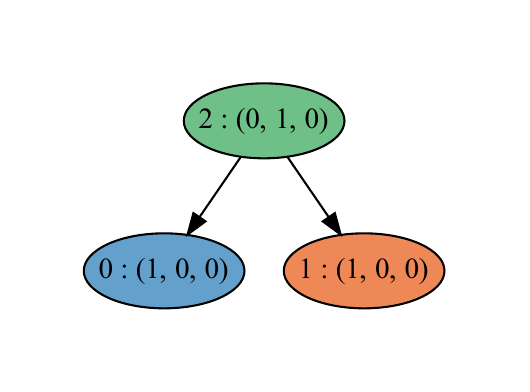}
\caption{Blowup complex $\cX_b$ (left) with the color-coded Morse sets for the Morse graph (right) of $\cF_3$ computed from the wall labeling in Figure~\ref{fig:wall_labeling} (A).}
\label{fig:example_1_2d_results}
\end{figure*}

\begin{ex}
\label{ex:multiple_connection_matrices_ex}
Consider the wall labeling on $\cX(\I)$ for $\I=\setof{0,1,2,3}^2$ shown in Figure~\ref{fig:ex2sec6}(A).
The associated multivalued map $\cF_3\colon \cX\mvmap \cX$  superimposed on $\cX_b$ is shown in Figure~\ref{fig:ex2sec6}(B).
The colored cells correspond to subset of $\cX_b$ that are preimages of $\sMG(\cF_3)$, i.e., $\setdef{\pi_b^{-1}(p)}{p\in \sMG(\cF_3)}$. 

For each $p\in \sMG(\cF_3)$, the Conley index $CH_*(p)$ is nonzero (the associated Hasse diagram is shown in Figure~\ref{fig:ex2sec6}(C)) and the chains for the Conley complex take the form
\[
\bigoplus_{p=0}^8 CH_*(p;\Z_2).
\]
The Conley indices are recorded in the Morse graph and thus the only nontrival portions of the connection matrices are
\[
\Delta_2\colon CH_2(8;\Z_2) \to\bigoplus_{p=4}^7 CH_1(p;\Z_2)
\]
and
\[
\Delta_1\colon \bigoplus_{p=4}^7 CH_1(p;\Z_2) \to \bigoplus_{p=0}^3 CH_0(p;\Z_2).
\]
However, there are two possible connection matrices
\begin{equation}
\label{eq:pairDelta}
\Delta_2 = \begin{bmatrix}
    1\\ {\color{blue}1}\\ 1\\ 1
\end{bmatrix}
\ ,\  
\Delta_1 = \begin{bmatrix}
    1 & 0 & 0 & 1\\ {\color{blue}1} & 1 & 0 & 0\\ {\color{blue}0} & 1 & 1 & 0\\ 0 & 0 & 1 & 1
\end{bmatrix}
\quad\text{and}\quad
\Delta'_2 = \begin{bmatrix}
    1\\ {\color{blue}0}\\ 1\\ 1
\end{bmatrix}
\ , \ 
\Delta'_1 = \begin{bmatrix}
    1 & 0 & 0 & 1\\ {\color{blue}0} & 1 & 0 & 0\\ {\color{blue}1} & 1 & 1 & 0\\ 0 & 0 & 1 & 1
\end{bmatrix}    
\end{equation}
where the distinctions are highlighted in blue. 
This lack of uniqueness is associated with the fact that the poset structure for the Morse graph has element $7$ above element $4$, but both have the same Conley index.
We discuss the implications of this non-uniqueness in the context of ODEs in Section~\ref{sec:examples}.
\end{ex}

\begin{figure}[ht]
    \centering
    \begin{subfigure}[t]{0.32\textwidth}
        \centering
        \includegraphics[scale=0.3]{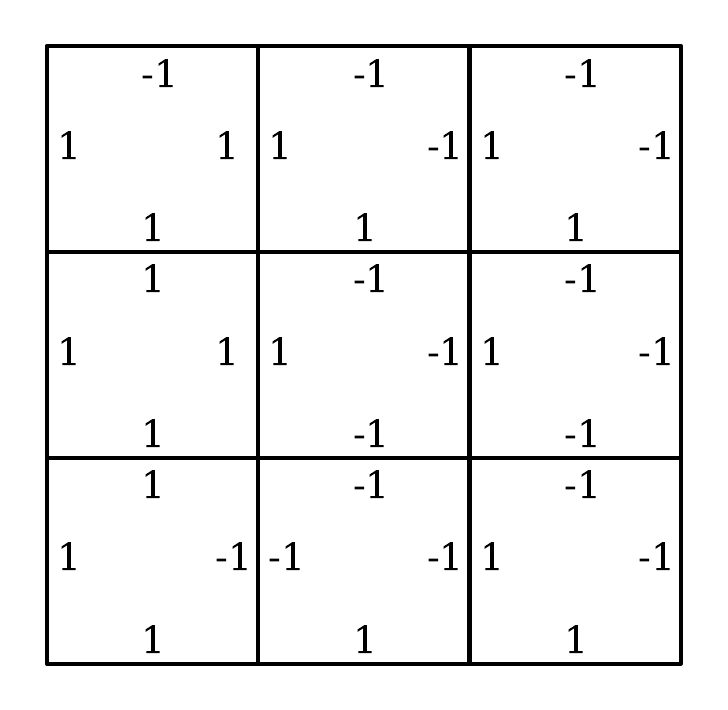}
        \caption{$\omega : W(\cX) \to \setof{\pm 1}^2$}
    \end{subfigure}
    \hfill
    \begin{subfigure}[t]{0.32\textwidth}
        \centering
        \includegraphics[scale=0.25]{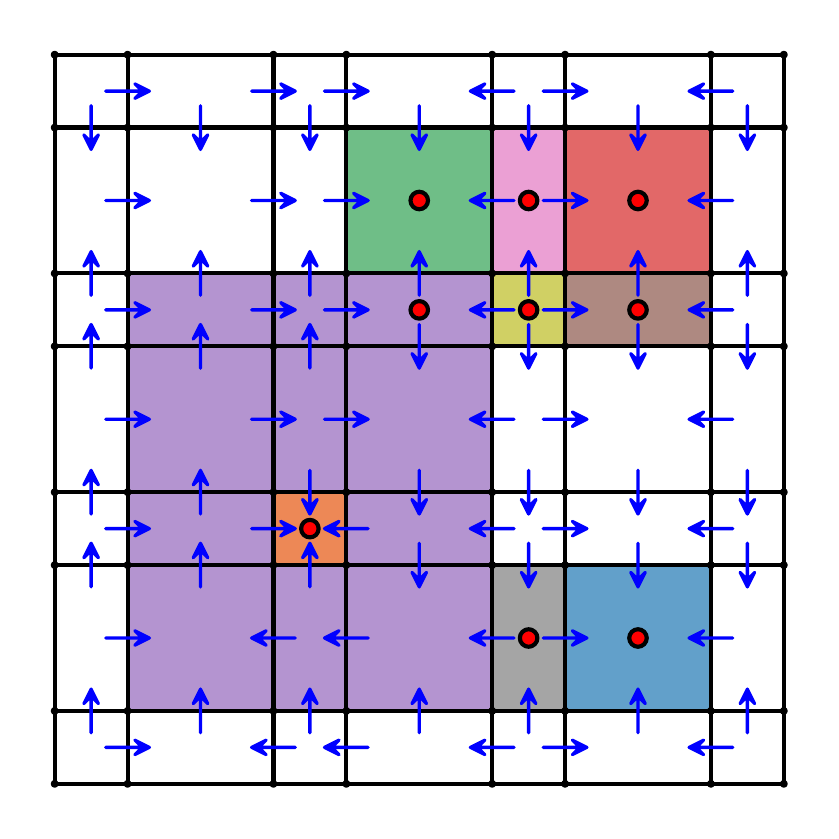}
        \caption{$\cF_3$ in $\cX_b$}
    \end{subfigure}
    \hfill
    \begin{subfigure}[t]{0.32\textwidth}
        \centering
        \includegraphics[scale=0.3]{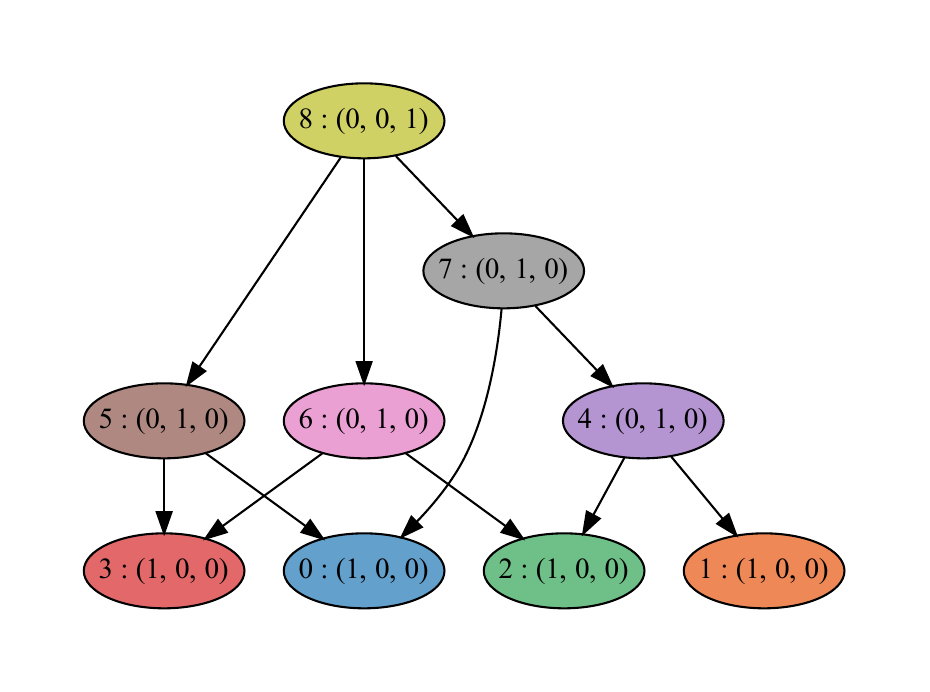}
        \caption{$\sMG(\cF_3)$}
    \end{subfigure}
    \caption{(A) Wall labeling. (B) Directed graph $\cF_3$ imposed on interior of $\cX_b$. Colored cells correspond to elements of $\setdef{\pi_b^{-1}(p)}{p\in \sMG(\cF_3)}$. (C) Hasse diagram of the Morse graph $\sMG(\cF_3)$ with corresponding color coding. The Betti numbers for the Conley index of each node in the Morse graph are denoted by $(\beta_0,\beta_1,\beta_2)$.}
    \label{fig:ex2sec6}
\end{figure}

\part{Continuous Dynamics}
\label{part:III}
\chapter{Ramp systems and DSGRN}
\label{sec:ramp}

We begin the third part of this manuscript; identifying differential equations to which the combinatorial/homological computations are applicable. In Section~\ref{sec:DefinitionRampSystems} we introduce a multiparameter system of ODEs called ramp systems. To tie ramp systems to the previously mentioned computations we need to be able to identify a wall labeling; this is done in Section~\ref{sec:ramp2rook}. 
To emphasize that the constructions of Part~\ref{part:II} can be extended beyond ramp systems in Section~\ref{sec:wall-ramp-property} we show that not every wall labeling can be achieved via a ramp system.
Finally, in Section~\ref{sec:DSGRN} we introduced DSGRN and show how to generate ramp systems from regulatory networks.

\section{Definition of Ramp Systems}
\label{sec:DefinitionRampSystems}

As a first step we introduce ramp functions and ramp nonlinearities.
\begin{defn}
\label{defn:rampfunction}
Let $J$ be a positive integer.
A \emph{ramp function} of type $J$ is a continuous function $r\colon \R\to \R$ of the form
\[
r(x;\nu,\theta,h) = \begin{cases}
\nu_0, & \text{if}~ x \leq \theta_1 - h_1 \\
\nu_j, & \text{if}~ \theta_j + h_j \leq x \leq \theta_{j+1} - h_{j+1} ~\text{for}~ j=1, \ldots, J-1 \\
\nu_J, & \text{if}~ x \geq \theta_J + h_J
\end{cases}
\]
that interpolates linearly. 
We refer to $\theta_j$ as a \emph{threshold} for the variable $x$.
See Figure~\ref{fig:rampfunction} for a ramp function of type 4.
\end{defn}

In this manuscript we restrict our attention to ramp systems that are defined on $[0,\infty)$ and take strictly positive values. Thus, a ramp function $r$ of type $J$ depends upon $3J+1$ parameters $(\nu,\theta,h)\in (0,\infty)^{3J+1}$ that, to avoid degeneracies, must satisfy the following constraints:
\begin{equation}
\label{eq:nu}
\nu \in \Lambda_{\nu} := \setdef{\nu = (\nu_0,\ldots,\nu_J) }{\nu_j\neq \nu_{j+1}} 
\end{equation}
\begin{equation}
\label{eq:theta}
\theta \in \Lambda_{\theta} := \setdef{ \theta = (\theta_1,\ldots,\theta_J)  }{ \ \theta_j<\theta_{j+1} }
\end{equation}
and
\begin{equation}
\label{eq:h}
h \in \Lambda_h := \setdef{ h = (h_1,\ldots,h_J) }{\theta_1 - h_1 > 0, \, \theta_j + \ h_j < \theta_{j+1} - h_{j+1} }.
\end{equation}
We set
\begin{equation}
    \label{eq:LambdaR}
    \Lambda(r) := \setof{(\nu,\theta,h)}\subset (0,\infty)^{3J+1}
\end{equation}
that satisfy \eqref{eq:nu}, \eqref{eq:theta}, and \eqref{eq:h}.

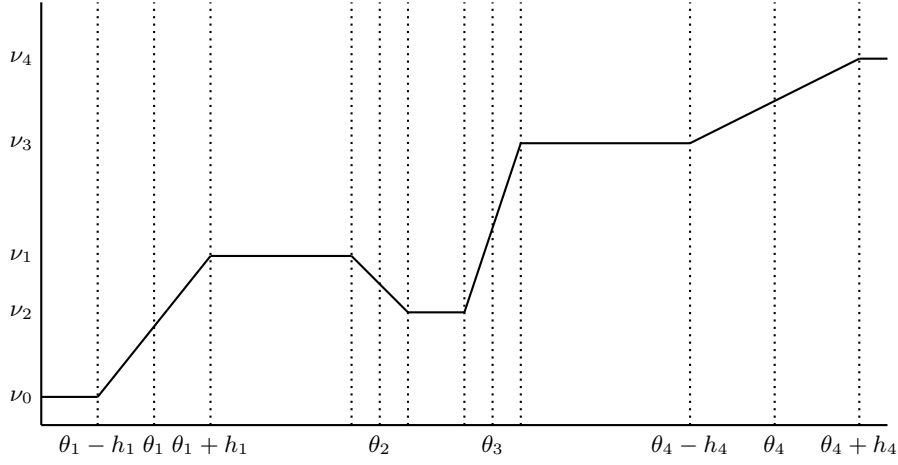
\begin{figure}
\begin{center}
\begin{tikzpicture}
[scale=0.375]

\draw[thick] (0,0) -- (30,0);
\draw[thick] (0,0) -- (0,15);

\draw[dotted, thick] (4,0) -- (4,15);
\draw(4,-0.7) node{\small{$\theta_1$}};
\draw[dotted, thick] (12,0) -- (12,15);
\draw(12,-0.7) node{\small{$\theta_2$}};
\draw[dotted, thick] (16,0) -- (16,15);
\draw(16,-0.7) node{\small{$\theta_3$}};
\draw[dotted, thick] (26,0) -- (26,15);
\draw(26,-0.7) node{\small{$\theta_4$}};
\draw[dotted, thick] (2,0) -- (2,15);

\draw(2,-0.7) node{\small{$\theta_1 - h_1$}};
\draw[dotted, thick] (6,0) -- (6,15);
\draw(6,-0.7) node{\small{$\theta_1 + h_1$}};
\draw[dotted, thick] (11,0) -- (11,15);
\draw[dotted, thick] (13,0) -- (13,15);
\draw[dotted, thick] (15,0) -- (15,15);
\draw[dotted, thick] (17,0) -- (17,15);
\draw[dotted, thick] (23,0) -- (23,15);
\draw(23,-0.7) node{\small{$\theta_4 - h_4$}};
\draw[dotted, thick] (29,0) -- (29,15);
\draw(29,-0.7) node{\small{$\theta_4 + h_4$}};

\draw[thick] (0,1) -- (2,1);
\draw[thick] (2,1) -- (6,6);
\draw[thick] (6,6) -- (11,6);
\draw[thick] (11,6) -- (13,4);
\draw[thick] (13,4) -- (15,4);
\draw[thick] (15,4) -- (17,10);
\draw[thick] (17,10) -- (23,10);
\draw[thick] (23,10) -- (29,13);
\draw[thick] (29,13) -- (30,13);

\draw (-0.7,1) node{\small{$\nu_0$}};
\draw (-0.7,6) node{\small{$\nu_1$}};
\draw (-0.7,4) node{\small{$\nu_2$}};
\draw (-0.7,10) node{\small{$\nu_3$}};
\draw (-0.7,13) node{\small{$\nu_4$}};
\end{tikzpicture}
\end{center}
\caption{Ramp function $r(x)$ of type $J=4$ and parameters $(\nu,\theta,h)$.}
\label{fig:rampfunction}
\end{figure}

The ramp nonlinearities $E_n$ are defined to be sums and/or products of ramp functions that take the form of interaction functions defined below.

\begin{defn}
\label{def:interaction_function}
A polynomial $f$ in $K$ variables $z = (z_1, \ldots, z_K)$ is called a \emph{type I interaction function} of order $K$ if it has the form
\[
f(z) := \prod_{j=1}^q f_j(z)
\]
where each factor has the form 
\[
f_j(z) = \sum_{i\in I_j}z_i
\]
and it is called a \emph{type II interaction function} of order $K$ if it has the form
\[
f(z) := \sum_{j=1}^q f_j(z)
\]
where each term has the form 
\[
f_j(z) = \prod_{i\in I_j}z_i
\]
and the indexing sets $\setof{I_j \mid 1 \leq j\leq q}$ form a partition of $\setof{1, \ldots, K}$.
\end{defn}

\begin{defn}
\label{defn:rampnonlinearity}
A \emph{type I} or \emph{type II ramp nonlinearity} is a continuous function from $[0,\infty)^K$ to $[0,\infty)$ of the form
\[
f(r_1(x_1),\ldots,r_K(x_K)) 
\]
where $r_\ell$ is a ramp function of type $J_\ell$, for $\ell = 1, \ldots, K$, and $f$ is a type I or type II interaction function of order $K$.
\end{defn}

\begin{defn}
\label{defn:ramp_system}
An $N$-dimensional \emph{ramp system} takes the form
\[
\dot{x} = -\Gamma x + E(x; \nu, \theta, h),
\]
where
\begin{equation}
\label{eq:rampODE}
\dot{x}_n  = -\gamma_n x_n + E_n(x), \quad n = 1, \ldots, N
\end{equation}
with $x_n \geq 0$, $\gamma = (\gamma_1,\ldots,\gamma_N)\in (0,\infty)^N$, and  $E_n$ is a type I or type II ramp nonlinearity.
\end{defn}

In general, $E_n$ depends on a subset of the variables $(x_1,\ldots, x_N)$ that we call the \emph{source coordinates} of $E_n$ and denote by $\source(n)$. 
Conversely, we keep track of the ramp nonlinearities impacted by a particular variable and hence define the \emph{target coordinates} of $x_n$ to be
\[
\target(n) := \setdef{ m }{ n \in \source(m) }.
\]

Fix a set of source coordinates $\source(n) = \setof{ m_\ell \mid \ell = 1,\ldots, K(n) }$ and an interaction function $f_n$, the subscript notation for the associated ramp functions takes the form
\[
E_n \left( x_{m_1},\ldots,x_{m_{K(n)}} \right) = f_n\left(r_{n,m_1}(x_{m_1}),\ldots ,r_{n,m_{K(n)}}(x_{m_{K(n)}})  \right).
\]

\begin{ex}
Examples of ramp nonlinearities using the above notation are
\begin{equation*}
E_n(x_1,x_3) = r_{n,1}(x_1) + r_{n,3}(x_3), \quad E_n(x_1,x_3) = r_{n,1}(x_1) r_{n,3}(x_3),
\end{equation*}
and
\[
\quad E_n(x_1,x_2,x_3) = r_{n,1}(x_1) \left( r_{n,2}(x_2) + r_{n,3}(x_3) \right)
\]
with source coordinates $\setof{1,3}$, $\setof{1,3}$, and $\setof{1,2,3}$, respectively.
\end{ex}

Given a ramp function $r_{n,m}$ of type $J_{n,m}$ the associated parameters  are denoted by 
\[
(\nu_{n,m},\theta_{n,m},h_{n,m}) \in \Lambda(r_{n,m})\subset [0,\infty)^{3J_{n,m}+1}
\]
where
\[
\nu_{n,m} = (\nu_{n,m,0},\ldots,\nu_{n,m,J_{n,m}}), \quad
\theta_{n,m} = (\theta_{n,m,1},\ldots,\theta_{n,m,J_{n,m}})
\]
and
\[
h_{n,m} = (h_{n,m,1},\ldots,h_{n,m,J_{n,m}}). 
\]
Using the notation of \eqref{eq:LambdaR}, the parameter space of $E_n$ is
\[
\Lambda(E_n)  = \prod_{m_{\ell} \in \source(n)} \Lambda(r_{n,m_{\ell}}).
\]
The set of parameters associated with an $N$-dimensional ramp system is $(\gamma, \nu, \theta, h) \in (0,\infty)^N \times \Lambda(E)$, where
\[
\Lambda(E) = \prod_{n=1}^N \Lambda(E_n).
\]

To conclude this section we turn to the dynamics of ramp systems.
In particular, we define an attracting block and prove that it contains the global attractor of the associated ramp system.

Fix a parameter $(\gamma, \nu, \theta, h) \in (0,\infty)^N \times \Lambda(E)$ and recall that, for each $n = 1, \ldots, N$, the ramp nonlinearity $E_n$ is a function $E_n \colon [0,\infty)^{K(n)} \to [0,\infty)$ of the form
\[
E_n \left( x_{m_1},\ldots,x_{m_{K(n)}} \right) = f_n\left(r_{n,m_1}(x_{m_1}),\ldots ,r_{n,m_{K(n)}}(x_{m_{K(n)}})  \right),
\]
where $f_n$ is an interaction function, $r_{n, m_{\ell}} \colon [0, \infty) \to [0, \infty)$ is a ramp function of type $J_{n, m_{\ell}}$, and $\source(n) = \setof{ m_\ell \mid \ell = 1,\ldots, K(n) }$.
Notice that $r_{n, m_{\ell}}$ is bounded by
\[
M^r_{n, m_{\ell}} = M^r_{n, m_{\ell}}(\nu) := \max_{x \in [0, \infty)}{ r_{n, m_{\ell}}(x) } = \max{ \setof{ \nu_{n, m_{\ell}, 0}, \ldots, \nu_{n, m_{\ell}, J_{n, m_{\ell}}} } },
\]
which implies that $E_n$ is bounded by
\[
M^E_n = M^E_n (\nu) := \max_{x \in [0,\infty)^{K(n)}}{E_n(x)} = f_n \left( M^r_{n, m_1}, \ldots, M^r_{n, m_{K(n)}}  \right).
\]
Define
\[
Gb_n = Gb_n (\gamma,\nu) := \frac{M^E_n (\nu)}{\gamma_n}.
\]
Observe that
\begin{equation}
\label{eq:dotx<0}
\text{if $x_n > Gb_n (\gamma,\nu)$, then $\dot{x}_n < 0$.}    
\end{equation}

We define $B\subset [0,\infty)^N$  such that the ramp nonlinearities $E_n$ are constant on $[0,\infty)^N\setminus B$.
For this purpose let $\widetilde{m}_n$ and $\widetilde{j}_n$ be such that
\[
\theta_{\widetilde{m}_n, n, \widetilde{j}_n} = \max{ \setof{ \theta_{m,n,j} \mid m \in \target(n), \, j \in J_{m,n} } }
\]
and define
\begin{equation}
\label{eq:GAB}
\gab_n = \gab_n(\gamma, \nu, \theta, h) := \max{ \setof{Gb_n (\gamma,\nu), \, \theta_{\widetilde{m}_n, n, \widetilde{j}_n} + h_{\widetilde{m}_n, n, \widetilde{j}_n}} } + 1
\end{equation}
where the $1$ is added in order to have $\gab_n(\gamma, \nu, \theta, h) > Gb_n (\gamma,\nu)$ and $\gab_n(\gamma, \nu, \theta, h) > \theta_{\widetilde{m}_n, n, \widetilde{j}_n} + h_{\widetilde{m}_n, n, \widetilde{j}_n}$. 
Define
\[
B := \prod_{n = 1}^{N}{ [0, \gab_n(\gamma, \nu, \theta, h)] }
\]
We have the following results.

\begin{lem}
\label{lem:global-bound}
Consider a ramp system $\dot{x} = -\Gamma x + E(x; \nu, \theta, h)$ for a fixed parameter $(\gamma, \nu, \theta, h) \in (0,\infty)^N \times \Lambda(E)$. 
If $x_n \geq \gab_n(\gamma, \nu, \theta, h)$, then $\dot{x}_n < 0$. Furthermore, if $x \in [0,\infty)^N \setminus B$, then $E(x; \nu, \theta, h)$ is constant.
\end{lem}

\begin{proof}
If $x_n \geq \gab_n(\gamma, \nu, \theta, h)$, then 
\[
x_n > Gb_n (\gamma,\nu) = \frac{ \max_{x \in [0,\infty)^{K(n)}}{E_n(x)} }{\gamma_n},
\]
which, by \eqref{eq:dotx<0}, implies that $\dot{x}_n < 0$. 
Additionally, if $x \in [0,\infty)^N \setminus B$, then  $x_n > \theta_{\widetilde{m}_n, n, \widetilde{j}_n} + h_{\widetilde{m}_n, n, \widetilde{j}_n}$ for each $n = 1, \ldots, N$, which implies that $E_n$ is constant. Therefore, $E$ is constant on $[0,\infty)^N \setminus B$.
\end{proof}

\begin{prop}
\label{prop:globalAttractor}
A ramp system $\dot{x} = -\Gamma x + E(x; \nu, \theta, h)$ where $(\gamma, \nu, \theta, h) \in (0,\infty)^N \times \Lambda(E)$ has a compact global attractor $K \subset B$. Furthermore, $B$ is an attracting block and $\omega(B) = K$.
\end{prop}

\begin{proof}
To prove the existence of a global attractor it is sufficient to prove that the flow $\varphi$ for the ramp system is point dissipative \cite{raugel}, i.e., that  for every $x\in [0,\infty)^N$ there exists a time $\tau(x)$ such that $\varphi([\tau(x),\infty),x) \subset B$.

By \eqref{eq:dotx<0}, if $x_n \geq \gab_n(\gamma, \nu, \theta, h)$ then $\dot{x}_n < 0$.
Thus, if $x\in \cl([0,\infty)^N \setminus B)$, then $\dot{x}_n <0$ for all $n=1,\ldots, N$, and hence, there exists $\tau(x) > 0$ such that $\varphi(\tau(x),x) \in B$.

We now show that $B$ is an attracting block. 
If $x \in \bdy(B)$, then there exists $n\in\setof{1, \ldots, N}$ such that $x_n=0$ or $x_n = \gab(\gamma, \nu, \theta, h)$.
Furthermore, if $x_n = 0$, then $\dot{x}_n > 0$ and if $x_n = \gab(\gamma, \nu, \theta, h)$, then $\dot{x}_n < 0$.
Thus, the vector field is transverse to the boundary of $B$ and points towards the interior of $B$. 
Therefore, $B$ is an attracting block.

By the definition of an attracting block, if $x \in B$, then $\varphi([0, \infty),x) \subset B$, thus every $x \in [0,\infty)^N$ is point dissipative.
\end{proof}

\section{From Ramp Systems to Wall Labeling}
\label{sec:ramp2rook}

The goal of this section is as follows. 
Given a ramp system define an associated cubical complex $\cX$ and  wall labeling $\omega \colon W(\cX) \to \setof{\pm 1}$.
A necessary first step is to identify the set of parameters $(\gamma, \nu, \theta, h) \in (0,\infty)^N \times \Lambda(E)$ for which a wall labeling can be constructed.
As a consequence of Proposition~\ref{prop:wall-labeling-const} the combinatorial/algebraic topological computations described in Section~\ref{sec:lattice} are dependent on the parameters $(\gamma, \nu, \theta)$.
However, the applicability of these computations to the ramp system of interest  depends on $h$.
With this in mind we provide explicit non-empty sets $\cH_i(\gamma, \nu, \theta)$ of parameters $h$ for which, in later sections, we validate Theorem~\ref{thm:dynamics}.

Throughout this section we consider an $N$-dimensional ramp system $\dot{x} = -\Gamma x + E(x;  \nu, \theta, h)$ defined on the parameter space $(0,\infty)^N \times \Lambda(E)$. Given a parameter $(\gamma, \nu, \theta, h) \in (0,\infty)^N \times \Lambda(E)$ we denote the set of all indices $(m, j)$ of $\theta_{m, n, j}$ corresponding to the ramp functions $r_{m,n}(x_n)$ for $m \in \target(n)$ by
\[
\mathcal{J}(n) = \setof{ (m, j) \mid m \in \target(n), \, j \in J_{m,n} }.
\]

We refer to $\theta_{m, n, j}$ as an \emph{output threshold} for the variable $x_n$ and as an \emph{input threshold} for the variable $x_m$. Denote by
\[
\Theta(n) = \setof{ \theta_{m,n,j} \mid (m, j) \in \mathcal{J}(n) }
\]
the set of the output thresholds for the variable $x_n$.
\begin{rem}
\label{rem:K(n)}
The number of indices in $\mathcal{J}(n)$, which is equal the number of output thresholds for the variable $x_n$, is independent of the parameters of the ramp system and is given by
\begin{equation}
\label{eq:num_out_thresholds}
K(n) := \sum_{m \in \target(n)} {J_{m, n}},
\end{equation}
where $J_{m, n}$ is the type of the ramp function $r_{m, n}$.    
\end{rem}
\begin{defn}
\label{defn:rampcellcomplex}
Let $\dot{x} = -\Gamma x + E(x; \nu, \theta, h)$ be an $N$-dimensional ramp system with output thresholds $\Theta(n)$, $n=1,\ldots,N$.
The \emph{ramp induced} cubical cell complex (see Definition~\ref{defn:Xcomplex}) is 
\[
\cX := \cX(\I) = (\cX, \preceq, \dim, \kappa)
\]
where
\[
\I = \prod_{n=1}^N \setof{0, \ldots, K(n)+1}
\]
and  $K(n)$ is as in \eqref{eq:num_out_thresholds}.
\end{defn}

\begin{defn}
\label{defn:admissiblehtheta}
Let $(\nu, \theta, h) \in \Lambda(E)$. We say that the parameter $(\theta,h)$ is \emph{admissible} if the following two conditions are satisfied for each $n = 1, \ldots, N$.
\begin{enumerate}
\item[(i)] If $(m, j), (m', j') \in \mathcal{J}(n)$ and $(m, j) \neq (m', j')$, then
\[
\theta_{m, n, j} \neq \theta_{m', n, j'}.
\]
\item[(ii)] Let $K(n)$ be the number of indices in $\mathcal{J}(n)$ and consider an indexing
\begin{equation}
    \label{eq:Jn}
    \mathcal{J}(n) = \setof{ (m_k, j_k) \mid k = 1, \ldots, K(n) }
\end{equation}
such that
\[
\theta_{m_k, n, j_k} < \theta_{m_{k+1}, n, j_{k+1}}, \quad \text{for}~ k = 1, \ldots, K(n) - 1.
\]
Then,
\begin{equation}
\label{eq:theta_h_inequalities}
\theta_{m_k, n, j_k} + h_{m_k, n, j_k} < \theta_{m_{k+1}, n, j_{k+1}} - h_{m_{k+1}, n, j_{k+1}},
\end{equation}
for all $k = 1, \ldots, K(n) - 1$.
\end{enumerate}
\end{defn}

Given the indexing of \eqref{eq:Jn} let $\bv = (k_1, k_2, \ldots, k_N) \in \prod_{n=1}^N \setof{0, \ldots, K(n)}$ and define
\[
D_{\bv}(\theta, h) := \prod_{n=1}^N (\theta_{m_{k_n}, n, j_{k_n}} + h_{m_{k_n}, n, j_{k_n}}, \theta_{m_{k_n + 1}, n, j_{k_n + 1}} - h_{m_{k_n + 1}, n, j_{k_n + 1}}) \subset (0,\infty)^N,
\]
where we set
$\theta_{m_0, n, j_0} = h_{m_0, n, j_0} = 0$,
$\theta_{m_{K(n)+1}, n, j_{K(n)+1}} = \gab_n(\gamma, \nu, \theta, h)$, and $h_{m_{K(n)+1}, n, j_{K(n)+1}} = 0$.
When the parameters are fixed we  denote $D_{\bv}(\theta, h)$ simply by $D_{\bv}$.

\begin{lem}
\label{lem:well_defined_DV}
Consider an $N$-dimensional ramp system $\dot{x} = -\Gamma x + E(x;  \nu, \theta, h)$ where $(\nu, \theta, h) \in \Lambda(E)$ and $(\theta, h)$ is admissible.
By \eqref{eq:num_out_thresholds},  this gives rise to $K(n)$, $n=1,\ldots,N$.
If $\bv \in \prod_{n=1}^N \setof{0, \ldots, K(n)}$, then the set $D_{\bv}(\theta, h)$ is well defined and nonempty. 
Furthermore $E_n$ is constant with respect to $x$, $\theta$, and $h$ on $D_{\bv}(\theta, h)$ for all $n = 1, \ldots, N$.
\end{lem}
\begin{proof}
Given $\bv = (k_1, k_2, \ldots, k_N) \in \prod_{n=1}^N \setof{0, \ldots, K(n)}$, by the definition of $K(n)$ and the indexing of $\mathcal{J}(n)$ introduced in Definition~\ref{defn:admissiblehtheta} it follows that $(m_k, j_k) \in \mathcal{J}(n)$ and hence, by \eqref{eq:theta_h_inequalities}, that
\[
\theta_{m_k, n, j_k} + h_{m_k, n, j_k} < \theta_{m_{k+1}, n, j_{k+1}} - h_{m_{k+1}, n, j_{k+1}}.
\]
Therefore $D_{\bv}(\theta, h)$ is well defined and nonempty. Additionally, notice that $D_{\bv}(\theta, h)$ are the rectangles in the domain of the ramp system where all the ramp functions are constant. Hence $E_n$ depends only on $\nu$ on $D_{\bv}(\theta, h)$.
\end{proof}

In a slight abuse of notation we denote the value of $E_n$ on $D_{\bv}$ by $E_n(D_{\bv})$.

\begin{defn}
\label{defn:admissible-parameters}
We say that $(\gamma, \nu, \theta, h) \in (0,\infty)^N \times \Lambda(\bE)$ is \emph{admissible} if $(\theta,h)$ is admissible, and for all $\bv \in \prod_{n=1}^N \setof{0, \ldots, K(n)}$ and all $n=1, \ldots, N$,
\[
\gamma_n \theta_{m_{k_n}, n, j_{k_n}} \neq E_n(D_{\bv}) \quad \text{and} \quad \gamma_n \theta_{m_{k_n + 1}, n, j_{k_n + 1}} \neq E_n(D_{\bv}).
\]

We denote the set of admissible parameters by 
\[
\Lambda(R) \subset (0,\infty)^N \times \Lambda(\bE).
\]
\end{defn}

We turn to the definition of the wall labeling for which we find the following notation useful.
Given a ramp induced cubical cell complex $\cX$, consider $\mu = [\bv, \bone] \in \cX^{(N)}$.
For clarity (we want to focus on the cell complex as opposed to the phase space) in what follows we set
\[
E_n(\mu) := E_n(D_{\bv}).
\]

\begin{defn}
\label{defn:ramp-wall-labeling}
Let $\dot{x} = -\Gamma x + E(x;  \nu, \theta, h)$ be a ramp system with a fixed parameter $(\gamma, \nu, \theta, h) \in \Lambda(R)$.
Let $\cX$ be the associated ramp induced cell complex.
The \emph{ramp induced} wall labeling $\omega \colon W(\cX) \to \setof{\pm 1}$ is defined as follows.
Given an $n$-wall $(\xi, \mu) \in W(\cX)$, 
with $\xi = \left[\bv, {\bf 1}^{(n)}\right]$ and $\bv = (k_1, k_2, \ldots, k_N)$, define
\[
\omega(\xi, \mu) =
\sgn( -\gamma_n \theta_{m_{k_n}, n, j_{k_n}} + E_n(\mu) ),
\]
where $\theta_{m_0, n, j_0} = 0$ and $\theta_{m_{K(n)+1}, n, j_{K(n)+1}} = \gab_n(\gamma, \nu, \theta, h)$.
\end{defn}

The following theorem proves that a ramp induced wall labeling is a wall labeling.

\begin{thm}
\label{thm:ramp-wall-labeling}
Let $\dot{x} = -\Gamma x + E(x;  \nu, \theta, h)$ be a ramp system with a fixed parameter $(\gamma, \nu, \theta, h) \in \Lambda(R)$.
The ramp induced wall labeling $\omega \colon W(\cX) \to \setof{\pm 1}$ is a wall labeling.
\end{thm}

\begin{proof}
Let $\dot{x} = -\Gamma x + E(x;  \nu, \theta, h)$ be a ramp system with a fixed parameter $(\gamma, \nu, \theta, h) \in \Lambda(R)$ and let $\omega \colon W(\cX) \to \setof{\pm 1}$ as in Definition~\ref{defn:ramp-wall-labeling} above.
To verify that $\omega$ is a wall labeling we need to show that for each $\sigma \in \cX^{(0)}$ there exists a local inducement map $\tilde{o}_{\sigma}$ (see Definition~\ref{def:wall_labeling}). 

Let $\sigma = [\bv, \bzero] \in \cX^{(0)}$, with $\bv = (k_1, k_2, \ldots, k_N)$. Using the notation in Definition~\ref{defn:ramp-wall-labeling}, define $\tilde{o}_{\sigma} \colon \setof{1, \ldots, N} \to \setof{1, \ldots, N}$ by
\[
\tilde{o}_{\sigma}(n) = m_{k_n}.
\]

To verify Definition~\ref{def:wall_labeling}(i) 
let $\mu, \mu' \in T(\sigma)$ be $n$-adjacent with $\mu = [\bv, \bone]$ and $\mu' = [\bv - \bzero^{(n)}, \bone] = [\bv', \bone]$. First notice that $E_{\ell} (\mu) = E_{\ell} (\mu')$ if $\ell \neq m_{k_n}$, since $m_{k_n}$ is the index of the only component of $E$ which may change when $x_n$ crosses the value $\theta_{m_{k_n}, n, j_{k_n}}$ corresponding to the shared $n$-wall of $\mu$ and $\mu'$. Hence we have that
\[
E_{\ell} (\mu) = E_{\ell} (\mu'), \quad \text{if}~ \ell \neq \tilde{o}_\sigma(n).
\]

Let $(\mu^-_{\ell}, \mu)$ and $(\mu'^-_{\ell}, \mu')$ be the left $\ell$-walls of $\mu$ and $\mu'$, respectively. 
Then we have $\mu^-_{\ell} = [\bv, \bone^{(\ell)}]$ and $\mu'^-_{\ell} = [\bv - \bzero^{(n)}, \bone^{(\ell)}] = [\bv', \bone^{(\ell)}]$. Denoting $\bv = (k_1, k_2, \ldots, k_N)$ and $\bv' = (k_1', k_2', \ldots, k_N')$ we have that
\[
\omega(\mu^-_{\ell}, \mu) = \sgn( -\gamma_{\ell} \theta_{m_{k_{\ell}}, n, j_{k_{\ell}}} + E_{\ell}(\mu) )
\]
and
\[
\omega(\mu'^-_{\ell}, \mu') = \sgn( -\gamma_{\ell} \theta_{m_{k'_{\ell}}, n, j_{k'_{\ell}}} + E_{\ell}(\mu') )
\]
Since $\bv' = \bv - \bzero^{(n)}$, if $\ell \neq n$ we have that $k_{\ell} = k'_{\ell}$, and hence that $\theta_{m_{k_{\ell}}, n, j_{k_{\ell}}} = \theta_{m_{k'_{\ell}}, n, j_{k'_{\ell}}}$. Therefore, if $\ell \neq \tilde{o}_\sigma(n)$ and $\ell \neq n$, it follows that $E_{\ell} (\mu) = E_{\ell} (\mu')$ and $\gamma_{\ell} \theta_{m_{k_{\ell}}, n, j_{k_{\ell}}} = \gamma_{\ell} \theta_{m_{k'_{\ell}}, n, j_{k'_{\ell}}}$, and hence that
\[
\omega(\mu^-_{\ell}, \mu) = \omega(\mu'^-_{\ell}, \mu').
\]

Analogously, if we let $(\mu^+_{\ell}, \mu)$ and $(\mu'^+_{\ell}, \mu')$ be the right $\ell$-walls of $\mu$ and $\mu'$, respectively, then a similar argument shows that
\[
\omega(\mu^+_{\ell}, \mu) = \omega(\mu'^+_{\ell}, \mu')
\]
if $\ell \neq \tilde{o}_\sigma(n)$ and $\ell \neq n$.

To prove Definition~\ref{def:wall_labeling}(ii) 
let $(\xi_n,\mu), (\xi_n, \mu') \in W(\sigma)$ be $n$-walls with $\mu = [\bv, \bone]$ and $\mu' = [\bv - \bzero^{(n)}, \bone] = [\bv', \bone]$, and assume that $n \neq \tilde{o}_\sigma(n)$. As shown above, since $n \neq \tilde{o}_\sigma(n)$, we have that
\[
E_n (\mu) = E_n (\mu').
\]
Since $(\xi_n, \mu)$ and $(\xi_n, \mu')$ represent the $n$-wall shared by $\mu$ and $\mu'$ it follows that
\[
\omega(\xi_n, \mu) = \sgn( -\gamma_n \theta_{m_{k_n}, n, j_{k_n}} + E_n(\mu) )
\]
and
\[
\omega(\xi_n, \mu') = \sgn( -\gamma_n \theta_{m_{k_n}, n, j_{k_n}} + E_n(\mu') ).
\]
Since both signs are evaluated at the same $\theta_{m_{k_n}, n, j_{k_n}}$ we conclude that $\omega(\xi_n, \mu) = \omega(\xi_n, \mu')$ as desired.
\end{proof}

\begin{prop}
\label{prop:ramp-dissipativewall}    
Let $\dot{x} = -\Gamma x + E(x;  \nu, \theta, h)$ be a ramp system with a fixed parameter $(\gamma, \nu, \theta, h) \in \Lambda(R)$.
The ramp induced wall labeling $\omega \colon W(\cX) \to \setof{\pm 1}$ is strongly dissipative.
\end{prop}

\begin{proof}
Let $(\xi, \mu) \in W(\cX)$ be such that $\xi \in \left( \bbdy(\cX) \right)^{(N-1)}$ with $\xi = [\bv, {\bf 1}^{(n)}]$ and $\bv = (k_1, k_2, \ldots, k_N)$. Since $\xi$ is a boundary $n$-wall of $\mu$ it follows that $k_n = 0$ or $k_n = K(n) + 1$, if $\xi$ is a left or a right wall, respectively. From the definition of $\omega$ it follows that
\[
\omega(\xi, \mu) = \sgn( -\gamma_n 0 + E_n(\mu) ) = \sgn( E_n(\mu) ) = 1 = -p(\xi,\mu)
\]
or
\[
\omega(\xi, \mu) = \sgn( -\gamma_n \gab_n(\gamma, \nu, \theta, h) + E_n(\mu) ) = -1 = -p(\xi,\mu),
\]
if $k_n = 0$ or $k_n = K(n) + 1$, respectively. Therefore (see Definition~\ref{defn:dissipativewall}) we have that $\omega$ is strongly dissipative.
\end{proof}

As indicated above, there is a dichotomy with respect to parameter dependence; the combinatorial/algebraic topological computations are determined by the parameters $(\gamma, \nu, \theta)$, while the validity of these computations with respect to a ramp system depends on $h$.
Thus, we are interested in the set of $h$ that appear in $\Lambda(R)$ for a fixed $(\gamma, \nu, \theta)$, i.e., 
\begin{equation}
\label{eq:H0}
\cH_0(\gamma,\nu,\theta) := \setof{ h \mid (\gamma, \nu, \theta, h) \in \Lambda(R)},    
\end{equation}
and conversely, the set of $(\gamma, \nu, \theta)$ that appear in $\Lambda(R)$ for some $h$, that we denote by
\begin{equation}
\label{eq:lambdaS}
\Lambda(S) := \setdef{(\gamma, \nu, \theta)}{\cH_0(\gamma,\nu,\theta) \neq \emptyset }.
\end{equation}

The following lemma, whose proof is left to the reader, shows that $\Lambda(S)$ is an open set and independent of $h$.

\begin{lem}
Given $(\gamma, \nu, \theta)$ we have that $(\gamma, \nu, \theta) \in \Lambda(S)$ if and only if the following conditions are satisfied:
\begin{enumerate}
\item[(i)] If $(m, j), (m', j') \in \mathcal{J}(n)$ and $(m, j) \neq (m', j')$, then
\[
\theta_{m, n, j} \neq \theta_{m', n, j'}.
\]
\item[(ii)] For all $\bv \in \prod_{n=1}^N \setof{0, \ldots, K(n)}$ and all $n=1, \ldots, N$ we have
\[
\gamma_n \theta_{m_{k_n}, n, j_{k_n}} \neq E_n(D_{\bv}) \quad \text{and} \quad \gamma_n \theta_{m_{k_n + 1}, n, j_{k_n + 1}} \neq E_n(D_{\bv}).
\]
\end{enumerate}
\end{lem}

\begin{rem}
In applications we are often interested in minimizing the slope of the ramps in the ramp system, which corresponds to maximizing the $h_{m, n, j}$ parameters. To that end, notice that given a parameter $(\gamma, \nu, \theta) \in \Lambda(S)$, the only restrictions on $h$ are the inequalities given by \eqref{eq:theta_h_inequalities} and hence we can use those inequalities to maximize the components of $h$.
\end{rem}

\begin{prop}
\label{prop:wall-labeling-const}
Let $(\gamma, \nu, \theta) \in \Lambda(S)$.
The wall labeling $\omega \colon W(\cX) \to \setof{\pm 1}$ is constant with respect to $h \in \cH_0(\gamma,\nu,\theta)$.
\end{prop}

\begin{proof}
Let $(\gamma, \nu, \theta) \in \Lambda(S)$ and consider $\bv = (k_1, \ldots, k_N$). By Lemma~\ref{lem:well_defined_DV} the value of $E_n(D_{\bv})$ is constant with respect to $h$. Thus, $\sgn( -\gamma_n \theta_* + E_n(D_{\bv}) )$ is constant with respect to $h$, and therefore, the wall labeling is independent of $h$.
\end{proof}

To summarize the results of this section, if $(\gamma, \nu, \theta, h) \in \Lambda(R)$ then the ramp induced wall labeling $\omega\colon \cW(\cX)\to \setof{\pm 1}$ is strongly dissipative and is constant with respect to $h\in\cH_0(\gamma, \nu, \theta)$.
As is demonstrated in Chapter~\ref{sec:R0}, if $h\in\cH_0(\gamma, \nu, \theta)$, then given $\cF_0$, via Theorem~\ref{thm:dynamics}, the associated Morse graph and connection matrix is applicable to the ramp system at parameter value $(\gamma, \nu, \theta, h)$.
To obtain similar results for $\cF_i$, for $i \in \setof{1, 2, 3}$, requires further restrictions on $h$.
In the remainder of this section we define $\cH_i(\gamma, \nu, \theta)$, which provides sufficient restrictions for $\cF_i$. 

\begin{defn}
\label{defn:H1}
Let $(\gamma, \nu, \theta) \in \Lambda(S)$. 
Define $\cH_1(\gamma, \nu, \theta)$ as the set of $h \in \cH_0(\gamma, \nu, \theta)$ satisfying the following conditions:
\begin{equation}
\label{eq:cH1right}
\frac{E_n(D_{\bv})}{\gamma_n} \not\in \left( \theta_{m_{k_n}, n, j_{k_n}}, \theta_{m_{k_n}, n, j_{k_n}} + h_{m_{k_n}, n, j_{k_n}} \right),    
\end{equation}
and
\begin{equation}
\label{eq:cH1left}
\frac{E_n(D_{\bv})}{\gamma_n} \not\in \left( \theta_{m_{k_n + 1}, n, j_{k_n + 1}} - h_{m_{k_n + 1}, n, j_{k_n + 1}}, \theta_{m_{k_n + 1}, n, j_{k_n + 1}} \right)
\end{equation}
for all $n \in \setof{1, \ldots, N}$ and all $\bv = (k_1, k_2, \ldots, k_N) \in \prod_{n=1}^N \setof{1, \ldots, K(n)}$.
\end{defn}

To obtain intuition concerning the constraints in Definition~\ref{defn:H1} observe that the assumption that $(\gamma, \nu, \theta) \in \Lambda(S)$ implies that 
\begin{equation}
\label{eq:cH1neq}
\frac{E_n(D_{\bv})}{\gamma_n}\neq \theta_{m_{k_n}, n, j_{k_n}}
\end{equation}
for all $n \in \setof{1, \ldots, N}$ and all $\bv = (k_1, k_2, \ldots, k_N) \in \prod_{n=1}^N \setof{1, \ldots, K(n)}$, or equivalently, that $\sgn\left(-\gamma_n \theta_{m_{k_n}, n, j_{k_n}} + E_n(D_{\bv})\right)$ is well defined.
Conditions \eqref{eq:cH1right} and \eqref{eq:cH1left} force $\sgn\left(-\gamma_n x_n + E_n(D_{\bv})\right)$ to agree with $\sgn\left(-\gamma_n \theta_{m_{k_n}, n, j_{k_n}} + E_n(D_{\bv})\right)$ for $x_n$ in the intervals given in \eqref{eq:cH1right} and \eqref{eq:cH1left}.

To define $\cH_2$ it is convenient to introduce the following notation.

Let $(\gamma, \nu, \theta) \in \Lambda(S)$.
Given $\xi=[\bv,\bw] \in \cX$ with $\bv=(k_1,\ldots,k_N) \in \prod_{n=1}^N \setof{1,\dots,K(n)}$, define the interval $I_n(\xi)$ by
\begin{equation}
I_n(\xi) = \begin{cases}
    \left[\theta_{m_{k_n},n,j_{k_n}}-h_{m_{k_n},n,j_{k_n}},\theta_{m_{k_n},n,j_{k_n}}+h_{m_{k_n},n,j_{k_n}}\right] & \text{if } n \in J_i(\xi), \\
\left[\theta_{m_{k_n},n,j_{k_n}}+h_{m_{k_n},n,j_{k_n}},\theta_{m_{k_n+1},n,j_{k_n+1}}-h_{m_{k_n+1},n,j_{k_n+1}}\right] & \text{if } n \in J_e(\xi),
    \end{cases}
\end{equation}
and denote the length of the interval by
\begin{equation}
    \label{eq:IntervalLength}
\Xi_n(\xi) := \begin{cases}
    2h_{m_{k_n},n,j_{k_n}} & \text{if $n \in J_i(\xi)$,} \\
\theta_{m_{k_n+1},n,j_{k_n+1}}-h_{m_{k_n+1},n,j_{k_n+1}} -(\theta_{m_{k_n},n,j_{k_n}}+h_{m_{k_n},n,j_{k_n}}), & \text{if $n \in J_e(\xi)$.}
\end{cases}
\end{equation}

We also make use of the upper and lower bounds of the ramp system on the cell $\xi$, which are defined by
\begin{align}
\label{eq:Ln}
        L_n(\xi) & = \min\setdef{\inf_{x_n \in I_n(\xi)} \left| -\gamma_n x_n + E_n(\mu) \right|}{\mu \in T(\xi)}, \\
\label{eq:Un}
U_n(\xi) & = \max\setdef{\sup_{x_n \in I_n(\xi)} \left| -\gamma_n x_n + E_n(\mu) \right|}{\mu \in T(\xi)}.
\end{align}

\begin{defn}
\label{defn:H2}
Let $(\gamma,\nu,\theta)\in\Lambda(S)$. 
Define $\cH_2(\gamma,\nu,\theta)$ to be the set of all $h \in \cH_1(\gamma,\nu,\theta)$ that satisfies the following conditions for each pair $(\xi,\xi') \in \cX^{(N-2)}\times \cX^{(N-1)}$ that exhibits indecisive drift with GO-pair $(n_g,n_o)$.  
\begin{itemize}
    \item[(i)] If $\xi \in \cF_2(\xi')$, then 
    \begin{equation}
        \label{eq:F2-bounds}
        2h_{{n_o},{n_g},j_{k_{n_g}}} < \frac{L_{n_g}(\xi')}{U_{n_o}(\xi')} \frac{\Xi_{n_o}(\xi')}{2}.
    \end{equation}
    \item[(ii)] If $\xi' \in \cF_2(\xi)$, $n_o$ is not cyclic at any $\sigma \prec \xi$ and $n_o\notin \activeset(\xi)$, then
    \begin{equation}
        \label{eq:F2-bounds-internal-not-active}
        2 h_{n_o,n_g,j_{k_{n_g}}} < \frac{L_{n_g}(\xi)}{U_{n_o}(\xi)} 2 h_{\rmap\xi(n_o),n_o,j_{k_{n_o}}}. 
    \end{equation}
    \item[(iii)] If $\xi' \in \cF_2(\xi)$, $n_o$ is not cyclic at any $\alpha \prec \xi$ and $n_o\in \activeset(\xi)$, then let $n_o'=\rmap\xi(n_o)$ so that for all $\alpha \prec \xi$ such that $\sigma \darrow_{\cF_2} \xi'$ and $\Ex(\alpha,\xi)=\setof{n_o'}$
    \begin{equation}
      \label{eq:F2-bounds-internal-not-cyclic}
      \begin{aligned} 
         & 2 h_{n_o,n_g,j_{k_{n_g}}} <  \\ 
         & \frac{L_{n_g}(\xi)}{U_{n_o}(\xi)} 
          \left( E_{n_o'}^{-1}\left( \gamma_{n_o'}\left( \theta_{*,n_o',j_{k_{n_o'}}} 
         + p_{n_o'}(\alpha, \xi)h_{*,n_o',j_{k_{n_o'}}} \right) \right)
        -\left( \theta_{n_o',n_o,j_{k_{n_o}}} + p_{n_o}(\xi, \xi') h_{n_o',n_o,j_{k_{n_o}}}\right)\right),
      \end{aligned}
    \end{equation}
    where $E_{n_o'}$ depends only on $x_{n_o}$. 
    \item[(iv)] If $\xi' \in \cF_2(\xi)$ and $n_o$ is cyclic at some $\alpha \prec \xi$, then let $n_o'=\rmap\xi(n_o)$ and for all $\alpha \prec \xi$ with $\Ex(\alpha,\xi)=\setof{n_o}$
    \begin{equation}
        \label{eq:F2-bounds-cyclic}
        \frac{\gamma_{n_g}(\theta_{*,n_g,*}+r_{n_g}h_{*,n_g,*})-E_{n_g}}{\gamma_{n_g}(\theta_{*,n_g,*}-r_{n_g}h_{*,n_g,*})-E_{n_g}} > \frac{\gamma_{n_o}E_{n_o'}^{-1}(\gamma_{n_o'}(\theta_{*,n_o',*}+p_{n_o'}(\alpha,\xi)h_{*,n_o',*})-E_{n_g}/\gamma_{n_g}}{\gamma_{n_o}(\theta_{*,n_o,*}+p_{n_o}(\xi,\xi')h_{*,n_o',*})-E_{n_g}/\gamma_{n_g}}
    \end{equation}
    where $E_{n_g}=E_{n_g}(\mu)$ for any $\mu \in \Top_{\cX}(\xi)$ and $E_{n_o'}$ depends only on $x_{n_o}$. 
\end{itemize}
\end{defn}

To obtain explicit bounds, notice that given $n_o \in J_e(\xi')$, the inequality~\eqref{eq:F2-bounds} may be rewritten as 
\begin{equation}
\label{eq:F2-bounds-h-explicit}
\begin{aligned}
2h_{{n_o},{n_g},j_{k_{n_g}}} + \frac{L_{n_g}(\xi')}{U_{n_o}(\xi')}\frac{\left(h_{m_{k_{n_g}+1},n_g,j_{k_{n_g}+1}}+h_{n_o,n_g,j_{k_{n_g}}}\right)}{2} & \\
< \frac{L_{n_g}(\xi')}{U_{n_o}(\xi')}\frac{(\theta_{m_{k_{n_g}+1},n_g,j_{k_{n_g}+1}}-\theta_{n_o,n_g,j_{k_{n_g}}})}{2}. &
\end{aligned}
\end{equation}

\begin{defn}
\label{defn:H3}
Let $(\gamma,\nu,\theta)\in\Lambda(S)$.
For $N=2$, define 
\[
\cH_3(\gamma,\nu,\theta)=\cH_1(\gamma,\nu,\theta).
\]
For $N=3$, let $\xi =[\bv,\bzero] \in \cX^{(0)}$ be a cell whose regulation map $\rmap\xi$ is a 3-cycle, i.e., $\rmap\xi=(1\, 2\, 3)$ or $\rmap\xi=(1\, 3\, 2)$. Let $\kappa^+(\xi)=[\bv,\bone]$ and $\kappa^-(\xi)=[\bv-\bone,\bone]$ be top cells of $\xi$. 
Define $\cH_3(\gamma,\nu,\theta)$ to be the set of all $h \in \cH_1(\gamma,\nu,\theta)$ such that 
    \begin{equation}
    \label{eq:F3-bounds}
        8 \prod_{n=1}^{3} h_{\rmap\xi(n),n,j_{k_n}} < \frac{\prod_{n=1}^{3} \left| E_n(\kappa^+(\xi))-E_n(\kappa^-(\xi))\right|}{-\gamma_1\gamma_2\gamma_3 + (\gamma_1+\gamma_2+\gamma_3)(\gamma_1\gamma_2+\gamma_1\gamma_3+\gamma_2\gamma_3)}
    \end{equation}
whenever the regulation map $\rmap\xi$ at $\xi \in \cX^{(0)}$ is a 3-cycle. 
\end{defn}
While we do not provide a proof that there always exist an $h \in \cH_1(\gamma,\nu,\theta)$ that satisfy \eqref{eq:F2-bounds-cyclic}, we verify numerically its existence in all necessary examples of Chapter~\ref{sec:examples}. For now we will provide a proof that $\cH_i(\gamma,\nu,\theta)\neq\emptyset$ for $i=0,1,3$.

\begin{prop}
\label{prop:smallh}
If $(\gamma, \nu, \theta) \in \Lambda(S)$, then for each $i \in \setof{0, 1, 3}$ there exists $\tilde{h}_i(\gamma, \nu, \theta) > 0$ such that $h = (\bar{h}, \bar{h}, \ldots, \bar{h}) \in \cH_i(\gamma, \nu, \theta)$ for any $\bar{h} \in (0, \tilde{h}_i(\gamma, \nu, \theta))$. In particular
\[
\cH_i(\gamma, \nu, \theta) \neq \emptyset\quad \text{for $i = 0, 1, 3$}.
\]
\end{prop}
\begin{proof}
For $i = 0$, notice that given $(\gamma, \nu, \theta) \in \Lambda(S)$, the only restrictions on $h$ for it to be in $\cH_0(\gamma, \nu, \theta)$ are given by equation \eqref{eq:theta_h_inequalities} in Definition~\ref{defn:admissiblehtheta}. Hence, with the notation in Definition~\ref{defn:admissiblehtheta}, if we take
\[
h_0(\gamma, \nu, \theta, 0) \, = \min_{ \substack{n = 1, \ldots, N \\ k = 1, \ldots, K(n) - 1} } { \frac{\theta_{m_{k+1}, n, j_{k+1}} - \theta_{m_k, n, j_k}}{2} }
\]
the result follows for $\cH_0(\gamma, \nu, \theta)$.

For $i = 1$, notice that given $(\gamma, \nu, \theta) \in \Lambda(S)$, the only additional restrictions on $h$ for it to be in $\cH_1(\gamma, \nu, \theta) \subset \cH_0(\gamma, \nu, \theta)$ are the conditions in Definition~\ref{defn:H1}. To define bounds on $h$ satisfying those conditions let $\tilde{h} = \tilde{h}(\gamma, \nu, \theta) > 0$ be such that if $E_n(D_{\bv}) / \gamma_n \in (\theta_{m_{k_n}, n, j_{k_n}}, \theta_{m_{k_n + 1}, n, j_{k_n + 1}})$, then
\[
\tilde{h} \leq \min{ \setof{ \left| \theta_{m_{k_n}, n, j_{k_n}} - \frac{E_n(D_{\bv})}{\gamma_n} \right|, \left| \theta_{m_{k_n + 1}, n, j_{k_n + 1}} - \frac{E_n(D_{\bv})}{\gamma_n} \right| } },
\]
for all $n \in \setof{1, \ldots, N}$ and all $\bv = (k_1, k_2, \ldots, k_N) \in \prod_{n=1}^N \setof{1, \ldots, K(n)}$. Notice that if $h = (\bar{h}, \bar{h}, \ldots, \bar{h}) \in \cH_0(\gamma, \nu, \theta)$ and $\bar{h} < \tilde{h}$, then $h$ satisfy the conditions of Definition~\ref{defn:H1}, which implies that $h \in \cH_i(\gamma, \nu, \theta)$. Therefore taking
\[
h_0(\gamma, \nu, \theta, 1) \, = \min { \setof{ h_0(\gamma, \nu, \theta, 0), \, \tilde{h}(\gamma, \nu, \theta) } }
\]
the result follows for $\cH_1(\gamma, \nu, \theta)$.

To conclude, we show that there exists $h \in \cH_1(\gamma,\nu,\theta)$ that satisfies \eqref{eq:F3-bounds} when $N = 3$. 
Indeed, if $h = (\bar{h}, \bar{h}, \ldots, \bar{h}) \in \cH_1(\gamma, \nu, \theta)$ satisfies
\[
    \overline{h} < \frac{1}{2} \left( \frac{\prod_{n=1}^{3} \left| E_n(\kappa^+(\xi))-E_n(\kappa^-(\xi))\right|}{-\gamma_1\gamma_2\gamma_3 + (\gamma_1+\gamma_2+\gamma_3)(\gamma_1\gamma_2+\gamma_1\gamma_3+\gamma_2\gamma_3)} \right)^{\frac{1}{3}},
\]
then $h$ satisfies the inequality \eqref{eq:F3-bounds}.
\end{proof}

\section{Properties of Wall Labeling induced by Ramp Systems}
\label{sec:wall-ramp-property}
A natural question to ask is  whether any wall labeling can be originated from ramp systems. The answer is negative. In what follows, we present properties of wall labelings that are derived from ramp systems nonlinearities and give an example of wall labeling that does not satisfy it. 

Let $\dot{x}=-\Gamma x + E(x;\nu,\theta,h)$ be a $N$-dimensional ramp system with a fixed admissible parameter $(\gamma,\nu,\theta,h) \in \Lambda(S)$ as in Definition~\ref{defn:admissible-parameters}. Let $\cX=\cX(\I)$ be the cubical complex associated to the ramp system and $\omega : W(\cX) \to \setof{\pm 1}$ the wall labeling given by Definition~\ref{defn:ramp-wall-labeling}. Given a wall labeling, we defined an associated Rook Field as in Definition~\ref{def:rookfield}. 

We remind the reader that for each $\bv \in \prod_{n=1}^N\setof{1,\ldots,K(n)}$, there is an indexing $\mathcal{J}(n)=\setdef{(m_k,j_k)}{k=1,\ldots,K(n)}$ (see \eqref{eq:Jn}) that provides a one-to-one correspondence between $\bv=(k_1,\ldots,k_N)$ and $(\theta_{m_{k_1},1,j_{k_1}},\ldots,\theta_{m_{k_N},N,j_{k_N}})$.

The following proposition set us up to classify locally the wall labeling derived from ramp systems. We say that $f : \mathbb{R}^n \to \mathbb{R}$ is monotone on the first variable if for all $y \in \mathbb{R}^{n-1}$, the function $x \mapsto f(x,y)$ is monotone. Moreover, we say that it has the same type of monotonicity for all $x \in \mathbb{R}$ if $x \leq x'$ implies that either $f(x,y)\leq f(x',y)$ for all $y \in \mathbb{R}^{n-1}$ or $f(x,y)\geq f(x',y)$ for all $y \in \mathbb{R}^{n-1}$. 
\begin{prop}
    \label{prop:local-monotone}
    For any $\bv=(k_1,\ldots,k_N) \in \prod_{n=1}^N\setof{1,\ldots,K(n)}$ and $n \in \setof{1,\ldots,N}$, $E_n(x;\nu,\theta,h)$ is monotone in each variable $x_i$ on
    \[
       S_\bv(\theta,h) \coloneqq \prod_{n=1}^N (\theta_{m_{k_n}-1, n, j_{k_n}-1} + h_{{k_n}-1, n, j_{k_n}-1}, \theta_{m_{k_n + 1}, n, j_{k_n + 1}} - h_{m_{k_n + 1}, n, j_{k_n + 1}}).
    \]
\end{prop}
\begin{proof}
    Recall that if $r(x;\nu,\theta,h)$ is a ramp function with thresholds $\theta_1, \ldots, \theta_J$, then $r$ is monotone in $(\theta_{j-1}+h_{j-1},\theta_{j+1}-h_{j+1})$ for each $i \in \setof{1,\ldots,J}$. In particular, if $x,y \in (\theta_{j-1}+h_{j-1},\theta_{j+1}-h_{j+1})$ and $0 \leq y-x$, then 
    \[
    0 \leq (r(y;\nu,\theta,h)-r(x;\nu,\theta,h) )(\nu_{j}-\nu_{j-1}).
    \]
    Let $\bv=(k_1,\ldots,k_N) \in \prod_{n=1}^N\setof{1,\ldots,K(n)}$ so that for each $n$, $E_n$ is given by
    \[
    E_n(x;\nu,\theta,h) = f_n\left(r_{n,m_1}(x_{m_1}),\ldots,r_{n,m_{K(n)}}(x_{m_{K(n)}})\right),
    \]
    where $f_n$ is a type I or II interaction function and $r_{n,m}$ are ramp functions. 
    Let $x,y \in S_\bv(\theta,h)$ be such that $x_n = y_n$ for $n\neq i$ and $x_i \leq y_i$.  
    Since $f_n$ is a sum of products or product of sums of ramp functions in $S_\bv(\theta,h)$, let $I \in \setdef{I_j}{1 \leq j \leq N}$ be the subset that contains $i$ in the partition of $\setof{1,\ldots,K}$ done by $f$. Then
    \[
        E_n(y;\nu,\theta,h) - E_n(x;\nu,\theta,h) = c(x;\nu,\theta,h)\left( r_{n,i}(y_i;\nu,\theta,h)-r_{n,i}(x_i;\nu,\theta,h) \right)
    \]
    where $c(x;\nu,\theta,h) = \prod_{\substack{j \in I\setminus\setof{i}}} r_{n,j}(x_j,\nu,\theta,h) \geq 0$ for all $x \in S_\bv(\theta,h)$. Thus, $E_n$ is monotone on each coordinate $x_i$ on $S_\bv(\theta,h)$ with the same monotonicity as $r_{n,i}$. 
\end{proof}
Observe that Proposition~\ref{prop:local-monotone} implies that $-\gamma_n x_n + E_n(x;\nu,\theta,h)$ is monotone with respect to each input $x_i$ in the region $S_\bv(\theta,h)$ with the same type of monotonicity. We remark, however, that it might be monotonically decreasing with respect to a variable $x_i$ while being monotonically increasing with respect to a variable $x_j$. 

In terms of wall labeling and rook fields, we obtain the following proposition. 
\begin{prop}
    \label{prop:wall-labeling-local-monotone}
    Let $\sigma=[\bv,\bzero]\in \cX^{(0)}\setminus\bbdy(\cX)$ with $\bv=(k_1,\ldots,k_N)$. If $\rook : TP(\cX) \to \setof{0,\pm 1}^N$ is the associated rook field, then $\Psi_n \colon \setof{0,1}^N \to \setof{\pm 1}$ given by
    \[
        \Psi_n(\bv') = \rook_n \left( [\bv,\bzero] , [\bv-\bone+\bv',\bone] \right)
    \]
    is monotone on each entry for all $n \in \setof{1,\ldots,N}$.
\end{prop}
\begin{proof}
    Without loss of generality, we assume $\bv=\bone$ and consider the function $\Psi_n \colon \setof{0,1}^N \to \setof{\pm 1}$ given by
    \[
        \Psi_n(\bv) = \rook_n \left( [\bone,\bzero] , [\bv,\bone] \right),
    \]
    for some each $n \in \setof{1,\ldots,N}$. Note that $\Psi_n(\bv)$ evaluates the Rook Field $\rook$ in the $n$-direction at each top dimensional $\mu = [\bv,\bone] \in \Top_\cX(\sigma)$ of $\sigma$. More precisely, 
    \[
        \Psi_n(\bv') = \rook_n(\sigma,\mu) = \omega(\mu^*(\sigma,\mu),\mu)
    \]
    where $\mu^*(\sigma,\mu)$ is the unique $n$-wall such that $\sigma \preceq \mu^*(\sigma,\mu) \prec \mu$ as in Corollary~\ref{cor:mu*}. Since for any $n$-wall $\xi=[\bv',\bone^{(n)}] \in \cX$ such that $\sigma \prec \xi$, it must be the case that $\bv'_n=1$, it follows that 
    \[
        \rook_n(\sigma,\mu) = \sgn(-\gamma_n\theta_{m_{1},n,j_{1}}+E_n(\mu)),
    \]
    where $E_n(\mu)$ corresponds to the constant value that $E_n$ assumes in $D_{\bv}(\theta,h)$. In particular, if one chooses $x(\mu)$ as the barycenter of $D_{\bv}(\theta,h)$, then 
    \[
         \rook_n(\sigma,\mu) =  \sgn(-\gamma_n\theta_{m_{1},n,j_{1}}+E_n(x(\mu);\nu,\theta,h))
    \]
    and the monotonicity on each component of $\Psi_n$ is of the same type as the monotonicity of $E_n$ proven in Proposition~\ref{prop:local-monotone} since $\bv \mapsto x([\bv,\bone])$ preserves the order in each coordinate.

    Indeed, if $\bv,\bv'\in \setof{0,1}^N$ with $\bv_i \leq \bv'_i$ and $\bv_n=\bv'_n$ for $n \neq i$, then $x([\bv,\bone])$ and $x([\bv'',\bone])$ satisfy 
    \[
        x_i([\bv,\bone]) \leq x_i([\bv',\bone]) \quad\text{and}\quad x_n([\bv,\bone]) = x_n([\bv',\bone]),\ \text{for $n \neq i$}. 
    \]
\end{proof}
Finally, we piece together the above results to obtain a characterization of the rook field around a vertex $\sigma \in \cX$ by denoting the top cells $\mu \in \Top_\cX(\sigma)$ using the boolean indexing of Proposition~\ref{prop:wall-labeling-local-monotone}, i.e., $\mu_{\bf i} = \mu_{{\bf i_1}\ldots{\bf i_N}}$. 
\begin{prop}
    \label{prop:subgraph-connected}
    Let $\sigma \in \cX^{(0)}\setminus\bbdy(\cX)$. 
    Consider the $N$-hypercube graph $G=(V,E)$ with $V=\Top_\cX(\sigma)$ and $E = \setdef{(\mu,\mu') \in V\times V}{\mu,\mu' \text{ are } k\text{-adjacent}}$. 
    Fix $n \in \setof{1,\ldots,N}$ and define $W \colon V \to \{\pm 1\}$ by $W(\mu) = \omega(\mu^*(\sigma,\mu),\mu)$. 
    Then both subgraphs with vertices $W^{-1}(+1)$ and $W^{-1}(-1)$ are connected. Moreover, each restriction to a face of $G$ is connected.
\end{prop}
\begin{proof}
    We prove by induction on $N$. The case $N=1$ is trivial since either there are no edges connecting vertices of $W^{-1}(+1)$ or it is the whole graph. 

    For $N=2$, the subgraphs are disconnected if and only if the vertices are given by $\setof{\mu_{00},\mu_{11}}$ and $\setof{\mu_{10},\mu_{01}}$. Without loss of generality, assume that the sets represent $W^{-1}(1)$ and $W^{-1}(-1)$ respectively. By identifying $W(\mu_{\bf i})$ with $\Psi_n({\bf i})$, we obtain
    \[
        \Psi_n(0,0) > \Psi_n(1,0) \text{ and } \Psi_n(0,1) < \Psi_n(1,1),
    \]
    which contradicts the monotonicity of $\Psi_n$ on the first coordinate. 

    For $N \geq 3$, note that if $\mu,\mu'\in W^{-1}(1)$ belong to the same face, then it is connected by the induction hypothesis. Suppose for the sake of contradiction that no two vertices in $W^{-1}(1)$ belong to the same face, so that $W^{-1}(+1)=\setof{\mu,\mu'}$ consists of opposing vertices. Without loss of generality, we may assume that $\mu=\mu_{\bzero}$ and $\mu'=\mu_{\bone}$. Then 
    \begin{equation}
        \label{eq:break-monotone}
        \Psi_n(\bzero) > \Psi_n(\bone^{(1)}) \text{ and } \Psi_n(\bzero^{(1)}) < \Psi_n(\bone),
    \end{equation}
    contradicts the monotonicity of $\Psi_n$ on the first coordinate. 

    Thus, both subgraphs with vertices $W^{-1}(+1)$ and $W^{-1}(-1)$ are connected.

    The additional result follows from a similar argument by observing that \eqref{eq:break-monotone} can be applied to any $\bzero, \bzero^{(n)},\bzero^{(k)}, \bzero^{(n)}+\bzero^{(k)}$ to contradict monotonocity in the $n$-th coordinate within the face that contains the $n$ and $k$ directions. 
\end{proof}
In what follows, we describe an example for $N=2$ that does not satisfy the monotonicity of the ramp systems but is a valid strongly dissipative wall labeling in the general setting. 
\begin{ex}
    \label{ex:7}
    Let $\I=\setof{0,1,2}\times\setof{0,1,2}$ and consider the cubical complex $\cX = \cX(\I)$ generated by $\I$. Define the wall labeling $\omega : W(\cX) \to \setof{-1,1}$ by 
    \begin{align*}
        \omega\left(\defcell{1}{0}{0}{1},\defcell{0}{0}{1}{1}\right) & = 1, & 
        \omega\left(\defcell{1}{0}{0}{1},\defcell{1}{0}{1}{1}\right) & = -1, \\
        \omega\left(\defcell{1}{1}{0}{1},\defcell{0}{1}{1}{1}\right) & = -1, &
        \omega\left(\defcell{1}{0}{0}{1},\defcell{1}{1}{1}{1}\right) & = 1,
    \end{align*}
    and
    \begin{equation*}
        \omega(\xi,\mu) = \begin{cases}
            1 & \text{ if } \xi \notin \bbdy(\cX) \text{ and } J_i(\xi)=\setof{2} \\
            -p_n(\xi,\mu) & \text{ if } \xi \in \bbdy(\cX) \text{ and } J_i(\xi)=\setof{n}
        \end{cases}
    \end{equation*}
    as in Figure~\ref{fig:ex7}. 
    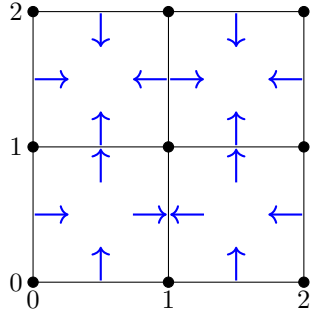
\begin{figure}[H]
        \centering
        \begin{tikzpicture}[scale=0.225]
        \draw[step=8cm,black] (0,0) grid (16,16);
        
        \foreach \i in {0,1,2}{
            \foreach \j in {0,1,2}{
            \draw[black, fill=black] (8*\i,8*\j) circle (2ex);
            }
        }
        
        \foreach \i in {0,1,2}{
            \draw(8*\i,-1) node{$\i$}; 
            \draw(-1,8*\i) node{$\i$}; 
        }
        
        \foreach \i in {0,1}{
            \draw[->, blue, thick] (4+8*\i,0.1) -- (4+8*\i,2.1); 
            \draw[->, blue, thick] (4+8*\i,15.9) -- (4+8*\i,13.9); 
            \draw[->, blue, thick] (0.1, 4+8*\i) -- (2.1,4+8*\i); 
            \draw[->, blue, thick] (15.9,4+8*\i) -- (13.9,4+8*\i); 
        }

        \draw[->, blue, thick] (4,5.9) -- (4,7.9); 
        \draw[->, blue, thick] (4,8.1) -- (4,10.1);
        \draw[->, blue, thick] (12,5.9) -- (12,7.9); 
        \draw[->, blue, thick] (12,8.1) -- (12,10.1);

        \draw[->, blue, thick] (5.9,4) -- (7.9,4); 
        \draw[->, blue, thick] (10.1,4) -- (8.1,4); 
        \draw[->, blue, thick] (7.9,12) -- (5.9,12); 
        \draw[->, blue, thick] (8.1,12) -- (10.1,12); 
        
        \end{tikzpicture}
        \caption{Wall labeling of Example~\ref{ex:7}}
        \label{fig:ex7}
    \end{figure}
    For any $\sigma \in \cX^{(0)}$, define the map $\tilde{o}_\sigma : \setof{1,2} \to \setof{1,2}$ by 
    \[\tilde{o}_\sigma(1)=\tilde{o}_\sigma(2)=1.\]
    Note that $\tilde{o}_\sigma$ is a local inducement map as in Definition~\ref{def:rookfield}. Indeed, if $(\xi,\mu)$ and $(\xi,\mu') \in W(\sigma)$ are $2$-walls, then 
    \[
    \omega(\xi,\mu)=\omega(\xi,\mu')=1.
    \]
    If $(\xi,\mu),(\xi,\mu')$ are $1$-walls, there is nothing to verify. The same is true for any $\mu,\mu' \in \Top_\cX(\sigma)$ that are $2$-adjacent, since there is no $k \in \{1,2\}$ such that $k \neq 2$ and $k\neq \tilde{o}_\sigma(2)=1$. Therefore, $\omega$ is a wall labeling. 
\end{ex}

In what follows, we set up a technical proposition that is  used in Chapter~\ref{sec:R2Dynamics}.
\begin{prop}
    \label{prop:opaque-pairs}
    Assume that $(\xi,\xi') \in \cX^{(N-k)} \times \cX^{(N-(k-1))}$ with $\Ex(\xi,\xi')=\setof{n}$ for some $k \geq 2$. If there exists $\mu,\mu' \in \Top_\cX(\xi')$ such that
    \[
        \rook_n(\xi,\mu) \neq \rook_n(\xi,\mu'),
    \]
    then there exists at least $k-1$ walls $\xi_i \in \cX^{(N-1)}$ with $\xi\preceq \xi_i$ such that $T(\xi_i)=\setof{\mu_i,\mu_i'}$ satisfies 
    \[
        \omega(\xi_i,\mu_i) \neq \omega(\xi_i,\mu_i'). 
    \]
\end{prop}
\begin{proof}
    Let $G=(V,E)$ be the $(k-1)$-hypercube graph given by $V=\Top_\cX(\xi')$ and edges connecting adjacent top dimensional cells. By Proposition~\ref{prop:wall-labeling-local-monotone}, the subgraphs obtained from $\rook_n(\xi,\mu)=+1$ and $\rook_n(\xi,\mu)=-1$ are connected. Note that for each edge $(\mu_i,\mu_i') \in E$ that connect vertices with opposite signs, there is a wall $\xi_i \in W(\xi)$ such that $\omega(\xi_i,\mu_i) \neq \omega(\xi_i,\mu_i')$. Thus, it is enough to show that there exists at least $k$ such edges. Indeed, if one assumes that there exists $\mu,\mu' \in T(\xi')$ such that
    \[
        \rook_n(\xi,\mu) \neq \rook_n(\xi,\mu'),
    \]
    then there exists at least one positive and one negative vertex.

    We prove by induction on $K=k-1$. If $K=1$, then there is an unique edge. If $K > 1$, then consider a face that has both positive and negative, which yields at least $K-1$ edges by the induction hypothesis. If a positive (resp. negative) vertex connects to a vertex with opposite sign on the opposing face, then we're done. Otherwise, the opposing face has negative and positive vertices too, which yields another set of $K-1$ edges by the induction hypothesis. Thus, $2(K-1)\geq K$ when $K>1$. 
\end{proof}

\section{DSGRN}
\label{sec:DSGRN}

The results of Chapters~\ref{sec:RookFields} through \ref{sec:lattice} provide an algorithm that takes a wall labeling as input and produces a Morse graph and Conley complex(es) as output.
Widespread applicability of this work depends on being able to translate problems of interest into wall graphs.
We do this using the DSGRN software \cite{DSGRN}. 
Motivation for DSGRN came from systems biology where models take the form of regulatory networks \cite{cummins:gedeon:harker:mischaikow:mok} and more recently from ecology where models take the form of trophic networks \cite{cuello:gameiro:bonachela:mischaikow}.
In both cases, the expectation is that an ordinary differential equation provides an adequate model for the associated dynamics, but explicit nonlinearities can not be derived from first principles and associated parameters are expensive to measure and not necessarily transferable between organisms or ecosystems.

For this manuscript the details of DSGRN are not necessary.
We present high level description of the simplest functional capabilities of DSGRN as this is helpful in understanding how DSGRN can be employed and provide references for more precise information.

A \emph{regulatory network} $RN$ is a finite directed graph with vertices $V=\setof{1,\ldots,N}$ and annotated directed edges $(n,m)\in V\times V$ of the form $n\to m$ (activation) or $n\dashv m$ (repression), such that for each node there is a nonlinear function $\Sigma_n$ representing the interaction of the incoming edges (see \cite{cummins:gedeon:harker:mischaikow:mok, kepley:mischaikow:zhang, cummins:gameiro:gedeon:kepley:mischaikow:zhang, cuello:gameiro:bonachela:mischaikow}
for more precise definitions).
From the biological perspective each node $n$ represents a species (biochemical or organismal) and $x_n\geq 0$ represents the quantity or density of that species.
Thus the associated phase space is $[0,\infty)^N$.

It useful to think that the software DSGRN \cite{DSGRN} takes a regulatory network as input, identifies a parameter space, provides a decomposition of parameter space into a finite number of explicit regions, and for each region produces a
wall labelling.

The parameter space is defined as follows.
To each node $n\in V$, there is a decay rate $\gamma_n >0$.
There are three  parameters $\theta_{n,m}>0$, $\ell_{n,m}>0$, and $\delta_{n,m}>0$ assigned to each edge from node $m$ to node $n$ that define an \emph{activating function}
\[
\sigma_{n,m}^+(x_m) =
\begin{cases}
\ell_{n,m} & \text{if $x_m< \theta_{n,m}$,} \\
\ell_{n,m} + \delta_{n,m} & \text{if $x_m > \theta_{n,m}$,} 
\end{cases}
\]
or
\[
\sigma_{n,m}^-(x_m) =
\begin{cases}
\ell_{n,m}  + \delta_{n,m} & \text{if $x_m< \theta_{n,m}$,} \\
\ell_{n,m} & \text{if $x_m > \theta_{n,m}$,} 
\end{cases}
\]
if $m \to n$ or $m \dashv n$, respectively. Thus given a regulatory network $RN$, the associated parameter space is $(0,\infty)^{N + 3L}$ where $N$ is the number of nodes and $L$ is the number of edges.

Consider a fixed regulatory network $RN$ with $N$ nodes, each node $n$ has a set of in-edges from nodes
$\source(n) = \setof{m_{i_1},\ldots m_{i_{k_n}}}$.
The nonlinearity $\Sigma_n$ is a sum and/or product of the functions $\sigma_{n,m}^\pm(x_{m})$ for $m \in \source(n)$.
More precisely,
\[
\Sigma_n (x) = \Sigma_n \left( x_{m_{i_1}}, \ldots, x_{m_{i_{k_n}}} \right) = f_n \left(\sigma^{\pm}_{n, m_{i_1}}(x_{m_{i_1}}), \ldots, \sigma^{\pm}_{n,m_{i_{k_n}}}(x_{m_{i_{k_n}}})  \right),
\]
where $f_n$ is an interaction function as defined in Definition~\ref{def:interaction_function}, and $\sigma^{\pm}_{n, m_{\ell}}$ is $\sigma^+_{n, m_{\ell}}$ or $\sigma^-_{n, m_{\ell}}$ if $m_{\ell} \to n$ or $m_{\ell} \dashv n$, respectively.

The \emph{rate of change} for node $n$ is given by
\[
-\gamma_n x_n + \Sigma_n(x)
\]
where $\Sigma_n(x)$ is as defined above and is independent of $x_m$ if $m \not\in \source(n)$.

Observe that $\Sigma_n(x)$ can take on $2^{k_n}$ values expressed in terms of the parameters $\ell_{n,m_i}$ and $\delta_{n,m_i}$, for $m_i\in \source(n)$. Restricting to a parameter at which these values are distinct for a fixed set of $\ell_{n,m_i}$ and $\delta_{n,m_i}$ we denote the range of $\Sigma_n(x)$ over all $x$ by
\begin{equation}
\label{eq:pvalues}
p_0, \, p_1, \, \cdots, \, p_{2^{k_n}-1}.
\end{equation}

We summarize the discussion so far by noting that for fixed parameter values $\gamma_n$, $\ell_{n,m_i}$ and $\delta_{n,m_i}$, for $m_i \in \source(n)$, whether the variable $x_n$ is increasing or decreasing when the system is at the state $x$ is determined by
\[
\sgn\left( -\gamma_n x_n + \Sigma_n(x) \right) = \sgn \left( -\gamma_n x_n + p_j \right),
\]
where $p_j = \Sigma_n(x)$.

For node $n$, let $\target(n) = \setof{m_{i_1},\ldots m_{i_{K(n)}}}$ be the set of nodes  which have an incoming edge from $n$. Note that we are using $K(n)$ to denote the number of edges that leave node $n$. We also assume that the threshold values $\theta_{m_i,n}$, for $m_i \in \target(n)$, are distinct.

\begin{defn}
\label{defn:DSRGNparameter}
For a given regulatory network $RN$ with $N$ nodes and parameter values in $(0, \infty)^D$ the \emph{DSGRN parameter space} is denoted by $\P(RN)$ and consist of  parameter values $(\gamma, \ell, \delta, \theta) \in (0, \infty)^D$ that satisfy the following conditions for each $n = 1, \ldots, N$.
\begin{itemize}
\item[(i)] The values $p_0, \, p_1, \, \cdots, \, p_{2^{k_n}-1}$ defined by \eqref{eq:pvalues} are distinct.
\item[(ii)] The threshold values $\theta_{m_i,n}$, for $m_i \in \target(n)$, are distinct.
\item[(iii)] $\gamma_n \theta_{m_i,n} \neq p_j$ for each $m_i \in \target(n)$ and $j = 0, \ldots, 2^{k_n}-1$.
\end{itemize}
\end{defn}

Given a DSGRN parameter $(\gamma, \ell, \delta, \theta) \in \P(RN)$ we have that
\[
\sgn\left( - \gamma_n \theta_{m_i,n} + p_j \right) \neq 0.
\]
The sets of parameter values at which the sign above is constant for each $n$, $m_i$, and $j$ is given as a finite set of explicit open semi-algebraic sets \cite{cummins:gedeon:harker:mischaikow:mok, kepley:mischaikow:zhang, cummins:gameiro:gedeon:kepley:mischaikow:zhang, cuello:gameiro:bonachela:mischaikow}. 
These semi-algebraic sets are disjoint, but the union of their closures covers the entire parameter space $(0, \infty)^D$. We refer to these sets as \emph{parameter regions} and they form a decomposition of the parameter space which is computable and for regulatory networks with limited numbers of in-edges they are precomputed as part of DSGRN. This decomposition of parameter space into parameter regions is represented by DSGRN as a \emph{parameter graph}, where the nodes are the parameter regions and edges represent co-dimension $1$ adjacency \cite{cummins:gedeon:harker:mischaikow:mok, kepley:mischaikow:zhang, cummins:gameiro:gedeon:kepley:mischaikow:zhang, cuello:gameiro:bonachela:mischaikow}. 
We denote the parameter graph of a regulatory network $RN$ by $PG(RN)$ and denote the nodes in the parameter graph by $k \in \setof{0, \ldots, M - 1}$, where $M$ is the number of nodes in the parameter graph. Given a node $k$ of the parameter graph we denote the parameter region associated to node $k$ by $R(k)$.

Let $RN$ be a regulatory network and let $(\gamma, \ell, \delta, \theta) \in \P(RN)$ be a fixed parameter. Associated to $(\gamma, \ell, \delta, \theta)$ we define a family of ramp systems, parameterized by $h$, as follows. For each edge $m \to n$ of $RN$ define
\[
\nu_{n, m, 1} = \ell_{n, m} \quad \text{and} \quad \nu_{n, m, 2} = \ell_{n, m} +\delta_{n, m},
\]
and for each edge $m \dashv n$ of $RN$ define
\[
\nu_{n, m, 1} = \ell_{n, m} +\delta_{n, m} \quad \text{and} \quad \nu_{n, m, 2} = \ell_{n, m}.
\]
Let $0 < h_{n, m} < \theta_{n, m}$ and define the ramp function
\[
r_{n, m}(x_m) =
\begin{cases}
\nu_{n, m, 1}, & \text{if}~ x_m < \theta_{n, m} - h_{n, m} \\
L(x_m), & \text{if}~ \theta_{n, m} - h_{n, m} \leq x_m \leq \theta_{n, m} + h_{n, m} \\
\nu_{n, m, 2}, & \text{if}~ x_m > \theta_{n, m} + h_{n, m},
\end{cases}
\]
where $L(x_m) = \dfrac{\nu_{n, m, 2} - \nu_{n, m, 1}}{2 h_{n, m}}(x_m - \theta_{n, m}) + \dfrac{\nu_{n, m, 1} + \nu_{n, m, 2}}{2}$.

Notice that $r_{n, m}$ is an increasing ramp function if $m \to n$ and it is a decreasing ramp function if $m \dashv n$. We define the ramp system associated to a DSGRN regulatory network $RN$ by replacing the activating functions $\sigma^{\pm}_{n, m}$ in the nonlinearity $\Sigma_n$ by the ramp functions $r_{n, m}$ defined above.

\begin{rem}
    Notice that for ramp systems derived from DSGRN the ramp functions are of type $J=1$.
\end{rem}

\begin{defn}
\label{defn:DSGRN_ramp}
Let $RN$ be a regulatory network with $N$ nodes and rate of change for node $n$ given by $-\gamma_n x_n + \Sigma_n(x)$, where
\[
\Sigma_n (x) = \Sigma_n \left( x_{m_{i_1}}, \ldots, x_{m_{i_{k_n}}} \right) = f_n \left(\sigma^{\pm}_{n, m_{i_1}}(x_{m_{i_1}}), \ldots, \sigma^{\pm}_{n,m_{i_{k_n}}}(x_{m_{i_{k_n}}})  \right)
\]
and $f_n$ is an activating function. The \emph{ramp system generated} by the regulatory network $RN$ is the $N$-dimensional ramp system $\dot{x} = -\Gamma x + E(x)$, where
\[
\dot{x}_n  = -\gamma_n x_n + E_n(x), \quad n = 1, \ldots, N,
\]
and
\[
E_n (x) = E_n \left( x_{m_{i_1}}, \ldots, x_{m_{i_{k_n}}} \right) = f_n \left(r_{n, m_{i_1}}(x_{m_{i_1}}), \ldots, r_{n,m_{i_{k_n}}}(x_{m_{i_{k_n}}})  \right),
\]
with $r_{n, m_{\ell}}$ as defined above.
\end{defn}

\begin{rem}
We denote the parameters in the DSGRN nonlinearity $\Sigma_n(x)$ of a regulatory network by $\ell_{n, m}$ and $u_{n, m} =\ell_{n, m}+\delta_{n,m}$ as this is the standard notation in the DSGRN literature \cite{cummins:gedeon:harker:mischaikow:mok, kepley:mischaikow:zhang, cummins:gameiro:gedeon:kepley:mischaikow:zhang, cuello:gameiro:bonachela:mischaikow} as well as in the DSGRN software \cite{DSGRN}. However when describing the ramp system generated by the regulatory network it is convenient to relabel these parameters as $\nu_{n, m, 1}$ and $\nu_{n, m, 2}$ as done above.
\end{rem}

In a slight abuse of notation we denote the parameters in the parameter space $\P(RN)$ of a regulatory network by $(\gamma, \nu, \theta)$, where $\nu$ is defined in terms of $\ell$ and $u$ as described above. Notice that the parameters in $\P(RN)$ must satisfy the restrictions $\nu_{n, m, 1} < \nu_{n, m, 2}$ if $m \to n$ and $\nu_{n, m, 1} > \nu_{n, m, 2}$ if $m \dashv n$. Hence it follows that the parameter space $\P(RN)$ of a regulatory network $RN$ is a proper subset of the $(\gamma, \nu, \theta)$ parameter space $\Lambda(S)$ of the ramp system generated by $RN$. Given a parameter $(\gamma, \nu, \theta) \in \P(RN)$ it follows by Proposition~\ref{prop:smallh} that $\cH_0(\gamma, \nu, \theta) \neq \emptyset$ and hence it defines a wall labeling $\omega \colon W(\cX) \to \setof{-1,1}$ which is constant with respect to $h \in \cH_0(\gamma, \nu, \theta)$. Furthermore, by the definition of the parameter graph, we have that the the signs
\[
\sgn\left( - \gamma_n \theta_{m_i,n} + p_j \right)
\]
are constant for every parameter in a given parameter region $R(k)$ of the parameter graph. This implies that the wall labeling $\omega \colon W(\cX) \to \setof{-1, 1}$ is constant for all parameter in $R(k)$. Hence we have the following Proposition.

\begin{prop}
\label{prop:DSGRN_wall_constant}
Let $RN$ be a regulatory network with associated DSGRN parameter graph $PG(RN)$ and let $\dot{x} = -\Gamma x + E(x)$ be the ramp system generated by $RN$.
Let $k$ be a node in the parameter graph and let $R(k)$ be the associated parameter region. The wall labeling $\omega \colon W(\cX) \to \setof{-1, 1}$ generated by the ramp system $\dot{x} = -\Gamma x + E(x)$ is constant with respect to $(\gamma, \nu, \theta) \in R(k)$ and $h \in \cH_0(\gamma, \nu, \theta)$.
\end{prop}

\begin{ex}
\label{ex:RNDSGRN}
Consider the regulatory network shown in Figure~\ref{fig:network_2_intro}(A). The parameter space is $(0,\infty)^{14}$. DSGRN decomposes the parameter space into exactly $1600$ regions. Furthermore, region $R(974)$ as enumerated by DSGRN produces the wall labeling shown in Figure~\ref{fig:wall_labeling}(A).
\end{ex}

\chapter{Rectangular Geometrizations of $\cX_b$}
\label{sec:geometrization01}

Consider a ramp system $\dot{x} = -\Gamma x + E(x; \nu, \theta, h)$ with fixed parameter values $(\gamma, \nu, \theta, h) \in \Lambda(R)$ (see Definition~\ref{defn:admissible-parameters}). Let $\cX$ be the associated ramp induced cubical complex (see Definition~\ref{defn:rampcellcomplex}) and let $\cX_b$ be the associated blow-up cubical complex (see Definition~\ref{defn:Xbcomplex}). The goal of this chapter is to provide a rectangular geometrization of the cell complex $\cX_b$.

Recall that $\cX = \cX(\I)$ is the cubical complex generated by
\[
\I = \prod_{n=1}^N \setof{0, \ldots, K(n) + 1}
\]
where $K(n)$ in the number of output thresholds for the variable $x_n$ (see \eqref{eq:num_out_thresholds}), and $\cX_b = \cX(\bar{\I})$ is the cubical complex generated by
\[
\bar{\I} = \prod_{n=1}^N \setof{0, \ldots, 2(K(n)+1) + 1}.
\]

Using the indexing set $\mathcal{J}(n)$ (see Definition~\ref{defn:admissiblehtheta})
\[
\theta_{m_k, n, j_k} + h_{m_k, n, j_k} < \theta_{m_{k+1}, n, j_{k+1}} - h_{m_{k+1}, n, j_{k+1}},
\]
for all $k = 1, \ldots, K(n) - 1$ and all $n=1,\ldots,N$.
For the purpose of defining a rectangular geometrization of $\cX_b$ set
\begin{equation}
\label{eq:m0nj0}
\theta_{m_0, n, j_0} = h_{m_0, n, j_0} := \frac{\theta_{m_1, n, j_1} - h_{m_1, n, j_1}}{4}
\end{equation}
and
\begin{equation}
\label{eq:mKnjK}
\theta_{m_{K(n) + 1}, n, j_{K(n) + 1}} := \gab_n(\gamma, \nu, \theta, h) - \frac{1}{4}, \quad h_{m_{K(n) + 1}, n, j_{K(n) + 1}} := \frac{1}{4}.
\end{equation}
Notice that
\begin{equation}
\label{eq:rectangular_geo_boundary}
\begin{aligned}
& \theta_{m_0, n, j_0} - h_{m_0, n, j_0} = 0 \\
& \theta_{m_{K(n) + 1}, n, j_{K(n) + 1}} + h_{m_{K(n) + 1}, n, j_{K(n) + 1}} = \gab_n(\gamma, \nu, \theta, h).
\end{aligned}
\end{equation}

\begin{rem}
The actual values of $\theta_{m_0, n, j_0}$, $h_{m_0, n, j_0}$, $\theta_{m_{K(n) + 1}, n, j_{K(n) + 1}}$, and $h_{m_{K(n) + 1}, n, j_{K(n) + 1}}$ above are not particularly important. 
They are set to insure that
\begin{align*}
    \theta_{m_0, n, j_0} + h_{m_0, n, j_0} &< \theta_{m_1, n, j_1} - h_{m_1, n, j_1},   \\
    \theta_{m_{K(n)}, n, j_{K(n)}} + h_{m_{K(n)}, n, j_{K(n)}} &< \theta_{m_{K(n) + 1}, n, j_{K(n) + 1}} - h_{m_{K(n) + 1}, n, j_{K(n) + 1}},
\end{align*}
and condition \eqref{eq:rectangular_geo_boundary} is satisfied.
\end{rem}

The first step towards the goal of defining a geometrization of $\cX_b$ is the identification of vertices in $\cX_b$ with points in the phase space $[0,\infty)^N$ of the ramp system. Given $\bv \in \cX_b^{(0)}$, by the definition of $\bar{\I}$, it follows that
\[
\bv = (2 k_1 + j_1, \ldots, 2 k_N + j_N)
\]
where $(k_1, \ldots, k_N) \in \I$ and $j_1, \ldots, j_N \in \setof{0, 1}$. Define
\[
e \colon \cX_b^{(0)} \to \prod_{n = 1}^N {[0, \gab_n(\gamma, \nu, \theta, h)]} \subset [0,\infty)^N
\]
by $e(\bv) = (e_1(\bv), \ldots, e_N(\bv))$ where
\[
e_n(\bv) :=
\begin{cases}
\theta_{m_k, n, j_k} - h_{m_k, n, j_k}, & \text{if}~ j_n = 0 \\
\theta_{m_k, n, j_k} + h_{m_k, n, j_k}, & \text{if}~ j_n = 1
\end{cases}
\]
If follows from \eqref{eq:rectangular_geo_boundary} that if $k_n = 0$ and $j_n = 0$ (left boundary of $\cX_b$), then
\[
e_n(\bv) = \theta_{m_0, n, j_0} - h_{m_0, n, j_0} = 0
\]
and if $k_n = K(n) + 1$ and $j_n = 1$ (right boundary of $\cX_b$), then
\begin{equation}
\label{eq:evertices}
e_n(\bv) = \theta_{m_{K(n) + 1}, n, j_{K(n) + 1}} + h_{m_{K(n) + 1}, n, j_{K(n) + 1}} = \gab_n(\gamma, \nu, \theta, h).
\end{equation}

Let $\bw \in \setof{0,1}^N$ and $[\bv,\bw] \in \cX_b$. Define
\[
I_n([\bv,\bw]) := [e_n(\bv), e_n(\bv + \bw)] =
\begin{cases}
\setof{e_n(\bv)},     & \text{if}~ w_n = 0 \\
[e_n(\bv), e_n(\bv + \bzero^{(n)})], & \text{if}~ w_n = 1
\end{cases}
\]
and
\[
\ell_n(\bw) :=
\begin{cases}
\setof{0}, & \text{if}~ w_n = 0 \\
[0, 1],    & \text{if}~ w_n = 1.
\end{cases}
\]

Using this notation we define the rectangular regions
\begin{equation}
\label{eq:Rxi}
\Rec([\bv,\bw]) := \prod_{n = 1}^{N} {I_n([\bv,\bw])} \subset \prod_{n = 1}^{N} {[0, \gab_n(\gamma, \nu, \theta, h)]},
\end{equation}
and
\begin{equation}
\label{eq:K(w)}
K(\bw) := \prod_{n = 1}^{N} {\ell_n(\bw)} \subset [0,1]^N.    
\end{equation}
Observe that $K(\bw)$ is a $k$-dimensional face of $[0,1]^N$, where $k = \dim(\bw)$.

We inductively define homeomorphisms
\[
a_\zeta \colon K(\bw) \to \Rec(\zeta),\quad \zeta =[\bv,\bw] \in \cX_b
\]
as follows. Let $\dim(\zeta) = 0$. Then $\bw =\bzero$ and hence $K(\bw) = \setof{(0, \ldots, 0)} \subset \R^N$.
Set
\[
a_\zeta(\bzero) = e(\bv).
\]
Assume that $a_\zeta$ is defined for all $\zeta\in\cX_b$ such that $\dim(\zeta)\leq k-1$.
Let $\dim(\zeta) = k$.
Define $a_{[\bv,\bw]} \colon K(\bw) \to \Rec([\bv,\bw])$ to be the affine linear map from $K(\bw)$ to $\Rec(\zeta)$ such that if $\bu <_\Z \bw$, then
\[
a_{[\bv,\bw]} \mid_{K(\bu)} = a_{[\bv,\bu]}.
\]

Let $\cG$ denote the set of homeomorphisms
\begin{equation}
\label{eq:G}
\Geo \colon \prod_{n = 1}^{N} {[0, \gab_n(\gamma, \nu, \theta, h)]} \to \prod_{n = 1}^{N} {[0, \gab_n(\gamma, \nu, \theta, h)]}
\end{equation}
such that 
\[
\Geo \mid_{\Rec(\zeta)} \colon \Rec(\zeta) \to \Rec(\zeta)
\]
is a homeomorphism for every $\zeta\in\cX_b$.

\begin{defn}
\label{defn:rectGeoXb}
Consider a ramp system $\dot{x} = -\Gamma x + E(x; \nu, \theta, h)$ for fixed parameter values $(\gamma, \nu, \theta, h) \in \Lambda(R)$.
An \emph{associated rectangular geometrization} $ \recG(\gamma,\nu,\theta,h)$ is defined as follows.
Fix $\Geo \in \cG$. Given $\zeta = [\bv,\bw]\in \cX_b$, define $g_\zeta \colon K(\bw) \to \prod_{n = 1}^{N} {[0, \gab_n(\gamma, \nu, \theta, h)]}$  by
\[
g_\zeta =  \Geo\circ a_\zeta
\]
and set 
\[
\recG(\gamma,\nu,\theta,h) := \setof{ g_\zeta \mid \zeta \in \cX_b }.
\]
\end{defn}

For most of this text we work with ramp systems at fixed parameter values and thus for simplicity we let $\recG = \recG(\gamma,\nu,\theta,h)$.
We remind the reader that for fixed $\recG$, given $\xi\in \cX_b^{(n)}$,
\[
\bg(\xi) := g_\xi(B^n)
\]
where $g_\xi\in\recG$.

\begin{rem}
    \label{rem:g(b(xi))}
Fix $\recG$. Given $\xi = [\bv,\bw] \in \cX$,
\[
    \bg(\blup(\xi)) = \bg([2\bv+\bw,\bone]) = \prod_{n=1}^N I_n([2\bv+\bw,\bone]) = \prod_{n=1}^N [e_n(2\bv+\bw),e_n(2\bv+\bw+\bzero^{(n)})].
\]
The expression on the right-hand side can be further simplified by verifying whether $w_n=0$ or $w_n=1$, i.e., if $n \in J_i(\xi)$ or $n \in J_e(\xi)$. For each $n \in \setof{1,\ldots,N}$, one obtains
\[
I_n(\blup(\xi)) = \begin{cases}
    [\theta_{m_k,n,j_k}+h_{m_k,n,j_k},\theta_{m_{k+1},n,j_{k+1}}-h_{m_{k+1},n,j_{k+1}}], & n \in J_e(\xi), \\
    [\theta_{m_k,n,j_k}-h_{m_k,n,j_k},\theta_{m_k,n,j_k}+h_{m_k,n,j_k}], & n \in J_i(\xi),
\end{cases}
\]
where $(m_k,j_k) \in \mathcal{J}(n)$ is the indexing introduced in Definition~\ref{defn:admissiblehtheta}. 
The length of the interval $I_n(b(\xi))$ is denoted by
\begin{equation}
\label{eq:Xi}
\Xi_n(\xi) = \begin{cases}
\left( \theta_{m_{k+1},n,j_{k+1}}-h_{m_{k+1},n,j_{k+1}} \right) - \left(\theta_{m_k,n,j_k}+h_{m_k,n,j_k},\right), & n \in J_e(\xi), \\
2 h_{m_k,n,j_k}, & n \in J_i(\xi),
\end{cases}
\end{equation}
\end{rem}

\chapter{Global Dynamics Derived from $\cF_0$}
\label{sec:R0}

As indicated in the introduction we include this section purely for pedagogical reasons. This is the most trivial application of Corollary~\ref{cor:existenceAttractorLattice} and Theorem~\ref{thm:dynamics} but nevertheless  the essential steps of the process of going from combinatorics to nonlinear dynamics are identified.

Consider a ramp system 
 \begin{equation}
    \label{eq:rampH0}
    \dot{x} = -\Gamma x + E(x; \nu, \theta, h)
    \end{equation}
with fixed parameter values $(\gamma, \nu, \theta) \in \Lambda(S)$ and $h \in \cH_0(\gamma, \nu, \theta)$, where $\Lambda(S)$ and $\cH_0(\gamma, \nu, \theta)$ are defined by \eqref{eq:lambdaS} and \eqref{eq:H0}, respectively. 
Let $\cX$ be the associated ramp induced cubical complex (see Definition~\ref{defn:ramp-wall-labeling}). 
The map $\cF_0 \colon \cX\mvmap \cX$ is independent of the wall labeling. As discussed in Example~\ref{ex:trivialSCC}, $\sInvset^+(\cF_0) = \setof{\emptyset, \cX}$. Thus, $\CRC(\cF_0) = \sMG(\cF_0) = \SCC(\cF_0)$ consists of a single element $p$. Furthermore,
\[
CH_k(p; \F) \cong H_k(\cX_b, \emptyset; \F) \cong
\begin{cases}
\F & \text{if $k=0$} \\
0 & \text{otherwise.}
\end{cases}
\]
Therefore, the chains for the Conley complex are $CH_0(p;\F)\cong\F$ with a trivial boundary map.

Set $\sN = \setof{\emptyset,\cX_b}$ and observe that $\sN$ is an AB-lattice (see Definition~\ref{defn:ABlattice}) with respect to $\cX_b$.
Choose an associated rectangular geometrization $\recG = \recG(\gamma,\nu,\theta,h)$ (see Definition~\ref{defn:rectGeoXb}).
By Proposition~\ref{prop:globalAttractor}, $\recG$ is aligned with $-\Gamma x + E(x;\nu,\theta,h)$ (see Definition~\ref{defn:aligned}).
Therefore, we can apply Theorem~\ref{thm:dynamics} and conclude that the Conley index of the associated global attractor $K$ is $CH_*(p)$.
Furthermore, by \cite{mccord:89} there exists a fixed point in $K$.

The points we wish to emphasize are as follows.
We begin with a specific ODE \eqref{eq:rampH0}.
Chapter~\ref{sec:ramp} prescribes a combinatorial model $\cF_0\colon \cX\mvmap \cX$. 
The results of Chapter~\ref{sec:lattice} identify and justify the combinatorial and homology computations.
An appropriate geometrization is provided by Chapter~\ref{sec:geometrization01}.
Therefore, the D-grading gives rise to a Morse decomposition and the Conley complex is a connection matrix for the dynamics of \eqref{eq:rampH0}.
Finally, we use a classical result involving the Conley index to rigorously identify an invariant set with a particular structure.

\chapter{Global Dynamics Derived from $\cF_1$}
\label{sec:R1Dynamics}

The standing hypothesis throughout this chapter is the following:
\begin{description}
    \item[H1] Consider an $N$-dimensional ramp system given by
    \begin{equation}
    \label{eq:rampH1}
    \dot{x} = -\Gamma x + E(x; \nu, \theta, h)
    \end{equation}
    with parameters $(\gamma, \nu, \theta)\in \Lambda(S)$ (see \eqref{eq:lambdaS}) and $h \in \cH_1(\gamma, \nu, \theta)$ (see Definition~\ref{defn:H1}). Let $\cX=\cX(\I)$ be the ramp induced cubical complex, $\omega \colon W(\cX) \to \setof{\pm 1}$ be the associated wall labeling and $\rook : TP(\cX) \to \setof{0,\pm 1}$ be the associated Rook Field (see Section~\ref{sec:ramp2rook}). Let $\cX_b$ be the associated blow-up complex (see Section~\ref{sec:blowup_complex}). Let $\cF_1 \colon \cX \mvmap \cX$ be the associated combinatorial multivalued map derived from Definition~\ref{def:Rule1} and let $\pi_b\colon (\cX_b,\preceq_b) \to (\SCC(\cF_1),\preceq_{\cF_1})$ be the D-grading as in Section~\ref{sec:CCalgorithm}. 
\end{description}

The goal of this chapter is to prove the following theorem. 
\begin{thm}
    \label{thm:R1ABlattice}
    Given the hypothesis {\bf H1}, any rectangular geometrization $\recG$ of $\cX_b$ is aligned with the vector field \eqref{eq:rampH1} for all $\cN \in \sN(\cF_1)$. 
\end{thm}
As a consequence, Theorem~\ref{thm:dynamics} implies that we have a characterization of the global dynamics of the ramp system.

In Section~\ref{sec:F1-transversality} we provide the necessary results to prove Theorem~\ref{thm:R1ABlattice}. In Section~\ref{sec:F1-results} we provide a proof and more refined results about the dynamics of \eqref{eq:rampH1}. 

\section{Transversality of $\cF_1$}
\label{sec:F1-transversality}

Recall that $\cX = \cX(\I)$ is the cubical cell complex (see Definition~\ref{defn:Xcomplex}) generated by
\[
\I = \prod_{n=1}^N \setof{0, \ldots, K(n) + 1}
\]
where $K(n)$ in the number of output thresholds for the variable $x_n$ (see \eqref{eq:num_out_thresholds}), and $\cX_b = \cX(\bar{\I})$ is the cubical cell complex (see Definition~\ref{defn:Xbcomplex}) generated by
\[
\bar{\I} = \prod_{n=1}^N \setof{0, \ldots, 2(K(n)+1) + 1},
\]
and the blowup map is the function $b : \cX \to \cX_b^{(N)}$ given by $\blup([\bv,\bw])=[2\bv+\bw,\bone]$ as in Section~\ref{sec:blowup_complex}.

Using the indexing set $\mathcal{J}(n)$ (see Definition~\ref{defn:admissiblehtheta})
\[
\theta_{m_k, n, j_k} + h_{m_k, n, j_k} < \theta_{m_{k+1}, n, j_{k+1}} - h_{m_{k+1}, n, j_{k+1}},
\]
for all $k = 1, \ldots, K(n) - 1$ and all $n=1,\ldots,N$. 
We assume that $\theta_{m_0,n,j_0}$ and $\theta_{m_{K(n)+1},n,j_{K(n)+1}}$ satisfy \eqref{eq:m0nj0} and \eqref{eq:mKnjK}, respectively.

The first step towards proving that a rectangular geometrization is aligned with the vector field is to identify the region in which the geometrization of two cells $\xi,\xi'\in \cX$ with $\xi \prec \xi'$ intersect.

\begin{prop}
\label{prop:intersection-of-geometrization-of-blup}
Given {\bf H1}, let $\recG$ be a rectangular geometrization of $\cX_b$.  
Let $\xi,\xi'\in\cX$ with 
$\xi\prec\xi'$
and $\xi=[\bv,\bw]$ with $\bv=(k_1,\ldots,k_N)$.  
Let $x = (x_1,\ldots, x_N) \in \bg(\blup(\xi)) \cap \bg(\blup(\xi'))$. 
If $n \in \Ex(\xi,\xi')$, then
    \[
        x_n = \theta_{m_{k_n}, n, j_{k_n}}-p_n(\xi,\xi') h_{m_{k_n}, n, j_{k_n}}.
    \]
\end{prop}
\begin{proof}
    Let $\xi = [\bv,\bw]$ and $\xi'=[\bv',\bw']$. If $n \in \Ex(\xi,\xi')$ then $n \in J_e(\xi)$ and $n \in J_i(\xi')$ by Definition~\ref{defn:Ex}. In terms of $\bv,\bv',\bw,\bw'$, Definition~\ref{defn:Xcomplex} implies that $\bw_n=0$ and $\bw'_n=1$ while $\bv_n=\bv_n'$ or $\bv_n=\bv_n'+1$.
    
    Assume that $\bv_n=\bv_n'$ so that $p_n(\xi,\xi')=-1$ by Definition~\ref{defn:rpvector}.
    Then, by Definition~\ref{defn:Xbcomplex}, the blowup map $b:\cX \to \cX^{(N)}_b$ maps $\xi$ and $\xi' $ to 
    \[
        \blup(\xi) = [2\bv+\bw,\bone],\quad\text{and} \quad \blup(\xi') = \blup([\bv,\bw]) = [2\bv'+\bw',\bone],
    \]
    respectively, with $2\bv_n+\bw_n=2\bv_n$ even and $2\bv'_n+\bw'_n=2\bv_n+1$ odd. 
    
    Now, recall \eqref{eq:Rxi},
    \[
    \Rec([2\bv+\bw,\bone]) = \prod_{n = 1}^{N} I_n([2\bv+\bw,\bone]) = \prod_{n=1}^ N [e_n(2\bv+\bw), e_n(2\bv+\bw + \bzero^{(n)})], 
    \]
    so if $x = (x_1,\ldots, x_N)  \in \bg(\blup(\xi))$, then 
\begin{align*}
    x_n  \in & [e_n(2\bv+\bw+\bzero),e_n(2\bv+\bw+\bone^{(n)})]  \\
    &  = [\theta_{m_{k_n}, n, j_{k_n}}-h_{m_{k_n}, n, j_{k_n}},\theta_{m_{k_n}, n, j_{k_n}}+h_{m_{k_n}, n, j_{k_n}}],
\end{align*}
 with indexing given by $\bv=(k_1,\ldots,k_N)$ as in Definition~\ref{defn:admissiblehtheta}. 
Similarly, if $x \in \bg(\blup(\xi'))$, then 
\begin{align*}
    x_n  \in & [e_n(2\bv+\bw),e_n(2\bv+\bw+\bone^{(n)}))] \\
    &  = [\theta_{m_{k_n}, n, j_{k_n}}+h_{m_{k_n}, n, j_{k_n}},\theta_{m_{k_n+1}, n, j_{k_n+1}}-h_{m_{k_n+1}, n, j_{k_n+1}}].
\end{align*}
    Thus, 
    \begin{align*}
        x_n & = \theta_{m_{k_n}, n, j_{k_n}}+h_{m_{k_n}, n, j_{k_n}} \\ 
        & = \theta_{m_{k_n}, n, j_{k_n}}-(-1) h_{m_{k_n}, n, j_{k_n}}\\
        & = \theta_{m_{k_n}, n, j_{k_n}}-p_n(\xi,\xi') h_{m_{k_n}, n, j_{k_n}}.
    \end{align*}
    
If one assumes $\bv_n=\bv_n'+1$, or equivalently that  $\bv'_n=\bv_n-1$,
then $p_n(\xi,\xi')=1$. If $x \in \bg(\blup(\xi'))$, then 
\begin{align*}
    x_n \in & [e_n(2\bv+\bw),e_n(2\bv+\bw+\bone)]   \\
    & = [\theta_{m_{k_n-1}, n, j_{k_n-1}}+h_{m_{k_n-1}, n, j_{k_n-1}},\theta_{m_{k_n}, n, j_{k_n}}-h_{m_{k_n}, n, j_{k_n}}],
\end{align*}
and hence, $x_n = \theta_{m_{k_n}, n, j_{k_n}}-p_n(\xi,\xi')h_{m_{k_n}, n, j_{k_n}}$. 
\end{proof}
The next result provides an identification of the sign of the vector field on $\bg(\blup(\xi))\cap\bg(\blup(\xi'))$ using the multivalued map $\cF_1 \colon \cX \mvmap \cX$. 
\begin{prop}
\label{prop:transversality-for-entrance-exit-faces}
Assume hypothesis \textbf{H1} and consider any rectangular geometrization $\recG$ of $\cX_b$. 
Let $\xi, \xi' \in \cX$ be such that $\xi \prec \xi'$ with $\Ex(\xi,\xi') = \{n\}$. 
For $x = (x_1, \ldots, x_N) \in \bg(\blup(\xi)) \cap \bg(\blup(\xi'))$, the following holds:
\begin{itemize}
    \item[(i)] If $\xi \notin \cF_1(\xi')$, then $\sgn(\dot{x}_n) = -p_n(\xi, \xi')$.
    \item[(ii)] If $\xi' \notin \cF_1(\xi)$, then $\sgn(\dot{x}_n) = p_n(\xi, \xi')$.
\end{itemize}
\end{prop}
\begin{proof}
    We prove that if $\xi \notin \cF_1(\xi')$, then $\sgn(\dot{x}_n) = -p_n(\xi,\xi')$.

    Assume that $\xi \notin \cF_1(\xi')$, so by Definitions~\ref{def:Rule1.1} and~\ref{def:Rule1}, we have that $\xi \in E^+(\xi')$. Thus, by Definition~\ref{def:exit_face}, $\rook_n(\xi,\mu) = p_n(\xi,\xi')$ for all $\mu \in \Top_\cX(\xi')$.

    By Proposition~\ref{prop:mu*}, for each $\mu \in \Top_\cX(\xi')$, there exists a unique $\mu_n^* \in \cX^{(N-1)}$, with $\xi \preceq \mu_n^*$, such that
    \[
        \rook_n(\xi,\mu) = \rook_n(\mu_n^*,\mu).
    \]
    From Definition~\ref{def:rookfield} of Rook Fields, we obtain that
    \[
        \omega(\mu_n^*,\mu) = p_n(\xi,\xi') \quad \text{for all $(\mu_n^*,\mu) \in W(\xi)$}.
    \]
    From Definition~\ref{defn:ramp-wall-labeling}, the associated wall labeling is given by
    \[
        \omega(\mu_n^*,\mu) = \sgn(-\gamma_n \theta_{m_{k_n}, n, j_{k_n}} + E_n(\mu)),
    \]
    for appropriate indices $m_{k_n}$ and $j_{k_n}$ as in equation~\eqref{eq:Jn}.    Considering $(\gamma,\nu,\theta) \in \Lambda(S)$ and $h \in \cH_1(\gamma,\nu,\theta)$, Definition~\ref{defn:H1} ensures
    \[
        -\gamma_n (\theta_{m_{k_n}, n, j_{k_n}} - p_n(\xi,\xi') h_{m_{k_n}, n, j_{k_n}}) + E_n(\mu) \neq 0,
    \]
    for all $\mu \in \Top_\cX(\xi')$. Therefore,
    \[
        \sgn(-\gamma_n (\theta_{m_{k_n}, n, j_{k_n}} - p_n(\xi,\xi') h_{m_{k_n}, n, j_{k_n}}) + E_n(\mu)) = -p_n(\xi,\xi').
    \]
    Since $E_n(x;\nu,\theta,h)$ is a ramp nonlinearity, for any $x \in \bg(\blup(\xi)) \cap \bg(\blup(\xi'))$, we have
    \[
        E_n(\mu_{\min}) \leq E_n(x;\nu,\theta,h) \leq E_n(\mu_{\max}),
    \]
    for some $\mu_{\min}, \mu_{\max} \in \Top_\cX(\xi')$. Thus,
    \[
        \sgn(-\gamma_n (\theta_{m_{k_n}, n, j_{k_n}} - p_n(\xi,\xi') h_{m_{k_n}, n, j_{k_n}}) + E_n(x;\nu,\theta,h)) = -p_n(\xi,\xi'),
    \]
    for all such $x$.
    
Finally, by Proposition~\ref{prop:intersection-of-geometrization-of-blup}, if $x = (x_1, \ldots, x_N) \in \bg(\blup(\xi)) \cap \bg(\blup(\xi'))$, then
    \[
        x_n = \theta_{m_{k_n}, n, j_{k_n}} - p_n(\xi,\xi') h_{m_{k_n}, n, j_{k_n}}.
    \]
Therefore, $\sgn(\dot{x}_n) = -p_n(\xi,\xi')$, for all $x \in \bg(\blup(\xi)) \cap \bg(\blup(\xi'))$.

    The proof for the case when $\xi' \notin \cF_1(\xi)$ follows analogously by interchanging the roles of $\xi$ and $\xi'$.
\end{proof}

Proposition~\ref{prop:transversality-for-entrance-exit-faces} identifies the sign of the vector field in the interior of $\recG(\cX_b)$, but not if the cell is in the boundary of $\cX_b$. 
The following proposition addresses this issue. 
\begin{prop}
\label{prop:transversality-at-the-boundary}
Assume {\bf H1} and a rectangular geometrization $\recG$ of $\cX_b$. 
Let $\zeta \in \cX_b^{(N-1)}\cap \bbdy(\cX_b)$  and let $J_i(\zeta)=\setof{n}$.
Then, for any $x \in \bg(\zeta)$, there exists an unique cell $\xi \in \bbdy(\cX)$ such that $\Top_{\cX_b}(\zeta)=\setof{\blup(\xi)}$, and for any $n$-wall $(\mu_n^*,\mu) \in W(\xi)$
    \[
        \sgn(\dot{x}_n) =-p_n(\zeta,\blup(\xi))=-p_n(\mu_n^*,\mu),
    \]
    for all $x = (x_1,\ldots, x_N)  \in \bg(\zeta)$.
\end{prop}
\begin{proof}
    If $\zeta \in \cX_b^{(N-1)}$ and $\zeta \in \bbdy(\cX_b)$, then by Definition~\ref{def:boundary_prime}, there exists an unique $\zeta' \in \cX_b$ such that $\zeta \prec \zeta'$ with $J_i(\zeta)=\setof{n}$ and $J_e(\zeta)=\setof{1,\ldots,N}\setminus\setof{n}$. By Definition~\ref{defn:Xbcomplex}, the set of top cells of $\cX_b$, $\cX_b^{(N)}$, is in 1-1 correspondence with the cells of $\cX$. Let $\xi=\blup^{-1}(\zeta')$ with $\blup(\xi)=[2\bv+\bw,\bone]$. Given that $\zeta \in \bbdy(\cX_b)$, either $2\bv_n+\bw_n=0$ or $2\bv_n+\bw_n=2(K(n)+1)+1$. 

Assume $2\bv_n+\bw_n=0$, so that $\bv_n=\bw_n=0$ implies $\xi \in \bbdy(\cX)$ and $n \in J_i(\xi)$. Let $\mu \in \Top_\cX(\xi)$. 
By Proposition~\ref{cor:mu*} and Corollary~\ref{cor:mu*}, there exists an unique $n$-wall $(\mu_n^*,\mu) \in W(\cX)$ such that $\rook_n(\xi,\mu) = \omega(\mu_n^*,\mu)$.
By Proposition~\ref{prop:ramp-dissipativewall}, $\omega(\mu_n^*,\mu)=-p_n(\mu_n^*,\mu)$. Moreover, since $\bv_n=0$, $p_n(\mu_n^*,\mu)=-1$. Thus, $\omega(\mu_n^*,\mu)=1$. Hence, by Definition~\ref{defn:ramp-wall-labeling},
\[
    \sgn(-\gamma_n 0 + E_n(\mu) )) = \sgn( E_n(\mu) ) = 1.
\]    
Note that for $x = (x_1,\ldots, x_N)  \in \bg(\zeta)$, $2\bv_n+\bw_n=0$ and $J_i(\zeta)=\setof{n}$ implies that $x_n = 0$. 
Also, $J_e(\xi)=\setof{1,\ldots,N}\setminus\setof{n}$ implies that for $i\neq n$, $x_i \in [\theta_{m_k,i,j_k}+h_{m_k,i,j_k},\theta_{m_k+1,i,j_k+1}-h_{m_k+1,i,j_k+1}]$. 
Thus, by Definition~\ref{defn:rampnonlinearity}, $E_n(x;\nu,\theta,h)$ is constant for $x \in \bg(\zeta)$. 
Therefore, $\dot{x}_n > 0$ for all $x \in \bg(\zeta)$. 
Since $2\bv_n+\bw_n=0$, $p_n(\zeta,\blup(\xi))=-1$, hence
$\sgn(\dot{x}_n) =-p_n(\zeta,\blup(\xi))$. 

 Similarly, if one assumes $2\bv_n+\bw_n+1=2(K(n)+1)+1$, then $2\bv_n+\bw_n=2K(n)$, $\bv_n=K(n)$ and $\bw_n=0$. Now,
    $x_n =\gab_n(\gamma, \nu, \theta, h)$ as in \eqref{eq:evertices}, 
    while $x_i$ does not affect $E_n$ for $i \neq n$. Again, since $\omega$ is strongly dissipative and $h \in \cH_1(\gamma,\nu,\theta)$, one obtains that $\dot{x}_n < 0$ for $x \in \bg(\zeta)$ and since $p_n(\zeta,\blup(\xi))=1$, the result follows.
\end{proof}

\section{Dynamics of $\cF_1$}
\label{sec:F1-results}

We are now in position to prove Theorem~\ref{thm:R1ABlattice}. 

\begin{proof}[Proof of Theorem~\ref{thm:R1ABlattice}]
Fix a rectangular geometrization $\recG$ of $\cX_b$ and let $\cN\in\sN(\cF_1)$.
Let $\zeta\in \bbdy(\cN)^{(N-1)}$.
Since $\recG$ is a rectangular geometrization of $\cX_b$ into $\R^N$, $\bg(\zeta)$ is an $N-1$ dimensional rectangular hyperplane in $B = [0,\gab(\gamma,\nu)^N]$, and therefore, is a smooth $(N-1)$ dimensional manifold with boundary. 

There are two possible cases, either $\zeta \in \bbdy(\cX_b)$ or $\zeta\not\in \bbdy(\cX_b)$. 

If $\zeta \in \bbdy(\cX_b)$, then $z_\zeta(x)=-p(\zeta,\blup(\xi))$, where $\Top_{\cX_b}(\zeta)=\setof{\blup(\xi)}$.
By Proposition~\ref{prop:transversality-at-the-boundary}, $\sgn(\dot{x}_n) = -p_n(\zeta,\blup(\xi))$ for all $x \in \bg(\zeta)$, thus
\[
    \langle \sgn(\dot{x}) , z_\zeta(x) \rangle = \left( p_n(\zeta,\blup(\xi)) \right)^2  > 0,
\]
and $z_\zeta(x) = \sgn(\dot{x})$. 

If $\zeta\not\in \bbdy(\cX_b)$, then there exist $\sigma_0,\sigma_1\in \cX_b^{(N)}$, such that $\zeta\preceq_b \sigma_i$, $\sigma_0\in\cN$, and $\sigma_1\not\in\cN$. Let $\xi_i = \blup^{-1}(\sigma_i)$.
By Lemma~\ref{lem:codimension1face}, $|\dim(\xi_0) - \dim(\xi_1)| = 1$, and by Lemma~\ref{lem:boundaryABlattice}, $\xi_0 \prec_{\bar{\cF_1}} \xi_1$. Since $\sigma_1 \notin \cN$, it follows that $\xi_1 \notin \cF_1(\xi_0)$. Then by Proposition~\ref{prop:transversality-for-entrance-exit-faces}, $\sgn(\dot{x}_n)=p_n(\xi_0,\xi_1)$ for all $x \in \bg(\sigma_0)\cap\bg(\sigma_1)$. Note that $\bg(\sigma_0)\cap\bg(\sigma_1)=\bg(\zeta)$ and that
\[
    z_\zeta(x)=p(\zeta,\sigma_0) = p(\xi_0,\xi_1)=\pm \bzero^{(n)}. 
\]
Thus, 
$\langle \sgn(\dot{x}) , z_\zeta(x) \rangle = (p_n(\xi_0,\xi_1))^2 > 0$.
If instead $\xi_1 \preceq \xi_0$, then by the same arguments 
$\langle \sgn(\dot{x}) , z_\zeta(x) \rangle = (-p_n(\xi_1,\xi_0))^2$. 

In either case, $z_\zeta(x) = \sgn(\dot{x})$ and
\[
    \langle \dot{x} , z_\zeta(x) \rangle = |\dot{x}_n| \geq \inf_{x \in \bg(\zeta)} |\dot{x}_n| > 0,
\]
where the last inequality follows from the choice of $h \in \cH_1(\gamma,\theta,\nu)$ in Definition~\ref{defn:H1}. Therefore, $\recG$ is aligned with the ramp system over $\cN$ for every $(N-1)$-dimensional cell $\zeta\in\bdy'(\cN)$.
\end{proof}

Having established Theorem~\ref{thm:R1ABlattice}, we turn to Corallary~\ref{cor:existenceAttractorLattice}.
Observe that $\sN$ in Corallary~\ref{cor:existenceAttractorLattice} can be replace by $\sN(\cF_1)$ and as a consequence we have Theorem~\ref{thm:dynamics} at our disposal where $\varphi$ is the flow associated with any ramp system \eqref{eq:rampH1} satisfying {\bf H1}.

In particular by Theorem~\ref{thm:dynamics}(2) we have computed a Morse decomposition for $\varphi$.
However, for a fixed parameter value the Morse representation may provide finer information.
Thus, we are interested in understanding general conditions to define a poset $(\sP,\preceq_{\bar{\cF}_1})$ such that $\sQ\subseteq \sP \subseteq \SCC(\cF_1)$ and under which $M(p)= \emptyset$ for all $p\in \SCC(\cF_1)\setminus\sP$.

\begin{lem}
\label{lem:gradient-direction-is-nonzero}
Let $\xi \in \cX$.
If $n \in G(\xi)$, then
\[
\left( -\gamma_n x_n + E_n(x;\nu,\theta,h) \right) r_n > 0, \ \text{for all $x = (x_1, \ldots, x_N)  \in \bg(\blup(\xi))$}.
\]
where $\setof{ r_n } = R_{n}(\xi)$.
\end{lem}
\begin{proof}
Set $f_n(x)=(-\gamma_n x_n + E_n(x;\nu,\theta,h))r_n$. 
Since $h \in \cH_1(\gamma,\nu,\theta)$, 
    \[
        \sgn(-\gamma_{n} x_{n} + E_{n}(\mu)) = r_n, \ \forall x_{n} \in I_n(\xi), \forall \mu \in \Top_\cX(\xi).
    \]
   Thus, $f_n(x) > 0$ for all $x \in \partial \bg(\blup(\xi))$. Note that $\partial^2 f_n/\partial x_i^2 = 0$ for all $i=1,\ldots,N$ as $E_n$ is either a type I or type II  ramp nonlinearity. Hence, by the maximum principle, $f_n$ attains its minimum at the boundary, so
   \[
        f_n( \bg(\blup(\xi)) )  \geq f_n (\partial \bg(\blup(\xi))) > 0.
   \]
\end{proof}

\begin{thm}
\label{thm:regular_cell}
Given {\bf H1} and an associated rectangular geometrization of $\cX_b$, 
let $\cS\subset \cX$ and assume that the elements of $\cS$ share a common gradient direction (see Definition~\ref{defn:commonGradient}).
Then, 
\[
\Inv\left( \bigcup_{\xi\in\cS}\bg(\blup(\xi)),\varphi\right) = \emptyset.
\]
In particular, if  $\xi\in\cX$ is not an equilibrium cell, then
\[
\Inv(\bg(\blup(\xi)),\varphi) = \emptyset.
\]
\end{thm}

\begin{proof}
For each $\xi\in \cS$ Lemma~\ref{lem:gradient-direction-is-nonzero} implies that
\[
    \left|\dot{x}_{n_g}\right| 
    \geq \inf_{x \in \bg(\blup(\xi))} \left| -\gamma_{n_g} x_{n_g} + E_{n_g}(x;\nu,\theta,h) \right| > 0.
\]
The inequality is strict and $\cS$ is finite, thus $x_{n_g}$ acts as a Lyapunov function over $\bigcup_{\xi\in\cX}\bg(\blup({\xi}))$.
Therefore, $\Inv\left(\bigcup_{\xi\in\cX}\bg(\blup({\xi})),\varphi\right) = \emptyset$.

By definition if $\xi$ is not an equilibrium cell, then $G(\xi)\neq \emptyset$, which is a special case of this Theorem.
\end{proof}

\begin{defn}
Consider a wall labeling $\omega\colon W(\cX) \to \setof{\pm1}$.
Let $\cF\colon \cX\mvmap \cX$ be a multivalued map with D-grading $\pi\colon \cX \to (\SCC(\cF),\preceq_{\bar{\cF}})$.
An element $p\in \SCC(\cF)$ is \emph{spurious} if $p\in\bar{\cF}(p)$ and the elements of $\pi^{-1}(p)$ share a common gradient direction.
\end{defn} 

\begin{prop}
    \label{prop:F1.3}
Given {\bf H1} and an associated rectangular geometrization of $\cX_b$,
let $\pi\colon \cX\to \SCC(\cF_1)$ denote the D-grading.
If $p\in \SCC(\cF_1)$ and 
$p\not\in \bar{\cF}_1(p)$,
then 
\[
\Inv(\bg(\blup(\pi^{-1}(p))) =\emptyset.
\]
\end{prop}
\begin{proof}
By assumption 
$p\not\in \bar{\cF}_1(p)$.
Therefore there exists unique cell $\xi$ such that $\pi^{-1}(p) =\setof{\xi}$. 
The assumption 
$p\not\in \bar{\cF}_1(p)$
also implies that $\xi\not\in\cF_1(\xi)$, i.e., Condition 1.1 is satisfied. 
This implies that $G(\xi) \neq \emptyset$.
Therefore the result follows from Theorem~\ref{thm:regular_cell}.
\end{proof}

\begin{prop}
\label{cor:F1.3}
Given {\bf H1} and an associated rectangular geometrization of $\cX_b$, let $\pi\colon \cX\to \SCC(\cF_1)$ denote the D-graph.
Define 
\[
\sP = \SCC(\cF_1)\setminus\left( \setof{p\not\in \bar{\cF}_1(p)}
\cup 
\setof{p\ \text{is spurious}}\right)
\]
Then, $(\sP,\preceq_{\bar{\cF}_1})$ is a Morse decomposition of $\varphi$ restricted to $B$.
\end{prop}

\begin{proof}
Since $(\SCC(\cF_1),\preceq_{\bar{\cF}_1})$ is a Morse decomposition of $\varphi$ restricted to $B$ it is sufficent to show that if $p\in \SCC(\cF_1)\setminus \sP$, then $\Inv\left(\bigcup_{\xi\in\cX}\bg(\blup({\pi^{-1}(p)})),\varphi\right) = \emptyset$.
The result follows from Proposition~\ref{prop:F1.3} and Theorem~\ref{thm:regular_cell}.     
\end{proof}

We conclude this section with a justification for the equilibrium cell (Definition~\ref{def:eqcell}) nomenclature.
\begin{prop}\label{prop:equilibrium-existence}
Consider {\bf H1} and an associated rectangular geometrization of $\cX_b$. 
Then, $\xi \in \cX$ is an equilibrium cell if and only if \eqref{eq:rampH1} contains an unique equilibrium in $\bg(\blup(\xi))$.
\end{prop}
\begin{proof}
Let $\xi=[\bv,\bw] \in \cX$ be an equilibrium cell, i.e., $G(\xi) = \emptyset$.  
Furthermore, set $\bv=(k_1,\ldots,k_N)$.    
    
If $\xi \in \cX^{(N)}$, then \eqref{eq:rampH1} restricted to $\bg(\blup(\xi))$ reduces to 
\begin{equation*}
        \dot{x} = -\Gamma x + C,
\end{equation*}
where $C = E(\xi)$ as in Definition~\ref{defn:ramp-wall-labeling}. 
Note that for any $n \in \setof{1,\ldots,N}$, 
\[
R_n(\xi) = \setof{\rook_n(\xi^-_n,\xi),\rook_n(\xi^+_n,\xi)} = \setof{\omega(\xi^-_n,\xi),\omega(\xi^+_n,\xi)},
\]
where $\rook$ is the rook field (see Definition~\ref{def:rookfield}) for the ramp induced wall labeling (see Definition~\ref{defn:ramp-wall-labeling}).   
If $\omega(\xi^-_n,\xi)=\omega(\xi^+_n,\xi)$, then $G(\xi)\neq\emptyset$. Hence, the contrapositive yields 
\[
\omega(\xi^-_n,\xi) = \sgn( -\gamma_n \theta_{m_{k_n}, n, j_{k_n}} + E_n(\xi) )\neq \sgn( -\gamma_n \theta_{m_{k_n+1}, n, j_{k_n+1}} + E_n(\xi) ) = \omega(\xi^+_n,\xi)
\]
for each $n \in \setof{1,\ldots,N}$. 
Therefore,
\begin{align*}
& \sgn\left( -\gamma_n (\theta_{m_{k_n}, n, j_{k_n}}+h_{m_{k_n}, n, j_{k_n}}) + E_n(\xi) \right) \neq \\
& \sgn\left( -\gamma_n (\theta_{m_{k_n+1}, n, j_{k_n+1}}-h_{m_{k_n+1}, n, j_{k_n+1}}) + E_n(\xi) \right) 
\end{align*}
since $h \in \cH_1(\gamma,\nu,\theta)$. We have concluded that \eqref{eq:rampH1} assumes opposite signs at opposite faces of the cube $\bg(\blup(\xi))$, therefore the Poincaré-Miranda Theorem yields the existence of an equilibrium point in $\bg(\blup(\xi))$. Since \eqref{eq:rampH1} is linear and $h \in \cH_1(\gamma,\nu,\theta)$, that point is unique. 

    The case when $\xi$ is not a top cell follows in a similar manner. Assume that $\xi \notin \cX^{(N)}$. Then by Proposition~\ref{prop:eqcell-is-opaque} $\xi$ is opaque. Therefore, Proposition~\ref{prop:opaque} implies that $\rmap\xi \colon J_i(\xi) \to J_i(\xi)$ is a bijection, so \eqref{eq:rampH1} restricted to $\bg(\blup(\xi))$ is given by 
    \begin{align*}
        \dot{x}_n = \begin{cases}
        -\gamma_n x_n + f_n\left(x_{\rmap\xi^{-1}(n)}\right), & n \in J_i(\xi), \\ 
        -\gamma_n x_n + c_n, & n \in J_e(\xi),
        \end{cases} 
    \end{align*}
    where
    \begin{align*}
        f_n\left(x_{\rmap\xi^{-1}(n)}\right) & = E_n\left(\mdpt_1(\xi),\ldots,x_{\rmap\xi^{-1}(n)},\ldots,\mdpt_n(\xi));\nu,\theta,h\right) \\ 
        c_n & = E_n\left( \mdpt(\xi); \nu,\theta,h\right),
    \end{align*}
    with $\mdpt(\xi)=(\mdpt_n(\xi))_n$. 
    Moreover, Proposition~\ref{defn:opaque} implies that $J_i(\xi) = O(\xi)$ and $J_e(\xi)=N_0(\xi)$. For each $n \in \activeset(\xi)$ (see Definition~\ref{def:active_regulation}), let
    \begin{align*}
        x^-_n & = \theta_{m_{k_n}, n, j_{k_n}}-h_{m_{k_n}, n, j_{k_n}} \\ 
        x^+_n & = \theta_{m_{k_n}, n, j_{k_n}}+h_{m_{k_n}, n, j_{k_n}}
    \end{align*}
    where $\bv=(k_1,\ldots,k_N)$ and $(m_{k_n},j_{k_n})$ as in \eqref{eq:Jn}. Note that 
    \[
        \sgn\left(-\gamma_n x_n + f_n\left(x^-_{\rmap\xi^{-1}(n)}\right)\right) \neq \sgn\left(-\gamma_n x_n + f_n\left(x^+_{\rmap\xi^{-1}(n)}\right)\right),
    \]
    otherwise $\sgn(\dot{x}_n)$ would be constant and $n \in G(\xi)$. If $n \in J_e(\xi)$, then $n \in N_0(\xi)$ and a similar argument to the previous case yields that $\dot{x_n}$ assumes opposite signs at opposing faces of the rectangular region $\bg(\blup(\xi))$. The result follows the Poincaré-Miranda Theorem, the linearity of \eqref{eq:rampH1} and the fact that $h \in \cH_1(\gamma,\nu,\theta)$. 

Conversely, let $\xi \in \cX$ and assume that there exists an unique equilibrium $x^* \in \bg(\blup(\xi))$. 
Suppose for the sake of contradiction that $G(\xi)\neq\emptyset$ and let $n \in G(\xi)$. 
Then, by Lemma~\ref{lem:gradient-direction-is-nonzero}, $\sgn(\dot{x}_n) \neq 0$ for all $x \in \bg(\blup(\xi))$, 
which contradicts the existence of $x^*$.
\end{proof}

\chapter{The Janus Complex}
\label{sec:janusComplex}

The results of Chapter~\ref{sec:R1Dynamics} are obtained using rectangular geometrizations to identify transversality along codimension-one faces of elements of $\sN(\cF_1)$, thus elucidating the fact that the local conditions used to define $\cF_1$ can be associated with regions of the phase space where the vector field is well aligned with the standard coordinate system of $\R^N$.
This is not the case for the quasilocal conditions that give rise to $\cF_2$.
Hence we  are forced to use geometrizations that involve curved hypersurfaces.
Based on the principle that it is easier to obtain explicit analytic estimates locally, our approach is to produce a finer cell complex $\Janus$, call the Janus complex, via subdivision of $\cX_b$, and then build geometrizations of $\Janus$.

\begin{defn}
    \label{defn:JanusComplex}
Given a cubical cell complex $\cX(\I) = (\cX,\prec,\dim,\kappa)$ the \emph{Janus complex} is the cubical complex 
\[
\cX\left(\bbarbI\right) = (\Janus,\precJanus,\dim,\kappa)
\]
generated by $\bbarbI = \prod_{n=1}^N \bbarI_n$ where
\[
\bbarI_n := \setof{k\mid k=0,\ldots, 8\bar{K}(n)},\quad n=1,\ldots, N.
\]
\end{defn}

\begin{ex}
\label{ex:simplexXr}
If $\I = \setof{0,1,2}^2$, then
$\bar{\I} = \setof{0,\ldots, 5}^2$ and $\bbarbI = \setof{k\mid k=0,\ldots, 40}^2$.
See Figure~\ref{fig:Xr}.
\end{ex}

\begin{figure}
\begin{tikzpicture}
[scale=0.3]

\draw[ultra thick, blue] (0,0) -- (16,0);
\draw[ultra thick, blue] (0,8) -- (16,8);
\draw[ultra thick, blue] (0,16) -- (16,16);
\draw[ultra thick, blue] (0,0) -- (0,16);
\draw[ultra thick, blue] (8,0) -- (8,16);
\draw[ultra thick, blue] (16,0) -- (16,16);

\draw[thick, red] (-2,-2) -- (18,-2);
\draw[thick, red] (-2,2) -- (18,2);
\draw[thick, red] (-2,6) -- (18,6);
\draw[thick, red] (-2,10) -- (18,10);
\draw[thick, red] (-2,14) -- (18,14);
\draw[thick, red] (-2,18) -- (18,18);
\draw[thick, red] (-2,-2) -- (-2,18);
\draw[thick, red] (2,-2) -- (2,18);
\draw[thick, red] (6,-2) -- (6,18);
\draw[thick, red] (10,-2) -- (10,18);
\draw[thick, red] (14,-2) -- (14,18);
\draw[thick, red] (18,-2) -- (18,18);

\draw (-2,-1.5) -- (18,-1.5);
\draw (-2,-1.0) -- (18,-1.0);
\draw (-2,-0.5) -- (18,-0.5);
\draw (-2,0) -- (18,0);
\draw (-2,0.5) -- (18,0.5);
\draw (-2,1.0) -- (18,1.0);
\draw (-2,1.5) -- (18,1.5);

\draw (-2,2.5) -- (18,2.5);
\draw (-2,3.0) -- (18,3.0);
\draw (-2,3.5) -- (18,3.5);
\draw (-2,4.0) -- (18,4.0);
\draw (-2,4.5) -- (18,4.5);
\draw (-2,5.0) -- (18,5.0);
\draw (-2,5.5) -- (18,5.5);

\draw (-2,6.5) -- (18,6.5);
\draw (-2,7.0) -- (18,7.0);
\draw (-2,7.5) -- (18,7.5);
\draw (-2,8.0) -- (18,8.0);
\draw (-2,8.5) -- (18,8.5);
\draw (-2,9.0) -- (18,9.0);
\draw (-2,9.5) -- (18,9.5);

\draw (-2,10.5) -- (18,10.5);
\draw (-2,11.0) -- (18,11.0);
\draw (-2,11.5) -- (18,11.5);
\draw (-2,12.0) -- (18,12.0);
\draw (-2,12.5) -- (18,12.5);
\draw (-2,13.0) -- (18,13.0);
\draw (-2,13.5) -- (18,13.5);

\draw (-2,14.5) -- (18,14.5);
\draw (-2,15.0) -- (18,15.0);
\draw (-2,15.5) -- (18,15.5);
\draw (-2,16.0) -- (18,16.0);
\draw (-2,16.5) -- (18,16.5);
\draw (-2,17.0) -- (18,17.0);
\draw (-2,17.5) -- (18,17.5);

\draw (-1.5,-2) -- (-1.5,18);
\draw (-1.0,-2) -- (-1.0,18);
\draw (-0.5,-2) -- (-0.5,18);
\draw (0,-2) -- (0,18);
\draw (0.5,-2) -- (0.5,18);
\draw (1.0,-2) -- (1.0,18);
\draw (1.5,-2) -- (1.5,18);

\draw (2.5,-2) -- (2.5,18);
\draw (3.0,-2) -- (3.0,18);
\draw (3.5,-2) -- (3.5,18);
\draw (4.0,-2) -- (4.0,18);
\draw (4.5,-2) -- (4.5,18);
\draw (5.0,-2) -- (5.0,18);
\draw (5.5,-2) -- (5.5,18);

\draw (6.5,-2) -- (6.5,18);
\draw (7.0,-2) -- (7.0,18);
\draw (7.5,-2) -- (7.5,18);
\draw (8.0,-2) -- (8.0,18);
\draw (8.5,-2) -- (8.5,18);
\draw (9.0,-2) -- (9.0,18);
\draw (9.5,-2) -- (9.5,18);

\draw (10.5,-2) -- (10.5,18);
\draw (11.0,-2) -- (11.0,18);
\draw (11.5,-2) -- (11.5,18);
\draw (12.0,-2) -- (12.0,18);
\draw (12.5,-2) -- (12.5,18);
\draw (13.0,-2) -- (13.0,18);
\draw (13.5,-2) -- (13.5,18);

\draw (14.5,-2) -- (14.5,18);
\draw (15.0,-2) -- (15.0,18);
\draw (15.5,-2) -- (15.5,18);
\draw (16.0,-2) -- (16.0,18);
\draw (16.5,-2) -- (16.5,18);
\draw (17.0,-2) -- (17.0,18);
\draw (17.5,-2) -- (17.5,18);

\draw(0,-3) node{$\color{blue}0$};
\draw(8,-3) node{$\color{blue}1$};
\draw(16,-3) node{$\color{blue}2$};

\draw(-3,0) node{$\color{blue}0$};
\draw(-3,8) node{$\color{blue}1$};
\draw(-3,16) node{$\color{blue}2$};

\draw(-2,-3) node{$\color{red}0$};
\draw(2,-3) node{$\color{red}1$};
\draw(6,-3) node{$\color{red}2$};
\draw(10,-3) node{$\color{red}3$};
\draw(14,-3) node{$\color{red}4$};
\draw(18,-3) node{$\color{red}5$};

\draw(-3,-2) node{$\color{red}0$};
\draw(-3,2) node{$\color{red}1$};
\draw(-3,6) node{$\color{red}2$};
\draw(-3,10) node{$\color{red}3$};
\draw(-3,14) node{$\color{red}4$};
\draw(-3,18) node{$\color{red}5$};

\end{tikzpicture}
\caption{Set $\I = \setof{0,1,2}$. 
Blue indicates $\cX = \cX(\I)$. 
Red indicates $\cX_b$.
Black indicates $\Janus$.
}
\label{fig:Xr}
\end{figure}

We need to relate  the chain complexes $C_*\left(\bar{\I};\F\right)$ and $C_*\left(\bbarbI;\F \right)$ corresponding to $\bar{\I}$ and $\cX\left(\bbarbI\right)$, respectively.
We declare the canonical bases for $C_j\left(\bar{\I};\F\right)$  and $C_j\left(\bbarbI;\F \right)$ to be $\cX_b^{(j)}$ and $\Janus^{(j)}$.
Using these bases we define a chain map $\btor\colon C_*\left(\bar{\I};\F \right) \to C_*\left( \bbarbI ;\F \right)$ as follows.
Let $[\bv,\bw]\in \cX_b$ and $[\bu,\bw]\in \cX_b$ where $\bv = (\bv_1,\ldots,\bv_N)$, $\bu = (\bu_1,\ldots,\bu_N)$, and $\bw \in \setof{0,1}^N$.
We define an acyclic carrier $\bar{S}\colon \cX(\bar{\I}) \to \cX\left(\bbarbI\right)$ by
\begin{equation}
\label{eq:barS}
    \bar{S}([\bv, \bw]) := \setof{ [\bu,\bw] \in \Janus \mid 8{\bv}_n \leq \bu_n < 8({\bv}_n+\bw_n)+1}.
\end{equation}

The subdivision map $\bar{S}$ has an associated chain map $\btor\colon C_*(\bar{\I};\F) \to C_*( \bbarbI ;\F)$ given by
\[
\btor_n([\bv,\bw]) := \sum_{\xi\in {\bar{S}}^{n}([\bv,\bw])}\xi.
\]

We leave it to the reader to check that under this identification $\bbarbI$ is a subdivision of $\bar{\I}$ (see \cite[Definition 23.1]{lefschetz}) and therefore that the associated chain complexes are chain homotopic.

In the next section we are interested in taking combinatorial information expressed by the multivalued maps on $\cX$ and deriving homological information using the Janus complex.
For this we make use of the following function,
\begin{equation}
\label{eq:defn_sb}
\begin{aligned}
\ctoy \colon \cX &\rightarrow  \sSub(\Janus) \\
[\bv, \bw] &\mapsto  \ \sO(\bar{S}(\blup([\bv, \bw]))),
\end{aligned}
\end{equation}
where $\sSub(\Janus)$ is the collection of subcomplexes of $\Janus$.

As is discussed in the introduction to this chapter appropriate geometrizations of $\sInvset^+(\cF_i)$, $i=1,2$, need to be constructed.
However, it is useful to view these constructions as modifications of  rectangular geometrizations.
With this in mind, let $\zeta =[\bu,\bw]\in \Janus$.
Then, for each $n=1,\ldots, N$ we can write $\bu_n = 8\bv_n + k_n$ where $0\leq k_n < 8$ and $\bv  \in \cX_b$.
Recall from \eqref{eq:K(w)} that $K(\bw)$ is defined in terms of $\setof{\bar{\bw}\mid \bar{\bw}\in \setof{0,1}^N,\ \bar{\bw}\leq_\Z \bw}\subset [0,1]^N$.
Define $\bar{a}_\zeta\colon K(\bw)\to K(\bw)$ to be the linear map that takes
\[
\bar{\bw} \to \frac{1}{8}\left(\bar{\bw}+k \right).
\]

We extend the rectangular geometrizations of $\cX_b$ (see Definition~\ref{defn:rectGeoXb}) to the Janus complex $\Janus$.

\begin{defn}
    \label{defn:rectGeoXr}
Fix $\Geo \in \cG$.
Consider $\zeta =[\bu,\bw]\in \Janus$.
Then, there exist $\xi = [\bv,\bw]\in \cX_b$ such that $\bu_n = 8\bv_n + k_n$ for some $0\leq k_n < 8$. 
Let $a_\xi \colon K(\bw)\to \Rec(\xi)$ be as defined in Section~\ref{sec:geometrization01}.
Define 
\[
    g_\zeta = \Geo\circ a_\xi \circ \bar{a}_\zeta.
\]
The collection of homeomorphisms $\recG_\mathtt{J} :=\setof{g_\zeta \mid \zeta \in \Janus}$ defines a \emph{rectangular geometrization} of $\Janus$.
\end{defn}

The notion of separation between cells, defined below, is used the following chapters to identify subsets of the Janus complex.

\begin{defn}
\label{defn:janus-separation}
Let $\cZ$ be an $N$-dimensional cubical complex and let $\zeta = [\bv,\bw], \bar{\zeta} = [\bar{\bv},\bar{\bw}]\in \cZ$.
The \emph{separation} between $\zeta$ and $\bar{\zeta}$ is given by
\[
d(\zeta,\bar{\zeta}) = \min\setdef{\sum_{n=1}^N |\bv_n-\bar{\bv}_n|}{[\bv,\bzero]\preceq\zeta,\ [\bar{\bv},\bzero]\preceq\bar{\zeta}}.
\]
\end{defn}

Given $\zeta\in\cZ$, set
\begin{equation}
\label{eq:Br}
B_r(\zeta):=\cl \left(\setof{\zeta' \mid d(\zeta, \zeta')\leq r}\right)\subset \cZ.
\end{equation}
Observe that by definition $B_r(\zeta)$ is a closed cell complex.

The following notation is defined for a general cubical complex and used in Chapter~\ref{sec:global_pi3}. 
\begin{defn}
\label{defn:cube-line}
Let $\cX$ be an $N$-dimensional cubical complex and let $\Janus$ be the Janus complex. For any $\bv,\bv' \in \bbarbI$, we define the \emph{cube between $\bv$ and $\bv'$} by
\begin{equation}
    \label{eq:cube}
\mathrm{Cube}(\bv,\bv') = \setdef{[\bv'',\bone]\in \Janus^{(N)}}{\bv \leq_\Z \bv'',\  \bv''+\bone \leq_\Z \bv'},
\end{equation}
and the \emph{line between $\bv$ and $\bv'$} by
\begin{equation}
    \label{eq:line}
    \mathrm{Line}(\bv,\bv') \coloneqq \setdef{[\bv'',\bw'']\in \Janus}{\bv \leq_\Z \bv''' \leq_\Z \bv' \text{ for all } \bv'''\in\cl([\bv'',\bw''])\cap \Janus^{(0)}}.
\end{equation}
For $\eta\in\cX$, let 
\begin{equation}\label{eq:T1}
   \cT^{(1)}(\eta) \coloneqq \setdef{\tau\in\cX^{(\dim(\eta)+1)} }{ \eta\prec \tau},
\end{equation}
and for
 any $\xi\in\cX^{(0)}$ define
\begin{equation}
\label{eq:cross}
    \cross(\xi)\coloneqq\ctoy(\xi)\cup \displaystyle\bigcup_{\eta\in\cT^{(1)}(\xi)}\ctoy(\eta).
\end{equation}
\end{defn}

\chapter{$D$-Gradings in the Janus Complex}
\label{sec:P2-grading}

In this chapter, given a  multivalued map $\cF_2 \colon \cX \mvmap \cX$ as in Definition~\ref{def:Rule2} we construct an admissible D-grading $\pi_2 \colon \Janus \to \SCC(\cF_2)$ that is used to  extend Theorem~\ref{thm:LipAttBlock} to the setting of $\cF_2$. 
We use Figure~\ref{fig:indecisive_drift_2_b1} to illustrate the construction.
The initial wall labeling is shown in Figure~\ref{fig:indecisive_drift_2_b1}(A).
The multivalued map $\cF_2 \colon \cX \mvmap \cX$ is obtained by refining $\cF_1$ at pairs that exhibit indecisive drift, e.g., $(\xi,\xi_0')$ and $(\xi,\xi_1')$.
Figure~\ref{fig:indecisive_drift_2_b1}(B) shows elements of $\cX_b$ that share faces with $\blup(\xi)$ along with the edges induced by
$\cF_1 \colon \cX \mvmap \cX$.
The D-grading $\pi\colon \cX \to \SCC(\cF_2)$ is determined by the global action of $\cF_2 \colon \cX \mvmap \cX$.
For purposes of this illustration we assume that
\begin{equation}
\label{eq:Dgrading}
    \pi(\xi'_0) = p_5,\ \pi(\xi) = p_4,\ \pi(\xi'_1) = p_3,\ \pi(\mu_1) = p_2,\ \pi(\tau) = p_1,\ \pi(\mu_0) = p_0, 
\end{equation}
where $p_0 \preceq_{\bar{\cF}_2} p_1 \preceq_{\bar{\cF}_2} \cdots \preceq_{\bar{\cF}_2} p_5$ are elements of $\SCC(\cF_2)$.

\begin{figure}[]
    \centering
    \begin{subfigure}{0.45\textwidth}
    \centering
    \begin{tikzpicture}[scale=0.20]
        \draw[step=8cm,ultra thick, blue] (0,0) grid (16,16);
        \foreach \i in {0,...,2}{
            \foreach \j in {0,...,2}{
            \draw[black, fill=blue] (8*\i,8*\j) circle (2ex);
            }
        }
        \foreach \i in {1,...,3}{
            \draw(-1,8*\i -8) node{\color{blue}$\i$}; 
        }
        \foreach \i in {0,...,2}{
            \draw(8*\i,-1) node{\color{blue}$\i$}; 
        }
        \foreach \i in {0,...,1}{
            \draw[->,  thick] (4+8*\i,0.1) -- (4+8*\i,2.1); 
            \draw[->,  thick] (4+8*\i,15.9) -- (4+8*\i,13.9); 
        }
        \foreach \i in {1,...,2}{
            \draw[->,  thick] (0.1, 4+8*\i -8) -- (2.1,4+8*\i -8); 
        }
        \foreach \i in {0,...,1}{
            \draw[->,  thick] (4+8*\i,5.9) -- (4+8*\i,7.9); 
            \draw[->, thick] (4+8*\i,8.1) -- (4+8*\i,10.1);
        }
            \draw[->,  thick] (8.1,4) -- (10.1,4);
            \draw[->,  thick] (5.9,4) -- (7.9,4);
            \draw[->,  thick] (13.9,4) -- (15.9,4); 
        \foreach \i in {2}{
            \draw[->,  thick] (7.9,4+8*\i -8) -- (5.9,4+8*\i -8); 
             \draw[->,  thick] (10.1,4+8*\i -8) -- (8.1,4+8*\i -8);
             \draw[->,  thick] (15.9,4+8*\i -8) -- (13.9,4+8*\i -8);  
        }
        \draw(4,12) node{$\mu_0$};
        \draw(12,12)  node{$\mu_1$};
        \draw(4,4) node{$\hat\xi_0'$};
        \draw(12,4)  node{$\hat\xi_1'$};
        \draw(5.5,7) node{$\xi'_0$};
        \draw(8.75,13.5)  node{$\tau$};
        \draw(8.75,2.75)  node{$\hat\xi$};
        \draw(13.5,7)  node{$\xi'_1$};
        \draw(9,9)  node{$\xi$};
        
        \end{tikzpicture}
        \caption{}
        \label{subfig:indecisive_drift_a}
        \end{subfigure}
        \hfill 
        \begin{subfigure}{0.45\textwidth}
        \centering
            \begin{tikzpicture}[scale=0.30]
                \draw[step=4cm, red] (0,0) grid (12,12);
            
                \foreach \i in {3,...,6}{
                    \draw(-1,-12+4*\i) node{\color{red}$\i$}; 
                }
                \foreach \i in {1,...,4}{
                    \draw(-4+4*\i,-1) node{\color{red}$\i$}; 
                }
            
                \draw(2,10) node{\tiny$\blup(\mu_0)$};
                \draw(10,10)  node{\tiny$\blup(\mu_1)$};
                \draw(2,2) node{\tiny$\blup(\hat\xi_0')$};
                \draw(10,2)  node{\tiny$\blup(\hat\xi_1')$};
                \draw(2,6) node{\tiny$\blup(\xi'_0)$};
                \draw(6,10)  node{\tiny$\blup(\tau)$};
                \draw(6,2)  node{\tiny$\blup(\hat\xi)$};
                \draw(10,6)  node{\tiny$\blup(\xi'_1)$};
                \draw(6,5.5)  node{\tiny$\blup(\xi)$};
            
                \draw[->,  thick] (5,10) -- (3,10);
                \draw[->,  thick] (9,10) -- (7,10);
                \draw[->,  thick] (3,2) -- (5,2);
                \draw[->,  thick] (7,2) -- (9,2);
            
                \draw[->,  thick] (2,7) -- (2,9);
                \draw[->,  thick] (6,7) -- (6,9);
                \draw[->,  thick] (10,7) -- (10,9);
                \draw[->,  thick] (2,3) -- (2,5);
                \draw[->,  thick] (6,3) -- (6,5);
                \draw[->,  thick] (10,3) -- (10,5);
            
                \draw[dashed,blue] (4,6) -- (8,6);
            \end{tikzpicture}
            \caption{}
            \label{subfig:indecisive_drift_b}
    \end{subfigure}
    
    \begin{subfigure}{0.45\textwidth}
    \centering
    \begin{tikzpicture}
    [scale=0.4]

    \fill[gray] (6,8.5) -- (5.5,8.5) -- (5.5,9.0) -- (5.0,9.0) -- (5.0,9.5) -- (4.5,9.5) -- (4.5,10) -- (6,10) -- (6,8.5);

    \fill[gray] (10,8.5) -- (9.5,8.5) -- (9.5,9.0) -- (9,9.0) -- (9,9.5) -- (8.5,9.5) -- (8.5,10) -- (10,10) -- (10,8.5);
    
    \draw[thick, red] (2,6) -- (14,6);
    \draw[thick, red] (2,10) -- (14,10);
    \draw[thick, red] (2,14) -- (14,14);
    \draw[thick, red] (2,6) -- (2,14);
    \draw[thick, red] (6,6) -- (6,14);
    \draw[thick, red] (10,6) -- (10,14);
    \draw[thick, red] (14,6) -- (14,14);
    
    \draw (2,6.5) -- (14,6.5);
    \draw (2,7.0) -- (14,7.0);
    \draw (2,7.5) -- (14,7.5);
    \draw (2,8.0) -- (14,8.0);
    \draw (2,8.5) -- (14,8.5);
    \draw (2,9.0) -- (14,9.0);
    \draw (2,9.5) -- (14,9.5);
    
    \draw (2,10.5) -- (14,10.5);
    \draw (2,11.0) -- (14,11.0);
    \draw (2,11.5) -- (14,11.5);
    \draw (2,12.0) -- (14,12.0);
    \draw (2,12.5) -- (14,12.5);
    \draw (2,13.0) -- (14,13.0);
    \draw (2,13.5) -- (14,13.5);
    
    \draw (2.5,6) -- (2.5,14);
    \draw (3.0,6) -- (3.0,14);
    \draw (3.5,6) -- (3.5,14);
    \draw (4.0,6) -- (4.0,14);
    \draw (4.5,6) -- (4.5,14);
    \draw (5.0,6) -- (5.0,14);
    \draw (5.5,6) -- (5.5,14);
    
    \draw (6.5,6) -- (6.5,14);
    \draw (7.0,6) -- (7.0,14);
    \draw (7.5,6) -- (7.5,14);
    \draw (8.0,6) -- (8.0,14);
    \draw (8.5,6) -- (8.5,14);
    \draw (9.0,6) -- (9.0,14);
    \draw (9.5,6) -- (9.5,14);
    
    \draw (10.5,6) -- (10.5,14);
    \draw (11.0,6) -- (11.0,14);
    \draw (11.5,6) -- (11.5,14);
    \draw (12.0,6) -- (12.0,14);
    \draw (12.5,6) -- (12.5,14);
    \draw (13.0,6) -- (13.0,14);
    \draw (13.5,6) -- (13.5,14);
    
    \draw(2,5) node{$\color{red}1$};
    \draw(6,5) node{$\color{red}2$};
    \draw(10,5) node{$\color{red}3$};
    \draw(14,5) node{$\color{red}4$};
    
    \draw(1,6) node{$\color{red}4$};
    \draw(1,10) node{$\color{red}5$};
    \draw(1,14) node{$\color{red}6$};
    
    \draw[red, fill=red] (6,10) circle (1.5ex);
    \draw[red, fill=red] (10,10) circle (1.5ex);
    
    \end{tikzpicture}
    \caption{}
    \label{subfig:indecisive_drift_c}
    \end{subfigure}
    \hfill 
    \begin{subfigure}{0.45\textwidth}
    \centering
    \begin{tikzpicture}[scale=0.4]
    
    \fill[grading12] (2,6) -- (2,10) -- (4.5,10) -- (4.5,9.5) -- (5.0,9.5) -- (5.0,9.0) -- (5.5,9.0) --  (5.5,8.5) -- (6,8.5) -- (6,6) -- (2,6);
    
    \fill[grading8] (4.5,10) -- (8.5,10) -- (8.5,9.5) -- (9.0,9.5) -- (9.0,9.0) -- (9.5,9.0) -- (9.5,8.5) -- (10,8.5) -- (10,6) -- (6,6) -- (6,8.5) -- (5.5,8.5) -- (5.5,9.0) -- (5.0,9.0) -- (5.0,9.5) -- (4.5,9.5) -- (4.5,10);
    
    \fill[grading8] (6,8.5) -- (5.5,8.5) -- (5.5,9.0) -- (5.0,9.0) -- (5.0,9.5) -- (4.5,9.5) -- (4.5,10.5) -- (5,10.5) -- (5,11) -- (5.5,11) -- (5.5, 11.5) -- (6,11.5) -- (6,8.5);
    
    \fill[grading7] (10,8.5) -- (9.5,8.5) -- (9.5,9.0) -- (9,9.0) -- (9,9.5) -- (8.5,9.5) -- (8.5,10) -- (14,10) -- (14,6) -- (10,6) -- (10,8.5);
    
    \fill[grading5] (10,10) rectangle (14,14);
    \fill[grading3] (6,10) rectangle (10,14);
    \fill[grading1] (2,10) rectangle (6,14);

    \draw[thick, red] (2,6) -- (14,6);
    \draw[thick, red] (2,10) -- (14,10);
    \draw[thick, red] (2,14) -- (14,14);
    \draw[thick, red] (2,6) -- (2,14);
    \draw[thick, red] (6,6) -- (6,14);
    \draw[thick, red] (10,6) -- (10,14);
    \draw[thick, red] (14,6) -- (14,14);
    
    \draw (2,6.5) -- (14,6.5);
    \draw (2,7.0) -- (14,7.0);
    \draw (2,7.5) -- (14,7.5);
    \draw (2,8.0) -- (14,8.0);
    \draw (2,8.5) -- (14,8.5);
    \draw (2,9.0) -- (14,9.0);
    \draw (2,9.5) -- (14,9.5);
    
    \draw (2,10.5) -- (14,10.5);
    \draw (2,11.0) -- (14,11.0);
    \draw (2,11.5) -- (14,11.5);
    \draw (2,12.0) -- (14,12.0);
    \draw (2,12.5) -- (14,12.5);
    \draw (2,13.0) -- (14,13.0);
    \draw (2,13.5) -- (14,13.5);
    
    \draw (2.5,6) -- (2.5,14);
    \draw (3.0,6) -- (3.0,14);
    \draw (3.5,6) -- (3.5,14);
    \draw (4.0,6) -- (4.0,14);
    \draw (4.5,6) -- (4.5,14);
    \draw (5.0,6) -- (5.0,14);
    \draw (5.5,6) -- (5.5,14);
    
    \draw (6.5,6) -- (6.5,14);
    \draw (7.0,6) -- (7.0,14);
    \draw (7.5,6) -- (7.5,14);
    \draw (8.0,6) -- (8.0,14);
    \draw (8.5,6) -- (8.5,14);
    \draw (9.0,6) -- (9.0,14);
    \draw (9.5,6) -- (9.5,14);
    
    \draw (10.5,6) -- (10.5,14);
    \draw (11.0,6) -- (11.0,14);
    \draw (11.5,6) -- (11.5,14);
    \draw (12.0,6) -- (12.0,14);
    \draw (12.5,6) -- (12.5,14);
    \draw (13.0,6) -- (13.0,14);
    \draw (13.5,6) -- (13.5,14);
    
    \draw(2,5) node{$\color{red}1$};
    \draw(6,5) node{$\color{red}2$};
    \draw(10,5) node{$\color{red}3$};
    \draw(14,5) node{$\color{red}4$};
    
    \draw(1,6) node{$\color{red}4$};
    \draw(1,10) node{$\color{red}5$};
    \draw(1,14) node{$\color{red}6$};
    
    \end{tikzpicture}
    \caption{}
    \label{subfig:indecisive_drift_d}
    \end{subfigure}
\caption{{
(A) $\cX(\I)$ and associated wall labeling. 
(B) $\cX_b$ and visualization of single edges of $\cF_1 \colon \cX \mvmap \cX$.
(C) In black, the subcomplex of $\Janus$ within which modifications to $\pi_{\mathtt{J}}$ will be made. 
Red points are $\Spineb(\xi, \xi'_0)$ and $\Spineb(\xi, \xi'_1)$.
The shaded regions are the modification tiles $\mTile(\xi,\xi_0')$ and $\mTile(\xi,\xi_1')$.
(D) Coloring of fibers of $\pi_2$. 
Cells in $\Janus$ that belong to $\pi_2^{-1}(p_i)$, for $i=0,\ldots,5$ are shown in the same color.}}
\label{fig:indecisive_drift_2_b1}
\end{figure}
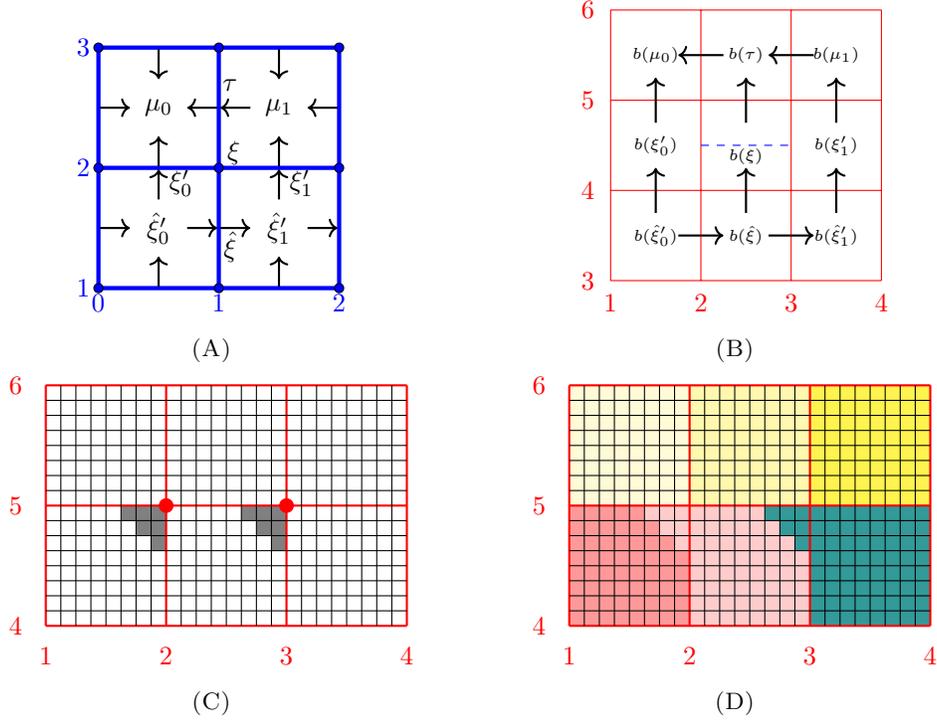

We use the subdivision operator $\bar{S}\colon \cX_b \to \Janus$ to define an admissible D-grading $\pi_{\mathtt{J}}\colon \Janus^{(N)}\to \SCC(\cF_2)$ as follows.
Given $\eta\in \bar{S}(\blup(\xi))\cap \Janus^{(N)}$ where $\xi\in\cX$ set  
\begin{equation}
\label{eq:piJ}
\pi_{\mathtt{J}}(\eta) = \pi(\xi).
\end{equation}
Then, we use local modifications to construct an admissible D-grading $\pi_2\colon \Janus \to \SCC(\cF_2)$.
As suggested by Figure~\ref{fig:indecisive_drift_2_b1}(B), this modification will involve elements of $\bar{S}(\blup(\alpha))$ when $\alpha\in\cX$ is a cell that belongs to a pair that exhibits indecisive drift.  

Our analysis makes use of computations involving chain complexes defined in terms of elements of $\cl(\blup(\cN))$ where $\cN\in\sInvset^+(\cF_2)$.
Thus, we need to keep track of cells that are lower than $\alpha$ with respect to the D-grading $\pi_{\mathtt{J}}$.
In the context of the example of Figure~\ref{fig:indecisive_drift_2_b1}, this implies that we can limit our attention to the cells of $\Janus$ shown in Figure~\ref{fig:indecisive_drift_2_b1}(C), i.e., 
\[
   \bar{S}(\blup(\xi'_0)),\ \bar{S}(\blup(\xi)),\ \bar{S}(\blup(\xi'_1)),\ \bar{S}(\blup(\mu_1)),\ \bar{S}(\blup(\tau)),\ \text{and}\ \bar{S}(\blup(\mu_0)). 
\]
As discussed in Section~\ref{sec:ConleyComplex}, the admissible D-grading determines a poset grading of $C_*(\cX;\F)$, which in turn leads to a Conley complex. By deriving $\pi_2$ via local modifications of $\pi_{\mathtt{J}}$, the two associated 
graded chain complexes are chain homotopic. To that end, we only perform modifications on closures of subsets called \emph{modification cells}, the grey $N$-dimensional cells in Figure~\ref{fig:indecisive_drift_2_b1}(C). 
For simplicity, the identification of the modification cells is done using cells in $\cX_b$ that we call \emph{modification spines} (see the red points in Figure~\ref{fig:indecisive_drift_2_b1}(C)).
Under $\pi_2$, the grading of the modification cells are different from those under $\pi_{\mathtt{J}}$ (see Figure~\ref{fig:indecisive_drift_2_b1}(D)). 

The modification spines and cells are defined in terms of pairs that exhibit indecisive drift.
Section~\ref{sec:2notation} outlines main definitions, properties and examples. 
Section~\ref{sec:pi2} defines $\pi_2$ and shows that the poset grading defined by $\pi_2$ is poset chain homotopic to the poset grading $\pi_b$ upon which the Conley complex was computed.

For the remainder of the chapter, recall that given a multivalued map $\cF\colon\cX\mvmap \cX$ the fibers of the D-grading $\pi\colon \cX \to \SCC(\cF)$ form a partition of $\cX$, i.e.,
\[
    \cX = \bigcup_{p\in\SCC(\cF)}\pi^{-1}(p), \text{ and } \pi^{-1}(p)\cap \pi^{-1}(p') = \emptyset \text{ when } p \neq p'.
\]
Furthermore, the composition with the bijection $\blup^{-1}$ defines a map $\pi_b = \pi \circ \blup^{-1} : \cX^{(N)} \to \SCC(\cF)$ whose fibers form a partition of $\cX_b^{(N)}$. 
Using the subdivision operation $\bar{S}\colon \cX_b \to \Janus$ we define a D-grading $\pi_{\mathtt{J}}\colon \Janus^{(N)}\to \SCC(\cF_2)$ as follows.
Given $\eta\in \bar{S}(\blup(\xi))\cap \Janus^{(N)}$ where $\xi\in\cX$ set  
\begin{equation}\label{eq:pi_Janus}
    \pi_{\mathtt{J}}(\eta) = \pi(\xi).
\end{equation}

\section{Subcomplexes derived from indecisive drift}
\label{sec:2notation}

Let $\Phi$ be a rook field derived from a ramp system (see Definition~\ref{defn:ramp-wall-labeling}).
This gives rise to multivalued map $\cF_2\colon \cX\mvmap \cX$ to $\cX_b$ in turn defines an admissible D-grading $\pi$, which via \eqref{eq:pi_Janus} defines an admissible D-grading $\pi_{\mathtt{J}}\colon \Janus^{(N)}\to \SCC(\cF_2)$. 
As indicated in the introduction to this chapter  the modifications to the  poset grading $\pi_{\mathtt{J}}$ that give rise to  $\pi_2$ are local in nature. 
In particular, they are performed on cells that arise from subdivisions of elements of $\cX_b^{(N)}$.
The drift map defined below is introduced inorder to identify these elements.

Recall that $\cD(\rook)$ denotes the set of pairs that exhibit indecisive drift (see Definition~\ref{defn:indecisive}), and given $(\xi,\xi')\in\cD(\rook)$, $\Dec(\xi,\xi')$ denotes the set of back walls for $(\xi,\xi')$ (see Definition~\ref{defn:back-walls}).

By Proposition~\ref{prop:back_wall_well_defined} if $\Ex(\xi,\xi')=\setof{n_o}$, then  $\rook_{n_o}$ is constant over $\Dec(\xi,\xi')$.
As a consequence the following map is well-defined.

\begin{defn}
\label{defn:drift_map}
Assume that $(\xi, \xi')\in\cD(\rook)$ with $\Ex(\xi,\xi')=\setof{n_o}$ and  $(\dec, \dec')\in \Dec(\xi,\xi')$.
The \emph{drift map} at $(\xi, \xi')$ is given by
\[
\begin{aligned}
D_{(\xi, \xi')}\colon \Top_\cX(\xi') & \to \setof{0,\pm 1} \\
\mu & \mapsto \begin{cases}
    \rook_{n_o}(\xi, \mu)  & \text{if  $\rook_{n_o}(\xi, \mu) \neq \rook_{n_o}(\dec, \dec')$,}\\
0 & \text{if $\rook_{n_o}(\xi, \mu) = \rook_{n_o}(\dec, \dec')$.}
\end{cases}
\end{aligned}
\]
For $(\xi,\xi') \in \cD(\rook)$, we define the \emph{support of $D_{(\xi,\xi')}$} by 
\[
\support\left(D_{(\xi,\xi')}\right) \coloneqq \setdef{\mu \in \Top_\cX(\xi')}{D_{(\xi,\xi')}(\mu) \neq 0}. 
\]
\end{defn}

\begin{prop}
    \label{prop:nontrivial-supp}
    If $(\xi,\xi') \in \cD(\rook)$, then $\support\left( D_{(\xi,\xi')} \right) \neq \emptyset$. 
\end{prop}
\begin{proof}
    Let $(\dec,\dec')\in \Dec(\xi,\xi')$ and $\Ex(\xi,\xi')=\setof{n_o}$. 
    Suppose for the sake of contradiction that $\support\left( D_{(\xi,\xi')} \right) = \emptyset$. 
    Then for any $\mu \in \Top_\cX(\xi')$, 
    \[
        \rook_{n_o}(\xi,\mu) = \rook_{n_o}(\dec,\dec').
    \]
    Therefore, for any n-wall $(\xi_n,\mu),(\xi_n,\mu') \in W(\xi')$, 
    \[
        \rook_{n_o}(\xi_n,\mu) = \rook_{n_o}(\xi_n,\mu').
    \]
    Thus, there is no $n_g \in G_i(\xi')$ that actively regulates $n_o$, which contradicts Proposition~\ref{prop:GO-pair-properties}. 
\end{proof}

\begin{rem}
    Note that for any $(\xi,\xi')\in \cD(\rook)$ and \( \mu \in \support(D_{(\xi,\xi')}) \), 
    \[
    D_{(\xi,\xi')}(\mu)=-\rook_{n_o}(\dec,\dec'),
    \]
    where $(\dec,\dec')\in \Dec(\xi,\xi')$ since $n_o \in J_i(\xi) \subset O(\xi) \cup G(\xi)$ by Proposition~\ref{prop:Direction-Decomposition}.  
\end{rem}
\begin{ex}
\label{ex:drift_2d}
We recall Examples~\ref{ex:indecisive_drift} and \ref{ex:backpairs} (or see Figure~\ref{fig:indecisive_drift_2_b1}) in which $(\xi,\xi_0')$ and $(\xi,\xi_1')$ exhibit indecisive drift, and $(\hat\xi,\hat\xi'_0)$ and $(\hat\xi,\hat\xi'_1)$ are the respective back walls.
In both cases $n_o = 1$.

Observe that $\Top_\cX(\xi'_0) = \setof{\hat\xi'_0,\mu_0}$.
Note that that $\rook_1(\hat{\xi},\hat{\xi}'_0) = 1$ and hence
$D_{(\xi, \xi_0')}(\hat\xi'_0)=0$. 
Since $\rook_1(\xi,\mu_0) = -1$ it follows that $D_{(\xi, \xi_0')}(\mu_0)=\rook_{1}(\xi, \mu_0) = -1$.
Similarly, $\Top_\cX(\xi'_1) = \setof{\hat\xi'_1,\mu_1}$ and  $\rook_1(\hat{\xi},\hat{\xi}'_1) = 1$.
Thus, $D_{(\xi, \xi_1')}(\hat\xi'_1)=0$ and $D_{(\xi, \xi_1')}(\mu_1)=-1$.
\end{ex}
\begin{ex}
\label{ex:3Db}
We return to the three dimensional Example~\ref{ex:3D_Dec_pairs}.
There we identified that the pairs $(\xi,\xi_i')$, $i=0,1$, exhibit indecisive drift where
\[
\xi  = \defcellb{v_1}{v_2}{v_3}{0}{0}{0}, \quad
\xi_0'  = \defcellb{v_1}{v_2-1}{v_3}{0}{1}{0}\quad\text{and}
\quad\xi_1'  = \defcellb{v_1}{v_2}{v_3}{0}{1}{0}.
\]
Furthermore, $(n_g,n_o) = (1,2)$ is a GO-pair for both pairs.
Observe that 
\[
\Top_\cX(\xi'_0) = \setof{
\defcellb{v_1}{v_2-1}{v_3}{1}{1}{1},
\defcellb{v_1-1}{v_2-1}{v_3}{1}{1}{1},
\defcellb{v_1}{v_2-1}{v_3-1}{1}{1}{1},
\defcellb{v_1-1}{v_2-1}{v_3-1}{1}{1}{1}
}.
\]
The unique back wall for $(\xi,\xi_0')$ is
\[
\Dec(\xi,\xi_0') = \setof{\left( \defcellb{v_1}{v_2}{v_3}{1}{0}{1} ,\defcellb{v_1}{v_2-1}{v_3}{1}{1}{1} \right)}
\]
and
\[
\rook_2\left( \defcellb{v_1}{v_2}{v_3}{1}{0}{1} ,\defcellb{v_1}{v_2-1}{v_3}{1}{1}{1} \right) = 1.
\]
Thus,
\[
D_{(\xi,\xi_0')}\left( \defcellb{v_1}{v_2-1}{v_3}{1}{1}{1} \right) = \rook_2\left( \xi,\defcellb{v_1}{v_2-1}{v_3}{1}{1}{1} \right) = 0
\]
while $D_{(\xi,\xi_0')}$ is $-1$ for the remaining elements of $\Top_\cX(\xi'_0)$.

We leave it to the reader to check that
\begin{align*}
D_{(\xi, \xi'_1)}\left(\defcellb{v_1-1}{v_2}{v_3-1}{1}{1}{1}\right) &=
D_{(\xi, \xi'_1)} \left( \defcellb{v_1}{v_2}{v_3-1}{1}{1}{1}\right)
\\&=D_{(\xi, \xi'_1)}\left( \defcellb{v_1-1}{v_2}{v_3}{1}{1}{1}\right)= -1
\end{align*}
and
\[
D_{(\xi, \xi'_1)} \left(\defcellb{v_1}{v_2}{v_3}{1}{1}{1}\right) = 0.
\]
\end{ex}
\begin{prop}\label{prop:drift}
For any $(\xi, \xi')\in\cD(\rook)$, $\support\left(D_{(\xi,\xi')}\right)\neq \emptyset$.
\end{prop}
\begin{proof}
Let $(\xi,\xi') \in \cD(\rook)$ with $\Ex(\xi,\xi')=\setof{n_o}$. 
By Definition~\ref{defn:indecisive}, $(\xi,\xi')$ has a GO-pair $(n_g,n_o)$. 
Thus, Definition~\ref{defn:GO-pair} yields $n_g$-adjacent top cells $\mu,\mu'\in \Top_\cX(\xi')$ such that $\rook_{n_o}(\xi,\mu) \neq \rook_{n_o}(\xi,\mu')$.
Note that $n_o \in J_i(\xi)\subset O(\xi) \cup G(\xi)$ by Proposition~\ref{prop:Direction-Decomposition}, so $\setof{\rook_{n_o}(\xi,\mu),\rook_{n_o}(\xi,\mu')} = \setof{\pm 1}$. 
Thus, either $\mu$ or $\mu'$ belongs to $\support(D_{(\xi,\xi')})$. 
\end{proof}

\begin{defn}
\label{defn:mcell}
The \emph{modification support} of $(\xi, \xi')\in \cD(\rook)$ is 
\[
    \mSuppX(\xi,\xi') \coloneqq \setdef{\xi'' \in \cX}{\xi' \preceq \xi'', \Top_{\cX}(\xi'') \subseteq \support\left(D_{(\xi,\xi')}\right)}. 
\]
\end{defn}
\begin{ex}
    \label{ex:mcell2}
Returning to Figure~\ref{fig:indecisive_drift_2_b1}(A), consider the  pairs 
\[
(\xi,\xi'_0) = \left( \defcell{1}{2}{0}{0},\defcell{0}{2}{1}{0} \right) \quad\text{and}\quad (\xi,\xi'_1) =\left( \defcell{1}{2}{0}{0},\defcell{1}{2}{1}{0} \right)
\]
that exhibit indecisive drift. 
In Example~\ref{ex:drift_2d}, we concluded that $\support(D_{(\xi,\xi_0')})=\setof{\mu_0}$ and $\support(D_{(\xi,\xi_1')})=\setof{\mu_1}$ where $\mu_0 = \defcell{0}{2}{1}{1}$ and $\mu_1=\defcell{1}{2}{1}{1}$. 
Therefore, 
\[
\mSuppX(\xi,\xi_i') = \setof{\mu_i}, \quad i=0,1.
\]
\end{ex}

We include Example~\ref{ex:mcell2} because Figure~\ref{fig:indecisive_drift_2_b1} provides a useful guide for determining $\mSuppX(\xi,\xi'_0)$. Indeed, since all pairs that exhibit indecisive drift are in $\cX^{(N-2)}\times\cX^{(N-1)}$, and each top cell in $\Top_\cX(\xi'_i)$ has a different value of $D_{(\xi,\xi'_i)}$, the modification support consists of a single top cell.  
As the following example shows, in higher dimensions, the pair that exhibits indecisive drift might have a lower dimension. 
\begin{ex}
\label{ex:mcell3}
Continuing with Example \ref{ex:3Db} , consider the pair
\[\xi = \defcellb{v_1}{v_2}{v_3}{0}{0}{0} \text{ and } \xi_1 = \defcellb{v_1}{v_2}{v_3}{0}{1}{0}\]
that exhibit indecisive drift. 
The goal is to find $\mSuppX(\xi,\xi_1')$. First, one computes the $\support(D_{(\xi,\xi_1')})$ as in Example~\ref{ex:mcell3}, that is, 
\[
    \support(D_{(\xi,\xi_1')}) = \setof{
    \defcellb{v_1-1}{v_2}{v_3-1}{1}{1}{1},
    \defcellb{v_1}{v_2}{v_3-1}{1}{1}{1},
    \defcellb{v_1-1}{v_2}{v_3}{1}{1}{1}
    }. 
\]
Then 
\begin{align*}
\mSuppX(\xi,\xi_1') = & \left\{ \defcellb{v_1-1}{v_2}{v_3}{1}{1}{0}, \defcellb{v_1}{v_2}{v_3-1}{0}{1}{1},
\defcellb{v_1}{v_2}{v_3-1}{1}{1}{1}, \right. \\
& \defcellb{v_1-1}{v_2}{v_3}{1}{1}{1},  
\left. \defcellb{v_1-1}{v_2}{v_3-1}{1}{1}{1}\right\}.
\end{align*}
\end{ex}
The discussion up to this point has made use of $\cX$.
We now turn attention to $\cX_b$.
\begin{defn}
    \label{defn:spine}
    We define the \emph{modification spine} of $(\xi,\xi') \in \cD(\rook)$ to be
    \[
        \Spineb(\xi,\xi') \coloneqq
        \cl(\blup(\xi)) \bigcap 
        \cl(\blup(\mSuppX(\xi,\xi'))) \subset \cX_b.
    \]
\end{defn}

Observe that by definition, if $\zeta\in \Spineb(\xi,\xi)$, then $\cl(\zeta) \subset \Spineb(\xi,\xi)$, and therefore the following lemma holds.

\begin{lem}
\label{lem:spineclosed}
The modification spine of $(\xi, \xi')\in \cD(\rook)$ is a closed subcomplex of $\cX_b$.
\end{lem}

The modification spine in the Janus complex is obtained by applying the subdivision map $\bar S$ from $\cX_b$ to $\Janus$ (see \eqref{eq:barS} and Lemma \ref{lem:spineclosed}).
More precisely it is given by  
\begin{equation}
\label{eq:SpineJ}
\SpineJ(\xi,\xi') \coloneqq \bar S(\Spineb(\xi,\xi')).
\end{equation}
Since $\Spineb(\xi,\xi')$ is closed and $\bar S$ is a subdivision map, $\SpineJ(\xi,\xi')$ is closed. 

\begin{ex}
\label{ex:spine2d}
Returning to Example~\ref{ex:mcell2}, consider  $(\xi,\xi_i')\in \cD(\rook)$ for $i=0,1$.
For each $i$, it follows from Example~\ref{ex:mcell2} that
\[
\mSuppX(\xi,\xi_i') = \setof{\mu_i}. 
\]
Thus, 
\[
\Spineb(\xi, \xi'_i)=\cl(\blup(\xi)) \cap \cl(\blup(\mu_i)),
\]
namely, 
\begin{align*}
\Spineb(\xi, \xi'_0) & =\setof{\defcell{2}{5}{0}{0}}\subset \cX_b \\
\Spineb(\xi, \xi'_1) & =\setof{\defcell{3}{5}{0}{0}}\subset \cX_b
\end{align*}
(see the vertices marked in blue in Figure~\ref{subfig:indecisive_drift_c}).
The corresponding modification spines in $\Janus$ are
\begin{align*}
\SpineJ(\xi, \xi'_0)& =\setof{\defcell{16}{40}{0}{0}}\subset \Janus \\
\SpineJ(\xi, \xi'_1)& =\setof{\defcell{24}{40}{0}{0}}\subset \Janus.
\end{align*}
\end{ex}

\begin{ex}\label{ex:spine3d}
Returning to Example \ref{ex:mcell3}
the support of $D_{(\xi,\xi_1')}$ is 
\[
    \support(D_{(\xi,\xi_1')}) = \setof{
    \defcellb{v_1-1}{v_2}{v_3-1}{1}{1}{1},
    \defcellb{v_1}{v_2}{v_3-1}{1}{1}{1},
    \defcellb{v_1-1}{v_2}{v_3}{1}{1}{1}.
    }. 
\]
Thus, the cells $\xi'' \in \cX$ that satisfy $\xi \preceq \xi \preceq \xi''$ and $\Top_\cX(\xi'') \subseteq \support(D_{(\xi,\xi_1')})$ are within the list of $16$ cells described in Example~\ref{ex:mcell3}. Naturally, each top dimensional cell in $\support(D_{\xi,\xi_1'})$ satisfies the condition. In addition, we have following two $2$-dimensional cells 
\[
    \defcellb{v_1}{v_2}{v_3-1}{0}{1}{1}, \defcellb{v_1-1}{v_2}{v_3}{1}{0}{1}.
\]
Therefore, 
\begin{multline*}
\Spineb(\xi,\xi_1') = \cl\left(\blup\left(\defcellb{v_1}{v_2}{v_3}{0}{0}{0}\right)\right) \bigcap \left( 
\cl\left( \blup\left( \defcellb{v_1-1}{v_2}{v_3-1}{1}{1}{1}\right)\right) \right.
\cup \\ 
\cl\left( \blup\left( \defcellb{v_1}{v_2}{v_3-1}{1}{1}{1}\right)\right) 
\cup  
\cl\left( \blup\left( \defcellb{v_1-1}{v_2}{v_3}{1}{1}{1}\right)\right) 
\cup \\ 
\left.
\cl\left( \blup\left( \defcellb{v_1}{v_2}{v_3-1}{0}{1}{1}\right)\right) 
\cup
\cl\left( \blup\left( \defcellb{v_1-1}{v_2}{v_3}{1}{0}{1}\right)\right) \right) 
= \\
\left\{ \defcellb{2v_1+1}{2v_2+1}{2v_3}{0}{0}{0}, \defcellb{2v_1}{2v_2+1}{2v_3}{0}{0}{0}, \defcellb{2v_1}{2v_2+1}{2v_3+1}{0}{0}{0},\right. \\
\left.
\defcellb{2v_1}{2v_2+1}{2v_3}{1}{0}{0}, \defcellb{2v_1}{2v_2+1}{2v_3}{0}{0}{1}
\right\}
\end{multline*}
and
\begin{multline*}
\SpineJ(\xi,\xi_1') = \bar S(\Spineb(\xi,\xi_1')) = \\
\left\{\defcellb{16v_1}{16v_2+8}{16v_3}{0}{0}{0}, \defcellb{16v_1+1}{16v_2+8}{16v_3}{0}{0}{0}, \ldots, \defcellb{16v_1+8}{16v_2+8}{16v_3}{0}{0}{0}, \right. 
\\
\defcellb{16v_1}{16v_2+8}{16v_3}{0}{0}{0}, \defcellb{16v_1}{16v_2+8}{16v_3+1}{0}{0}{0},\ldots, 
\defcellb{16v_1}{16v_2+8}{16v_3+8}{0}{0}{0},
\\
\defcellb{16v_1}{16v_2+8}{16v_3}{1}{0}{0}, \defcellb{16v_1+1}{16v_2+8}{16v_3}{1}{0}{0}, \ldots, \defcellb{16v_1+7}{16v_2+8}{16v_3}{1}{0}{0},
\\
\left.
\defcellb{16v_1}{16v_2+8}{16v_3}{0}{0}{1}, \defcellb{16v_1}{16v_2+8}{16v_3+1}{0}{0}{1},\ldots, 
\defcellb{16v_1}{16v_2+8}{16v_3+7}{0}{0}{1}
\right\}.
\end{multline*}
\end{ex}

The modification to the D-grading of the Janus complex induced by an indecisive pair is not symmetric. 
To break the symmetry, we begin by defining the notion of direction as follows.

\begin{defn}
\label{defn:signature}
Given  $(\xi,\xi')\in\cD(\rook)$ for which $\Ex(\xi,\xi')=\setof{n_o}$  the \emph{direction} of $(\xi,\xi')$ is defined as
\[
\direc(\xi,\xi') = D_{(\xi,\xi')}(\mu_0)
\]
where $\mu_0\in \support\left(D_{(\xi,\xi')}\right)$.
\end{defn}
To show that $\direc(\xi,\xi')$ is well defined we first note that by Proposition~\ref{prop:drift}, $\support\left(D_{(\xi,\xi')} \right)\neq\emptyset$.
Then, by Proposition~\ref{prop:back_wall_well_defined},
$D_{(\xi,\xi')}(\mu_0)$ is independent of the choice of $\mu_0\in \support\left(D_{(\xi,\xi')}\right)$.
As indicated by the following result, direction is consistent pairs that exhibit indecisive drift are faces of one another.
\begin{prop}
    \label{prop:consistent-direction}
If $(\xi_0,\xi_0'),(\xi_1,\xi_1') \in \cD(\rook)$ with $\xi_0 \preceq \xi_1$, then 
\[
    \direc(\xi_0,\xi_0')=\direc(\xi_1,\xi_1'). 
\]
\end{prop}
\begin{proof}
Assume that $(\xi_0,\xi_0'),(\xi_1,\xi_1') \in \cD(\rook)$ with $\xi_0 \preceq \xi_1$. 
First, by Proposition~\ref{prop:unique_ext_GO_pairs}, it follows that $\Ex(\xi_0,\xi_0')=\Ex(\xi_1,\xi_1')=\setof{n_o}$. 
Moreover, by Proposition~\ref{prop:GO-pair-properties}, if $(n_g,n_o)$ is a GO-pair for $(\xi_1,\xi_1')$, then it is a GO-pair for $(\xi_0,\xi_0')$.

Let $(\dec_0,\dec_0') \in \Dec(\xi_0,\xi_0')$ and $(\dec_1,\dec_1') \in \Dec(\xi_1,\xi_1')$. 
We show that $\rook_{n_o}(\dec_0,\dec_0')=\rook_{n_o}(\dec_1,\dec_1')$ by contradiction. 
Suppose without loss of generality that 
$\rook_{n_o}(\dec_0,\dec_0') = -1$ and $\rook_{n_o}(\dec_1,\dec_1') = 1$.
Note that $p_{n_g}(\xi_0,\dec_0)=p_{n_g}(\xi_1,\dec_1)=-r_{n_g}$ where $R_{n_g}(\xi_0)=\setof{r_{n_g}}$. 
Additionally, since $n_g$ is a gradient inessential direction for both cells and actively regulates $n_o$, there are $n_o$-walls $(\mu_0^*,\mu_0)$ and $(\mu_1^*,\mu_1)$ of $\xi_0$ and $\xi_1$ such that $p_{n_g}(\xi_0,\mu_0^*)=p_{n_g}(\xi_1,\mu_1^*)=r_{n_g}$, $\rook_{n_o}(\mu_0^*,\mu_0) = 1$, and $\rook_{n_o}(\mu_1^*,\mu_1) = -1$.
Thus, $\rook_{n_o}$ is both increasing and decreasing on the $n_g$ direction, which contradicts Proposition~\ref{prop:local-monotone}. 
\end{proof}

\begin{ex}
\label{ex:direction_2d}
Continuing with Example~\ref{ex:mcell2}, 
\[
\direc(\xi,\xi'_0) = D_{(\xi,\xi'_0)}(\mu_0) = -1
\quad
\text{and} 
\quad
\direc(\xi,\xi'_1) = D_{(\xi,\xi'_1)}(\mu_1) = -1.
\]
\end{ex}

\begin{ex}
\label{ex:direction_3d}
Continuing with Example~\ref{ex:mcell3},
\[
\direc(\xi,\xi'_1) = D_{(\xi,\xi'_1)}\left(\mu\right) = -1, \mu \in \support(D_{(\xi,\xi_1')}).
\]
Note that both pairs $\left( \defcellb{v_1}{v_2}{v_3}{1}{0}{0},\defcellb{v_1}{v_2}{v_3}{1}{1}{0}\right)$ and $\left( \defcellb{v_1}{v_2}{v_3}{0}{0}{1},\defcellb{v_1}{v_2}{v_3}{0}{1}{1}\right)$ exhibit indecisive drift with GO-pairs $(3,2)$ and $(1,2)$ respectively. For each of them,
\[
\direc\left( \defcellb{v_1}{v_2}{v_3}{1}{0}{0},\defcellb{v_1}{v_2}{v_3}{1}{1}{0}\right)= \left( \defcellb{v_1}{v_2}{v_3}{0}{0}{1},\defcellb{v_1}{v_2}{v_3}{0}{1}{1}\right) = -1 = \direc(\xi,\xi'_1),
\]
as in Proposition~\ref{prop:consistent-direction}. 
\end{ex}
We can decompose any pair $(\xi,\xi') \in \cD(\rook)$  into cells $\alpha^+, \alpha^-\in\setof{\xi,\xi'}$, where the local modifications will involve elements of $\bar S(\blup(\alpha^-))\cap B_2(\SpineJ(\xi, \xi'))$, while no modification will involve elements of $\bar S(\blup(\alpha^+)\cap B_2(\SpineJ(\xi, \xi'))$, see Figure~\ref{fig:indecisive_drift_2_b1}(C).
\begin{defn}\label{defn:gamma+-}
The \emph{positive} and \emph{negative drift cells} of  $(\xi,\xi')\in\cD(\rook)$ are
\[
\PosDriftCells(\xi,\xi')=
\begin{cases}
    \xi' & \text{ if $p_{n_o}(\xi,\xi') \direc(\xi,\xi')   = 1$}, \\
    \xi & \text{ if  $p_{n_o}(\xi,\xi') \direc(\xi,\xi')   = - 1$}, 
\end{cases}
\]
and
\[
\NegDriftCells(\xi,\xi')=
\begin{cases}
    \xi' & \text{ if $p_{n_o}(\xi,\xi') \direc(\xi,\xi')   = - 1$} \\
    \xi & \text{ if  $p_{n_o}(\xi,\xi') \direc(\xi,\xi')   =  1$},
\end{cases}
\]
respectively. 
\end{defn}
\begin{ex}
\label{ex:plusminus_2d}
Continuing with Example~\ref{ex:direction_2d}, recall that $n_o=1$.
Observe that $p_1(\xi,\xi'_0) = 1$ and $p_1(\xi,\xi'_1) = -1$ while $\direc(\xi,\xi_i')=-1$ for each $i=1,2$. 
Therefore,
\[
\PosDriftCells(\xi,\xi'_0) = \xi \quad\text{and}\quad
\PosDriftCells(\xi,\xi'_1) = \xi'_1
\]
or equivalently
\[
\NegDriftCells(\xi,\xi'_0) = \xi'_0\quad\text{and}\quad
\NegDriftCells(\xi,\xi'_1) = \xi.
\]
\end{ex}

\begin{ex}
\label{ex:plusminus_3d}
Continuing with Example~\ref{ex:mcell3}, recall that $n_o =2$.
Observe that $p_2(\xi,\xi'_1) = -1$ while $\direc(\xi,\xi_i')=-1$. 
Therefore,
\[
\PosDriftCells(\xi,\xi'_1) =  \defcellb{v_1}{v_2}{v_3}{0}{1}{0}
, \quad
 \NegDriftCells(\xi,\xi'_1) =  \defcellb{v_1}{v_2}{v_3}{0}{0}{0}.
\]
\end{ex}
\begin{rem}
    \label{rem:D+-}
    It is important to note that a cell $\alpha\in\cX$ can be associated with multiple pairs that exhibit indecisive drift via $\PosNegDriftCells(\xi,\xi')$.
However, it is clear that if $(\xi,\xi') \in \cD(\rook)$, then $\PosDriftCells(\xi,\xi') \neq \NegDriftCells(\xi,\xi')$. 
\end{rem}
A summary of the construction to this point is as follows. 
Given $(\xi,\xi')\in \cD(\rook)$, we have defined 
the modification spine $\Spineb(\xi,\xi')\subset \cX_b$. 
We now apply the subdivision operator to produce a subset of $B_2(\SpineJ(\xi, \xi'))\subset \Janus$ (see \eqref{eq:Br}) where
the local modifications to $\pi_{\mathtt{J}}$ will be performed. 
Below we describe how to identify this subset.

\begin{defn}
\label{defn:mTile}  
We define the \emph{modification tiles} for $(\xi,\xi')\in \cD(\rook)$ by
\[
\mTile(\xi,\xi'):=B_2(\SpineJ(\xi, \xi')) \cap \bar{S}(\blup(\NegDriftCells(\xi, \xi'))) \subset \Janus^{(N)}.
\]
\end{defn}
Modifications to the D-grading will involve cells that belong to $\cl(\mTile(\xi,\xi'))$.

\begin{ex}
\label{ex:mtiles2d}
Continuing with Example~\ref{ex:spine2d}.
The two-dimensional gray cells shown in Figure~\ref{fig:indecisive_drift_2_b1}(C) indicate the modification tiles 
\begin{align*}
    \mTile(\xi,\xi'_0) & = B_2\left(\bar{S}(\setof{\defcell{2}{5}{0}{0}}) \right) \cap \bar{S}(b(\setof{\xi_0'})),
\end{align*}
and
\begin{align*}
    \mTile(\xi,\xi'_1) & = B_2\left(\bar{S}(\setof{\defcell{3}{5}{0}{0}}) \right) \cap \bar{S}(b(\setof{\xi}))
\end{align*}
in the Janus complex.
\end{ex}
\begin{ex}
\label{ex:mtiles3d}
Returning to Example~\ref{ex:spine3d}. Since $\NegDriftCells(\xi, \xi_1')=\xi_1'$ then
by definition
\[
    \mTile(\xi,\xi_1'):=B_2(\SpineJ(\xi, \xi_1')) \cap \bar{S}(\blup(\xi_1')) \subset \Janus^{(N)}.
\]
\end{ex}

The following result is used to guarantee that  if  $(\xi_0,\xi_0'),(\xi_1,\xi_1')\in \cD(\rook)$, then modifications of the D-grading associated with $(\xi_0,\xi_0')$ are independent of modifications associated with $(\xi_1,\xi_1')$.

\begin{prop}
\label{prop:Ts}
Let $(\xi_0,\xi'_0), (\xi_1,\xi'_1) \in \cD(\rook)$ be distinct pairs.
Then,
\begin{enumerate}
\item[(i)]  $\mTile(\xi_0,\xi_0') \cap \mTile(\xi_1,\xi_1')=\emptyset$;
\item[(ii)]  if $(\xi_0,\xi'_0)$ and $(\xi_1,\xi'_1)$ share a common element ($\xi_0=\xi_1$ or $\xi'_0=\xi'_1$), then
$
\cl(\mTile(\xi_0,\xi_0')) \cap \cl(\mTile(\xi_1,\xi_1'))=\emptyset$; 
and
\item[(iii)] if 
$
\cl (\mTile(\xi_0,\xi_0')) \cap \cl (\mTile(\xi_1,\xi_1'))\neq\emptyset$,
then $\Ex(\xi_0,\xi_0')=\Ex(\xi_1,\xi_1')$.
\end{enumerate}
\end{prop}

\begin{proof}
Let $(\xi_0,\xi_0'), (\xi_1,\xi_1') \in \cD(\rook)$ and assume without loss of generality that $\dim(\xi_0)\leq\dim(\xi_1)$. Note that $B_2(\SpineJ(\xi_0,\xi_0')) \cap B_2(\SpineJ(\xi_1,\xi_1')) \neq \emptyset$ if and only if $\Spineb(\xi_0,\xi_0')\cap\Spineb(\xi_1,\xi_1') \neq \emptyset$. By Definition~\ref{defn:spine}, that is equivalent to
\[
        \left( \cl(\blup(\xi_0)) \bigcap 
        \cl(\blup(\mSuppX(\xi_0,\xi_0'))) \right)\bigcap \left( \cl(\blup(\xi_1)) \bigcap 
        \cl(\blup(\mSuppX(\xi_1,\xi_1'))) \right)  \neq \emptyset,
\]
which is nontrivial if and only if $\xi_0 \preceq \xi_1$ or $\xi_0' \preceq \xi_1'$. If $\dim(\xi_0)<\dim(\xi_1)$, then $\xi_0 \prec \xi_1$ and $\xi_0'\prec \xi_1'$ and item (i) follows by Definition~\ref{defn:mcell} since $\setof{\xi_0,\xi_0'}$ and $\setof{\xi_1,\xi_1'}$ are disjoint. Indeed, 
\[
    \bar{S}(\blup(\NegDriftCells(\xi_0, \xi'_0))) \cap \bar{S}(\blup(\NegDriftCells(\xi_1, \xi'_1))) \subset \setof{\xi_0,\xi_0'} \cap \setof{\xi_1,\xi_1'} = \emptyset.
\]
To prove (ii), assume that $\dim(\xi_0)=\dim(\xi_1)$ and that $(\xi_0,\xi_0')$ and $(\xi_1,\xi_1')$ share a common element. Note that either $\xi_0=\xi_1$ or $\xi_0'=\xi_1'$ as $\xi_0 = \xi_1'$ or $\xi_1 = \xi_0'$ would contradict Proposition~\ref{prop:not_L_GO_pairs}. Moreover, by Proposition~\ref{prop:unique_ext_GO_pairs}, $\Ex(\xi_0,\xi_0')=\Ex(\xi_1,\xi_1')$. By Lemma \ref{lem:spineclosed}, $\Spineb$ is a closed subcomplex of $\cX_b$, then $\Spineb(\xi_0, \xi_0')$ and $\Spineb(\xi_1, \xi_1')$ are separated by at least 1. Hence, any $\zeta_0 \in \SpineJ(\xi_0,\xi_0')$ and $\zeta_1 \in \SpineJ(\xi_1,\xi_1')$ are separated by at least 8. Therefore, 
\[ 
d(B_2(\SpineJ(\xi_0, \xi_0')),B_2(\SpineJ(\xi_1, \xi_1'))) \geq 2.
\]
Thus, $B_2(\SpineJ(\xi_0, \xi_0')$ and $B_2(\SpineJ(\xi_1, \xi_1'))$ are disjoint, which yields the result. 

To show (iii), note that $B_2(\SpineJ(\xi_0,\xi_0')) \cap B_2(\SpineJ(\xi_1,\xi_1')) \neq \emptyset$ if and only if $\Spineb(\xi_0,\xi_0')\cap\Spineb(\xi_1,\xi_1') \neq \emptyset$. By Definition~\ref{defn:spine}, that is equivalent to
\[
        \left( \cl(\blup(\xi_0)) \bigcap 
        \cl(\blup(\mSuppX(\xi_0,\xi_0'))) \right)\bigcap \left( \cl(\blup(\xi_1)) \bigcap 
        \cl(\blup(\mSuppX(\xi_1,\xi_1'))) \right)  \neq \emptyset,
\]
which is nontrivial if and only if $\xi_0 \preceq \xi_1$ or $\xi_0' \preceq \xi_1'$. In either case, the result follows from Proposition~\ref{prop:unique_ext_GO_pairs}. 
\end{proof}

\begin{defn}
\label{defn:Add&Rem}
Let $\alpha\in\cX$. 
The \emph{addition set} of $\alpha$ is
\[
\Add(\alpha) \coloneqq \bigcup_{\substack{(\xi,\xi')\in\PosIndecisiveDrifts(\alpha)}}  \mTile(\xi,\xi') \cap \bar{S}(b(\NegDriftCells(\xi,\xi'))).
\]
The \emph{removal set} of $\alpha$ is
\[
\Rem(\alpha) \coloneqq \bigcup_{\substack{(\xi,\xi')\in\NegIndecisiveDrifts(\alpha)}} \mTile(\xi,\xi') \cap \bar{S}(b(\NegDriftCells(\xi,\xi'))).
\]

\end{defn}

\begin{ex}\label{ex:Add&Rem_ex}
    Recall from Example~\ref{ex:mtiles2d} that 
    \[
    \xi = \PosDriftCells(\xi,\xi_0') = \NegDriftCells(\xi,\xi_1'),
    \]
    so 
    \[
        \Add(\xi) = \mTile(\xi,\xi'_0) \cap \bar{S}(\blup(\xi'_0)), \quad \Rem(\xi) = \mTile(\xi,\xi'_1) \cap \bar{S}(\blup(\xi)). 
    \]
\end{ex}

Note that if there is no $(\xi,\xi') \in \cD(\rook)$ such that $\alpha = \NegDriftCells(\xi,\xi')$ or $\alpha=\PosDriftCells(\xi,\xi')$, then $\Add(\alpha)=\Rem(\alpha)=\emptyset$. Additionally, from Remark~\ref{rem:D+-} and Proposition~\ref{prop:Ts}, we have the following properties of addition and removal sets.

\begin{cor}
\label{cor:prop_Add_Rem}
If $\alpha_1,\alpha_2\in\cX$ are distinct cells, then 
\[
\Add(\alpha_1) \cap \Add(\alpha_2) = \Rem(\alpha_1) \cap \Rem(\alpha_2) 
    = \Add(\alpha_1) \cap \Rem(\alpha_1)
    = \emptyset.
\]
\end{cor}

\begin{prop}
\label{prop:Add&RemProp}
If $\alpha \in \cX$, then $
\cl(\Add(\alpha)) \cap \cl(\Rem(\alpha)) = \emptyset.$
\end{prop}

\begin{proof}
Given Remark~\ref{rem:D+-}, it is enough to consider distinct pairs that exhibit indecisive drift. If $(\xi_0,\xi_0'),(\xi_1,\xi_1') \in \cD(\rook)$ are distinct and $\cl(\Add(\alpha)) \cap \cl(\Rem(\alpha)) \neq \emptyset$, then 
\[
\cl (\mTile(\xi_0,\xi_0')) \cap \cl (\mTile(\xi_1,\xi_1'))\neq\emptyset,
\] 
and by Proposition~\ref{prop:back_subset}, $\Ex(\xi_0,\xi_0')=\Ex(\xi_1,\xi_1')$. Therefore, the spines in $\cX_b$ are separated by at least 1, hence $\PosNegDriftCells(\xi_0,\xi_0')$ and $\PosNegDriftCells(\xi_1,\xi_1')$ are separated by at least 8 in $\Janus$, thus $\mTile(\xi_0,\xi_0')$ and $\mTile(\xi_1,\xi_1')$ are disjoint, and so are $\Add(\alpha)$ and $\Rem(\alpha)$. 
\end{proof}

\begin{prop}
\label{prop:Wdiff}
Let $\alpha_1, \alpha_2 \in \cX$. 
If 
$\PosIndecisiveDrifts(\alpha_1)\cap \NegIndecisiveDrifts(\alpha_2)=\emptyset,$ then 
\[
\cl (\Add(\alpha_1)) \cap \cl(\Rem(\alpha_2)) = \emptyset.
\]
\end{prop}
\begin{proof}
    By definition, $\Add(\alpha_1) \cap \bar{S}(\blup(\alpha_1)) = \emptyset$ and $\Rem(\alpha_2) \subset \bar{S}(\blup(\alpha_2)) = \emptyset$. If $\alpha_i$ is not a face of $\alpha_j$ for any distinct $i,j\in\setof{1,2}$ then $\cl (\Add(\alpha_i)) \cap \cl(\Rem(\alpha_j)) = \emptyset$.\\
    Assume, with loss of generality, that $\alpha_1\prec \alpha_2$. We will prove by contradiction, thus suppose that  
    \begin{equation}\label{eq:mTile_not_empty}
        \cl (\mTile(\xi_0,\xi_0')) \cap \cl (\mTile(\xi_1,\xi_1'))\neq\emptyset
    \end{equation}
    for some $(\xi_0,\xi_0')\in \PosIndecisiveDrifts(\alpha_1)$ and $(\xi_1,\xi_1')\in \NegIndecisiveDrifts(\alpha_2)$. 
    Then $(\xi_0,\xi_0')\neq (\xi_1,\xi_1')$ since $\alpha_1\prec \alpha_2$. Proposition~\ref{prop:back_subset}, $\Ex(\xi_0,\xi_0')=\Ex(\xi_1,\xi_1') = \setof{n_o}$.\\
    If $(\xi_0,\xi_0')$ and $(\xi_1,\xi_1')$ have a common element, by Proposition \ref{prop:Ts}, 
    \[
    \cl (\mTile(\xi_0,\xi_0')) \cap \cl (\mTile(\xi_1,\xi_1')) = \emptyset,
    \]
    a contradiction. Thus, assume that $(\xi_0,\xi_0')$ and $(\xi_1,\xi_1')$ does not have any element in common, hence, $\xi_0 =\alpha_1$ and $\xi_1=\alpha_2$, since $\Ex(\xi_0,\xi_0')=\Ex(\xi_1,\xi_1') = \setof{n_o}$. As a consequence 
    \begin{equation}\label{eq:same_position}
        p_{n_o}(\xi_0,\xi_0')=p_{n_o}(\xi_1,\xi_1'). 
    \end{equation}\\
    Notice that, Equation \eqref{eq:mTile_not_empty} implies that $\setof{\xi_0,\xi_0'}\cap\setof{\xi_1,\xi_1'}\neq\emptyset$, i.e., there exists $\mu\in\Top_\cX(\xi_1)\cap \Top_\cX(\xi_0)$ such that $D_{(\xi_0,\xi_0')}(\mu) \neq D_{(\xi_1,\xi_1')}(\mu)\neq 0$. By Proposition \ref{prop:drift}, $D_{(\xi_0,\xi_0')}(\mu)=D_{(\xi_1,\xi_1')}(\mu)$, i.e., $\direc(\xi_0,\xi_0')=\direc(\xi_1,\xi_1')$. Hence, by Equation \eqref{eq:same_position}, we have that 
    $\PosDriftCells(\xi_0,\xi_0') = \PosDriftCells(\xi_1,\xi_1')$ and $\NegDriftCells(\xi_0,\xi_0') = \NegDriftCells(\xi_1,\xi_1')$. Therefore
    either $(\xi_0,\xi_0') \in \PosIndecisiveDrifts(\alpha_1)$ and $(\xi_1,\xi_1') \in \PosIndecisiveDrifts(\alpha_2)$, or $(\xi_0,\xi_0') \in \NegIndecisiveDrifts(\alpha_1)$ and $(\xi_1,\xi_1') \in \NegIndecisiveDrifts(\alpha_2)$, a contradiction to the hypotheses.
\end{proof}
\begin{defn}
\label{defn:pi2fibers}   
For each $\alpha\in\cX$ the associated \emph{$2$-fiber} is defined by
\[
\Pl(\alpha) := \bar{S}(\blup(\alpha))\cup \Add(\alpha) - \Rem(\alpha) \subset \Janus^{(N)}.
\]
\end{defn}
\begin{ex}\label{ex:pi2fibers}
    Continuing Example~\ref{ex:Add&Rem_ex}, we have that 
    \[
    \Pl(\xi) := \bar{S}(\blup(\xi))\cup (\mTile(\xi,\xi'_0)\cap \bar{S}(\blup(\xi_0')) - \mTile(\xi,\xi'_1).
    \]
\end{ex}
\begin{prop}
    \label{prop:2FisFibre}  
The $2$-fibers define a partition of $\Janus^{(N)}$, i.e.,
\begin{enumerate}
    \item if $\alpha_1$ and $\alpha_2$ are distinct elements of $\cX$, then $\Pl(\alpha_1)\cap \Pl(\alpha_2)=\emptyset$, and
    \item \(\displaystyle \Janus^{(N)} = \bigcup_{\alpha\in\cX}\Pl(\alpha)\).
\end{enumerate}
\end{prop}
\begin{proof}
To prove item (1) consider distinct cells $\alpha_1,\alpha_2\in\cX$. 
Corollary \ref{cor:prop_Add_Rem} implies that $\Add(\alpha_1) \cap \Add(\alpha_2)=\emptyset$ and $\Rem(\alpha_1) \cap \Rem(\alpha_2)=\emptyset$.
Thus, 
\[
\Pl(\alpha_1) \cap \Pl(\alpha_2)=\emptyset.
\]
Now, we prove the item (2). Note that given any $(\xi,\xi') \in \cD(\rook)$, 
\begin{align*}
    \Add(\PosDriftCells(\xi,\xi')) = \Rem(\NegDriftCells(\xi,\xi')),
\end{align*}
so any element removed by $\Rem(\NegDriftCells(\xi,\xi'))$ is accounted for by $\Add(\PosDriftCells(\xi,\xi'))$. Thus, 
\[
\bigcup_{\alpha \in \cX} \Pl(\alpha) = \bigcup_{\alpha \in \cX} \bar{S}(b(\alpha)) = \Janus^{(N)}.
\]
\end{proof}

\section{Construction of the Grading $\pi_2$}
\label{sec:pi2}

We begin this section with the definition of the desired grading
$\pi_2\colon \Janus \to \SCC(\cF_2)$.
The remainder of the section is dedicated to showing that $\pi_2$ is an admissible D-grading.

\begin{defn}
\label{defn:pi2}   
Define $\pi_2 \colon \Janus \rightarrow \SCC(\cF_2)$ as follows. 

If $\zeta\in\Janus^{(N)}$, then
\begin{equation}
\label{eq:pi21}
\pi_2(\zeta) := \pi_b(\blup(\alpha))\quad\text{where $\zeta\in \Pl(\alpha)$}.    
\end{equation}
If $\zeta\in\Janus^{(n)}$, $0\leq n <N$, then
\begin{equation}
\label{eq:pi22}
\pi_2(\zeta) := \min\setdef{\pi_2(\mu)}{\mu\in\Top_{\Janus}(\zeta)}.
\end{equation}
\end{defn}
Notice that $\Pl(\xi)$ is the fiber $\pi_2^{-1}(\xi)$ in Figure~\ref{fig:indecisive_drift_2_b1}(D). 

As indicated above the goal of this section is to prove the following theorem.

\begin{thm}
\label{thm:pi2-welldefined}
The function $\pi_2 : \Janus \to \SCC(\cF_2)$ is 
an admissible D-grading.
\end{thm}

In what follows we identify which cells influence the minimum value of the grading $\pi_2$ when $\dim(\zeta) < N$. 
\begin{defn}
\label{defn:PDcells}
Given $(\xi,\xi')\in \cD(\rook)$ and let $\Ex(\xi,\xi')=\setof{n_o}$. 
We define the \emph{Post-Drift cells} of $(\xi,\xi')$ by
\begin{multline*}
    \PDcells(\xi,\xi') \coloneqq \{\xi'' \in \cl(\support\left(D_{(\xi,\xi'))}\right) \ \mid\ \NegDriftCells(\xi,\xi') \prec \xi'',\
     \rmap\xi^{-1}(n_o)\cap \rmap{\xi''}^{-1}(n_o) = \emptyset, \\
     \text{ and } 
    \Ex(\NegDriftCells(\xi,\xi'), \xi'') \subseteq \rmap\xi^{-1}(n_o)\backslash \setof{n_o}
    \}.
\end{multline*}
\end{defn}

\begin{prop}
\label{prop:PDcell-indecisive}
Let $(\xi,\xi')\in \cD(\rook)$ with $\Ex(\xi,\xi')=\setof{n_o}$. 
If $\alpha \in \PDcells(\xi,\xi')$, then there is no $\alpha'\in \cX$ with such that $(\alpha,\alpha')$ or $(\alpha',\alpha)$ exhibit indecisive drift with GO-pair $(n_g,n_o)$ for some $n_g$. 
\end{prop}
\begin{proof}
We provide a proof by contradiction.
Assume that there exists such an $\alpha' \in \cX$ and suppose that $\alpha \prec \alpha'$. 
Let $n_g\in G_i(\alpha')$ actively regulate the opaque direction $n_o$.
By Proposition~\ref{prop:GO-pair-properties}, $n_g \in \rmap\alpha^{-1}(n_o)$ and $n_g \in \rmap{\alpha'}^{-1}(n_o)$.
Since $\alpha \in \PDcells(\xi,\xi')$, $\NegDriftCells(\xi,\xi') \prec \alpha$, and hence, either $\xi \prec \alpha$ or $\xi' \prec \alpha'$. 
In either case, $n_g \in J_i(\xi)$. 
Since $\rmap\xi(n_g) = n_o$, $n_g \in \rmap\xi^{-1}(n_o)$, which contradicts the fact that $\rmap\xi^{-1}(n_o) \cap \rmap\alpha^{-1}(n_o) = \emptyset$. 

The same argument may be used when $\alpha' \prec \alpha$. 
\end{proof}

Our proof of Theorem~\ref{thm:pi2-welldefined} is based on partitioning $\Janus$ into three cases.
\begin{itemize}
    \item[(i)] \(\displaystyle\zeta\in\Janus \backslash \bigcup_{(\xi,\xi')\in \cD(\rook)}\cl(\mTile(\xi,\xi'));\)
    \item[(ii)] 
    \(\displaystyle\zeta \in \bigcup_{(\xi,\xi')\in \cD(\rook)}\cl(\mTile(\xi,\xi')) \backslash \cl\left( \bar S \left( \blup \left(\PDcells(\xi,\xi')\right)\right)\right)\)
    \item[(iii)]  \(\displaystyle\zeta \in \bigcup_{(\xi,\xi')\in \cD(\rook)}\cl(\mTile(\xi,\xi')) \cap  \cl\left( \bar S \left( \blup \left(\PDcells(\xi,\xi')\right)\right)\right).\)
\end{itemize}
The first case is addressed by the following two propositions. 
\begin{prop}
\label{prop:Pl-trivial}
Let $\alpha\in\cX$. 
If there is no $(\xi,\xi)\in\cD(\Phi)$ such that $\xi=\alpha$ or $\xi'=\alpha$, 
then $\Pl(\alpha) = \bar{S}(\blup(\alpha))$. 
Moreover, $\gradJanus(\mu)=\pi_2(\mu)$, for any $\mu\in\Pl(\alpha)$.
\end{prop}

\begin{proof}
Note that {\bf Condition 2.1} is not satisfied and therefore $\cF_2(\alpha)=\cF_1(\alpha)$ and $\cF_2^{-1}(\alpha)=\cF_1^{-1}(\alpha)$.
Therefore, $\Pl(\alpha) = \bar{S}(\blup(\alpha))$. 
Furthermore, \eqref{eq:pi_Janus} implies that $\gradJanus(\zeta)=\pi_2(\zeta)$, for any $\zeta\in \bar{S}(\blup(\alpha))=\Pl(\alpha)$.
\end{proof}
\begin{prop}\label{prop:same-grading2_1}
    If $\zeta\in\Janus \backslash \bigcup_{(\xi,\xi')\in \cD(\rook)}\cl(\mTile(\xi,\xi'))$, then
    $\pi_2(\zeta)=\gradJanus(\zeta)$.
\end{prop}
\begin{proof}
    If $\pi_2(\zeta) \neq \gradJanus(\zeta)$, then there exists $\mu \in \Top_{\Janus}(\zeta)$ such that 
    \[
        \pi_2(\mu) \neq \gradJanus(\mu).
    \]
    Therefore, $\mu \in \mTile(\xi,\xi')$ for some $(\xi,\xi') \in \cD(\rook)$, so $\zeta \in \cl(\mTile(\xi,\xi'))$. 
\end{proof}

To address the second case, we make us of the following results.
\begin{prop}
    \label{prop:same-grading2_2}
    Let $(\xi,\xi') \in \cD(\rook)$. If 
    \[\zeta \in \cl(\mTile(\xi,\xi'))\backslash \cl\left( \bar S \left( \blup \left(\PDcells(\xi,\xi')\right)\right)\right),\] then there exists $\beta \in \cX_b$ such that 
    \[
    \setdef{\pi_2(\mu)}{\mu \in \Top_{\Janus}(\zeta)} = \setdef{\pi_b(\mu)}{ \mu \in \Top_{\cX_b}(\beta)} \subset \SCC(\cF_2). 
    \]
\end{prop}
\begin{proof}
    Assume that $(\xi,\xi')$ exhibit indecisive drift and let 
    \[ 
    \zeta \in \cl(\mTile(\xi,\xi'))\backslash \cl\left( \bar S \left( \blup \left(\PDcells(\xi,\xi')\right)\right)\right). 
    \]
    Let $\mathcal{A} = \setdef{ \alpha \in \cX }{\Top_{\Janus}(\zeta) \cap \Pl(\alpha) \neq \emptyset}$. Note that $\mathcal{A} \neq \emptyset$ since either $\xi$ or $\xi'$ is in $\mathcal{A}$. Let $\beta = \wedge_{\alpha \in \mathcal{A}} \blup(\alpha)$ be the cell with highest dimension in $\bigcap_{\alpha \in \mathcal{A}} \cl(b(\alpha))$.
    By construction of $\mathcal{A}$,
    \[
        \setdef{\pi_2(\mu)}{\mu \in \Top_{\Janus}(\zeta)} \subseteq \setdef{ \pi_b(\mu)}{\mu \in \Top_{\cX_b}(\beta)}. 
    \]
    Conversely, if $\mu \in \Top_{\cX_b}(\beta)$, then $b^{-1}(\mu) \in \mathcal{A}$. Note that for any $\mu' \in \Top_{\Janus}(\zeta) \cap \Pl(\mu)$, $\pi_2(\mu')=\pi_b(\mu)$, hence 
    \[
        \setdef{\pi_2(\mu)}{\mu \in \Top_{\Janus}(\zeta)} \supseteq \setdef{ \pi_b(\mu)}{\mu \in \Top_{\cX_b}(\beta)}. 
    \]
\end{proof}
As the next proposition shows the following cells 
\begin{equation}
\label{eq:Dzeta}
D(\zeta) := \setdef{\xi}{\zeta\in\cl(\mTile(\tilde\xi,\tilde\xi'))
\text{ or $\zeta\in\cl(\mTile(\tilde\xi',\tilde\xi))$}}.    
\end{equation}
provide an upper bound on post drift cells.  

\begin{lem}\label{lem:Dzeta}
If $(\xi,\xi') \in \cD(\rook)$, $\tau\in \PDcells(\xi,\xi')$ and 
$\zeta \in \cl\left( \bar S \left( \blup \left(\PDcells(\xi,\xi')\right)\right)\right)$, then
\begin{equation}\label{eq:inequalities_rule2}
\pi(\tau) \posetF{2} \pi(\delta),
\end{equation}
for any $\delta\in D(\zeta)$. 
\end{lem}
\begin{proof}
    Let $(\xi,\xi') \in \cD(\zeta)$ with $\Ex(\xi,\xi')=\setof{n_o}$. 
    
    For simplicity, we denote $\xi^+ = \PosDriftCells(\xi,\xi')$ and $\xi^- = \NegDriftCells(\xi,\xi')$. 
    Since \[
    \zeta \in \cl\left( \bar S \left( \blup \left(\PDcells(\xi,\xi')\right)\right)\right),
    \] 
    then there exists $\tau\in \PDcells(\xi,\xi')$ such that $\zeta\in\cl(\bar S(\blup(\tau)))$. It follows from Definition~\ref{defn:PDcells} that 
    \[
        \Ex(\xi^-,\tau) \subseteq \rmap\xi^{-1}(n_o).
    \]
    Since $(\xi,\xi')$ exhibits indecisive drift (see Definition~\ref{defn:indecisive}), the set $\rmap\xi^{-1}(n_o)$ doesn't contain any other opaque direction other than $n_o$. However, $n_o \notin \Ex(\xi^-,\tau)$, so $\Ex(\xi^-,\tau)$ only consists of gradient directions in $G_i(\xi^-)$. Since $\tau \in \cl(\support(D_{(\xi,\xi')}))$, $\tau \notin \cl(\Back(\xi,\xi'))$. Thus, $p_n(\xi^-,\tau)=-r_{n}$ for all $n \in \Ex(\xi^-,\tau)$. 
    Let $n_1,\ldots,n_k\in\setof{1,\ldots, N}$ such that $\Ex(\xi^-,\tau)=\setof{n_1,\ldots,n_k}.$ Since $\Ex(\xi^-,\tau)$ only contains gradient directions, then $\Ex(\xi^-,\tau)$ does not have drift direction (see Definition~\ref{defn:adrift}). Hence, by  Proposition~\ref{prop:mu*} and Lemma~\ref{lem:F_2-leq-geq}
    there are elements $\xi_i \in \cX$ such that 
    \[
        \xi^- = \xi_0 \prec_\cX \xi_1 \prec_\cX \ldots \prec_\cX \xi_k = \tau,
    \]
    such that $\Ex(\xi_{i-1},\xi_i) = \setof{n_i}$ and $\pi(\xi_{i-1}) \posetF{2} \pi(\xi_{i})$. Thus, $\pi(\tau) \posetF{2} \pi(\xi^-)$. 
    Now, let $\tau' \in \cX$ be such that $\xi \prec_\cX \tau'$ and $\Ex(\tau,\tau')=\setof{n_o}$. By the same analysis, $\pi(\tau') \posetF{2} \pi(\xi^+)$. Since $\tau \in \cl(\support(D_{(\xi,\xi')}))$, it must be the case that $\tau \in \cF_2(\tau')$. Thus, $\pi(\tau) \posetF{2} \pi(\tau') \posetF{2} \pi(\xi^+)$. 
\end{proof}
Finally, the third case is considered in what follows.
\begin{prop}
\label{prop:same-grading2_3}
Let $(\xi,\xi') \in \cD(\rook)$. 
If 
\[
\zeta \in \cl(\mTile(\xi,\xi'))\cap \cl\left( \bar S \left( \blup \left(\PDcells(\xi,\xi')\right)\right)\right),
\] 
then there exists $\beta_\zeta \in \cX_b$ such that 
\[
\min\setdef{\pi_2(\mu)}{\mu \in \Top_{\Janus}(\zeta)} = \min\setdef{\pi_b(\mu)}{ \mu \in \Top_{\cX_b}(\beta_\zeta)}. 
\]
\end{prop}

\begin{proof}
Given $\zeta$ as in the hypothesis, set
\[
\beta_\zeta := \wedge_\preceq \setdef{\blup(\alpha)}{\zeta\in \cl(\blup(\alpha))}.
\]
We claim that
\begin{align}
  \setdef{\gradJanus(\mu)}{\mu \in \Top_{\Janus}(\zeta)} &= \setdef{\pi_b(\mu)}{ \mu \in s(\Top_{\Janus}(\zeta))}\label{eq:piJpib} \\
  &= \setdef{\pi_b(\mu)}{ \mu \in \Top_{\cX_b}(\beta_\zeta)}.   
\end{align}
where $s(\mu):= \mu'$ and $\mu \in \bar{S}(\mu')$.

Indeed, assume the second equality follows from the fact that $s(\Top_{\Janus}(\zeta))=\Top_{\cX_b}(\beta_\zeta)$. Now we prove the first equality. Assume
that $\mu\in \Top_{\Janus}(\zeta)$. By \eqref{eq:pi_Janus}, there exists $\alpha\in\cX^{(N)}$ such that $\mu\in\bar S(\blup(\alpha))$ and $\gradJanus(\mu) = \pi(\alpha)$. Since $\mu\in \Top_{\Janus}(\zeta)$ and $\zeta\in \cl(\bar S(\beta_\zeta))$, it follows that $b(\alpha)\in\Top_{\cX_b}(\beta_\zeta)$, thus $s(\mu)=b(\alpha)$. Hence,
    \[
    \gradJanus(\mu) = \pi(\alpha) = \pi_b(\blup(\alpha)) = \pi_b(s(\mu)).
    \]

    Conversely, assume that $\mu'\in s(\Top_{\Janus}(\zeta))$. We denote $\alpha = \blup^{-1}(\mu')$ so that from \eqref{eq:pi_Janus}, we derive
    \[
    \pi_b(\mu') = \pi(\alpha) =  \gradJanus(\mu),
    \]
    for all $\mu\in \bar S(\blup(\alpha))$. Therefore any $\mu\in s^{-1}(\mu')$ is in $\Top_{\Janus}(\zeta)$ since $s^{-1}(\mu') \subseteq \bar S(\blup(\alpha)) = \bar S(\mu')$.
Thus, $\pi_b(\mu')=\gradJanus(\mu)$ and  
\eqref{eq:piJpib} 
follows.

Since $\zeta \in \cl\left( \bar S \left( \blup \left(\PDcells(\xi,\xi')\right)\right)\right)$ there exists $\tau\in \PDcells(\xi,\xi')$ such that $\zeta\in\cl(\bar S(\blup(\tau)))$. 
Hence, \eqref{eq:inequalities_rule2} implies that 
\begin{equation}\label{eq:min_pib}
       \min\setdef{\pi_b(\mu)}{ \mu \in \Top_{\cX_b}(\beta_\zeta)}
    =
    \min\left(\setdef{\pi_b(\mu)}{ \mu \in \Top_{\cX_b}(\beta_\zeta)\setminus \blup(D(\zeta))}\right).  
\end{equation}
and
\begin{equation}\label{eq:min_pi2}
    \min\left(\setdef{\pi_2(\mu)}{ \mu \in \Top_{\Janus}(\zeta)\setminus\Pl(D(\zeta))}\right)
    =
    \min\setdef{\pi_2(\mu)}{ \mu \in \Top_{\Janus}(\zeta)},    
\end{equation}
where $D(\zeta)$ is given by \eqref{eq:Dzeta}.

Observe that
\begin{align*}
    \min\setdef{\pi_b(\mu)}{ \mu \in \Top_{\cX_b}(\beta_\zeta)}
    &=
    \min\left(\setdef{\pi_b(\mu)}{ \mu \in \Top_{\cX_b}(\beta_\zeta)\setminus \blup(D(\zeta))}\right) \\
    &=
    \min\left(\setdef{\gradJanus(\mu)}{ \mu \in \Top_{\Janus}(\zeta)\setminus \bar S(\blup(D(\zeta)))}\right)\\
    &=
    \min\left(\setdef{\pi_2(\mu)}{ \mu \in \Top_{\Janus}(\zeta)\setminus \Pl(D(\zeta))}\right)\\
    &=
    \min\setdef{\pi_2(\mu)}{ \mu \in \Top_{\Janus}(\zeta)},
\end{align*}

where the first and the last equalities follow from \eqref{eq:min_pib} and \eqref{eq:min_pi2}, respectively, the second equality follows from 
\eqref{eq:piJpib}
and the third equality comes from the fact that the modification done in $\left.\gradJanus\right|_{\Top_{\Janus}(\zeta)}$ to obtain $\left.\pi_2\right|_{\Top_{\Janus}(\zeta)}$ consists in changing $\gradJanus(\mu)=\pi(\NegDriftCells(\tilde\xi,\tilde\xi'))$ to 
    $\pi_2(\mu)=\pi(\PosDriftCells(\tilde\xi,\tilde\xi'))$ for some cells $\mu\in\Top_{\Janus}(\zeta)$ with $\PosDriftCells(\tilde\xi,\tilde\xi'), \NegDriftCells(\tilde\xi,\tilde\xi')\in D(\zeta)$.
\end{proof}
\begin{proof}[Proof of Theorem~\ref{thm:pi2-welldefined}]
    Let $\zeta \in \Janus$. If $\dim(\zeta) = N$, Proposition \ref{prop:2FisFibre} implies that there exists a unique $\alpha\in\cX$ such that $\zeta$ in $\Pl(\alpha)$.
    Therefore, \eqref{eq:pi21} is uniquely defined. 

    If $\dim(\zeta) < N$, then by Proposition~\ref{prop:same-grading2_1}, Proposition~\ref{prop:same-grading2_2} and Proposition~\ref{prop:same-grading2_3}, it follows that
    \[
    \min\setdef{\pi_2(\mu)}{\mu \in \Top_{\Janus}(\zeta)} = \min\setdef{\pi_b(\mu)}{ \mu \in \Top_{\cX_b}(\beta)},
    \]
    for some $\beta \in \cX_b$. 
    Notice that Theorem~\ref{thm:singlevalue} implies that $\pi_b$ is an admissible D-grading, i.e., $\setdef{\pi_b(\mu)}{ \mu \in \Top_{\cX_b}(\beta)}$ has a unique minimum element. Therefore $\pi_2$ is also an admissible D-grading.
\end{proof}

\section{Chain Equivalence of $(C_\ast(\Janus), \gradJanus)$ and $(C_\ast(\Janus), \pi_2)$.}
\label{sec:pi2Equivalent}

The goal of this section is to prove the following theorem.
\begin{thm}
\label{thm:gradedchainequiv}
Consider $\cF_2\colon \cX\mvmap \cX$ with D-grading $\pi\colon \cX \to \SCC(\cF_2)$ as in \ref{def:Rule2}. Then, 
$(C_\ast(\Janus), \gradJanus)$ and $(C_\ast(\Janus), \pi_2)$ are $\SCC(\cF_2)$-graded chain equivalent.
\end{thm}
To achieve this, we turn our focus to the subcomplex of $\Janus$ containing the local modifications done for a grading value $p\in\SCC(\cF_2)$, denoted by 
\[\mFiber(p) := \left((\pi_2)^{-1}(p) -  (\gradJanus)^{-1}(p) \right) \cup \left( (\gradJanus)^{-1}(p) - (\pi_2)^{-1}(p)
\right)\] 
the symmetric difference of the complexes $(\pi_2)^{-1}(p)$ and  $(\gradJanus)^{-1}(p)$. 
Given $p\in\SCC(\cF_2)$ denote 
\[\cX^p:= \pi^{-1}(p).\] 
Set 
\[\mAdd(\alpha, p):=\ \cl 
\left(\Add(\alpha)\right) \cap \mFiber(p)\]
and
\[\mRem(\alpha, p):=\ \cl 
\left(\Rem(\alpha)\right) \cap \mFiber(p),\]
for $\alpha\in\cX^p$. Observe that, $\mFiber(p)$ can be represented as \[\mFiber(p) = \bigcup_{\alpha\in\cX^p} \mAdd(\alpha, p)\cup \mRem(\alpha, p).\] 
\begin{ex}
\label{ex:mF2}
Recall from Example~\ref{ex:mtiles2d} that $\mTile(\xi,\xi'_0)$ and $\mTile(\xi,\xi'_1)$ are the top cells in the left and right gray region in Figure~\ref{subfig:indecisive_drift_c}, respectively.
From Equation \eqref{eq:Dgrading} and Example \ref{ex:pi2fibers}, we have that 
$\pi(\xi) = p_4$ and 
\[\Pl(\xi)= \bar S(\blup(\xi))\cup \mTile(\xi,\xi'_0) - \mTile(\xi,\xi'_1),\] then 
\[\pi_2^{-1}(p_4)=\Pl(\xi)=\bar S(\blup(\xi))\cup \mTile(\xi,\xi'_0) - \mTile(\xi,\xi'_1)\] 
and $\gradJanus^{-1}(p_4)=\bar S(\blup(\xi))$. 
Using the fibers $\pi_2^{-1}(p_4)$ and $\gradJanus^{-1}(p_4)$ we compute
\begin{align*}
    \mFiber(p_4) &=  \left((\pi_2)^{-1}(p_4) -  (\gradJanus)^{-1}(p_4) \right) \cup \left( (\gradJanus)^{-1}(p_4) - (\pi_2)^{-1}(p_4)
\right)
\\ &= \mTile(\xi,\xi'_0) \cup \mTile(\xi,\xi'_1).
\end{align*}
Furthermore, 
\[\mAdd(\xi,p_4) = \mTile(\xi,\xi'_0) \quad \text{and} \quad \mRem(\xi,p_4) = \mTile(\xi,\xi'_1).\] 
Another way to represent $\mFiber(p_4)$ is
\[\mFiber(p_4) = \mAdd(\xi,p_4) \cup \mRem(\xi,p_4) = \mTile(\xi,\xi'_0) \cup \mTile(\xi,\xi'_1).\]
\end{ex}

\begin{lem}\label{lem:no_overlap}
    Given $p\in\SCC(\cF_2)$, then
    \[ \left( \bigcup_{\alpha\in\cX^p} \mAdd(\alpha, p) \right) \cap \left( \bigcup_{\alpha\in\cX^p} \mRem(\alpha, p) \right)=\emptyset.\]
\end{lem}
\begin{proof}
    It is enough to prove that $\mAdd(\alpha_1, p) \cap \mRem(\alpha_2, p) = \emptyset$, for any $p\in\SCC(\cF_2)$, $\alpha_1, \alpha_2\in\cX^p$. 
    
    First, assume that $\alpha_1=\alpha_2$. Thus the result follows from Proposition \ref{prop:Add&RemProp}.

Second, assume that $\alpha_1\neq\alpha_2$. Hence
\[\mAdd(\alpha_1, p) \cap \mRem(\alpha_2, p) =\ \cl 
\left(\Add(\alpha_1)\right) \cap \cl 
\left(\Rem(\alpha_2)\right) \cap  \mFiber(p).\]
Notice that, if either $\PosIndecisiveDrifts(\alpha_1)$ or $\NegIndecisiveDrifts(\alpha_2)$ (or both) is empty, then \[\cl 
\left(\Add(\alpha_1)\right) \cap \cl 
\left(\Rem(\alpha_2)\right) \cap  \mFiber(p)=\emptyset,\] 
since $\Add(\alpha_1)=\emptyset$ or $\Rem(\alpha_2)=\emptyset$. Therefore, the result follows. 

Now, assume that there exist $(\xi,\xi')\in \PosIndecisiveDrifts(\alpha_1)$ and $(\sigma,\sigma')\in \NegIndecisiveDrifts(\alpha_2)$. Hence
\begin{equation*}
      \cl 
\left(\Add(\alpha_1)\right) \cap \cl 
\left(\Rem(\alpha_2)\right) \cap  \mFiber(p) = \cl (\mTile(\xi,\xi')) \cap \cl (\mTile(\sigma,\sigma'))\cap \mFiber(p).  
\end{equation*}

Observe that it is enough to prove that 
\[\cl (\mTile(\xi,\xi')) \cap \cl (\mTile(\sigma,\sigma'))\cap \mFiber(p)=\emptyset\] 
to obtain the result.

First assume that $(\xi,\xi')=(\sigma,\sigma')$, then without loss of generality, assume that $(\xi,\xi')=(\alpha_1,\alpha_2)$. Since $\alpha_1, \alpha_2\in\cX^p$ then 
\[\cl(\mTile(\alpha_1,\alpha_2)) \subset \ctoy(\alpha_1)\cup\ctoy(\alpha_2) \subset (\pi_2)^{-1}(p)\] 
and 
\[\cl(\mTile(\alpha_1,\alpha_2)) \subset \ctoy(\alpha_1)\cup\ctoy(\alpha_2) \subset (\gradJanus)^{-1}(p),\] 
thus
\begin{align*}
    \cl (\mTile(\xi,\xi')) \cap \cl (\mTile(\sigma,\sigma'))\cap \mFiber(p) = \cl (\mTile(\alpha_1,\alpha_2))\cap \mFiber(p)
    = \emptyset,
\end{align*}
where the last equality follows from the definition of $ \mFiber(p)$.

Finally, assume that $(\xi,\xi')$ and $(\sigma,\sigma')$ are distinct pairs. We break into two cases: 
\begin{enumerate}
    \item If all elements in $(\xi,\xi')$ and $(\sigma,\sigma')$ are distinct, 
then 
$(\alpha_2,\alpha_1)\not\in\cD(\rook)$ and $(\alpha_1,\alpha_2)\not\in\cD(\rook)$. 
By Proposition~\ref{prop:Wdiff}, we have that $\cl (\mTile(\xi,\xi')) \cap \cl (\mTile(\sigma,\sigma')))=\emptyset.$
    \item If $(\xi,\xi')$ and $(\sigma,\sigma')$ have a common element, then by Proposition~\ref{prop:Ts}, $\cl (\mTile(\xi,\xi')) \cap \cl (\mTile(\sigma,\sigma'))=\emptyset.$
\end{enumerate}
Both cases imply that
\[
 \cl (\mTile(\xi,\xi')) \cap \cl (\mTile(\sigma,\sigma'))\cap \mFiber(p) = \emptyset \cap \mFiber(p) = \emptyset.
\]
\end{proof}
Let $(\xi,\xi')\in\cD(\rook)$, 
$\Ex(\xi,\xi')=\setof{n_o}$, 
$[\bv'',\bone]=\blup(\xi)$ and  $v=\bv''_{n_o}+(1-p(\xi,\xi'))/2$.
We define the projection of $\cl (\mTile(\xi,\xi'))$ to its face subcomplex $\cl(\mTile(\xi,\xi')) \cap \ctoy(\PosDriftCells_{(\xi, \xi')})$ as
the map
    \[\proj{(\xi,\xi')}:\cl \left(\mTile(\xi,\xi')\right) \rightarrow \ctoy(\PosDriftCells_{(\xi,\xi')})\cap\cl \left(\mTile(\xi,\xi')\right)\]
given by $\proj{(\xi,\xi')}([\bv,\bw]):=[\bv',\bw']$, where 
\[\bv'_n=
\begin{cases}
     v, &\text{if } n=n_o\\
    \bv_n, &\text{else}\\
\end{cases}
\quad \text{ and } \quad
\bw'_n=
\begin{cases}
     0, &\text{if } n=n_o,\\
    \bw_n, &\text{else}.
\end{cases}
\]
It is straightforward to check the following lemma.
\begin{lem}
\label{lem:basic_homotopy1}
Let $(\xi,\xi')\in \cD$ and $p\in\SCC(\cF_2)$, then the projection map $\proj{(\xi,\xi')}$  restricted to $\cl \left( \mTile(\xi,\xi') \right)\cap \mFiber(p)$ induces a null homotopy $h_{(\xi,\xi')}$ for $C_\ast(\cl \left(\mTile(\xi,\xi')\right)\cap \mFiber(p))$.
\end{lem}
The next lemma is a direct consequence of \ref{lem:basic_homotopy1}.
\begin{lem}\label{lem:basic_homotopy2}
    Let $h_{(\xi,\xi')}$ be the null homotopy from Lemma \ref{lem:basic_homotopy1}, for any $(\xi,\xi')\in \cD$. Then, the chain maps 
    \[h_{\alpha}^+=\bigoplus_{(\xi,\xi')\in \PosIndecisiveDrifts(\alpha)} h_{(\xi,\xi')}  \quad \text{ and } \quad  h_{\alpha}^-= \bigoplus_{(\xi,\xi')\in \NegIndecisiveDrifts(\alpha)} h_{(\xi,\xi')}\] are null homotopies for $C_\ast(\mAdd(\alpha, p))$ and $C_\ast(\mRem(\alpha, p))$, respectively.
\end{lem}

\begin{lem}\label{lem:W-homotopies}
    Let $p\in\SCC(\cF_2)$ and $\alpha_1, \alpha_2\in \cX^p$. Then 
    $C_\ast(\mAdd(\alpha_1, p)\cap \mAdd(\alpha_2, p))$ and $C_\ast(\mRem(\alpha_1, p)\cap \mRem(\alpha_2, p))$ are contractible. Moreover,
\[C\left(\bigcup_{\alpha\in\cX^p} \mAdd(\alpha, p)\right) \quad \text{ and } \quad C\left(\bigcup_{\alpha\in\cX^p} \mRem(\alpha, p)\right)\] are contractible.
\end{lem}
\begin{proof}
 From Lemma \ref{lem:basic_homotopy2}, we have the null homotopies $h_{\alpha_1}^+$ and $h_{\alpha_2}^+$ for $C_\ast(\mAdd(\alpha_1, p))$ and $ C_\ast(\mAdd(\alpha_2, p))$, respectively. Since, $h_{\alpha_1}^+$ and $h_{\alpha_2}^+$ are the same chain map when defined over $\mAdd(\alpha_1, p)\cap \mAdd(\alpha_2, p)$, then $\mAdd(\alpha_1, p)\cap \mAdd(\alpha_2, p)$ is also reduced to the null complex.
 
 It follows from Mayer-Vetoris sequence
 that $C_\ast(\mAdd(\alpha_1, p)) \oplus C_\ast(\mAdd(\alpha_2, p))$ is homotopic equivalent to $C_\ast(\mAdd(\alpha_1, p)\cup \mAdd(\alpha_2, p))$ via the homotopy $h_{\alpha_1}^+\oplus h_{\alpha_2}^+$. 

 By induction, we have that $C\left(\bigcup_{\alpha\in\cX^p} \mAdd(\alpha, p)\right)$ is contractible via \\ $\oplus_{\alpha \in \cX^p} h_\alpha^+$. 
 The same proof follows analogously for $C\left(\bigcup_{\alpha\in\cX^p} \mRem(\alpha, p)\right)$ if we change $\mAdd$ for $\mRem$ and $h^+$ for $h^-$, respectively.
\end{proof}
Given $p\in\SCC(\cF_2)$, let $A^p := (\gradJanus)^{-1}(p)$, $B^p := (\pi_2)^{-1}(p)$, $A^{\downarrow p} := (\gradJanus)^{-1}(\cl p)$ and $B^{\downarrow p} := (\pi_2)^{-1}(\cl (p))$ be subsets of $\Janus$. We have the following technical result.
\begin{prop}\label{prop:Ap_Bp}
    Let $p\in\SCC(\cF_2)$, then $C_\ast(A^p)$ and $C_\ast(B^p)$ are chain equivalent. Moreover, for $q\in\SCC(\cF_2)$ an immediate successor of $p$, then $C_\ast(A^{\downarrow q})$ and $C_\ast(B^{\downarrow q})$ are chain equivalent.
\end{prop}
\begin{proof}
    Consider the projections  $B^p \twoheadrightarrow A^p\cap B^p$ and $A^p \twoheadrightarrow A^p\cap B^p$, obtained from the projections $\proj{(\xi,\xi')}:\cl (\mTile(\xi,\xi'))\rightarrow \cl(\mTile(\xi,\xi')) \cap \ctoy(\PosDriftCells_{(\xi, \xi')})$
    for any $(\xi,\xi')\in\cD$ such that $\mTile(\xi,\xi')\subset \mFiber(p)$. We will use this projection to obtain two short exact sequences.
    
    First, notice that, the cells added to the fiber $A^p$ are 
    \[\bigcup_{\alpha\in\cX^p} \mAdd(\alpha, p)\] then $\bigcup_{\alpha\in\cX^p} \mAdd(\alpha, p)\subset B^p-B^p\cap A^p$. Hence, we have the following short exact sequence
    \[
    C\left(\bigcup_{\alpha\in\cX^p} \mAdd(\alpha, p)\right) \hookrightarrow  C_\ast(B^p) \twoheadrightarrow C_\ast(A^p\cap B^p),\]
    where $\hookrightarrow$ is the inclusion.
    
    Similarly, we can obtain another short exact sequence by noticing that the cells removed from the fiber $A^p$ are $\bigcup_{\alpha\in\cX^p} \mRem(\alpha, p)$, and by Lemma \ref{lem:no_overlap}, the cells added to $A^p$ and removed from $A^p$ do not overlap, then $\bigcup_{\alpha\in\cX^p} \mRem(\alpha, p)\subset A^p - B^p\cap A^p$. Thus
    \[ C\left(\bigcup_{\alpha\in\cX^p} \mRem(\alpha, p)\right) \hookrightarrow C_\ast(A^p) \twoheadrightarrow C_\ast(A^p\cap B^p).
    \]
    Given that 
    \[C\left(\bigcup_{\alpha\in\cX^p} \mAdd(\alpha, p)\right) \quad \text{ and }  \quad  C\left(\bigcup_{\alpha\in\cX^p} \mRem(\alpha, p)\right)\] are contractible by Lemma \ref{lem:W-homotopies} then $C_\ast(A^p)\simeq C_\ast(A^p\cap B^p) \simeq C_\ast(B^p)$.
    
    Consider the following diagram 
    \begin{equation}\label{eq:diagram_h1_h2}
    \begin{tikzcd}
    0  \arrow[r] & C_\ast(A^{\downarrow p}) \arrow[r, hookrightarrow] \arrow[d, "h_1"] & 
    C_\ast(A^{\downarrow q}) \arrow[r, twoheadrightarrow] \arrow[d, dotted] & 
    C_\ast(A^{\downarrow q}, A^{\downarrow p}) \arrow[d, "h_2"] \arrow[r]  & 0\\
    0  \arrow[r] & C_\ast(B^{\downarrow p}) \arrow[r, hookrightarrow]  &
    C_\ast(B^{\downarrow q}) \arrow[r, twoheadrightarrow]  &
    C_\ast(B^{\downarrow q}, B^{\downarrow p}) \arrow[r]  & 0,
    \end{tikzcd}
    \end{equation}
     where $h_1$ is the chain homotopy derived from induction and $h_2$ is the chain homotopy obtained from $C_\ast(A^q) \simeq C_\ast(B^q)$ and the excision theorem \[C_\ast(A^{\downarrow q}, A^{\downarrow p})\simeq C_\ast(A^q) \text{ and } C_\ast(B^{\downarrow q}, B^{\downarrow p})\simeq C_\ast(B^q).\] 
    Notice that the top and bottom sequences in diagram \eqref{eq:diagram_h1_h2} are split exact sequences, thus 
    \[C_\ast(A^{\downarrow q})\simeq C_\ast(A^{\downarrow p})\oplus C_\ast(A^{\downarrow q}, A^{\downarrow p}) \quad \text{and} \quad C_\ast(B^{\downarrow q})\simeq C_\ast(B^{\downarrow p})\oplus C_\ast(B^{\downarrow q}, B^{\downarrow p}).\]
    Hence, $ C_\ast(A^{\downarrow q})$ is homotopic equivalent to $ C_\ast(B^{\downarrow q})$ via the chain homotopy         
    $h_1\oplus h_2$.
\end{proof}
\begin{proof}[Proof of Theorem~\ref{thm:gradedchainequiv}]
    By applying Proposition~\ref{prop:Ap_Bp} inductively we obtain a $\SCC(\cF_2)$-graded chain homotopy between $(C_\ast(\Janus), \gradJanus)$ and $(C_\ast(\Janus), \pi_2)$.
\end{proof}

\chapter{Global Dynamics of Ramp Systems via $\cF_2$}
\label{sec:R2Dynamics}

The standing hypothesis throughout this chapter is the following:
\begin{description}
    \item[H2] Consider an $N$-dimensional ramp system given by
    \begin{equation}
    \label{eq:rampH2}
    \dot{x} = -\Gamma x + E(x; \nu, \theta, h)
    \end{equation}
    with parameters $(\gamma, \nu, \theta)\in \Lambda(S)$ (see \eqref{eq:lambdaS}) and $h \in \cH_2(\gamma, \nu, \theta)$ (see Definition~\ref{defn:H2}). Let $\cX=\cX(\I)$ be the ramp induced cubical complex, $\omega \colon W(\cX) \to \setof{\pm 1}$ be the associated wall labeling and $\rook : TP(\cX) \to \setof{0,\pm 1}$ be the associated Rook Field (see Section~\ref{sec:ramp2rook}). Let $\cX_b$ and $\Janus$ be the associated blow-up and Janus complexes respectively (see  Section~\ref{sec:blowup_complex} and Section~\ref{sec:janusComplex}). Let $\recG$ and $\recG_J$ be an associated rectangular geometrization of $\cX_b$ and $\Janus$ respectively. Let $\cF_2 \colon \cX \mvmap \cX$ be the associated combinatorial multivalued map (see Section~\ref{sec:quasi-local}) and let $\pi_2 \colon \Janus \to \SCC(\cF_2)$ be the D-grading derived from $\cF_2$ (see Section~\ref{sec:pi2}). 
\end{description}
The primary goal of this chapter is to prove the following theorem. 

\begin{thm}
\label{thm:R2ABlattice}
Given  hypothesis {\bf H2}, there exists a geometrization $\bG_2$ of $\Janus$ that is aligned with $\eqref{eq:rampH2}$ over all $\cN \in \sN(\cF_2)$. 
\end{thm}

Recall from Definition~\ref{defn:aligned} that a geometrization $\bG_2$ is aligned with \eqref{eq:rampH2} over $\cN\in \sInvset^+(\cF_2)$ if given any $(N-1)$-dimensional cell $\zeta\in \bbdy(\blup(\cN))$, it follows that $\bg(\zeta)$ is a smooth $(N-1)$-dimensional manifold over which the vector field is transverse inwards. We provide a constructive proof of Theorem~\ref{thm:R2ABlattice}. 

In order to provide manifolds for each $\zeta \in \bbdy(\blup(\cN))$, Section~\ref{sec:nullclines} introduces the regions of interest and Section~\ref{sec:F2-bounds} describes analytical bounds that allow one to construct such manifolds. Finally, Section~\ref{sec:Geo2} summarizes the construction and verifies Theorem~\ref{thm:R2ABlattice}. 

In light of Theorem~\ref{thm:R2ABlattice}, Theorem~\ref{thm:dynamics} implies that if a ramp system satisfies the hypothesis {\bf H2}, then its global dynamics is characterized by the combinatorial/homological computations of Part~\ref{part:II}. 

\section{Nullclines of Ramp Systems}
\label{sec:nullclines}
Consider $\xi \prec \xi' \in \cX$. 
As indicated in Remark~\ref{rem:g(b(xi))}, via the geometrization we obtain explicit subsets of phase space $\bg(\blup(\xi)),\bg(\blup(\xi'))\subset [0,\infty)^N$. 
Let
\begin{equation}
\label{eq:nuln}
\Nul_{n}(\xi,\xi') := \setdef{ x \in \bg(\blup(\xi))\cap \bg(\blup(\xi')) }{ -\gamma_{n} x_{n} + E_{n}(x;\nu,\theta,h) = 0}
\end{equation}
denote the intersection of the $n$-nullcline with the geometric realization of the shared faces of $\blup(\xi)$ and $\blup(\xi')$.

\begin{prop}
\label{prop:opaque-implies-nontrivial-nullcline}
If $\xi,\xi'\in\cX$ and $\xi \darrow_{\cF_1}\xi'$, then $\Nul_{n}(\xi,\xi') \neq \emptyset$, where $\Ex(\xi,\xi')=\setof{n}$. 
\end{prop}
\begin{proof}
Let $\xi, \xi'$ be cells in $\cX$ and assume $\xi \darrow_{\cF_1} \xi'$. 
By Definition~\ref{def:Rule1}, $\xi \prec \xi'$ and $\dim(\xi')=\dim(\xi)+1$. 
Since $\xi \darrow_{\cF_1} \xi'$, $\xi \notin E^\pm(\xi')$, so there are top cells $\mu^{+},\mu^{-} \in \Top_\cX(\xi')$ such that $\Phi_{n}(\xi,\mu^{+}) = 1$ and $\Phi_{n}(\xi,\mu^{-}) = -1$.
By Definition~\ref{def:rookfield}, there exist $n$-walls $(\xi_{n}^{+},\mu^{+}),(\xi_{n}^{-},\mu^{-}) \in W(\cX)$ such that $\omega(\xi_{n}^{+},\mu^{+}) = 1$ and $\omega(\xi_{n}^{-},\mu^{-})=-1$, with $\xi \preceq \xi_{n}^{\pm}$ and $p_n(\xi_{n}^{\pm},\mu^{\pm})=p_n(\xi,\xi')$. 
    
Note that for any $x \in \bg(\blup(\xi)) \cap \bg(\blup(\xi'))$, Definition~\ref{defn:rectGeoXb} implies that
    \[
        x_n = \theta_{m_{k_n},n,j_{k_n}} - p_n(\xi,\xi') h_{m_{k_n},n,j_{k_n}}
    \]
    for some $m_{k_n}$ and $j_{k_n}$ (see \eqref{eq:Jn}).
    By Definition~\ref{defn:ramp-wall-labeling}, 
    \[
    \sgn( -\gamma_n \theta_{m_{k_n}, n, j_{k_n}} + E_n(\mu^{+}) ) = 1 \quad\text{and}\quad \sgn( -\gamma_n \theta_{m_{k_n}, n, j_{k_n}} + E_n(\mu^{-}) ) = -1.
    \]
By Definition~\ref{defn:H1}, $h \in \cH_2(\gamma,\nu,\theta) \subset \cH_1(\gamma,\nu,\theta)$ implies that the sign is constant for any value between $\theta_{m_{k_n},n,j_{k_n}}$ and $\theta_{m_{k_n},n,j_{k_n}}- p_n(\xi,\xi')h_{m_{k_n},n,j_{k_n}}$, so
    \begin{align*}
        -\gamma_n \left( \theta_{m_{k_n}, n, j_{k_n}}- p_n(\xi,\xi') h_{m_{k_n},n,j_{k_n}} \right) + E_n(\mu^{+}) & > 0,\ \text{and} \\
        -\gamma_n \left( \theta_{m_{k_n}, n, j_{k_n}}- p_n(\xi,\xi') h_{m_{k_n},n,j_{k_n}} \right) + E_n(\mu^{-}) & < 0.
    \end{align*}
    Define the points $x^{+},x^{-} \in [0,\infty)^N$ by
    \[
        x_i^{\pm} = \theta_{m_{k_i},i,j_{k_i}} - p_i(\xi,\mu^{\pm}) h_{m_{k_i},i,j_{k_i}},
    \]
    and note that 
    \[
    x^{\pm} \in \left( \bg(\blup(\xi))\cap \bg(\blup(\xi')) \right) \cap \left( \bg(\blup(\xi_{n}^{\pm})) \cap \bg(\blup(\mu^{\pm}) \right). 
    \]
    Thus, there exist points $x^{\pm} \in \bg(\blup(\xi))\cap \bg(\blup(\xi'))$ with 
    \[
        \sgn\left( -\gamma_n x_{n}^{\pm} + E_{n}(x^{\pm},\nu,\theta,h) \right) = \pm 1.
    \]
    Therefore, $\Nul_{n}(\xi,\xi') \neq \emptyset$ by the Intermediate Value Theorem. 
\end{proof}

\begin{rem}
Recall by Definition~\ref{defn:rectGeoXr} that $\recG_\mathtt{J}$ is a rectangular geometrization of $\Janus$.
Thus, Theorem~\ref{thm:R1ABlattice} remains valid in the context of the Janus complex when $\recG$ is replaced by $\recG_\mathtt{J}$. 
The same is true for Proposition~\ref{prop:opaque-implies-nontrivial-nullcline}.
\end{rem}

For technical reasons it is more convenient to work on regions in phase space slightly larger or smaller than those defined directly by $\recG$ (see Remark~\ref{rem:g(b(xi))}).
With this in mind, we define the following.
\begin{defn}
\label{defn:admissible-delta}
We say that $\delta \in [0,\infty)$ is
an \emph{admissible spatial perturbation}
if for each $n \in \setof{1,\ldots,N}$ and each top cell $\mu \in \Top_\cX(\cX)$
\begin{equation}
\label{eq:delta-restriction}
\delta < \frac{\Xi_n(\mu)}{2}, 
\end{equation}
where $\Xi_n(\mu)$ is the length of the interval $I_n(b(\mu))$ (see \eqref{eq:Xi}).
    If $\delta$ is admissible, for each $\xi \in \cX$ and $n \in \setof{1,\ldots,N}$, we define the intervals $I_n^\delta(\xi)$ by
    \begin{equation}
    \label{eq:Idelta}
    I_n^\delta(\xi) \coloneqq \begin{cases}
            [\theta_{m_k,n,j_k}-h_{m_k,n,j_k}-\delta,\theta_{m_k,n,j_k}+h_{m_k,n,j_k}+\delta], & n \in J_i(\xi) \\
                [\theta_{m_k,n,j_k}+h_{m_k,n,j_k}+\delta,
            \theta_{m_{k+1},n,j_{k+1}}-h_{m_{k+1},n,j_{k+1}}-\delta], & n \in J_e(\xi)
            \end{cases}    
    \end{equation}
    where $(m_k,j_k),(m_{k+1},j_{k+1}) \in \mathcal{J}(n)$ as in \eqref{eq:Jn}. Given a subset of indices $\cI \subseteq \setof{1,\ldots,N}$, define the rectangular region
    \begin{equation}
        \label{eq:Qdelta}
        Q^\delta(\xi,\cI) \coloneqq \prod_{n \in \cI} I_n(\blup(\xi)) \times \prod_{n \notin \cI} I_n^\delta(\xi) \subset [0,\infty)^N.
    \end{equation}
\end{defn}
\begin{defn}
\label{defn:Ixi}
Assume that $(\xi,\xi') \in \cX \times \cX$ has a GO-pair $(n_g,n_o)$. 
The \emph{immutable indices of $(\xi,\xi')$} are
    \begin{equation}
        \label{eq:Ixi}
        \cI(\xi,\xi') = \begin{cases}
            \setof{n_g,n_o} & \text{if } \dim(\xi) = N-2 \\ 
            \setof{n_o} & \text{if } \dim(\xi) < N-2
        \end{cases}
    \end{equation}
\end{defn}
Observe that if $n \in \cI$, the rectangular region $Q^\delta(\xi,\cI)$ is obtained via the same interval $I_n(\blup(\xi))$ that defined $\bg(\blup(\xi))=\prod I_n(\blup(\xi))$ (see Remark~\ref{rem:g(b(xi))}). However, if $n \notin \cI$, the interval $I_n^\delta(\xi)$ is slightly larger when $n \in J_i(\xi)$ or smaller when $n \in J_e(\xi)$. 
To accommodate that change in \eqref{eq:Idelta}, we obtain the following modification of the intersection of the $n$-nullcline of shared faces of $\bg(\blup(\xi))$ and $\bg(\blup(\xi'))$ for any rectangular geometrization $\recG$ and $\xi\prec\xi'\in \cX$

\begin{equation}
\label{eq:nulndelta}   
 \Nul_{n}^\delta(\xi,\xi') = \setdef{x \in Q^\delta(\xi,\cI(\xi,\xi')) \cap Q^\delta(\xi',\cI(\xi,\xi'))}{-\gamma_n x_n + E_n(x;\nu,\theta,h)=0}. 
\end{equation}

\begin{cor}
    \label{cor:delta-nullcline}
If $\xi,\xi'\in\cX$, $\xi \preceq \xi'$, $\xi \darrow_{\cF_1}\xi'$, and $\Ex(\xi,\xi')=\setof{n_o}$, then for all $\delta >0$
\[
\Nul_{n_o}^\delta(\xi,\xi') \neq \emptyset. 
\]
\end{cor}
\begin{proof}
The result follows directly from Proposition~\ref{prop:opaque-implies-nontrivial-nullcline} since $\Ex(\xi,\xi') \subseteq \cI(\xi,\xi')$ implies that
    \(
        \Nul_{n_o}(\xi,\xi') \subseteq \Nul_{n_o}^\delta(\xi,\xi').
    \)
\end{proof}

\section{Quasi-local bounds on $h$}
\label{sec:F2-bounds}

The desired geometrization $\bG_2$ for $\cF_2$ is obtained from $\recG_\mathtt{J}$ via local modifications.
These local modifications will insure alignment of the geometrization in neighborhoods of points where solutions of the ramp system are tangent to the rectangular faces of $\recG$.
This section provides the estimates needed to insure both that local modifications are possible and that the modifications lead to the desired form of transversality.

Recall from Chapter~\ref{sec:geometrization01} that the images of the vertices in $\cX_b$ are mapped to phase space $[0,\infty)^N$ by means of \eqref{eq:evertices}, which modulo essential subscripts is defined in terms of $\theta - h$ and $\theta + h$.

Therefore, the rectangles (as in \eqref{eq:Rxi}) -- over which we want to control the modifications -- are bounded by either intervals of the form $[\theta-h,\theta+h]$, which has midpoint $\theta$,  or $[\theta +h, \theta'-h']$, which has midpoint $(\theta+h + \theta' - h')/2$.
This latter observation leads to the following definition.
\begin{defn}
    \label{defn:midpoint}
    Given $\xi=[\bv,\bw] \in \cX$ with $\bv=(k_1,\ldots,k_N)$, for each $n \in \setof{1,\ldots,N}$ define the \emph{n-midpoint of $\xi$} by
    \begin{equation}
    \label{eq:IntervalMidpoint}
    \mdpt_n(\xi) := \begin{cases}
        \theta_{m_{k_n},n,j_{k_n}} & \text{if $n \in J_i(\xi)$,} \\
    \frac{(\theta_{m_{k_n+1},n,j_{k_n+1}}-h_{m_{k_n+1},n,j_{k_n+1}}) + (\theta_{m_{k_n},n,j_{k_n}}+h_{m_{k_n},n,j_{k_n}})}{2}, & \text{if $n \in J_e(\xi)$.}
\end{cases}
\end{equation}
The \emph{midpoint of $\xi$} is \( \displaystyle\mdpt(\xi) = \left( \mdpt_1(\xi),\ldots,\mdpt_N(\xi) \right).\)
\end{defn}
\begin{prop}
\label{prop:FlowFromExternalTangency}
Let $\varphi$ denote the flow associated with the ramp system given the standing hypothesis {\bf H2}.
Assume $(\xi,\xi')\in \cX^{(N-2)} \times \cX^{(N-1)}$ has a GO-pair $({n_g},{n_o})$.
Consider $x^0 \in \Nul_{n_o}(\xi,\xi')$ such that  $\varphi((0,s),x^0) \subset \Int(\bg(\blup(\xi')))$ for sufficiently small $s>0$.

Then, there exists a finite exit time $T_{x^0} > 0$ such that $\varphi([0,T_{x^0}],x^0)\subseteq \bg(\blup(\xi'))$, but $\varphi([0,T_{x^0}+t],x^0)\not\subset \bg(\blup(\xi'))$ for all $t>0$.

Furthermore,
\begin{enumerate}
\item[(i)] if $p_{n_o}(\xi,\xi')=1$, then for all $t \in [0,T_{x^0}]$
    \[
    \varphi_{n_o}(t,x^0) > \mdpt_{n_o}(\xi'),
    \]
    and
\item[(ii)] if $p_{n_o}(\xi,\xi')=-1$, then for all $t \in [0,T_{x^0}]$
    \[
    \varphi_{n_o}(t,x^0) < \mdpt_{n_o}(\xi').
    \]
\end{enumerate}  
\end{prop}

\begin{proof}
By Definition~\ref{defn:GO-pair} of GO-pair, the direction $n_g$ is a gradient direction of $\xi'$, that is, $n_g \in G(\xi')$.
Thus, by Lemma~\ref{lem:gradient-direction-is-nonzero}, $|\dot{x}_{n_g}| > 0$  on the compact set $\bg(\blup(\xi'))$. 
Define 
\[
T_{x^0} = \inf \setdef{ t\geq 0 }{ \varphi(t,x^0) \notin \bg(\blup(\xi'))}.
\]
The assumption that $\varphi((0,s),x^0) \subset \Int(\bg(\blup(\xi')))$ for sufficiently small $s>0$ implies that $T_{x^0}>0$, and the assumption that $x^0 \in \Nul_{n_o}(\xi,\xi')$ implies that $x^0_{n_g} \neq \theta_{{n_o},{n_g},j} \pm h_{{n_o},{n_g},j}$. Therefore, the same argument as above implies the existence of
\[
T_{{n_g}}(x^0)  \coloneqq \inf \setdef{ t\geq 0 }{ \varphi_{n_g}(t,x^0) \notin \left[\theta_{{n_o},{n_g},j}-h_{{n_o},{n_g},j},\theta_{{n_o},{n_g},j}+h_{{n_o},{n_g},j}\right]} >0.  
\]
Note that $\varphi(T_{{n_g}}(x^0),x^0) = \theta_{{n_o},{n_g},j}+r_{n_g} h_{{n_o},{n_g},j}$, where $R_{n_g}(\xi') = \setof{r_{n_g}}$, and that by definition $T_{x^0} \leq T_{{n_g}}(x^0)$.

Assume $p_{n_o}(\xi,\xi')=1$.
Define
\begin{equation}
\label{eq:midtime}
T_{n_o}(x^0) \coloneqq \inf\setdef{ t \geq 0 }{ \varphi_{n_o}(t,x^0) < \mdpt_{n_o}(\xi')}.    
\end{equation}
To prove (i) it is sufficient to show that $T_{n_g}(x^0) < T_{n_o}(x^0)$.
To this end note that 
\begin{align*}
        2 h_{{n_o},{n_g},j} & \geq |\varphi_{n_g}(T_{n_g}(x^0)-\varphi_{n_g}(0,x^0)| \\
        & \geq \left| \int_0^{T_{n_g}(x^0)} \dot{x}_{{n_g}} dt \right| \\
        & \geq \inf_{x \in \bg(\blup(\xi'))}|-\gamma_{n_g} x_{n_g} + E_{n_g}(x)| \cdot T_{n_g}(x^0)  \\
        & \geq L_{n_g}(\xi') \cdot T_{n_g}(x^0) > 0,
\end{align*}
where $L_{n_g}(\xi')$ is given by $\eqref{eq:Ln}$.
Similarly, 
    \begin{align*}
        \Xi_{n_o}(\xi') & = \varphi_{n_o}(0,x^0)-\varphi_{n_o}(T_{n_o}(x^0),x^0) \\ 
        & \leq \left| \int_0^{T_{n_o}(x^0)} \dot{x}_{n_o} dt \right| \\
        & \leq \sup_{x \in \bg(\blup(\xi'))} |-\gamma_{n_o}x_{n_o}+E_{n_o}(x)| \cdot T_{n_o}(x_{n_o}) \\
        & \leq U_{n_o}(\xi') \cdot T_{n_o}(x^0), 
    \end{align*}
where  $\Xi_{n_o}(\xi')$ is given by \eqref{eq:IntervalLength} and  $U_{n_o}(\xi')$ by \eqref{eq:Un}.  

By Definition~\ref{defn:H2}, i.e., the assumption that $h\in\cH_2(\nu,\gamma,\theta)$, it follows that
\begin{equation*}
        2h_{{n_o},{n_g},j} < \frac{L_{n_g}(\xi')}{U_{n_o}(\xi')} \frac{\Xi_{n_o}(\xi')}{2}.
    \end{equation*}
Therefore,
\[
T_{n_g}(x^0) \leq \frac{2 h_{{n_o},{n_g},j}}{L_{n_g}(\xi')} < \frac{\Xi_{n_o}(\xi')}{U_{n_o}(\xi')} \leq T_{n_o}(x^0),
\]
which is the desired result.

The proof for (ii) is essentially the same, though the inequality in \eqref{eq:midtime} is reversed.
\end{proof}

We now turn to the question of identifying the manifolds $\bg(\zeta)$ described in the introduction to this Chapter. We begin our discussion by focusing on pairs $(\xi,\xi') \in \cX^{(N-2)} \times \cX^{(N-1)}$ that have a GO-pair $(n_g,n_o)$.
By Proposition~\ref{prop:GO-pair-properties}
\[
J_i(\xi)=\setof{n_g,n_o}\quad\text{and}\quad J_i(\xi')=\setof{n_g}.
\]
Therefore, $n \in J_e(\xi)$ for all $n \neq n_g,n_o$.
\begin{defn}
\label{defn:GO-manifold}
Assume that $(\xi,\xi') \in \cX^{(N-2)} \times \cX^{(N-1)}$ has a GO-pair $(n_g,n_o)$.
Given an admissible spatial perturbation $\delta$ and $\varepsilon>0$, we define the associated \emph{GO-manifold} by
\begin{equation}
\label{eq:GO-manifold}
    \mathcal{M}_{\delta,\varepsilon}(\xi,\xi') \coloneqq \bigcup_{\substack{x^0 \in \Nul_{n_o}^\delta(\xi,\xi')}} \varphi_\varepsilon\left( [0,T_{n_g}(x^0)],x^0 \right).
\end{equation}
where $\varphi_\varepsilon(\cdot,x^0)$ are solutions of $\dot{x}=F_\varepsilon(x)$ given by
\begin{equation}
    \label{eq:rampSysPerturbed}
    \dot{x}_n = \begin{cases} 0 & n \neq n_o,n_g \\
        -\gamma_{n_o} x_{n_o} + E_{n_o}(x;\nu,\theta,h) & n = {n_o} \\
        (-\gamma_{n_g}-{r_{n_g}}\varepsilon) x_{n_g} + E_{n_g}(x;\nu,\theta,h) & n = {n_g}
    \end{cases}
\end{equation}
with $R_{n_g}(\xi')=\setof{r_{n_g}}$. 
\end{defn}
\begin{cor}
\label{cor:FlowFromExternalTangency-Perturbation}
Consider the standing hypothesis {\bf H2} and assume that $(\xi,\xi') \in \cX^{(N-2)} \times \cX^{(N-1)}$ has a GO-pair $(n_g,n_o)$. Let $\delta \in [0,\infty)$ be an admissible spatial perturbation.  
If $\direc(\xi,\xi') = -p_{n_o}(\xi,\xi')$, then there exist $\varepsilon > 0$ such that for each $x^0 \in \Nul_{n_o}^\delta(\xi,\xi')$, there is a finite exit time $T_{x^0,\varepsilon} > 0$ satisfying
\[
\varphi_\varepsilon([0,T_{x^0,\varepsilon}],x^0)\subseteq Q^\delta(\xi',\cI(\xi,\xi'))
\]
and
\[
\varphi_\varepsilon([0,T_{x^0,\varepsilon}+t],x^0)\not\subset Q^\delta(\xi',\cI(\xi,\xi')),\, \forall t>0.
\]
Furthermore,
\begin{enumerate}
\item[(i)] if $p_{n_o}(\xi,\xi')=1$, then for all $t \in [0,T_{x^0,\varepsilon}]$
    \[
    \varphi_{n_o}(t,x^0) > \mdpt_{n_o}(\xi'),
    \]
    and
\item[(ii)] if $p_{n_o}(\xi,\xi')=-1$, then for all $t \in [0,T_{x^0,\varepsilon}]$
    \[
    \varphi_{n_o}(t,x^0) < \mdpt_{n_o}(\xi').
    \]
\end{enumerate}
\end{cor}
\begin{proof}
    Notice that when $\direc(\xi,\xi')=-p_{n_o}(\xi,\xi')$, for any 
    \[
    x^0 \in \Nul_{n_o}^\delta(\xi,\xi')\subset \Nul_{n_o}(\xi,\xi')
    \] 
    there exists $s >0$ such that $\varphi((0,s),x^0) \subset \Int(\bg(\blup(\xi')))$, hence Proposition~\ref{prop:FlowFromExternalTangency} yields the result.
\end{proof}
\begin{cor}
\label{cor:FlowFromExternalTangency-Perturbation-2}
Given the standing hypothesis {\bf H2}, assume that $(\xi,\xi') \in \cX^{(N-2)} \times \cX^{(N-1)}$ has a GO-pair $(n_g,n_o)$. Let $\delta \in [0,\infty)$ be admissible.
If $\direc(\xi,\xi') = p_{n_o}(\xi,\xi')$, then there exist $\varepsilon > 0$ such that for any $x^0 \in \Nul_{n_o}^\delta(\xi,\xi')$, there is a finite exit time $T_{x^0,\varepsilon} > 0$ satisfying
\[
    \varphi_\varepsilon([0,T_{x^0,\varepsilon}],x^0)\subseteq Q^\delta(\xi,\cI(\xi,\xi')),
\]
and for all $t>0$
\[
\varphi_\varepsilon([0,T_{x^0,\varepsilon}+t],x^0)\not\subset Q^\delta(\xi,\cI(\xi,\xi')).
\]
Furthermore, $\varphi_{n_o}(t,x^0) \in \Int(I_{n_o}(\xi))$ for every $t \in (0,T_{x^0,\varepsilon}]$.
\end{cor}
\begin{proof}
    Let $x^0 \in \Nul_{n_o}^\delta(\xi,\xi')$ and define 
    \[
    T_{{n_g}}(x^0)  \coloneqq \inf \setdef{ t\geq 0 }{ \varphi_{n_g}(t,x^0) \notin I_{n_g}(\xi) } >0.  
    \]
    We observe that if $\varphi(t,x^0) \in \partial I_{n_o}(\xi)$ for some $t \in  (0,T_{x^0}]$, then $t = T_{x^0}$. Recall that by the \eqref{eq:F2-bounds-internal-not-active} in Definition~\ref{defn:H2}, $h$ satisfies 
    \[
        2 h_{n_o,n_g,j} < \frac{L_{n_g}(\xi)}{U_{n_o}(\xi)} 2 h_{\rmap\xi(n_o),n_o,j}.
    \]
    Using the same argument as Proposition~\ref{prop:FlowFromExternalTangency}, the time $T_{n_g}(x^0)$ to exit through $x_{n_g}=\theta_{n_o,n_g,j}\pm h_{n_o,n_g,j}$ is less than the time $T_{n_o}(x^0)$ to exit through $x_{n_o}=\theta{*,n_o,j'}\pm h_{*,n_o,j}$. Therefore, $\varphi_{n_o}(t,x^0) \in \Int(I_{n_o}(\xi))$. By continuity, one obtains $\varepsilon>0$. 
\end{proof}
\begin{prop}
    \label{prop:continuous-exit-time}
    Given the standing hypothesis {\bf H2}, assume that $(\xi,\xi') \in \cX^{(N-2)} \times \cX^{(N-1)}$ has a GO-pair $(n_g,n_o)$. 
    If $\varphi_\varepsilon$ denotes the flow of the perturbed ramp system $\dot{x}=F_\varepsilon(x)$ in \eqref{eq:rampSysPerturbed}, then the exit times $T_{x^0,\varepsilon}$ vary continuously as a function of $x^0 \in \Nul_{n_o}^\delta(\xi,\xi')$.
\end{prop}
\begin{proof}
Given that $\dot{x}_n = 0$ when $n \neq n_g, n_o$, it follows that for any $x^0 \in \Nul_{n_o}^\delta(\xi,\xi')$, the exit time $T_{x^0,\varepsilon}$ is either equal to $T_{n_g,\varepsilon}(x^0)$ or equal to $T_{n_o,\varepsilon}(x^0)$. By Corollary~\ref{cor:FlowFromExternalTangency-Perturbation} and Corollary~\ref{cor:FlowFromExternalTangency-Perturbation-2}, it must be the case that $T_{x^0,\varepsilon} = T_{n_g,\varepsilon}(x^0)$. In particular, for every $x^0 \in \Nul_{n_o}^\delta(\xi,\xi')$
\[ 
\varphi_{\varepsilon,n_g}(T_{x^0},x^0) = \theta_{n_o,n_g,j}+r_{n_g} h_{n_o,n_g,j},
\]
for some $j \in \setof{1,\ldots,K(n_g)}$. 
Then for $\varepsilon' >0$ sufficiently small,
\begin{equation*}
    \varphi_{\varepsilon}(T_{x^0}-\varepsilon',x^0) \in \Int(\bg(\blup(\xi'))) \text{ and } \varphi_{\varepsilon}(T_{x^0}+\varepsilon',x^0) \notin \bg(\blup(\xi')).
\end{equation*}
since $|\dot{x}_{n_g}| > 0$ for all $x \in \bg(\blup(\xi'))$. Fix $\varepsilon' > 0$ that verifies the condition above. By continuity with respect to the initial conditions, there are neighborhoods $U_+, U_-$ of $x^0$ such that 
\[
    \varphi_\varepsilon(T_{x^0}-\varepsilon',U_+) \subset \Int(\bg(\blup(\xi'))) \text{ and } \varphi_\varepsilon(T_{x^0}+\varepsilon',U_-) \subset (\bg(\blup(\xi')))^\complement.
\]
Thus, for any neighborhood $U$ of $x^0$ contained in $U_+ \cap U_-$, $y^0 \in U$ implies
\[
    T_{x^0} - \varepsilon' < T_{y^0} < T_{x^0} + \varepsilon',
\]
hence $|T_{y^0}-T_{x^0}|<\varepsilon'$. 
\end{proof}
\begin{prop}
    \label{prop:transversality-manifold}
    Under the assumptions of {\bf H2}, if $(\xi, \xi') \in \cX^{(N-2)} \times \cX^{(N-1)}$ admits a GO-pair $(n_g, n_o)$, then the solutions of the ramp system $\dot{x} = F(x)$ in \eqref{eq:rampH2} are transverse to $\cM_{\delta,\varepsilon}(\xi,\xi')\setminus \Nul_{n_o}(\xi,\xi')$. 
\end{prop}
\begin{proof}
Given that $\cM_{\delta,\varepsilon}(\xi,\xi')$ is given by solutions of $\dot{x}=F_\varepsilon(x)$ for
\begin{equation*}
    F_{\varepsilon,n}(x) =  \begin{cases} 0 & n \neq n_o,n_g \\
       F_{n_o}(x) & n = {n_o} \\
        F_{n_g}(x) - r_{n_g}\varepsilon x_{n_g} & n = {n_g},
    \end{cases}
\end{equation*}
it is sufficient to show that $F(x) \neq \lambda F_\varepsilon (x)$ for any $\lambda \neq 0$, that is, $F(x) \notin T_{x} \cM_{\delta,\varepsilon}(\xi,\xi')$. Let $x \in \cM_{\delta,\varepsilon}(\xi,\xi')\setminus \Nul_{n_o}(\xi,\xi')$ and suppose that $F(x)=\lambda F_\varepsilon(x)$. Since $F_{\varepsilon,n_o}(x)=F_{n_o}(x)$, if $F_{n_o}(x)\neq 0$, then $\lambda=1$. However, $F(x)- \lambda F_\varepsilon(x) = \varepsilon r_{n_g} x_{n_g} \neq 0$. Thus, there is no $\lambda \neq 0$ such that $F(x)=\lambda F_\varepsilon(x)$, therefore the flow is transverse to $\cM_{\delta,\varepsilon}(\xi,\xi')\setminus \Nul_{n_o}(\xi,\xi')$. 
\end{proof}
We have so far defined the GO-manifold at $(\xi,\xi')$ when $\dim(\xi)=N-2$, which is the highest possible dimension in which GO-pairs may be exhibited. We now extend the GO-manifold to any pair $(\xi,\xi') \in \cX^{(N-k-1)} \times \cX^{(N-k)}$ that has GO-pairs, where $k\geq 2$.
\begin{defn}
    \label{defn:R2-Manifold}
    Assume that $(\xi,\xi') \in \cX \times \cX$ has a GO-pair $(n_g,n_o)$ with $\Ex(\xi,\xi')=\setof{n_o}$ and $\dim(\xi)<N-2$. Let $(\tilde{\xi}_i,\tilde{\xi}'_i) \in \cX^{(N-2)} \times \cX^{(N-1)}$ be all $k$ pairs that have unique GO-pairs $(n_{g_i},n_o)$ with $\xi \preceq \tilde{\xi}_i \prec \tilde{\xi}_i'$ and $\xi'\preceq \tilde{\xi}_i'$ for $i=1,\ldots,k$ and denote by $F_{\varepsilon_i}(x;\tilde{\xi}_i,\tilde{\xi}'_i)$ their corresponding system as in \eqref{eq:rampSysPerturbed}. 
    
    Let $\Delta^{k-1}\subseteq \mathbb{R}^{k}$ denote the standard $(k-1)$-simplex
    \[
        \Delta^{k-1} = \setdef{ \alpha \in \mathbb{R}^{k}}{\sum_{i=1}^{k} \alpha_i = 1,\ \alpha_i \geq 0}
    \]
    and let $\psi : \Delta^{k-1}\times [0,1]^{N-k} \to \Nul_{n_o}^\delta(\xi,\xi')$ be a parametrization of $\Nul_{n_o}^\delta(\xi,\xi')$ where $\psi_i = \psi|_{\alpha_i = 0}$ is a parametrization of
    \(
    Q^\delta \left(\tilde{\xi}_i,\cI(\tilde{\xi},\tilde{\xi}_i') \right) \cap Q^\delta \left(\tilde{\xi}_i',\cI(\tilde{\xi},\tilde{\xi}_i')\right)
    \)
    for each $i = 1, \ldots, k$. 
    
    Given an admissible $\delta \in [0,\infty)$ and  $\varepsilon=(\varepsilon_1,\ldots,\varepsilon_k)$ with $\varepsilon_i > 0$, define the \emph{GO-manifold at $(\xi,\xi')$} by 
    \[
        \cM_{\delta,\varepsilon}(\xi,\xi') \coloneqq \bigcup_{x^0 \in \Nul_{n_o}^\delta(\xi,\xi')} \varphi_{\varepsilon}([0,T_{n_g}(x^0)],x^0),
    \]
    where $\varphi_{\varepsilon}(\cdot,x^0)$ are solutions of
    \begin{equation}
        \label{eq:rampSysPerturbed-general}
        \dot{x} = \sum_{i=0}^{k-1} \alpha_i F_{\varepsilon_i}(x;\tilde{\xi}_{i},\tilde{\xi}'_{i}).
    \end{equation}
\end{defn}

Corollaries~\ref{cor:FlowFromExternalTangency-Perturbation},~\ref{cor:FlowFromExternalTangency-Perturbation-2} and Proposition~\ref{prop:transversality-manifold} yield the following corollary. 
\begin{cor}
\label{cor:R2Manifold}
Assume the standing hypothesis {\bf H2}. 
If $(\xi,\xi') \in \cD(\rook)$ and $\delta \in [0,\infty)$ is an admissible spatial perturbation, then there exist $\varepsilon>0$ such that solutions of the ramp system \eqref{eq:rampH2} are transverse to $\cM_{\delta,\varepsilon}(\xi,\xi')\setminus \Nul_{n_o}(\xi,\xi')$
where $\dot{x}=F(x)$ is the ramp system \eqref{eq:rampH2} and $\dot{x}=F_\varepsilon(x)$ is the perturbed system \eqref{eq:rampSysPerturbed-general}. 
\end{cor}
\begin{proof}
    When $\dim(\xi)=N-2$, Proposition~\ref{prop:transversality-manifold} yields the result directly. Assume that $\dim(\xi)<N-2$.
    Note that if $(\xi,\xi') \in \cD(\rook)$, then $\direc(\tilde{\xi}_i,\tilde{\xi}_i')=\direc(\xi,\xi')$ for any pair $(\tilde{\xi}_i,\tilde{\xi}_i')\in \cX^{(N-2)} \times \cX^{(N-1)}$ that have GO-pairs and $\Ex(\xi,\xi')=\Ex(\tilde{\xi}_i,\tilde{\xi}_i')$, by Proposition~\ref{prop:consistent-direction}. Therefore, by Proposition~\ref{prop:local-monotone}, $\Nul_{n_o}^\delta(\xi,\xi')$ splits $Q^\delta(\xi,\cI(\xi,\xi'))\cap Q^\delta(\xi',\cI(\xi,\xi'))$ into two connected regions that satisfy either $\dot{x}_{n_o}>0$ or $\dot{x}_{n_o}<0$. In either case, either all trajectories starting at $x^0 \in \Nul_{n_o}^\delta(\xi,\xi')$ go into $Q^\delta(\xi,\cI(\xi,\xi'))$ or all trajectories go into $Q^\delta(\xi',\cI(\xi,\xi'))$--that is, for each $x^0 \in \Nul_{n_o}^\delta(\xi,\xi')$ there exists $s>0$ sufficiently small such that
    \[
        \varphi_\varepsilon((0,s),x^0) \subset \Int(Q^\delta(\xi,\cI(\xi,\xi'))), 
    \]
    or $ \varphi_\varepsilon((0,s),x^0) \subset \Int(Q^\delta(\xi',\cI(\xi,\xi')))$. Without loss of generality, assume the former. By Proposition~\ref{prop:FlowFromExternalTangency}, $\cM_{\delta,\varepsilon}(\xi,\xi')$ does not intersect the plane $x_{n_o}=\mdpt_{n_o}(\xi')$, and $F(x) \notin T_x \cM_{\delta,\varepsilon}(\xi,\xi')$ for sufficiently small $\varepsilon_i > 0$. 
\end{proof}
\begin{defn}
We call $\epsilon >0$ an \emph{admissible GO-perturbation} if it satisfies the conclusion of Corollary~\ref{cor:R2Manifold}.
\end{defn}

\section{Geometrization of $\Janus$ for $\cF_2$}
\label{sec:Geo2}

We remind the reader that the primary goal of this chapter is to prove Theorem~\ref{thm:R2ABlattice}. For that purpose, we construct a geometrization $\bG_2$ of the Janus complex $\Janus$ that is aligned with the vector field. The desired geometrization $\bG_2$ is obtained from $\recG_J$ by changing the restrictions in Definition~\ref{defn:rectGeoXr} that $\bg_\zeta$ must satisfy when $\zeta \in \Janus$. Of utmost interest are the cells in the boundary of $\mTile(\xi,\xi')$ when $(\xi,\xi') \in \cD(\rook)$.

First, we define the cells in $\Janus$ that will be associated to $Q^\delta(\xi,I)$ by identifying a similar structure in the cell complex, i.e., given a certain list of indices $\cI \subseteq \setof{1,\ldots,N}$, we add or subtract cells in a given direction $n \in \cI$ whenever $n \in J_i(\xi)$ or $n \in J_e(\xi)$. 
\begin{defn}
    Let $\xi=[\bv,\bw] \in \cX$ and let $\Janus$ be the Janus complex. Let $\cI \subseteq \setof{1,\ldots,N}$. We define the collection of top cells $\cQ(\xi,\cI) \subseteq \Top_{\Janus}(\Janus)$ by $\zeta=[\bv',\bone] \in \cQ(\xi,\cI)$ if and only if
    \begin{align*}
        8(2\bv_n)+2 \leq \bv_n' < 8(2\bv_n+1)-2 & \text{ if } n \notin \cI \text{ and } n \in J_e(\xi) \\ 
        8(2\bv_n+1)-2 \leq \bv_n' < 8((2\bv_n+1)+1)+2 & \text{ if } n \notin \cI \text{ and } n \in J_i(\xi) \\
        8(2\bv_n+\bw_n) \leq \bv_n' < 8((2\bv_n+\bw_n)+1)+1 & \text{ if } n \in \cI
    \end{align*}
    and $\ctoy_2 \colon \cX \times \mathcal{P}(\setof{1,\ldots,N}) \to \sSub(\Janus)$ by
    \[
        \ctoy_2(\xi,\cI) = \sO\left(\setdef{[\bv',\bone] \in\Janus}{[\bv',\bone]\in \cQ(\xi,\cI)} \right). 
    \]
\end{defn}
To simplify the notation on what follows, we identify the cells of the Janus complex that will be mapped to the manifolds $\cM_{\delta,\varepsilon}(\xi,\xi')$. Recall that Definition~\ref{defn:pi2fibers} associates to each $\alpha\in\cX$ the $2$-fiber given by
\[
\Pl(\alpha) = \bar{S}(\blup(\alpha))\cup \Add(\alpha) - \Rem(\alpha) \subset \Janus^{(N)}.
\]
\begin{defn}
    Given $(\xi,\xi') \in \cD(\rook)$ with GO-pair $(n_g,n_o)$, 
    the set of \emph{transfiguration cells of $(\xi,\xi')$} is given by 
    \[
    \trans(\xi,\xi') = \sO(\Pl(\NegDriftCells(\xi,\xi'))\cap \sO(\Pl(\PosDriftCells(\xi,\xi')).
    \]
    The \emph{Q-transfiguration cells of $(\xi,\xi')$} is given by
    \begin{equation}
    \label{eq:transQ}
        \trans_Q(\xi,\xi') = 
        \ctoy_2(\NegDriftCells(\xi,\xi'),\cI(\xi,\xi'))
    \end{equation}
\end{defn}
\begin{defn}
    \label{defn:R2-Geometrization}
    Let $\delta>0$ be an admissible spatial perturbation, and let $\varepsilon>0$ be a GO-perturbation. A rectangular geometrization $\bG_2$ of $\Janus$ in $\mathbb{R}^N$ satisfies the \emph{GO-conditions} if
    \begin{equation}
    \label{eq:transfig-geo}
        \bg \left( \trans_Q(\xi,\xi') \right) = \cM_{\delta,\varepsilon}(\xi,\xi'), \quad \forall (\xi,\xi')\in\cD(\rook).
    \end{equation}
\end{defn}
We show that if a rectangular geometrization $\bG_2$ satisfies the GO-conditions, then the geometrization is aligned with the vector field in \eqref{eq:rampH2}. 

\begin{proof}[Proof of Theorem~\ref{thm:R2ABlattice}]
    Let $\bG_2$ be a rectangular geometrization that satisfies the GO-conditions. Let $\cN \in \sN(\cF_2)$ and let $\zeta \in \bbdy(\cN)^{(N-1)}$. We show that $\varphi$ is inward-pointing at every $x \in \bg(\zeta)$. First, we show that 
    \begin{equation}\label{eq:R2alignment}
        \langle F(x) , z_\zeta(x) \rangle > 0,
    \end{equation}
    for a inward-ponting normal vector whenever the flow is tranverse, and then show that $\varphi$ is inward-pointing at $x$ when the flow is tangent to $\bg(\zeta)$. 
    
    In light of Theorem~\ref{thm:R1ABlattice} and Definition~\ref{def:Rule2}, it is enough to \eqref{eq:R2alignment} when $\zeta \in \Pl(\xi) \cap \Pl(\xi')$ for some pair $(\xi,\xi')\in \cD(\rook)$. Let $(n_g,n_o)$ be a GO-pair for $(\xi,\xi')$ and let $x=(x_1,\ldots,x_N) \in \bg(\zeta)$. 

    Assume that $\xi \in \cF_2(\xi')$, so that $\pi(\xi) < \pi(\xi')$ and $\xi \in \cN$ while $\xi'\notin \cN$. If $\zeta \in \bar{S}(\xi) \cap \bar{S}(\xi')$, then 
    \[
        \sgn(-\gamma_{n_o} x_{n_o} + E_{n_o}(x;\nu,\theta,h)) = p_n(\xi,\xi').
    \]
    If $\zeta \notin \bar{S}(\xi) \cap \bar{S}(\xi')$, then $\zeta \in \trans_Q(\xi,\xi')$ and $x \in \cM_{\delta,\varepsilon}(\xi,\xi')\setminus\Nul_{n_o}(\xi,\xi')$. In that case, Proposition~\ref{prop:transversality-manifold} implies that the flow is transverse to $\cM_{\delta,\varepsilon}(\xi,\xi')$. Note that a inward pointing normal vector at $x$ is given by
    \[
        z_\zeta(x) = r_{n_g} p_{n_o}(\xi,\xi') \left( -F_{\varepsilon,n_g}(x;\xi,\xi')\bzero^{(n_o)}+F_{\varepsilon,n_o}(x;\xi,\xi')\bzero^{(n_g)}\right). 
    \]
    Since $\sgn(F_{n_o}(x)) = -p_{n_o}(\xi,\xi')$,
    \begin{align*}
         \langle F(x) , z_\zeta(x) \rangle & =  r_{n_g} p_{n_o}(\xi,\xi')\left( -r_{n_g}\varepsilon x_{n_g} F_{n_o}(x) \right) = \varepsilon x_{n_g} |F_{n_o}(x)| > 0
    \end{align*}
    It remains to show that $\varphi$ is inward-pointing at $x \in \Nul_{n_o}(\xi,\xi')$. Assume that $x \in \Nul_{n_o}(\xi,\xi')$, we must show that there exists $t_0$ such that $\varphi(t,x) \in \Int(\bg(\cN))$ for all $0<t<t_0$. Note that $F(x)-F_\varepsilon(x) = r_{n_g}\varepsilon x_{n_g}$, so there exists $t_0$ such that $\varphi(t_0,x) \notin \cM_{\delta,\varepsilon}(\xi,\xi')$. By Theorem~\ref{thm:inward}, $\varphi(t_0,x) \in \Int(\bg(\cN))$. Therefore, by Corollary~\ref{cor:inward}, $\bG$ is aligned with $f$ over $\cN$.
    
    The case $\xi' \in \cF_2(\xi)$ is analogous. 
\end{proof}

We showcase the construction when $N=2$ in Figure~\ref{fig:GeometrizationF2}.
\begin{figure}[H]
    \centering
    \begin{subfigure}{0.4\textwidth}
        \centering
        \begin{tikzpicture}[scale=0.25]
    \draw[step=8cm,ultra thick, blue] (0,0) grid (16,16);
    
    \foreach \i in {0,...,2}{
        \foreach \j in {0,...,2}{
        \draw[black, fill=blue] (8*\i,8*\j) circle (2ex);
        }
    }
    
    \foreach \i in {1,...,3}{
        \draw(-1,8*\i -8) node{\color{blue}$\i$}; 
    }
    \foreach \i in {0,...,2}{
        \draw(8*\i,-1) node{\color{blue}$\i$}; 
    }

    \foreach \i in {0,...,1}{
        \draw[->,  thick] (4+8*\i,0.1) -- (4+8*\i,2.1); 
        \draw[->,  thick] (4+8*\i,15.9) -- (4+8*\i,13.9); 
    }
    \foreach \i in {1,...,2}{
        \draw[->,  thick] (0.1, 4+8*\i -8) -- (2.1,4+8*\i -8); 
    }
    \foreach \i in {0,...,1}{
        \draw[->,  thick] (4+8*\i,5.9) -- (4+8*\i,7.9); 
        \draw[->, thick] (4+8*\i,8.1) -- (4+8*\i,10.1);
    }
        \draw[->,  thick] (8.1,4) -- (10.1,4);
        \draw[->,  thick] (5.9,4) -- (7.9,4);
        \draw[->,  thick] (13.9,4) -- (15.9,4); 
    \foreach \i in {2}{
        \draw[->,  thick] (7.9,4+8*\i -8) -- (5.9,4+8*\i -8); 
         \draw[->,  thick] (10.1,4+8*\i -8) -- (8.1,4+8*\i -8);
         \draw[->,  thick] (15.9,4+8*\i -8) -- (13.9,4+8*\i -8);  
    }
    
    \draw(4,12) node{$\mu_0$};
    \draw(12,12)  node{$\mu_1$};
    \draw(4,4) node{$\hat\xi_0'$};
    \draw(12,4)  node{$\hat\xi_1'$};
    \draw(5.5,7) node{$\xi'_0$};
    \draw(8.75,13.5)  node{$\tau$};
    \draw(8.75,2.75)  node{$\hat\xi$};
    \draw(13.5,7)  node{$\xi'_1$};
    \draw(9,9)  node{$\xi$};
    \end{tikzpicture}
        \caption{Wall labeling of Example~\ref{ex:drift_2d} around $\xi=\defcell{1}{2}{0}{0}$.}
    \end{subfigure}%
    \hfill
    \begin{subfigure}{0.5\textwidth}
        \centering
        \begin{tikzpicture}[scale=0.5]
    
    \fill[gray] (6,8.5) -- (5.5,8.5) -- (5.5,9.0) -- (5.0,9.0) -- (5.0,9.5) -- (4.5,9.5) -- (4.5,10.0) -- (6,10) -- (6,8.5);
    
    \fill[gray] (10,8.5) -- (9.5,8.5) -- (9.5,9.0) -- (9,9.0) -- (9,9.5) -- (8.5,9.5) -- (8.5,10) -- (10,10) -- (10,8.5);
    
    \draw[thick, red] (2,6) -- (14,6);
    \draw[thick, red] (2,10) -- (14,10);
    \draw[thick, red] (2,14) -- (14,14);
    \draw[thick, red] (2,6) -- (2,14);
    \draw[thick, red] (6,6) -- (6,14);
    \draw[thick, red] (10,6) -- (10,14);
    \draw[thick, red] (14,6) -- (14,14);
    
    \draw (2,6.5) -- (14,6.5);
    \draw (2,7.0) -- (14,7.0);
    \draw (2,7.5) -- (14,7.5);
    \draw (2,8.0) -- (14,8.0);
    \draw (2,8.5) -- (14,8.5);
    \draw (2,9.0) -- (14,9.0);
    \draw (2,9.5) -- (14,9.5);
    
    \draw (2,10.5) -- (14,10.5);
    \draw (2,11.0) -- (14,11.0);
    \draw (2,11.5) -- (14,11.5);
    \draw (2,12.0) -- (14,12.0);
    \draw (2,12.5) -- (14,12.5);
    \draw (2,13.0) -- (14,13.0);
    \draw (2,13.5) -- (14,13.5);
    
    \draw (2.5,6) -- (2.5,14);
    \draw (3.0,6) -- (3.0,14);
    \draw (3.5,6) -- (3.5,14);
    \draw (4.0,6) -- (4.0,14);
    \draw (4.5,6) -- (4.5,14);
    \draw (5.0,6) -- (5.0,14);
    \draw (5.5,6) -- (5.5,14);
    
    \draw (6.5,6) -- (6.5,14);
    \draw (7.0,6) -- (7.0,14);
    \draw (7.5,6) -- (7.5,14);
    \draw (8.0,6) -- (8.0,14);
    \draw (8.5,6) -- (8.5,14);
    \draw (9.0,6) -- (9.0,14);
    \draw (9.5,6) -- (9.5,14);
    
    \draw (10.5,6) -- (10.5,14);
    \draw (11.0,6) -- (11.0,14);
    \draw (11.5,6) -- (11.5,14);
    \draw (12.0,6) -- (12.0,14);
    \draw (12.5,6) -- (12.5,14);
    \draw (13.0,6) -- (13.0,14);
    \draw (13.5,6) -- (13.5,14);
    
    \draw(2,5) node{$\color{red}1$};
    \draw(6,5) node{$\color{red}2$};
    \draw(10,5) node{$\color{red}3$};
    \draw(14,5) node{$\color{red}4$};
    
    \draw(1,6) node{$\color{red}4$};
    \draw(1,10) node{$\color{red}5$};
    \draw(1,14) node{$\color{red}6$};
    
    \draw[red, fill=red] (6,10) circle (1.5ex);
    \draw[red, fill=red] (10,10) circle (1.5ex);
    
    \end{tikzpicture}
        \caption{Subcomplex of $\Janus$ within modifications to $\pi_J$ are made.}
    \end{subfigure}\\
    
    \begin{subfigure}{0.4\textwidth}
        \centering
        \begin{tikzpicture}[scale=0.2]
            \draw[step=8cm,black,very thick] (0,0) grid (24,24);
            \draw[step=1cm,black,thin] (0,0) grid (24,24);
            \draw[blue,very thick] (5,16) -- (5,15) -- (6,15) -- (6,14) -- (7,14) -- (7,13) -- (8,13);
            \draw[red,very thick] (5+8,16) -- (5+8,15) -- (6+8,15) -- (6+8,14) -- (7+8,14) -- (7+8,13) -- (8+8,13);
        \end{tikzpicture}
        \caption{Highlighted cells in the Janus complex with proposed geometrization.}
    \end{subfigure}%
    \hfill
    \begin{subfigure}{0.5\textwidth}
        \centering
        \includegraphics[scale=0.5]{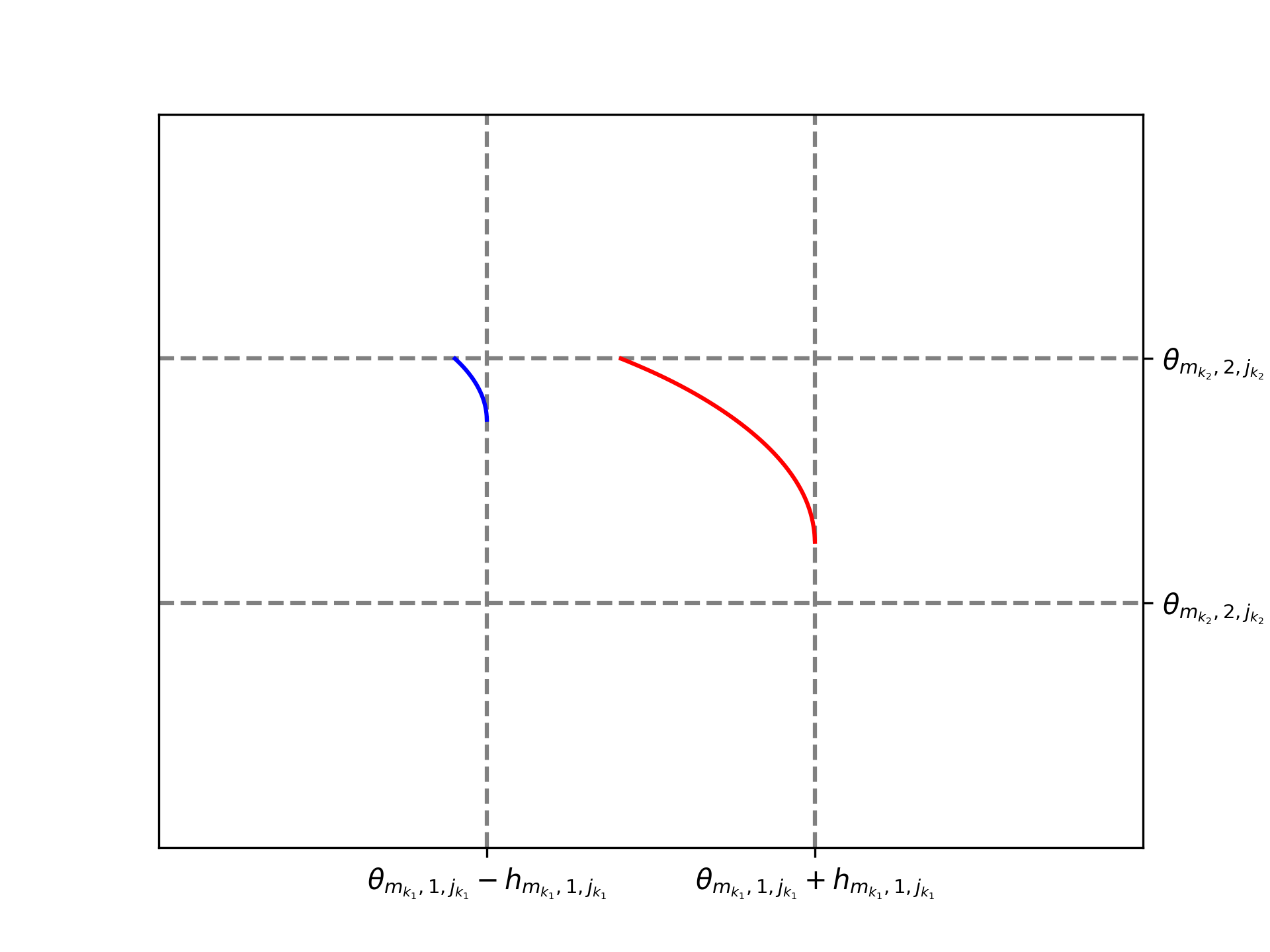}
        \caption{Geometrization in the phase space.}
    \end{subfigure}
    \caption{Geometrization that provides transversality for $\mathcal{F}_2$.}
    \label{fig:GeometrizationF2}
\end{figure}
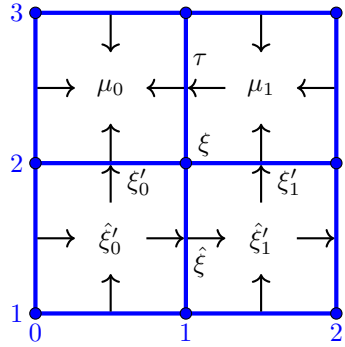
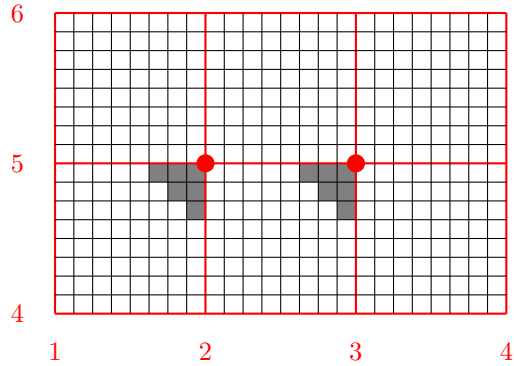
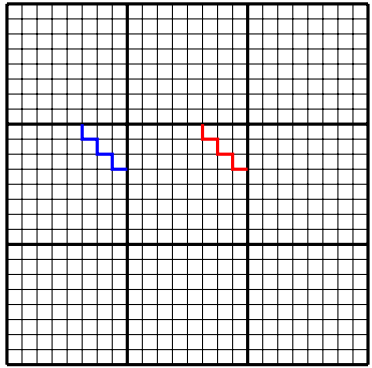
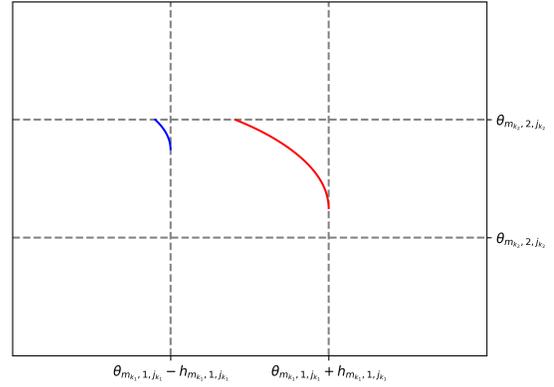

\chapter{$\cD$-Grading in the Janus Complex for $\cF_3$}\label{sec:P3-Grading}

This chapter follows similar steps to those in Chapter \ref{sec:P2-grading}. Specifically, given  a cubical complex $\cX$ and a multivalued map $\cF_3\colon \cX \mvmap \cX$ as in Definition~\ref{defn:Rule3} 
an admissible
D-grading $\pi_3\colon\Janus \to (\SCC(\cF_3),\posetF{3})$ is constructed over the associated Janus complex $\Janus$. As discussed in Section~\ref{sec:ConleyComplex}, 
an admissible D-grading
determines a poset grading of $C_*(\cX;\F)$, which in turn leads to a Conley complex. As done in Chapter~\ref{sec:P2-grading}, we 
derive $\pi_3$ via local modifications of $\pi_2$ that preserve up to chain homotopy the poset grading of $C_*(\cX;\F)$. Specifically, 
the modifications are performed on closures of subsets called \emph{modification neighborhood} of $\mBn{\xi}$ for any $\xi\in\Xi$, where $\Xi$ is the set of cells in $\cX$ that satisfy Conditions 3.1 and 3.2 in Definitions~\ref{defn:Rule3.1} and \ref{defn:Rule3.2} and $N=2,3$.

In Section \ref{sec:signature_of_cycles}, we identify an analogous to the spine in Definition~\ref{defn:spine}. Then we use this identification to perform the  modifications of $\pi_2$ on the modification neighborhoods of a two dimensional Janus complex, in Section 
\ref{sec:local_2D_Janus}, and a three dimensional Janus complex, in Section \ref{sec:local_3D_Janus}. Finally, in Section \ref{sec:global_pi3}, we merge the local modification of $\pi_2$ to define $\pi_3$ over $\Janus$. Furthermore, we prove that $\pi_3$ is an extendable D-grading such that preserve up to chain homotopy the original poset grading of $C_*(\cX;\F)$.

\section{Constructing D-grading $\pi_3$ via signature of cycles}
\label{sec:signature_of_cycles}

We start by identifying the cells in which Conditions 3.1 and 3.2 are applied.
By Definition~\ref{defn:Rule3}, these are the semi-opaque cells $\xi \in \cX$ whose cycle decomposition of its regulation map $\rmap\xi$ contains cycles of length $k\geq 2$.
When $N=2$ and $N=3$, the possibilities are as follows:
\begin{itemize}
    \item[(i)] if $N=2$, then $\rmap\xi = (n_1 \ n_2)$; 
    \item[(ii)] if $N=3$, then $\rmap\xi = (n_1 \ n_2 )$, $(n_1 \ n_2) (n_3)$, or $(n_1 \ n_2 \ n_3)$.
\end{itemize}
Note that when $N=2$, $\xi$ is necessarily an equilibrium cell, while when $N=3$ it is not. 
The construction of $\pi_3$ is carried out locally, following the cases outlined in items (i) and (ii), which can be further categorized based on the signature of the cycles, as defined below.
\begin{defn}
\label{defn:feedback}
Let $\rook\colon \TP(\cX)\to \setof{0,\pm 1}^N$ be a rook field on an $N$-dimensional cubical complex $\cX$.
Let $\xi=[\bv,\bw] \in \cX$ be a semi-opaque cell (see Definition~\ref{defn:partially_opaque}). 
Let  $\rmap\xi = \sigma_{\xi,1} \sigma_{\xi,2} \ldots \sigma_{\xi,r}$ be the decomposition of the permutation $\rmap\xi$ into disjoint cycles. 
For each cycle $\sigma=(n_1 \ \ldots \ n_k)$ of length $k$ let $S_\sigma \coloneqq \setof{n_1,\ldots,n_k} \subset \activeset(\xi)$ denote the support of the cycle.
Define the \emph{signature} of $\sigma$ to be
\[
\delta_\sigma = \prod_{i\in S_\sigma} \rook_i(\xi,[\bv,\bone]) \in\setof{\pm 1}.
\]
If $\rmap\xi = \sigma_{\xi}$ is a cycle of length $N$, then we say that $\xi$ has a \emph{positive (negative) feedback} if $\delta_\sigma=+1$ ($\delta_\sigma=-1$).
If, in addition, $\xi$ is an equilibrium cell, then we say that $\xi$ is an \emph{equilibrium cell with positive (negative) feedback}.
\end{defn}
Figure \ref{fig:pos_neg_cycles} provides examples of equilibrium cells with positive and negative feedback. 
\begin{figure}
\begin{center}
\begin{tikzpicture}[scale=0.25]
\draw[step=8cm,black] (0,0) grid (16,16);
\foreach \i in {0,...,2}{
    \foreach \j in {0,...,2}{
    \draw[black, fill=black] (8*\i,8*\j) circle (2ex);
    }
}
\foreach \i in {1,...,3}{
    \draw(-1,8*\i -8) node{$\i$}; 
}
\foreach \i in {1,...,3}{
    \draw(8*\i-8,-1.1) node{$\i$}; 
}
\def\i{0}
    \draw[->, blue, thick] (4+8*\i,5.9) -- (4+8*\i,7.9); 
    \draw[->, blue, thick] (4+8*\i,8.1) -- (4+8*\i,10.1);
\def\i{1}
    \draw[->, blue, thick] (4+8*\i,7.9) -- (4+8*\i,5.9); 
    \draw[->, blue, thick] (4+8*\i,10.1) -- (4+8*\i,8.1);
    \draw[->, blue, thick] (8.1,4) -- (10.1,4);
    \draw[->, blue, thick] (5.9,4) -- (7.9,4);
\def\i{2}
\draw[->, blue, thick] (7.9,4+8*\i -8) -- (5.9,4+8*\i -8); 
\draw[->, blue, thick] (10.1,4+8*\i -8) -- (8.1,4+8*\i -8);  
\end{tikzpicture}
\begin{tikzpicture}[scale=0.25]
\draw[step=8cm,black] (0,0) grid (16,16);
\foreach \i in {0,...,2}{
    \foreach \j in {0,...,2}{
    \draw[black, fill=black] (8*\i,8*\j) circle (2ex);
    }
}
\foreach \i in {1,...,3}{
    \draw(-1,8*\i -8) node{$\i$}; 
}
\foreach \i in {1,...,3}{
    \draw(8*\i-8,-1.1) node{$\i$}; 
}
\def\i{0}
    \draw[->, blue, thick] (4+8*\i,7.9) -- (4+8*\i,5.9); 
    \draw[->, blue, thick] (4+8*\i,10.1) -- (4+8*\i,8.1);
\def\i{1}
    \draw[->, blue, thick] (4+8*\i,5.9) -- (4+8*\i,7.9); 
    \draw[->, blue, thick] (4+8*\i,8.1) -- (4+8*\i,10.1);
    \draw[->, blue, thick] (8.1,4) -- (10.1,4);
    \draw[->, blue, thick] (5.9,4) -- (7.9,4);
\def\i{2}
\draw[->, blue, thick] (7.9,4+8*\i -8) -- (5.9,4+8*\i -8); 
\draw[->, blue, thick] (10.1,4+8*\i -8) -- (8.1,4+8*\i -8);  
\end{tikzpicture}
\end{center}
\caption{Equilibrium cell $\xi=\defcell{2}{2}{0}{0}$ with positive and negative feedback on the left and right of the figure, respectively.
}
\label{fig:pos_neg_cycles}
\end{figure}
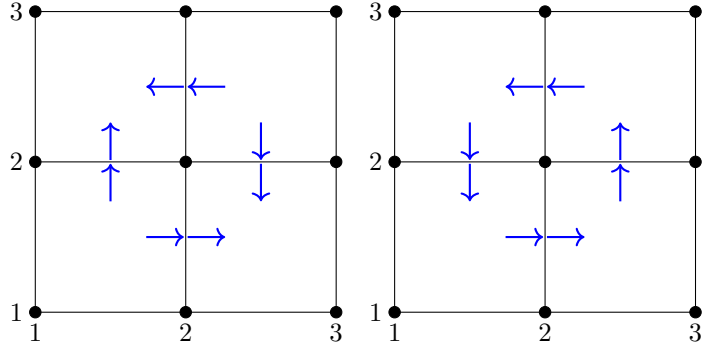

We extend the definition of spine, modification cells, and tiles to the setting of semi-opaque cells.
\begin{defn}
\label{defn:semi-opaqueMod}
Let $\xi \in \cX$ be a semi-opaque cell and let $\rmap\xi = \sigma_{\xi,1}\ldots \sigma_{\xi,r} $ be the decomposition of the permutation $\rmap\xi$ into disjoint cycles. 
For each $\sigma = (n_1 \ \ldots \ n_k)$ of length $k\geq 2$, define the \emph{modification spine of $\xi$ with respect to $\sigma$} to be 
\[
\Spineb(\xi,\sigma) \coloneqq \cl(b(\xi)) \bigcap \left( \bigcup_{\mu \in \cU(\xi,\sigma)} \cl(b(\mu)) \right)
\]
where $\cU(\xi,\sigma)$ is defined in \eqref{defn:set_cells_Uns&Stable}.
The \emph{modification spine of $\xi$ with respect to $\sigma$ in the Janus complex} is given by 
\[
\SpineJ(\xi,\sigma) \coloneqq \bar{S}(\Spineb(\xi,\sigma)).
\]
\end{defn}
\begin{ex}\label{ex:3Dpos_1}
Consider $N=3$ and $\xi=[\bv,\bzero]\in\cX^{(0)}$ a semi-opaque cell. Assume that $o_\xi=\sigma=(1 \ 2 \ 3)$ and \[w(\kappa_{1}^-,\kappa)=w(\kappa_{2}^-,\kappa)=w(\kappa_{3}^-,\kappa)=1\] 
where $\kappa^{-1}_n=[\bv,\bone^{(n)}]$ and $\kappa=[\bv,\bone]$, see Definition~\ref{defn:lap-number}.

Observe that the signature of $\sigma$ is 
\[
\delta_\sigma = \prod_{i\in \setof{1,2,3}} \rook_i(\xi,[\bv,\bone]) = \prod_{i\in \setof{1,2,3}} w(\kappa_{n_i}^-,\kappa) = 1.
\]
Hence, $\xi$ is an equilibrium cell with positive feedback.

From Definition~\ref{defn:lap-number}, it follows that
\[\lap_{\xi,\sigma}(\kappa)=\#\setdef{i\in\setof{1,\ldots,3}}{w(\kappa_{i}^-,\kappa) = -1}=0 \leq \frac{3}{2},\] 
hence $\kappa\in \cU(\xi)$,
see Figure~\ref{fig:wall_labelling_3Dpos}.
     
Define $\mu_\bu\coloneq [\bv -\bu, \bone]$, where $\bu\in\setof{0,1}^3$. With the notation above,  
\[\Top_\cX(\xi)=\setdef{\mu_\bu\in\cX^{(3)}}{\bu\in\setof{0,1}^3},\] 
and $p(\xi,\mu_\bu)=2\bu-\bone$, see Figure~\ref{fig:wall_labelling_3Dpos}.

Recall that 
$\lap_{\xi,\sigma}(\mu_\bu)$ (see Definition~\ref{defn:Rule3}) counts how many times $-1$'s are in
     \[
     \setof{(2\bu_1-1)(2\bu_2-1),(2\bu_2-1)(2\bu_3-1),(2\bu_3-1)(2\bu_1-1)}.
     \]
     Hence, 
     $\lap_{\xi,\sigma}(\mu_\bu)=3>\frac{3}{2}$ for \[\bu\in\setof{\bzero^{(1)},\bzero^{(2)},\bzero^{(3)},\bone^{(1)},\bone^{(2)},\bone^{(3)}}\] 
     and 
    $\lap_{\xi,\sigma}(\mu_\bu)=1<\frac{3}{2}$ for $\bu\in\setof{\bzero,\bone}$. As a consequence, the opposite top cells $[\bv-\bone, \bone]$ and $\kappa=[\bv-\bzero,\bone]$ are the only top cells in $\cU(\xi)$. Hence, the modification spine of $\xi$ is the set
    \begin{align*}
        \Spineb(\xi,\sigma) &= \cl(b(\xi)) \bigcap \left( \cl(b(\mu_\bzero)) \cup \cl(b(\mu_\bone)) \right) \\
        &=  \cl(b(\xi))\cap \cl(b(\mu_\bzero)) \bigcup \cl(b(\xi)) \cap\cl(b(\mu_\bone))= \setof{[\bar\bv,\bzero], [\bar\bv+\bone,\bzero]},
    \end{align*}
    where $\bar\bv\in\N^3$ such that $\blup(\xi)=[\bar\bv,\bone]$.
\end{ex}
\begin{figure}
    \centering
\begin{tikzpicture}[scale=3.5]
    \coordinate (A) at (0,0,0);
    \coordinate (B) at (1,0,0);
    \coordinate (C) at (1,1,0);
    \coordinate (D) at (0,1,0);
    \coordinate (E) at (0,0,1);
    \coordinate (F) at (1,0,1);
    \coordinate (G) at (1,1,1);
    \coordinate (H) at (0,1,1);

    \draw[thick, dashed, blue] (A) -- (B);
    \draw[thick] (B) -- (C);
    \draw[thick] (C) -- (D);
    \draw[thick, dashed, blue] (D) -- (A);

    \draw[thick] (E) -- (F);
    \draw[thick, blue] (F) -- (G);
    \draw[thick, blue] (G) -- (H);
    \draw[thick] (H) -- (E);

    \draw[thick, dashed] (A) -- (E);
    \draw[thick, blue] (B) -- (F);
    \draw[thick] (C) -- (G);
    \draw[thick, blue] (D) -- (H);
    
    \node at (A) [below right] {$\mu_{\bone^{(1)}}$};
    \node at (B) [below right] {$\mu_{\bzero^{(3)}}$};
    \node at (C) [below right] {$\mu_{\bzero}$};
    \node at (D) [above left] {$\mu_{\bzero^{(2)}}$};
    \node at (E) [below right] {$\mu_{\bone}$};
    \node at (F) [below right] {$\mu_{\bone^{(2)}}$};
    \node at (G) [below right] {$\mu_{\bzero^{(1)}}$};
    \node at (H) [above left] {$\mu_{\bone^{(3)}}$};

    \draw[->, thick, red] (A) -- ($(A) + (0.1,0,0)$); 
    \draw[->, thick, red] (A) -- ($(A) + (0,-0.1,0)$); 
    \draw[->, thick, red] (A) -- ($(A) + (0,0,0.2)$); 
    
    \draw[->, thick, red] (B) -- ($(B) + (0.1,0,0)$); 
    \draw[->, thick, red] (B) -- ($(B) + (0,0.1,0)$); 
    \draw[->, thick, red] (B) -- ($(B) + (0,0,0.2)$); 
    
    \draw[->, thick, red] (C) -- ($(C) + (0.1,0,0)$); 
    \draw[->, thick, red] (C) -- ($(C) + (0,0.1,0)$); 
    \draw[->, thick, red] (C) -- ($(C) + (0,0,-0.2)$); 
    
    \draw[->, thick, red] (D) -- ($(D) + (0.1,0,0)$); 
    \draw[->, thick, red] (D) -- ($(D) + (0,-0.1,0)$); 
    \draw[->, thick, red] (D) -- ($(D) + (0,0,-0.2)$); 
    
    \draw[->, thick, red] (E) -- ($(E) + (-0.1,0,0)$); 
    \draw[->, thick, red] (E) -- ($(E) + (0,-0.1,0)$); 
    \draw[->, thick, red] (E) -- ($(E) + (0,0,0.2)$); 
    
    \draw[->, thick, red] (F) -- ($(F) + (-0.1,0,0)$); 
    \draw[->, thick, red] (F) -- ($(F) + (0,0.1,0)$); 
    \draw[->, thick, red] (F) -- ($(F) + (0,0,0.2)$); 
    
    \draw[->, thick, red] (G) -- ($(G) + (-0.1,0,0)$); 
    \draw[->, thick, red] (G) -- ($(G) + (0,0.1,0)$); 
    \draw[->, thick, red] (G) -- ($(G) + (0,0,-0.2)$); 
    
    \draw[->, thick, red] (H) -- ($(H) + (-0.1,0,0)$); 
    \draw[->, thick, red] (H) -- ($(H) + (0,-0.1,0)$); 
    \draw[->, thick, red] (H) -- ($(H) + (0,0,-0.2)$); 

    \coordinate (AB_mid) at ($(A)!0.5!(B)$);
    \coordinate (BC_mid) at ($(B)!0.5!(C)$);
    \coordinate (CD_mid) at ($(C)!0.5!(D)$);
    \coordinate (DA_mid) at ($(D)!0.5!(A)$);
    \coordinate (EF_mid) at ($(E)!0.5!(F)$);
    \coordinate (FG_mid) at ($(F)!0.5!(G)$);
    \coordinate (GH_mid) at ($(G)!0.5!(H)$);
    \coordinate (HE_mid) at ($(H)!0.5!(E)$);
    \coordinate (AE_mid) at ($(A)!0.5!(E)$);
    \coordinate (BF_mid) at ($(B)!0.5!(F)$);
    \coordinate (CG_mid) at ($(C)!0.5!(G)$);
    \coordinate (DH_mid) at ($(D)!0.5!(H)$);
\end{tikzpicture}
    \caption{The cube represents the zero dimensional cell $\xi$ in the dual complex of $\cX$. The top dimensional cell of $\xi$, $\Top_\cX(\xi)=\setdef{\mu_\bu\in\cX^{(3)}}{\bu\in\setof{0,1}^3}$, are represented by vertices in the dual complex of $\cX$. The wall labelling of each top cell in $\Top_\cX(\xi)$ are represented by vectors in red. The edges of the cube are the dual of the two dimensional cells and the edges 
    in blue are the dual of two dimensional cells in the stable cell of $\xi$,  $\Stable(\xi)$.}
    \label{fig:wall_labelling_3Dpos}
\end{figure}
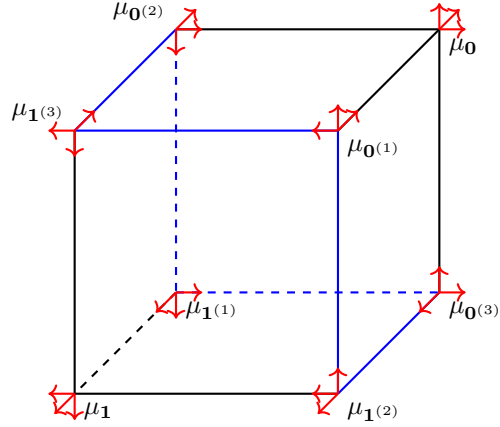
\begin{defn}
    \label{defn:modification-equilibrium}
    The set of \emph{modification cells of $\xi$ with respect to $\sigma$} is given by
    \[
        \mCellX(\xi,\sigma) = \setdef{ \xi' \in \cX}{ \xi \prec \xi',\ \dim(\xi')=\dim(\xi)+1}\cap O(\cU_\sigma(\xi)),
    \]
    and the set of \emph{modification tiles for $\xi$ with respect to $\sigma$} is given by 
    \[
        \mTile(\xi,\sigma) = B_2(\SpineJ(\xi,\sigma)) \cap \bar{S}(\blup(\mCellX(\xi,\sigma))).
    \]
\end{defn}
\begin{ex}\label{ex:3Dpos_2}
    Continuing with Example~\ref{ex:3Dpos_1}, the modification cells of $\xi=[\bv, \bzero]$ with respect $\sigma=(1\ 2\ 3)$ is the set
    \[
        \mCellX(\xi,\sigma) = \setdef{[\bv -a\bzero^{(k)}, \bzero^{(k)}]}{k=0, 1, 2 \text{ and } a=0, 1},
    \]
    i.e., all one-dimensional cells whose faces contain $\xi=[\bv, \bzero]$.
\end{ex}

\section{Local D-grading $\pi_3$ for two dimensional Janus Complex}
\label{sec:local_2D_Janus}
As discussed at the beginning of Section~\ref{sec:signature_of_cycles} we restrict our attention to cells $\xi$ whose regulation map $\rmap\xi$ contains cycles of length greater than or equal to two.
For $N=2$, this implies that $\xi\in\cX^{(0)}$.
To define $\pi_3\colon \Janus \to \SCC(\cF_3)$ we modify $\pi_2$ on the cells $\mBn{\xi}$, see Figure~\ref{fig:star_xi}.

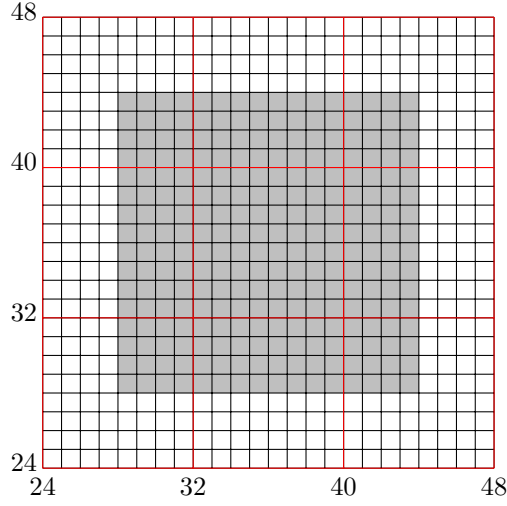
\begin{figure}
\begin{center}
    \begin{tikzpicture}[scale=0.25]
        \fill[lightgray] (4,4) rectangle (20,20);
        \foreach \i in {24,32,40,48}{
            \draw(\i-24,-1) node{$\i$}; 
        }
        \foreach \i in {24,32,40,48}{
            \draw(-1,\i-24+0.3) node{$\i$}; 
        }        
        \draw[step=1cm,black] (0,0) grid (24,24);
        \draw[step=8cm,red] (0,0) grid (24,24);
    \end{tikzpicture}
\end{center}
\caption{The gray two dimensional cells represent the modification neighborhood $\mBn{\xi}$ of the equilibrium cell $\xi=\defcell{2}{2}{0}{0}$ in Figure~\ref{fig:pos_neg_cycles}.
}
\label{fig:star_xi}
\end{figure}

\subsection{Negative Feedback}
\label{sec:2Dneg_fb}

We begin by considering the case that $\xi=[\bv', \bzero]$ is an equilibrium cell with negative feedback. Consider the D-grading $\pi \colon (\cX,\preceq_\cX) \rightarrow (\SCC(\cF_3),\posetF{3})$. 
We leave it to the reader to check that if $\xi\prec_\cX \tau_j$, for $j=0,1$, then 
\[
\pi(\xi)\posetF{3} \pi(\tau_0) =_{\bar{\cF_3}}\pi(\tau_1).
\]
Consider $[\bv,\bzero]\in\Janus$ where $\bv = 16\bv'$.
Define $\pi_3^\xi\colon \mBn{\xi} \rightarrow \SCC(\cF_3)$ by 
\[
\pi_3^\xi(\eta) \coloneqq\begin{cases}
\pi(\xi), & \text{if } \eta\in\mathrm{Cube}(\bv-\mathbf{3}, \bv+\mathbf{5}); \\
\pi(\tau_0), & \text{otherwise},
\end{cases}
\]
where $\mathrm{Cube}(\cdot,\cdot)$ is defined in \eqref{eq:cube}. 
Figure \ref{fig:2DnegJanus} shows the impact of this modification.

The following result can be obtained by visual inspection of Figure \ref{fig:2DnegJanus}. 
\begin{prop}\label{prop:2Dneg_fb_extendable}
    Let $\xi$ be an equilibrium cell with negative feedback in a two dimensional Janus complex $\Janus$. Then $\pi_3^\xi\colon \mBn{\xi} \rightarrow (\SCC(\cF_3),\posetF{3})$ is extendable.
\end{prop}

We now consider the case that $\xi=[\bv', \bzero]$ is an equilibrium cell with positive feedback.

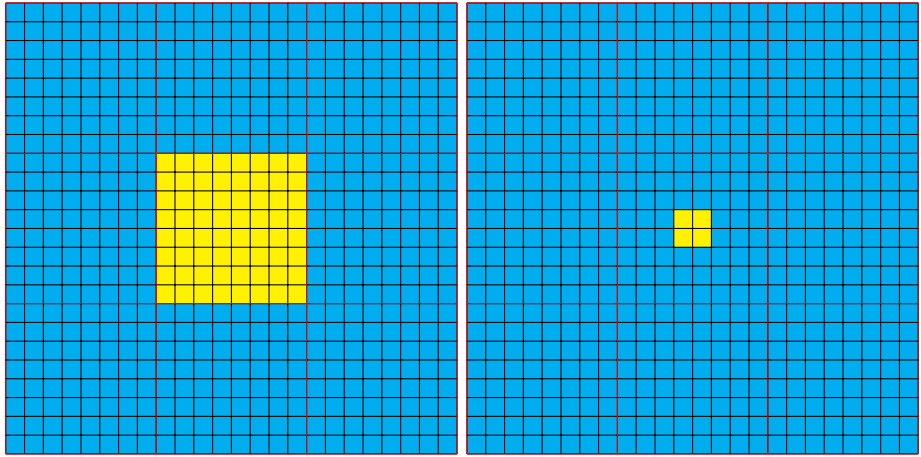
\begin{figure}[H]
    \centering
    \begin{tikzpicture}[scale=0.25]
    \fill[cyan] (0,0) rectangle (24,24); 
    \fill[yellow] (8,8) rectangle (16,16);
    \draw[step=1cm,black] (0,0) grid (24,24);
    \draw[step=8cm,red] (0,0) grid (24,24);
    \end{tikzpicture}
    \begin{tikzpicture}[scale=0.25]
    \fill[cyan] (0,0) rectangle (24,24); 
    \fill[yellow] (11,11) rectangle (13,13);
    \draw[step=1cm,black] (0,0) grid (24,24);
    \draw[step=8cm,red] (0,0) grid (24,24);
    \end{tikzpicture}

    \caption{On the right and on the left are the $\gradJanus$ and $\pi_3^\xi$ gradings for the subcomplex $\bar S(\blup(\sstarr(\xi)))\subset \Janus$ of an equilibrium cell with positive feedback, respectively. And the colors cyan and yellow represents the grading $\pi(\tau_0)$ and $\pi(\xi)$.}
    \label{fig:2DnegJanus}
\end{figure}

\subsection{Positive Feedback.}\label{sec:2Dpos_fb}

Let $\xi \in \cX$ be an equilibrium cell with positive feedback, that is, $\rmap\xi=\sigma$ with length equal to $2$ and $\delta_{\sigma} = 1$.

Set
\[
\FibTri(\xi)\coloneqq 
\bar{S}(\blup(\xi)) \cup \mTile(\xi,\sigma)
\]
and for any $\tau\in\cT^{(1)}(\xi)$ set
\[
\FibTri(\tau;\xi)\coloneqq 
\bar{S}(\blup(\tau)) - \FibTri(\xi),
\] 
where $\cT^{(1)}$ is defined in \eqref{eq:T1}. 

Define the grading $\pi_3^\xi\colon \mBn{\xi} \rightarrow (\SCC(\cF_3),\posetF{3})$ by
\[
\pi_3^\xi(\eta) \coloneqq\begin{cases}
\pi(\xi) & \text{if } \eta\in\FibTri(\xi); \\
\pi(\tau) &\text{if } \eta\in\FibTri(\tau;\xi), \\
    \pi_2(\eta),& \text{ elsewhere}.
\end{cases}
\]
Figure \ref{fig:2DposJanus} depicts the local grading assignment for an equilibrium cell $\xi$ with positive feedback. 
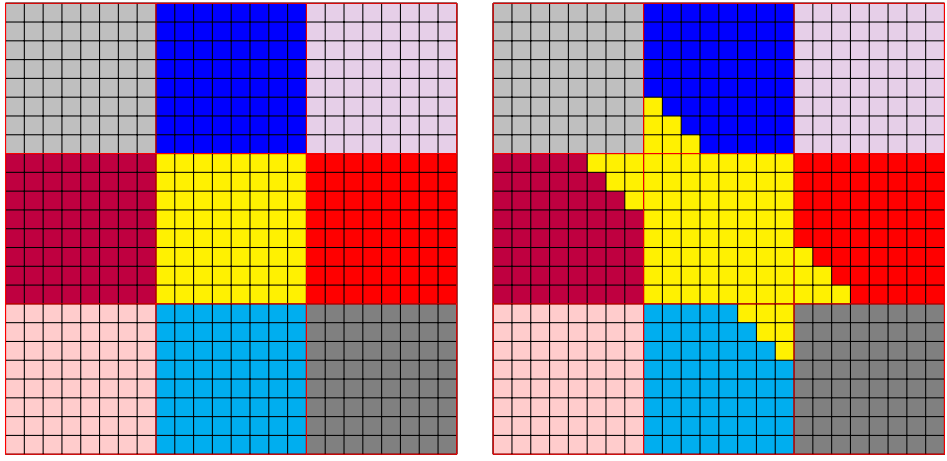
\begin{figure}[H]
    \centering
    \begin{tikzpicture}[scale=2]
    \draw[step=1cm,black] (0,0) grid (3,3);

    \fill[cyan] (1,0) rectangle (2,1); 
    \fill[purple] (0,1) rectangle (1,2); 
    \fill[red] (2,1) rectangle (3,2); 
    \fill[blue] (1,2) rectangle (2,3); 
    \fill[yellow] (1,1) rectangle (2,2); 
    \fill[lightgray] (1,2) rectangle (0,3); 
    \fill[gray] (2,1) rectangle (3,0);

    \fill[red!20!white] (0,0) rectangle (1,1);
    \fill[red!10!white!90!blue] (2,2) rectangle (3,3);
    
    \draw[step=0.125cm,black] (0,0) grid (3,3);
    \draw[step=1cm,red] (0,0) grid (3,3);

    \end{tikzpicture} \quad
    \begin{tikzpicture}[scale=2]
    \draw[step=1cm,black] (0,0) grid (3,3);

    \fill[cyan] (1,0) rectangle (2,1); 
    \fill[purple] (0,1) rectangle (1,2); 
    \fill[yellow] (1,1) rectangle (2,2); 
    \fill[red] (2,1) rectangle (3,2); 
    \fill[blue] (1,2) rectangle (2,3); 
    \fill[lightgray] (1,2) rectangle (0,3); 
    \fill[gray] (2,1) rectangle (3,0);

    \fill[red!20!white] (0,0) rectangle (1,1);
    \fill[red!10!white!90!blue] (2,2) rectangle (3,3);

    \fill[yellow] (2,1) -- ++(-0.375,0)
        -- ++(0,-0.125) -- ++(0.125,0)
        -- ++(0,-0.125) -- ++(0.125,0)
        -- ++(0,-0.125) -- ++(0.125,0) 
        -- cycle;
    \fill[yellow] (1,2) -- ++(0,-0.375)
        -- ++(-0.125,0) -- ++(0,0.125)
        -- ++(-0.125,0) -- ++(0,0.125)
        -- ++(-0.125,0) -- ++(0,0.125)
        -- cycle;
    \fill[yellow] (2,1) -- ++(0,0.375)
        -- ++(0.125,0) -- ++(0,-0.125)
        -- ++(0.125,0) -- ++(0,-0.125)
        -- ++(0.125,0) -- ++(0,-0.125)
        -- cycle;
    \fill[yellow] (1,2) -- ++(0.375,0)
        -- ++(0,0.125) -- ++(-0.125,0)
        -- ++(0,0.125) -- ++(-0.125,0)
        -- ++(0,0.125) -- ++(-0.125,0) 
        -- cycle;
    \draw[step=0.125cm,black] (0,0) grid (3,3);
    \draw[step=1cm,red] (0,0) grid (3,3);
    \end{tikzpicture}
    \caption{On the right and on the left are the $\gradJanus$ and $\pi_3^\xi$ gradings for the subcomplex $\bar S(\blup(\sstarr(\xi)))\subset \Janus$ of an equilibrium cell with positive feedback, respectively. Moreover, the colors represent the grading values where: gray $\posetF{3}$ yellow $\posetF{3}$ dark red $\posetF{3}$ pink; 
    dark gray $\posetF{3}$ yellow $\posetF{3}$ red $\posetF{3}$ purple;
    yellow $\posetF{3}$ light blue $\posetF{3}$ pink; and yellow $\posetF{3}$ dark blue $\posetF{3}$ purple.}
    \label{fig:2DposJanus}
\end{figure}

We leave to the reader to check that $\pi_3^\xi$ satisfies the three conditions in Theorem~\ref{thm:extension}, where $\pi=\pi_2$ and $\pi'=\pi_3^ \xi$.
As a consequence we obtain the following result.

\begin{prop}\label{prop:2Dpos_fb_extendable}
Consider $\cF_3\colon \cX\mvmap \cX$, where $\cX$  is a two-dimensional cubical complex.
Let $\xi$ be an equilibrium cell with positive feedback.
Then, 
\[
\pi_3^\xi\colon \mBn{\xi} \rightarrow (\SCC(\cF_3),\posetF{3})
\]
is extendable.
\end{prop}

\section{Local D-grading $\pi_3$ for the three dimensional Janus Complex}\label{sec:local_3D_Janus}
In this section, we define a new grading $\pi_3$ for the Janus complex $\Janus$ when $N=3$. In the following subsections we address each local modifications for the possible cases: Positive 3-Cycle, Negative 3-Cycle and 2-Cycle. 
\subsection{Positive 3-Cycle}

Let $\xi \in \cX^{(0)}$ be an equilibrium cell with positive feedback.
This implies that $\rmap\xi=(n_1 \ n_2 \ n_3)$ and $\delta_{(n_1 \, n_2 \, n_3)}=1$.
Recall from Example~\ref{ex:3Dpos_1} that $\Spineb(\xi,\rmap\xi)$ has two zero-dimensional cells such that $p_n(\zeta_1,b(\xi))=-p_n(\zeta_2,b(\xi))$, where $\blup(\xi)=[\bar\bv,\bone]$, $\zeta_1=[\bar\bv,0]$ and $\zeta_2=[\bar\bv+\bone,\bzero]$. In general, due to symmetry, the same holds for $\rmap\xi=(n_1 \ n_2 \ n_3)$, i.e.,
\[
    \Spineb(\xi,\rmap\xi) = \setof{\zeta_1,\zeta_2} \subset \cX_b^{(0)},
\]
with $\zeta_1,\zeta_2\in\cX_b^{(0)}$,
$p_n(\zeta_1,b(\xi))=-p_n(\zeta_2,b(\xi))$ for each $n=1,2,3$. 
Let $\bar{S}(\zeta_i) = [\bv^i,\bzero] \in \Janus$ and define the set
\[
\FibTri(\xi) \coloneqq \mathrm{Line}(\bv^1, \bv^2) \cup \left(  \bigcup_{i=1}^2  B_2([\bv^i,\bzero]) \setminus  \left(\bar{S}(\blup(\xi)) \cup \bigcup_{\beta\in\Stable(\xi)} \bar{S}(\blup(\beta)) \right)  \right),
\]
where $\mathrm{Line}(\cdot,\cdot)$ is defined in \eqref{eq:line}.
We leave it to the reader to check that if $\tau_j\in \Stable(\xi)$, for $j=0,1$, then \[
\pi(\xi)\posetF{3} \pi(\tau_0) =_{\bar{\cF_3}}\pi(\tau_1).
\]
Define the grading $\pi_3^\xi\colon \mBn{\xi} \rightarrow (\SCC(\cF_3),\posetF{3})$ by:
\[ 
\pi_3^\xi(\eta)\coloneqq
\begin{cases}
    \pi(\xi)& \text{ if } \eta \in \FibTri(\xi),\\
    \pi(\tau_0)& \text{ if } \eta \in \bar{S}(\blup(\xi)) - \FibTri(\xi), \\
    \pi_2(\eta),& \text{ elsewhere}.
\end{cases}
\] 
Figure \ref{fig:pos3cycle} illustrates the change in grading when $\rmap{\xi}=(1\ 2\ 3)$, $\delta_{(1\ 2\ 3)}$ and $\kappa\in\Top_\cX(\xi)$ from Definition~\ref{defn:lap-number}, such that $\omega(\kappa_n^{-},\kappa) =1,$ for all $n=1,2,3.$
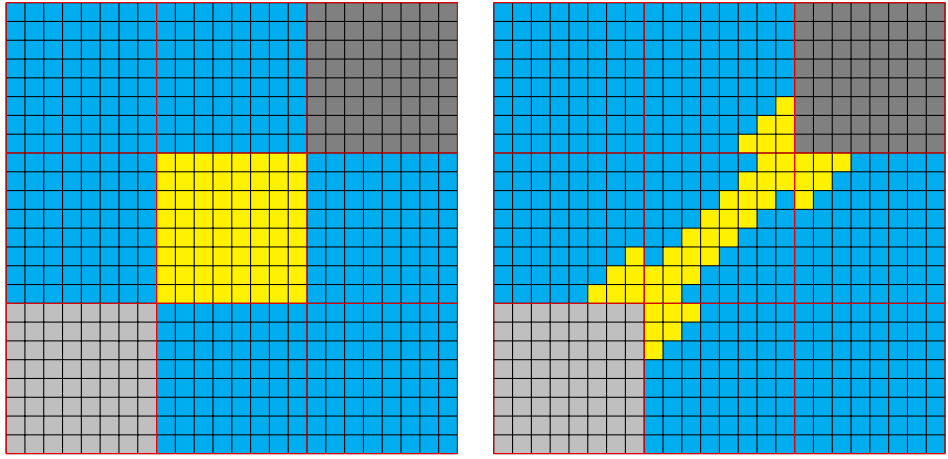
\begin{figure}[H]
    \centering
    \begin{tikzpicture}[scale=0.25]
    \fill[cyan] (0,0) rectangle (24,24); 
    \fill[lightgray] (0,0) rectangle (8,8);
    \fill[gray] (16,16) rectangle (24,24);
    \fill[yellow] (8,8) rectangle (16,16);
    \draw[step=1cm,black] (0,0) grid (24,24);
    \draw[step=8cm,red] (0,0) grid (24,24);
    \end{tikzpicture}
    \quad
    \begin{tikzpicture}[scale=.25]
        \fill[cyan] (0,0) rectangle (24,24); 
        \fill[lightgray] (0,0) rectangle (8,8);
        \fill[gray] (16,16) rectangle (24,24);
        \foreach \x in {8,...,14} {
            \fill[yellow] (\x,\x) rectangle (\x+2,\x+2);
        }        
        \fill[yellow] (5,8) -- (5,9) -- (6,9) -- (6,10) -- (7,10) -- (7,11) -- (8,11) -- (8,8) -- (5,8);
        \fill[yellow] (8,5) -- (9,5) -- (9,6) -- (10,6) -- (10,7) -- (11,7) -- (11,8) -- (8,8) -- (8,5);
        \fill[yellow] (16,16) -- (19,16) -- (19,15) -- (18,15) -- (18,14) -- (17,14) -- (17,13) -- (16,13) -- (16,16);
        \fill[yellow] (16,16) -- (16,19) -- (15,19) -- (15,18) -- (14,18) -- (14,17) -- (13,17) -- (13,16) -- (16,16);
        
        \fill[lightgray] (7,7) rectangle (8,8);
        \fill[gray] (16,16) rectangle (17,17);
        \draw[step=1cm,black] (0,0) grid (24,24);
        \draw[step=8cm,red] (0,0) grid (24,24);
\end{tikzpicture}
    
    \caption{Fixing the third coordinate of $\bv + \mathbf{4}$, on the right and on the left are the slices of the $\gradJanus$ and $\pi_3^\xi$ gradings for the subcomplex $\bar S(\blup(\sstarr(\xi)))\subset \Janus$ of an equilibrium cell with positive feedback, respectively. Moreover, the colors represent the grading where: gray $\posetF{3}$ yellow $\posetF{3}$ light blue; yellow $\posetF{3}$ light gray.}
    \label{fig:pos3cycle}
\end{figure}
\begin{prop}\label{prop:3Dpos_fb_extendable}
    Let $\xi$ be an equilibrium cell with positive feedback in a three dimensional Janus complex $\Janus$. Then 
    \[\pi_3^\xi\colon \mBn{\xi} \rightarrow (\SCC(\cF_3),\posetF{3})\] 
    is extendable.
\end{prop}
\begin{proof}
    Let 
    \[\cV=\setdef{\eta\in \mBn{\xi}}{\pi_2(\eta)\neq\pi_3^ \xi(\eta)}.\] We will show that $\pi=\pi_2$ and $\pi'=\pi_3^ \xi$ as in Theorem~\ref{thm:extension} satisfy the conditions (1)-(3). Consequently, $\pi_3^\xi$ is extendable. 
    
    Clearly $\cV$ is defined to satisfy condition (1) of Theorem~\ref{thm:extension}.
    
    Now, observe that the minimal elements of the set  $\pi_3^\xi\left(\mBn{\xi}\right)$ are precisely the grading values $\pi(\mu)$ for $\mu\in\cU(\xi)$. From the definition of $\pi_3^\xi$, it follows that 
    \[\pi_2(\eta)=\pi_3^\xi(\eta), \text{ for all } \eta\in \displaystyle\bigcup_{\mu\in\cU(\xi)} \bar S(\blup(\mu)).\] 
    As a consequence, condition (2) in Theorem~\ref{thm:extension} is satisfied.
    
    Finally, notice that the elements that are not comparable in \[\pi_3^\xi\left(\mBn{\xi}\right)\] 
    are the minimal elements, i.e., the values $\pi(\mu)$ such that $\mu\in\cU(\xi)$. Furthermore, for any distinct $\mu,\mu'\in\cU(\xi)$, the sets $\bar S(\blup(\mu))$ and $\bar S(\blup(\mu'))$ are separated by at least 8. Consequently,  $\cl(\mu)\cap\cl(\mu') = \emptyset$. Thus $\cl(\mu)\cap\cl(\mu')\cap\cl(\mu'') = \emptyset$ for any $\mu''\in\cV$. Therefore, the condition (3) in Theorem~\ref{thm:extension} is satisfied. 
\end{proof}

\subsection{Negative Feedback}
Let $\xi = [\bv', \bzero]\in \cX^{(0)}$ an equilibrium cell with negative feedback. This implies that $\rmap{\xi}=(n_1\ n_2\ n_3)$ and $\delta_{(n_1\ n_2\ n_3)}=-1$.

Without loss of generality, assume that $\rmap{\xi}=(1\ 2\ 3)$ and there exist $\mu_\bzero,\mu_\bone\in \cU(\xi)$ such that  $p(\xi,\mu_\bzero) = (-1,-1,-1)$ and $p(\xi,\mu_\bone) = (1,1,1)$, note that this is precisely the case described in Figure~\ref{fig:wall_labelling_repressilator}. Let $\bv = 16\bv'$, then set
\[
\FibTri(\xi) \coloneqq \mathrm{Line}(\bv^1, \bv^2),
\]
and
\begin{multline*}
\mathrm{Uns} \coloneqq \mathrm{Cube}({\bf 8}, {\bf 16})\cup \\
\setof{[(v_1,v_2,v_3), \bone]\in \Janus \mid 12 \leq v_1+v_2+v_3 < 24 \text{ and at least one } v_i\in \setof{7,16}} \\
- \FibTri(\xi),
\end{multline*}
where $\mathrm{Cube}(\cdot,\cdot)$ is defined in \eqref{eq:cube}.

We leave to the reader to check that Lemma~\ref{lem:F_3-q-same_grading} and Lemma~\ref{lem:F_3-minimal} ensure that if $\tau_j\in \cU(\xi)$, for $j=0,1$, then 
\[
\pi(\tau_0) =_{\bar{\cF_3}}\pi(\tau_1) \posetF{3} \pi(\xi).
\]
Define the grading $\pi_3^\xi\colon \mBn{\xi} \rightarrow (\SCC(\cF_3),\posetF{3})$ by
\[ 
\pi_3^\xi(\eta)\coloneqq 
\begin{cases}
    \pi(\xi),& \text{ if } \eta \in \FibTri(\xi);\\
    \pi(\tau_0),& \text{ if } \eta \in \mathrm{Uns};\\
    \pi_2(\eta),& \text{ elsewhere}.
\end{cases}
\] 
Figures \ref{fig:neg3cycle} and \ref{fig:neg3cycle2}  illustrate the change.

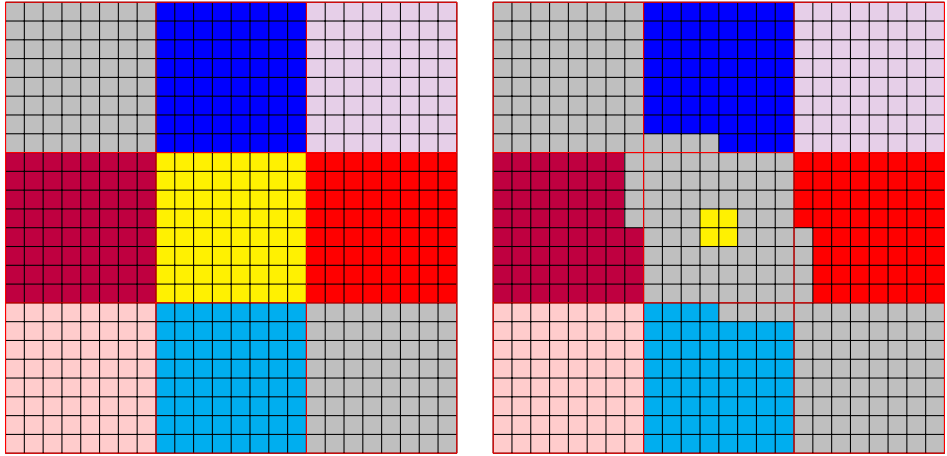
\begin{figure}[H]
\label{fig:geometrization-rule3-2D-negative}
    \centering    
    \begin{tikzpicture}[scale=0.25]

        \fill[cyan] (8,0) rectangle (16,8); 
        \fill[purple] (0,8) rectangle (8,16); 
        \fill[red] (16,8) rectangle (24,16); 
        \fill[blue] (8,16) rectangle (16,24); 
        \fill[yellow] (8,8) rectangle (16,16); 
        \fill[lightgray] (8,16) rectangle (0,24); 
        \fill[lightgray] (16,8) rectangle (24,0);

        \fill[red!20!white] (0,0) rectangle (8,8);
        \fill[red!10!white!90!blue] (16,16) rectangle (24,24);
        
        \draw[step=1cm,black] (0,0) grid (24,24);
        \draw[step=8cm,red] (0,0) grid (24,24);
    
    \end{tikzpicture} 
    \quad
    \begin{tikzpicture}[scale=0.25]
        \fill[cyan] (8,0) rectangle (16,8); 
        \fill[purple] (0,8) rectangle (8,16); 
        \fill[red] (16,8) rectangle (24,16); 
        \fill[blue] (8,16) rectangle (16,24); 
        
        \fill[red!20!white] (0,0) rectangle (8,8);
        \fill[red!10!white!90!blue] (16,16) rectangle (24,24);
        
        \fill[lightgray] (8,8) rectangle (16,16);
        \fill[lightgray] (8,16) rectangle (0,24); 
        \fill[lightgray] (16,8) rectangle (24,0);
        \fill[lightgray] (12,12) rectangle (7,17);
        \fill[lightgray] (12,12) rectangle (17,7);
        
        \fill[yellow] (11,11) rectangle (13,13);
        
        \draw[step=1cm,black] (0,0) grid (24,24);
        \draw[step=8cm,red] (0,0) grid (24,24);
    \end{tikzpicture}
    \caption{ Fixing the third coordinate of $\bv + \mathbf{4}$, on the right and on the left are the slices of the $\gradJanus$ and $\pi_3$ gradings for the subcomplex $\bar S(\blup(\sstarr(\xi)))\subset \Janus$ of an equilibrium cell with negative feedback, respectively. 
    Moreover, the colors represent the grading where: 
    gray $\posetF{3}$ yellow $\posetF{3}$ dark red $\posetF{3}$ pink;
    yellow $\posetF{3}$ light blue $\posetF{3}$ pink;
    yellow $\posetF{3}$ dark blue $\posetF{3}$ purple; and 
    yellow $\posetF{3}$ red $\posetF{3}$ purple.
    }
    \label{fig:neg3cycle}
\end{figure}

\begin{figure}[H]
\label{fig:geometrization-rule3-2D-negative}
    \centering    
    \begin{tikzpicture}[scale=0.25]
        \fill[cyan] (8,0) rectangle (16,8); 
        \fill[purple] (0,8) rectangle (8,16); 
        \fill[red] (16,8) rectangle (24,16); 
        \fill[blue] (8,16) rectangle (16,24); 

        \fill[red!20!white] (0,0) rectangle (8,8);
        \fill[red!10!white!90!blue] (16,16) rectangle (24,24);
        
        \fill[lightgray] (8,8) rectangle (16,16);
        \fill[lightgray] (8,16) rectangle (0,24); 
        \fill[lightgray] (16,8) rectangle (24,0);
        \fill[lightgray] (12,12) rectangle (7,17);
        \fill[lightgray] (12,12) rectangle (17,7);

        \foreach \x in {8,...,14} {
            \fill[yellow] (\x,\x) rectangle (\x+2,\x+2);
        }

        \draw[step=1cm,black] (0,0) grid (24,24);
        \draw[step=8cm,red] (0,0) grid (24,24);
        
    \end{tikzpicture}
    \caption{Fixing the third coordinate of $\bv +\mathbf{4}$, the figure is a slice of  $\pi_3$ grading for the subcomplex $\bar S(\blup(\sstarr(\xi)))\subset \Janus$ of an equilibrium cell with negative feedback, where the stable cells $\Stable(\xi)$ are projected to this slice.}
    \label{fig:neg3cycle2}
\end{figure}
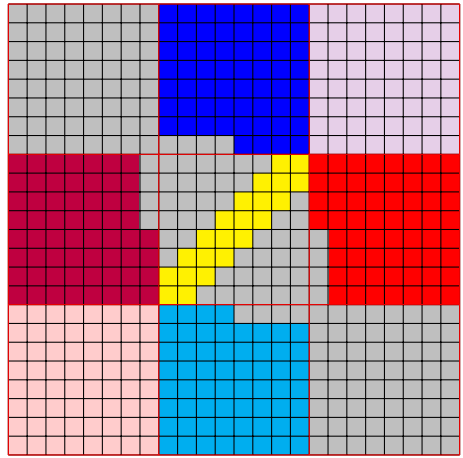

\begin{prop}
\label{prop:3Dneg_fb_extendable}
    Let $\xi$ be an equilibrium cell with positive feedback in a three dimensional Janus complex $\Janus$. Then \[\pi_3^\xi\colon \mBn{\xi} \rightarrow (\SCC(\cF_3),\posetF{3})\] is extendable.
\end{prop}

\begin{proof}
    Let $\mu_\bzero,\mu_\bone\in \cU(\xi)$ be distinct unstable cells. To simplify the notation denote 
    \[\cB \coloneqq \left((\pi_3^\xi)^{-1}(\pi(\mu_\bone))\right)\bigcup
    \left((\pi_3^\xi)^{-1}(\pi(\mu_\bzero))\right).\]
    The strategy is to prove that $\pi_3^\xi$ is extendable over both $\mBn{\xi} \setminus \cB$ and $\cB$. Furthermore, we use them to define 
    \[
    \pi_3^\xi(\zeta)\coloneqq 
    \begin{cases}
         \left.\pi_3^\xi\right|_{\mBn{\xi} \setminus \cB}(\zeta), & \text{if } \zeta\in \cl\left(\mBn{\xi}\setminus \cB\right)\\
         \left.\pi_3^\xi\right|_{\cB}(\zeta), & \text{if } \zeta\in \cl(\cB) \setminus \cl\left(\mBn{\xi}\setminus \cB\right) \\
    \end{cases}
    \] for any $\zeta\in \cl\left(\mBn{\xi}\right)$, and verify that $\pi_3^\xi(\zeta)$ is admissible grading on
    \[
        \cl\left(\mBn{\xi} \setminus \cB\right) 
        \bigcap \cl\left(\cB\right).
    \]

    First, observe that $\pi(\xi)^<$ is the unique minimum of $\pi_3^\xi\left(\mBn{\xi}\right)$. By Proposition~\ref{prop:min_extendable},  $\pi_3^\xi$ is extendable over \[\mBn{\xi} \setminus \cB.\]
    
    Second, since
    $(\pi_3^\xi)^{-1}(\pi(\mu_\bone))$ and $(\pi_3^\xi)^{-1}(\pi(\mu_\bone))$ are separated by at least one, by Theorem~\ref{thm:extension}, $\pi_3^\xi$ is extendable over $\cB$.

    It remains to show that $\pi_3^\xi$ is admissible grading on
    \[
        \cl\left(\mBn{\xi} \setminus \cB\right) 
        \bigcap \cl\left(\cB\right).
    \]
    In fact, given that $\pi(\mu_\bone)$ and $\pi(\mu_\bone)$ are the maximal elements of \[\pi_3^\xi(\mBn{\xi}),\] then 
    $
        \pi_3^\xi(\mu')\posetF{3}\pi_3^\xi(\mu), \text{ for all } \mu'\in\mBn{\xi} \setminus \cB \text{ and } \mu\in\cB.
    $ 
    Hence, for all 
    \[
        \zeta \in \cl\left(\mBn{\xi} \setminus \cB\right) 
        \bigcap \cl\left(\cB\right)
    \]
    it follows that
    \begin{align*}
        \pi_3^\xi(\zeta) &= \min\setdef{\pi_3^\xi(\mu)}{\mu\in\Top_{\mBn{\xi}}(\zeta)}\\
        &= \min\setdef{\pi_3^\xi(\mu)}{\mu\in\Top_{\mBn{\xi} \setminus \cB}(\zeta)}\\
        &= \left.\pi_3^\xi\right|_{\mBn{\xi} \setminus \cB}(\zeta),
    \end{align*}
    where the last equality is well defined since $\pi_3^\xi$ is extendable over \[\mBn{\xi} \setminus \cB.\]
\end{proof}

\subsection{2-Cycles}
The strategy to define locally a D-grading for a cell $\xi$, whose regulation map $o_\xi$ is a bijection and has a 2-cycle $\alpha_\xi$, consists in using the D-grading already defined in Subsections \ref{sec:2Dneg_fb} and \ref{sec:2Dpos_fb} for a two dimensional Janus complex. To accomplish this, we introduce the notion of projection of a $N$ dimensional cubical complex $\cX\left(\prod_{n=1}^N I_n\right)$ to a low dimensional cubical complex $\cX\left(\prod_{n\in \setof{i_1, \ldots, i_\ell}} I_n\right)$, where $\setof{i_1, \ldots, i_\ell} \subset \setof{1, \ldots, N}$.

Let $\xi\in\cX\left(\prod_{n=1}^N I_n\right)$ and $\alpha_\xi$ a cycle of the regulation map $o_\xi$, then we define the projection map of $\xi$ onto its cycle $\alpha_\xi$ by
\begin{align*}
   \mathrm{proj}_{\alpha_\xi}: \cX\left(\prod_{n=1}^N I_n\right) &\to \cX\left(\prod_{n\in\alpha_\xi} I_n\right)  \\
    \defcellb{v_1}{\vdots}{v_N}{w_1}{\vdots}{w_N} &\mapsto  \defcellb{v_{i_1}}{\vdots}{v_{i_\ell}}{w_{i_1}}{\vdots}{w_{i_\ell}},
\end{align*}
where $\alpha=(i_1\ldots i_\ell)$. Naturally, $\mathrm{proj}_{\alpha_\xi}$ does not have an inverse map; however, we can define a right inverse map of $\mathrm{proj}_{\alpha_\xi}$, i.e., a map $\mathrm{proj}_{\alpha_\xi}^{-1}$ such that $\mathrm{proj}_{\alpha_\xi} \circ \mathrm{proj}_{\alpha_\xi}^{-1} = \id_{\cX\left(\prod_{n\in\alpha_\xi} I_n\right)}$. The right inverse map of interesting is defined as follows
\begin{align*}
   \mathrm{proj}_{\alpha_\xi}^{-1}\coloneqq \cX\left(\prod_{n\in\alpha_\xi} I_n\right) &\to \cX\left(\prod_{n=1}^N I_n\right)   \\
    \defcellb{v_{i_1}}{\vdots}{v_{i_\ell}}{w_{i_1}}{\vdots}{w_{i_\ell}} &\mapsto  \defcellb{v_0'}{\vdots}{v_N'}{w_0'}{\vdots}{w_N'},
\end{align*}
where $\xi=[\bv,\bw]$, 
\[ 
v'_i = 
\begin{cases}
    v_i, &\text{if } i \in \alpha_\xi;\\
    \bv_i,& \text{elsewhere},
\end{cases}
\quad \text{and} \quad
w'_i = 
\begin{cases}
    w_i, &\text{if } i \in \alpha_\xi;\\
    0,& \text{elsewhere}.
\end{cases}
\]

Let $\cX_{\alpha_\xi}=\mathrm{proj}_{\alpha_\xi}(\sstarr(\xi))$ and 
$\pi_{\alpha_\xi}\colon \cX_{\alpha_\xi} \rightarrow (\SCC(\cF_3),\posetF{3})$ be defined by 
\[
     \pi_{\alpha_\xi}(\sigma)\coloneqq \pi\left(\mathrm{proj}_{\alpha_\xi}^{-1}(\sigma)\right),
\]
where $\pi\colon \cX \rightarrow (\SCC(\cF_3),\posetF{3})$.

Let $\xi=[\bv,\bw]$ be a cell whose regulation map $o_\xi$ is a bijection and has a 2-cycle $\alpha_\xi$. Notice that $\mathrm{proj}_{\alpha_\xi}(\xi)$ is an equilibrium cell with either positive feedback or negative feedback on a two dimensional complex $\Bar{S}(\blup(\cX_{\alpha_\xi}))$. Furthermore, $\cX_{\alpha_\xi}$ is a two dimensional complex, hence, we can use the D-grading $\pi_{\alpha_\xi}\colon \cX_{\alpha_\xi} \rightarrow (\SCC(\cF_3),\posetF{3})$ to obtain  
\[
\pi_3^{\mathrm{proj}_{\alpha_\xi}(\xi)}\colon \mBn{\mathrm{proj}_{\alpha_\xi}(\xi)} \to (\SCC(\cF_3),\posetF{3})
\]
by following the same procedure done in Subsections \ref{sec:2Dneg_fb} and \ref{sec:2Dpos_fb}. 

Now, we use $\pi_3^{\mathrm{proj}_{\alpha_\xi}(\xi)}$
to define the D-grading 
\[
\pi_3^\xi\colon \mBn{\xi} \rightarrow (\SCC(\cF_3),\posetF{3})
\]
by setting
\[ 
\pi_3^\xi(\eta)\coloneqq 
\begin{cases}
    \pi_3^{\mathrm{proj}_{\alpha_\xi}(\xi)}(\mathrm{proj}_{\alpha_\xi}(\eta)), &\text{if } \eta \in \Bar{S}(\blup(\mathrm{proj}_{\alpha_\xi}^{-1}(\cX_{\alpha_\xi})));\\
    \pi_2(\eta),& \text{ elsewhere}.
\end{cases}
\] 

\begin{prop}
\label{prop:2cycle_extendable}
    Let $\xi$ be a cell whose regulation map has a 2-cycle in a three dimensional Janus complex $\Janus$. Then $\pi_3^\xi\colon \mBn{\xi} \rightarrow (\SCC(\cF_3),\posetF{3})$ is extendable.
\end{prop}
\begin{proof}
Since, $\pi_3^\xi=\pi_2$ over $\mBn{\xi}\setminus \Bar{S}(\blup(\mathrm{proj}_{\alpha_\xi}^{-1}(\cX_{\alpha_\xi})))$ then \[\left.\pi_3^\xi\right|_{\mBn{\xi}\setminus \Bar{S}(\blup(\mathrm{proj}_{\alpha_\xi}^{-1}(\cX_{\alpha_\xi})))}\] is extendable.
    
    Now, we prove that $\left.\pi_3^\xi\right|_{\Bar{S}(\blup(\mathrm{proj}_{\alpha_\xi}^{-1}(\cX_{\alpha_\xi})))}$ is extendable. In fact, for all $\zeta\in\cl(\Bar{S}(\blup(\mathrm{proj}_{\alpha_\xi}^{-1}(\cX_{\alpha_\xi}))))$, it follows that 
        \begin{align*}
        \left.\pi_3^\xi\right|_{\Bar{S}(\blup(\mathrm{proj}_{\alpha_\xi}^{-1}(\cX_{\alpha_\xi})))}(\zeta) &= \min\setdef{\pi_3^\xi(\mu)}{\mu\in\Top_{\Bar{S}(\blup(\mathrm{proj}_{\alpha_\xi}^{-1}(\cX_{\alpha_\xi})))}(\zeta)}\\
        &= \min\setdef{\pi_3^{\mathrm{proj}_{\alpha_\xi}(\xi)}(\mathrm{proj}_{\alpha}(\mu))}{\mu\in\Top_{\Bar{S}(\blup(\mathrm{proj}_{\alpha_\xi}^{-1}(\cX_{\alpha_\xi})))}(\zeta)}\\
        &= \min\setdef{\pi_3^{\mathrm{proj}_{\alpha_\xi}(\xi)}(\eta)}{\eta\in \mathrm{proj}_{\alpha}\left(\Top_{\Bar{S}(\blup(\mathrm{proj}_{\alpha_\xi}^{-1}(\cX_{\alpha_\xi})))}(\zeta)\right)}\\
        &= \min\setdef{\pi_3^{\mathrm{proj}_{\alpha_\xi}(\xi)}(\eta)}{\eta\in \Top_{\mathrm{proj}_{\alpha}(\Bar{S}(\blup(\mathrm{proj}_{\alpha_\xi}^{-1}(\cX_{\alpha_\xi}))))}(\mathrm{proj}_{\alpha}\left(\zeta\right))}\\
        &= \min\setdef{\pi_3^{\mathrm{proj}_{\alpha_\xi}(\xi)}(\eta)}{\eta\in \Top_{\Bar{S}(\blup(\cX_{\alpha_\xi}))}(\mathrm{proj}_{\alpha}\left(\zeta\right))}\\
        &= \pi_3^{\alpha_\xi}(\zeta),
    \end{align*}
    the last equality implies that $\left.\pi_3^\xi\right|_{\Bar{S}(\blup(\mathrm{proj}_{\alpha_\xi}^{-1}(\cX_{\alpha_\xi})))}(\zeta)$ has a unique element, i.e., $\left.\pi_3^\xi\right|_{\Bar{S}(\blup(\mathrm{proj}_{\alpha_\xi}^{-1}(\cX_{\alpha_\xi})))}$ is extendable. 

    It remains to show that $\pi_3^\xi(\zeta)$ is admissible grading for any 
    \[
    \zeta\in\cl\left(\mBn{\xi}\setminus \Bar{S}(\blup(\mathrm{proj}_{\alpha_\xi}^{-1}(\cX_{\alpha_\xi})))\right)
    \bigcap
    \cl\left(\Bar{S}(\blup(\mathrm{proj}_{\alpha_\xi}^{-1}(\cX_{\alpha_\xi})))\right).
    \]
    Indeed, either
    \[
    \pi_3^\xi(\zeta) = 
    \left.\pi_3^\xi\right|_{\mBn{\xi}\setminus \Bar{S}(\blup(\mathrm{proj}_{\alpha_\xi}^{-1}(\cX_{\alpha_\xi})))}(\zeta)
    \]
    or
    \[
    \pi_3^\xi(\zeta) = 
    \left.\pi_3^\xi\right|_{\Bar{S}(\blup(\mathrm{proj}_{\alpha_\xi}^{-1}(\cX_{\alpha_\xi})))}(\zeta),
    \]
    since 
    either 
    $\gradJanus(\mu) \leq \gradJanus(\mu')$ or 
    $\gradJanus(\mu') \leq \gradJanus(\mu)$ for any 
    \[\mu\in\Top_{\mBn{\xi}\setminus\Bar{S}(\blup(\mathrm{proj}_{\alpha_\xi}^{-1}(\cX_{\alpha_\xi})))}(\zeta)\] and $\mu'\in \Top_{\Bar{S}(\blup(\mathrm{proj}_{\alpha_\xi}^{-1}(\cX_{\alpha_\xi})))}(\zeta)$. As a consequence, either 
    $\pi_3^\xi(\mu) \leq \pi_3^\xi(\mu')$
    or 
    $\pi_3^\xi(\mu') \leq \pi_3^\xi(\mu)$. Hence, $\pi_3^\xi(\zeta)$ is the unique minimal element in 
    \[
    \min\setdef{\pi_3^\xi(\mu)}{\mu\in\Top_{\Janus}(\zeta)},
    \]
    i.e., $\pi_3^\xi(\zeta)$ is admissible grading.
\end{proof}

\section{Global D-grading $\pi_3$}\label{sec:global_pi3}

In this section, we prove that $\pi_3\colon \Janus^{(N)} \rightarrow (\SCC(\cF_3),\posetF{3})$ defined by
\[ 
\pi_3(\eta)\coloneqq 
\begin{cases}
    \pi_3^ \xi(\eta), &\text{if } \exists~\xi\in\Xi \text{ such that }\eta\in\mBn{\xi};\\
    \pi_2(\eta),& \text{ elsewhere}.
\end{cases}
\] 
is an extendable D-grading over $\Janus$. Furthermore, we conclude this section by demonstrating that $(C_\ast(\Janus), \gradJanus)$ and $(C_\ast(\Janus), \pi_3)$ are $\SCC(\cF_3)$-graded chain equivalent.
\begin{thm}
\label{thm:gradedchainequiv3}
Consider $\cF_3\colon \cX\mvmap \cX$ with D-grading $\pi\colon \cX \to \SCC(\cF_3)$. Then, 
$(C_\ast(\Janus), \gradJanus)$ and $(C_\ast(\Janus), \pi_3)$ are $\SCC(\cF_3)$-graded chain equivalent.
\end{thm}

First, notice that $\left(\mBn{\xi}\right) \cap \left(\mBn{\xi'}\right) = \emptyset$ for different $\xi,\xi'\in\Xi$, then $\pi_3$ is a well defined map. 
Second, we state the following proposition which guarantees that $\pi_3$ is extendable.

\begin{prop}
    The map $\pi_3 \colon \Janus^{(N)} \rightarrow (\SCC(\cF_3),\posetF{3})$ is an extendable D-grading over $\Janus$.
\end{prop}
\begin{proof}
    Since $\pi_2$ is extendable and \[\left.\pi_3\right|_{\Janus^ {N}\setminus\cup_{\xi\in\Xi}\mBn{\xi}} = \left.\pi_2\right|_{\Janus^ {N}\setminus\cup_{\xi\in\Xi}\mBn{\xi}}\]
    then 
    \[
    \left.\pi_3\right|_{\Janus^ {N}\setminus\cup_{\xi\in\Xi}\mBn{\xi}}
    \]
    is extendable.
    By Propositions~\ref{prop:2Dneg_fb_extendable}, \ref{prop:2Dpos_fb_extendable}, \ref{prop:3Dneg_fb_extendable} and \ref{prop:3Dpos_fb_extendable}, it follows that $\left.\pi_3\right|_{\mBn{\xi}}$ is extendable for any $\xi\in\Xi$.

    It remains to show that $\pi_3(\zeta)$ is admissible grading for any 
    \[
    \zeta\in
    \cl\left( \Janus^ {N}\setminus\cup_{\xi\in\Xi}\mBn{\xi} \right) 
    \cap 
    \cl\left( \mBn{\xi} \right)
    \]
    and any
    \[
    \zeta\in
    \cl\left( \mBn{\xi} \right) 
    \cap 
    \cl\left( \mBn{\xi'} \right),
    \]
    where $\xi,\xi'\in\Xi$ and $\xi\neq\xi'$. In fact, $\pi_3(\zeta)$ is admissible grading for both cases, since for any 
    \[\zeta\in \bdy\left( \cl\left( \mBn{\xi} \right) \right)\] it follows that $\pi_3(\mu)=\pi_2(\mu)$ for all $\mu\in\Top_{\Janus}(\zeta)$, i.e., $\pi_3(\zeta)=\pi_2(\zeta)$. Hence $\pi_3(\zeta)$ is admissible grading.
\end{proof}

Observe that the $\SCC(\cF_3)$-graded chain homotopies can be visualized or implemented computationally on a case-by-case basis. This procedure is analogous to the homotopies detailed in Chapter~\ref{sec:P2-grading}, leading in the following proposition.

\begin{prop}
\label{prop:chain32}
    $(C_\ast(\Janus), \pi_2)$ and $(C_\ast(\Janus), \pi_3)$ are $\sP$-graded chain equivalent.
\end{prop}

\begin{proof}[Proof of Theorem~\ref{thm:gradedchainequiv3}]
    The result follows directly from Theorem~\ref{thm:gradedchainequiv} and Proposition~\ref{prop:chain32}.
\end{proof}

\chapter{Global Dynamics of Ramp Systems via $\cF_3$}
\label{sec:R3Dynamics}

Unless otherwise stated the standing hypothesis throughout this chapter is the following:
\begin{description}
\item[H3] Consider an $N$-dimensional ramp system given by
\begin{equation}
    \label{eq:rampH3}
    \dot{x} = -\Gamma x + E(x; \nu, \theta, h)
\end{equation}
with parameters $(\gamma, \nu, \theta)\in \Lambda(S)$ (see \eqref{eq:lambdaS}) and $h \in \cH_3(\gamma, \nu, \theta)$ (see Definition~\ref{defn:H3}). 
Let $\cX=\cX(\I)$ be the ramp induced cubical complex, $\omega \colon W(\cX) \to \setof{\pm 1}$ be the associated wall labeling and $\rook : TP(\cX) \to \setof{0,\pm 1}$ be the associated Rook Field (see Section~\ref{sec:ramp2rook}).  
Let $\cX_b$ and $\Janus$ be the associated blow-up and Janus complexes respectively (see Section~\ref{sec:blowup_complex} and Section~\ref{sec:janusComplex}).
Let $\recG$ be an associated rectangular geometrization of $\cX_b$.  
Let $\cF_3 \colon \cX \mvmap \cX$ be the associated combinatorial multivalued map derived from Definition~\ref{defn:Rule3} and let $\pi_3 \colon \Janus \to \SCC(\cF_3)$ be the D-grading as in Section~\ref{sec:global_pi3}. 
\end{description}
The primary goal of this chapter is to prove the following theorem. 
\begin{thm}
\label{thm:R3ABlattice}
Given  hypothesis {\bf H3} and assuming $N=2$ or $3$, there exists a geometrization $\bG_3$ of $\Janus$ that is aligned with $\eqref{eq:rampH3}$ over all $\cN \in \sN(\cF_3)$. 
\end{thm}

As in Chapter~\ref{sec:R2Dynamics}, we construct a geometrization $\bG_3$ and show that given any $(N-1)$-dimensional cell $\zeta \in \bbdy(\blup(\cN))$,
$\bg(\zeta)$ is a smooth $(N-1)$-dimensional manifold over which the vector field is transverse.

We recall the identification of the cells in which Conditions 3.1 and 3.2 were applied. By Definition~\ref{defn:Rule3}, those are the semi-opaque cells $\xi \in \cX$ whose cycle decomposition of its regulation map $\rmap\xi$ contains cycles of length $k\geq 2$. When $N=2$ and $N=3$, the possibilities are as follows.
\begin{itemize}
    \item[(i)] $N=2$ and $\rmap\xi = (1 \ 2)$,
    \item[(ii)] $N=3$ and $\rmap\xi = (n_1 \ n_2 )$, $(n_1 \ n_2) (n_3)$, or $(n_1 \ n_2 \ n_3)$.
\end{itemize}
Note that when $N=2$, $\xi$ is necessarily an equilibrium cell. While $N=3$ it might be the case that $3$ is either a gradient or opaque direction. We recall that the identification of equilibrium cells is preserved by $\cF_3$. Since $\cH_3(\nu,\theta,h) \subset \cH_1(\nu,\theta,h)$, one obtains the following corollary of Proposition~\ref{prop:equilibrium-existence}. 
\begin{cor}
    Consider {\bf H3} and let $\recG$ be a rectangular geometrization of $\cX_b$. If $\xi \in \cX$ is an equilibrium cell, then the ramp system \eqref{eq:rampH1} contains an unique equilibrium in $\bg(\blup(\xi))$.
\end{cor}
\section{Geometrization for two dimensional Janus complex}
\label{sec:R3N2}
Throughout this section, we fix $N=2$ and consider the hypothesis {\bf H3}.
\begin{defn}
    \label{defn:F3N2-manifold}
    Let $\xi \in \cX$ be an equilibrium cell and let $x^* \in \bg(\blup(\xi))$ be the unique equilibrium point. Let $\rmap\xi=(1 \ 2)$ be its regulation map and let $\delta_{\rmap\xi}$ be the signature of $\rmap\xi$. Let $\varepsilon>0$.
    \\ 
    If $\delta_{\rmap\xi}=-1$, then we define $\cM_\varepsilon^i(\xi)$, the \emph{internal transversal manifold at $\xi$}, as follows.
    \begin{equation}
        \label{eq:F3N2-neg}
        \cM_\varepsilon^i(\xi)\coloneqq L^{-1}(\varepsilon)
    \end{equation}
    where $L(x)=\left(P^{-1}(x-x^*)\right)^T \left(P^{-1}(x-x^*)\right)$, $J$ is the linearization of \eqref{eq:rampH3} at $x^*$, $\alpha\pm i\beta$ the eigenvalues of $J$, and $P^-1 J P = D$ for 
    \[
    D = \begin{bmatrix}
        -\alpha & \beta \\ 
        -\beta & -\alpha
    \end{bmatrix}.
    \]
    If $\delta_{\rmap\xi}=1$, then we define $\cM_\varepsilon^e(\xi)$, the \emph{external transversal manifold at $\xi$}, as follows.
    \begin{equation}
        \label{eq:F3N2-pos}
        \cM_\varepsilon^e(\xi) \coloneqq \bigcup_{\substack{\xi \prec \xi'\\\dim(\xi')=1}} \cM_{0,\varepsilon}(\xi,\xi'),
    \end{equation}
    where $\cM_{\delta,\varepsilon}(\xi,\xi')$ is the GO-manifold (see Definition~\ref{defn:GO-manifold}). 
\end{defn}
\begin{thm}
    \label{thm:F3N2-trans}
    Let $N=2$ and consider {\bf H3}. Let $\xi \in \cX$ be an equilibrium cell with $\delta_{\rmap\xi}=-1$ (resp. $\delta_{\rmap\xi}=1$). Then, there exists $\varepsilon_0 > 0$ such that solutions of \eqref{eq:rampH3} are transverse to $\cM_\varepsilon^i(\xi)$ (resp. $\cM_\varepsilon^e(\xi)$) for any $0 < \varepsilon < \varepsilon_0$. 
\end{thm}
\begin{proof}
    The case for $\delta_{\rmap\xi}=1$ follows directly from Proposition~\ref{prop:transversality-manifold}. 
    
    Let $\delta_{\rmap\xi}=-1$. Note that since $\rmap\xi(1)=2$ and $\rmap\xi(2)=1$, the associated ramp system within $\bg(\blup(\xi))$ is given by
    \begin{align*}
        \dot{x}_1 & = -\gamma_1 x_1 + f_1(x_2) \\
        \dot{x}_2 & = -\gamma_2 x_2 + f_2(x_1)
    \end{align*}
    where
    \begin{align*}
        f_1\left( x_2 \right) & = E_1 \left( (\mdpt_1(\xi),x_2),\nu,\theta,h \right), \\ 
        f_2\left( x_1 \right) & = E_2 \left( (x_1,\mdpt_2(\xi)),\nu,\theta,h \right).
    \end{align*}
    The Jacobian at the equilibrium point $x^* \in \bg(\blup(\xi))$ is given by 
    \[
    J = \begin{bmatrix} 
    -\gamma_1 & d_{1,2} \\
    d_{2,1} & -\gamma_2,
    \end{bmatrix} 
    \]
    where $d_{i,j} = \partial f_i / \partial x_j$. The trace of $J$ is negative 
    \[
    \mathrm{tr}(J) = - (\gamma_1 +\gamma_2),
    \]
    and the determinant depends on $\delta_{(1 2)}$, namely, \[
    \det(J)=\gamma_1\gamma_2 -d_{1,2}d_{2,1}=\gamma_1\gamma_2-\delta_{(1 2)}|d_{1,2}||d_{2,1}|.
    \]
    Since $\delta_{\rmap\xi}=\delta_{(1 2)} = -1$, $\det(J)>0$ and the eigenvalues of $J$ are $-\alpha\pm\beta$ with $\alpha>0$ and $\beta \neq 0$. Let $P$ be an invertible matrix such that $P^{-1}JP=D$ for 
    \[
    D = \begin{bmatrix}
        -\alpha & \beta \\ 
        -\beta & -\alpha
    \end{bmatrix}. 
    \]
    We claim that $L(x) = \left(P^{-1}(x-x^*)\right)^T \left(P^{-1}(x-x^*)\right)$
    is a Lyapunov function. If that is true, then for any $\varepsilon_0 >0$ sufficiently small such that $L^{-1}(\varepsilon_0)$ is contained in the interior of $\bg(\blup(\xi))$, the result follows. To prove the claim, consider the change of variables $y=P^{-1}(x-x^*)$ so that $\dot{y}=Dy$. It follows that
    \begin{align*}
        \dot{L}(y) & = \dot{y}^T y + y^T\dot{y} \\ 
        & = (Dy)^T y + y^T(Dy) \\ 
        & = y^T D^T y + y^T D y \\
        & = y^T(D^T+D)y \\
        & = y^T( -2\alpha I) y \\
        & = -2\alpha y^T y
    \end{align*}
    which is zero at $y=0$, i.e., when $x=x^*$ and strictly negative for any other $x \in \bg(\blup(\xi))$, $x\neq x^*$. 
\end{proof}

We illustrate in Figure~\ref{fig:2D-Janus-embedding} the role of $\cM_\varepsilon^i(\xi)$ and $\cM_\varepsilon^e(\xi)$ in the geometric realization $\bG_3$ that verifies Theorem~\ref{thm:R3ABlattice}. 
\begin{figure}
    \centering
    \begin{subfigure}[t]{0.45\textwidth}
        \begin{tikzpicture}[scale=0.25]
            \fill[cyan] (0,0) rectangle (24,24);
            \fill[yellow] (11,11) rectangle (13,13);
            \draw[step=1cm,black] (0,0) grid (24,24);
            \draw[step=8cm,red] (0,0) grid (24,24);
            \end{tikzpicture}
        \subcaption{$\Janus$ with $\pi_3$-grading for $\delta_{\rmap\xi}=-1$.}
    \end{subfigure}
    \hfill 
    \begin{subfigure}[t]{0.45\textwidth}
    \begin{tikzpicture}[scale=0.25]
    \fill[cyan] (0,0) rectangle (24,24);
    \draw[step=8cm,red] (0,0) grid (24,24);
    \fill[yellow] (12,12) ellipse [x radius=3.2, y radius=2.4];
    \draw (12,12) ellipse [x radius=3.2, y radius=2.4];
    \end{tikzpicture}
    \subcaption{Geometrization of $\Janus$ for $\delta_{\rmap\xi}=-1$.}
    \end{subfigure}
    
    \vspace{0.5cm}
    
    \begin{subfigure}[t]{0.45\textwidth}
        \begin{tikzpicture}[scale=2]
            \fill[cyan] (1,0) rectangle (2,1);
            \fill[purple] (0,1) rectangle (1,2); 
            \fill[yellow] (1,1) rectangle (2,2);
            \fill[red] (2,1) rectangle (3,2);
            \fill[blue] (1,2) rectangle (2,3);
            \fill[lightgray] (1,2) rectangle (0,3); 
            \fill[gray] (2,1) rectangle (3,0);
            \fill[red!20!white] (0,0) rectangle (1,1);
            \fill[red!10!white!90!blue] (2,2) rectangle (3,3);
            \fill[yellow] (2,1) -- ++(-0.250,0)
                -- ++(0,-0.125) -- ++(0.125,0)
                -- ++(0,-0.125) -- ++(0.125,0) 
                -- cycle;
            \fill[yellow] (1,2) -- ++(0,-0.250)
                -- ++(-0.125,0) -- ++(0,0.125)
                -- ++(-0.125,0) -- ++(0,0.125)
                -- cycle;
            \fill[yellow] (2,1) -- ++(0,0.250)
                -- ++(0.125,0) -- ++(0,-0.125)
                -- ++(0.125,0) -- ++(0,-0.125)
                -- cycle;
            \fill[yellow] (1,2) -- ++(0.250,0)
                -- ++(0,0.125) -- ++(-0.125,0)
                -- ++(0,0.125) -- ++(-0.125,0) 
                -- cycle;
            \draw[step=0.125cm,black] (0,0) grid (3,3);
            \draw[step=1cm,red] (0,0) grid (3,3);
        \end{tikzpicture}
    \subcaption{$\Janus$ with $\pi_3$-grading for $\delta_{\rmap\xi}=1$.}
    \end{subfigure}
    \hfill 
    \begin{subfigure}[t]{0.45\textwidth}
        \begin{tikzpicture}[scale=2]
            
            \fill[cyan] (1,0) rectangle (2,1); 
            \fill[purple] (0,1) rectangle (1,2); 
            \fill[yellow] (1,1) rectangle (2,2); 
            \fill[red] (2,1) rectangle (3,2); 
            \fill[blue] (1,2) rectangle (2,3); 
            \fill[lightgray] (1,2) rectangle (0,3);
            \fill[gray] (2,1) rectangle (3,0);
            \fill[red!20!white] (0,0) rectangle (1,1);
            \fill[red!10!white!90!blue] (2,2) rectangle (3,3);
            \def\SouthCurve {(1.625,1) .. controls (1.9,0.9) .. (2,0.625)};
            \def\WestCurve {(1,1.625) .. controls (0.9,1.9) .. (0.625,2)};
            \def\EastCurve {(2,1.375) .. controls (2.1,1.1) .. (2.375,1)};
            \def\NorthCurve {(1.375,2) .. controls (1.1,2.1) .. (1,2.375)};
            \fill[yellow] (1,2) -- (1,1.625) -- \WestCurve -- cycle;
            \fill[yellow] (2,1) -- (1.625,1) -- \SouthCurve -- cycle; 
            \fill[yellow] (2,1) -- (2,1.375) -- \EastCurve -- cycle; 
            \fill[yellow] (1,2) -- (1.375,2) -- \NorthCurve -- cycle; 
            \draw \WestCurve;
            \draw \EastCurve;
            \draw \SouthCurve;
            \draw \NorthCurve;
            \draw[step=1cm,red] (0,0) grid (3,3);
        \end{tikzpicture}
        \subcaption{Geometrization of $\Janus$ for $\delta_{\rmap\xi}=1$.}
    \end{subfigure}
    \caption{On the left, the Janus complex $\Janus$ with $\pi_3$-grading. On the right, a geometric realization of $\Janus$ that verifies transversality of the vector field along the boundary.}
    \label{fig:2D-Janus-embedding}
\end{figure}
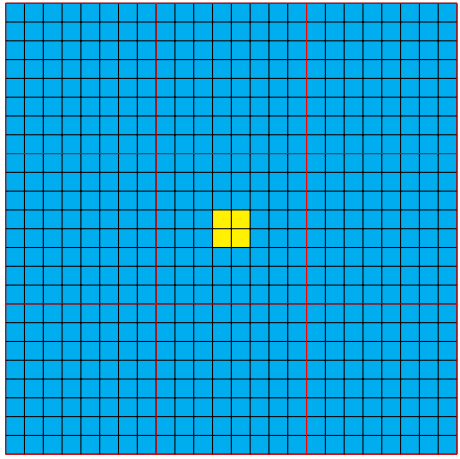
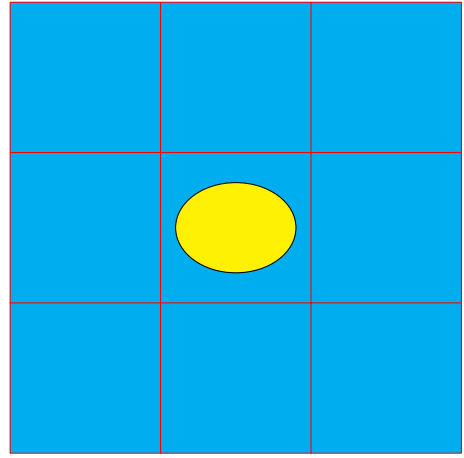
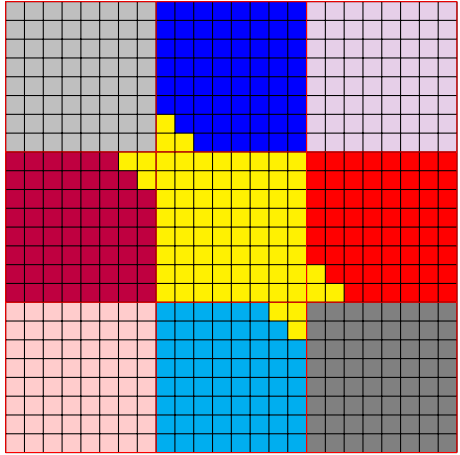
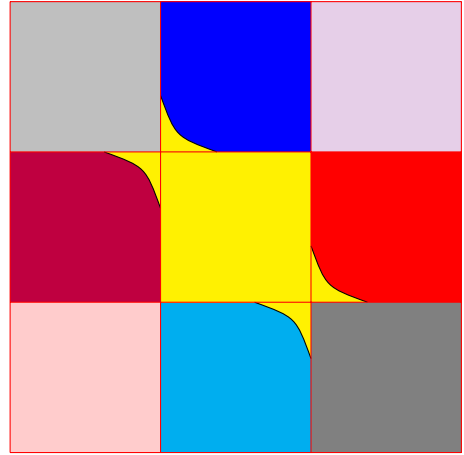

\section{Geometrization of three dimensional Janus Complex for $\cF_3$}
\label{sec:F3N3}

Throughout this section, we fix $N=3$ and consider the hypothesis {\bf H3}. 
\begin{defn}
    Let $\xi \in \cX$ be an equilibrium cell with $\rmap\xi=(n_1 \ n_2 \ n_3 )$.  Let $x^* \in \bg(\blup(\xi))$ denote the unique equilibrium point. Let $\delta,\varepsilon > 0$. We define $\cM_\varepsilon^i(\xi)$, the \emph{internal transversal manifold at $\xi$}, as follows. 

    If $\delta_{\rmap\xi}=1$, then
    \[
    \cM_\varepsilon^i(\xi) \coloneqq L_s^{-1}(\varepsilon) \cap \bg(\blup(\xi)),
    \]
    where $L_s(x) = y_1^2+y_2^2$ for $y= P^{-1}(x-x^*)$, $J$ is the linearization of $\eqref{eq:rampH3}$ at $x^*$, $-\alpha\pm \beta$ and $\lambda>0$ are the eigenvalues of $J$, and $P^{-1}JP=D$ for 
    \begin{equation}
    \label{eq:J-neg}
    D = \begin{bmatrix}
        -\alpha & \beta & 0 \\ 
        -\beta & -\alpha & 0 \\ 
        0 & 0 & \lambda
    \end{bmatrix}.
    \end{equation} 
    If $\delta_{\rmap\xi}=-1$, then
    \[
        \cM_\varepsilon^i(\xi) \coloneqq L_u^{-1}(\varepsilon) \cap \bg(\blup(\xi)),
    \]
    where $L_s(x) = y_1^2+y_2^2$ for $y= P^{-1}(x-x^*)$, $J$ is the linearization of $\eqref{eq:rampH3}$ at $x^*$, $\alpha\pm \beta$ and $\lambda<0$ are the eigenvalues of $J$, and $P^{-1}JP=D$ for 
    \[
    D = \begin{bmatrix}
        \alpha & \beta & 0 \\ 
        -\beta & \alpha & 0 \\ 
        0 & 0 & \lambda
    \end{bmatrix}.
    \]
    To define $\cM_\varepsilon^e(\xi)$, the \emph{external transversal manifold at $\xi$}, we first define 
    \[
        Cone(\xi) \coloneqq \setdef{x \in \bg(\blup(\xi))}{\delta_{\rmap\xi}(-\gamma_n x_n + E_n(x;\nu,\theta,h))(x_n-x^*_n) \geq 0, \forall n=1,2,3.},
    \]
    and let 
    \[
        BC(\xi) = Cone(\xi) \bigcap \left( \bigcup_{n=1}^3 \bigcup_{\substack{\xi\prec\xi'\\\dim(\xi')=1}} \Nul_n(\xi,\xi') \right). 
    \]
    If $\delta_{\rmap\xi} = -1$, then 
    let $\mu \in \Stable(\xi)$ a stable cell (see \eqref{eq:stable_cells}).
    
    Denote by $\xi_{n_1},\xi_{n_2},\xi_{n_3}$ the cells of dimension $1$ with $J_e(\xi_{n_i})=\setof{n_i}$ such that $\xi \prec \xi_{n_i} \prec \mu$ and let $\xi_{n_i}'$ be the corresponding $2$-cells such that $(\xi_{n_i},\xi_{n_i}')$ has a GO-pair. For each $i$, let $\psi_{n_i} : [0,1] \to \bg(\blup(\xi)) \cap \bg(\blup(\xi_{n_i}))$ be a parametrization of $BC(\xi)\cap\bg(\blup(\xi))\cap\bg(\blup(\xi_{n_i'}))$ with 
    \begin{align*}
         \psi_{n_i}(0) & \in \bg(\blup(\xi_{n_i})) \cap \bg(\blup(\xi_{n_i}')), \\
         \psi_{n_i}(1/2) & \in \Nul_{n_i}(\xi,\xi_{n_i})\cap\Nul_{\rmap\xi^{-1}(n_{i})}(\xi,\xi_{n_i}), \\ 
         \psi_{n_i}(1) & \in \bg(\blup(\xi))\cap\bg(\blup(\xi_{\rmap\xi(n_{i})})).
    \end{align*}
    Define the external manifold at $\xi$ relative to $\xi_{n_i}$ by
    \[
        \cM_\varepsilon^e(\xi,\xi_{n_i}) = \bigcup_{\alpha \in [0,1]} \varphi_{\varepsilon,n_i}([0,T_\xi(\psi_{n_i}(\alpha)),\psi_{n_i}(\alpha)]
    \]
    where $\varphi_{\varepsilon,n_i}$ are solutions of 
    \[
        \dot{x} = (1-\alpha) F_{\varepsilon}(x;\xi_{n_i},\xi_{n_i}') + \alpha F_{\varepsilon}(x;\xi_{{\rmap\xi(n_{i})}},\xi_{{\rmap\xi(n_{i})}}')
    \]
    and the exit time $T$ is given by 
    \[
        T_\xi(x) = \inf\setdef{t > 0}{\varphi_{\varepsilon,n_i}(t,x) \in \bg(\blup(\xi))}. 
    \]
    Finally, define the external manifold at $\xi$ relative to $\mu$ by
    \[
        \cM_\varepsilon^e(\xi,\mu) = \bigcup_{i=1}^3 \cM_\varepsilon^e(\xi,\xi_{n_i})
    \]
    and the external transversal manifold at $\xi$ is given by 
    \[
        \cM_\varepsilon^e(\xi) = \bigcup_{\mu \in \Stable(\xi)} \cM_\varepsilon^e(\xi,\mu).
    \]
    If $\delta_{\rmap\xi} = 1$, then 
    let $\mu \in \cU(\xi)$ a unstable cell (see Definition~\ref{defn:unstable_cells}).
    
    Denote by $\xi_{n_1},\xi_{n_2},\xi_{n_3}$ the cells of dimension $1$ with $J_e(\xi_{n_i})=\setof{n_i}$ such that $\xi \prec \xi_{n_i} \prec \mu$ and let $\xi_{n_i}'$ be the corresponding $2$-cells such that $(\xi_{n_i},\xi_{n_i}')$ has a GO-pair. For each $i$, let $\psi_{n_i} : [0,1] \to \bg(\blup(\xi)) \cap \bg(\blup(\xi_{n_i}))$ be a parametrization of $BC(\xi)\cap\bg(\blup(\xi))\cap\bg(\blup(\xi_{n_i'}))$ with 
    \begin{align*}
         \psi_{n_i}(0) & \in \bg(\blup(\xi_{n_i})) \cap \bg(\blup(\xi_{n_i}')), \\
         \psi_{n_i}(1/2) & \in \Nul_{n_i}(\xi,\xi_{n_i})\cap\Nul_{\rmap\xi^{-1}(n_{i})}(\xi,\xi_{n_i}), \\ 
         \psi_{n_i}(1) & \in \bg(\blup(\xi))\cap\bg(\blup(\xi_{\rmap\xi(n_{i})})).
    \end{align*}
    Define the external manifold at $\xi$ relative to $\xi_{n_i}$ by
    \[
        \cM_\varepsilon^e(\xi,\xi_{n_i}) = \bigcup_{\alpha \in [0,1]} \varphi_{\varepsilon,n_i}([0,T_\mu(\psi_{n_i}(\alpha)),\psi_{n_i}(\alpha)]
    \]
    where $\varphi_{\varepsilon,n_i}$ are solutions of 
    \[
        \dot{x} = (1-\alpha) F_{\varepsilon}(x;\xi_{n_i},\xi_{n_i}') + \alpha F_{\varepsilon}(x;\xi_{{\rmap\xi(n_{i})}},\xi_{{\rmap\xi(n_{i})}}')
    \]
    and the exit time $T$ is given by 
    \[
        T_\mu(x) = \inf\setdef{t > 0}{\varphi_{\varepsilon,n_i}(t,x) \in \bg(\blup(\mu))}. 
    \]
    Finally, define the external manifold at $\xi$ relative to $\mu$ by
    \[
        \cM_\varepsilon^e(\xi,\mu) = \bigcup_{i=1}^3 \cM_\varepsilon^e(\xi,\xi_{n_i})
    \]
    and the external transversal manifold at $\xi$ is given by 
    \[
        \cM_\varepsilon^e(\xi) = \bigcup_{\mu \in \cU(\xi)} \cM_\varepsilon^e(\xi,\mu).
    \]
\end{defn}
We observe that transversality at the internal transversal manifolds at $\xi$ follows from the proof of Theorem~\ref{thm:F3N2-trans}. 
\begin{cor}
    \label{cor:F3N3-internal-3cycle}
    Let $\xi \in \cX$ be an equilibrium cell with $\rmap\xi=(n_1 \ n_2 \ n_3)$. Then there exists $\varepsilon_0 > 0$ such that solutions of \ref{eq:rampH3} are transverse to $\cM_\varepsilon^i(\xi)$ for any $0<\varepsilon<\varepsilon_0$. 
\end{cor}
\begin{proof}
    First, we observe that since $h \in \cH_3(\gamma,\nu,\theta)$, if $\delta_{\rmap\xi}=-1$, then the eigenvalues of the linearization at the equilibrium $x^* \in \bg(\blup(\xi))$ are $\alpha\pm i\beta$ and $\lambda < 0$ with $\alpha>0,\beta\neq0$.

    Similarly, $h \in \cH_3(\gamma,\nu,\theta)$ implies that $-\alpha\pm i\beta$ and $\lambda >0$ with $\alpha>0$, $\beta \neq 0$ when $\delta_{\rmap\xi}=1$.

    Therefore, $L_u$ and $L_s$ are well-defined. Thus, it is enough to take $\varepsilon_0>0$ such that $L_s^{-1}(\varepsilon_0)$ and $L_u^{-1}(\varepsilon_0)$ are contained in $\bg(\blup(\xi))$. Transversality follows the linearization in $\bg(\blup(\xi))$. 
\end{proof}
We now prove the transversality result for the external manifolds.
\begin{thm}
    \label{thm:F3N3-external-3cycle}
     Let $\xi \in \cX$ be an equilibrium cell with $\rmap\xi=(n_1 \ n_2 \ n_3)$. Then there exists $\varepsilon_0 > 0$ such that solutions of \ref{eq:rampH3} are transverse to $\cM_\varepsilon^e(\xi)$ for any $0<\varepsilon<\varepsilon_0$. 
\end{thm}
\begin{proof}
    First, we observe that since $h \in \cH_2(\gamma,\nu,\theta)$, Proposition~\ref{prop:FlowFromExternalTangency} implies that any solution with initial condition $x^0 \in BC(\xi)$ is bounded by the hyperplanes $x_n = \mdpt_n(\xi')$ for $\xi \prec \xi'$, $\dim(\xi')=1$, so $\cM_\varepsilon^e(\xi)$ is bounded the same hyperplanes. 

    Assume that $\delta_{\rmap\xi}=1$. Observe that by Proposition~\ref{cor:FlowFromExternalTangency-Perturbation-2}, the manifold of parametrized solutions between GO-manifolds is transverse to the vector field for sufficiently small $\varepsilon_0>0$. 
    
    Assume that $\delta_{\rmap\xi}=-1$. Let $x \in \cM_\varepsilon^e(\xi)$. Let $x^*\in\bg(\blup(\xi))$ be the unique equilibrium in $\bg(\blup(\xi))$ and consider the change of variables $y = P^{-1}(x^*-x)$. If $x \in \bg(\blup(\xi))$, then 
    \[
        \langle F(x) , F(x)-\varepsilon(x-x^*) \rangle =  \|F(x)\|^2 - \varepsilon_0(x-x^*)F(x) > 0,
    \]
    for sufficiently small $\varepsilon_0>0$.
\end{proof}

\begin{defn}
    Let $\xi \in \cX$ be a pseudo-opaque cell with $\rmap\xi=(n_1 \ n_2 )$ or $(n_1 \ n_2 ) (n_3)$.  Fix $x_{n_3} \in I_{n_3}(\xi)$ and let $x^*(x_{n_3})=(x_{n_1}^*,x_{n_2}^*,x_{n_3}) \in \bg(\blup(\xi))$ be such that $(x_1^*,x_2^*) \in I_{n_1}(\xi) \times I_{n_2}(\xi)$ is the unique equilibrium point of the restricted system
    \begin{align*}
        \dot{x}_{n_1} & = -\gamma_{n_1}x_{n_1} + E_{n_1}(x;\nu,\theta,h) \\ 
        \dot{x}_{n_2} & = -\gamma_{n_2}x_{n_2} + E_{n_2}(x;\nu,\theta,h) 
    \end{align*}
    Let $\varepsilon > 0$. We define the \emph{transversal manifold at $\xi$} by
    \[
    \cM_\varepsilon(\xi) \coloneqq \bigcup_{x_{n_3} \in I_{n_3}(\xi)} \cM_{\varepsilon}(\proj{(n_1 n_2)}(\xi),x_{n_3})\times \setof{x_{n_3}}.
    \]
    where $\cM_{\varepsilon}(\proj{(n_1 n_2),x_{n_3}}(\xi))$ is the internal/external transversal manifold at $\proj{(n_1 n_2)}(\xi)$ at a fixed $x_{n_3} \in I_{n_3}(\xi)$.  
\end{defn}
\begin{cor}
    \label{prop:F3-projmanifold}
    Let $\xi \in \cX$ be a pseudo-opaque cell with $\rmap\xi=(n_1 \ n_2 )$ or $(n_1 \ n_2 ) (n_3)$. Then there exists $\varepsilon_0>0$ such that solutions of \eqref{eq:rampH3} are transverse to $\cM_\varepsilon(\xi)$ for $0<\varepsilon<\varepsilon_0$. 
\end{cor}
\begin{proof}
    Let $x \in \cM_\varepsilon(\xi)$. Note that solutions of \eqref{eq:rampH3} are transverse to $x_{n_3} = \theta_{n_3}\pm h_{*,n_3,j}$ by Theorem~\ref{prop:transversality-for-entrance-exit-faces}. Otherwise, there exists $\varepsilon_0>0$ by Theorem~\ref{thm:F3N2-trans} such that transversality follows from by the projection into the $n_1,n_2$ coordinates. 
\end{proof}

\section{Geometrization $\bG_3$}

We remind the reader that the primary goal of this chapter is to prove Theorem~\ref{thm:R3ABlattice} for $N=2$ and $N=3$. For that purpose, we construct a geometrization $\bG_3$ of the Janus complex $\Janus$ that is aligned with the vector field for each $N$. The desired geometrization $\bG_3$ is obtained from $\recG_J$ by changing the restrictions in Definition~\ref{defn:rectGeoXr} that $\bg_\zeta$ must satisfy when $\zeta \in \Janus$.

First, we choose an appropriate rectangular geometrization. Then we proceed to modify it in such a way to obtain the manifolds defined in Sections~\ref{sec:R3N2}. 
\begin{defn}
    \label{defn:compatible-recG-F3N2}
    Let $N=2$. We say that a rectangular geometrization $\recG_J$ of $\Janus$ is compatible with $\cF_3 \colon \cX \mvmap \cX$ if it satisfies the GO-conditions and at each equilibrium cell $\xi = [\bv,\bzero] \in \cX$ the following is satisfied. 
    \begin{enumerate}
        \item If $\delta_{\rmap\xi} = -1$, then
        \( \displaystyle 
            \bg(\bbdy(\mathrm{Cube}(16\bv-{\bf 3}, 16\bv+{\bf 5}))) = \cM_\varepsilon^i(\xi). 
        \)
        \item If $\delta_{\rmap\xi} = 1$, then 
        \( \displaystyle 
            \bg\left( \bbdy(\mTile(\xi,\rmap\xi))\setminus\bigcup_{\xi'\in\cT^{(1)}(\xi)}\bbdy(\bar{S}(\xi'))\right)= \cM_\varepsilon^e(\xi).
        \)
    \end{enumerate}
    We denote it by $\bG_3$. 
\end{defn}

\begin{theorem}
    \label{thm:F3N2-setup-ABlattice}
    Consider the standing hypothesis, fix $N=2$ and let $\bG_3$ be a rectangular geometrization that is compatible with $\cF_3$. Then $\bG_3$ is aligned with any $\cN \in \sN(\cF_3)$ such that $\zeta \in \bbdy(\cN)^{(N-1)} \cap \bar{S}(\xi)$ for some equilibrium cell $\xi \in \cX$.
\end{theorem}

\begin{proof}
    Let $\xi=[\bv,\bzero] \in \cX$ be an equilibrium cell such that $\zeta \in \bbdy(\cN)^{(N-1)} \cap \bar{S}(\xi)$. If $\delta_{\rmap\xi}(\xi)=-1$, then 
    \[
        \zeta \in \mathrm{Cube}(16\bv-{\bf 3},16\bv+{\bf 5}), 
    \]
    so $\bg(\zeta) \subset \cM_\varepsilon^i(\xi)$, to which, by Proposition~\ref{thm:F3N2-trans}, the flow is transverse inwards. Since $\pi_3(\xi) < \pi_3(\xi')$ for any $\xi \preceq \xi'$, it follows that $\langle F(x) , z_{\zeta}(x) \rangle > 0$. 

    If $\delta_{\rmap\xi}=1$ and $\bg(\zeta)$ is contained in a hyperplane $x_{n}=\theta_{k_n,n,j_n}\pm h_{k_n,n,j_n}$, then Theorem~\ref{thm:R1ABlattice} yields the result. Otherwise, $\zeta \in \bbdy(\mTile(\xi,\rmap\xi))$ and Theorem~\ref{thm:R2ABlattice} yields the result. 
\end{proof}

Note that when $N=2$, Theorem~\ref{thm:R3ABlattice} is a corollary of Theorem~\ref{thm:F3N2-setup-ABlattice} by combining the restrictions of $\bG_2$ and $\bG_3$. 

We now turn our attention to the case when $N=3$. 
\begin{defn}
\label{defn:F3N3-Geo}
Let $N=3$. 
We say that a rectangular geometrization $\bG_3$ of $\Janus$ is compatible with $\cF_3 \colon \cX \mvmap \cX$ if it satisfies the following conditions. 
    \begin{enumerate}
        \item $\bG_3$ satisfies the GO-conditions (see Definition~\ref{defn:R2-Geometrization}),
        \item if $\xi$ is an equilibrium cell with $\rmap\xi=(n_1 \ n_2 \ n_3)$, then
        \begin{align*}
            \bg(\bbdy(\FibTri(\xi))\cap\bar{S}(\xi)) & = \cM_\varepsilon^i(\xi) \\ 
            \bg(\bbdy(\FibTri(\xi)\cap \bar{S}(\cT^{(1)}(\xi)))) & = \cM_\varepsilon^e(\xi).
        \end{align*}
        \item if $\xi$ is an pseudo-opaque whose regulation map contains a 2-cycle $(n_1 \ n_2)$, then $\proj{(n_1 n_2)}(\bG_3)$ is compatible with $\cF_2 \colon \cX_{(n_1 n_2)}\mvmap \cX_{(n_1 n_2)}$. 
    \end{enumerate}
\end{defn}

\begin{proof}[Proof of Theorem~\ref{thm:R3ABlattice} for $N=3$]
    Fix $N=3$ and let $\bG_3$ be a rectangular geometrization of $\Janus$ that is compatible with $\cF_3$.  
    
    We verify the alignment along the boundary of $\cN \in \sN(\cF_3)$ by analyzing each case. 
    
    Obviously, if $\zeta \in \bbdy(\cN)^{(N-1)}$ and $\zeta \in \bbdy(\FibTri(\xi))$ when $\xi$ is a pseudo-opaque cell with a $2$-cycle in the decomposition of $\rmap\xi$, then either $J_e(\zeta)=\setof{n}$ is part of the cycle decomposition of $\rmap\xi$ or not. If $n \in G(\xi)$, then $\bg(\zeta)$ is contained in the hyperplanes defined by $x_n = \theta_{*,n,*}\pm h_{*,n,*}$ and the alignment follows by Theorem~\ref{thm:R1ABlattice}. Otherwise, $n \in S_{\rmap\xi}$ and the result follows from Theorem~\ref{thm:R2ABlattice} by the compatibility with $\cF_2$.  

    If $\zeta \in \bar{S}(\xi)$, then $\bg(\zeta) \subset L_s^{-1}(\varepsilon)$ or $\bg(\zeta) \subset L_u^{-1}(\varepsilon)$. In either case, the flow is transverse inwards to $\bg(\blup(\cN))$ by construction, and $\langle F(x), z_{\zeta}(x) \rangle > 0$. 

    If $\zeta \notin \bar{S}(\xi)$, then either $\bg(\zeta)$ is contained in a hyperplane $x_{n} = \theta_{m_{k_n},n,j_{k_n}}\pm h_{m_{k_n},n,j_{k_n}}$ or $\bg_\zeta(\Int(B^{N-1})) \subset \Int(\bg(\xi'))$. In the former case, Theorem~\ref{thm:R1ABlattice} yields the alignment. In the latter, note that Theorem~\ref{thm:F3N3-external-3cycle} implies that the flow is transverse inwards-- that is, $\langle F(x) , z_{\zeta}(x) \rangle > 0$.  
\end{proof}

\part{Applications and Prospectives}
\label{part:IV}
\chapter{Examples}
\label{sec:examples}

The constructions and results of Parts~\ref{part:II} and \ref{part:III} are the focus of this manuscript.
However, as discussed in Part~\ref{part:I} we developed these tools with the goal of being able to efficiently understand the global dynamics of systems of ODEs.
In this chapter we return to the discussion of the applicability of our techniques by systematically describing how information about continuous dynamics can be extracted from the order theoretic information codified by Morse graphs and the homological information codified by Conley complexes.

We begin in Section~\ref{sec:saddlesaddlebif} by considering two examples where each node in the Morse graph has a nontrivial Conley index.
As a consequence the Morse decompositions for the associate ramp systems are Morse representations, i.e., each node is associated with a nontrivial recurrent set.
This allows us to easily identify the existence of nontrivial global nonrecurrent dynamics.
In both examples the Conley complex is not unique, which provides us with the  opportunity to  explore the existence of potential global bifurcations.

In Section~\ref{sec:3d_example_periodic_orbit} we turn our attention to the identification of nontrivial recurrent dynamics.
In particular, we demonstrate how the multivalued map $\cF$, the  Conley index information, and the analytic estimates of Part~\ref{part:III} can be used to identify stable periodic orbits and discuss the potential for using these ideas to identify more general examples of recurrent dynamics.
To emphasize the need to strengthen our techniques we include an example for which a Morse set has the Conley index of a fixed point, but for which we conjecture there that the associated invariant set contains a fixed point and a periodic orbit.

As is shown in Sections~\ref{sec:R1Dynamics} and ~\ref{sec:R2Dynamics} the results obtained using the combinatorial models $\cF_1$ and $\cF_2$ are dimension independent.
However, as is well known, the dynamics on the global attractor of a high dimensional ODE can be low dimensional.
In Section~\ref{sec:semiconjugacy} we describe this lower dimensional dynamics by using the Morse graph and the Conley complex to construct a semi-conjugacy to lower dimensional system.

Our calculations show that it is quite common to identify Morse graphs where some of the nodes have trivial Conley indices.  
As discussed in Section~\ref{sec:CCalgorithm} while a nontrivial Conley index implies the existence of a nontrivial invariant set, the converse is not true, and thus, while the techniques of Part~\ref{part:II} allow us to identify a Morse decomposition for the ramp system we cannot conclude that we have found a Morse representation.
With this in mind, in Section~\ref{sec:notMorseRep} we discuss  why our multivalued maps  that are defined on cubical complexes can produce strongly connected components with trivial Conley indices even if in the ramp system the associated invariant set is empty.
We also argue for caution in dismissing the importance of nodes with trivial index.

In the introduction and at other points in this manuscript we highlight the DSGRN software.
However, as is discussed in Section~\ref{sec:NotDSGRN} our theoretical results extend beyond the current capabilities of DSGRN.
Nevertheless, we demonstrate that the desired computations can be carried out.

Finally, in Section~\ref{sec:beyondRamp} we briefly touch upon the question of extending our results beyond ramp systems.

An implementation of the algorithms described in this paper is available as part of the DSGRN software \cite{DSGRN_Rook_Field}. A Jupyter notebook with code to perform the computations in this monograph, with a description on how to use our software to perform the computations, is available at \cite{Rook_Field_Paper_Repo}. All the computations presented in this monograph were performed on a laptop using a single core.

\section{Identifying Global Dynamics and Global Bifurcations}
\label{sec:saddlesaddlebif}

\begin{ex}
\label{ex:saddlesaddlebif}
Example~\ref{ex:multiple_connection_matrices_ex}  begins with the wall labeling of Figure~\ref{fig:ex2sec6}(A) from which $\cF_3$ is determined.
The D-gradings $\pi\colon \cX\to \SCC(\cF_3)$ and $\pi_b\colon \cX_b\to \SCC(\cF_3)$ can be determined from Figure~\ref{fig:ex2sec6}(B).
The associated Morse graph $\sMG(\cF_3)$ is shown in Figure~\ref{fig:ex2sec6}(C) as are the Conley indices of the nodes.
The associated pair of connection matrices is given by \eqref{eq:pairDelta}.

To link these combinatorial and algebraic topological computations to the dynamics of differential equations, we take a step back and consider the regulatory network $RN$ shown in Figure~\ref{fig:example_2_2d_network}.
The parameter graph obtained by applying DSGRN to this regulatory network has $1600$ nodes. 
To compute the Morse graphs and a Conley complex for all $1600$ nodes of the parameter graph takes about $8$ seconds.

At node $752$ the DSGRN derived wall labeling is that of Figure~\ref{fig:ex2sec6}(A).
The region of parameter space associated with parameter node $R(752)$ is given by the following inequalities,
\begin{equation}
\label{eq:752}
\begin{aligned}
& x_1 : p_0 < \gamma_1 \theta_{2,1} < p_2 < \gamma_1 \theta_{1,1} < p_1, p_3 \\
& x_2 : p_0 < \gamma_2 \theta_{1,2} < \gamma_2 \theta_{2,2} < p_1, p_2, p_3,
\end{aligned}
\end{equation}
where for node $x_1$: $p_0 = \ell_{1,1} + \ell_{1,2}$, $p_1 = u_{1,1} + \ell_{1,2}$, $p_2 = \ell_{1,1} + u_{1,2}$, and $p_3 = u_{1,1} + u_{1,2}$, and for node $x_2$: $p_0 = \ell_{2,1} \ell_{2,2}$, $p_1 = u_{2,1} \ell_{2,2}$, $p_2 = \ell_{2,1} u_{2,2}$, and $p_3 = u_{2,1} u_{2,2}$.

\begin{figure*}[!htb]
\centering
\begin{tikzpicture}
    [main node/.style={circle,fill=white!20,draw},scale=2.5]
    \node[main node] (1) at (0,0.75) {1};
    \node[main node] (2) at (0,0) {2};
        
    \path[thick]
        (1) edge[-|,shorten <= 2pt, shorten >= 2pt, bend left] (2)
        (2) edge[loop left, distance=10pt, shorten >= 2pt, thick, in=300, out=240, ->] (2)
        (2) edge[->,shorten <= 2pt, shorten >= 2pt, bend left] (1)
        (1) edge[loop right, distance=10pt, shorten >= 2pt, thick, in=120, out=60, ->] (1); 
\end{tikzpicture}
\caption{Regulatory network.}
\label{fig:example_2_2d_network}
\end{figure*}
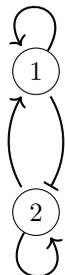

Consider the ramp system
\begin{equation}
\label{eq:ramp_system_SSbif}
\begin{aligned}
\dot{x}_1 & = -\gamma_1 x_1 + r_{1,1}(x_1) + r_{1,2}(x_2) \\
\dot{x}_2 & = -\gamma_2 x_2 + r_{2,1}(x_1)  r_{2,2}(x_2) 
\end{aligned}
\end{equation}
where the $r_{n,m}$ are ramp functions of type $J=1$ (see Definition~\ref{defn:rampfunction}) and the parameters $(\gamma,\nu,\theta)$ are chosen from $R(752)$. See Definition~\ref{defn:DSGRN_ramp} for the relationship between $\ell$, $u$ and $\nu$.
To obtain results concerning the dynamics of \eqref{eq:ramp_system_SSbif} we need to insure that the parameters $(\gamma,\nu,\theta,h)$ are admissible and hence assume that $h\in \cH_3(\gamma,\nu,\theta)$.

We repeat, with more detail, the sequence of arguments used in Example~\ref{ex:saddle_saddle_3D_intro}.
Let $\bG(\gamma,\nu,\theta,h)$ be a geometrization associated with $\cF_3$.
Let $\varphi$ denote the flow associated with \eqref{eq:ramp_system_SSbif} at $(\gamma,\nu,\theta,h)$.
Proposition~\ref{prop:globalAttractor} implies that $\varphi$ possesses a compact global attractor $K \subset [0,\gab_1(\gamma, \nu, \theta, h)] \times [0,\gab_2(\gamma, \nu, \theta, h)]$.
Theorem~\ref{thm:R3ABlattice} guarantees that the Morse graph $\sMG(\cF_3)$ is a Morse decomposition for $\varphi$ restricted to $K$.
The Conley index of each node in the Morse graph is nonzero and thus for every $p\in\sMG(\cF_3)$,
\[
M(p) := \Inv\left(
\bg\left(\left(\cl(\sO(\pi_b^{-1}(p)))\right)\right)
\setminus \bg\left(\left(\cl(\sO(\pi_b^{-1}(p))^<) \right) \right),\varphi \right) \neq \emptyset
\]
Therefore, $\sMG(\cF_3)$ is a Morse representation for $\varphi$ restricted to $K$.

The Conley index of each node is the index of a hyperbolic fixed point.
In particular, nodes labeled $(1,0,0)$, $(0,1,0)$, and $(0,0,1)$ have Conley indices of hyperbolic fixed points with unstable manifolds of dimension $0$, $1$, and $2$, respectively.
By \cite{mccord:89} we know that each invariant set $M(p)$ contains a fixed point.
However, based on this homological information we cannot conclude that $M(p)$ contains a unique fixed point, nor that it contains a fixed point with an unstable manifold of dimension corresponding to the nontrivial Betti number of the Conley index. 
We return to a more detailed discussion concerning fixed points below.

The fact that the Morse graph describes a Morse decomposition, or more specifically for this example a Morse representation, leads to the following result.
Let $x\in [0,\infty)^2$.
Then $\omega(x;\varphi) \subset M(p)$ for some $p\in \setof{0, 1, \ldots, 8}$.
Since the global attractor $K$ is a subset of $[0,\infty)^2$, this results applies to all $x\in K$.
However, in addition, if $x\in K$, then $\alpha(x;\varphi) \subset M(q)$ and $p\leq q$ under the partial order described by the Morse graph.
Furthermore, $p=q$ if and only if $x\in M(p)$.
Thus, the partial order of the Morse graph provides considerable restrictions on the existence of possible connecting orbits and hence on the global structure of the dynamics.

Existence results with respect to connecting orbits are obtained using the connection matrices.
Let $p, q \in \sMG(\cF_3)$ and assume $p<q$.
By \cite[Proposition 5.3]{franzosa:89} if $p$ is an immediate predecessor of $q$ and for all connection matrices the submatrix
\[
\Delta(p,q) \colon CH_*(q) \to CH_*(p)
\]
is nontrivial, then there exists $x\in K$ such that
\[
\alpha(x;\varphi) \subset M(q)\quad\text{and}\quad \omega(x;\varphi) \subset M(p),
\]
i.e., there is a connecting orbit from $M(q)$ to $M(p)$.
For our example, we can conclude the existence of connecting orbits from $M(q)$ to $M(p)$ for the following edges
\begin{equation}
\label{eq:sadtoatt}
4 \to 1,\ 4 \to 2,\ 5 \to 0,\ 5 \to 3,\ 6 \to 2,\ 6 \to 3,\ 7 \to 0
\end{equation}
and
\begin{equation}
\label{eq:sourcetosad}
8 \to 5,\ 8 \to 6,\ 8 \to 7
\end{equation}
since for the respective $q$ and $p$ for both connection matrices  $\Delta(p,q)\neq 0$. 
The nodes $0$, $1$, $2$, and $3$ are minimal nodes of $\sMG(\cF_3)$.
Thus, the existence of the connecting orbits of \eqref{eq:sadtoatt} indicates that $M(4)$ lies on the separatrix between $M(1)$ and $M(2)$, $M(5)$ lies on the separatrix between $M(0)$ and $M(3)$, and $M(6)$ lies on the separatrix between $M(2)$ and $M(3)$. Similarly, we can extend our understanding of these separatrices by noting the existence of the connections from $M(8)$ to $M(5)$ and $M(8)$ to $M(6)$.

As in Example~\ref{ex:saddle_saddle_3D_intro}, 
 we explain the implications of the edge $7 \to 4$ in $\sMG(\cF_3)$ in detail.
Assume each Morse set consists of a unique hyperbolic fixed point. 
The $7 \to 4$ edge suggests the existence of a heteroclinic connection between saddles in $M(7)$ and $M(4)$.
However, such a connecting orbit is not expected to exist at a typical parameter value.
We remind the reader that the results described above are valid for all admissible parameter values for which $(\gamma,\nu,\theta)\in R(752)$.
Thus, a better interpretation of the $7 \to 4$ edge is that it does not exclude the existence of a connecting orbit from $M(7)$ to $M(4)$.
This raises the question: \emph{are there  parameter values $(\gamma,\nu,\theta,h)$ with  $(\gamma,\nu,\theta)\in R(752)$ for which there exists a connecting orbit from $M(7)$ to $M(4)$?}

To answer this question we again turn to the connection matrices \eqref{eq:pairDelta}.
The difference between $\Delta_1$ and $\Delta'_1$ is in the first column that records 
\[
CH_1(7) \to \bigoplus_{p=0}^3 CH_0(p),
\]
and more precisely in the $CH_1(7) \to CH_0(1)\oplus CH_0(2)$ entries.
The connection matrix $\Delta$ suggests a connecting orbit from $M(7)$ to $M(1)$, while the connection matrix $\Delta'$ suggests a connecting orbit from $M(7)$ to $M(2)$.

Let us assume that $M(7)$ is a hyperbolic fixed point, then its Conley index implies that it has a one-dimensional unstable manifold.
We have already established that a connecting orbit exists from $M(7)$ to $M(0)$.
Thus, at most one more connecting orbit can emanate from $M(7)$.
Hence, if there is a connecting orbit from $M(7)$ to $M(1)$, then there is no connecting orbit from $M(7)$ to $M(2)$, and vice versa.
Furthermore, in either case there is no connecting orbit from $M(7)$ to $M(4)$ and thus the ordering of the Morse representation changes to that of Figure~\ref{fig:newMG}.
The connection matrix associated with the Morse representation in Figure~\ref{fig:newMG}(A) is $\Delta$, while the Morse representation in Figure~\ref{fig:newMG}(B) leads to $\Delta'$.

\begin{figure}[!htpb]
\begin{subfigure}{0.47\textwidth}
\centering
\includegraphics[width=\linewidth]{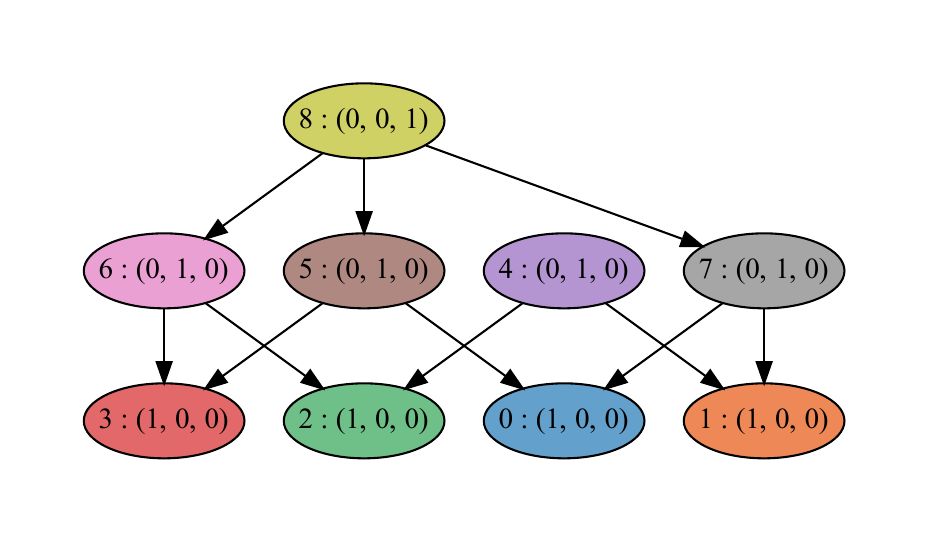}
\caption{Connecting orbit from $M(7)$ to $M(1)$.}
\label{fig:newMG_1}
\end{subfigure}
\begin{subfigure}{0.47\textwidth}
\centering
\includegraphics[width=\linewidth]{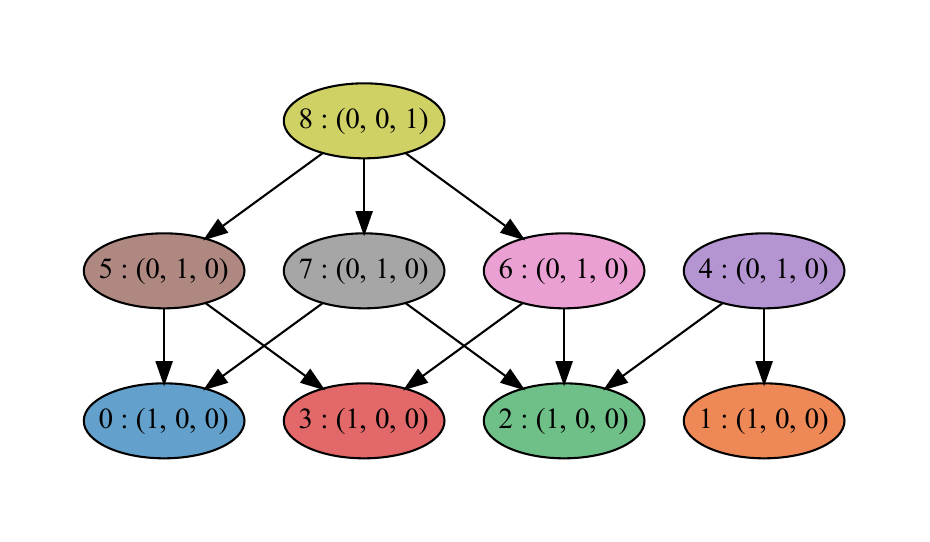}
\caption{Connecting orbit from $M(7)$ to $M(2)$.}
\label{fig:newMG_2}
\end{subfigure}
\caption{(A) Morse graph associated with existence of $M(7)$ to $M(1)$ connecting orbit. (B) Morse graph associated with existence of $M(7)$ to $M(2)$ connecting orbit.}
\label{fig:newMG}
\end{figure}

Though validated numerics \cite{vandenberg:sheombarsing} can rigorously identify connections of this form -- from saddles to attractors -- this is beyond the scope of this manuscript.  Thus we only  remark that at the parameter values
\begin{equation}
\label{eq:params_connection_1}
\begin{aligned}
& \nu_{1, 1, 1} = 0.03, \ \nu_{1, 1, 2} = 2.10, \ \nu_{1, 2, 1} = 0.33, \ \nu_{1, 2, 2} = 0.94 \\
& \nu_{2, 1, 1} = 0.44, \ \nu_{2, 1, 2} = 3.15, \ \nu_{2, 2, 1} = 4.01, \ \nu_{2, 2, 2} = 0.51 \\
& \theta_{1, 1} = 1.68, \ \theta_{1, 2} = 0.75, \ \theta_{2, 1} = 0.32, \ \theta_{2, 2} = 0.74 \\
& h_{1, 1} = h_{1, 2} = h_{2, 1} = h_{2, 2} = 0.1, \ \gamma_1 = \gamma_2 = 1
\end{aligned}
\end{equation}
standard numerical calculations suggest the existence of  a connecting orbit from $M(7)$ to $M(1)$, and at the parameter values
\begin{equation}
\label{eq:params_connection_2}
\begin{aligned}
& \nu_{1, 1, 1} = 0.37, \ \nu_{1, 1, 2} = 3.32, \ \nu_{1, 2, 1} = 0.37, \ \nu_{1, 2, 2} = 1.19 \\
& \nu_{2, 1, 1} = 1.14, \ \nu_{2, 1, 2} = 0.26, \ \nu_{2, 2, 1} = 1.29, \ \nu_{2, 2, 2} = 4.76 \\
& \theta_{1, 1} = 3.37, \ \theta_{1, 2} = 0.87, \ \theta_{2, 1} = 0.87, \ \theta_{2, 2} = 1.02 \\
& h_{1, 1} = h_{1, 2} = h_{2, 1} = h_{2, 2} = 0.1, \ \gamma_1 = \gamma_2 = 1
\end{aligned}
\end{equation}
we observe the existence of a  connecting orbit from $M(7)$ to $M(2)$.

We return to the question of the existence of a connecting orbit from $M(7)$ to $M(4)$. The parameter region $R(752)$ is path connected \cite{cummins:gedeon:harker:mischaikow:mok}. Therefore, there exists a path of admissible parameters from \eqref{eq:params_connection_1} to \eqref{eq:params_connection_2} that lies within $R(752)$. Using continuity (see \cite{reineck}) it can be shown that at some point on this path there must exist a connecting orbit from $M(7)$ to $M(4)$.

The DSGRN parameter region $R(752)$ is explicitly defined by the inequalities of \eqref{eq:752}.
Given $(\gamma,\nu,\theta) \in R(752)$ we have explicit bounds for $h\in \cH_3(\gamma,\nu,\theta)$.
We discuss the importance of this knowledge.

As indicated in the introduction the development of DSGRN was motivated by problems from systems biology.
In this setting the ramp functions serve as heuristic models for nonlinear chemical reactions \cite{POLYNIKIS2009511,MESTL1995291} and  a biochemical perspective suggest that the slopes of the ramps should not be too steep. 
Consider the following parameters that belong to $R(752)$,
\[
\begin{aligned}
& \nu_{1, 1, 1} = 1.21, \ \nu_{1, 1, 2} = 14.95, \ \nu_{1, 2, 1} = 1.47, \ \nu_{1, 2, 2} = 3.73 \\
& \nu_{2, 1, 1} = 16.30, \ \nu_{2, 1, 2} = 6.36, \ \nu_{2, 2, 1} = 6.36, \ \nu_{2, 2, 2} = 19.73 \\
& \theta_{1, 1} = 15.34, \ \theta_{1, 2} = 55.59, \ \theta_{2, 1} = 2.85, \ \theta_{2, 2} = 60.02 \\
& h_{1, 1} = 2, \ h_{1, 2} = 5, \ h_{2, 1} = 2, \ h_{2, 2} = 5, \ \gamma_1 = \gamma_2 = 1.
\end{aligned}
\]
These are admissible parameter values for which the slopes of the ramp function are between $2$ and $4$.
This suggests that the methods of this manuscript can produce meaningful results in the setting of systems biology.
\end{ex}

The discussion of the previous example does not require that the Morse sets be equilibria. This is  an important point since in general it is easier to identify the Conley indices of Morse sets than to determine the exact structure of the Morse set.
However, by Proposition~\ref{prop:equilibrium-existence} for the ramp system \eqref{eq:ramp_system_SSbif} we can conclude that the Morse sets consist of unique hyperbolic equilibria.

Turning to a different direction of applications there is extensive literature on the study of piecewise-smooth systems \cite{piecewisesmooth} and it is worth noting that our methods can be viewed from this perspective.

Observe that under the conditions of Proposition~\ref{prop:smallh} if $h=(\bar{h},\ldots, \bar{h})$ and $\bar{h}\in (0,\tilde{h}_i(\gamma,\nu,\theta))$, then the Morse graph $\sMG(\cF_i)$, the associated Conley indices, and the collection of connection matrices is fixed and computable. 
Observe that in the limit when $h_{n,m}=0$, a ramp function is locally piecewise constant and thus the associated ramp system is a piecewise-smooth system. 
Thus, we can view the analysis of the dynamics of \eqref{eq:ramp_system_SSbif} as a singular bifurcation from a piecewise-smooth system.

\begin{ex}
\label{ex:saddle_saddle_3D_example}
We return to Example~\ref{ex:saddle_saddle_3D_intro} that is presented in the context of the  system of ODEs \eqref{eq:ramp_system_intro_1} at the parameter values indicated in Table~\ref{tab:parameters_ramp_system_intro_2}.
As discussed there, rather than performing computations on the ODEs directly, we represent the system via the regulatory network of Figure~\ref{fig:3d_example_1RN}.
As discussed in Section~\ref{sec:ramp2rook} given the regulatory network and the parameter values we determine the chain complex $\cX$ and the appropriate wall labeling; in this case wall labeling 52,718,681,992 out of 87,280,405,632 potential wall labelings.
Since this is a three dimensional system we can construct the multivalued map $\cF_3 \colon \cX \mvmap \cX$ from which we extract the Morse graph of Figure~\ref{fig:mg_3d_example_intro_1} and an associated connection matrix.
This computation takes $0.19$ seconds.

Theorem~\ref{thm:R3ABlattice} guarantees that the Morse graph  is a Morse decomposition for the flow on the global attractor.
As indicated in Figure~\ref{fig:mg_3d_example_intro_1} each of the Morse nodes has a nontrivial index, and therefore, the Morse graph is a Morse representation.
Furthermore, for each node in the Morse graph the associated recurrent set consists of a single cell and thus by \cite[Theorem 1]{duncan21:ramp} we can conclude that the Morse representation consists of 25 equilibria, i.e., there is no nontrivial recurrent dynamics.
It is hard to imagine that more tradition computational methods can obtain such results as efficiently.

Recall from Example~\ref{ex:saddlesaddlebif} that the $7\to 4$ edge  suggested the possibility of a saddle to saddle connecting orbit.
This example is awash in such edges: $15 \to 9$, $16 \to 9$, $16 \to 10$, $17 \to 9$, $18 \to 9$, $18 \to 10$, $21 \to 16$, and $22 \to 16$. 
As a consequence there are 4,096 admissible connection matrices (the computation of all 4,096 admissible connection matrices took $4.5$ seconds).
This implies that we cannot immediately identify the heteroclinic orbits that occur for the parameter values in Table~\ref{tab:parameters_ramp_system_intro_2}.
However, this can be established by numerically following the one-dimensional unstable manifolds of the equilibria associated with Morse nodes 15 through 18, 21, and 22.
Doing this provides us with the Morse graph of Figure~\ref{fig:mg_3d_example_intro_new}.

\begin{figure*}[!htb]
\centering
\includegraphics[width=1.0\textwidth]{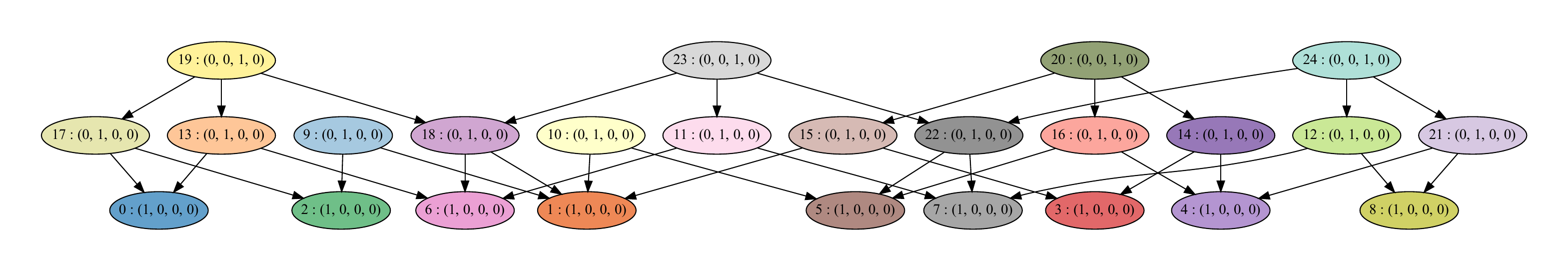}
\caption{Morse graph for the ODE system \eqref{eq:ramp_system_intro_1} with the parameter values in Table~\ref{tab:parameters_ramp_system_intro_2}, obtained by numerically following the one-dimensional unstable manifolds of the equilibria associated with Morse nodes $15$, $16$, $17$, $18$, $21$, and $22$ in Figure~\ref{fig:mg_3d_example_intro_1}.}
\label{fig:mg_3d_example_intro_new}
\end{figure*}

We hasten to add that for many applications, especially those associated with multiscale systems such as biology, there is tremendous parameter uncertainty.
Having identified dynamics at one parameter value, e.g., Table~\ref{tab:parameters_ramp_system_intro_2}, does not guarantee similar dynamics for nearby parameter values.
We include this example precisely because of the multitude of connection matrices suggests that $R(\text{52,718,681,992})$ is a range of parameter values at which the structure of the global dynamics can undergo a variety of significant changes.
It remains an open problem how to effectively identify which of the admissible connection matrices are actually realized for a specific family of ODEs, e.g., ramp systems.
Our example involves 39 parameters and so it is not inconceivable that there are hypersurfaces of saddle-saddle connections $15 \to 9$, $16 \to 9$, $16 \to 10$, $17 \to 9$, $18 \to 9$, $18 \to 10$, $21 \to 16$, and $22 \to 16$. The number of Morse graphs that will be realized depends on the intersections of these hypersurfaces. There are theoretical results that the structure of connecting orbits can be even more complicated \cite{mccord:mischaikow:92}. 
At the moment we do not have efficient mechanisms for completely identifying the set of heteroclinic connections that occur.
However, we know of no other techniques that provide such a clear identification of the potential set of bifurcations.
\end{ex}

\section{Identifying Periodic Orbits}
\label{sec:3d_example_periodic_orbit}

The focus of much of this text has been on identifying Morse graphs and the associated connection matrices.
Obviously recurrent dynamics, which is associated with nodes in the Morse graph, are of equal interest.
In this section we argue that our techniques are capable of identifying nontrivial recurrent dynamics by identifying periodic orbits for a three-dimensional example (see Example~\ref{ex:periodicOrbit3}) and a five-dimensional example  (see Example~\ref{ex:periodicOrbit5}).
We conclude with a discussion about the challenge of finding more complicated recurrent dynamics.

\begin{ex}
\label{ex:periodicOrbit3}
We return to Example~\ref{ex:introPeriodic}.
The Morse graph  presented in Figure~\ref{fig:mg_3d_example_intro_periodic} is obtained using the combinatorial multivalued map $\cF_3\colon \cX(\I) \mvmap \cX(\I)$ where 
\[
\I = \setof{0,1,2,3,4} \times \setof{0,1,2,3,4}\times \setof{0,1,2,3}.
\]
The computation of the Morse graph and connection matrix for this example takes about $0.12$ seconds.
In particular, this gives rise to D-gradings $\pi\colon \cX\to \SCC(\cF_3)$ and $\pi_b\colon \cX_b\to \SCC(\cF_3)$.
The associated Conley index data is computed using $\Z_2$ coefficients.
Let $\bG = \bG(\gamma,\nu,\theta,h)$ be a geometrization associated with $\cF_3$ and let $\varphi$ denote the flow generated by \eqref{eq:ramp_system_intro_periodic}.

Theorem~\ref{thm:R3ABlattice} and the fact that each of the Conley indices is nontrivial implies that the Morse graph defines a Morse representation for $\varphi$.
In particular, given $\bg\in \bG$
\[
M(p) = \Inv(\bg(\pi_b(\sO(p)))\setminus\bg(\pi_b(\sO(p)^<));\varphi) \neq \emptyset
\]
for all nodes $p$ of the Morse graph.

We focus on the Morse set $M(1)$ with Conley index 
\[
CH_k(1; \Z_2) \cong \begin{cases}
    \Z_2 & \text{if $k=0,1$}\\
    0 & \text{otherwise.}
\end{cases}
\]
This allows for the possibility that $M(1)$ consists of a stable periodic orbit.
Our goal is to prove that $M(1)$ contains a periodic orbit using \cite[Theorem 1.3]{mccord:mischaikow:mrozek}.

We begin by collecting the necessary information.
Globally, since $1$ is a minimal node in the Morse graph, $\bg(\pi_b(1))$ is an attracting block, i.e., if $x \in \bg(\sO(\pi_b(1)))$, then $\varphi((0,\infty),x)\subset \bg(\sO(\pi_b(1)))$.
More locally, let $\xi\in \pi^{-1}(1)$, then $\xi\not\in\cF_3(\xi)$.
This implies that $\Inv(\bg(\blup(\xi)),\varphi) = \emptyset$ and hence if $x\in \bg(\blup(\xi))$, then there exists $t>0$ such that $\varphi(t,x)\not\in \bg(\blup(\xi))$.

Figure~\ref{fig:SCC2} shows the edges of $\cF_3$ restricted to $\pi^{-1}(1)$. 
Consider the sequence of cells and edges
\[ 
\begin{tikzpicture}
            [main node/.style={rectangle,fill=white!20,
            font=\sffamily\tiny\bfseries},scale=2.5]
            \node[main node] (1) at (0,0) {$\defcellb{2}{2}{1}{1}{1}{1}$};
            \node[main node] (2) at (1,0) {$\defcellb{2}{2}{2}{1}{1}{0}$};
            \node[main node] (3) at (2,0) {$\defcellb{2}{2}{2}{1}{1}{1}$};

            \path[thick]
            (1) edge[->,shorten <= 2pt, shorten >= 2pt] (2)
            (2) edge[->,shorten <= 2pt, shorten >= 2pt] (3)
            ;
\end{tikzpicture} 
\]
corresponding to a 3-dimensional cell mapping to a 2-dimensional cell mapping to a 3-dimensional cell. 
Both edges arises from \textbf{Condition 1.1}. 
Observe that it is the third coordinate that is changing and increasing.
It is left to the reader to check that 
\[
\blup\left(\defcellb{2}{2}{1}{1}{1}{1} \right) = \defcellb{5}{5}{3}{1}{1}{1}.
\]
By Lemma~\ref{lem:gradient-direction-is-nonzero}, $\dot{x}_3 >0$ on $\bg\left(\blup\left( \defcellb{2}{2}{1}{1}{1}{1} \right)\right)$, and furthermore, the flow $\varphi$ is transverse across $\bg\left( \defcellb{5}{5}{4}{1}{1}{0} \right)$.
This leads to two fundamental observations.
First, $\bg\left( \defcellb{5}{5}{4}{1}{1}{0} \right)$ is a local section  $\varphi$.
Second, if $x\in \blup\left(\defcellb{2}{2}{1}{1}{1}{1} \right)$, then there exist $t_1 > t_0 \geq 0$ such that 
\[
\varphi([0,t_1],x)\subset \blup\left(\defcellb{2}{2}{1}{1}{1}{1} \right) \cup \blup\left(\defcellb{2}{2}{2}{1}{1}{0} \right)
\]
and
\[
\varphi([0,t_1],x)\cap \blup\left(\defcellb{2}{2}{1}{1}{1}{1} \right)  = \varphi([0,t_0)],x). 
\]
In words, the trajectory starting at $x$ leaves $\blup\left(\defcellb{2}{2}{1}{1}{1}{1} \right)$ and enters $\blup\left(\defcellb{2}{2}{2}{1}{1}{0} \right)$.
A similar argument  shows that the trajectory starting at $\varphi(t_1,x)$ leaves $\blup\left(\defcellb{2}{2}{2}{1}{1}{0} \right)$ and enters $\blup\left(\defcellb{2}{2}{2}{1}{1}{1} \right)$.

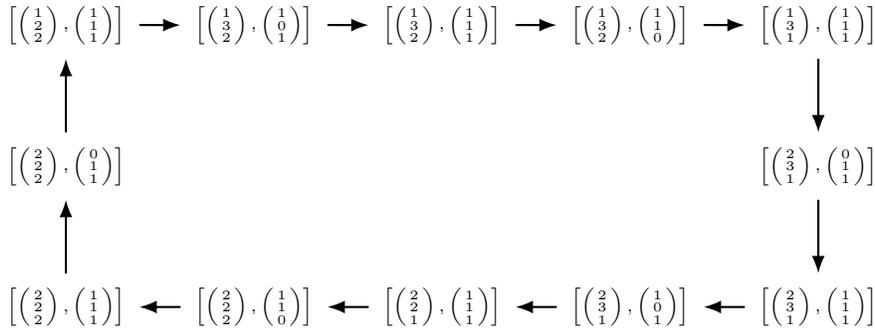
\begin{figure}[!htb]
\begin{tikzpicture}
            [main node/.style={rectangle,fill=white!20,
            font=\sffamily\tiny\bfseries},scale=2.5]
            \node[main node] (1) at (2,2) {$\defcellb{2}{2}{1}{1}{1}{1}$};
            \node[main node] (2) at (1,2) {$\defcellb{2}{2}{2}{1}{1}{0}$};
            \node[main node] (3) at (0,2) {$\defcellb{2}{2}{2}{1}{1}{1}$};
            \node[main node] (4) at (0,2.75) {$\defcellb{2}{2}{2}{0}{1}{1}$};
            \node[main node] (5) at (0,3.5) {$\defcellb{1}{2}{2}{1}{1}{1}$};
        \node[main node] (6) at (1,3.5) {$\defcellb{1}{3}{2}{1}{0}{1}$};
        \node[main node] (7) at (2,3.5) {$\defcellb{1}{3}{2}{1}{1}{1}$};
        \node[main node] (8) at (3,3.5) {$\defcellb{1}{3}{2}{1}{1}{0}$};
            \node[main node] (9) at (4,3.5) {$\defcellb{1}{3}{1}{1}{1}{1}$};
            \node[main node] (14) at (4,2.75) {$\defcellb{2}{3}{1}{0}{1}{1}$};
            \node[main node] (15) at (4,2) {$\defcellb{2}{3}{1}{1}{1}{1}$};
            \node[main node] (16) at (3,2) {$\defcellb{2}{3}{1}{1}{0}{1}$};

            \path[thick]
            (1) edge[-Latex,shorten <= 2pt, shorten >= 2pt] (2)
            (2) edge[-Latex,shorten <= 2pt, shorten >= 2pt] (3)
            (3) edge[-Latex,shorten <= 2pt, shorten >= 2pt] (4)
            (4) edge[-Latex,shorten <= 2pt, shorten >= 2pt] (5)
            (5) edge[-Latex,shorten <= 2pt, shorten >= 2pt] (6)
            (6) edge[-Latex,shorten <= 2pt, shorten >= 2pt] (7)
            (7) edge[-Latex,shorten <= 2pt, shorten >= 2pt] (8)
            (8) edge[-Latex,shorten <= 2pt, shorten >= 2pt] (9)
            (9) edge[-Latex,shorten <= 2pt, shorten >= 2pt] (14)

            (14) edge[-Latex,shorten <= 2pt, shorten >= 2pt] (15)
            (15) edge[-Latex,shorten <= 2pt, shorten >= 2pt] (16)
            (16) edge[-Latex,shorten <= 2pt, shorten >= 2pt] (1)
;
\end{tikzpicture}
\caption{$\cF_3$ restricted to cells in $\cX$ associated with strongly connected component $1$.}
\label{fig:SCC2}
\end{figure}

Repeating this argument shows that given $x\in M(1)$, its forward trajectory passes through the cells associated with $M(1)$ in the order shown in Figure~\ref{fig:SCC2}.
In particular, for any $x\in M(1)$, there exists $t_x >0$, such that $\varphi(t_x,x)\in \bg\left( \defcellb{5}{5}{4}{1}{1}{0} \right)$.
Therefore, we have identified a section for $M(1)$, and hence by \cite[Theorem 1.3]{mccord:mischaikow:mrozek} $M(1)$ contains a periodic orbit.
\end{ex}

\begin{ex}
\label{ex:periodicOrbit5}

Consider the regulatory network shown in Figure~\ref{fig:network_5D}.
The associated abstract cubical complex is $\cX(\I)$ where
\[
\I = \setof{0,1,2,3} \times \setof{0,1,2,3,4,5}\times \setof{0,1,2,3,4}\times \setof{0,1,2}\times \setof{0,1,2}.
\]
The associated parameter graph has 13,608,000,000 vertices and for the purpose of this example we focus on parameter node 5,103,162,287.
The associated rook field $\Phi\colon TP(\cX) \to \setof{0,\pm1}^5$ has seven equilibrium cells (see Definition~\ref{def:eqcell}).
Since $N=5$, we are restricted to working with $\cF_2\colon\cX\mvmap \cX$.

The associated Morse graph is shown in Figure~\ref{fig:example_5d_periodic_mg}. The computation of the Morse graph and associated connection matrix for this example took $9.15$ seconds.
Each of the nodes has a nontrivial Conley index and therefore by Theorem~\ref{thm:R2ABlattice} the Morse graph defines a Morse representation of the associated ramp system.

\begin{figure*}[!htb]
\centering
\begin{tikzpicture}
[main node/.style={circle,fill=white!20,draw},scale=2.5]
\node[main node] (5) at (0.3, 0) {$5$};
\node[main node] (3) at (0.0, 0.6) {$3$};
\node[main node] (4) at (0.6, 0.6) {$4$};
\node[main node] (2) at (0.3, 1.2) {$2$};
\node[main node] (1) at (0.3, 1.8) {$1$};
\path[thick]
(1) edge[-|, shorten <= 2pt, shorten >= 2pt, bend right] (2)
(2) edge[-|, shorten <= 2pt, shorten >= 2pt, bend right] (1)
(2) edge[-|, shorten <= 2pt, shorten >= 2pt] (3)
(2) edge[->, shorten <= 2pt, shorten >= 2pt] (4)
(3) edge[-|, shorten <= 2pt, shorten >= 2pt, bend left, in=140] (1)
(3) edge[->, shorten <= 2pt, shorten >= 2pt] (5)
(4) edge[->, shorten <= 2pt, shorten >= 2pt] (5)
(1) edge[->, loop, shorten <= 2pt, shorten >= 2pt, distance=8pt, thick, out=-35, in=35] (1)
(2) edge[->, loop, shorten <= 2pt, shorten >= 2pt, distance=8pt, thick, out=-35, in=35] (2)
(3) edge[->, loop, shorten <= 2pt, shorten >= 2pt, distance=8pt, thick, out=-35, in=35] (3)
(5) edge[->, loop, shorten <= 2pt, shorten >= 2pt, distance=8pt, thick, out=-35, in=30] (5);
\end{tikzpicture}
\caption{Five node network.}
\label{fig:network_5D}
\end{figure*}
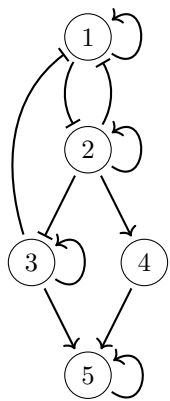

\begin{figure*}[!htb]
\centering
\includegraphics[width=0.95\textwidth]{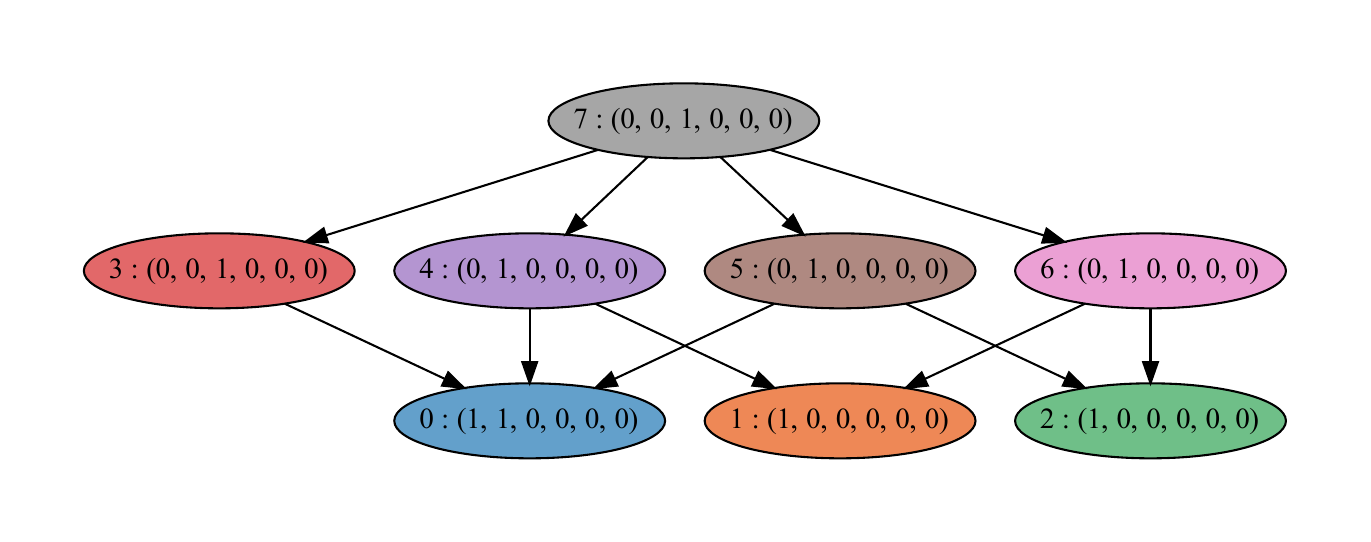}
\caption{Morse graph for the map $\cF_2$ generated by the wall labeling associated with parameter node $5,103,162,287$ of the parameter graph of the network in Figure~\ref{fig:network_5D}.}
\label{fig:example_5d_periodic_mg}
\end{figure*}

The above mentioned seven equilibrium cells are associated with the nodes $1$ through $7$ and have Conley indices associated with hyperbolic fixed points with unstable manifolds of dimension $0$ through $2$.
We draw attention to this fact to emphasize that even though it is natural to associate {\bf Condition 3.1} (and hence $\cF_3$) with identification of equilibria, in many cases   $\cF_2$ provides enough resolution to separate equilibrium cells into distinct Morse nodes. 

The Conley index of node $0$ is that of a stable periodic orbit.
Figure~\ref{fig:example_5d_periodic_scc} shows the edges of $\cF_2$ restricted to $\pi^{-1}(0)$.
The same argument presented in Example~\ref{ex:periodicOrbit3} implies the existence of a periodic orbit for any associated ramp system.

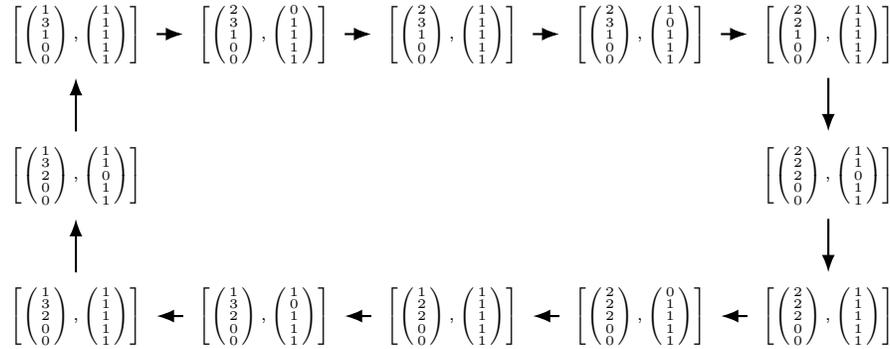
\begin{figure}[!htb]
\begin{tikzpicture}
    [main node/.style={rectangle,fill=white!20,
            font=\sffamily\tiny\bfseries},scale=2.5]
    \node[main node] (1) at (2,2) {$\defcellc{1}{2}{2}{1}{1}{1}{1}{1}$};
    \node[main node] (2) at (1,2) {$\defcellc{1}{3}{2}{1}{0}{1}{1}{1}$};
    \node[main node] (3) at (0,2) {$\defcellc{1}{3}{2}{1}{1}{1}{1}{1}$};
    \node[main node] (4) at (0,2.75) {$\defcellc{1}{3}{2}{1}{1}{0}{1}{1}$};
    \node[main node] (5) at (0,3.5) {$\defcellc{1}{3}{1}{1}{1}{1}{1}{1}$};
    \node[main node] (6) at (1,3.5) {$\defcellc{2}{3}{1}{0}{1}{1}{1}{1}$};
    \node[main node] (7) at (2,3.5) {$\defcellc{2}{3}{1}{1}{1}{1}{1}{1}$};
    \node[main node] (8) at (3,3.5) {$\defcellc{2}{3}{1}{1}{0}{1}{1}{1}$};
    \node[main node] (9) at (4,3.5) {$\defcellc{2}{2}{1}{1}{1}{1}{1}{1}$};
    \node[main node] (14) at (4,2.75) {$\defcellc{2}{2}{2}{1}{1}{0}{1}{1}$};
    \node[main node] (15) at (4,2) {$\defcellc{2}{2}{2}{1}{1}{1}{1}{1}$};
    \node[main node] (16) at (3,2) {$\defcellc{2}{2}{2}{0}{1}{1}{1}{1}$};

            \path[thick]
            (1) edge[-Latex,shorten <= 2pt, shorten >= 2pt] (2)
            (2) edge[-Latex,shorten <= 2pt, shorten >= 2pt] (3)
            (3) edge[-Latex,shorten <= 2pt, shorten >= 2pt] (4)
            (4) edge[-Latex,shorten <= 2pt, shorten >= 2pt] (5)
            (5) edge[-Latex,shorten <= 2pt, shorten >= 2pt] (6)

            (6) edge[-Latex,shorten <= 2pt, shorten >= 2pt] (7)

            (7) edge[-Latex,shorten <= 2pt, shorten >= 2pt] (8)
            (8) edge[-Latex,shorten <= 2pt, shorten >= 2pt] (9)
            (9) edge[-Latex,shorten <= 2pt, shorten >= 2pt] (14)
 
            (14) edge[-Latex,shorten <= 2pt, shorten >= 2pt] (15)
            (15) edge[-Latex,shorten <= 2pt, shorten >= 2pt] (16)
            (16) edge[-Latex,shorten <= 2pt, shorten >= 2pt] (1);
\end{tikzpicture}
\caption{Strongly connected component corresponding to node $0$ of the Morse graph in Figure~\ref{fig:example_5d_periodic_mg}.}
\label{fig:example_5d_periodic_scc}
\end{figure}
\end{ex}

\begin{ex}
\label{ex:trivial_index_3D_F2}
Let $\cF_3$ be the map generated by the parameter node 2,023,186 of the parameter graph of the network in Figure~\ref{fig:example_1_3d_network}. 
The corresponding Morse graph is shown in  Figure~\ref{fig:missing_periodic_3D}. 
The recurrent component corresponding to Morse node $3$ has $29$ cells and $4$ pairs of double edges. 
For all the other Morse nodes the corresponding recurrent components consist of a single equilibrium cell. 
For Morse node $3$ the restrictive condition on the definition of $\cF_2$ 
prevents the double edges from being resolved and hence it is not possible to separate the equilibrium cell from the remaining cells in the recurrent component. 

For this example we conjecture that $M(3)$ contains  a fixed point and a periodic orbit.

\begin{figure*}[!htb]
\centering
\begin{tikzpicture}
[main node/.style={circle,fill=white!20,draw,font=\sffamily\tiny\bfseries},scale=2.5]
\node[main node] (1) at (0,0.5) {1};
\node[main node] (2) at (0,0) {2};
\node[main node] (3) at (0.5,0) {3};
\path[thick]
(1) edge[loop, shorten <= 2pt, shorten >= 2pt, distance=10pt, thick, out=155, in=205, ->] (1)
(1) edge[-|, shorten <= 2pt, shorten >= 2pt, bend right] (2)
(2) edge[-|, shorten <= 2pt, shorten >= 2pt, bend right] (1)
(2) edge[loop, shorten <= 2pt, shorten >= 2pt, distance=10pt, thick, out=155, in=205, ->] (2)
(2) edge[-|, shorten <= 2pt, shorten >= 2pt] (3)
(3) edge[-|, shorten <= 2pt, shorten >= 2pt, bend right] (1)
(3) edge[loop, shorten <= 2pt, shorten >= 2pt, distance=10pt, thick, out=25, in=-25, ->]  (3);
\end{tikzpicture}
\caption{Three node regulatory network whose parameter graph has $3,600,000$ parameter nodes.}
\label{fig:example_1_3d_network}
\end{figure*}
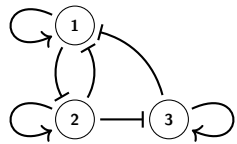

\begin{figure*}[!htb]
\centering
\includegraphics[width=0.7\textwidth]{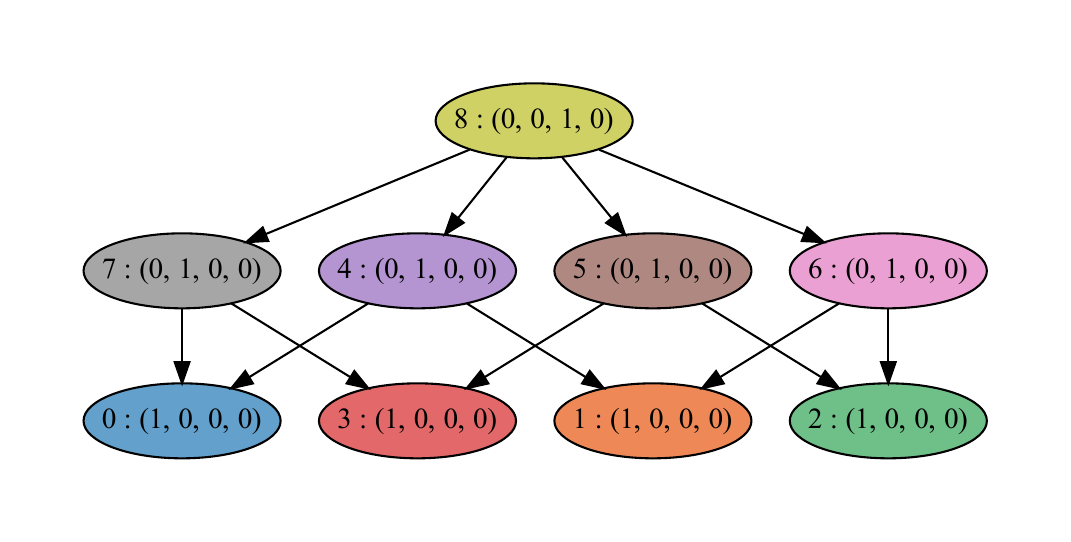}
\caption{Morse graph for node $2,023,186$ of the parameter graph of the network in Figure~\ref{fig:example_1_3d_network}.}
\label{fig:missing_periodic_3D}
\end{figure*}
\end{ex}

The  discussion in this section focuses on the identification of a periodic orbit, the simplest example of nontrivial recurrent dynamics.
Observe that the essential input to the process is a wall labeling and the results of this monograph provide a means by which combinatorial computations and analytic estimates allow conclusions about dynamics to be drawn concerning a specific family of ODEs.
We note that the work of \cite{mrozek:srzednicki:thorpe:wanner} is similar in spirit.
The starting point is the concept of a multivector field upon which similar combinatorial and homological computations are performed and for which similar conclusions can be obtained.
The setting of \cite{mrozek:srzednicki:thorpe:wanner} is more general in nature, but at the cost of weaker ties to explicit differential equations.

There are at least two challenges associated with identifying recurrent dynamics.
The first is to develop explicit algorithms that take strongly connected components of a combinatorial multivalued map $\cF$ and identify that a periodic orbit exists. 
A primary obstacle appears to be that of identifying appropriate local sections.
 
The second is to develop explicit algorithms that identify more complex dynamics.
It is well established that the machinery of the Conley index is capable of identifying chaotic dynamics \cite{mischaikow:mrozek:95a,szymczak:96,day:frongillo}.
Direct computational results in the context of flows are less well developed, but do exist \cite{mrozek:srzednicki,mrozek:srzednicki:thorpe:wanner}.

\section{Global Semi-conjugacies}
\label{sec:semiconjugacy}

The discussion of the previous two sections focused on the combinatorial and homological information can be used to prove the existence of particular trajectories; connecting orbits in Section~\ref{sec:saddlesaddlebif} and periodic orbits in Section~\ref{sec:3d_example_periodic_orbit}.
In this section we show how this information can be used to obtain statements about the global structure of the dynamics.
We begin with a particularly simple example that we use to motivate the use of semiconjugacies and conclude with a nontrivial example.

\begin{ex}
\label{ex:mccord}
Returning to Example~\ref{ex:connection_matrix_running_ex}, 
consider the regulatory network of Figure~\ref{fig:network_parameter_regions_intro} (A).
This network has a parameter graph with $1600$ nodes in the DSGRN parameter graph. 
For the parameter node $974$ DSGRN produces the wall labeling shown in Figure~\ref{fig:wall_labeling} (A).
As a consequence the map $\cF_3$ and the corresponding Morse graph and connection matrix computed in Example~\ref{ex:connection_matrix_running_ex} are valid for all parameter values $(\gamma,\nu,\theta)$ in this DSGRN parameter node.
Furthermore, Theorem~\ref{thm:R1ABlattice} implies that we have computed a Morse representation and the unique connection matrix for the global attractor $K$ of the ramp system $\dot{x} = -\Gamma x + E(x; \gamma,\nu,\theta,h)$ for all $h\in \cH_1(\gamma,\nu,\theta)$.
As a consequence a Morse representation on $K$ is
\begin{equation}
    \label{eq:MR6.4.7}
    \sMR = \setdef{M(0),M(1),M(2)}{0<2,\ 1<2}
\end{equation}
for which the associated Conley indices are
\begin{equation}
    \label{eq:CIMR6.4.7}
CH_k(0)\cong CH_k(1) \cong\begin{cases}
    \Z & \text{if $k=0$,}\\
    0 & \text{otherwise}
\end{cases} \quad\text{and}\quad
CH_k(2) \cong\begin{cases}
    \Z & \text{if $k=1$,}\\
    0 & \text{otherwise.}
\end{cases}
\end{equation}

By \cite[Theorem 1]{duncan21:ramp}, we can conclude that $M(0)$ and $M(1)$ are stable hyperbolic fixed points, and $M(2)$ is a hyperbolic fixed point with a one-dimensional unstable manifold.
The unique connection matrix implies that there exists a connecting orbit from $M(2)$ to $M(0)$ and from $M(2)$ to $M(1)$.
Since $M(2)$ has a one-dimensional unstable manifold this implies that these are the unique connecting orbits.
Furthermore, since $\sMR$ is a Morse representation we have obtained a complete description of the dynamics on $K$.

To put this discussion into a broader context let $\bar{K}$ be the geometric simplicial complex consisting of the intervals $\setof{[-1,0],[0,1]}$ and vertices ${-1,0,1}$.
Let $\psi\colon \R\times\bar{K} \to \bar{K}$ defined by $\dot{y} = y(1-y^2)$.
What we have argued above is that for any $(\gamma,\nu,\theta)\in R(974)$ and $h\in \cH_3(\gamma,\nu,\theta)$ the flow $\varphi\colon \R\times K \to K$, obtained by restricting the associated ramp system to its global attractor, is conjugate to $\psi$.
Motivated by {\bf Goal 1} we want to address the following question:
\begin{quote}
\emph{Are these results valid for a larger class of differential equations?}
\end{quote}

If we replace the ramp functions by Hill functions (see \eqref{eq:Hillrepressor}), then for $s_{i,j}$ sufficiently large each Morse set consists of hyperbolic fixed point \cite{duncan21:ramp}, and therefore, the same argument applies.
However, without additional arguments we cannot assume that this is the case for the specific values of  $s_{i,j}$ used for system \eqref{eq:Hill_2D_example_ODE}.
The perspective of this paper is that the direct identification of invariant sets, in this case Morse sets, is not in general computationally tractable.
Instead, one should focus on verifying the existence of attracting blocks as is discussed in Chapter~\ref{sec:prelude} and illustrated in Figure~\ref{fig:phase_portrait_Hill_intro}.

With this in mind we turn to \cite[Theorem 1.2]{mccord}.
In particular, given $\varphi\colon \R\times K \to K$, the Morse representation \eqref{eq:MR6.4.7}, the Conley indices \eqref{eq:CIMR6.4.7}, the uniqueness of the connection matrix and the flow $\psi\colon \R\times\bar{K} \to \bar{K}$, the hypotheses {\bf H0}-{\bf H4} of \cite[Theorem 1.2]{mccord} are satisfied and therefore, there exists a surjective semi-conjugacy $f\colon K\to \bar{K}$ from the flow $\varphi$ to the flow $\psi$ such that $f(M(0))=-1$, $f(M(1))=1$, and $f(M(2))=0$.
Note that it is not necessary to understand the structure of the Morse sets, just their homological Conley indices.
\end{ex}

\begin{ex}
\label{ex:3d_example_1}
The results of Example~\ref{ex:mccord} were included to provide intuition. 
We now repeat the same application for a nontrivial example.
Consider the network in Figure~\ref{fig:example_1_3d_network}. 
The parameter graph for this network has $3,600,000$ parameter nodes. 
Consider the map $\cF_3$ generated by the parameter node $2,472,286$. 
The Morse graph $\sMG(\cF_3)$ is presented in Figure~\ref{example_single_CM_3d_mg}. The computation of the Morse graph and the associated connection matrix for this example took $0.14$ seconds.
Observe that the Conley index associated with each node is nontrivial, 
and therefore, we have a Morse representation
\[
\sMR = \setof{M(p) \mid p=0,\ldots,12}
\]
where the partial order is that of the Morse graph.
Furthermore, each Morse set has the index of a hyperbolic fixed point and the connection matrix is unique.

\begin{figure*}[!htb]
\centering
\includegraphics[width=0.85\textwidth]{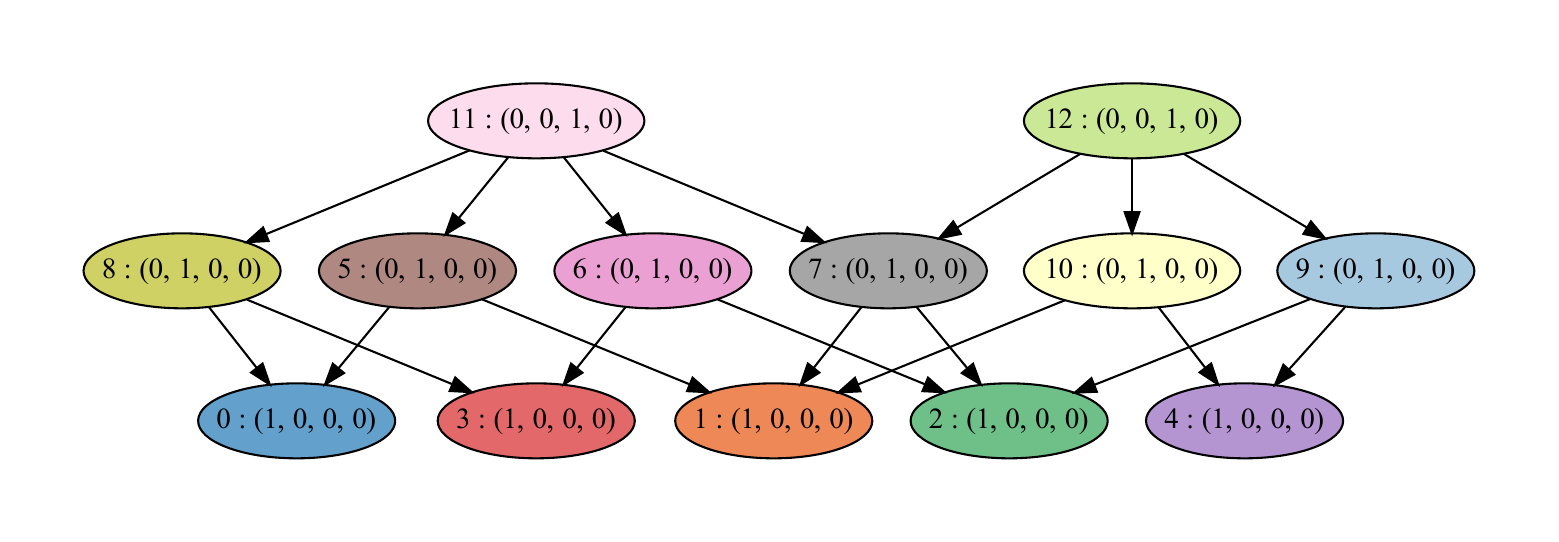}
\caption{Morse graph for the map $\cF_3$ generated by the parameter node $2,472,286$ of the parameter graph of the network in Figure~\ref{fig:example_1_3d_network}.}
\label{example_single_CM_3d_mg}
\end{figure*}

Following the prescription of \cite[Section 3]{mccord} we construct the simplicial complex and flow shown in Figure~\ref{fig:mccordflow}.
The complex $\bar{K}$ has 13 vertices, one for each Morse set.
The faces of $\bar{K}$ have the form $\langle p_i,p_j \rangle$ and $\langle p_i,p_j,p_k \rangle$ where $p_i<p_j$ and $p_i<p_j<p_k$ in the Morse graph, respectively.
The fixed points of $\psi \colon \R\times \bar{K}\to \bar{K}$ are given by the vertices.
The direction of the flow $\psi \colon \R\times \bar{K}\to \bar{K}$ along the edges is indicated by the arrows.
The flow on each two-dimensional cell is non-recurrent and compatible with the arrows on its edges. 

\begin{figure}[!htb]
\centering
\begin{tikzpicture}[scale=0.25]
\fill[gray!10!white] (0,0) -- (0,6) -- (12,12) -- (24,12) -- (24,0) -- (0,0);

\draw[step=6cm,black] (0,0) grid (24,6);
\draw (12,6) -- (12,12);
\draw (12,12) -- (18,12);
\draw (18,6) -- (18,12);
\draw (18,12) -- (24,12);
\draw (24,6) -- (24,12);

\draw (0,0) -- (12,12);
\draw (12,0) -- (24,12);
\draw (12,12) -- (24,0);
\draw (6,6) -- (12,0);
\draw (0,6) -- (12,12);

\foreach \i in {0,...,1}{
\draw [-{Latex[length=3mm]}] (6,6*\i) -- (2.5,6*\i);
}
\foreach \i in {0,...,1}{
\draw [-{Latex[length=3mm]}] (6,6*\i) -- (9.5,6*\i);
}
\foreach \i in {0,...,4}{
\draw [-{Latex[length=3mm]}] (6*\i,6) -- (6*\i,2.5);
}
\foreach \i in {0,...,2}{
\draw [-{Latex[length=3mm]}] (12+6*\i,6) -- (12+6*\i,9.5);
}
\foreach \i in {0,...,2}{
\draw [-{Latex[length=3mm]}] (18,6*\i) -- (14.5,6*\i);
}
\foreach \i in {0,...,2}{
\draw [-{Latex[length=3mm]}] (18,6*\i) -- (21.5,6*\i);
}

\draw [-{Latex[length=3mm]}] (0,6) -- (6,9);
\draw [-{Latex[length=3mm]}] (6,6) -- (2.5,2.5);
\draw [-{Latex[length=3mm]}] (6,6) -- (9.5,2.5);
\draw [-{Latex[length=3mm]}] (6,6) -- (9.5,9.5);

\draw [-{Latex[length=3mm]}] (18,6) -- (14.5,2.5);
\draw [-{Latex[length=3mm]}] (18,6) -- (14.5,9.5);
\draw [-{Latex[length=3mm]}] (18,6) -- (21.5,2.5);
\draw [-{Latex[length=3mm]}] (18,6) -- (21.5,9.5);

\foreach \i in {0,...,4}{
  \foreach \j in {0,...,1}{
    \draw[black, fill=black] (6*\i,6*\j) circle (2ex);
  }
}
\draw[black, fill=black] (12,12) circle (2ex);
\draw[black, fill=black] (18,12) circle (2ex);
\draw[black, fill=black] (24,12) circle (2ex);

\draw[blue] (-1,0) node{$4$};
\draw[blue] (25,0) node{$0$};
\draw[blue] (25,6) node{$8$};
\draw[blue] (12,-1) node{$1$};
\draw[blue] (18,-1) node{$5$};
\draw[blue] (-1,6) node{$9$};
\draw[blue] (18,13) node{$6$};
\draw[blue] (11,5) node{$7$};
\draw[blue] (6,-1) node{$10$};
\draw[blue] (25,12) node{$3$};
\draw[blue] (11,13) node{$2$};
\draw[blue] (20,7) node{$11$};
\draw[blue] (5,7) node{$12$};
\end{tikzpicture}
\caption{Two-dimensional complex $\bar{K}$ with flow $\psi \colon \R\times \bar{K}\to \bar{K}$ where labeled vertices correspond to fixed points and arrows indicated the direction of the nonrecurrent dynamics.}
\label{fig:mccordflow}
\end{figure}
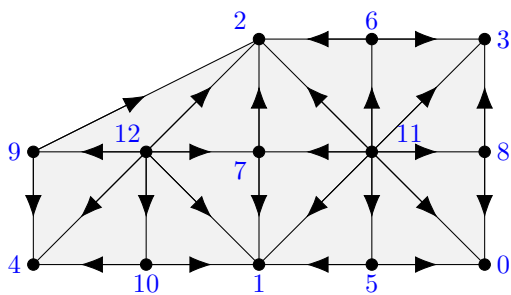

Consider a ramp system $\dot{x} = -\Gamma x + E(x;\nu,\theta,h)$ for any $h\in \cH_3(\gamma,\nu,\theta)$ where $(\gamma,\nu,\theta)\in R(\text{2,472,286})$.
Observe that this is a three-dimensional system with a $31$ dimensional parameter space.
Let $\varphi\colon \R\times K\to K$ denote the flow obtained by restricting the dynamics to its global attractor. Then \cite[Theorem 1.2]{mccord} implies that there exists a semi-conjugacy from $\varphi$ to $\psi$ for which the Morse set $M(p)$ is sent to vertex $p$.

Figure~\ref{fig:mccordflow} provides a qualitative understanding of the dynamics of the above mentioned ramp system for an unbounded set of parameter values.
It is possible that the above mentioned semi-conjugacy is in fact a conjugacy.
We can identifying equilibria for ramp systems \cite{duncan21,duncan21:ramp}.
However, beyond the existence of connecting orbits guaranteed by the semi-conjugacy, we do not know how to identify the fine structure of the connections, e.g., the topology of the set of connecting orbits from Morse set $12$ to Morse set $2$ is homeomorphic to a two-dimensional disk. 
It is worth noting that more quantitative information can be extracted from the geometric realization $\bG$, but how to employ this information effectively to prove the above mentioned fine structure remains an open question.
\end{ex}

\begin{ex}
To emphasize that the application of \cite[Theorem 1.2]{mccord} does not depend on the extrinsic dimension of the differential we consider of the six-dimensional network in Figure~\ref{fig:RN}. 
As indicated in the introduction the network in Figure~\ref{fig:RN} is a proposed regulatory network for the epithelial-to-mesenchymal transition and thus the labeling of the nodes corresponds to the associate genes and micro RNA.
The translation to the notation of this paper is as follows
\begin{align*}
\text{node 1} &\leftrightarrow \text{TGFb},\quad
\text{node 2} \leftrightarrow \text{miR200},\quad
\text{node 3} \leftrightarrow \text{Snail1}, \\
\text{node 4} &\leftrightarrow \text{Ovol2},\quad
\text{node 5} \leftrightarrow \text{Zeb1},\quad
\text{node 6} \leftrightarrow \text{miR34a}.
\end{align*}

Because this system is six-dimensional we cannot make use of $\cF_3$.
Let $\cF_2$ be the map generated by the parameter node 2,684,686,006 out of the 4,429,771,960,320 nodes in the parameter graph.
The Morse graph computed from $\cF_2$ is presented in Figure~\ref{fig:example_6d_mg} (Left). The computation of the Morse graph and the associated connection matrix for this example took $2.2$ minutes.

Consider a ramp system $\dot{x} = -\Gamma x + E(x;\nu,\theta,h)$ for any $h\in \cH_2(\gamma,\nu,\theta)$ where $(\gamma,\nu,\theta)\in R(\text{2,684,686,006})$.
This is a six-dimensional system with a 58 dimensional parameter space.
Let $\varphi\colon \R\times K\to K$ denote the flow obtained by restricting the dynamics to its global attractor.
We leave it to the reader to check that the hypotheses of \cite[Theorem 1.2]{mccord} are satisfied and therefore there exists a semi-conjugacy from $\varphi$ to $\psi$ where the latter is described in Figure~\ref{fig:example_6d_mg} (Right).

\begin{figure*}[!htb]
\centering
\includegraphics[width=0.7\textwidth]{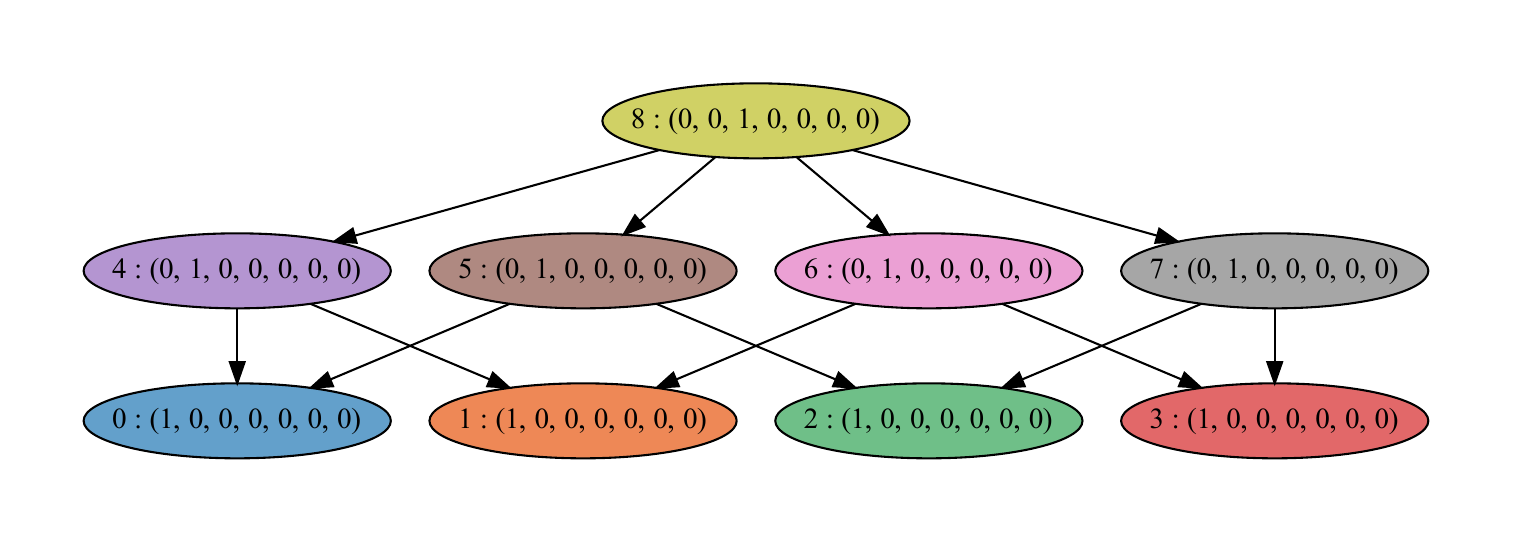}
\begin{tikzpicture}[scale=0.2]
\fill[gray!10!white] (12,0) -- (12,12) -- (24,12) -- (24,0) -- (12,0);

\draw[step=6cm,black] (12,0) grid (24,12);

\draw (12,0) -- (24,12);
\draw (12,12) -- (24,0);

\foreach \i in {2,...,4}{
\draw [-{Latex[length=3mm]}] (6*\i,6) -- (6*\i,2.5);
}
\foreach \i in {0,...,2}{
\draw [-{Latex[length=3mm]}] (12+6*\i,6) -- (12+6*\i,9.5);
}
\foreach \i in {0,...,2}{
\draw [-{Latex[length=3mm]}] (18,6*\i) -- (14.5,6*\i);
}
\foreach \i in {0,...,2}{
\draw [-{Latex[length=3mm]}] (18,6*\i) -- (21.5,6*\i);
}

\draw [-{Latex[length=3mm]}] (18,6) -- (14.5,2.5);
\draw [-{Latex[length=3mm]}] (18,6) -- (14.5,9.5);
\draw [-{Latex[length=3mm]}] (18,6) -- (21.5,2.5);
\draw [-{Latex[length=3mm]}] (18,6) -- (21.5,9.5);

\foreach \i in {2,...,4}{
  \foreach \j in {0,...,1}{
    \draw[black, fill=black] (6*\i,6*\j) circle (2ex);
  }
}
\draw[black, fill=black] (12,12) circle (2ex);
\draw[black, fill=black] (18,12) circle (2ex);
\draw[black, fill=black] (24,12) circle (2ex);

\draw[blue] (25,0) node{$0$};
\draw[blue] (25,6) node{$5$};
\draw[blue] (12,-1) node{$1$};
\draw[blue] (18,-1) node{$4$};
\draw[blue] (18,13) node{$7$};
\draw[blue] (11,5) node{$6$};
\draw[blue] (25,12) node{$2$};
\draw[blue] (11,13) node{$3$};
\draw[blue] (20,7) node{$8$};
\end{tikzpicture}
\caption{(Left) Morse graph for the map $\cF_2$ generated by the parameter node 2,684,686,006 of the parameter graph of the network in Figure~\ref{fig:RN}.
(Right) Two-dimensional complex $\bar{K}$ with flow $\psi \colon \R\times \bar{K}\to \bar{K}$ where labeled vertices correspond to fixed points and arrows indicated the direction of the nonrecurrent dynamics.}
\label{fig:example_6d_mg}
\end{figure*}
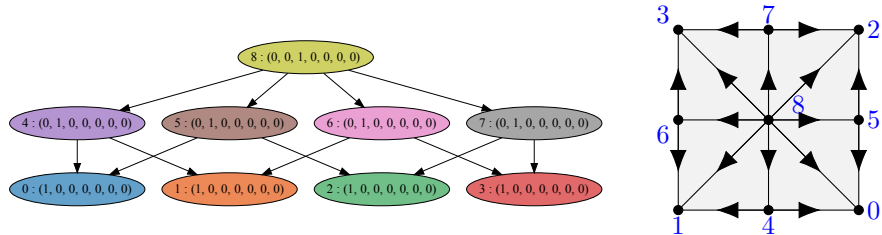

To emphasize the power of \cite[Theorem 1.2]{mccord} we note that thought the Conley indices of the Morse sets are those of hyperbolic fixed points, some of the Morse nodes consist of multiple cells.
In particular, Morse nodes 4, 5, 6, and 7 each consist of 5 cells, while  Morse node 8 consists of 35 cells.
This suggests the possibility that the associated Morse sets contain more complicated dynamics than just that of a fixed point, and  that  potentially the dynamics within the Morse set is sensitive to specific choices of parameter values for the ramp system.
Nevertheless, the semi-conjugacy provides us with a rigorous  understanding of the global dynamics.
\end{ex}

\section{Morse graphs that are not Morse representations}
\label{sec:notMorseRep}

In general we cannot expect to have as succinct an understanding of the global dynamics as presented in Section~\ref{sec:semiconjugacy}.
There are two challenges.
The first is that representation of the dynamics of \cite[Theorem 1.2]{mccord} is restricted a gradient-like system consisting of fixed points and connecting orbits.
Perhaps the work of this manuscript will provide motivation to extend \cite[Theorem 1.2]{mccord}.
The second is more fundamental and is associated with nodes in the Morse graph that have trivial Conley index; the fact that a node has been identified via the combinatorial computations suggests that recurrent dynamics may occur, but the fact that the Conley index is trivial does not provide us with any guarantee that such recurrence occurs.
Stated using the language of this monograph, the Morse graph computed using the techniques developed in Part~\ref{part:II} defines a Morse decomposition but not a Morse representation for the associated ramp system.

The examples in this section are meant to provide insight with respect to this second challenge.

\begin{ex}
\label{ex:trivial_index_2D}

We consider again the network in Figure~\ref{fig:example_2_2d_network}. Let $\cF_3$ be the map generated by the parameter node $47$. 
The associated Morse graph $\sMG(\cF_3)$ with Conley indices is given in Figure~\ref{fig:example_trivial_index_2D}. 
Morse nodes 0, 1, and 3 have nontrivial Conley indices and thus correspond to nontrivial invariant sets. 
In fact, by Theorem~\ref{prop:equilibrium-existence} the associate Morse sets $M(i)$, $i=0,1,3$ are fixed points.
One can check that for a particular choice of parameter values the equilibrium at $M(0)$ has complex eigenvalues and thus trajectories that converge to $M(0)$ do so by spiralling inward.
This oscillation is captured by $\cF_3 \colon \cX\mvmap \cX$ and results in Morse node 2, which lies below node 3 and above node 0.
By Poincar\'{e}-Bendixson Theorem, if there is a connecting orbit from $M(3)$ to $M(0)$ then $M(2) = \emptyset$ (this happens for ramp systems at tested parameter values), but without further analysis we cannot conclude that $M(2) = \emptyset$ for all parameter values in $R(47)$.

\begin{figure*}[!htb]
\centering
\includegraphics[width=0.45\textwidth]{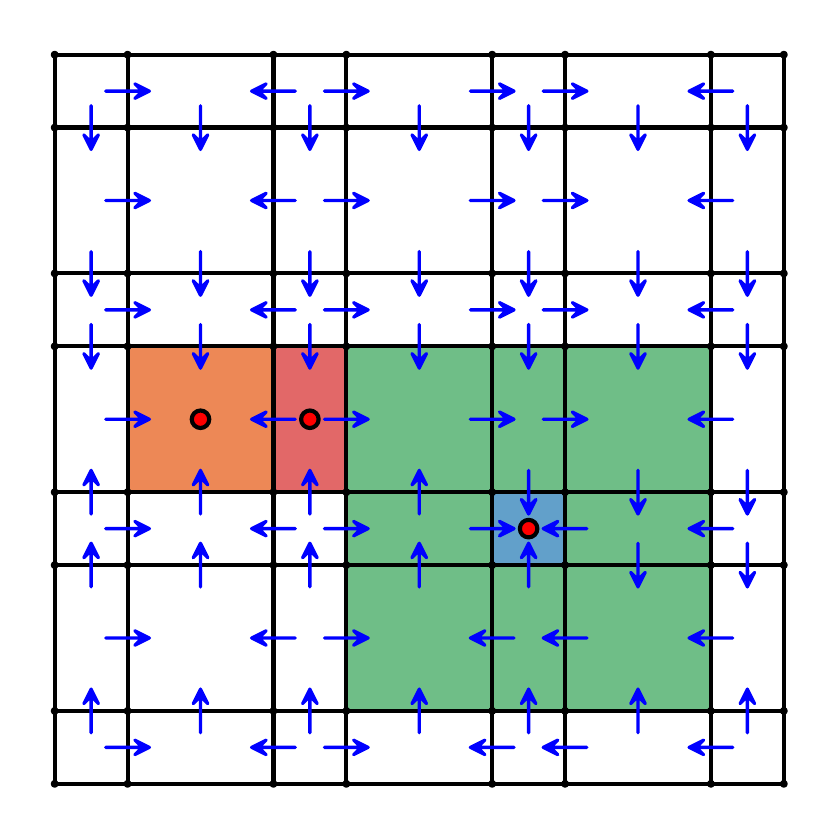}
\includegraphics[width=0.45\textwidth]{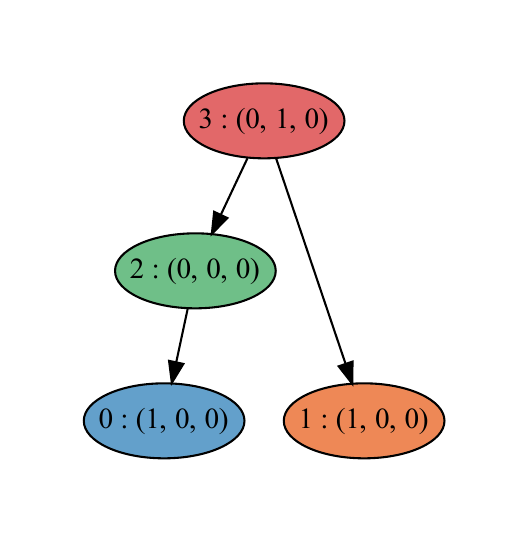}
\caption{Directed graph $\cF_3$ superimposed on $\cX_b$ and the corresponding Morse graph for node $47$ of the parameter graph of the network in Figure~\ref{fig:example_2_2d_network}.}
\label{fig:example_trivial_index_2D}
\end{figure*}
\end{ex}

\begin{ex}
\label{ex:trivial_index_3D_F2}
Let $\cF_3$ be the map generated by node 65,571,607,721 in the parameter graph of the network in Figure~\ref{fig:3d_example_1RN}. The corresponding Morse graph is shown in Figure~\ref{fig:trivial_index_3D_F2}. The map $\cF_3$ has three equilibrium cells. Each of the recurrent components corresponding to nodes $0$, $1$, and $2$ consists of a single equilibrium cell of $\cF_3$. The recurrent component corresponding to Morse node $3$ has $61$ cells and $12$ pairs of double edges.
For this example we conjecture that $M(3) = \emptyset$.
However, as with the previous example, we cannot conclude that $M(3) = \emptyset$ without further analysis.

\begin{figure*}[!htb]
\centering
\includegraphics[width=0.45\textwidth]{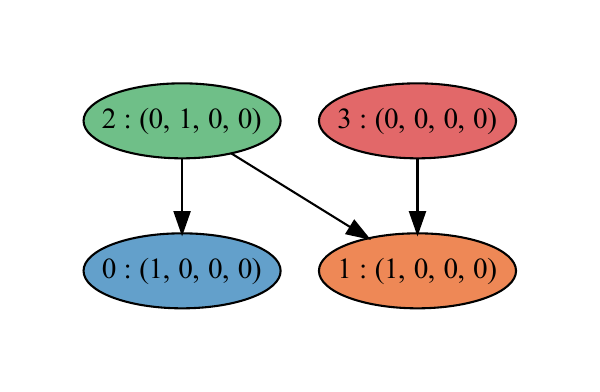}
\caption{Morse graph for node 65,571,607,721 of the parameter graph of the network in Figure~\ref{fig:3d_example_1RN}.}
\label{fig:trivial_index_3D_F2}
\end{figure*}
\end{ex}

\begin{ex}
\label{ex:trivial_indices}
We return to the three-node network illustrated in Figure~\ref{fig:3d_example_1RN}. As indicated in Chapter~\ref{sec:intro} the associated parameter graph has 87,280,405,632 nodes. Let $\cF_3$ be the map generated by the parameter node 52,717,613,010. 
The Morse graph computed from $\cF_3$ is presented in Figure~\ref{fig:example_2_3d_mg}. The computation of the Morse graph and the associated connection matrix for this example took $0.2$ seconds. For this system there are $8$ possible connection matrices.

Consider a ramp system $\dot{x} = -\Gamma x + E(x;\nu,\theta,h)$ for some $h \in \cH_3(\gamma,\nu,\theta)$ where $(\gamma,\nu,\theta)\in R(\text{52,717,613,010})$.
Since the Conley indices for nodes $p \in \setof{10, 11, 13, 18, 27}$ are trivial, we cannot conclude that the associated invariant set $M(p) \neq \emptyset$.
It is possible that the answer depends on the choice of parameter $(\gamma, \nu, \theta)$ within $R(\text{52,717,613,010})$ and $h \in \cH_3(\gamma, \nu, \theta)$.

The recurrent components corresponding to Morse nodes $10$, $11$, $13$, $18$, and $27$ contain $44$, $42$, $42$, $14$, and $14$ cells, respectively. The recurrent components corresponding to nodes $10$, $11$, and $13$ contain double edges that were not resolved by $\cF_2$ as in Example~\ref{ex:trivial_index_3D_F2}. The recurrent components corresponding to nodes $18$ and $27$ do not contain double edges as in Example~\ref{ex:trivial_index_2D}.

\begin{figure*}[!htb]
\centering
\includegraphics[width=1.0\textwidth]{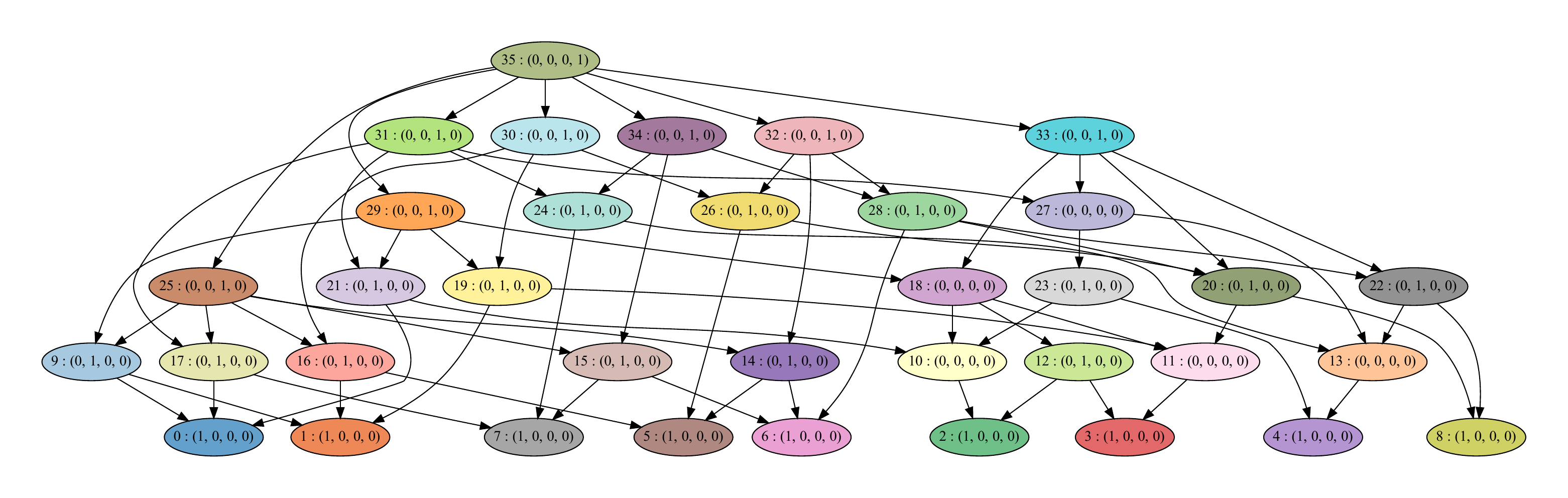}
\caption{Morse graph for the map $\cF_3$ generated by the parameter node 52,717,613,010 of the parameter graph of the network in Figure~\ref{fig:3d_example_1RN}.}
\label{fig:example_2_3d_mg}
\end{figure*}
\end{ex}

\begin{ex}
We again consider the network in Figure~\ref{fig:RN}. 
Let $\cF_2$ be the map generated by the parameter node 1,739,757,491,101.
The Morse graph computed from $\cF_2$ is presented in Figure~\ref{fig:example_6d_mg_2}. The computation of the Morse graph and the associated connection matrix for this example took $2.5$ minutes.

Consider a ramp system $\dot{x} = -\Gamma x + E(x;\nu,\theta,h)$ for some $h \in \cH_2(\gamma,\nu,\theta)$ where $(\gamma,\nu,\theta)\in R(\text{1,739,757,491,101})$.
Since the Conley indices for nodes $p \in \setof{15, 16, 17, 18, 21, 22, 23, 24}$ are trivial, we cannot conclude that the associated invariant set $M(p) \neq \emptyset$. For the remaining nodes $p \not\in \setof{15, 16, 17, 18, 21, 22, 23, 24}$, the Conley indices imply that $M(p)$ contains a fixed point.

\begin{figure*}[!htb]
\centering
\includegraphics[width=1.0\textwidth]{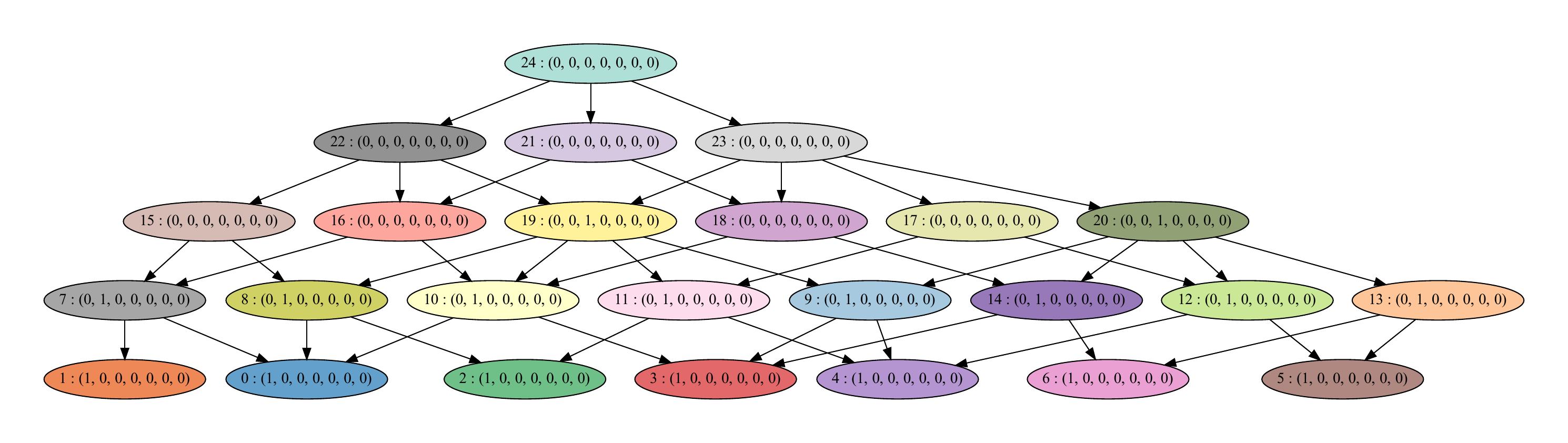}
\caption{Morse graph for the map $\cF_2$ generated by the parameter node 1,739,757,491,101 of the parameter graph of the network in Figure~\ref{fig:RN}.}
\label{fig:example_6d_mg_2}
\end{figure*}
\end{ex}

\section{Systems which are not DSGRN derived ramp systems}
\label{sec:NotDSGRN}

All the computations performed in the previous sections were done using DSGRN.
Recall that DSGRN takes a regulatory network and generates a parameter graph with the property that each node in the parameter graph gives rise to a unique wall labeling. 
Section~\ref{sec:ramp2rook} provides a mapping from  ramp systems to wall labelings.
However, the only ramp systems that provide wall labelings that are compatible with the current implementation of DSGRN are ramp systems for which each ramp function is to type $J=1$ (see Definition~\ref{defn:rampfunction}).
In this section we demonstrate that it is possible to perform computations for more general ramp systems.

\begin{ex}
\label{ex:ramp_van_der_pol}
To illustrate the fact that our method can be applied to ramp systems not derived from DSGRN (ramp functions with more than one ramp), we consider the ramp system
\begin{equation}
\label{eq:van_der_pol_ramp_system}
\begin{aligned}
\dot{x_1} & = - \gamma_1 x_1 + r_{1,1}(x_1) + r_{1,2}(x_2) \\
\dot{x_2} & = - \gamma_2 x_2 + r_{2,1}(x_1) + r_{2,2}(x_2)
\end{aligned}
\end{equation}
where $\gamma_1 = \gamma_2 = 2$ and $r_{i,j}$ are ramp functions with parameter values given in Table~\ref{tab:parameters_van_der_pol_ramp_system}. Plots of $r_{i, j}$ are shown in Figure~\ref{fig:van_der_pol_plots}. Notice that $r_{1,1}$ is cubic like and the other ramp functions are monotone with multiple ramps. The multivalued map $\cF_3$ and the corresponding Morse graph for \eqref{eq:van_der_pol_ramp_system} are shown in Figure~\ref{fig:stg_mg_van_der_pol}. The computation of the Morse graph and the associated connection matrix for this example took $0.03$ seconds.

\begin{table}[!htpb]
\centering
\renewcommand{\arraystretch}{1.2}
\begin{tabular}{@{}lll@{}}
\toprule
$\nu_{1,1} = (7, 3, 7, 3)$      & $\theta_{1,1} = (3, 5.1, 7.17)$  & $h_{1,1} = (0.15, 0.15, 0.15)$ \\
$\nu_{1,2} = (3.1, 4.9, 7.3)$   & $\theta_{1,2} = (5.3, 6.1)$      & $h_{1,2} = (0.11, 0.11)$       \\
$\nu_{2,1} = (7.5, 5.5, 4)$     & $\theta_{2,1} = (4, 5.5)$        & $h_{2,1} = (0.15, 0.15)$       \\
$\nu_{2,2} = (1.4, 5, 7.55, 9)$ & $\theta_{2,2} = (4, 5.85, 7.55)$ & $h_{2,2} = (0.11, 0.11, 0.11)$ \\
$\gamma_1 = \gamma_2 = 2$ \\
\bottomrule
\end{tabular}
\caption{Parameters values for \eqref{eq:van_der_pol_ramp_system}.}
\label{tab:parameters_van_der_pol_ramp_system}
\end{table}

\begin{figure*}[!htpb]
\centering
\includegraphics[width=0.45\textwidth]{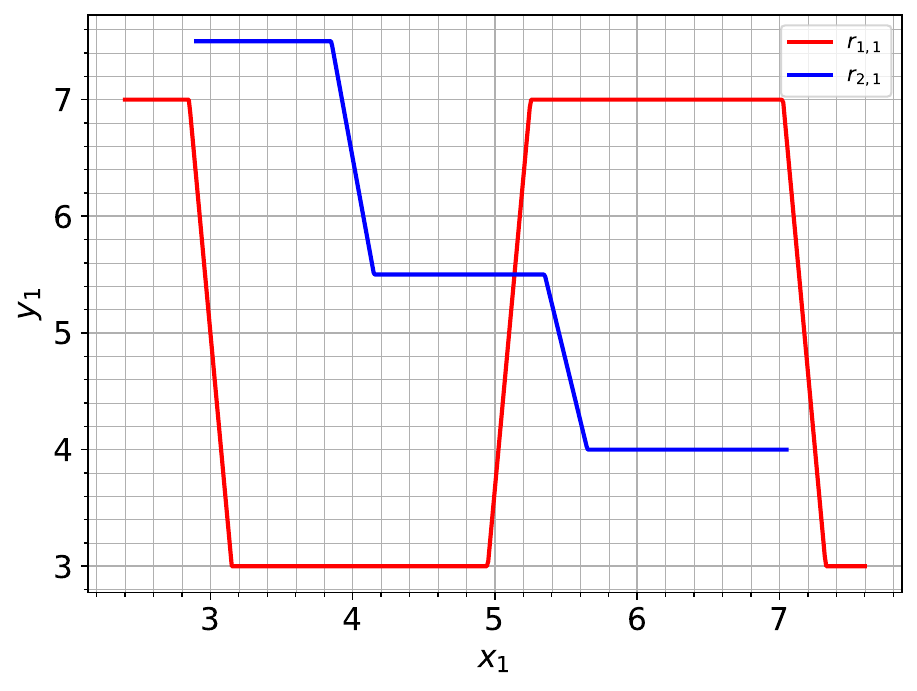}
\includegraphics[width=0.45\textwidth]{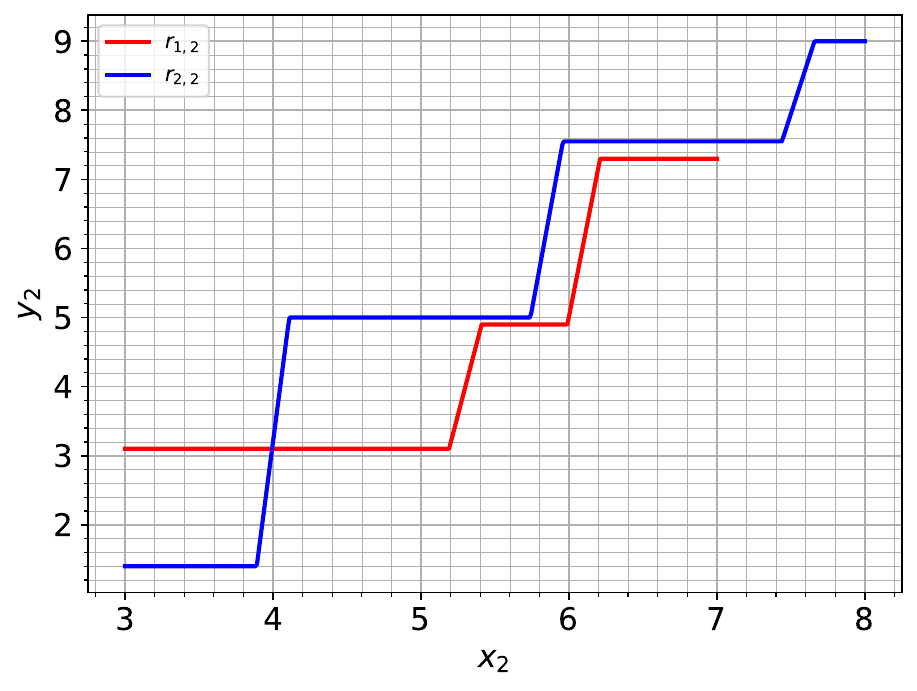}
\caption{Plots of $r_{i, j}$ from \eqref{eq:van_der_pol_ramp_system} with parameter values given in Table~\ref{tab:parameters_van_der_pol_ramp_system}.}
\label{fig:van_der_pol_plots}
\end{figure*}

\begin{figure*}[!htb]
\centering
\includegraphics[width=0.65\textwidth]{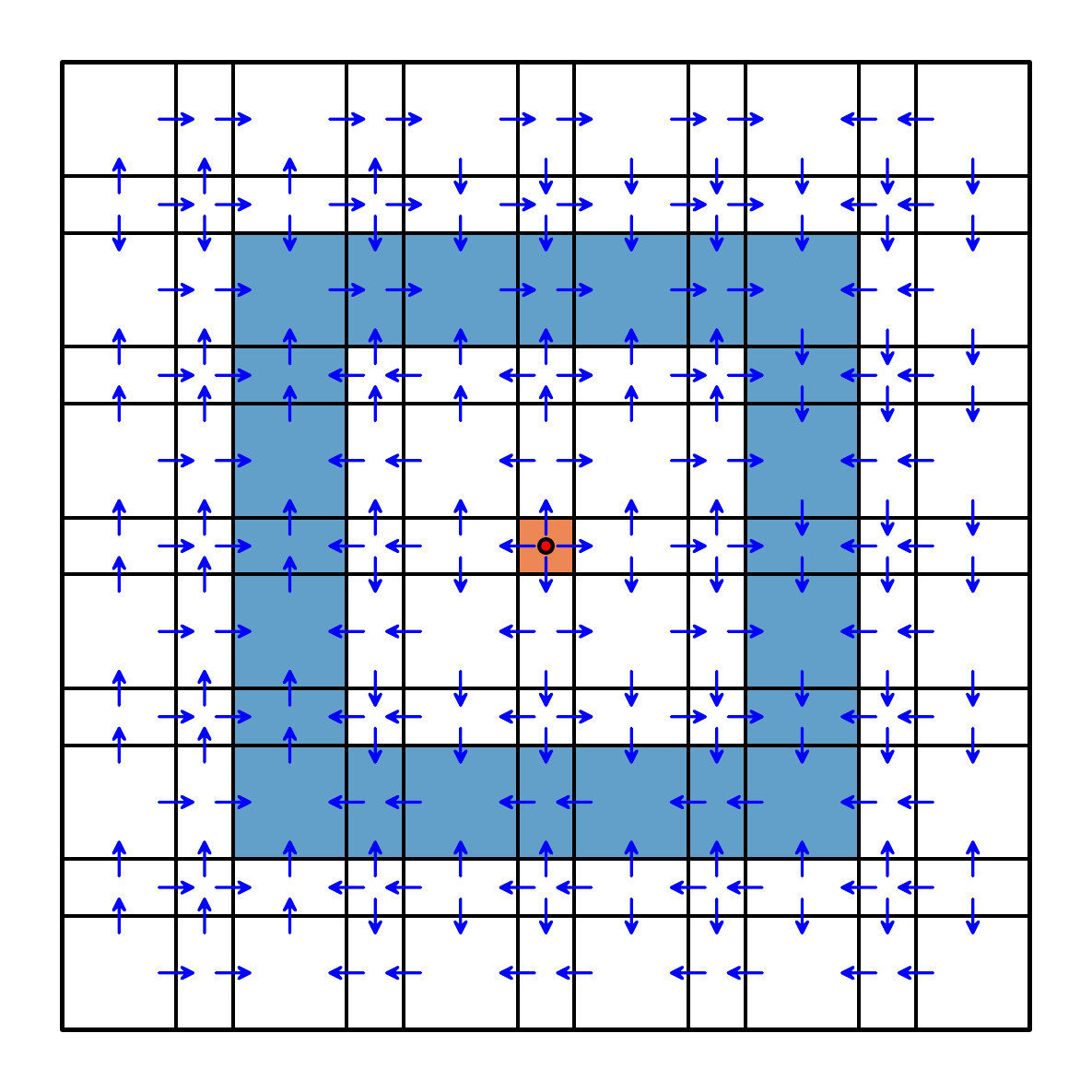}
\includegraphics[width=0.3\textwidth]{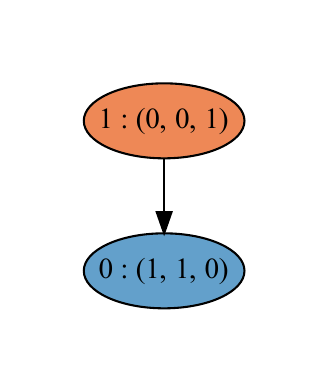}
\caption{Multivalued map $\cF_3$ and Morse sets (left) and the corresponding Morse graph (right) for \eqref{eq:van_der_pol_ramp_system}.}
\label{fig:stg_mg_van_der_pol}
\end{figure*}
\end{ex}

\section{Dynamics beyond ramp systems}
\label{sec:beyondRamp}

We chose ramp systems for this monograph in order to simplify the necessary analysis for Part~\ref{part:III}.
As shown in the previous examples of this Chapter, the results of Part~\ref{part:III} allow us to move directly from the combinatorial/homological computations described in Part~\ref{part:II} to conclusions about the dynamics of ramp system differential equations.
However, we claim that the computations provide insight to a much broader class of ODEs, e.g., systems where the nonlinearities $r_{i,j}$ are replaced by piecewise monotone functions.
Any formal justification for this claim greatly exceeds the scope of this monograph.
Nevertheless as the following simple example shows, we believe that  the constructive proofs of geometrizations of Chapters~\ref{sec:R0}, \ref{sec:R1Dynamics}, \ref{sec:R2Dynamics}, and \ref{sec:R3Dynamics} can be used to suggest geometrizations that are aligned with the vector fields of other ODEs.

\begin{ex}
\label{ex:basic2DODE}
The nonlinearities for the system of ODEs given by \eqref{eq:Hill_2D_example_ODE} are based on Hill functions that are a typical modeling choice for systems biology \cite{Somvanshi2013}.
The $H^-_{i,j}$ are monotone nonlinearities and furthermore there is a easy identification between the parameter values of the Hill system and the ramp system.
Thus to characterize the dynamics of \eqref{eq:Hill_2D_example_ODE} we choose a ramp system based on the regulatory network of Figure~\ref{fig:network_parameter_regions_intro}(A) with parameter values motivated by those of Table~\ref{tab:parameters_intro1}.
Since this is a two-dimensional system we can use the combinatorial model given by $\cF_3$ to identify the Morse graph and the connection matrix.
We choose a particularly simple geometrization for the \eqref{eq:Hill_2D_example_ODE} that is shown explicitly in Figure~\ref{fig:phase_portrait_Hill_intro}:
the rectangular cells come from the cells of the rectangular geometrization used for $\cF_3$; and the non-rectangular cells  are built by using the simplest piecewise linear approximations to the manifolds identified in Chapter~\ref{sec:R3Dynamics}.
Numerically, this geometrization is aligned with the vector field \eqref{eq:Hill_2D_example_ODE} for all $\cN\in \sN(\cF_3)$.
This in turn implies that we have produced a lattice of attracting blocks for \eqref{eq:Hill_2D_example_ODE}, and thus the Morse graph and connection matrix are valid for \eqref{eq:Hill_2D_example_ODE}.
\end{ex}

\chapter{Future Directions and Concluding Remarks}
\label{sec:futurework}

As is emphasized in Chapter~\ref{sec:intro} we view this monograph as providing a novel, computationally feasible approach to the analysis of the global dynamics of ODEs. 
Furthermore, we believe that the examples presented in Chapter~\ref{sec:examples} affirm the potential of this approach.
However, there are limitations to the results presented here and we take this final chapter to discuss the opportunities and challenges  to extending this work.

The most obvious restriction is that the multivalued map $\cF_3$, which addresses the dynamics near equilibria, is only developed in dimensions two and three.
As discussed in Section~\ref{sec:equilibria} we believe that the function $\lap_{\xi,\sigma}\colon \Top_\cX(\xi) \to \mathbb{N}$ (see Definition~\ref{defn:lap-number}) provides the necessary information to define $\cF_3$ for general dimensions.
However, at the time of this writing the proofs of transversality have not yet been obtained.

Because of their subtlety we leave discussion and investigation of  double edges (even in dimension $N=3$) that remain for all $\cF_i$, $i=1,2,3$ to future work.

On a more positive note, Section~\ref{sec:alternativeNonlinearities} discusses the opportunity to use the constructions of this paper to address systems of ODEs involving nonlinearities different from the ramp systems, e.g., the example of presented in Chapter~\ref{sec:intro} involving  Hill functions.
However, there are important caveats and we use the classical van der Pol equation to demonstrate the associated challenges and discuss potential solutions.

The focus of this monograph is identifying dynamics of an ODE at a fixed parameter value.
However, as is discussed in Chapter~\ref{sec:DSGRN} the combinatorial and homological computations are constant over open regions of parameter space that are explicitly defined and organized via the parameter graph.
This raises the possibility of extending the machinery of this paper to understand the dynamics over the entire high dimensional parameter space.

\section{Equilibria and Lap Number}
\label{sec:equilibria}

Motivation for construction of the map $\cF_3$ comes from theory of linear cyclic feedback systems~\cite{gedeon:mischaikow:95,gedeon:memoirs}. 
The dynamics of the ramp system in equilibrium cells $\xi \in \cX$  with $o_\xi = \sigma$ consisting of a single cycle, is governed, after a change of variables, by a linear system 
\begin{equation}\label{cfs}
    \dot{x}_j = -\gamma_j x_j + b_{jk} x_k, \qquad j=\sigma(k),
\end{equation} 
where $0$ corresponds to the equilibrium $p \in \bg(b(\xi))$ of the ramp system and $j \in S_\sigma$. For a fixed set of parameters $(\gamma,\nu,\theta) \in \cH_3(\gamma,\nu,\theta)$, one can show~\cite{gedeon:memoirs} that the eigenvalues of this linear system have stereotypical structure when the total gain
\[ \Pi_{i\in S} b_{o_\sigma(j)j}\]
is sufficiently large, that is, when $h \in \cH_3(\gamma,\nu,\theta)$ since the total gain is given by our choice of $h$. 

In this framework, the position of stable and unstable linear subspaces can be characterized by an integer-valued lap number which is assigned to orthants of $\mathbb{R}^N$ and extended, whenever possible, by continuity to hyperplanes in between them. 
In our context, that characterization is captured by $\lap_{\xi,\sigma}(\mu)$ (see Definition~\ref{defn:lap-number}), which assigns to each top cell $\mu\in \Top_\cX(\xi)$ the corresponding value that it would associate to the orthant. 

Importantly, the unstable (linear) manifold of the fixed point $0$ in (\ref{cfs}) is, for sufficiently large gain, contained in a set of orthants of $\mathbb{R}^N$ characterized by a low lap number. 
In terms of our ramp systems, the local unstable manifold is contained in the geometrization of a subset of cells
\[ W^u(p) \subseteq \bg(\blup(\xi)) \cup \bigcup_{\zeta \in U(\xi)} \bg(b(\zeta)) .\] 
This observation is used to define the map $\cF_3$ in dimensions $N=2,3$. 

While the cyclic feedback systems~\cite{gedeon:mischaikow:95,gedeon:memoirs}
describe in detail the behavior of the systems around an equilibrium when $o_\xi=\sigma$ is a single cycle, the lap number $\lap_{\xi,\sigma}(\mu)$ was defined in Definition~\ref{defn:lap-number} for any cycle $\sigma$ in the decomposition $o_\xi =\sigma_{\xi,1} \ldots \sigma_{\xi,r}$ of an equilibrium cell. 
For such equilibrium cells $\xi$, a top cell $\mu \in \Top_\cX(\xi)$ will be assigned $r$ different lap numbers. This collection of lap numbers still contains the information about the linear dynamics of the ramp system in the neighborhood of $\xi$. If all lap numbers $\lap_{\xi,\sigma_{\xi,j}}(\mu) <\frac{k_j}{2}$ assigned to $\mu \in T(\xi)$ are below the cycle-length dependent threshold $\frac{k_j}{2}$, then this indicates that the unstable manifold $W^u(0)$ of equilibrium $0$ in (\ref{cfs}) intersects the orthant with same relative position as $\mu$ and the linear flow of (\ref{cfs}) between cells corresponding to $\xi$ and $\mu$ generally goes from $\bg(b(\xi))$ to $\bg(b(\mu))$. If all lap numbers satisfy $\lap_{\xi,\sigma_{\xi,j}}(\mu) \geq \frac{k_j}{2}$, then the linear flow of (\ref{cfs}) between cells corresponding to $\xi$ and $\mu$ generally goes from $\bg(b(\mu))$ to $\bg(\b(\xi))$. For all the other cells $\mu \in T(\xi)$ for which lap numbers associated to some cycles are less than the threshold, and some that are larger than the threshold, the dynamics is "saddle-like". For such cells some choices of geometrizations of $\Janus$ will lead to sections where the flow leaves the neighborhood of $\xi$ and other choices to sections where flow is entering the neighborhood of $\xi$. 
While we do not expect a conceptual difficulty in extension of $\cF_3$ to dimension $N>4$, since  $o_\xi =\sigma_{\xi,1} \ldots \sigma_{\xi,r}$ may  have two or more cycles $\sigma_{\xi,j} $ of length $k_j \geq 2$, the construction of a geometrization $\bG_3$ aligned with the flow is technically challenging and thus left as future work.

\section{Alternative Nonlinear Systems}
\label{sec:alternativeNonlinearities}

This monograph grew out of the need to give mathematical justification that the computations of the DSGRN software was accurately representing the dynamics of nonlinear ODEs that are reasonable models for gene regulatory networks that arise is applications associated with systems biology.
As described in Section~\ref{sec:DSGRN} the input for DSGRN is a regulatory network and the output is a parameter graph and to each node the assignment of a Conley complex.
The process of arriving at the Conley complex is the content of Part~\ref{part:II}, none of which is dependent on the choice of an ODE.

There are two reasons for the choice of ramp systems in Part~\ref{part:III}.
The first, which we return to below, is to match the form of models associated with gene regulatory networks:
\begin{equation}
\label{eq:geneRE}
    \dot{x}_n = -\gamma x_n + F_n(x)
\end{equation}
where $x_n$ can be thought of the quantity of  protein produced by gene $n$, $\gamma >0$ is the decay rate of this protein, and $F_n$ is a nonlinear function indicating how the proteins associated with other genes affects the rate of production of $x_n$.
The second is that the piecewise linear nature of the ramp functions allows for an exact description of the dynamics near equilibria.
 
The most important feature of Conley theory is the following (see \cite{conley:cbms,franzosa:89}).
Let $N$ be an  isolating neighborhood for a Lipschitz continuous differential equation $\dot{x} = f(x)$ on $R^N$. 
Let $\sJ^\vee(\sK)$ be a Morse decomposition of $\sInvset(N,f)$ with connection matrix 
\[
\Delta \colon \bigoplus_{K\in\sJ^\vee(\sK)} CH_*(K;\F) \to \bigoplus_{K\in\sJ^\vee(\sK)} CH_*(K;\F).
\]
Then, there exists $\epsilon >0$ such that if $g$ is Lipschitz continuous and $\sup_{x\in S}\| f(x)-g(x)\| <\epsilon$, then $\sJ^\vee(\sK)$ is a Morse decomposition of $\sInvset(N,g)$ with connection matrix 
\[
\Delta \colon \bigoplus_{K\in\sJ^\vee(\sK)} CH_*(K;\F) \to \bigoplus_{K\in\sJ^\vee(\sK)} CH_*(K;\F).
\]
In other words, the gradient-like structure and the homological information is preserved for sufficiently small perturbations.

The implication for this work is as follows.
Given a ramp system $\dot{x} = -\Gamma x + E(x;\nu,\theta,h)$ we can compute a Morse decomposition and connection matrix.
Assume that $F$ is a nonlinearity that captures the nonlinear interactions associated with gene regulation (in particular we are assuming that we have identified the appropriate regulatory network -- this is not a given in most biological settings).
Since gene regulation is a multiscale phenomenon it is impossible to derive $F_n$ from first principles. 
In applications explicit choices of the nonlinearity are typical made based on heuristics and thus it is reasonable to assume that $F = (F_1,\ldots, F_n)$ is not known.
However, if $\|F-E\|$ is sufficiently small, then we have characterized the dynamics of the associated biology.

Typically the statement ``if $\|F-E\|$ is sufficiently small'' appears  to be a rather strong assumption.
However, via the DSGRN parameter graph (see \cite{cummins:gedeon:harker:mischaikow:mok,kepley:mischaikow:zhang} for more detail) we can determine the complete set of wall labelings.\footnote{For sparse regulatory networks, e.g., less than five incoming edges per node, this is pre-computed and stored in DSGRN.} 
Thus, we know what finite set of computations identify  Morse graphs and connection matrices of the ramp systems over all of parameter space.
Thus, if the nonlinearity of \eqref{eq:geneRE} can be approximated by type I or type II ramp functions $E$ (see Definition~\ref{defn:rampnonlinearity}) then our approach provides a characterization of the global dynamics.

We hasten to remind the reader that if $F$ is specified, then there are at least two alternative directions that could be pursued.
First, as described in Section~\ref{sec:beyondRamp} we believe that this framework provides a potentially new approach to approximating and characterizing the dynamics of \eqref{eq:geneRE}.
Second, assume for the sake of simplicity that $F$ is constructed using Hill functions.
In principle one could attempt to repeat the estimates and constructions of Part~\ref{part:III} where $r^\pm$ is replaced by $H^\pm$. 
This appears to be a challenging project since even the  identification of optimal rectangular geometrizations as an explicit function of parameters is not clear. 

Having emphasized the flexibility and potential for wider applicability of our framework it is important to point out that we need to expand its capabilities.
Consider the $N$-dimensional generalized Lotka-Volterra equations that take the form
\begin{equation}
\label{eq:genLV}
\dot{x} = -\Gamma x +  (Ax)^T \cdot x
\end{equation}
where the \emph{interaction matrix} $A$ is a real $N\times N$ matrix.
As in this monograph the phase space of interest is $[0,\infty)^N$, the matrix $\Gamma$ is diagonal, and the interaction matrix takes the place of the regulatory network.
However, there are three significant differences:
(i) the diagonal elements of $\Gamma$ are assumed to be non-zero as opposed to positive; (ii) the nonlinearities are unbounded; and (iii) the nonlinearities can be quadratic with respect to any given variable.

These differences have consequences. 
On the level of the dynamics of \eqref{eq:genLV}  there need not be a global compact attractor and it 
is not clear how (or if) ramp systems provide optimal approximations to \eqref{eq:genLV}.
On the combinatorial level under appropriate modification the computational machinery of DSGRN is still applicable (see \cite{cuello:gameiro:bonachela:mischaikow}) and with respect to these modifications it makes sense to define wall labelings.
However, we cannot assume that the wall labelings are strongly dissipative and it is not clear that the conditions that lead to $\cF_2$ and $\cF_3$ are optimal in this context.
We do not believe that these challenges are insurmountable, but they have not been addressed. 

Consider the van der Pol equation \cite{hirsch:smale}
\[
\begin{aligned}
\dot{u} & = v - u^3 + u \\
\dot{v} & = - u
\end{aligned}
\]
where the phase space is $\R^2$ and the global dynamics consists of a unique unstable fixed point at the origin whose two-dimensional unstable manifold connects to the unique stable periodic orbit. Making the change of variables $x_1 = u + 5$ and $x_2 = v + 5$ we can write this system as
\begin{equation}
\label{eq:van_der_pol}
\begin{aligned}
\dot{x_1} & = - \gamma_1 x_1 + f_1(x_1) + g_1(x_2) \\
\dot{x_2} & = - \gamma_2 x_2 + f_2(x_1) + g_2(x_2)
\end{aligned}
\end{equation}
where
\[
f_1(x_1) = - (x_1 - 5)^3 + x_1 - 10 + \gamma_1 x_1, \quad g_1(x_2) = x_2,
\]
and
\[
f_2(x_1) =  - x_1 + 10, \quad g_2(x_2) = \gamma_2 x_2 - 5.
\]
Taking $\gamma_1 = \gamma_2 = 2$ we can try to compute the dynamics of this system by approximate \eqref{eq:van_der_pol} by a ramp system of the form
\begin{equation}
\label{eq:van_der_pol_ramp_system_new}
\begin{aligned}
\dot{x_1} & = - \gamma_1 x_1 + r_{1,1}(x_1) + r_{1,2}(x_2) \\
\dot{x_2} & = - \gamma_2 x_2 + r_{2,1}(x_1) + r_{2,2}(x_2)
\end{aligned}
\end{equation}
where $r_{i,j}$ are ramp functions, and then compute the map $\cF_3$ and the corresponding Morse graph for this ramp system. This is in fact what is done in Example~\ref{ex:ramp_van_der_pol}, where \eqref{eq:van_der_pol_ramp_system} is a ramp system approximation to \eqref{eq:van_der_pol}. Plots of $f_i$, $g_i$, and $r_{i, j}$ from \eqref{eq:van_der_pol_ramp_system} are shown in Figure~\ref{fig:van_der_pol_f_ramp_plots}. The multi-valued map $\cF_3$ and the corresponding Morse graph for \eqref{eq:van_der_pol_ramp_system} are shown in Figure~\ref{fig:stg_mg_van_der_pol}.

\begin{figure*}[!htpb]
\centering
\includegraphics[width=0.45\textwidth]{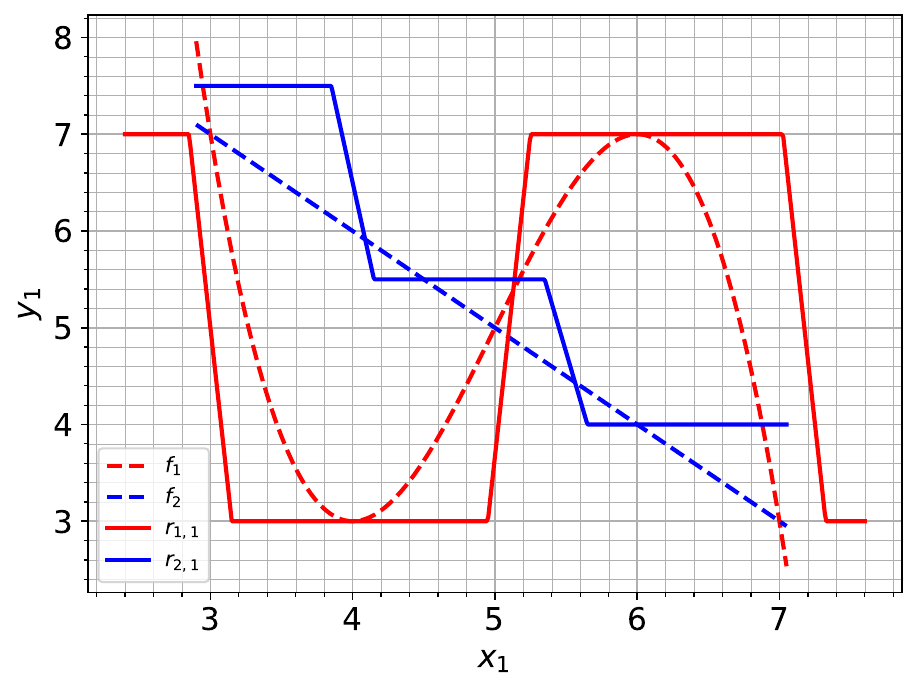}
\includegraphics[width=0.45\textwidth]{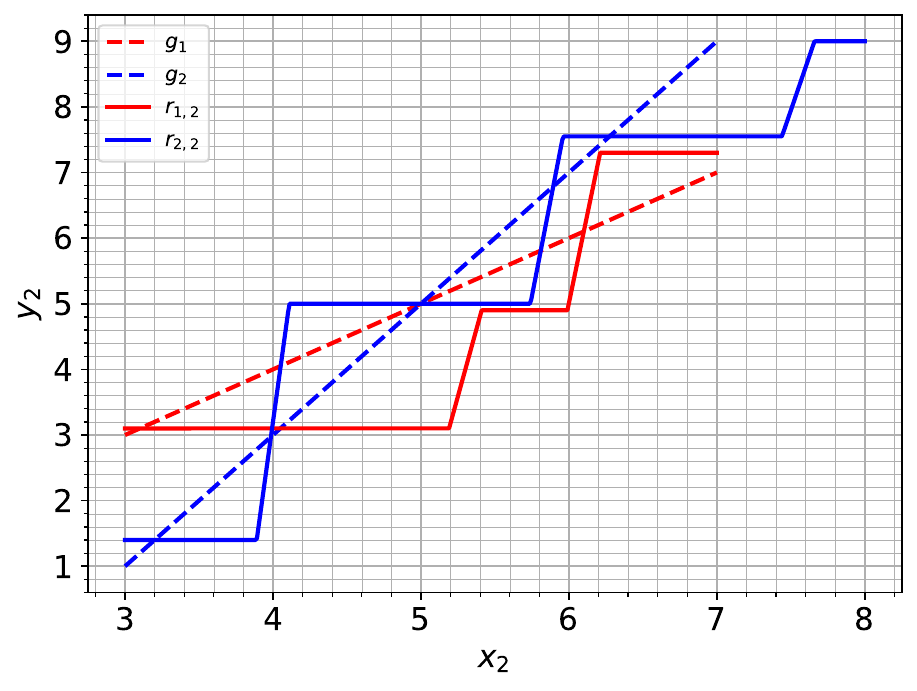}
\caption{Plots of $f_i$ and $g_i$ from \eqref{eq:van_der_pol} and $r_{i, j}$ from \eqref{eq:van_der_pol_ramp_system} with parameter values given in Table~\ref{tab:parameters_van_der_pol_ramp_system}.}
\label{fig:van_der_pol_f_ramp_plots}
\end{figure*}

As one can see in Figure~\ref{fig:stg_mg_van_der_pol} we can indeed capture the dynamics of \eqref{eq:van_der_pol} by a ramp system. However, a fair question is how can we justify that the dynamics in Figure~\ref{fig:stg_mg_van_der_pol}, computed for the ramp system \eqref{eq:van_der_pol_ramp_system}, is valid for the original system \eqref{eq:van_der_pol}. As we can see from Figure~\ref{fig:van_der_pol_f_ramp_plots} the ramp system \eqref{eq:van_der_pol_ramp_system} does not provide a good approximation to \eqref{eq:van_der_pol}. The obvious possible solution to this question is to try to use a better approximation, which can be achieved by using ramp functions with more ramps. Ramp functions with $10$ ramps each approximating \eqref{eq:van_der_pol} are shown in Figure~\ref{fig:van_der_pol_f_10_ramps_plots}.

\begin{figure*}[!htpb]
\centering
\includegraphics[width=0.45\textwidth]{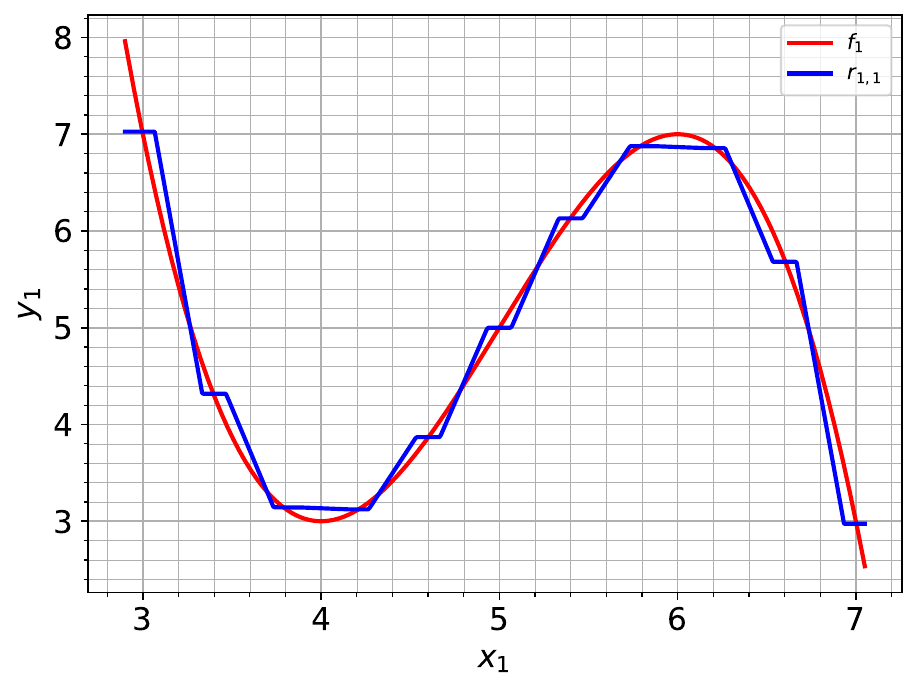}
\includegraphics[width=0.45\textwidth]{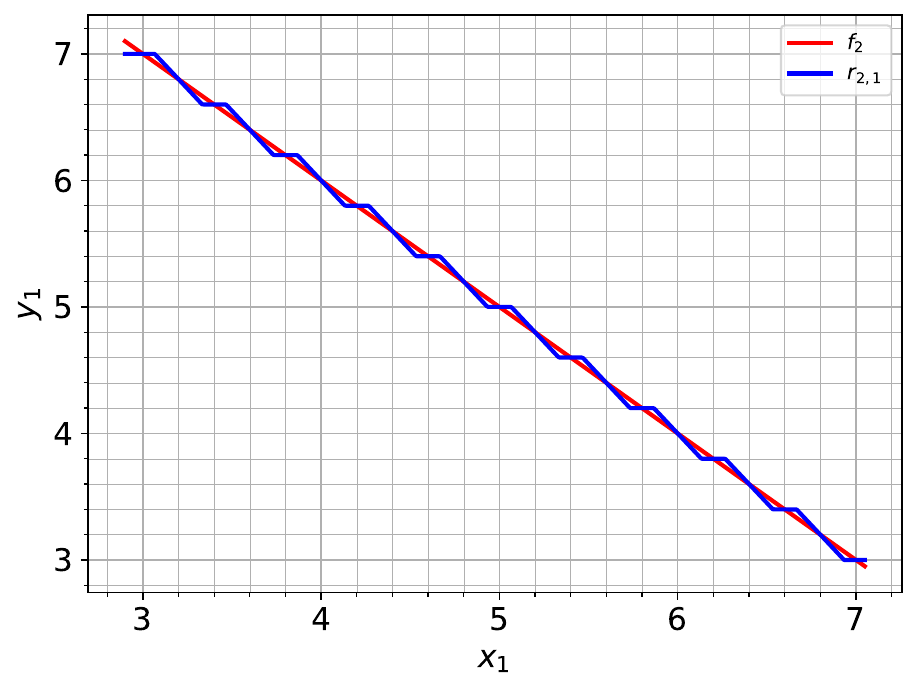}
\includegraphics[width=0.45\textwidth]{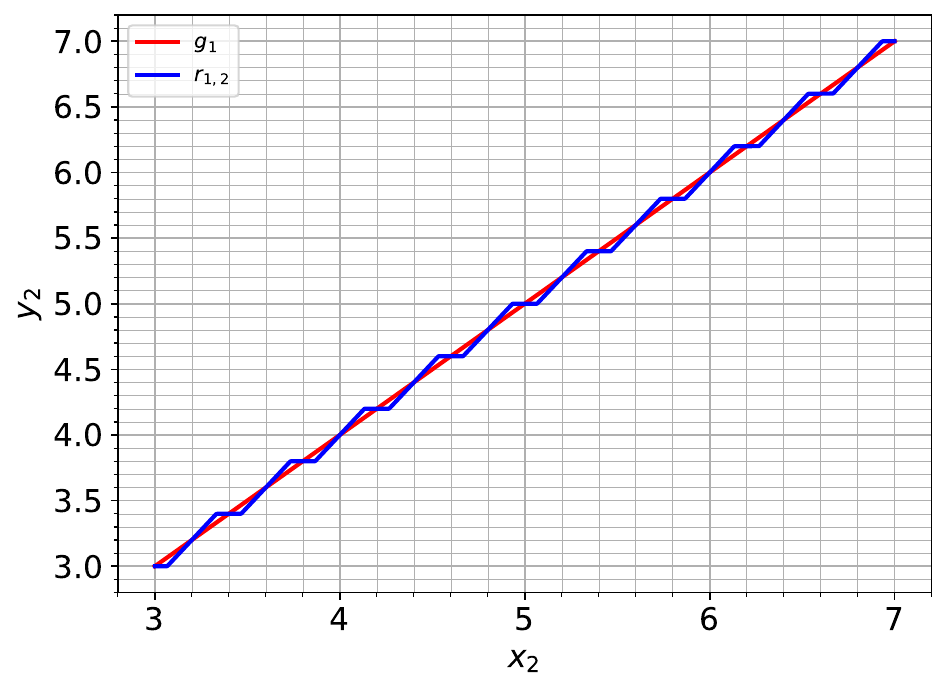}
\includegraphics[width=0.45\textwidth]{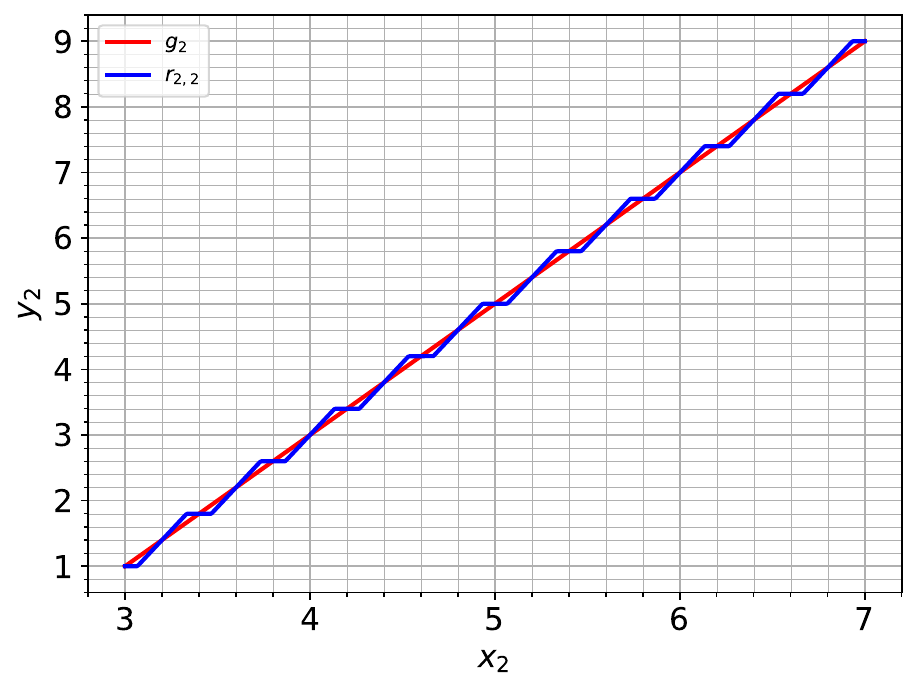}
\caption{Plots of $f_i$ and $g_i$ from \eqref{eq:van_der_pol} and $r_{i, j}$ from \eqref{eq:van_der_pol_ramp_system_new}, where the $r_{i, j}$ are ramp functions with $10$ ramps each.}
\label{fig:van_der_pol_f_10_ramps_plots}
\end{figure*}

The Morse graph for $\cF_3$ computed from the ramp functions in Figure~\ref{fig:van_der_pol_f_10_ramps_plots} is presented in Figure~\ref{fig:morse_graph_van_der_pol_10_ramps}. This is not the expected Morse graph for \eqref{eq:van_der_pol}. Using better approximations with more ramps do not solve this problem. This example indicates that to obtain the dynamics of a given ODE it is not enough to approximate the ODE by ramp functions. Additional restrictions need to be imposed in the ramp system approximation.

\begin{figure*}[!htb]
\centering
\includegraphics[width=0.8\textwidth]{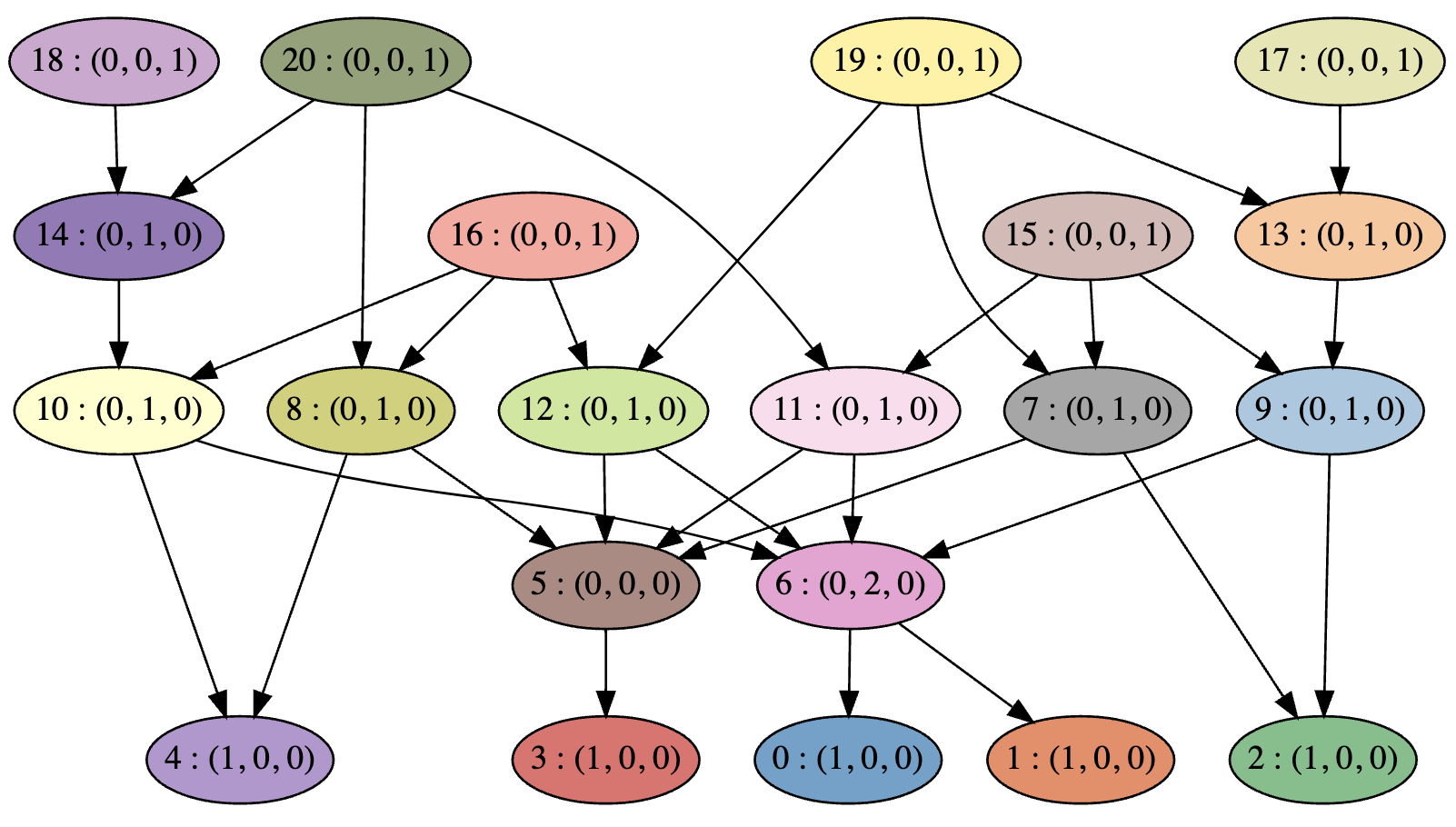}
\caption{Morse graph for $\cF_3$ computed using the ramp functions in Figure~\ref{fig:van_der_pol_f_10_ramps_plots}.}
\label{fig:morse_graph_van_der_pol_10_ramps}
\end{figure*}

\section{Concluding Remarks}

The monograph provides a novel approach to understanding the global dynamics of ODEs.
The optimistic perspective is that the techniques presented here could be generalized to dissipative systems whose nonlinearities can be decomposed into finitely many monotone functions, and furthermore, that the most fundamental information to apply these techniques takes the form of a network that describes the interactions between variables. 

While we are far from attaining this optimistic dream, we have shown that it can be at least partially realized in the context of ramp systems.
We have provided an algorithm that takes a ramp system, identifies an appropriate wall labeling, produces combinatorial models of the dynamics, and from this rigorously characterizes the global dynamics via a Morse graph and Conley complexes. 
Furthermore, to the best of our knowledge our approach is computationally much more efficient than  any other numerical method.

The fact that the Morse graph and Conley complex information is robust with respect to perturbation, leads us to believe that computations done with the DSGRN technology can at the very least rapidly provide intuition about the dynamics of systems of moderate dimension. 
This is a new capability that we hope finds broad applications.

We have argued (see Section~\ref{sec:alternativeNonlinearities}) that the constructive proofs provided in Part~\ref{part:III} have the potential to be extended either computationally or analytically to obtain rigorous results for systems with a broader class of nonlinearities.
We also argued that  \cite{cuello:gameiro:bonachela:mischaikow} suggests that it might be possible to consider systems with a broader class of linear terms.
We hope we have inspired readers to consider pursuing these possibilities.

\chapter*{Acknowledgments}

M.G. and K.M. were partially supported by DARPA contract HR0011-16-2-0033,  National Institutes of Health award R01 GM126555, and Air Force Office of Scientific Research under award number FA9550-23-1-0011 and FA9550-23-1-0400. M.G. was partially supported by FAPESP grant 2019/06249-7 and CNPq grant 309073/2019-7. K.M. was partially support by the National Science Foundation under awards DMS-1839294 and a grant from the Simons Foundation. K.M. and E.V. were partially supported by the National Science Foundation under awards HDR TRIPODS award CCF-1934924. T.G. was partially supported by NSF grant DMS-1839299,  DARPA FA8750-17-C-0054 and NIH 5R01GM126555-01. H.K. was partially supported by JSPS KAKENHI Grant Numbers 18H03671 and 20K20520. H.O. was partially supported by JSPS KAKENHI Grant Numbers 19K03644 and 23K03240.

\appendix
\chapter{DSGRN Visualization Tool}
\label{sec:appendix_visualization}

\begin{center}
By Daniel Gameiro
\end{center}

\bigskip
\bigskip

The DSGRN visualization tool is an interactive website that can be used to visualize the DSGRN results for 2D and 3D networks. It is available at

\medskip

\begin{center}
\url{https://chomp.rutgers.edu/projects/dsgrn_viz/index.html}
\end{center}

\medskip

The DSGRN visualization tool requires a JSON database with the DSGRN results. Instructions for creating a JSON database are available on the DSGRN repository \cite{DSGRN}. The user can upload a JSON database generated by DSGRN or select one of the available sample databases to view. Once a database is uploaded or selected, the first parameter node in the JSON file is automatically selected and the results are displayed. The top panel shows the network on the left with information about the network and the selected parameter node under it, the Morse graph in the middle, and the cell complex on the right with the Morse sets highlighted, as shown in Figure~\ref{fig:vis_top}. The bottom panel shows the parameter graph, with the selected node and its neighbors highlighted, as shown in Figure~\ref{fig:vis_bottom}. To select a different parameter node to visualize, click on the node in the parameter graph or select from the drop-down menu at the top. To view a specific Morse set, click on a node in the Morse graph. The user can also select more than one Morse set by changing the \emph{selection method} option to \emph{multiple Morse sets}, or select an interval of Morse sets by changing it to \emph{interval}. To return to viewing all Morse sets after a selection, simply click on one of the the selected Morse graph nodes. There are sliders available to change the line width, face opacity, and arrow size in the cell complex. It is also possible to select whether to show the wireframe of the cell complex or not, and which arrows to show (all, single only, double only, self only, single and double, or none). Finally, the user can change the colormap used for the Morse sets from the drop-down menu at the top of the page.

\begin{figure}[!htbp]
\centering
\includegraphics[width=1.0\textwidth]{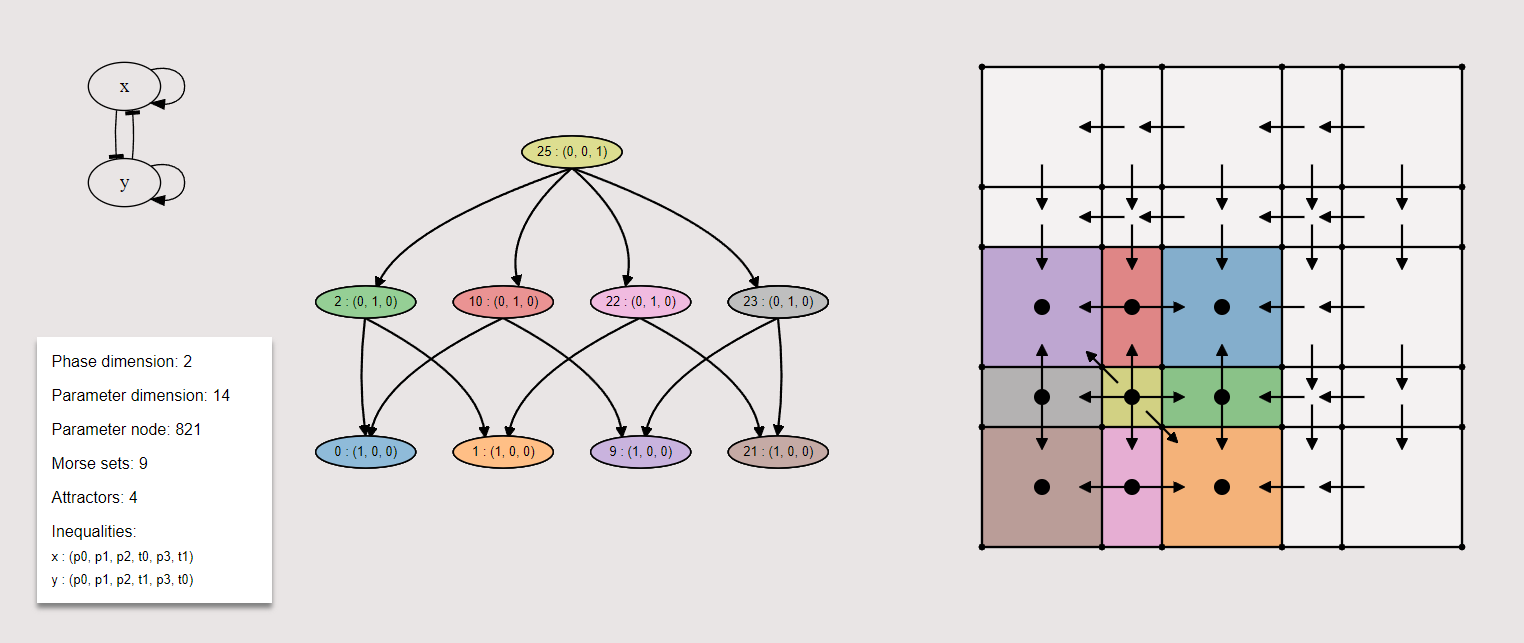}
\caption{Top panel of the visualization tool.}
\label{fig:vis_top}
\end{figure}

\begin{figure}[!htbp]
\centering
\includegraphics[width=0.7\textwidth]{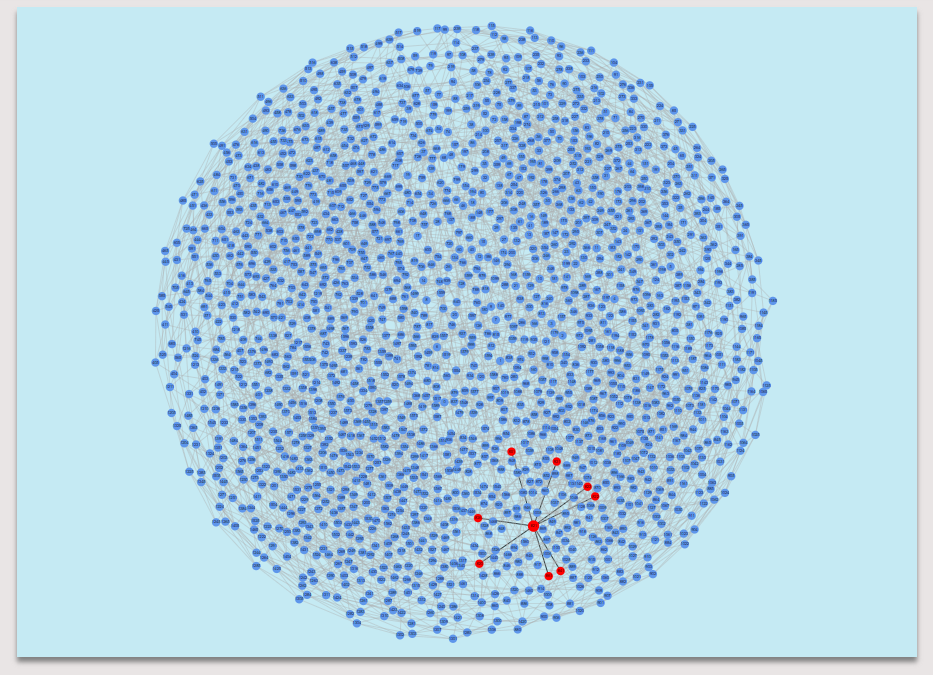}
\caption{Bottom panel of the visualization tool.}
\label{fig:vis_bottom}
\end{figure}

\backmatter
\bibliographystyle{amsplain}
\bibliography{rookbib}

\end{document}